\documentclass[10pt]{article}
\usepackage[margin=1in]{geometry}

\usepackage{setspace}
\usepackage[toc,page,header]{appendix}
\usepackage{minitoc}

\RequirePackage{amsthm,amsmath,amsfonts,amssymb}
\RequirePackage{natbib}
\RequirePackage[colorlinks,allcolors=blue,citecolor=blue,urlcolor=blue]{hyperref}
\RequirePackage{graphicx}
\usepackage{xcolor} 

\usepackage{enumitem,kantlipsum}
\usepackage{caption}
\usepackage{dsfont}
\usepackage{subcaption}
\usepackage{makecell}
\usepackage{mathtools}
\usepackage{comment}
\usepackage{booktabs}
\usepackage[ruled,linesnumbered]{algorithm2e}
\theoremstyle{plain}
\newtheorem{axiom}{Axiom}
\newtheorem{claim}[axiom]{Claim}
\newtheorem{theorem}{Theorem}[section]
\newtheorem*{theorem*}{Theorem}
\newtheorem*{lemma*}{Lemma}
\newtheorem{corollary}{Corollary}[section]
\newtheorem{lemma}[theorem]{Lemma}
\newtheorem{proposition}[theorem]{Proposition}
\theoremstyle{remark}

\newtheorem{assumption}[theorem]{Assumption}

\newtheorem{remark}{Remark}

\DeclareMathOperator*{\argmin}{arg\,min}
\def\ind{\mathds{1}}
\newcommand{\bb}[1]{\mathbb{#1}}
\newcommand{\mk}[1]{\mathfrak{#1}}
\SetKwComment{Comment}{/* }{ */}
\newcommand{\norm}[1]{\left\lVert#1\right\rVert}
\newcommand{\bignorm}[1]{\big\lVert#1\big\rVert}
\newcommand{\Bignorm}[1]{\Big\lVert#1\Big\rVert}

\newcommand{\ti}{_{2,\infty}}
\def\mod{\mathsf{mod}}
\def\exc{\mathsf{exc}}

\newcommand{\eq}[1]{\begin{equation} #1\end{equation}}
\newcommand{\ip}[1]{\big\langle#1 \big\rangle}
\newcommand{\mc}[1]{\mathcal{#1}}
\newcommand{\longeq}[1]{\begin{equation}
    \begin{split}
        #1
    \end{split}
\end{equation}}

\newcommand\blfootnote[1]{%
  \begingroup
  \renewcommand\thefootnote{}\footnote{#1}%
  \addtocounter{footnote}{-1}%
  \endgroup
}
\def\tr{\mathrm{Tr}}
\def\deltaop{\delta_{\mathsf{op}}}
\def\deltau{\delta_{\mathsf{u}}}
\def\bias{\mathsf{bias}}

\newcommand{\M}{\mathbf M}

\def\lv{\lVert}
\def\rv{\rVert}
\newcommand{\diag}{\mathrm{\normalfont{diag}}}
\newcommand{\bo}{\boldsymbol }

\def\snr{\mathsf{SNR}}
\newcommand{\snrfull}{{\mathsf{SNR}^{\mathsf{full}}}}
\newcommand{\rbayes}{\mc R^{\mathsf{Bayes}}}
\allowdisplaybreaks
\mathtoolsset{showonlyrefs}

\def\xicov{\xi_{\mathsf{cov}}}

\newcommand{\mb}{\mathbf }
\def\hat{\widehat}
\def\tilde{\widetilde}

\def\t{^{\top}}
\def\sigmamin{{\sigma_{\min}^*}}

\def\cov{\mathrm{Cov}}

\def\oracle{\mathsf{oracle}}

\def\spacingset#1{\renewcommand{\baselinestretch}%
{#1}\small\normalsize} 
\spacingset{1}

\title{Minimax-Optimal Spectral Clustering with Covariance Projection for High-Dimensional Anisotropic Mixtures}
\author{Chengzhu Huang \and Yuqi Gu}
\date{Department of Statistics, Columbia University
}

\begin{document}
\maketitle
\begin{abstract}
In mixture models, anisotropic noise within each cluster is widely present in real-world data. This work investigates both computationally efficient procedures and fundamental statistical limits for clustering in high-dimensional anisotropic mixtures. We propose a new clustering method, Covariance Projected Spectral Clustering (COPO), which adapts to a wide range of dependent noise structures. We first project the data onto a low-dimensional space via eigen-decomposition of a diagonal-deleted Gram matrix. Our central methodological idea is to sharpen clustering in this embedding space by a covariance-aware reassignment step, using quadratic distances induced by estimated projected covariances. Through a novel row-wise analysis of the subspace estimation step in weak-signal regimes, which is of independent interest, we establish tight performance guarantees and algorithmic upper bounds for COPO, covering both Gaussian noise with flexible covariance and general noise with local dependence. To characterize the fundamental difficulty of clustering high-dimensional anisotropic Gaussian mixtures, we further establish two distinct and complementary minimax lower bounds, each highlighting different covariance-driven barriers. Our results show that COPO attains minimax-optimal misclustering rates in Gaussian settings. Extensive simulation studies across diverse noise structures, along with a real data application, demonstrate the superior empirical performance of our method. 
\end{abstract}


\blfootnote{Emails: \texttt{ch3786@columbia.edu, yuqi.gu@columbia.edu}. Address for correspondence: Yuqi Gu, Department of Statistics, Columbia University, New York, NY 10027, United States of America.}

\section{Introduction}

Mixture models capture the foundational clustering structure widely present in many machine learning and statistical applications. In a mixture model, consider an $n\times p$ data matrix $\mb Y \coloneqq (\mb y_1, \cdots, \mb y_n)\t$ that collects $n$ independent samples $\mb y_1,\ldots,\mb y_n\in \bb R^p$. Each $\mb y_i$ is equipped with a latent label $z_i^*\in [K]$ and comes from a distribution $\mc D_{z_i^*}$ with expectation $\bo \theta_{z_i^*}^*$. In mixture models with additive noise, we can write
\begin{equation}\label{eq: additive noise formulation}
    \mb Y = \bb E[\mb Y]  + \mb E,\quad \mb Y^* := \bb E[\mb Y]=\mb Z^* {\mb \Theta^*}^\top,
\end{equation}
where $\mb E = (\mb E_1, \cdots, \mb E_n)\t\in\bb R^{n\times p}$ denotes the mean-zero noise matrix.
The $p\times K$ matrix $\mb \Theta^* = (\bo \theta_1^*,\cdots, \bo \theta_K^*)$ collects the $K$ cluster centers $\bo\theta_k^*\in\bb R^p$. In the $n\times K$ matrix $\mb Z^*$, the $i$th row is $\mb Z^*_{i,\cdot}= \mb e_{z_i^*}\t$, where $\mb e_k$ is the $k$th canonical basis of $\bb R^K$. 

We study the clustering problem under possible high-dimensionality ($p\gtrsim n$) and nonspherical (anisotropic) noise; i.e., $\text{Cov}(\mc D_k), k\in[K]$ are not identity matrices multiplied by scalars. For a true latent label vector $\mb z^*$ and an estimated latent label vector $\mb z$, the clustering performance of $\mb z$ is measured by the Hamming distance up to a label permutation: 
$
    h(\mb z, \mb z^*) \coloneqq \min_{\pi\in \Pi_K} \frac{1}{n}\sum_{i=1}^n \ind\{z_i \neq \pi(z_i^*)\},
$
where $\Pi_K$ is the set of all permutations of $[K]$.

Most theoretical results on high-dimensional clustering
have focused on isotropic or sub-Gaussian mixture models \citep{lu2016statistical,loffler2021optimality,zhang2024leave,ndaoud2018sharp}.
In practice, however, nonspherical noise structures are ubiquitous in real-world datasets.
When both high dimensionality and anisotropy are present in the mixture model, two natural yet challenging questions emerge:
\begin{itemize}[itemsep=1pt, topsep=2pt]
    \item \textit{Can we design a clustering algorithm to effectively capture distributional heterogeneity?}
    \item \textit{What is the information-theoretic limit for clustering under high-dimensional anisotropic noise, and is it attainable by an efficient algorithm?}
\end{itemize}
These questions echo unresolved open problems highlighted in \cite{chen2024optimal}, which studied anisotropic Gaussian mixtures with a fixed or slowly growing dimension.

In this work, we address the above problems by 
(i) proposing an efficient spectrum-based clustering method, \emph{Covariance Projected Spectral Clustering} (COPO); (ii) establishing a sharp upper bound for the clustering risk of COPO in terms of a key signal-to-noise ratio $\snr$; and (iii) deriving two statistical lower bounds that characterize the fundamental clustering hardness from distinct perspectives, both coinciding with degenerate forms of $\snr$.

\subsection{Covariance-Projected Spectral Clustering}\label{sec-contribution2}
Spectral methods stand out for their simplicity and strong theoretical foundations. A representative example is \cite{zhang2024leave}, which applies K-means to the rows of $\mathbf U\mathbf \Lambda$, where $\mathbf U$ contains the top-$K$ left singular vectors of $\mathbf Y$ and $\mathbf \Lambda$ the corresponding singular values.
These methods are motivated by a simple yet powerful insight: the left singular subspace of $\mb Y^*$ encodes the underlying cluster structure, and its empirical counterpart often provides a reliable estimate, up to an orthogonal rotation. A central theoretical challenge, therefore, is to understand how these empirical quantities fluctuate around their population counterparts. With such an understanding, one can incorporate noise structures into the clustering criterion to enhance clustering accuracy.

In this work, we propose a novel and easy-to-implement clustering method, \emph{Covariance Projected Spectral Clustering} (COPO, Algorithm~\ref{algorithm: CPSC}). Starting from an initialization, COPO proceeds in two stages.
First, it estimates the left singular subspace of $\mb Y$ via eigen-decomposition of a diagonal-deleted Gram matrix, producing a
$K$-dimensional embedding of the samples. Second, it iteratively refines the clustering in this embedding space by updating cluster assignments using a covariance-aware quadratic-distance reassignment rule based on estimated projected covariances.
To justify why these low-dimensional quantities can be reliably estimated and exploited for clustering,
we conduct a fine-grained subspace analysis. 
Define $\mc H(\cdot)$ as the \emph{diagonal deletion operator}, which sets the diagonal entries of a square matrix to zero and keeps all off-diagonal entries.
Let $(\mb U^*, \bo \Lambda^*, \mb V^*)$ be the top-$K$ SVD of $\mb Y^*$. Informally, given the top-$K$ eigen subspace $\mb U\in \bb R^{n \times r}$ of $\mc H(\mb Y \mb Y\t)$, 
each row of $\mb U$ admits the following approximation: 
\begin{align}
    & \mb U_{i,\cdot}(\mb U\t \mb U^*) - \mb U_{i,\cdot}^* \approx 
    \underbrace{\mb E_{i,\cdot} \mb V^* {\bo \Lambda^*}^{-1} + \mc H(\mb E\mb E\t)_{i,\cdot} \mb U^* {\bo \Lambda^*}^{-2}}_{
    \eqqcolon \mathfrak L_i\t}
    \label{eq: linear approximation of U}
\end{align}
This form of linear approximation has appeared in earlier works such as \cite{abbe2022}; however, our theory (Theorem~\ref{thm:subspace perturbation}) extends to the challenging \emph{weak-signal regime in a row-wise manner}, which departs from previous studies and is of independent interest. 
Focusing on this low-dimensional embedding of samples offers two main benefits: it ensures statistical consistency of estimating the low-dimensional quantities and also delivers superior computational efficiency compared to traditional {EM}-type methods {in the $p$-dimensional space}.

We give an illustrative example in Figure \ref{fig: decision boundaries}, applying COPO to data generated from a high-dimensional two-component Gaussian mixture model with $p=2000$ and $n=500$. When initialized by the spectral clustering algorithm studied in \cite{abbe2022}, our method accurately captures the projected cluster shapes in the subspace of the top $K=2$ right singular vectors of $\mb Y$ and effectively reduces the clustering errors as the iterations proceed. This is done by depicting and refining the elliptical (Figure \ref{fig: ellipse}) and hyperbolic (Figure \ref{fig: hyperbola}) decision boundaries in this space. In contrast, spectral clustering has the limitation that it uses {K-Means} in the 2-dimensional space and hence splits the point cloud by an unsatisfactory linear decision boundary.

\begin{figure}[ht]
    \centering
    \begin{subfigure}[b]{\textwidth}
        \centering
        \includegraphics[width=\textwidth]{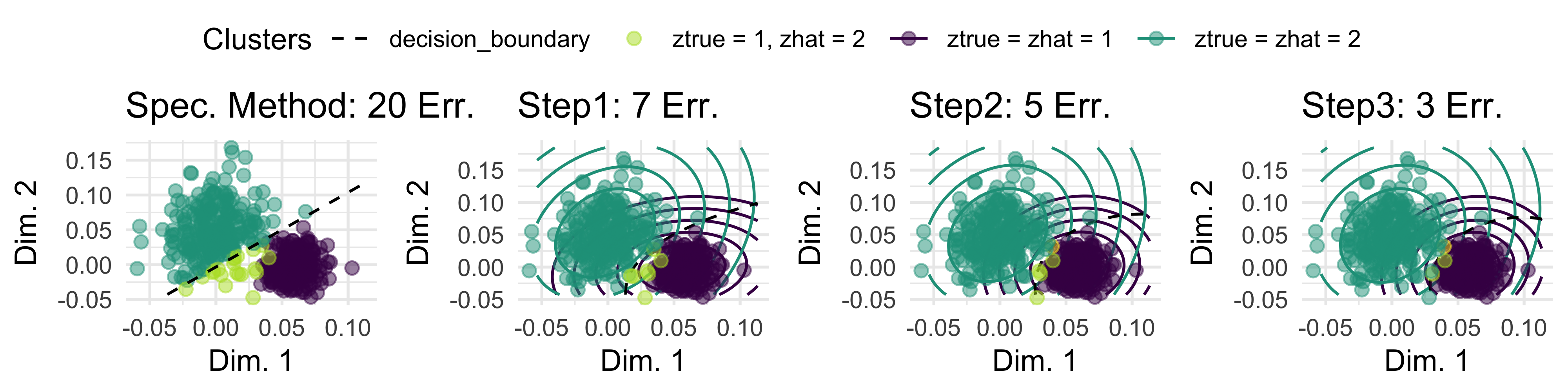}
        \caption{A case with elliptical decision boundaries of COPO}
        \label{fig: ellipse}
    \end{subfigure}
    \begin{subfigure}[b]{\textwidth}
        \centering
        \includegraphics[width=\textwidth]{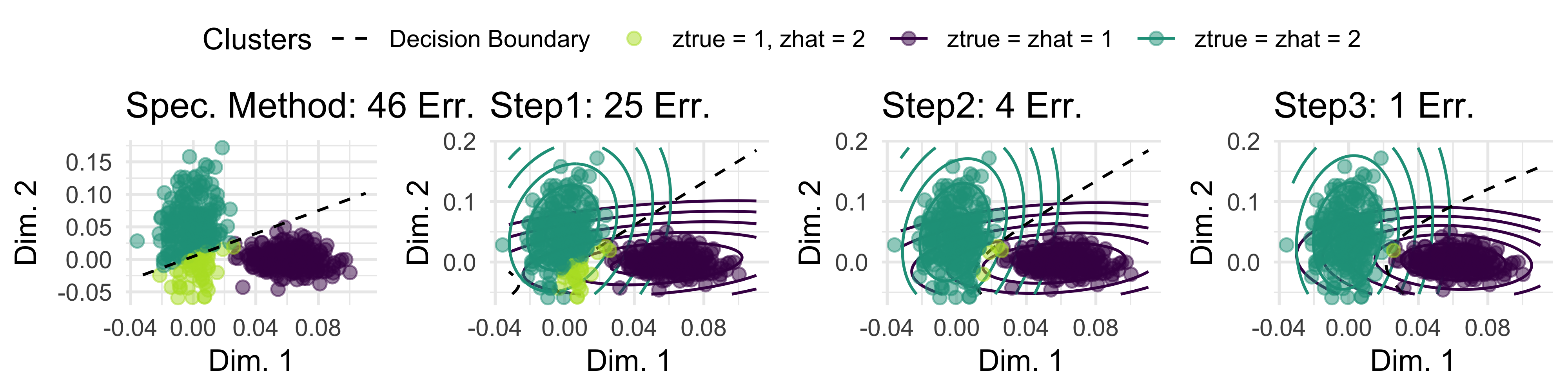}
        \caption{A case with hyperbolic decision boundaries of COPO}
        \label{fig: hyperbola}
    \end{subfigure}
    \caption{
    Comparing spectral clustering \cite{abbe2022}  and COPO in the top-2 right singular subspace of $\mb Y_{n\times p}$, with $n = 500$ and $p = 1000$. 
    From left to right: spectral clustering and the first three COPO iterations. Points are colored by label estimates (\texttt{zhat}) and ground truth (\texttt{ztrue}); ``Err.'' denotes the number of misclustered (light-green) points. Dashed lines indicate decision boundaries: linear for spectral clustering, and elliptical (Fig.~\ref{fig: ellipse}) or hyperbolic (Fig.~\ref{fig: hyperbola}) for COPO.}
    \label{fig: decision boundaries}
\end{figure}

\subsection{Theoretical Guarantees for COPO}\label{sec-contribution3} 
We establish sharp upper bounds on the misclustering error that characterize when and how the proposed COPO method succeeds in high-dimensional settings with anisotropic and dependent noise. Our analysis covers both anisotropic Gaussian mixtures and more general mixture models with locally dependent noise, and reveals a key signal-to-noise ratio (SNR) that governs clustering performance. These results demonstrate that COPO adaptively exploits covariance heterogeneity through dimension reduction and covariance-aware refinement, while remaining computationally efficient in high dimensions.

Informally, the clustering error of COPO is upper bounded by an exponential term in $\snr$, where the signal-noise-ratio $\snr$, as elaborated in Section~\ref{section: upper bound}, intrinsically represents the clustering hardness in anisotropic mixture models. 
\begin{theorem*}[Informal Upper Bound; formal version in Theorem~\ref{theorem: upper bound for algorithm (gaussian)}]
    In a wide range of noise environments, the 
    misclustering
    rate of COPO has the following upper bound given a proper initialization and a diverging $\mathsf{SNR}$:
    $
    \bb E[h(\hat{\mb z}, \mb z^*)] \leq \exp\Big(-(1 + o(1))\frac{\snr^2}{2}\Big). 
    $
\end{theorem*}

An implication of the upper bound is, in the partial recovery regime, our theoretical guarantee significantly improves upon the one for (hollowed) spectral clustering in the presence of imbalanced covariances, 
 as validated by simulation studies in Section \ref{sec:simulation}. 

Interestingly, as becomes evident later, both the procedures of the algorithm and the form of $\snr$ only rely on partial covariance information, rather than the full structure exploited by likelihood-based methods --- in other words, certain structural details are inevitably discarded when operating in low-dimensional embeddings. This naturally prompts the question: does neglecting the complete covariance structure compromise clustering accuracy? Surprisingly, our minimax lower bound analysis reveals that this omission does not entail a loss of optimality, a phenomenon not previously examined in the literature. 
\subsection{Fundamental Limits and Minimax Lower Bounds}
\label{subsec: lower bound}
Complementing the achievable guarantees in Section~\ref{sec-contribution3}, we investigate the fundamental limits of clustering in high-dimensional anisotropic Gaussian mixtures.
We derive minimax lower bounds for anisotropic Gaussian mixtures that quantify the intrinsic difficulty of clustering under heterogeneous covariance structures. The lower bounds expose two distinct sources of clustering hardness. Together with our upper bounds, these results establish the minimax optimality of COPO in the Gaussian setting.

To identify the sources of clustering hardness, we first gain intuition by examining the approximation \eqref{eq: linear approximation of U} under the \emph{homogeneous} anisotropic Gaussian mixtures ($\bo \Sigma_k = \bo \Sigma,~k\in[K]$), subject to the condition $\lim_{n\rightarrow \infty}\tr(\bo \Sigma^2) / p(n) = c >0$ for some constant $c$. The fluctuation $\mk L_i$ in \eqref{eq: linear approximation of U} arises from two uncorrelated components in the linear term. This perspective allows us to partition the high-dimensional regime ($p /n \rightarrow \infty$) into two distinct regimes where one term dominates the other. 
Moreover, the key signal-noise-ratio $\snr$ degenerates to $\snr^{\sf{mod}}$ or $\snr^{\sf{exc}}$ in these two regimes, respectively. This phase transition motivates the development of two distinct minimax lower bounds, each corresponding to one of these degenerate forms.

Compared to existing lower bounds in the literature \citep{lu2016statistical,zhang2023fundamental,chen2024optimal,ndaoud2018sharp}, our results offer two key insights: \emph{First}, in the moderately high-dimensional case, the minimax lower bound reveals the impossibility of achieving the Bayesian oracle risk in the presence of both heteroskedasticity and high-dimensionality. Instead, an \emph{informational dimension-reduction phenomenon} emerges with the ``constrained'' signal-to-noise-ratio $\snr^{\mod}$.
Our proof strategy diverges from the standard proof technique of reduction to a two-point testing problem commonly used in analyses of Gaussian mixtures \cite{lu2016statistical}, stochastic block models \cite{gao2017achieving,zhang2016minimax}. \emph{Second}, in the excessively high-dimensional case, our approach generalizes the results in \cite{ndaoud2018sharp,chen2021cutoff} by incorporating $\tr(\bo \Sigma^2)$ into the lower bound, 
highlighting its role as a measure of the intrinsic hardness of estimating cluster centers.

\paragraph{Notations.}
For any positive integer $n$, denote $[n]:=\{1,\ldots,n\}$. Denote the collection of $p$-by-$p$ orthogonal matrices by $O(p)=\{\mb U\in \bb R^{p\times p}, \mb U\t \mb U =\mb I_p\}$ and the collection of $p$-by-$r$ orthonormal matrices by $O(p,r) = \{\mb U\in \bb R^{p\times r}, \mb U\t \mb U =\mb I_{r}\}$ with $r< p$. Denote the group of invertible matrices in $\bb R^{n\times n}$ by $\mathrm{GL}_n(\bb R)$. 
For any matrix $\mb{M} \in \bb R^{n \times m}$, denote its $m$-th row (as a column vector) by  $\mb{M}_k$ or $\mb M_{k,\cdot}\t$ and its $k$-th column (as a column vector) by $\mb{M}_{\cdot,k}$. Define $\mc P_{-i,\cdot}(\cdot)$(resp. $\mc P_{\cdot,-i}(\cdot)$) as the operator that sets the $i$-th row (resp. column) of a matrix to zero while leaving all other entries unchanged.
Let $\norm{\M}$ and $\norm{\M}_F$ denote its spectral norm and Frobenius norm, respectively. 
Define the $\ell\ti$ metric as $\norm{\mb M}\ti \coloneqq \max_{i\in[n]}\norm{\mb M_i}_2$. Denote by $\sigma_k(\mb{M})$ the $k$-th largest singular value of $\mb{M}$ and by $\sigma_{\min}(\mb{M})$ the smallest nonzero one. 
If $\mb M$ is square, then define its effective rank $\mathrm{rk}(\mb M)$ as $\mathrm{rk}(\mb M) \coloneqq \frac{\tr(\mb M)}{\norm{\mb M}}$. 
Denote by $\mathbb{P}_{\bo{\theta}, \mathbf{\Sigma}}$ the probability measure and by $\phi_{\bo{\theta}, \mathbf{\Sigma}}$ the probability density function of a Gaussian distribution with mean $\boldsymbol{\theta}$ and covariance matrix $\mathbf{\Sigma}$, respectively. 
For any real valued functions $f(n)$ and $g(n)$, write $f(n) \lesssim g(n)$ if $\left|f(n)\right| \leq C \left|g(n)\right|$ for some constant $C$. Similarly, write $f(n) \gtrsim g(n)$ if $\left|f(n)\right| \geq  C' \left|g(n)\right|$ for some constant $C'$. Write $f(n) \asymp g(n)$ if $f(n) \lesssim g(n) \lesssim f(n)$. 
Write $f(n) \ll g(n)$ when there exists some sufficiently small constant $c$ such that $\left|f(n)\right| \leq c\left|g(n)\right|$ for any sufficiently large $n$. Write $f(n) = o(1)g(n) = o\big(g(n)\big)$ if ${\left|f(n)\right|}/{\left|g(n)\right|} \rightarrow 0$, and write $f(n) = \omega\big(1\big) g(n)= \omega\big( g(n)\big)$ if ${\left|f(n)\right|}/{\left|g(n)\right|} \rightarrow \infty $, as $n$ goes to infinity. For brevity, we may refer to the functions $\snr(\cdot)$, $\mathsf{SNR}^{\mathsf{mod}}(\cdot)$, and $\snr^{\mathsf{exc}}$ simply as $\snr$, $\mathsf{SNR}^{\mathsf{mod}}$, and $\snr^{\mathsf{exc}}$ when applied to a tuple of parameters, if the context makes it clear.

\paragraph{Organization.} 
Section~\ref{sec: alg} develops a new singular subspace perturbation theory and introduces our COPO method. Section~\ref{section: upper bound} provides comprehensive theoretical guarantees for our clustering method universally for both Gaussian mixtures and general mixture models. 
Section~\ref{section: lower bounds} establishes the minimax lower bounds for general high-dimensional anisotropic Gaussian mixture models. 
Simulation studies and real data analysis in Sections~\ref{sec:simulation} and \ref{sec: real data analysis}  validate our theoretical findings and demonstrate our method's superior performance.
Section~\ref{sec: conclusion and discussion} concludes. Proofs of the theoretical results are included in the Supplementary Material.

\section{New Subspace Perturbation Theory and COPO}
\label{sec: alg}
Our clustering algorithm COPO consists of two main components: the estimation of left singular subspace $\mb U^*$ of the data matrix, and an iterative refinement based on  operating upon the estimated subspace. In the refinement step, we incorporate the estimation of the covariance structure $\{\mb S_k^*\}$ (defined after \eqref{eq: SNR definition} in Section~\ref{section: upper bound}) into the criterion, which is roughly a ``projection'' of the original covariance matrix. 
The theoretical foundation for this procedure is given by our new subspace perturbation theory and a delicate characterization of clustering errors throughout the iteration process.

\subsection{New Analysis of Subspace Estimation via Diagonal Deletion} 
\label{subsec: subspace perturbation}
For mixture models in \eqref{eq: additive noise formulation}, it is well known that the clustering information is encoded in the left singular subspace of the expectation of the data matrix $\mathbb E[\mathbf Y]$. Many clustering methods have been proposed based on the empirical left singular subspace estimates \citep{zhang2019spectral,zhang2024leave,han2022exact}. 
However, in high dimensions, using the ``vanilla SVD'' of the data matrix $\mb Y$ to estimate $\mb Y^*$'s left singular subspace (equivalent to performing a eigenvalue decomposition of the Gram matrix $\mb Y\mb Y^\top$) has unsatisfactory performance, because $\bb E[\mb Y\mb Y^\top] = \mb Y^*\mb Y^{*\top} + \mathrm{diag}(\bb E[\norm{\mb E_i}_2^2])_{i\in[n]} \neq \mb Y^*\mb Y^{*\top}$. In particular, the quantities $\bb E[\|\mb E_i\|_2^2]$ may vary substantially across $i=1,\ldots,n$ (heteroskedasticity) \citep{zhang2022heteroskedastic}. This often happens, especially in the context of anisotropic mixture models.

To remedy the above bias issue along the diagonal entries, we resort to the eigen-decomposition of the Gram matrix after diagonal deletion; that is, we perform the top-$K$ eigen-decomposition on $\mc H(\mb Y\mb Y\t )$, where the operator $\mc H$ zeroes out all diagonal entries of $\mb Y \mb Y\t$. This procedure provides us with a robust left singular subspace estimation against possible heteroskedasticity without introducing much bias. We develop the following new approximation theory, which delivers a sharp row-wise characterization (not uniform type) for $\mb U_{i,\cdot}\mb U\t \mb U^* - \mb U^*_{i,\cdot}$, which differs from previous results as elaborated right after the theorem.

\begin{theorem}[New Singular Subspace Perturbation Theory]\label{thm:subspace perturbation}
    Consider a matrix $\mb M = \mb M^* + \mb E\in \bb R^{n \times p}$ with $\mathrm{rank}(\mb M^*) = r$, $\bb E[\mb E] = \mb 0$, and the top-$r$ SVD of $M^*$ being $(\mb U^*, \bo \Lambda^*, \mb V^*)$ where $\bo \Lambda^* = \diag(\sigma_1^*, \cdots, \sigma_r^*)$. Define $d \coloneqq n \vee p$, $\sigma^2 \coloneqq \max_{i\in[n]} \tr(\mathrm{Cov}(\mb E_i))/p$, $\tilde\sigma^2 \coloneqq \max_{i\in[n]}\norm{\mathrm{Cov}(\mb E_i)}$, and $\bar \sigma^2 \coloneqq \max_{i\in[n]} \bignorm{{\mb V^*}\t\mathrm{Cov}(\mb E_i) \mb V^*}$. 
    Assume the noise matrix $\mb E$ has independent rows, which satisfies \textbf{either} Assumption~\ref{assumption: gaussian noise} (Gaussian noise)
    \textbf{or} Assumption~\ref{assumption: bounded noise} (general noise). 
    Define $ \deltaop \coloneqq C_{W} \big[ \tilde \sigma^2 n + \sigma \tilde \sigma\sqrt{np}   +  \bar \sigma  \sqrt{n}  \sigma_1^* + \frac{\mu_1 r }{n} {\sigma_1^*}^2 \big] $ for $C_W >0$ and assume $\deltaop \ll (\sigma_r^*)^2$ and $r \log d \mu_1 r^2 \lesssim n$.  Then for the top-$r$ left eigenvector matrix $\mb U$ of $\mc H (\mb M \mb M\t)$, $t_0 \gg \sqrt{\log \log n}$, and any fixed $c_p>0$, with probability \underline{$1 - O(e^{-t_0^2 /2} \vee d^{-c_p})$},
    \eq{
         \Bignorm{\mb U_{i,\cdot} (\mb U\t \mb U^*) - \mb U^*_{i,\cdot}- \mk L_i\t - \bias_{z_i^*}\t }_2  
        \lesssim   \sqrt{\frac{\mu_1 r}{n}}\cdot \Big[ \frac{\deltaop^2 t_0  }{(\sigma_r^*)^4} + \frac{\mu_1^{\frac12} r \sigma_1^* \bar \sigma \sqrt{\log d / n} }{(\sigma_r^*)^2} \Big].
        \label{eq: final bound in singular subspace perturbation theory}
    } 
    Moreover, $\max_{k \in[K]}\norm{\bias_k}_2 \lesssim \sqrt{\frac{\mu_1 r}{n}}\cdot  \frac{r \deltaop \sqrt{\log d / n} + \mu_1 r (\sigma_1^*)^2 / n}{(\sigma_r^*)^2}$ with probability at least $1 - O(d^{-c_p})$. 
    \label{thm: singular subspace perturbation theory}
\end{theorem}

To facilitate comparison between our new theory and existing results, we focus on the regime where $\deltaop \ll (\sigma_r^*)^2 \ll \sqrt{n} \deltaop / r$ and $\mu_1 = O(1)$, under which the bound in \eqref{eq: final bound in singular subspace perturbation theory} simplifies to $\sqrt{\frac{\mu_1 r}{n}}\cdot \frac{\deltaop^2 t_0  }{(\sigma_r^*)^4} $. 
Existing $\ell_{2,\infty}$-type result (e.g., \cite{cai2021subspace}) requires the magnitude of the approximation terms to at least exceed $\sqrt{\frac{\mu_1 r}{n}} \frac{\deltaop^2}{(\sigma_r^*)^4}\cdot \mathrm{poly}(\log d)$ with probability $1 - O(\mathrm{poly}(d^{-1}))$ in order to dominate. In contrast, 
our result reveals that, for any fixed row and whenever $e^{-\frac{t_0^2}{2}} = \omega(\mathrm{poly}(d^{-1}))$, the approximation terms dominate as long as their magnitude exceeds $\sqrt{\frac{\mu_1 r}{n}} \frac{\deltaop^2 t_0}{(\sigma_r^*)^4}$ with probability $1 - O(e^{-\frac{t_0^2}{2}})$. The improvements in our result over existing results are twofold: (i) the logarithmic factor in the approximation error is replaced by $t_0 $, which may be smaller; and (ii) the allowed exceptional probability can be much larger than $\mathrm{poly}(d^{-1})$. 
We now explain why such refinement is necessary in mixture models. Under isotropic Gaussian noise, existing $\ell\ti$ analysis of the singular subspace perturbation \cite{fan2018eigenvector,chen2021spectral,cai2021subspace,agterberg2022entrywise} characterizes the maximum row-wise  $\ell_2$ fluctuation error, by approximating the subspace fluctuation with some linear forms. However, the resulting bound fails when $\deltaop \ll (\sigma_r^*)^2 \ll \deltaop\log d$, because a uniform guarantee over all rows requires each per-row approximation failure probability $p_i$ to be exceedingly small ($\mathrm{poly}(d^{-1})$). 
Such a stringent requirement is unattainable in the regime of interest. In fact, our algorithm only requires the approximation to hold for most rows, not all of them. This allows much milder conditions on the $p_i$'s and enables us to handle this weak-signal regime. Similar ideas have appeared in \cite{abbe2022}. With our refined row-wise approximation, the per-row failure probability can be as large as $p_i = e^{-\frac{t_0^2}{2}} = \omega( \mathrm{poly}(d^{-1}))$, while the approximation error becomes smaller under weaker conditions. While the proof of Theorem~\ref{thm: singular subspace perturbation theory} starts from the expansion formula in \cite{xia2021normal}, its main contribution is a delicate induction argument that substantially extends previous approaches and carefully decouples dependencies throughout the analysis. 
To build intuition for the linear form $\mk L_i$ defined in \eqref{eq: linear approximation of U}, we note that under some regularity conditions, the first term $\mb E_i\mb V^* {\bo \Lambda^*}^{-1}$ and the second terms $\mb E_i\mc P_{-i,\cdot}(\mb E)\t \mb U^*{\bo \Lambda^*}^{-2}$
are both \emph{asymptotically normal} and they are \emph{mutually uncorrelated}.   
This implies that the rotated low-dimensional embeddings $(\mb U_{i,\cdot})_{i\in[n]}$ can be approximated as Gaussian vectors with structured covariances, even when the original $\mb Y_i$ are non-Gaussian. This insight motivates a Gaussian pseudo-likelihood refinement in the $K$-dimensional space beyond K-means clustering.

\subsection{Covariance Projected Spectral Clustering}
\label{subsection: algorithm}

The insight described in the last subsection suggests interpreting  each row of $\mb U(\mb U\t \mb U^*)$ as a $K$-dimensional Gaussian vector. This motivates us to propose an iterative refinement after the singular subspace estimation (Algorithm~\ref{algorithm: CPSC}), which we call the \textbf{Co}variance \textbf{P}r\textbf{o}jected Spectral Clustering (COPO).  The term “covariance projection” reflects the fact that the population counterpart of $\hat{\bo \Omega}_k^{(s)}$ represents a partial projection of the original covariance information.
Compared with previous methods, COPO is computationally efficient for high-dimensional data and adaptive to nonspherical and dependent noise.

\begin{algorithm}[ht]
    \caption{(Iterative) \textbf{Co}variance \textbf{P}r\textbf{o}jected Spectral Clustering (COPO)}
    \label{algorithm: CPSC}
    \KwIn{Data matrix $\mb{Y} =(\mb{y}_1,\ldots,\mb{y}_n)^\top \in \mathbb{R}^{n\times p}$, number of clusters $K$, an initial cluster estimate $\hat{\mb z}^{(0)}$}
    \KwOut{Cluster assignment vector $\hat{\mb{z}}^{(t)}\in [K]^n$}

    Perform top-$K$ eigen decomposition of $\mc H(\mb Y \mb Y\t)$ and obtain its top-$K$ eigen subspace $\mb U \in O(n,K)$.
    
       \For{$s =1,\cdots, t$}{ 
         For each $k\in [K]$,
    estimate the embedded centers 
        by $
        \hat{\mb c}_k^{(s)} = \frac{\sum_{i\in [n]} \mb U_{i} \ind\{\hat{\mb z}^{(s-1)} = k \}}{ \sum_{i\in [n]}  \ind\{\hat{\mb z}^{(s-1)} = k \}}$,
        and estimate the covariance matrices of embedded vectors by
        $
        \hat{\bo \Omega}^{(s)}_k =
        \frac{\sum_{i\in [n]} \mb U_{i} \mb U_{i}\t \ind\{\hat{\mb z}^{(s-1)} = k \}}{ \sum_{i\in [n]}  \ind\{\hat{\mb z}^{(s-1)} = k \}} - \hat{\mb c}_k^{(s)} {{}\hat{\mb c}_k^{(s)}}\t$.
        
        Update the cluster memberships via
        $\hat{z}_i^{(s)}
        = \argmin\limits_{k\in [K]} \big(\mb U_{i} - {{}\hat{\mb c}_{k}^{(s)}} \big)\t \big(\hat{\bo \Omega}_k^{(s)}\big)^{-1} \big(\mb U_{i} - \hat{\mb c}_{k}^{(s)}  \big).
        $
       }
\end{algorithm}
In words, each iteration of Algorithm~\ref{algorithm: CPSC} proceeds in two steps: it sketches the centers $\{\hat{\mb c}_k^{(s)}\}_{k\in[K]}$ of the embedded vectors in $\mathbb R^K$ and the associated covariance matrices $\{\hat{\mb \Omega}_k^{(s)}\}_{k \in[K]}$ in $\mathbb R^{K\times K}$ based on the previous clustering assignment. Next, it reassigns each data point to the cluster that minimizes the Mahalanobis distance, computed using the current estimates $\hat{\bo \Omega}_k^{(s)}$. This modified criterion essentially corresponds to the negative log-likelihood function for $K$-dimensional Gaussian distributions, with the logarithmic determinant terms omitted --- an omission justified by the fact that, in the regimes of interest, these terms have negligible impact on clustering. Conceptually, COPO can be regarded as applying a modified hard-EM algorithm for Gaussian mixtures to the rows of $\mb U$ in $\bb R^K$, a formulation that appears straightforward and intuitive. However, it essentially originates from and is theoretically enabled by the understanding of the linear approximation of $\mb U_{i,\cdot}$ in Section~\ref{subsec: subspace perturbation}. 

Before continuing, we selectively remark on related methods and their limitations in high-dimensional anisotropic mixtures, mostly concerned about the Gaussian mixtures. Regarding the spectral methods, the algorithms studied in \cite{loffler2021optimality,abbe2022,zhang2024leave} apply K-Means to rows of an estimate $\mb U$ of the top-$K$ left singular subspace of $\mb Y^*$. While they are minimax optimal under isotropic Gaussian mixtures \cite{zhang2024leave,abbe2022}, they unsurprisingly fail to adapt to anisotropic noise, as seen in Lemma~C.6 and the remark therein: our upper bound exponent $-\frac{\snr^2}{2}$ is generally smaller than $-\frac{\min_{a\neq b\in[K]} \norm{\bo \theta_a^* - \bo \theta_b^*}_2^2}{8 \max_{k\in[K]} \norm{\bo \Sigma_k}}$ in \cite{zhang2024leave}. 
On the other hand, 
     classical EM and hard-EM (adjusted Lloyd’s algorithm \cite{chen2024optimal}) iteratively estimate centers and full covariance matrices but degrade in high dimensions due to the difficulty of estimating $p\times p$ covariances. Our Algorithm~\ref{algorithm: CPSC} can be viewed as a hard-EM variant on a $K$-dimensional model that avoids full covariance estimation, making COPO feasible in high dimensions. Relatedly, \cite{cai2019chime} extends EM under homogeneous covariances via sparsity assumptions on the discriminant vector $\bo\Sigma^{-1}(\bo \theta_2^* - \bo \theta_1^*)$.

 \noindent

In summary, existing clustering methods are either confined in low-dimensional regimes, or require some specific sparsity assumptions, or deal with the spherical noise case. In contrast, our COPO method handles high-dimensional data and adapts to heterogeneous covariances without sparsity requirements. 
Besides, our algorithm presents robustness against possible ill-conditioned {and even non-invertible $p\times p$} covariance matrices $\mb \Sigma_k$.

\subsection{Related Work}
Spectral methods, rooted in early works \cite{fiedler1973algebraic,hall1970r}, are a central tool for revealing low-rank structures in statistical models. By leveraging leading eigenvectors or singular vectors, they have been widely used in clustering and network analysis \citep{zhang2024leave,loffler2021optimality,srivastava2023robust,kumar2010clustering,vempala2004spectral,kannan2009spectral,davis2021clustering,tang2024factor,rohe2011spectral,qin2013regularized,lei2020consistent,jin2018score+}, supported by recent advances in random matrix perturbation theory \citep{abbe2020entrywise,agterberg2022entrywise,yan2024inference,fan2018eigenvector,lei2019unified,cape2019two,cai2021subspace,abbe2022}.

Besides, multiple research directions have established theoretical guarantees for clustering. These include the method of moments \citep{hsu2013learning,anandkumar2012method,ge2015learning} and likelihood-based approaches such as the EM algorithm and its variants \citep{xu1996convergence,dasgupta2007probabilistic,balakrishnan2017statistical}, with Lloyd’s algorithm as a well-studied special case of hard EM \citep{lu2016statistical,chen2024optimal,giraud2019partial,gao2022iterative}.

To understand how the unknown covariance matrices affect clustering, the works \cite{belkin2009learning,ge2015learning,chen2024optimal,wang2020efficient} focused on learning this heterogeneity. However, the statistical limits and methods for clustering in the high-dimensional regime where $p  \gg  n$, remain largely unexplored. To the best of our knowledge, the closest attempt to our discussion in this direction is \cite{davis2021clustering}, which established a statistical guarantee for an integer program with $p \log n \ll n$.

Finally, we connect our work to classification, the supervised analogue of clustering. Linear and quadratic discriminant classifiers have been extensively studied for their simplicity and interpretability \cite{cai2011direct,cai2019chime,cai2021convex,mai2012direct,bing2023optimal}. Our method can be viewed as an iterative low-dimensional quadratic discriminant procedure with unknown labels.

\section{Theoretical Guarantees for COPO}
\label{section: upper bound}
    Following the discussion in Section~\ref{sec: alg}, our theory accommodates two distinct noise settings: Gaussian noise and general noise with local dependence. 
In both settings, the misclustering error is governed by a central signal-to-noise ratio parameter, denoted by $\snr$, defined through the geometry of the embedded vectors. Specifically,
\begin{align}\label{eq: SNR definition}
&\mathsf{SNR}(\mb z^*, \{\bo \theta_k^*\}_{k\in[K]}, \{\mb \Sigma_k\}_{k\in[K]})^2
\coloneqq \min_{j_1\neq j_2\in [K]}\min_{\mb x \in \bb R^K}\big\{(\mb x - \mb w_{j_1}^*)\t {\mb S^*_{j_1}}^{-1}(\mb x- \mb w_{j_1}^*): \\ 
&\qquad  \qquad\qquad (\mb x - \mb w_{j_1}^*)\t {\mb S^*_{j_1}}^{-1}(\mb x- \mb w_{j_1}^*) = (\mb x - \mb w_{j_2}^*)\t {\mb S^*_{j_2}}^{-1}(\mb x- \mb w_{j_2}^*)\big\},
\end{align}
where $\mb w_k^*\coloneqq  {\mb V^*}\t \bo \theta_k^* = \bo \Lambda^* \mb U_i^*$ and $\mb S_k^* \coloneqq {\mb V^*}\t \bo \Sigma_k \mb V^* +  {\bo \Lambda^*}^{-1} {\mb U^*}\t \bb E\big[ \mc H(\mb E\mb E\t)_{\cdot,i}\mc H(\mb E\mb E\t)_{i,\cdot} \big] \mb U^* \bo {\Lambda^*}^{-1}$ with any $i$ with $z_i^* = k$. Throughout this paper, for convenience, we use “embedded vectors” to refer interchangeably to either the rows of $\mb U$ or the rows of $\mb U$ rescaled by singular value estimates. 
Under the approximation $\bo \Lambda^* {\mb U_{i,\cdot}^*}\t +\bo \Lambda^* \mc L_i + \bo \Lambda^* \mathsf{bias}_{z_i^*} \approx \bo \Lambda^* ({\mb U^*}\t \mb U )\mb U_{i,\cdot}\t $ for the embeddings, $\mb w_k^*$'s represent the population counterpart (center) of the rescaled rows of $\mb U$, and two summands in $\mb S_k^*$ are the covariance matrices of the two terms in $\bo \Lambda^* \mc L_i ={ \mb V^*}\t \mb E_{i,\cdot}\t + {\bo \Lambda^*}^{-1}{\mb U^*}\t \mc H(\mb E \mb E\t)_{\cdot,i}$, respectively. Accordingly, $\mb S_k^*$ can be interpreted as the limiting covariance matrix of the rescaled embedded vectors within cluster $k$. 

For reference, we summarize the notation used throughout the paper in the table below.
\begin{table*}[htbp]
\centering
\label{tab:notation-summary}
\setlength{\tabcolsep}{4pt}
\renewcommand{\arraystretch}{1.05}
\begin{tabular}{ll ll}
\toprule
Notation & Definition/Interpretation & Notation & Definition/Interpretation \\
\midrule
$\hat{\mb c}_k(\mb z)$ 
& Center estimate of $\mb U_i$ (Alg.~\ref{algorithm: CPSC})
& $\mb w_k^*$   
& Center of rescaled $\mb U_i$ \\

$\hat{\bo \Omega}_k(\mb z)$  
& Covariance estimate of $\mb U_i$ (Alg.~\ref{algorithm: CPSC})
& $\mb S_k^*$ 
& Limiting covariance of rescaled $\mb U_i$ \\
$\mc I_k(\mb z)$ & $\{i: z_i = k\} $& $n_k$ & $\big|\mc I_k(\mb z^*) \big|$ \\ 
$\beta$
& $\big(\tfrac{K\max_k n_k}{n}\big)\vee
   \big(\tfrac{n}{K\min_k n_k}\big)$
& $d$
& $n \vee p$
 \\
 $\tilde{\sigma}$ 
& $\max_{k\in[K]} \|\mathrm{Cov}(\mb E_k)\|^{1/2}$ & $\sigma$ 
& $\sqrt{\max_{i}\tr(\mathrm{Cov}(\mb E_i))/p}$\\
 $\bar{\sigma}$ 
& $\max_i 
   \|{\mb V^*}^{\!\top}\mathrm{Cov}(\mb E_i)\mb V^*\|^{1/2}$  & $\deltaop$ & defined in Theorem~\ref{thm: singular subspace perturbation theory} with $r = K$ \\ 
$\kappa$
& $\sigma_1^*/\sigma_K^*$ & $\kappa_{\mathsf{cov}}$
& $\deltaop / ( \sigma_K^*
   \sqrt{n\min_k \sigma_K(\mb S_k^*)})$ \\ 
$\omega_{a,b}$ & $ (\mb w_a^*-\mb w_b^*)^\top(\mb S_a^*)^{-1}(\mb w_a^*-\mb w_b^*)$ & $\nu$ 
& $\max_{a\neq b} \omega_{a,b}^{\frac12}/\min_{a\neq b} \omega_{a,b}^{\frac12}$ \\
$h(\mb z,\mb z^*)$
& $\tfrac{1}{n}\min_{\pi\in\Pi_K}
\sum_i \ind\{z_i\neq\pi(z_i^*)\}$
& $l(\mb z,\mb z^*)$
& $\sum_{i\in[n]} \omega_{z_i, \pi^*(z_i^*)} \ind\{z_i\neq\pi^*(z_i^*)\}$ \\
\bottomrule
\end{tabular}

\emph{Note.}
$\pi^*\coloneqq\argmin_{\pi\in\Pi_K}\sum_{i\in[n]}\ind\{z_i\neq\pi(z_i^*)\}$.
\end{table*}

We impose the following assumption on the signal strength and regularity for both the Gaussian mixture and general mixture settings. 
        \begin{assumption}\label{assumption: algorithm}
Assume that the cluster centers $\{\bo \theta_k^*\}_{k\in[K]}$ are linearly independent and the projected covariance matrices $\mb S_1^*,\ldots,\mb S_K^*$ have rank $K$, $\mathsf{SNR} = \omega( \kappa^2 \kappa_{\mathsf{cov}}^{8} K^3 \beta^3  \nu^2  \sqrt{\log \log d})$, $\sigma_r^* \gg \tilde \sigma^2 n + \sigma \tilde \sigma\sqrt{np}   +  \bar \sigma  \sqrt{n}  \sigma_1^*$, and $\min\{\frac{n}{(\log d)^{10}},  \frac{p^{\frac16} \sigma^{\frac13}}{\tilde\sigma^{\frac13} (\log d)^2} \} =  \omega( \kappa^4 \kappa_{\mathsf{cov}}^4 \beta^4 K^5)$.  Additionally, $\min_{k \in [K]} \mathrm{rk}(\bo \Sigma_k) \gtrsim (\log(p\vee n))^3 $ and $\min_{k_1, k_2 \in [K]} \mathrm{rk}(\bo \Sigma_{k_1} \bo \Sigma_{k_2}) \gtrsim (\log(p\vee n))^3 $. 
        \end{assumption}

\subsection{Guarantee for Anisotropic Gaussian Mixtures}
The assumption for Gaussian case is as follows: 
\begin{assumption}[Gaussian Noise]\label{assumption: gaussian noise}
    We assume that $\mb E_{i}\in \bb R^p$, $i\in[n]$ independently follow a multivariate Normal distribution with mean zero and covariance matrix $\mb \Sigma_{z_i^*}$.  
\end{assumption}

\begin{theorem}
        \label{theorem: upper bound for algorithm (gaussian)}
        Suppose that Assumption~\ref{assumption: gaussian noise} holds and the clustering error of the initialization satisfies $l(\hat{\mb z}^{(0)}, \mb z^*)\leq c_1 \frac{n}{\beta K	(\log d)^4}$ with probability at least $1- O(n^{-2}) \vee \exp\big(-(1+o(1)\frac{\snr^2}{2}\big)$ for some sufficiently small constant $c_1$. 
        Then given arbitrary positive $c_2$ and $\epsilon < 1$, the following holds for all $t \geq c_3\log n$ for some constant $c_3$:
            \begin{enumerate}[itemsep=1pt, topsep=2pt]
        \item If $\mathsf{SNR} < \sqrt{(2 + \epsilon)\log n}$, then
        \eq{\bb E[h(\hat{\mb z}^{(t)},\mb z^*)] \leq \exp\Big(- (1+ o(1)) \frac{\snr^2}{2}\Big).
        \label{eq: upper bound for algorithm} 
        }
        \item If $\mathsf{SNR}  \geq \sqrt{(2 + \epsilon)\log n }$ and $\nu = o(d)$, then $h(\hat{\mb z}^{(t)},\mb z^*) = 0$ with probability $1 - o(1)$. 
    \end{enumerate}
        \end{theorem}

\begin{remark}
    Our work diverges from another line of studies \cite{azizyan2013minimax,azizyan2015efficient,cai2019chime,cai2021convex} in Gaussian mixture models, which focused on the constant separation regime, i.e., $\snr \asymp 1$, under certain sparsity assumptions. In contrast, we address the growing separation regime with $\snr \rightarrow \infty$, and our approach to the upper bounds described later drastically differs from those works. Specifically, we focus directly on characterizing the actual clustering risk, which we show to exhibit an exponential decay in $\snr$, whereas previous studies centered on the gap between the actual risk and the Bayesian oracle risk—quantities typically of polynomial order, arising from errors in estimating the population parameters. 
    Interestingly, as a side note, our lower bounds in Section~\ref{subsection: minimax lower bound for anisotropic gaussian mixtures} reveal that, from a information-theoretic perspective, a degenerate form of $\snr$ is intrinsically tied to a gap between the actual risk and the Bayesian oracle risk, sharing similarities with the lower bounds in those works. 
\end{remark}

In a nutshell, we establish exponential bounds $\exp\big(-(1 + o(1)) \frac{\snr^2}{2}\big)$ for the clustering risk in the weak-signal regime and establish exact recovery guarantees for the strong-signal regime. To grasp the geometric meaning of $\snr$, recall its definition: 
\begin{align}
&\mathsf{SNR}(\mb z^*, \{\bo \theta_k^*\}_{k\in[K]}, \{\mb \Sigma_k\}_{k\in[K]})^2
\coloneqq \min_{j_1\neq j_2\in [K]}\min_{\mb x \in \bb R^K}\big\{(\mb x - \mb w_{j_1}^*)\t {\mb S^*_{j_1}}^{-1}(\mb x- \mb w_{j_1}^*): \\ 
&\qquad  \qquad\qquad (\mb x - \mb w_{j_1}^*)\t {\mb S^*_{j_1}}^{-1}(\mb x- \mb w_{j_1}^*) = (\mb x - \mb w_{j_2}^*)\t {\mb S^*_{j_2}}^{-1}(\mb x- \mb w_{j_2}^*)\big\}.
\end{align}
Note that $\mb S_k^*$ coincides with the covariance of $\bo \Lambda^* \mk L_i$, and ${\mb w_k^*}\t =  \mb U_{i,\cdot}^* \bo \Lambda^*$ for $i \in \mc I_k(\mb z^*)$. To illustrate, consider the case $K = 2$, if we apply a log-likelihood classifier (with the log terms ignored) to a two-component Gaussian mixture model with centers $\{\mb w_k^*\}_{k\in[2]}$ and covariance matrices $\{ \mb S_k^*\}_{k\in[2]}$, i.e., $\hat{z}(\mb y) = \argmin_{k\in[2]} (\mb y - \mb w_k^*)\t {\mb S_k^*}^{-1} (\mb y - \mb w_k^*)$, the decision boundary coincides exactly with the equality condition in the definition of $\snr$. Geometrically, $\snr$ represents the minimum weighted distance from the centers to this boundary. Moreover, under the condition $\snr \rightarrow \infty$, one has 
$$\bb P_{\mb y \sim \mc N(\mb w_1^*, \mb S_1^*)}\Big[\hat{z}(\mb y) \neq 1\Big] \leq \bb P_{\bo \epsilon \sim \mc N(\mb 0, \mb S_1^*)}\Big[\bignorm{{\mb S_1^*}^{-\frac12} \bo \epsilon}_2 \geq \snr\Big]= \exp\big(-(1+o(1))\frac{\snr^2}{2}\big),
$$
and the same holds for $\bb P_{\mb y \sim \mc N(\mb w_2^*, \bo \Sigma_2)}\Big[\hat{z}(\mb y) \neq 2\Big]$. In other words, any algorithm that effectively approximates the (pseudo) likelihood ratio classifier achieves the same exponential misclustering rate. This means our algorithm is optimal among all the spectral-decomposition-based algorithms \cite{abbe2022,zhang2024leave}. This heuristic is rigorously justified by our perturbation theory developed in Section~\ref{subsec: subspace perturbation}, together with a refined iterative characterization in our proofs, with mild requirements on $\snr$ and $l(\hat{\mb z}^{(0)}, \mb z^*)$.

\subsection{Guarantee for General Mixtures with Local Dependence}

Thanks to the universality result, we can accommodate arbitrary forms of dependence within each block and establish theoretical guarantees for COPO in a wide range of general mixture models. Naturally, we need to take into account both the degree of incoherence and the extent of local dependence for regularity. 
\begin{assumption}[General Noise with Local Dependence] We assume the following (unknown) block structures and regularity conditions on the noise matrix $\mathbf E$: 
    \label{assumption: bounded noise}
        \begin{enumerate}[itemsep=1pt, topsep=2pt]
            \item  There exists a partition $\{S_b\}_{b \in[l]}$ of $[p]$ such that $|S_b| \leq m$ for every $b\in[l]$ and $\{\mb E_{i,S_{b}}\}_{b\in[l]}$ are mutually independent for each $i\in[n]$. \label{item: bounded noise assumption 1}
            \item Either $|E_{i,j}| \leq B$ for all $i\in[n]$, $ j\in[p]$, or there exists a random matrix $\mb E'=(E_{i,j}') \in\mathbb{R}^{n\times p}$ obeying the same dependence structure in Assumption~\ref{assumption: bounded noise}.\ref{item: bounded noise assumption 1}, such that for any $i\in[n], j\in[p]$, it holds that $\norm{E'_{i,j}}_\infty\leq B$, $\mathbb E[E'_{i,j}] = 0$, $\big\|\mathrm{Cov}(\mb E'_{i,\cdot})\big\| \lesssim \norm{\mathrm{Cov}(\mb E_{i,\cdot})}$, and $\bb P(E_{i,j} = E'_{i,j}) \ge 1 - O(d^{-c-2})$. \label{item: bounded noise assumption 2}
            \item Define the incoherence degrees of $\mb U^*$ and $\mb V^*$ as 
            $\mu_1 \coloneqq {\norm{\mb U^*}\ti^2 n}/{K}$, and $\mu_2 \coloneqq {\norm{\mb V^*}\ti^2 p}/{K}.$
            Assume $B\Big\{ \frac{m^{\frac32}(\log d)^6}{\sigma \sqrt{n}}, \frac{m\sqrt{\mu_1 K} \log d}{\bar \sigma \sqrt{n}}, \frac{m\sqrt{\mu_2 K} \log d}{\bar \sigma \sqrt{p}}\Big\}\ll 1 $ and $ml \asymp p$. 
        \label{item: bounded noise assumption 3}
        \end{enumerate}
    \end{assumption}
We remark that Assumption~\ref{assumption: bounded noise}.\ref{item: bounded noise assumption 2} covers all the sub-Gaussian/sub-Exponential distributions with a bounded ratio between the sub-Gaussian/sub-Exponential norm and the standard deviation (e.g., distributions with constant parameters).    
Assumption~\ref{assumption: bounded noise}.\ref{item: bounded noise assumption 3} is aligned with conditions considered in \cite{cai2021subspace,yan2024inference}. Specifically, for the incoherence degrees $\mu_1$ and $\mu_2$, these allow us to overcome the irregularity in non-Gaussian cases so as to study their Gaussian-like behaviors, via the singular subspace perturbation theory.

\begin{theorem}
        \label{theorem: upper bound for algorithm (bounded)}
        Suppose that Assumption~\ref{assumption: bounded noise} holds and the clustering error of the initialization satisfies $l(\hat{\mb z}^{(0)}, \mb z^*)\leq c_1 \frac{n}{\beta K	(\log d)^4}$ with probability at least $1- O(n^{-2}) \vee \exp\big(-(1+o(1)\frac{\snr^2}{2}\big)$ for some sufficiently small constant $c_1$. 
        Then given arbitrary positive $c_2$ and $\epsilon$, the following holds for all $t \geq c_3\log n$ for some constant $c_3$: 
            \begin{enumerate}[itemsep=1pt, topsep=2pt]
        \item If $\mathsf{SNR} < \sqrt{(2 + \epsilon)\log n}$, then
        \eq{\bb E[h(\hat{\mb z}^{(t)},\mb z^*)] \leq \exp\Big(- (1+ o(1)) \frac{\snr^2}{2}\Big).\label{eq: upper bound for algorithm (bounded)} 
        }
        \item If $\mathsf{SNR}  \geq \sqrt{(2 + \epsilon)\log n }$ and $\nu = o(d)$, then $h(\hat{\mb z}^{(t)},\mb z^*) = 0$ with probability $1 - o(1)$. 
    \end{enumerate}
        \end{theorem}

We emphasize that Theorem~\ref{theorem: upper bound for algorithm (bounded)} attains the sharp constant multiplier $-1/2$ of $\snr^2$ in the exponent. This sharp rate hinges on precise control of the $2k$-moment of $\norm{\mb E_{i,\cdot}\mb A}_2$ given a matrix $\mb A \in \bb R^{p \times K}$ via the universality results of \cite{brailovskaya2022universality} (see Lemma~C.9). This allows us to precisely differentiate between the regimes of weak consistency (almost exact recovery) and strong consistency (exact recovery), i.e., whether we can obtain $\bb P[h(\hat{\mb z}^{(t)}, \mb z^*) = 0] = 1 - o(1)$. In addition, establishing consistency for center and covariance estimation under sample misspecification introduces further technical difficulties, due to the dependence across $\mk L_i$’s. Addressing this requires a careful decoupling argument.

\subsection{Initialization}
The required condition on the initialization error can be achieved by existing spectral methods.
In particular, the procedure of \cite{abbe2022}, which applies $(1+\epsilon)$-approximate $K$-means
to the rows of $\mb U$, yields the desired guarantee.

\begin{proposition}
\label{proposition: spectral initialization}
Suppose the assumptions of
Theorem~\ref{theorem: upper bound for algorithm (gaussian)}
or
Theorem~\ref{theorem: upper bound for algorithm (bounded)}
hold.
Then the clustering output $\hat{\mb z}^{(0)}$ produced by the algorithm in \cite{abbe2022}
satisfies
$
l(\hat{\mb z}^{(0)}, \mb z^*)
\;\ll\;
\frac{n}{\beta K(\log d)^4}
$
with probability at least
$1-O(n^{-2}) \,\vee\, \exp\!\bigl(-(1+o(1))\tfrac{\snr^2}{2}\bigr)$.
\end{proposition}

We emphasize that this initialization consistency does not follow directly from the expectation bound
on $\bb E[h(\hat{\mb z}^{(0)},\mb z^*)]$ established in \cite{abbe2022}.
Instead, it relies on our singular subspace perturbation theory combined with a careful decoupling argument; see details in Section~C.7.

\subsection{Implications of the Upper Bounds}\label{subsection: theoretical guarantee}
\label{subsec: implications of upper bounds}
We discuss the applicability of the conditions imposed in Theorem~\ref{theorem: upper bound for algorithm (gaussian)} and Theorem~\ref{theorem: upper bound for algorithm (bounded)}. For clarity of the following discussion on the signal strength, dependence structure, and noise pattern, we focus on the case with $\tau, K, \beta$, and $\mu_1 \vee \mu_2$ are $O(1)$.

\noindent\emph{Signal Strength.} To begin with, our theory extends to regimes with weak signal-to-noise ratios, namely when the $\mathsf{SNR}$ defined in \eqref{eq: SNR definition} grows just beyond $\sqrt{\log \log d}$. This extension is made possible by our new subspace perturbation analysis in Section~\ref{subsec: subspace perturbation}. Since exact recovery requires $\snr \geq \sqrt{(2 + \epsilon)\log n}$, our results cover the intermediate regime where a small fraction of samples may be misclustered. This is consistent with the discussion in \cite{abbe2022}, though in our case the analysis applies to a more adaptive algorithm, which necessitates additional technical effort in the proofs. 

\noindent\emph{Block Size.} We also comment on the block size under the general noise with the local dependence assumption. Assumption~\ref{assumption: bounded noise} implies that if $B$ is logarithmically greater (in terms of $d$) than  $\sigma$ {(the upper bound of the noise standard deviation)}, then the block size $m$ can scale as the order $O(p^{a})$ with $a \in (0,1)$, which corresponds to cases with severely dependent entries in the noise matrix.

\noindent\emph{Spiked Noise Cases.} Our theory allows for some spiked directions of the covariance matrices $\mb\Sigma_k$ that do not align with the subspace spanned by the cluster centers but lead to a large $\tilde\sigma$. This can be interpreted as the influence of some latent factors within the noise \cite{tang2024factor} and is commonly observed in real-world data. While a logarithmic requirement on the effective rank of $\bo \Sigma_k$ and $\bo \Sigma_{k_1}\bo \Sigma_{k_2}$ is imposed in Assumption~\ref{assumption: algorithm}--- primarily to ensure consistent  estimation of $\{\mb S_k^*\}_{k\in[K]}$---our theory remains applicable in scenarios where noise is small along most directions.

We further provide several important implications of our upper bound in comparison with the state-of-the-art analysis.  

\noindent\emph{Covering High-dimensional Regimes and Matching Sharp Thresholds in Special-case Models.} The sharp thresholds for isotropic Gaussian mixtrues have been investigated in \cite{ndaoud2018sharp,chen2021cutoff}. Taking a symmetric two-component isotropic Gaussian mixture with $\sigma =1$ and $n_1 = n_2$ for example, the sharp lower bound is given by $\exp\big(-(1 + o(1))\frac{\|\bo \theta_1^*\|_2^4}{2(\|\bo \theta_1^*\|_2^2 + p/n)} \big)$ in \cite[Theorem~2]{ndaoud2018sharp}. 
    When COPO is applied to a randomly perturbed version of this model, it achieves the same rate\footnote{
    Since the data matrix of the two-component symmetric Gaussian mixture model is rank-one, 
    we perturb it to rank two by adding a random unit vector $\mb v\in \bb R^p$, scaled by $\eta$, to all rows, i.e., we let $\mb R' = \mb R + \eta \mb 1_n \mb v\t $.
    Moreover, if $\eta = c \norm{\bo \theta_1^*}_2$ with $0 < c < 1$, one can prove by some concentration inequalities that $\frac{\|\bo \theta_1^*\|_2^4}{\|\bo \theta_1^*\|_2^2 + p/n} \Big/ \snr^2 \rightarrow 1$ by proving that the left top singular vector of $\mb R^* + \eta \mb 1_n \mb v\t $ asymptotically aligns with the direction of $(1\cdot \ind\{z_i^* = 1\} + 2 \cdot \ind\{z_i^* = 2\})_{i\in[n]}$. Hence $\exp\big(-(1+ o(1)) \snr^2/2\big) = \exp\big(-(1 + o(1))\frac{\|\bo \theta_1^*\|_2^4}{2(\|\bo \theta_1^*\|_2^2 + p/n)} \big)$. 
    }. 
    Moreover, for $p \gg n$, their lower bound indicates that almost exact recovery is possible only when $\norm{\bo \theta_1^*}_2 \gg (\frac{p}{n} )^{\frac14}$, whereas
    our method only requires $\norm{\bo \theta_1^*}_2 \gg (\frac{\log \log d \cdot p}{n} )^{\frac14}$ with an additional factor $(\log\log d)^{\frac14}$. Importantly, while their algorithms rely on specific structures in toy models, ours adapts to a much broader range of settings. 
    On the other hand, in the regime $(\frac{p}{n} )^{\frac12} \gg  \norm{\bo \theta_1^*}_2 \gg (\frac{p}{n} )^{\frac14}$,  the spectral clustering studied in \cite{loffler2021optimality,zhang2024leave} fail to achieve consistent estimation due to the insufficiency of the vanilla SVD.
    
\noindent\emph{Surpassing Homogeneous-Covariance-Focused Methods in High Dimensions.} 
The homogeneous covariance case is subsumed in our general inhomogeneous covariance case, while the former has been extensively studied in \cite{davis2021clustering,chen2024optimal}. Specifically for the cases with two centers symmetric about zero, \cite{davis2021clustering} provided an upper bound guarantee for their integer program but requires $n / p\log n \rightarrow \infty$, {i.e., not a high-dimensional scenario}. For more general $K$-component Gaussian mixtures with homogeneous covariance matrices, the hard-EM method proposed in \cite{chen2024optimal} requires $Kp = O(\sqrt{n})$, again not high-dimensional. In contrast, our method
    offers a robust solution to challenging high-dimensional mixture models.

\noindent\emph{Computational Efficiency.} We highlight the computational efficiency of the proposed method compared with the EM-based algorithm. COPO only requires performing the truncated eigen-decomposition on $\mc H(\mb Y \mb Y\t)$, which has a computation complexity of $O(npK)$. Additionally, COPO involves iterative averaging over the {projected} center space $\bb R^K$ and the {projected} covariance matrix space $\bb R^{K \times K}$ in $O(\log n)$ iterations. 

\section{Minimax Lower Bounds and Fundamental Limits}\label{section: lower bounds}
So far we have presented the recovery guarantees for the proposed COPO algorithm (Algorithm~\ref{algorithm: CPSC}). 
We next turn back to the question raised at the beginning: what is the information-theoretic limit of the clustering error? 
In the literature, this question has remained largely unexplored for cases with general covariance matrices, as it is intricate to determine which configurations of centers and covariances capture the true complexity of the problem. 
In this section, we provide two different perspectives to characterize the information-theoretic limits, which depend on different aspects of the covariance structure, respectively. 

Before proceeding, we first interpret the interplay between the separation between cluster centers and high dimensionality in the definition of $\mathsf{SNR}$ in Eq.\eqref{eq: SNR definition}; i.e., between $\sigma_K^*$ and $\sqrt{p}$. For simplicity, assume that the singular values of ${\mb V^*}\t \bo \Sigma_k \mb V^*$ for $k \in[K]$, the quantities $\min_{k_1,k_2 \in [K]}\tr(\bo \Sigma_{k_1} \bo \Sigma_{k_2})/p$ and $\max_{k_1,k_2 \in [K]}\tr(\bo \Sigma_{k_1} \bo \Sigma_{k_2})/p$, as well as the condition number $\kappa$ are all of constant orders. In this setting, the scale of $\mathsf{SNR}$ hinges on the scale of 
$\mb S_k^* = \mb S_k^{\mathsf{mod}} + \mb S_k^{\mathsf{exc}}$,
where $\mb S_k^{\mathsf{mod}} \coloneqq {\mb V^*}\t \bo \Sigma_k \mb V^* \asymp \mb I$ and $ \mb S_k^{\mathsf{exc}} \coloneqq {\bo \Lambda^*}^{-1} {\mb U^*}\t \bb E\big[ \mc H(\mb E\mb E\t)_{\cdot,i}\mc H(\mb E\mb E\t)_{i,\cdot} \big] \mb U^* {\bo \Lambda^*}^{-1} \asymp \frac{p}{{\sigma_K^*}^2} \mb I$ for any $i$ with $z_i^* = k$. 
We therefore consider two distinct regimes, where the dominant contribution to $\mb S_k^*$ comes from either $\mb S_k^{\mathsf{mod}}$ or $\mb S^{\mathsf{exc}}_k$, respectively.

\noindent\emph{Moderately high dimension ($p = o({\sigma_K^*}^2) $).   } When the dimension $p$ is not extremely large, the term $\mb S_k^{\mod}$ is the primary contributor to $\mb S_k^*$.  A straightforward calculation shows that
    \begin{align}
    &\mathsf{SNR} \approx \mathsf{SNR}^{\mathsf{mod}}, \text{  where } {\mathsf{SNR}^{\mathsf{mod}}}^2 \coloneqq \min_{j_1\neq j_2\in [K]}\min_{\mb x \in \bb R^K}\big\{(\mb x - \mb w_{j_1}^*)\t {\mb S^{\mod}_{j_1}}^{-1}(\mb x- \mb w_{j_1}^*): \\ 
&\qquad  \qquad\qquad (\mb x - \mb w_{j_1}^*)\t {\mb S^{\mod}_{j_1}}^{-1}(\mb x- \mb w_{j_1}^*) = (\mb x - \mb w_{j_2}^*)\t {\mb S^{\mod}_{j_2}}^{-1}(\mb x- \mb w_{j_2}^*)\big\}. 
\label{eq: SNR approx SNR^pn}
    \end{align}
This also corresponds to the case where $\mb U \mb U\t \mb U^* - \mb U^* \approx \mb E\mb V^*{\bo \Lambda^*}^{-1}$ holds, i.e., the fluctuation of the singular subspace estimation is primarily due to the projected noise.

\noindent\emph{Excessively high dimension ($p= \omega({\sigma_K^*}^2) = o({\sigma_K^*}^4 / n)$). } The excess growth of $p$ gives rise to the relation $\mb S_k^* \approx \mb S^{\exc}_k$. Consequently, one has 
    \begin{align}
         &\mathsf{SNR} \approx \mathsf{SNR}^{\mathsf{exc}}, \text{  where } {\mathsf{SNR}^{\mathsf{exc}}}^2 \coloneqq \min_{j_1\neq j_2\in [K]}\min_{\mb x \in \bb R^K}\big\{(\mb x - \mb w_{j_1}^*)\t {\mb S^{\exc}_{j_1}}^{-1}(\mb x- \mb w_{j_1}^*): \\ 
&\qquad  \qquad\qquad (\mb x - \mb w_{j_1}^*)\t {\mb S^{\exc}_{j_1}}^{-1}(\mb x- \mb w_{j_1}^*) = (\mb x - \mb w_{j_2}^*)\t {\mb S^{\exc}_{j_2}}^{-1}(\mb x- \mb w_{j_2}^*)\big\}. 
\label{eq: SNR approx SNR^qn}
    \end{align}
    In light of \eqref{eq: linear approximation of U}, the condition $p = o({\sigma_K^*}^4 / n)$ also implies that the difference $\mb U \mb U\t \mb U^* - \mb U^*$ is dominated by the quadratic noise term $\mc H(\mb E \mb E\t) \mb U^*{\bo \Lambda^*}^{-2}$.

In summary, in the idealized setting above, as long as the growth rates of $\sigma_K^*$ and $\sqrt{p}$ differ, the upper bound established in Theorem~\ref{theorem: upper bound for algorithm (gaussian)} and Theorem~\ref{theorem: upper bound for algorithm (bounded)} degenerates to either $\exp\big(-(1+o(1)) \frac{{\mathsf{SNR}^{\mathsf{mod}}}^2}{2}\big)$ or $\exp\big(-(1+o(1)) \frac{{\mathsf{SNR}^{\mathsf{exc}}}^2}{2}\big)$. These expressions will reappear, yet as minimax lower bounds, revealing the information-theoretic limits of anisotropic Gaussian mixtures as well as indicating the optimality of Algorithm~\ref{algorithm: CPSC}. Notably, the two quantities are tied respectively to the projection of $\bo \Sigma_k$ onto the subspace spanned by the centers and to the trace terms $\tr(\bo \Sigma_{k_1} \bo \Sigma_{k_2})$ for $k_1,k_2\in[K]$. Yet neither fully captures the information in $\bo \Sigma_k$, revealing the inherent barrier to retrieving the complete parameters in high dimensions. We summarize the results in Table~\ref{table: phase transition and lower bounds}.

\begin{table}[htbp]
\centering
\footnotesize
\setlength{\tabcolsep}{6pt}
\renewcommand{\arraystretch}{1.1}
\caption{Phase transition regimes and corresponding lower bounds}
\label{table: phase transition and lower bounds}

\begin{tabular}{llll}
\toprule
Regime 
& $\snr$ 
& Lower bound 
& $(\mb U\mb O-\mb U^*)_{i,\cdot}$ \\
\midrule
$n \ll p \ll (\sigma_K^*)^2$ \;(\emph{moderate})
& $\mathsf{SNR}^{\mathsf{mod}}$
& $\exp\!\big(-(1+o(1))\frac{(\snr^{\mathsf{mod}})^2}{2}\big)$
& $\mb E_{i,\cdot}\mb V^*{\bo\Lambda^*}^{-1}$ \\[0.3em]

$(\sigma_K^*)^2 \ll p \ll (\sigma_K^*)^4/n$ \;(\emph{excessive})
& $\mathsf{SNR}^{\mathsf{exc}}$
& $\exp\!\big(-(1+o(1))\frac{(\snr^{\mathsf{exc}})^2}{2}\big)$
& $\mc H(\mb E\mb E^\top)_{i,\cdot}\mb U^*{\bo\Lambda^*}^{-2}$ \\
\bottomrule
\end{tabular}
\end{table}

\subsection{Minimax Lower Bound Regarding $\mathsf{SNR}^{\mathsf{mod}}$ }\label{subsection: minimax lower bound for anisotropic gaussian mixtures}
As we already alluded to, the definition of $\mathsf{SNR}^{\mathsf{mod}}$ hints a subspace perspective that exclusively concerns the information contained within the span of the cluster centers $\bo\theta_1^*,\ldots,\bo\theta_K^*$. 
Next, we investigate the role of $\snr^{\mod}$ within the minimax framework. However, rather than simply selecting a least favorable subset and reducing the minimax rate to a two-point testing problem as in \cite{lu2016statistical,gao2017achieving,zhang2016minimax}, we delve into the dimension-reduction phenomenon associated with $\snr^{\mod}$ and analyze the gap between the actual risk and the Bayesian oracle risk, which is in stark contrast to the two-point testing approach.

\noindent\emph{Bayesian Oracle Risk. } An important quantity to assess clustering hardness is the Bayesian oracle risk, which assumes access to the full information of the cluster centers and covariance matrices and is related to complete distributional information, except the unknown label assignments.
For clarity, we consider two balanced Gaussian mixture components $\mc N(\bo \theta_k^*, \mb \Sigma_k)$, $k\in[2]$
with a prior $\frac{1}{2}\delta_{z_i^* = 1} + \frac{1}{2} \delta_{z_i^* = 2}$ for each sample. We examine the likelihood-ratio estimator $\tilde{\mb z}$ equipped with the oracle information {when parameters are known}:
\begin{equation}
     \tilde z(\mb y_i) =1 \cdot \ind{\big\{\phi_{\bo \theta_1^*, \bo \Sigma_1}(\mb y_i) \geq \phi_{\bo \theta_2^*, \bo \Sigma_2}(\mb y_i)\big\}} + 2 \cdot \ind{\big\{\phi_{\bo \theta_1^*, \bo \Sigma_1}(\mb y_i) > \phi_{\bo \theta_2^*, \bo \Sigma_2}(\mb y_i)\big\}}. \label{eq: likelihood ratio estimator}
\end{equation}
By Neyman-Pearson's theorem, the Bayesian oracle risk, arising from two-point testing, is
\begin{align}
     & \mc R^{\mathsf{Bayes}}(\{\bo \theta_j^*\}_{j\in[2]}, \{\mb \Sigma_j\}_{j\in[2]})
     = \frac{1}{2} \bb E_{\mb y\sim \mc N(\bo \theta_1^*, \mb \Sigma_1)}\big[\tilde z(\mb y) = 2\big] +  \frac{1}{2}\bb E_{\mb y \sim \mc N(\bo \theta_2^*, \mb \Sigma_2)}\big[\tilde z(\mb y)=  1\big].
    \label{eq: relation between bayesian oracle risk and likelihood ratio}
\end{align}
A common approach to lower bound the minimax clustering risk—developed in \cite{zhang2016minimax,gao2017achieving,gao2018community,chen2024achieving}—is to equate it to the Bayesian oracle risk:  
\eq{
    \inf_{\hat{\mb z}} \sup_{(\mb z^*, \bo \eta) \in \mb \Theta_z \times\{ (\bo \theta_1^*, \bo \theta_2^*, \mb \Sigma_1, \mb \Sigma_2)\}}\bb E[h(\hat{\mb z}, \mb z^*)] \gtrsim  \rbayes(\{\bo\theta_j^*\}_{j\in[2]}, \{\mb \Sigma_j\}_{j\in[2]}), 
}
where $\mb \Theta_z$ denotes a collection of approximately balanced clustering. 
The lower bounds in \cite{chen2024achieving} follow this strategy, which is effective when it is feasible to estimate complete distribution information so as to approximate the likelihood ratio estimator $\tilde z$.
However, in high-dimensional anisotropic settings, estimating the full $p \times p$ covariance matrices is impossible, rendering $\tilde z$ infeasible. This raises the question of whether the Bayesian oracle risk still tightly characterizes the minimax risk. In what follows, we first show that there exists such a parameter set that these two coincide. Beyond that, however, we justify that, for more natural and challenging settings where the two diverges, the quantity $\exp\big(-(1 +o(1))\frac{{\snr^{\mod}}^2}{2}\big)$ remains a valid lower bound.

With the above characterization, we will develop a preliminary understanding of the minimax misclustering rate. We define an approximately balanced cluster assignment set
\eq{
\mb \Theta_z \coloneqq \mb \Theta_z(\beta) = \left\{\mb z\in [2]^n: |\mathcal I_k(\mb z) | \in \left[n/2\beta, \beta n/2\right], k = 1,2\right\}, \label{definition: Theta_z}
}
and a parameter set of cluster centers and covariance matrices
\begin{align}
	& \tilde{\mb \Theta} \coloneqq \tilde{\mb\Theta}(n, p,  \mb S_1^{\mod}, \mb S_2^{\mod}, \mathsf{SNR}_0)= \Big\{( \bo \theta_1^*, \bo \theta_2^*, \mb \Sigma_1, \mb \Sigma_2): \mathsf{SNR}^{\mathsf{mod}}(\{\bo \theta_k^*\}, \{\mb \Sigma_k\}) =   \mathsf{SNR}_0^{\mathsf{mod}}; \\&   (\bo \theta_1^*, \bo \theta_2^*)=\mb V^* \mb R \text{ for $\mb V^*\in O(p,2)$ and $\mb R \in \mathrm{GL}_2(\bb R)$};~
\quad  {\mb V^{*\top}} \mb \Sigma_k \mb V^* = \mb S_k^{\mod}, k \in [2]\Big\}. 
\end{align}
Note that $\tilde{\mb \Theta}$ contains a group of parameters with easy-to-handle covariance matrices, for example, 
$\mb \Sigma_k = \mb V^*\mb S_k^{\mod}{\mb V^{*\top}} + \mb V_\perp^*{\mb V^*_\perp}$, for $k = 1,2.$
This structure allows the likelihood-ratio estimator to be reduced to that of a $K$-dimensional Gaussian mixture model, whose Bayesian oracle risk is tractable and takes the form $\exp\big(- (1 + o(1)){\mathsf{SNR}_0^{\mod}}^2/ 2\big)$. Hence, we are able to connect the problem to the Bayesian oracle risk and obtain the following lower bound. 
\begin{corollary}
	\label{corollary: simple gaussian lower bound}
	Consider two fixed projected covariance matrices $\mb S_1^{\mod}$, $\mb S_2^{\mod}$ and a parameter set $\mb \Theta \coloneqq \mb \Theta_z \times \tilde{\mb \Theta}$. If $\mathsf{SNR}_0^{\mod} \rightarrow\infty$, then the minimax misclustering rate over $\mb \Theta$ satisfies 
	\eq{
	\inf_{\hat{\mb z}} \sup_{(\mb z^*, {\bo \eta}) \in \mb \Theta}\bb E[h(\hat{\mb z}, \mb z^*)] \geq \exp\Big(- (1 + o(1))\frac{{\mathsf{SNR}_0^{\mod}}^2}{2}\Big).
	}
\end{corollary}

Although the above lower bounds presents our desired form, its proof hinges on the existence of parameters satisfying that the Bayesian oracle risk coincides with $\exp\big(-(1 + o(1){\snr^{\mod}_0}^2 / 2\big)$, a condition rarely met in practice. Typically, noise in the subspace spanned by the centers is intricately correlated with that in its orthogonal complement, often resulting $\mc R^{\mathsf{Bayes}}$ being much smaller than $\exp\big(-(1 + o(1))\frac{{\mathsf{SNR}^{\mathsf{mod}}}^2}{2}\big)$.

 \paragraph*{Challenging Cases: $\frac{-\log(\mc R^{\mathsf{Bayes}})}{\mathsf{SNR}^{\mathsf{mod}}} \geq \alpha >1$.}
The above discussion raises an intriguing question: if it is known \emph{a priori} that $\mc R^{\mathsf{Bayes}} $ is much smaller than $\exp\big(-(1 + o(1)) \frac{{\mathsf{SNR}^{\mathsf{mod}}}^2}{2}\big)$, does the latter still provide a valid lower bound for the minimax rate? 
Toward this, we consider the minimax rate of the following challenging restricted parameter space:
\begin{align}
& \tilde{\mb \Theta}_{\alpha} \coloneqq \tilde{\mb\Theta}_\alpha(n, p, \tilde\sigma, \mb S_1^{\mod}, \mb S_2^{\mod}, \mathsf{SNR}_0^{\mod})= \Big\{( \bo \theta_1^*, \bo \theta_2^*, \mb \Sigma_1, \mb \Sigma_2):  \\
& \quad (\bo \theta_1^*, \bo \theta_2^*)=\mb V^* \mb R \text{ for some $\mb V^*\in O(p,2)$ and $\mb R \in \mathrm{GL}_2(\bb R)$}; ~~\max_{k\in[2]} \bignorm{\mb \Sigma_k} \leq \tilde\sigma^2  ; \\ 
& \quad  {\mb V^{*\top}} \mb \Sigma_k \mb V^* = \mb S_k^{\mod}, k \in [2];~~
\mathsf{SNR}^{\mathsf{mod}}(\{\bo \theta_k^*\}, \{\mb \Sigma_k\}) =   \mathsf{SNR}^{\mathsf{mod}}_0; \quad \frac{-\log(\rbayes) }{{\mathsf{SNR}_0^{\mathsf{mf}}}^2 /2} \geq \alpha^2\Big\},\\ 
&\mb \Theta_\alpha \coloneqq \mb \Theta_\alpha (n, p, \tilde\sigma , \mb S_1^*, \mb S_2^*, \mathsf{SNR}_0^{\mod}, \beta ) =  \mb \Theta_z  \times \tilde{\mb \Theta}_\alpha \label{eq: parameter space for two-component cases}
\end{align}
with $\mathsf{SNR}_0^{\mod} >0$, $\alpha >1$, $\beta > 0$, and $\mb S_1^{\mod},\mb S_2^{\mod} \succ 0$, where $\mb \Theta_z$ is defined in \eqref{definition: Theta_z}. 

Note that the condition $\frac{-\log(\rbayes) }{{\mathsf{SNR}_0^{\mathsf{mf}}}^2 /2} \geq \alpha^2 >1$ implies that $\mc R^{\mathsf{Bayes}} $ is much smaller than $\exp\big(-(1 + o(1)) \frac{{\mathsf{SNR}^{\mathsf{mod}}_0}^2}{2}\big)$. 
Surprisingly enough, even though what we are left with is a more challenging problem, the minimax rate is shown to be of the form $\exp(-(1 + o(1)) \frac{{\mathsf{SNR}_0^{\mod}}^2}{2})$ and thus is solely related to the information in the subspace spanned by the cluster centers. We have the following main result.

\begin{theorem}[Minimax Lower Bound for Two-component Gaussian Mixtures]\label{theorem: gaussian lower bound}
	Consider the two-component Gaussian mixture model and the parameter space $\mb \Theta_\alpha = \mb \Theta_\alpha(n,p,\tilde\sigma$, $\mb S_1^{\mod}, \mb S_2^{\mod}, \mathsf{SNR}_0^{\mod}, \beta)$ with $\alpha > 1$, $\mb S_1^{\mod}$, and $\mb S_2^{\mod}$ being fixed.
 Then given $\mathsf{SNR}_0^{\mod} \rightarrow \infty$, $\frac{\log \beta}{{\mathsf{SNR}_0^{\mod}}^2}\rightarrow 0$, and $\tilde\sigma = \max_{k\in[2]}\norm{\mb S_k^{\mod}}^{\frac12}(\mathsf{SNR}_0^{\mod})^\iota$ for some $\iota >0$, the following holds if
    ${n (\mathsf{SNR}_0^{\mod})^{4\iota}} = o(p)$: 
	\eq{
	\inf_{\hat{\mb z}}\sup_{(\mb z^*, \bo \eta) \in \mb\Theta_\alpha}\bb E[h(\hat{\mb z}, \mb z^*)] \geq \exp\Big(-(1 + o(1))\frac{{\mathsf{SNR}_0^{\mod}}^2}{2}\Big).
	}
\end{theorem}

We briefly remark on the conditions in Theorem \ref{theorem: gaussian lower bound}: (i) The last condition ${n {\mathsf{SNR}_0^{\mod}}^{4\iota}} = o(p)$ enforces high-dimensionality ($n=o(p)$) of a sequence of mixture models. (ii) The condition $\tilde\sigma = \max_{k\in[2]}\norm{\mb S_k^{\mod}}{\mathsf{SNR}_0^{\mod}}^\iota$ allows for covariance matrices to exhibit larger variability in directions not aligned with $\mb V^*$ compared to those within $\mb V^*$, which is crucial in our proof. 

The lower bound in Theorem \ref{theorem: gaussian lower bound} has an exponent related to $\snr^{\mod}$. A direct consequence of the condition in $\bo \Theta_\alpha$ is that 
$$\log\Big[\inf_{\hat{\mb z}}\sup_{(\mb z^*, \bo \eta) \in \mb\Theta_\alpha }\bb E[h(\hat{\mb z}, \mb z^*)]\Big]  > \log\Big[\max_{ \bo \eta \in \tilde{\mb \Theta}_\alpha} \rbayes(\bo \eta ) \Big]. $$ 

As far as we know, this is the first result of proving the substantial discrepancy between the actual risk and the Bayesian oracle risk in general anisotropic Gaussian mixtures. This also suggests that the lower bound derived from the two-point testing argument is not tight. 
Achieving a tight lower bound therefore requires explicitly quantifying the discrepancy between the minimax rate and the Bayesian oracle risk, which constitutes the main technical challenge of our analysis.
A key insight underlying our proof is that this discrepancy in high dimensions stems from the ambiguity between parameter configurations that share identical projected covariances and cluster centers, yet differ in their full covariance structures. We construct such examples to induce smaller Bayesian oracle risks; see the paragraph \emph{Covariance Construction} in the proof of Theorem~\ref{theorem: gaussian lower bound} for details. 

For general $K$-component cases, we also establish the lower bound; see Section~\ref{subsec: K component lower bounds}.

\subsection{Minimax Lower Bound Regarding $\snr^{\mathsf{exc}}$}
In this part, we discuss another type of lower bound associated with $\snr^{\mathsf{exc}}$ that comes into play when the impact of high dimensionality exceeds that of cluster centers' signals. Our attention herein focuses on the homogeneous covariance settings. The reason is that in the inhomogeneous case, even with vanishing signals, differences in covariance structure alone can sometimes make clustering feasible and complicates the discussion\footnote{For example, imagine a two-component Gaussian mixture model with zero mean separation but distinct covariances, namely, $\mc N(\mb 0, \mb I_p)$ and $\mc N(\mb 0, 2 \mb I_p)$. When $p$ is large enough, elementary concentrations inform us that the $\ell_2$ norm of a sample concentrates around either $\sqrt{p}$ or $\sqrt{2p}$, which allows one to reliably recover $\mb z^*$ with high probability by simply checking the samples' $\ell_2$ norms.}.

Our lower bound regarding $\snr^{\exc}$ is stated as follows.
\begin{theorem}
        Consider the two-component Gaussian mixture model and the following parameter space for some sufficiently large constant $C$: 
    \begin{align}
        & \bo \Theta^{\mathsf{exc}} \coloneqq \bo \Theta^{\mathsf{exc}} (n, p, \bo \Sigma,  \snr^{\exc}_0) = \Big\{(\mb z^*, \{\bo \theta_k^*\}_{k\in[2]}, \{\bo \Sigma_k\}_{k\in[2]}): \bo \Sigma_k = \bo \Sigma;
        \\
        &
n_k(\mb z^*) \in \Big[\big(1 - C\sqrt{\frac{\log n} {n}}\big) \frac{n}{2}, \big(1 + C\sqrt{\frac{\log n} {n}}\big) \frac{n}{2}\Big]; ~
\snr^{\mathsf{exc}}(\mb z^*, \{\bo \theta_k^*\}, \{\mb \Sigma_k\}) \geq    \snr^{\mathsf{exc}}_0 \Big\}. 
\end{align}
Assume $(\mathsf{SNR}^{\mathsf{exc}}_0)^{\frac12} (n / \tr(\bo \Sigma^2))^{\frac14} \norm{\bo \Sigma}^{\frac12} \ll 1$ and $\mathsf{SNR}^{\mathsf{exc}}_0 \rightarrow \infty$. 
Then 
    \begin{align}
        & \inf_{\hat{\mb z}} \sup_{(\mb z, \{\bo \theta_k^*\}_{k\in[K]}, \{\bo \Sigma_k\}_{k\in[K]})\in \bo \Theta^{\mathsf{exc}} } \bb E\big[h(\hat{\mb z}, \mb z^*) \big] \geq \exp\Big(-(1 +o(1)) \frac{{\snr^{\mathsf{exc}}_0}^2}{2} \Big). 
    \end{align} 
    \label{thm: lower bound regarding SNReh}
\end{theorem}

From the expression of $\snr^{\exc}$, we observe that it relies solely on the traces $\tr(\bo \Sigma^2)$. This reflects the underlying principle that, when the structure is overly complex, the most effective strategy is to distill the essential information. The proof of Theorem~\ref{thm: lower bound regarding SNReh} builds on the approach used in \cite{ndaoud2018sharp, chen2021cutoff}: by introducing a Gaussian prior on the centers, the minimax problem is reduced to a Bayes risk problem. Nevertheless, accommodating the anisotropic covariance structure and selecting an appropriate prior demand additional technical care.

\section{Simulation Studies}\label{sec:simulation}
We conduct extensive simulations to evaluate COPO under a variety of settings and compare Algorithm~\ref{algorithm: CPSC} with several standard clustering methods. In all experiments, COPO is run for $\lfloor \log n \rfloor$ iterations, and performance is assessed via clustering error across different signal strengths and dimensions. We consider balanced two-component mixtures with $n_1 = n_2 = n/2$. For the simplest Gaussian setting with dense centers, comparisons with EM and spectral methods are deferred to the supplementary (Section~\ref{sec: additional experiments}); here we focus on non-Gaussian mixtures and Gaussian mixtures with sparse centers.

\subsection{Non-Gaussian Distributions}
To assess the performance for non-Gaussian data under flexible local dependent noise, we compare COPO with the {K-Means} algorithm and the spectral methods in the following four data generation settings, where we always fix $n = 200$ and vary the dimension $p$ from $100$ to $240$ and let the sizes of two clusters be equal. Spectral methods apply K-means to the embedding $\hat{\mb U} \hat{\bo \Lambda}$: in \textbf{spectral clustering}, $\hat{\mb U}$ and $\hat{\bo\Lambda}$ are the top-$K$ singular vectors and singular values of $\mb R$, respectively. In \textbf{diagonal-deleted spectral clustering}, they are the top-$K$ eigenvectors of $\mc H(\mb R\mb R^\top)$ and the square roots of the associated eigenvalues.

\begin{figure}[htbp]
    \centering
    \begin{subfigure}[b]{0.24\textwidth}
        \centering
        \includegraphics[width=\textwidth]{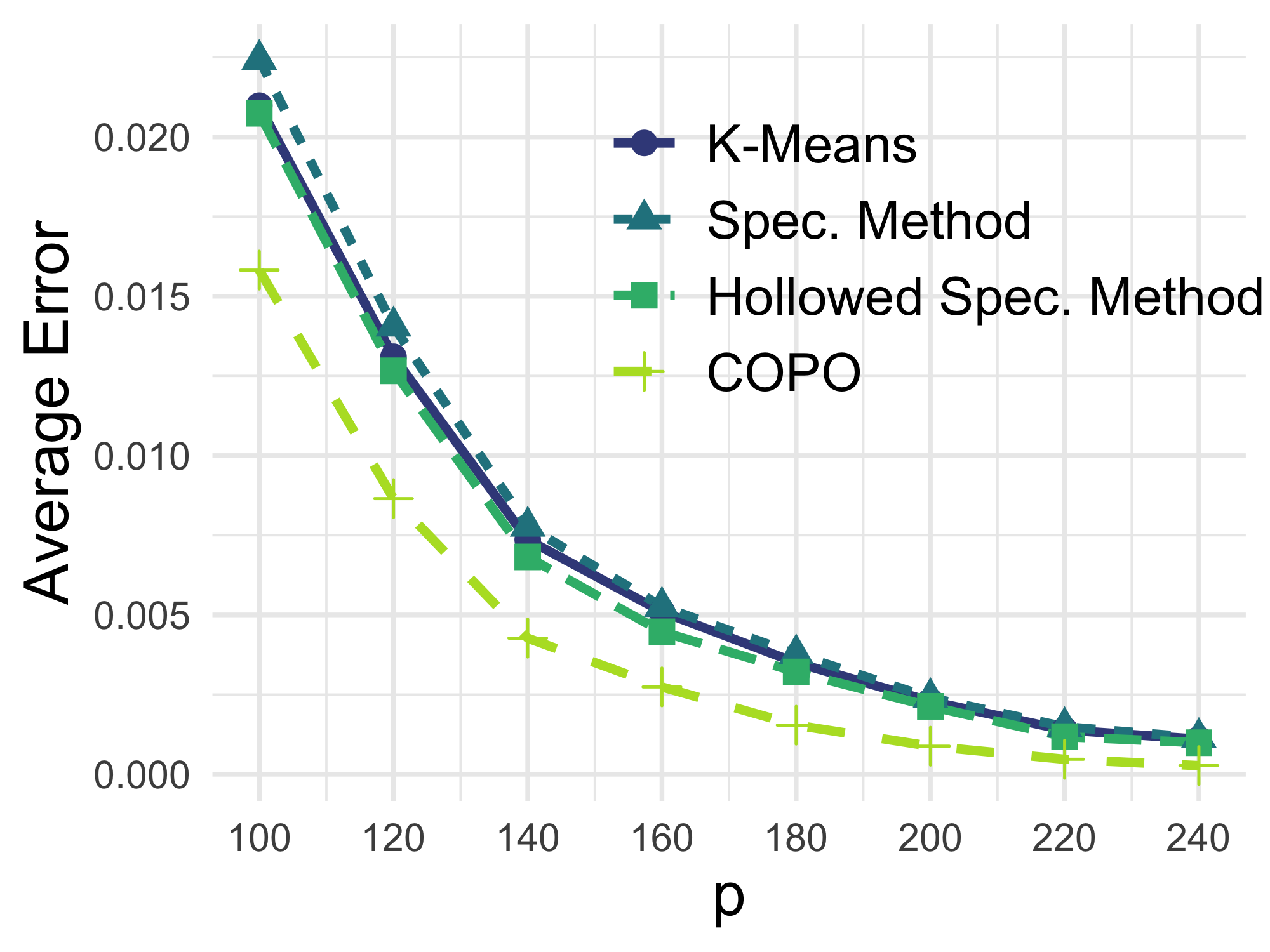}
        \caption{{\footnotesize Ising Mixtures}}
        \label{fig: ising mixture}
    \end{subfigure}
            \begin{subfigure}[b]{0.24\textwidth}
        \centering
        \includegraphics[width=\textwidth]{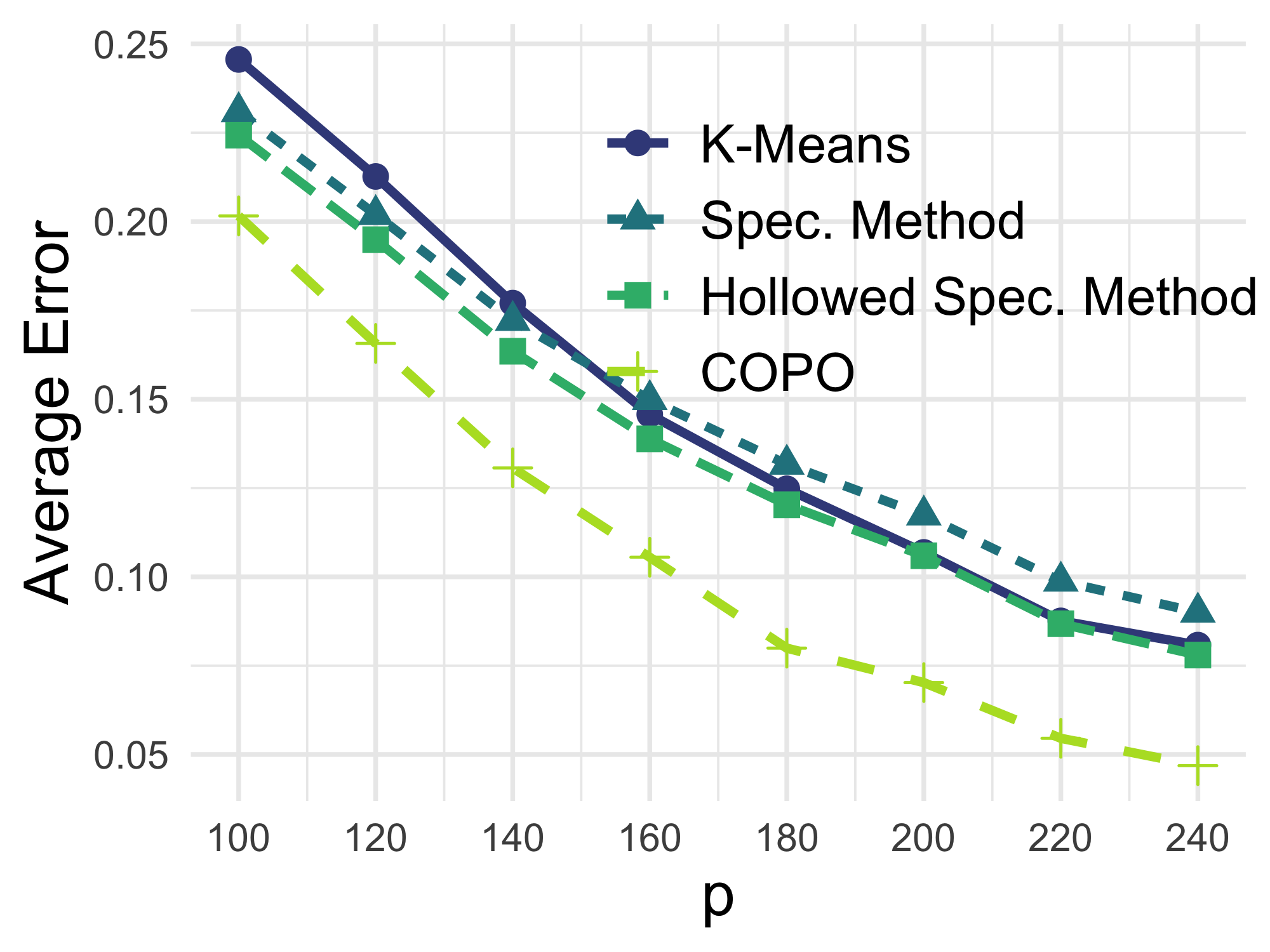}
        \caption{{\footnotesize Probit Mixtures}}
        \label{fig: dichotomous mixture}
    \end{subfigure}
    \begin{subfigure}[b]{0.24\textwidth}
        \centering
        \includegraphics[width=\textwidth]{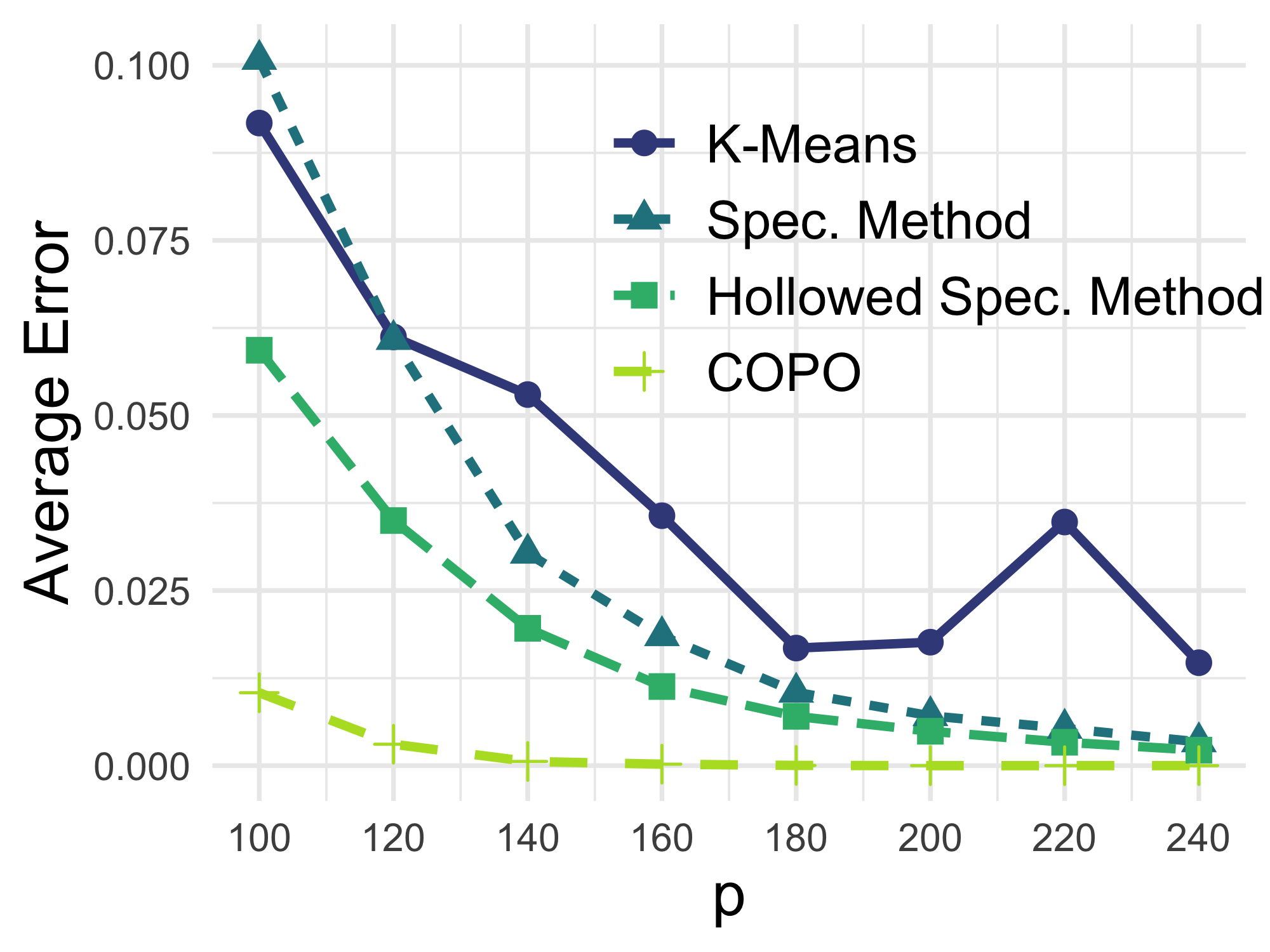}
        \caption{{\footnotesize Gamma Mixtures}}
        \label{fig: gamma mixture}
    \end{subfigure}
    \begin{subfigure}[b]{0.24\textwidth}
        \centering
        \includegraphics[width=\textwidth]{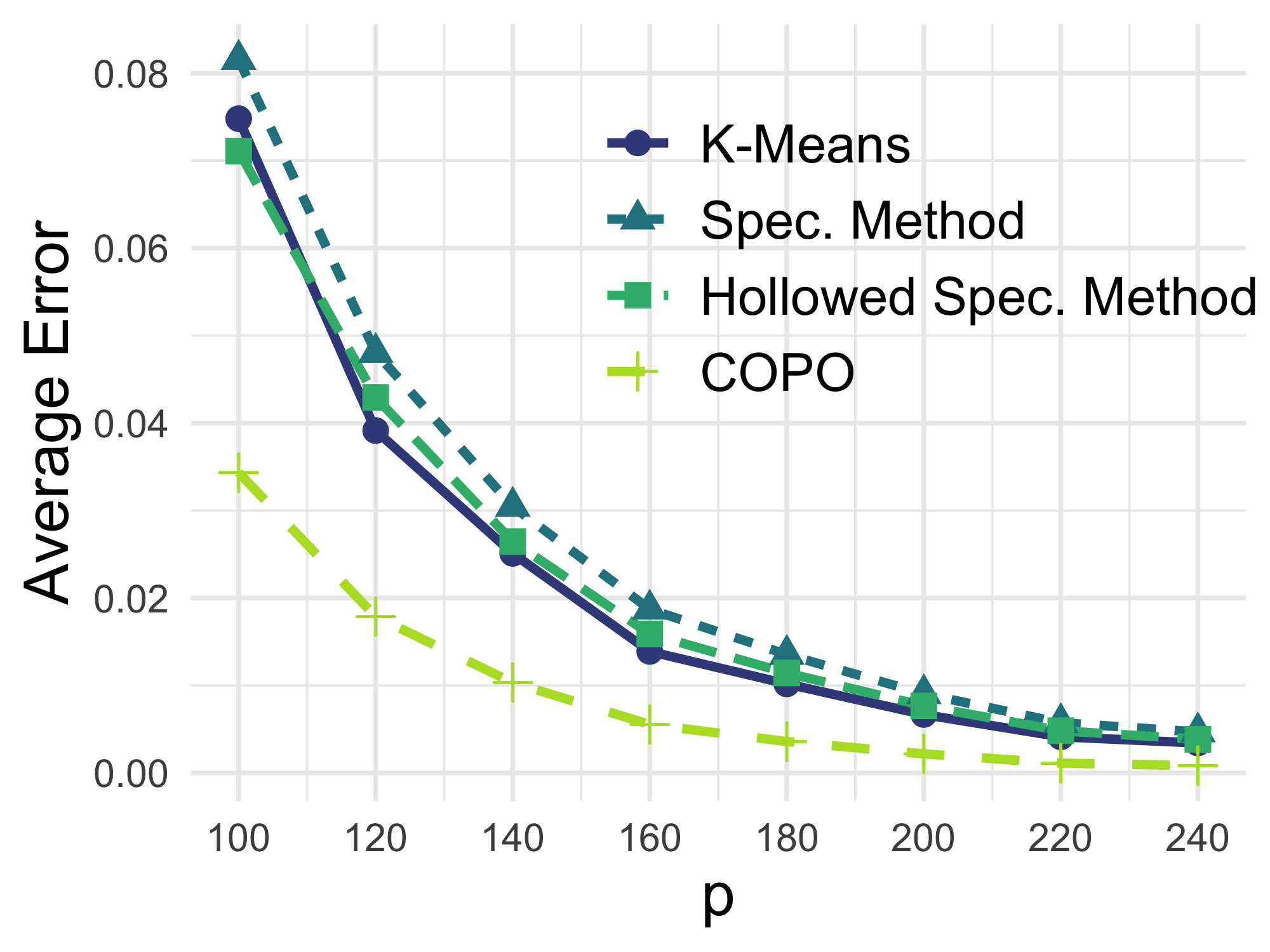}
        \caption{{\footnotesize Neg.~Binom.~Mixtures}}
        \label{fig: binomial mixture}
    \end{subfigure}
    \caption{Clustering error rates with varying dimensions for Ising mixtures, multivariate Probit mixtures, multivariate Gamma mixtures, multivariate Negative Binomial mixtures.
    }
    \label{fig: non-Gaussian noise}
\end{figure}

\emph{Mixtures of Ising Models. }
We generate multivariate binary data from the mixtures of Ising models.
For convenience, we first introduce two interaction matrices $\mb G_1, \mb G_2 \in \bb R^{4\times 4}$ with $(\mb G_1)_{i,j} = 0.1^{|i - j|} \ind_{\{i \neq j\}}$ and  $(\mb G_2)_{i,j} = 0.3^{|i - j|} \ind_{\{i \neq j\}}, i, j \in[4]$ and two thresholding vectors $\mb v_1 = (-1,-1,-1,-1)\t $, $ \mb v_2 = (-3,-3,-1,-1)\t$. For $\mb y_i\in \bb R^p$ belonging to the $k$-th component, we independently sample $(\mb y_i)_{4(l-1)+1:4l}$ from the distribution 
\eq{
\bb P\big[(\mb y_i)_{4(l-1)+1:4l} = \mb x \big] = \frac{\exp(\mb x\t \mb G_k\mb x + \mb v_k\t \mb x)}{\sum_{\mb z \in \{-1,1\}^4}\exp(\mb z\t \mb G_k\mb z + \mb v_k\t \mb z)}
}
for $x \in \{-1, 1\}^4$ and $l \in [p / 4]$. 
Figure~\ref{fig: ising mixture}  presents the clustering results.

\emph{Mixtures of Multivariate Probit Models. }
We generate data from the mixtures of multivariate probit models. The multivariate binary data have latent dependence structures across different features induced by dichotomizing an underlying Gaussian random vector. 
Define an autoregressive matrix 
$\mb A_\rho=\left(\begin{matrix}
    1 & \rho\\
    \rho& 1
\end{matrix}\right)$. In each trial, we independently generate $\rho_{k,j}$ ($k\in[2], j\in[p/2]$) and set two underlying covariance matrices to be $\tilde{\mb \Sigma}_{k} = \text{diag}(\mb A_{\rho_{k,1}},\cdots, \mb A_{\rho_{k,p / 2}})$, $k \in [2]$. Then we draw an underlying Gaussian matrix $\check{\mb Y} = (\check{\mb y}_1,\cdots, \check{\mb y}_n)\t  \in \bb R^{n\times p}$ where $\check{\mb y}_i \sim \mc N(\mb 0, \mb \Sigma_{z_i^*})$. 
The binary data matrix $\mb Y = (y_{i,j})_{i\in[n], j\in[p]}$ is constructed using thresholding vectors $\mb v_1 = (\mb 1_{p/2}, 0.1 \cdot \mb 1_{p/2})$ and $\mb v_2 = (1.5 \cdot \mb 1_{p/2}, -0.2 \cdot \mb 1_{p/2})$ by $(\mb y_{i})_j = \ind{\{(\check {\mb y}_{i})_j \geq (\mb v_{z_i^*})_j\}}$. 
Figure~\ref{fig: dichotomous mixture} presents the clustering results.

\emph{Mixtures of Multivariate Gamma Distributions. }
As mentioned earlier, COPO is also applicable to unbounded sub-Gaussian / sub-exponential data. We examine a mixture of two Multivariate Gamma distributions. Let $\mathrm{Gamma}(k, \theta)$ be a gamma distribution with shape $k$ and  scale $\theta$. For the first component, we set $k = 1,~\theta=1$ for the first $p/2$ entries and $k = 0.2,~\theta = 10$ for the last $p/2$ entries. For the second component, we set $r = 2,~\theta=1$ for the first $p/2 $ entries and $k = 1,~\theta=1$ for the last $p/2$ entries. Figure~\ref{fig: gamma mixture}  presents the results.

\emph{Mixtures of Multivariate Negative Binomial Distributions. }
We consider unbounded count data, where each entry follows a negative binomial distribution $\mathrm{NB}(r,p)$, with $r$ as the number of successes and $p$ as the success probability. For the first component, we set $r=6,~p = 0.48$ for the first $p/2$ entries and $r = 1, ~p  = 0.08$ for the last $p/2$ entries. For the second component, we let $r = 3, ~p = 0.24$ for all entries. Figure~\ref{fig: binomial mixture}  presents the results.

\medskip
In summary, Figures~\ref{fig: ising mixture}--\ref{fig: binomial mixture} demonstrate that  COPO uniformly outperforms the K-means and spectral clustering methods across various types of data. This demonstrates COPO's strong adaptivity to various mixture distributions and dependence structures.

\subsection{Comparison with Sparsity/Selection-Based Algorithms} 
Due to the intractability of full parameter recovery in high dimensions, another line of work diverges by pursuing estimation consistency under certain structural assumptions such as sparsity, to name a few, see \cite{azizyan2013minimax,cai2019chime,JinWang2016_IFPCA,WangGuNingLiu2015_HighDimEM}. Here we adopt the Influential PCA in \cite{JinWang2016_IFPCA} and the CHIME \footnote{Since CHIME requires both an initialization and a tuning parameter $\lambda_n$, we initialize the algorithm using K-Means and then determine an appropriate $\lambda_n$	by examining the support size of the estimates along a decreasing geometric sequence. } in \cite{cai2019chime} as the benchmark. Our comparison highlights two key aspects: (i) how the support size impacts estimation accuracy for COPO relative to the benchmark methods, and (ii) how increasing dimensionality affects performance when the support size is fixed as a proportion of the total dimension. The results are summarized in Table~\ref{tab:comparison-sparsity-methods}. 

\begin{table}[htbp]
\centering

\setlength{\tabcolsep}{4pt}
\renewcommand{\arraystretch}{1.05}
\caption{Misclustering error under varying sparsity and dimensionality.}
\label{tab:comparison-sparsity-methods}
\parbox{0.45\linewidth}{
\centering
\textbf{(a) Varying support size $s$} \\[0.3em]
\begin{tabular}{ccc ccc}
\toprule
$n$ & $p$ & $s$ & COPO & IF-PCA & CHIME \\
\midrule
100 & 500 & 10 & \textbf{0.145} & 0.444 & 0.148 \\
100 & 500 & 30 & 0.157 & 0.432 & \textbf{0.156} \\
100 & 500 & 50 & \textbf{0.148} & 0.431 & 0.160 \\
100 & 500 & 70 & \textbf{0.143} & 0.431 & 0.149 \\
100 & 500 & 90 & \textbf{0.147} & 0.431 & 0.161 \\
\bottomrule
\end{tabular}
}
\qquad
\parbox{0.45\linewidth}{
\centering
\textbf{(b) Varying dim. $p$ with fixed $s/p$} \\[0.3em]
\begin{tabular}{ccc ccc}
\toprule
$n$ & $p$ & $s$ & COPO & IF-PCA & CHIME \\
\midrule
100 & 200  & 10 & \textbf{0.028} & 0.399 & 0.069 \\
100 & 400  & 20 & \textbf{0.014} & 0.398 & 0.022 \\
100 & 600  & 30 & \textbf{0.009} & 0.395 & 0.019 \\
100 & 800  & 40 & \textbf{0.007} & 0.387 & 0.020 \\
100 & 1000 & 50 & \textbf{0.005} & 0.386 & 0.030 \\
\bottomrule
\end{tabular}
}

\end{table}

Specifically, we consider a two-component Gaussian mixture model with centers $\bo \theta_1^* = \frac{\alpha \sqrt{2}}{\sqrt{s}} (\mb 1_s, \mb 0_{p-s})$ and $\bo \theta_2^* = \frac{\alpha \sqrt{2}}{\sqrt{s}} (\mb 0_s, \mb 1_s, \mb 0_{p-2s})$ and a common covariance $\mb I_p$. 
In Table~\ref{tab:comparison-sparsity-methods}(a), we fix $n$ and $p$, set $\alpha = 4$, and vary the support size $s$, while keeping the singular values of $\bo \Theta^*$ constant by our construction. When $s$ is small -- favoring sparsity-based approaches -- COPO performs comparably to CHIME and significantly better than IF-PCA, which fails to achieve competitive accuracy. As $s$ increases, COPO achieves consistently lower error rates, demonstrating the benefit of not relying on sparsity assumptions. In Table~\ref{tab:comparison-sparsity-methods}(b), we fix $n$ and the ratio $s/p$ while setting $\alpha = 4(p/n)^{1/4}$, thereby isolating the effect of dimensionality. Across all regimes, COPO consistently outperforms IF-PCA and CHIME, with the performance gap widening as $p$ grows. This highlights COPO’s robustness in high dimensions.

\section{Real Data Analysis}
\label{sec: real data analysis}
We consider a high-dimensional single-cell 10x scATAC-seq dataset\footnote{The dataset is publicly available at \url{https://cellxgene.cziscience.com/collections/d36ca85c-3e8b-444c-ba3e-a645040c6185}.} \citep{lengyel2022molecular}. It includes five types of annotated cell: endothelial cell, smooth muscle cell, stromal cell, pericyte, and leukocyte, of sizes $156,~165,~15607,~2104,$ and $283$, respectively. 
The data contain the non-negative counts of gene expressions in each cell with $1604$ cells and $p = 19298$ genes in total. 
Since the numbers of samples across clusters are unbalanced, we choose the first $100$ samples for each cluster, leading to $n=500$ samples. Thus, this sample size is much smaller than the dimension of features, which brings extra challenge for clustering problem.

To study the mixture patterns and the impact of diagonal deletion, we apply two approaches: (i) vanilla SVD on the data matrix $\mathbf{Y}$, and (ii) eigen-decomposition on $\mathcal{H}(\mathbf{Y}\mathbf{Y}^\top)$. We then examine the pairwise plots of the top five singular vectors, color-coded by the ground-truth subpopulation labels (see Figures~\ref{fig: svd of single cell} and~\ref{fig: eigen-decomposition of H(YYt) of single cell}). 

Interestingly, several scatter plots—for instance, those of $X_4$ versus $X_5$—reveal notable differences between SVD and diagonal-deleted eigen-decomposition: different clusters in the latter case are more clearly separated, whereas the SVD plots contain hazardous outliers and weaker separation between clusters. This suggests that the leading eigenvectors of $\mathcal{H}(\mathbf{Y}\mathbf{Y}^\top)$ capture clustering structure more effectively, consistent with theoretical results on the benefits of diagonal deletion \cite{cai2021subspace}. This observation also empirically justifies the diagonal deletion step in the COPO algorithm.

Moreover, the pair plots in Figure~\ref{fig: eigen-decomposition of H(YYt) of single cell} exhibits clear nonspherical shapes in each cluster, which suggests the existence of distributional heterogeneity. It is worth mentioning that a recent paper \cite{lyu2024degree} tried to interpret such phenomena by introducing a degree parameter for each sample to capture the within-cluster heterogeneity. 
Nonetheless, if one considers that the degree parameter is independently sampled from a distribution and views the shape of each cluster as part of the noise, then the model setting in \cite{lyu2024degree} can be viewed as a special case of mixture models with nonspherical additive noise, which can be tackled by COPO. 

\begin{figure}[htbp]
    \centering
    \begin{subfigure}[b]{0.43\textwidth}
        \centering
        \includegraphics[width=\textwidth]{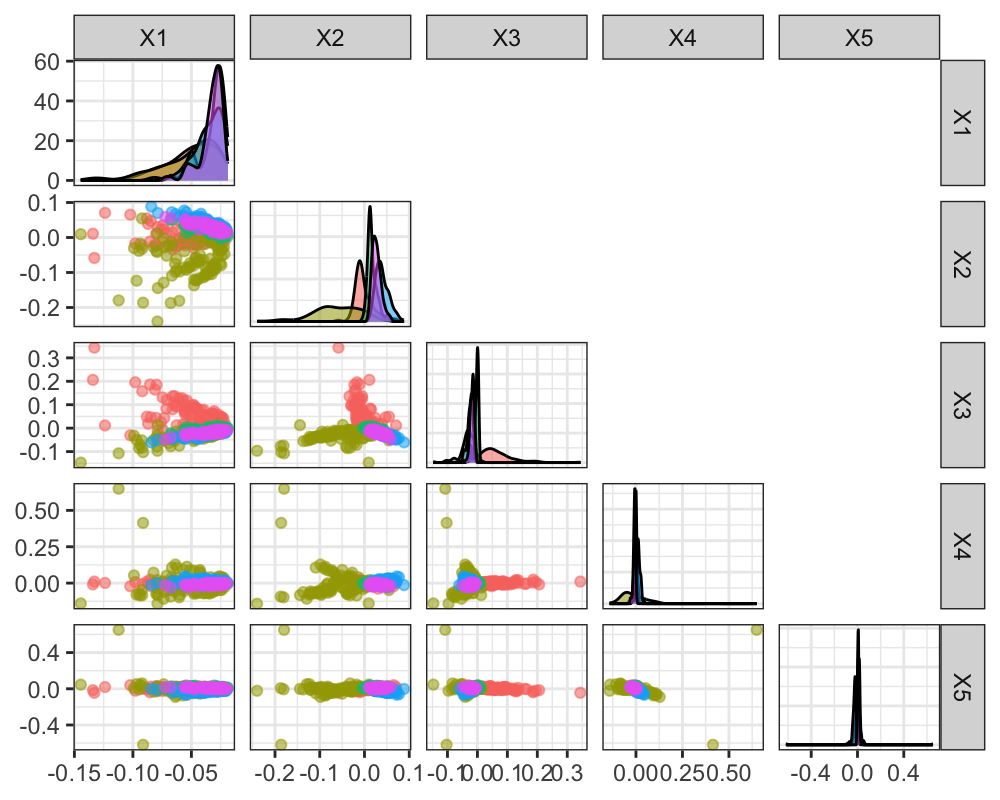}
        \caption{Truncated SVD of $\mb Y$}
        \label{fig: svd of single cell}
    \end{subfigure}\quad 
            \begin{subfigure}[b]{0.43\textwidth}
        \centering
        \includegraphics[width=\textwidth]{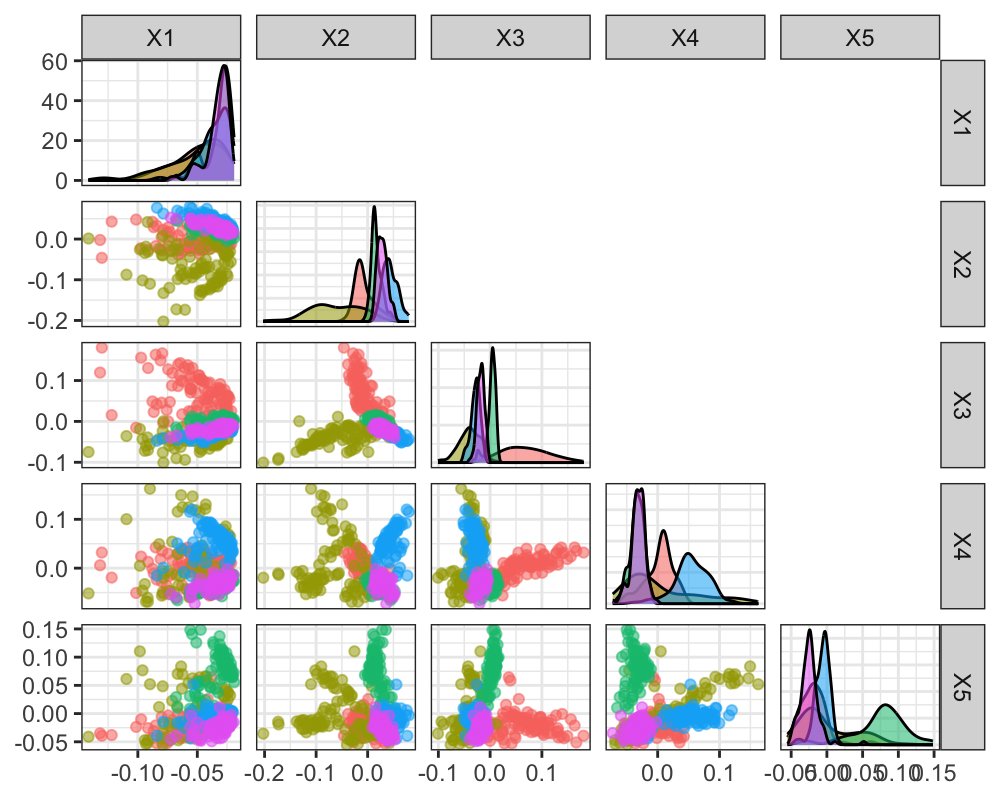}
        \caption{Leading Eigenvectors of $\mc H(\mb Y \mb Y\t)$}
        \label{fig: eigen-decomposition of H(YYt) of single cell}
    \end{subfigure}
    \label{figure: single cell pairplots}
    \caption{Pairplots of Low-Dimensional Embeddings for the Single-Cell Dataset. The left panel shows pairwise plots of the top-5 left singular vectors of $Y$, taken two columns at a time. The right panel shows pairwise plots of the top-5 eigenvectors of $\mc H(\mb Y \mb Y\t)$, also taken two columns at a time. 
}
\end{figure}

In Table~\ref{tab: pairwise comparison table for single cell data}, we compare our method with K-mean and spectral clustering \cite{zhang2024leave}, diagonal-deleted spectral clustering \cite{abbe2022}, IFPCA \cite{JinWang2016_IFPCA}, and CHIME \cite{cai2019chime}.
COPO consistently attains the lowest error rates in all comparisons, with the sole exception of the pericyte–leukocyte pair. For each pair of cell types, we plot the contours and decision boundaries from COPO, based on its cluster assignments along with the estimated centers and covariance structures. In many cases, COPO successfully captures the nonspherical shapes of the underlying clusters and adapts the decision boundaries accordingly, leading to a more faithful representation of the population heterogeneity.

\begin{table}[ht]
\centering

\setlength{\tabcolsep}{3pt}
\renewcommand{\arraystretch}{1.05}
\caption{Pairwise misclustering rates $h(\hat{\mb z},\mb z^*)$ for single-cell data}
\label{tab: pairwise comparison table for single cell data}
\begin{tabular}{lcccccc}
\toprule
Cell type pair & K-Means & Spectral & Diagonal-deleted & \textbf{COPO} & IFPCA & CHIME \\
\midrule
Pericyte vs Stromal              & 0.500 & 0.490 & 0.495 & \textbf{0.445} & 0.495 & 0.500 \\
Pericyte vs Smooth muscle        & 0.305 & 0.310 & 0.310 & \textbf{0.245} & 0.395 & 0.405 \\
Pericyte vs Endothelial          & 0.325 & 0.325 & 0.325 & \textbf{0.145} & 0.335 & 0.405 \\
Pericyte vs Leukocyte            & 0.310 & 0.310 & \textbf{0.305} & 0.315 & 0.335 & 0.405 \\
Stromal vs Smooth muscle         & 0.330 & 0.305 & 0.305 & \textbf{0.230} & 0.340 & 0.390 \\
Stromal vs Endothelial           & 0.350 & 0.315 & 0.325 & \textbf{0.170} & 0.320 & 0.455 \\
Stromal vs Leukocyte             & 0.345 & 0.310 & 0.310 & \textbf{0.260} & 0.360 & 0.375 \\
Smooth muscle vs Endothelial     & 0.370 & 0.370 & 0.370 & \textbf{0.180} & 0.500 & 0.425 \\
Smooth muscle vs Leukocyte       & 0.475 & 0.490 & 0.490 & \textbf{0.350} & 0.480 & 0.500 \\
Endothelial vs Leukocyte         & 0.410 & 0.390 & 0.395 & \textbf{0.220} & 0.495 & 0.445 \\
\midrule
All cell types                   & 0.638 & 0.636 & 0.576 & \textbf{0.438} & 0.594 & -- \\
\bottomrule
\end{tabular}

Pairwise comparison for single-cell data. Each entry reports the clustering performance for distinguishing two cell types using different methods: K-means, spectral clustering, hollowed spectral clustering, COPO, IFPCA, and CHIME.  Bold values indicate the lowest misclustering rate within each row.
CHIME is designed for two-component Gaussian mixtures and is therefore not applied to the all-cell-type setting.
\end{table}

	\begin{figure}[htbp]
	\centering
    \includegraphics[width = \linewidth]{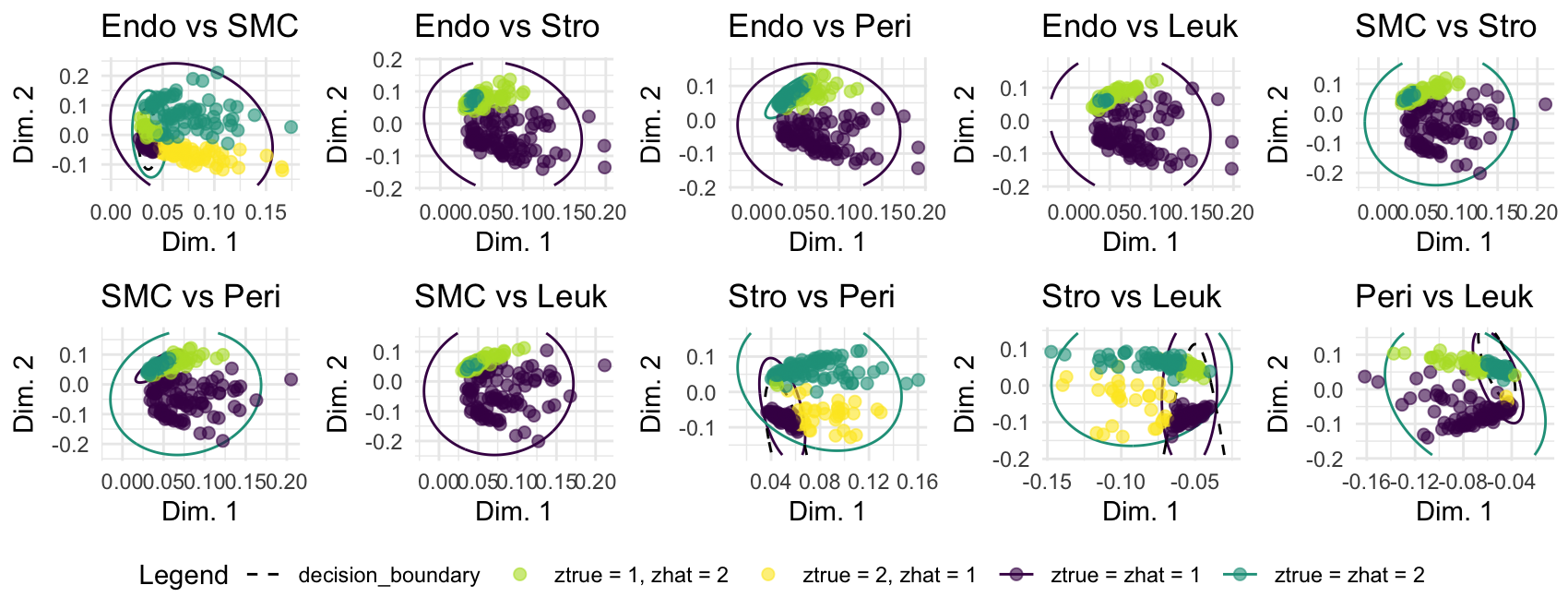}
  \caption{Single-cell data. For each cluster, we display contours defined by the Mahalanobis distance together with the decision boundaries estimated by COPO. Points are color-coded to reflect both the estimated cluster labels (\texttt{zhat}) and the ground truth (\texttt{ztrue}). }
      \label{figure: pairwise contour plot for single cell}
  \end{figure}

\section{Conclusion and Discussion} 
\label{sec: conclusion and discussion}
This paper proposes COPO, a novel and easy-to-implement clustering algorithm that adaptively learns projected covariance structures to improve clustering accuracy. 
To provide theoretical guarantees, we develop a refined subspace perturbation theory for the leading eigenvectors of $\mathcal{H}(\mathbf{R}\mathbf{R}^\top)$. This result yields universal upper bounds for COPO under a wide class of flexible noise distributions. To further probe the information-theoretic limits of anisotropic Gaussian mixture models, we derive two distinct lower bounds that highlight separate sources of clustering difficulty arising from covariance heterogeneity. Extensive numerical experiments corroborate the strong empirical performance of COPO.

There are several promising directions for future work. First, the current dependence of $\snr$ on the number of clusters $K$ and the condition number $\kappa$ does not appear to be optimal compared with spectral clustering results \cite{zhang2024leave}. An interesting question is whether the analysis can be sharpened to achieve optimal dependence. 
Additionally, the estimation of the number of clusters $K$ is an important problem in practice. A series of studies have addressed this problem in clustering and network analysis \cite{zhang2024leave,ma2021determining,jin2023optimal,lei2016goodness}. It would be interesting to explore how heteroskedasticity affects the estimation of $K$ and the performance of our algorithm when $K$ is unknown. Finally, the refined subspace perturbation theory developed in Section~\ref{subsec: subspace perturbation} may be of independent interest for analyzing weak-consistency regimes in other clustering settings, such as bipartite stochastic block models \cite{florescu2016spectral,zhou2020optimal,ndaoud2021improved} , tensor block models \cite{wang2019multiway,han2022exact}, and multi-layer networks \cite{lei2020consistent,chen2022global,agterberg2025joint,bhattacharyya2018spectral,arroyo2021inference}.

\bigskip
\spacingset{0.8}
\paragraph*{Supplementary Material. }
The Supplementary Material includes full proofs of all theoretical results, along with additional lower bounds and simulation studies.

{
\bibliographystyle{apalike} 
\bibliography{bibliography.bib}  }

\appendix
\tableofcontents

\section{Proofs of Lower Bounds}\label{sec-proof of lower bounds}
This section provides proofs of the lower bounds stated in the main text. It also presents a $K$-component version of Theorem~\ref{theorem: gaussian lower bound} along with its proof. 

\subsection{Lower Bound for K-Component Mixtures}
\label{subsec: K component lower bounds}
Beyond the presented two-component cases, we extend our minimax lower bound to the $K$-component Gaussian mixture model with general covariance matrix structures. Following the notations in the main text, we generalize the definition of $\rbayes$ to the $K$-component case by defining 
$$\rbayes(\{\bo \theta_k^*\}_{k\in[K]}, \{\mb \Sigma_k \}_{k\in[K]}) \coloneqq  \max_{a\neq b\in [K]}\rbayes(\{\bo \theta_a^*, \bo \theta_b^*\}, \{\mb \Sigma_a, \mb \Sigma_b \}).$$ 
Analogous to \eqref{eq: parameter space for two-component cases}, the parameter space is defined as:
    \begin{align}
       & \tilde{\mb \Theta}_{\alpha,K}  \coloneqq  \tilde{\mb \Theta}_{\alpha,K}(n, p, \tilde\sigma, \bar \sigma, \underline\sigma, {\mathsf{SNR}_0^{\mod}}) \coloneqq   
         \Big\{(\{\bo \theta_k^*\}_{k\in[K]}, \{\mb \Sigma_k\}_{k\in[K]}): \\ 
        &  (\bo \theta_1^*, \cdots,\bo \theta_K^*) = \mb V^* \mb R\text{ for some $\mb V^* \in O(p,K)$ and $\mb R\in \mathrm{GL}_K(\bb R)$}; ~ \max_{k\in[K]}\norm{\mb \Sigma_k} \leq \tilde\sigma^2; \\ 
        & \mathsf{SNR}^{\mod}(\{\bo \theta_k^*\}, \{\mb \Sigma_k\}) =  \mathsf{SNR}^{\mod}_0  ; \max_{k\in[K]} \norm{\mb S_k^{\mod}} \leq \bar\sigma^2; \min_{k\in[K]}\sigmamin(\mb S_k^{\mod})\geq \underline\sigma^2;  \frac{-\log(\rbayes ) }{{\mathsf{SNR}^{\mod}_0}^2 /2} \geq \alpha^2
        \Big\}, 
         \\
          &\mb \Theta_{z, K} \coloneqq \mb \Theta_{z, K}(\beta) = \left\{\mb z\in [K]^n: |\mathcal I_k(\mb z) | \in \left[\frac{n}{K\beta}, \frac{\beta n}{K}\right], k \in [K]\right\},\label{definition: Theta_z K}\\
             & \mb \Theta_{\alpha , K}    \coloneqq  \mb \Theta_{z, K} \times \tilde{\mb \Theta}_{\alpha,K}  =  \mb \Theta_{\alpha, K}(n,p, \tilde\sigma, \bar \sigma, \underline\sigma, {\mathsf{SNR}_0^{\mod}}, \beta)\label{eq: parameter space for K-component cases}.
    \end{align}

For ease of presentation, we no longer explicitly specify the forms of projected covariance matrices $\mb S_k^{\mod}$ as done in the two-component case. What remains unchanged is that we still focus on the challenging cases where $\frac{-\log(\mc R^{\mathsf{Bayes}})}{{\mathsf{SNR}^{\mod}}^2 / 2} \geq \alpha >1$ so as to illustrate the information-theoretic difficulty to achieve the Bayesian oracle risk. 
For a sequence of possibly growing numbers of components $K$, the following theorem offers a lower bound for the $K$-component Gaussian mixture model, whose proof is deferred to Section~\ref{subsec: proof of theorem: gaussian lower bound with K components}.

\begin{theorem}[Minimax Lower Bound for $K$-component Gaussian Mixtures]\label{theorem: gaussian lower bound with K components}
    Consider the $K$-component Gaussian mixture model and the parameter space $\mb \Theta_{\alpha, K}$ with $1<\alpha<\frac43$, and $\bar\sigma, \underline\sigma$ being some positive constants. Given ${\mathsf{SNR}_0^{\mod}} \rightarrow \infty$, 
    $\frac{K(\log \beta \vee 1)}{{\mathsf{SNR}_0^{\mod}}^2} \rightarrow 0$, $\tilde \sigma = \bar\sigma{\mathsf{SNR}_0^{\mod}}^{\iota}$, and $n{\mathsf{SNR}_0^{\mod}}^{4\iota} = o(p)$ for an arbitrary $\iota >0$, one has 
	\eq{
	\inf_{\hat{\mb z}}\sup_{(\mb z^*, \bo \eta ) \in \mb\Theta_{\alpha,K}}\bb E[h(\hat{\mb z}, \mb z^*)] \geq \exp\Big(-(1 + o(1))\frac{{\mathsf{SNR}_0^{\mod}}^2}{2}\Big). 
	} 
\end{theorem}

\subsection{Characterization of the Bayesian Oracle Risk}
We introduce a quantity determined by the full population parameters $\{\bo \theta_k^*\}_{k\in[K]}$ and $\{ \bo \Sigma_k\}_{k\in[K]}$: 
\longeq{
    &\snrfull(\{\bo \theta_k^*\}, \{\bo \Sigma_k\})
    \coloneqq  \min_{k_1\neq k_2 \in [K]} \min_{\mb y\in \bb R^{p}} \Big\{ (y - \bo \theta_{k_1}^*)\t \bo \Sigma_{k_1}^{-1} (\mb y - \bo \theta_{k_1}^*): (y - \bo \theta_{k_1}^*)\t \bo \Sigma_{k_1}^{-1} (\mb y - \bo \theta_{k_1}^*) + \frac{1}{2} \log |\bo \Sigma_{k_1}| \\ 
    & \qquad \qquad \qquad \qquad = (y - \bo \theta_{k_2}^*)\t \bo \Sigma_{k_2}^{-1} (\mb y - \bo \theta_{k_2}^*) + \frac{1}{2} \log |\bo \Sigma_{k_2}| \Big\} . 
    \label{eq: SNRfull definition}
}
To keep things concise, here and throughout, we may refer to the functions $\mathsf{SNR}^{\mathsf{mod}}(\cdot)$, $\mathsf{SNR}^{\mathsf{exc}}(\cdot)$, $\snr(\cdot)$, and $\snrfull(\cdot)$ simply as $\mathsf{SNR}^{\mathsf{mod}}$, $\mathsf{SNR}^{\mathsf{exc}}$, $\snr$, and $\snrfull$ when applied to a tuple of parameters, with the context making this clear. 

\subsubsection{Useful Facts of $\mathsf{SNR}^{\mathsf{mod}}$ and $\snrfull$ }
We begin with a few elementary facts about $\mathsf{SNR}^{\mathsf{mod}}$ defined in Eq.~\eqref{eq: SNR definition} and $\snrfull$ defined in Eq.~\eqref{eq: SNRfull definition} for \emph{two-component Gaussian mixtures}, in order to understand how they are related to covariance structures and Bayesian oracle risk.
When the dimension $p$ is fixed, the Bayesian oracle risk admits a clean expression (cf. \cite{chen2024optimal}): 
\eq{
   \mc R^{\mathsf{Bayes}}(\{\bo \theta_j^*\}_{j\in[2]}, \{\mb \Sigma_j\}_{j\in[2]}) = \exp\Big( -(1 + o(1) )\frac{{\snrfull}^2}{2}\Big),\label{eq: simple relation between SNRfull and Rbayes}
}
provided $\snrfull \rightarrow \infty$ and regularity conditions on $\mb \Sigma_k$ hold. 

In general high-dimensional settings, this relation may break down. Nevertheless, \eqref{eq: simple relation between SNRfull and Rbayes} remains valid in key scenarios--—specifically, when covariances are identical across all directions (homogeneous covariance case) or nearly so in most directions (inhomogeneous covariance case)---and is formalized in the following propositions. 
Consider a sequence of orthonormal matrices $\{\mb V^*\}_{n \in \bb N^+}\subset O(p,2)$ representing subspaces spanned by the cluster centers,  two fixed positive-definite projected covariance matrices $\mb S_k^{\mod}\in \bb R^{2\times 2}, k =1, 2$. Given those configuration, imagine a sequence of two-component anisotropic Gaussian mixtures with centers $\{\{\bo \theta^*_k\}_{k \in [2]}\}_{n \in \bb N^+} $ aligned with the subspace spanned by $\mb V^* \in\bb R^{p \times 2}$, covariances $\{\{\mb \Sigma_{k}\}_{k\in[2]}\}_{n\in \bb N^+}$ such that ${\mb V^*}\t \mb \Sigma_k \mb V^* = \mb S_k^{\mod}$ for $k\in[2]$, and  $\mathsf{SNR}^{\mod}\big(\{\bo \theta_k^*\}_{k\in[2]}, \{ \mb \Sigma_k\}_{k\in[2]}\big) \rightarrow \infty$ as $n$ goes to infinity. To recall, we omit the subscript $n$ for all sequences indexed by $n$ throughout the discussion. 
\begin{proposition}[Homogeneous covariance matrices]
\label{proposition: SNR' and Bayesian oracle risk 1}
    Suppose $\mb \Sigma_k,~k\in[2]$ are positive-definite and $\mb \Sigma_{1} = \mb \Sigma_{2}$ for $n\in \bb N^+$. Then  $
        \mathsf{SNR}^{\mathsf{mod}} = \bignorm{{\mb S_1^{\mod}}^{-\frac{1}{2}}(\mb w_{1}^*- \mb w_{2}^*)}_2 / 2,\quad \snrfull = \bignorm{\mb \Sigma_{1}^{-\frac{1}{2}}(\bo \theta_{1}^* - \bo \theta_{2}^*)}_2 / 2, \quad \snrfull \geq \mathsf{SNR}^{\mathsf{mod}}$, 
        where $\mb w_k^* \coloneqq {\mb V^{*\top}} \bo \theta^*_k$ denotes the projected centers. Further, it holds for the Bayesian oracle risk that
        \eq{\label{eq: bayesian oracle risk and SNR' lemma}
           \rbayes(\{\bo \theta_{j}^*\}_{j\in[2]}, \{\mb \Sigma_k\}_{k\in[2]})  = \exp\Big(-(1+ o(1))\frac{{\snrfull}^2}{2} \Big). 
        } 
        \end{proposition}

        \begin{proposition}[Covariance matrices homogeneous in most directions]
        \label{proposition: SNR' and Bayesian oracle risk 2}
        Suppose that there exists a sequence of orthogonal matrices $\big(\tilde{\mb V}, \tilde{\mb V}_{\perp} \big) \in O(p), n\in\bb N^+$ with 
        $\tilde{\mb V} \in O(p,a)$ and $\tilde{\mb V}_{\perp} \in O(p,p-a)$
        for some fixed integer $a > 2$ such that (a)
$\mb V^*$ coincides with the first two columns of $\tilde{\mb V}$ (i.e., $\mb V^* = (\tilde{\mb V})_{:,1:2})$, 
(b) ${\tilde{\mb V}_{\perp}}\t (\mb \Sigma_{1} - \mb \Sigma_{2})\tilde{\mb V}_{\perp}  = \mb 0$ (similarity of covariance matrices in most directions),
(c) $\tilde{\mb V}\t \mb \Sigma_{1}\tilde{\mb V}_{\perp} = \tilde{\mb V}\t \mb \Sigma_{2}\tilde{\mb V}_{\perp} = \mb 0$ (uncorrelatedness of noise in the directions of $\tilde{\mb V}$ and $\tilde{\mb V}_{\perp}$),
(d) the eigenvalues of $\tilde{\mb V}\t \mb \Sigma_k\tilde{\mb V}, k\in[2], n \in \bb N^+$ are lower and upper bounded by positive constants.     Then 
        it holds that $
            {\snrfull}^2 \geq {\mathsf{SNR}^{\mathsf{mod}}}^2 -  \big|\log|\tilde{\mb V}\t \mb \Sigma_{2}\tilde{\mb V}| - \log|\tilde{\mb V}\t \mb \Sigma_{1}\tilde{\mb V}|\big|, \label{eq: SNRfull and SNR}
        $
        and 
        \eq{
        \rbayes(\{\bo \theta_{k}^*\}_{k\in[2]}, \{\mb \Sigma_k\}_{k\in[2]})  = \exp\Big(-(1+ o(1))\frac{{\snrfull}^2}{2} \Big).\label{eq: inhomo-cov rbayes}
        }
\end{proposition}

\begin{remark}
    Note that we are interested in the regime where the signal strength goes to infinity compared with the noise and thus assume that $\mathsf{SNR}^{\mathsf{mod}}\rightarrow \infty$. Therefore, if we fix the matrices $\tilde{\mb V}\t \mb \Sigma_k\tilde{\mb V}, k = 1,2$, the logarithmic terms in \eqref{eq: SNRfull and SNR} are always negligible and imply ${\snrfull}^2 \geq (1 + o(1)){\mathsf{SNR}^{\mathsf{mod}}}^2$.
\end{remark}
\subsubsection{Proof of Proposition~\ref{proposition: SNR' and Bayesian oracle risk 1}}
\label{sec: proof of SNR' and Bayesian oracle risk}
The explicit forms of $\mathsf{SNR}$ and $\snrfull$
follow from their definition thanks to the homogeneous covariances. 
To apply \cite[Lemma~A.1]{chen2024optimal}  on testing error for Linear Discriminant Analysis to the Bayesian oracle risk, it suffices to verify that $\snrfull \rightarrow \infty$ as $n$ goes to infinity. By definition, we have 
\longeq{
    \snrfull(\{\mb \Sigma_k\}_{k\in[2]}, \{\bo \theta_{k}^*\}_{k\in[2]}) =&  \bignorm{\mb \Sigma_{1}^{-\frac{1}{2}}(\bo \theta_{1}^* - \bo \theta_{2}^*)}_2 / 2 \\ 
    =& \bignorm{(\bo \theta_{1}^* - \bo \theta_{2}^*)\t \mb V^*{\mb V^*}\t \mb \Sigma_{1}^{-1}\mb V^*{\mb V^*}\t (\bo \theta_{1}^* - \bo \theta_{2}^*)}_2^{\frac{1}{2}} / 2 \\ 
    = & \bignorm{(\mb w_{1}^* - \mb w_{2}^*)\t \big({\mb V^*}\t \mb \Sigma_{1}^{-1}\mb V^*\big)(\mb w^*_{1} - \mb w_2^*)}_2^{\frac{1}{2}} / 2 \\ 
    = & \bignorm{\big({\mb V^{*\top}} \mb \Sigma_{1}^{-1}\mb V^*\big)^{\frac{1}{2}}(\mb w^*_{1} - \mb w_2^*)}_2 / 2.
}
Then $\snrfull(\{\bo \theta_{k}^*\}_{k\in[2]}, \{\mb \Sigma_k\}_{k\in[2]}) \geq \mathsf{SNR}( \{\bo \theta_{k}^*\}_{k\in[2]}, \{\mb \Sigma_k\}_{k\in[2]})$ follows from the fact that ${\mb V^{*\top}} \mb \Sigma_k^{-1}\mb V^*\succeq {\mb S_k^{\mod}}^{-1}$ for $k\in[2]$. Since $\mathsf{SNR}\rightarrow \infty$, we therefore have $\snrfull\rightarrow \infty$ as $n$ goes to infinity. From the proof of \cite[Lemma A.1]{chen2024optimal} on testing error for Linear Discriminant Analysis and the fact that $\snrfull =\frac{1}{2}\bignorm{(\mb \Sigma_{1})^{-\frac{1}{2}}(\bo \theta_{1}^* - \bo \theta_{2}^*)}_2 \rightarrow \infty$, we have
\longeq{
    &\min_{\hat {\mb z}}\bb E_{z^* \sim \frac{1}{2}\delta_1 + \frac{1}{2}\delta_2,  \mb y\sim\mc N(\bo \theta_{z^*}^*, \mb \Sigma_{z^*}) }\big[\tilde z(\mb y) \neq z^*\big]  = \bb P\big[\epsilon \geq \frac{1}{2}\bignorm{(\mb \Sigma_{1})^{-\frac{1}{2}}(\bo \theta_{1}^* - \bo \theta_{2}^*)}_2 \big] \\ 
    = & \exp\Big(-(1 + o(1) )\frac{\snrfull^2}{2}\Big),
}
where $\epsilon$ is a standard Gaussian random variable. 
\subsubsection{Proof of Proposition~\ref{proposition: SNR' and Bayesian oracle risk 2}}
We point out that the quantities $\mathsf{SNR}$, $\snrfull$, and $\rbayes$ are invariant under rotations induced by orthogonal transformations. Specifically, for any orthogonal matrix $\mb R \in O(p)$, the following equalities hold:
\begin{align}
    &\mathsf{SNR}^{\mod}(\{\bo\theta^*_{k,n}\}_{k\in[2]}, \{\mb \Sigma_k\}_{k\in[2]}) = \mathsf{SNR}^{\mod}(\{\mb R\bo\theta^*_{k,n}\}_{k\in[2]}, \{\mb R\mb \Sigma_k \mb R\t\}_{k\in[2]}), \\ 
    & \snrfull(\{\bo\theta^*_{k,n}\}_{k\in[2]}, \{\mb \Sigma_k\}_{k\in[2]}) = \snrfull(\{\mb R\bo\theta^*_{k,n}\}_{k\in[2]}, \{\mb R\mb \Sigma_k \mb R\t\}_{k\in[2]}), \\ 
    & \rbayes(\{\bo\theta^*_{k,n}\}_{k\in[2]}, \{\mb \Sigma_k\}_{k\in[2]}) = \rbayes(\{\mb R\bo\theta^*_{k,n}\}_{k\in[2]}, \{\mb R\mb \Sigma_k \mb R\t\}_{k\in[2]}).
\end{align}  
Therefore, it suffices to consider the case where 
$$
\tilde {\mb V}_{n} =
\begin{pmatrix}
    \mb I_{a\times a} \\
     \mb 0_{ (p - a)\times a}
\end{pmatrix} 
= \big(\mb V^*, \check{\mb V}_n\big),
\text{ with }\mb V^* = \left(\begin{matrix}\mb I_{2\times 2 }\\ 
\mb 0_{ (p-2) \times 2}
\end{matrix}\right), \text{ and }\tilde {\mb V}_{n,\perp} =
\begin{pmatrix}
\mb 0_{a\times (p-a)}\\
\mb I_{(p-a)\times (p-a)}
\end{pmatrix}.
$$

The expression in the definition of $\snrfull$ is reduced to 
\longeq{
    &\left({\snrfull}(\{\bo \theta^*_{k}\}_{k\in[2]}, \{\mb \Sigma_k\}_{k\in[2]})\right)^2 \\
    =& \min_{i,j \in[2], i \neq j} ~\min_{\mb x \in \bb R^p}\big\{\mb x\t\tilde{\mb V}(\tilde{\mb V}\t \mb \Sigma_{i,n} \tilde{\mb V})^{-1}\tilde{\mb V}\t \mb x + \mb x\t\tilde{\mb V}_{\perp}(\tilde{\mb V}_{\perp}\t \mb \Sigma_{i,n} \tilde{\mb V}_{\perp})^{-1}\tilde{\mb V}_{\perp}\t \mb x: \\
    & \qquad \frac{1}{2}\mb x\t \tilde{\mb V}((\tilde{\mb V}\t \mb \Sigma_{j,n} \tilde{\mb V})^{-1} - (\tilde{\mb V}\t \mb \Sigma_{i,n} \tilde{\mb V})^{-1})\tilde{\mb V}\t \mb x\\ 
    & \qquad + \mb x\t \tilde{\mb V} (\tilde{\mb V}\t \mb \Sigma_{j,n} \tilde{\mb V})^{-1}\tilde{\mb V}\t  
    (\bo \theta_{i,n}^* - \bo \theta_{j,n}^*)
    \\
   &\qquad 
   + \frac{1}{2}(\bo \theta_{i,n}^* - \bo \theta_{j,n}^*)\t\tilde{\mb V}(\tilde{\mb V}\t  \mb \Sigma_{j,n}\tilde{\mb V})^{-1} \tilde{\mb V}\t (\bo \theta_{i,n}^* - \bo \theta_{j,n}^*)\\ 
   & \qquad  - \frac{1}{2}\log|\tilde{\mb V}\t \mb \Sigma_{i,n}^*\tilde{\mb V}| + \frac{1}{2}\log |\tilde{\mb V}\t \mb \Sigma_{j,n}^* \tilde{\mb V}| = 0\big\}\label{eq: SNR' reduction 0}
}
where we use the fact that $\tilde{\mb V}\t \mb \Sigma_{i,n}^{-1} \tilde{\mb V}= (\tilde{\mb V}\t \mb \Sigma_{i,n} \tilde{\mb V})^{-1}$, $\tilde{\mb V}_{\perp}\t \mb \Sigma_{i,n}^{-1} \tilde{\mb V}_{\perp}= (\tilde{\mb V}_{\perp}\t \mb \Sigma_{i,n} \tilde{\mb V}_{\perp})^{-1}$ for $i\in[2]$ since ${\tilde{\mb V}}\t \mb \Sigma_{i,n}\tilde{\mb V}_{\perp} = \mb 0$. 

Without loss of generality, we assume that $i=1$ is the minimizer of the above expression.
To facilitate the comparison with $\mathsf{SNR}$, we introduce two functions $f_1^{\mathsf{full}}$, $f_2^{\mathsf{full}}$ of $x \in \bb R^a$ and rewrite \eqref{eq: SNR' reduction 0} as taking the minimum over the $a$-dimensional, rather than the $p$-dimensional, space:
\longeq{
&\left({\snrfull}(\{\bo \theta^*_{k}\}_{k\in[2]}, \{\mb \Sigma_k\}_{k\in[2]})\right)^2 \\ 
  = & \min_{\mb x' \in \bb R^{a}}\big\{(\mb x' - \tilde{\mb V}\t \bo \theta_{1}^*)\t(\tilde{\mb V}\t \mb \Sigma_{1} \tilde{\mb V})^{-1} (\mb x' - \tilde{\mb V}\t \bo \theta_{1}^*): \\
   & \qquad \underbrace{(\mb x' - \tilde{\mb V}\t \bo \theta_{1}^*)\t(\tilde{\mb V}\t \mb \Sigma_{1} \tilde{\mb V})^{-1} (\mb x' - \tilde{\mb V}\t \bo \theta_{1}^*)}_{\eqqcolon f_1^{\textsf{full}}(\mb x') } = \\ 
   & \qquad \underbrace{(\mb x' - \tilde{\mb V}\t \bo \theta_{2}^*)\t(\tilde{\mb V}\t \mb \Sigma_{2} \tilde{\mb V})^{-1} (\mb x' - \tilde{\mb V}\t \bo \theta_{2}^*)  + \log|\tilde{\mb V}\t \mb \Sigma_{2}\tilde{\mb V}|  - \log|\tilde{\mb V}\t \mb \Sigma_{1}\tilde{\mb V}|}_{\eqqcolon f_2^{\textsf{full}}(\mb x') }\big\} \\ 
   = & \big(\snrfull(\{\tilde{\mb V}\t \bo \theta^*_{k}\}_{k\in[2]}, \{\tilde{\mb V}\t\mb \Sigma_k\tilde{\mb V}\}_{k\in[2]})\big)^2\label{eq: SNR' reduction},
}
where we employ the change of variables $\mb x' = \tilde{\mb V}^{\top}\mb x + \tilde{\mb V}\t \boldsymbol \theta_{1,n}^*$ and $\mb y = \tilde{\mb V}_{\perp}^{\top}\mb x$ for $\mb x\in \bb R^p$ in the expression of \eqref{eq: SNR' reduction 0} and the fact that the minimizer over all possible $\mb y\in \bb R^{p-a}$ is always the zero vector.

To lower bound $\snrfull$ using $\mathsf{SNR}^{\mod}$, we turn to simplify the expression of ${\mathsf{SNR}^{\mod}}^2$ in the same way: 
\begin{align}
    {\mathsf{SNR}^{\mathsf{mod}}}^2 = & \min_{\mb x \in \bb R^{2}}\big\{(\mb x - \mb w_{1}^*)\t{\mb S_1^{\mod}}^{-1} (\mb x - \mb w_{1}^*):  \\
    & \quad  \underbrace{(\mb x - \mb w_{1}^*)\t(\mb S_1^{\mod})^{-1} (\mb x - \mb w_{1}^*)}_{f_1(\mb x)} =  \underbrace{(\mb x - \mb w_{2}^*)\t(\mb S_2^{\mod})^{-1} (\mb x - \mb w_{2}^*)}_{f_2(\mb x)} \big\}, \label{eq: SNR reduction}
\end{align}
where $f_1$, $f_2$ are introduced for the comparison to $f_1^{\mathsf{full}}$, $f_2^{\mathsf{full}}$, respectively. 

Recap that $\mb S_k^{\mod} = {\mb V^{*\top}} \mb \Sigma_k \mb V^* \in\bb R^{2\times 2}$ and $\tilde{\mb V}$ consists of  the first $a$ canonical basis vectors of $\bb R^p$ as columns. By basic algebra, we have
$$\big(\tilde{\mb V}_{\perp}\t \mb \Sigma_k \tilde{\mb V}_{\perp}\big)^{-1}=\big((\mb \Sigma_k)_{1:a,1:a}\big)^{-1} = \left(\begin{matrix}
    (\mb S_k^{\mod})^{-1} + \mb B_n\t \mb D_n^{-1}\mb B_n &\quad \mb B_n\t  \\ 
    \mb B_n &\quad \mb D_n
\end{matrix} \right),$$
for some suitably defined matrices $\mb B_n$ and $\mb D_n$.
For each $k\in[2]$ and arbitrary $\mb y \in \bb R^{a} = (\mb y_1\t , \mb y_2\t )\t $ where $\mb y_1$ denotes the first two entries of $\mb y$ and $\mb y_2$ denotes the remaining entries, we have
\begin{align}
    & f_1^{\mathsf{full}}(\mb y) = (\mb y - \tilde{\mb V}\t \bo \theta_{1}^*)\t(\tilde{\mb V}\t \mb \Sigma_{1} \tilde{\mb V})^{-1} (\mb y - \tilde{\mb V}\t \bo \theta_{1}^*) \\ 
    = & (\mb y_1 - \mb w_{1}^*)\t\big(({\mb S^*_{1}})^{-1}+ \mb B_n\t \mb D_n^{-1}\mb B_n 
    \big)(\mb y_1 - \mb w_{1}^*) - 2 \mb y_2\t \mb B_n (\mb y_1 - {\mb V^*}\t \bo \theta_{1}^*) \\
    &\qquad + \mb y_2 \t \mb D_n\mb y_2 \\
    \geq & (\mb y_1 - \mb w_{1}^*)\t({\mb S^*_{1}})^{-1}(\mb y_1 - \mb w_{1}^*)=  f_1(\mb y_1), \\
    & f_2^{\mathsf{full}}(\mb y) = (\mb y - \tilde{\mb V}\t \bo \theta_{2}^*)\t(\tilde{\mb V}\t \mb \Sigma_{2} \tilde{\mb V})^{-1} (\mb y - \tilde{\mb V}\t \bo \theta_{2}^*)+ \log|\tilde{\mb V}\t \mb \Sigma_{2}\tilde{\mb V}|  - \log|\tilde{\mb V}\t \mb \Sigma_{1}\tilde{\mb V}|
    \\ 
    = &(\mb y_1 - \mb w_{2}^*)\t\big(({\mb S^*_{2}})^{-1}+ \mb B_n\t \mb D_n^{-1}\mb B_n 
    \big)(\mb y_1 - \mb w_{2}^*) - 2 \mb y_2\t \mb B_n (\mb y_1 - {\mb V^*}\t \bo \theta_{2}^*) \\
    &\qquad + \mb y_2 \t \mb D_n\mb y_2 + \log|\tilde{\mb V}\t \mb \Sigma_{2}\tilde{\mb V}|  - \log|\tilde{\mb V}\t \mb \Sigma_{1}\tilde{\mb V}| \\
    \geq & (\mb y_1 - \mb w_{2}^*)\t({\mb S^*_{2}})^{-1}(\mb y_1 - \mb w_{2}^*) + \log|\tilde{\mb V}\t \mb \Sigma_{2}\tilde{\mb V}|  - \log|\tilde{\mb V}\t \mb \Sigma_{1}\tilde{\mb V}| \\ 
    =& f_2(\mb y_1) + \log|\tilde{\mb V}\t \mb \Sigma_{2}\tilde{\mb V}|  - \log|\tilde{\mb V}\t \mb \Sigma_{1}\tilde{\mb V}|, \label{eq: f_1 / f_2 comparison}
\end{align}
where the inequalities are obtained by taking the minimization with respect to $\mb y_2$.

By \eqref{eq: SNR' reduction} and \eqref{eq: SNR reduction}, $\snrfull$ is defined by taking the minimum of $f_1^{\mathsf{full}}$ over all possible $\mb y \in \bb R^a$ with $f_1^{\mathsf{full}} = f_2^{\mathsf{full}}$, while $\mathsf{SNR}$ is defined in a similar way. 
The inequality \eqref{eq: f_1 / f_2 comparison} then leads to the conclusion that
\begin{align}
    & {\snrfull}^2 \geq {\mathsf{SNR}^{\mathsf{mod}}}^2 - \big|\log|\tilde{\mb V}\t \mb \Sigma_{2}\tilde{\mb V}| - \log|\tilde{\mb V}\t \mb \Sigma_{1}\tilde{\mb V}|\big|. 
\end{align}
This completes the proof of \eqref{eq: SNRfull and SNR}.

Now we set out to prove \eqref{eq: inhomo-cov rbayes}. 
Similar to the reduction in \eqref{eq: SNR' reduction}, the likelihood-ratio test is reduced to 
\longeq{
    & \tilde z(\mb y) = \ind\big\{\big(\mb y - \bo \theta_{1}^*\big)\t \tilde{\mb V}\big(\tilde{\mb V}\t  \mb \Sigma_{1} \tilde{\mb V} \big)^{-1} \tilde{\mb V}\t \big(\mb y - \bo \theta_{1}^*\big) + \log|\tilde{\mb V}\t \mb \Sigma_{1} \tilde{\mb V}| \leq \\ 
    & \qquad  \big(\mb y - \bo \theta_{2}^*\big)\t \tilde{\mb V} (\tilde{\mb V}\t \mb \Sigma_{2} \tilde{\mb V})^{-1} \tilde{\mb V}\t \big(\mb y - \bo \theta_{2}^*\big) + \log|\tilde{\mb V}\t \mb \Sigma_{2} \tilde{\mb V}|\big\}\\ 
    & + 2 \cdot \ind\big\{\big(\mb y - \bo \theta_{1}^*\big)\t \tilde{\mb V} (\tilde{\mb V}\t  \mb \Sigma_{1} \tilde{\mb V}\t )^{-1} \tilde{\mb V}\big(\mb y - \bo \theta_{1}^*\big) + \log|\tilde{\mb V}\t \mb \Sigma_{1} \tilde{\mb V}|\leq\\ 
    & \qquad \big(\mb y - \bo \theta_{2}^*\big)\t \tilde{\mb V}(\tilde{\mb V}\t  \mb \Sigma_{2} \tilde{\mb V}\t )^{-1} \tilde{\mb V}\big(\mb y - \bo \theta_{2}^*\big) + \log|\tilde{\mb V}\t \mb \Sigma_{2} \tilde{\mb V}|\big\}.
}
So, the Bayesian oracle risk $\rbayes(\{\bo \theta^*_{k}\}_{k\in[2]}, \{\mb \Sigma_k\}_{k\in[2]})$ and $\snrfull(\{\bo \theta^*_{k}\}_{k\in[2]}, \{\mb \Sigma_k\}_{k\in[2]})$ are equivalent to the $\rbayes$ and the $\snrfull$ of two $a$-dimensional Gaussian components $\mc N(\tilde{\mb V}\t \bo \theta_{1}^*, \tilde{\mb V}\t \mb \Sigma_{1}\tilde{\mb V})$ and $\mc N(\tilde{\mb V}\t \bo \theta_{2}^*, \tilde{\mb V}\t \mb \Sigma_{2}\tilde{\mb V})$, respectively. Recall that $a$ is a fixed integer not less than $2$. For a fixed-dimensional anisotropic Gaussian mixture model, \cite[Lemma 3.1]{chen2024optimal} implies that
\begin{align}
    & \rbayes(\{\bo \theta^*_{k}\}_{k\in[2]}, ~\{\mb \Sigma_k\}_{k\in[2]}) \\ 
    = &  \rbayes(\{\tilde{\mb V}\t \bo \theta^*_{k}\}_{k\in[2]},~ \{\tilde{\mb V}\t\mb \Sigma_k\tilde{\mb V}\}_{k\in[2]})\label{eq: equivalence between two mixtures}
    \\ 
    \stackrel{\text{\cite[Lemma 3.1]{chen2024optimal}}}{\geq} &  \exp\Big(-(1 + o(1) )\frac{{\snrfull}^2}{2}\Big).
\end{align}

On the other hand, the minimum of the weighted distances from the centers to the decision boundary in the definition of $\snrfull$ yields that 
\begin{align}
    & \bb P_{\mb y\sim \mc N(\tilde{\mb V}\t \bo \theta_{1}^*, \tilde{\mb V}\t \mb \Sigma_{1}\tilde{\mb V})}\big[\tilde z(\mb y) = 2 \big] \leq \bb P_{\bo \epsilon\sim \mc N(\mb 0, \tilde{\mb V}\t \mb \Sigma_{1}\tilde{\mb V})}\big[\norm{\bo \epsilon}_2 \geq \snrfull \big], \\ 
    & \bb P_{\mb y\sim \mc N(\tilde{\mb V}\t \bo \theta_{2}^*, \tilde{\mb V}\t \mb \Sigma_{2}\tilde{\mb V})}\big[\tilde z(\mb y) = 1 \big] \leq \bb P_{\bo \epsilon \sim \mc N(\mb 0, \tilde{\mb V}\t \mb \Sigma_{2}\tilde{\mb V})}\big[\norm{\bo \epsilon}_2 \geq \snrfull \big], \label{eq: likelihood-ratio upper bound}
\end{align}
where we recap that $\tilde z: \bb R^p \rightarrow [2]$ denotes the likelihood ratio estimator introduced in \eqref{eq: likelihood ratio estimator}.
Therefore, invoking the Hanson-Wright inequality \cite[Theorem 6.2.1]{vershynin2018high} together with \eqref{eq: equivalence between two mixtures} and \eqref{eq: likelihood-ratio upper bound} yields that 
\begin{align}
    &  \rbayes(\{\bo \theta^*_{k}\}_{k\in[2]}, \{\mb \Sigma_k\}_{k\in[2]}) \leq \exp\Big(-(1+o(1)) \frac{{\mathsf{SNR}^{\mathsf{full}}}^2}{2}\Big)
\end{align}
since $a$ is a fixed constant, $\mathsf{SNR} \rightarrow \infty$, and $\snrfull \gtrsim \mathsf{SNR}$. Therefore, we obtain the desired conclusion $\rbayes(\{\bo \theta^*_{k}\}_{k\in[2]}, \{\mb \Sigma_k\}_{k\in[2]}) = \exp\big(-(1+o(1)) \frac{{\mathsf{SNR}^{\mathsf{full}}}^2}{2}\big)$.

\subsection{Proof of Theorem~\ref{theorem: gaussian lower bound}}
\label{sec: proof of gaussian lower bound}
We now present a more general version of Theorem~\ref{theorem: gaussian lower bound} that permits flexibility in the choice of $\tilde\sigma$ and $\mb S_k^{\mod}$. In fact, Theorem~\ref{theorem: gaussian lower bound} will follow as an immediate corollary of the following one. 
\begin{theorem*}[Minimax Lower Bound for Two-component Gaussian Mixtures]\label{theorem: gaussian lower bound (general)}
	Consider the two-component Gaussian mixture model and the parameter space $\mb \Theta_\alpha = \mb \Theta_\alpha(n,p,\tilde\sigma$, $\mb S_1^{\mod}, \mb S_2^{\mod}, \mathsf{SNR}^{\mod}_0, \beta)$ with a fixed $\alpha > 1$.
 Then given $\snr^{\mod}_0 \rightarrow \infty$ and $\frac{\log \beta}{{\mathsf{SNR}^{\mod}_0}^2}\rightarrow 0$, one has 
	\eq{
	\inf_{\hat{\mb z}}\sup_{(\mb z^*, \bo \eta) \in \mb\Theta_\alpha}\bb E[h(\hat{\mb z}, \mb z^*)] \geq \exp\Big(-(1 + o(1))\frac{{\mathsf{SNR}^{\mod}_0}^2}{2}\Big), 
	}
    if $\tilde \sigma = \omega(\max_{k\in[2]}\norm{\mb S_k^{\mod}}^{
    \frac12})$, $\max_{k\in[2]}\norm{\mb S_k^{\mod}} / \min_{k\in[2]}\sigmamin(\mb S_k^{\mod})= O(1)$, $\log(\tilde\sigma^2 / \max_{k\in[2]}\norm{\mb S_k^{\mod}}) = o({\mathsf{SNR}_0^{\mod}}^2)$, and $n\tilde\sigma^{2(1+\epsilon)} = o(p \max_{k\in[2]}\norm{\mb S_k}^{1+\epsilon})$ for some constant $\epsilon > 0$.
\end{theorem*}

The proof consists of three main steps, detailed in Sections~\ref{subsubsec: step 1}, \ref{subsubsec: step 2}, and \ref{subsubsec: step 3}. Once these steps are established, the proof is concluded in Section~\ref{subsubsec: wrap up the proof of lower bound}.

\subsubsection{Step 1: Reduction to a Subset of \texorpdfstring{$\mb \Theta_z$}{Thetaz}}
\label{subsubsec: step 1}
The first step is to reduce the Hamming distance under all possible permutations over $[K]$ to that under a deterministic one, which is in the same spirit as the proof of Theorem~1 in \cite{gao2018community}.
For an arbitrary fixed $\mb z^{(0)} \in \mb \Theta_z$, define $\mc I_k(\mb z^{(0)}) = \{i\in[n]: z_i^{(0)} = k\}$,
then we can choose a subset $\mc B_k\subset \mc I_k(\mb z^{(0)})$ such that $|\mc B_k| =  |\mc I_k(\mb z^{(0)})| - \lfloor \frac{n}{8\beta } \rfloor$. We denote $\mc B = \mc B_1 \cup \mc B_2$. Then we define a subset $\mb Z_{\mc B}$ of $\mb \Theta_z$ which remains consistent with $\mb z^{(0)}$ at the locations of $\mc B$, i.e., $\mb Z_{\mc B} = \{\mb z \in \mb \Theta_z:~ z_i = z_i^{(0)} ~\forall i \in \mc B\}$. Therefore, for any two $\mb z^{(1)} \neq \mb z^{(2)}\in \mb Z_{\mc B}$, we have 
\eq{
 \frac{1}{n} \sum_{i=1}^n  \ind{\{z_i^{(1)} \neq  z_i^{(2)} \}} \leq \frac{n - |\mc B|}{n} \leq \frac{1}{4\beta }. 
}
However, for $\pi\in \Pi_2$ with $\pi(1) = 2$, $\pi(2) = 1$, one has 
\begin{align}
    & 
    \frac{1}{n}\sum_{i=1}^n\ind\{\pi(z_i^{(1)}) \neq \pi(z_i^{(2)})\} \geq \frac{1}{2\beta } - \frac{1}{n}\lfloor \frac{n}{8\beta}\rfloor \geq \frac{1}{4\beta},
\end{align}
which implies that 
\eq{
h(\mb z^{(1)}, \mb z^{(2)} ) = \frac{1}{n} \sum_{i=1}^n  \ind{\{z_i^{(1)} \neq  z_i^{(2)} \}}. 
}

Recall that $\mb \Theta_\alpha =\bo\Theta_z \times  \tilde{\bo\Theta}_\alpha $, where $\tilde{\mb \Theta}_\alpha$ denotes the parameter space for the continuous parameters $(\bo \theta_1^*, \bo \theta_2^*, \mb \Sigma_1, \mb \Sigma_2)$ and $\bo\Theta_z$ denotes the parameter space for the cluster label vectors. In the following, the expectation $\bb E$ and the probability measure $\bb P$ are taken with respect to the Gaussian mixture model uniquely determined by the parameter set $ (\mb z^*, \bo \theta_1^*, \bo \theta_2^*, \mb \Sigma_1, \mb \Sigma_2)$.
Setting a uniform prior on $\mb Z_{\mc B}\subset \mb \Theta_z$, we deduce that 
    \begin{align}
        &\inf_{\hat{\mb z}} \sup_{ (\mb z^*, \{\bo \theta_k^*\}_{k\in[2]}, \{\mb \Sigma_k\}_{k\in[2]}) \in \mb \Theta_\alpha} \bb E h(\hat {\mb z}, \mb z^*)  \\
    \geq &\inf_{\hat{\mb z}} \sup_{ (\mb z^*, \{\bo \theta_k^*\}_{k\in[2]}, \{\mb \Sigma_k\}_{k\in[2]}) \in \mb \Theta_\alpha} \Big[\bb E\big[ h(\hat {\mb z}, \mb z^*)\big]  -  \frac{1}{4\beta}  \big(\bb P_{\mb y\sim \mc N(\bo \theta_1^*, \mb \Sigma_1)}[\tilde z(\mb y) = 2] + \bb P_{\mb y\sim \mathcal N(\bo \theta_2^*, \mb \Sigma_2)}[\tilde z(\mb y) = 1] \big) \Big]\\ 
    = & \inf_{\hat{\mb z}} \sup_{(\bo\theta^*_1, \bo \theta^*_2, \mb \Sigma_1, \mb \Sigma_2) \in \tilde{\mb \Theta}_\alpha} \sup_{\mb z^* \in \mb \Theta_z}\Big[\bb E \big[h(\hat {\mb z}, \mb z^*) \big]  -  \frac{1}{4\beta} \big(\bb P_{\mb y\sim \mc N(\bo \theta_1^*, \mb \Sigma_1)}[\tilde z(\mb y) = 2] + \bb P_{\mb y\sim \mathcal N(\bo \theta_2^*, \mb \Sigma_2)}[\tilde z(\mb y) = 1] \big) \Big]\\
    \geq &\inf_{\hat{\mb z}} \sup_{(\bo\theta^*_1, \bo \theta^*_2, \mb \Sigma_1, \mb \Sigma_2) \in \tilde{\mb \Theta}_\alpha} \frac{1}{|\mb Z_{\mc B}|}\\ 
    & \quad \cdot \sum_{\mb z^*\in \mb Z_{\mc B}} \Big[ \frac{1}{n} \sum_{i \in \mc B^\complement} \bb P[\hat z_i \neq z_i^*] -  \frac{1}{4\beta}  \big(\bb P_{\mb y\sim \mc N(\bo \theta_1^*, \mb \Sigma_1)}[\tilde z(\mb y) = 2] + \bb P_{\mb y\sim \mathcal N(\bo \theta_2^*, \mb \Sigma_2)}[\tilde z(\mb y) = 1] \big)\Big]  \\ 
    \geq &\frac{1}{4\beta }\inf_{\hat{\mb z}} \sup_{(\bo\theta^*_1, \bo \theta^*_2, \mb \Sigma_1, \mb \Sigma_2) \in \tilde{\mb \Theta}_\alpha} \frac{1}{|\mb Z_{\mc B}|}\\ 
    & \quad \cdot \sum_{\mb z^*\in \mb Z_{\mc B}} \Big[ \frac{1}{|\mc  B^{\complement}|}\sum_{i \in \mc B^\complement} \bb P[\hat z_i \neq z_i^*]  -  \big(\bb P_{\mb y\sim \mc N(\bo \theta_1^*, \mb \Sigma_1)}[\tilde z(\mb y) = 2] + \bb P_{\mb y\sim \mathcal N(\bo \theta_2^*, \mb \Sigma_2)}[\tilde z(\mb y) = 1] \big) \Big] 
    \label{eq: reduction to subset of theta}
\end{align}
since $|\mc{B}^\complement | \leq n/(4\beta )$.

\subsubsection{Step 2: Reduction to the Local Minimax Rate} 
\label{subsubsec: step 2}
This step aims to reduce the global discrepancy appearing in \eqref{eq: reduction to subset of theta} to a local quantity, exploiting the exchangeability of the parameter space. This approach aligns with the spirit in \cite[Lemma 2.1]{zhang2016minimax} for the network stochastic block model. Note that we have fixed the permutation over different clusters in $h(\hat{\mb z}, \mb z^*)$ in Step~1, which is different from the proof in \cite[Lemma 2.1]{zhang2016minimax}. What remains to be done is to account for permutations over different rows of $\mb Y$, so as to represent the global clustering error over all samples via the misclustering probability of a single (local) sample.

Without loss of generality, we assume that $1 \notin \mc B$. Given a permutation $\pi$ on $[n]$ and an estimator $\hat{\mb z}$ based on data $\mb Y$, we define an estimator $\hat{\mb z}^\pi$ as $ \hat z^{\pi}_i(\mb Y) = (\hat{\mb z}(\mb Y^\pi))_{\pi(i)}$, $i \in[n]$, where the permuted data $\mb Y^\pi$ is defined as $\mb Y^{\pi}_{i,:} = \mb Y_{\pi^{-1}(i),:}$ for $i\in[n]$. 
Intuitively, we implement the estimator $\hat{\mb z}$ on the row-permuted data matrix $\mb Y^{\pi}$, then restore the original order of rows by applying the inverse permutation.
By introducing the above ``permuted'' version of $\hat{\mb z}$, we are able to redistribute the ``non-symmetric'' effect of $\hat{\mb z}$ across various rows while maintaining the order of the samples. For convenience, given a label vector $\mb z$ and a permutation $\pi$ over $[n]$, we also introduce a permuted label vector $\mb z_\pi$ by letting $ (\mb z_{\pi})_i= z_{\pi^{-1}(i)}$. 

Given an arbitrary $\hat{\mb z}$, the core step of the symmetrization argument lies in the randomized estimator $\hat{\mb z}^{\mathsf{sym}}$ that $\bb P[\hat{\mb z}^{\mathsf{sym}} = \hat{\mb z}^\pi| \mb Y] = {1}/({|\mc B^\complement|!})$ for each $\pi \in \Gamma_{\mc B}$,  where $\Gamma_{\mc B}$ denotes the collection of permutations on $[n] \to [n]$ that preserves indices $i\in\mc B$ but permutes those $i\in\mc B^{\complement}$. The symmetry of $\hat{\mb z}^{\mathsf{sym}}$ arises from averaging over all possible permuted estimators, canceling out any "non-symmetric" effects.

We fix arbitrary continuous parameters $(\bo \theta_1^*, \bo \theta_2^*, \mb \Sigma_1, \mb \Sigma_2) \in \tilde{\mb \Theta}_\alpha$ and denote the probability measure of $\mb Y$ corresponding to a given label $\mb z^*$ by $\bb P_{\mb z^*}$ herein. 
We make the following claim, {which will be proved at the end of this Step 2.}

\begin{claim}\label{claim: symmetrization argument}
The following holds for an arbitrary $\hat{ \mb z}$:
\eq{
    \frac{1}{|\mb Z_{\mc B}|}\sum_{\mb z^*\in \mb Z_{\mc B}} \frac{1}{|\mc  B^{\complement}|} \sum_{i \in \mc B^\complement} \bb P_{\mb z^*}[\hat z_i \neq z_i^*]  = \frac{1}{|\mb Z_{\mc B}|}\sum_{\mb z^*\in \mb Z_{\mc B}} \frac{1}{|\mc  B^{\complement}|} \sum_{i \in \mc B^\complement} \bb P_{\mb z^*}[\hat z_i^{\mathsf{sym}} \neq z_i^*]\label{eq: symmetrization argument}
}
\end{claim}

Invoking \eqref{eq: reduction to subset of theta} and Claim \ref{claim: symmetrization argument}, we first have: 
\begin{align}
    &\inf_{\hat{\mb z}} \sup_{(\mb z^*, \bo \theta_1^*, \bo \theta_2^*, \mb \Sigma_1, \mb \Sigma_2) \in \mb \Theta_\alpha} \Big(\bb E h(\hat {\mb z}, \mb z^*) -  \frac{1}{4\beta} \big(\bb P_{\mb y\sim \mc N(\bo \theta_1^*, \mb \Sigma_1)}[\tilde z(\mb y) = 2] + \bb P_{\mb y\sim \mathcal N(\bo \theta_2^*, \mb \Sigma_2)}[\tilde z(\mb y) = 1] \big) \Big)\\
   =& \frac{1}{4\beta } \inf_{\hat{\mb z}} \sup_{(\bo \theta_1^*, \bo \theta_2^*, \mb \Sigma_1, \mb \Sigma_2) \in \tilde{\mb \Theta}} \Big(\frac{1}{|\mb Z_{\mc B}|}\sum_{z^*\in \mb Z_{\mc B}} \frac{1}{|\mc  B^{\complement}|} \sum_{i \in \mc B^\complement} \bb P_{\mb z^*}[\hat z^{\mathsf{sym}}_i \neq z_i^*]  - \big(\bb P_{\mb y\sim \mc N(\bo \theta_1^*, \mb \Sigma_1)}[\tilde z(\mb y) = 2] + \bb P_{\mb y\sim \mathcal N(\bo \theta_2^*, \mb \Sigma_2)}[\tilde z(\mb y) = 1] \big)\Big) .
\end{align}

We then denote by $\pi^{(i)}$ the permutation on $[n] \to [n]$ that exchanges $1$ with $i$. Note that $\mb z_{\pi^{(i)}} = \mb z_{(\pi^{(i)})^{-1}}$ for every label vector $\mb z$. 
One has 
\begin{align}
 &\frac{1}{4\beta} \inf_{\hat{\mb z}} \sup_{(\bo \theta_1^*, \bo \theta_2^*, \mb \Sigma_1, \mb \Sigma_2) \in \tilde{\mb \Theta}} \Big(\frac{1}{|\mb Z_{\mc B}|}\sum_{z^*\in \mb Z_{\mc B}} \frac{1}{|\mc  B^{\complement}|} \sum_{i \in \mc B^\complement} \bb P_{\mb z^*}[\hat z^{\mathsf{sym}}_i \neq z_i^*] \\
    & \qquad - \big(\bb P_{\mb y\sim \mc N(\bo \theta_1^*, \mb \Sigma_1)}[\tilde z(\mb y) = 2] + \bb P_{\mb y\sim \mathcal N(\bo \theta_2^*, \mb \Sigma_2)}[\tilde z(\mb y) = 1] \big)\Big)  \\
    = &\frac{1}{4\beta} \inf_{\hat{\mb z}}\sup_{(\bo \theta_1^*, \bo \theta_2^*, \mb \Sigma_1, \mb \Sigma_2) \in \tilde{\mb \Theta}}  \frac{1}{|\mb Z_{\mc B}|}\sum_{z^*\in \mb Z_{\mc B}}\frac{1}{|\mc  B^{\complement}|} \sum_{i \in \mc B^\complement}\Big(\bb P_{\mb z^*}[\hat z^{\mathsf{sym}}_{i} \neq (\mb z_{\pi^{(i)}}^*)_1]& \\ 
    &\qquad  -  \big(\bb P_{\mb y\sim \mc N(\bo \theta_1^*, \mb \Sigma_1)}[\tilde z(\mb y) = 2] + \bb P_{\mb y\sim \mathcal N(\bo \theta_2^*, \mb \Sigma_2)}[\tilde z(\mb y) = 1] \big)\Big). 
\end{align}

Thanks to the symmetric property of $\hat{\mb z}^{\mathsf{sym}}$, $\bb P_{\mb z^*}[\hat z_i^{\mathsf{sym}} \neq (\mb z^*_{\pi^{(i)}})_1]$ is equivalent to the misclustering probability of the first sample under a permuted label. Formally, we derive that
\longeq{
     &\frac{1}{4\beta} \inf_{\hat{\mb z}}\sup_{(\bo \theta_1^*, \bo \theta_2^*, \mb \Sigma_1, \mb \Sigma_2) \in \tilde{\mb \Theta}}  \frac{1}{|\mb Z_{\mc B}|}\sum_{z^*\in \mb Z_{\mc B}}\frac{1}{|\mc  B^{\complement}|} \sum_{i \in \mc B^\complement}\Big(\bb P_{\mb z^*}[\hat z^{\mathsf{sym}}_{i} \neq (\mb z_{\pi^{(i)}}^*)_1]  -  \big(\bb P_{\mb y\sim \mc N(\bo \theta_1^*, \mb \Sigma_1)}[\tilde z(\mb y) = 2] + \bb P_{\mb y\sim \mathcal N(\bo \theta_2^*, \mb \Sigma_2)}[\tilde z(\mb y) = 1] \big)\Big)\\ 
    \stackrel{\text{(I)}}{=} & \frac{1}{4\beta}\inf_{\hat{\mb z}} \sup_{(\bo\theta^*_1, \bo \theta^*_2, \mb \Sigma_1, \mb \Sigma_2)\in \tilde{\mb \Theta}} 
    \Big(\bb P_{*,1, {\bo \eta}}[\hat z^{\mathsf{sym}}_1 =2 ] + \bb P_{*,2, {\bo \eta}}[\hat z^{\mathsf{sym}}_1 =1]  - \big(\bb P_{\mb y\sim \mc N(\bo \theta_1^*, \mb \Sigma_1)}[\tilde z(\mb y) = 2] + \bb P_{\mb y\sim \mathcal N(\bo \theta_2^*, \mb \Sigma_2)}[\tilde z(\mb y) = 1] 
    \big)\Big), \label{eq: symmetrization argument 2}
}
    where $\bb P_{*, k,\bo \eta}$ denotes the marginal probability measure of $\mb y$  with the uniform prior measure over $\{\mb z \in \mb Z_{\mc B}: z_1 = k\}$ for $k = 1,2$ and parameters $\bo \eta = (\bo\theta_1^*,\bo\theta_2^*, \mb \Sigma_1, \mb \Sigma_2)$. The equality (I) above holds since 
    \begin{equation}
        \bb P_{\mb z^*}[\hat z^{\mathsf{sym}}_{i} \neq (\mb z_{\pi^{(i)}}^*)_1]\stackrel{\text{by symmetry}}{=} \bb P_{\mb z^*}[(\hat {\mb z}^{\mathsf{sym}})_i^{\pi^{(i)}} \neq(\mb z_{\pi^{(i)}}^*)_1] = \bb P_{\mb z^*}[\hat z^{\mathsf{sym}}_{1}(\mb Y^{\pi^{(i)}}) \neq (\mb z_{\pi^{(i)}}^*)_1] =  \bb P_{{\mb z^*_{\pi^{(i)}}}}[\hat z^{\mathsf{sym}}_{1}(\mb Y) \neq (\mb z^*_{\pi^{(i)}})_1
         ]. 
    \end{equation}

    Conditional on $\tilde{\mb Y} \coloneqq (\mb y_2, \cdots,\mb y_n)\t $, we rewrite \eqref{eq: symmetrization argument 2} as 
    \begin{align}
        & \frac{1}{4\beta}\inf_{\hat{\mb z}} \sup_{(\bo\theta^*_1, \bo \theta^*_2, \mb \Sigma_1, \mb \Sigma_2)\in \tilde{\mb \Theta}} 
    \Big(\bb P_{*,1, {\bo \eta}}[\hat z^{\mathsf{sym}}_1 =2 ] + \bb P_{*,2, {\bo \eta}}[\hat z^{\mathsf{sym}}_1 =1] \\ 
    &\qquad - \big(\bb P_{\mb y\sim \mc N(\bo \theta_1^*, \mb \Sigma_1)}[\tilde z(\mb y) = 2] + \bb P_{\mb y\sim \mathcal N(\bo \theta_2^*, \mb \Sigma_2)}[\tilde z(\mb y) = 1] 
    \big)\Big) \\ 
    = & \frac{1}{4\beta}\inf_{\hat{\mb z}} \sup_{(\bo\theta^*_1, \bo \theta^*_2, \mb \Sigma_1, \mb \Sigma_2)\in \tilde{\mb \Theta}} 
    \bb E\Big[\Big(\bb P_{*,1, {\bo \eta}}[\hat z^{\mathsf{sym}}_1 =2 | \tilde{\mb Y}] + \bb P_{*,2, {\bo \eta}}[\hat z^{\mathsf{sym}}_1 =1 | \tilde{\mb Y}] \\ 
    &\qquad - \big(\bb P_{\mb y\sim \mc N(\bo \theta_1^*, \mb \Sigma_1)}[\tilde z(\mb y) = 2] + \bb P_{\mb y\sim \mathcal N(\bo \theta_2^*, \mb \Sigma_2)}[\tilde z(\mb y) = 1] \big)\Big)\Big]. 
    \end{align}

    Combining the above steps, we finally arrive at 
    \begin{align}
    &\inf_{\hat{\mb z}} \sup_{(\mb z^*, \bo \eta) \in \mb \Theta} \big(\bb E h(\hat {\mb z}, \mb z^*) -  \frac{1}{4\beta} \big(\bb P_{\mb y\sim \mc N(\bo \theta_1^*, \mb \Sigma_1)}[\tilde z(\mb y) = 2] + \bb P_{\mb y\sim \mathcal N(\bo \theta_2^*, \mb \Sigma_2)}[\tilde z(\mb y) = 1] \big) \big)\\
        = & \frac{1}{4\beta}\inf_{\hat{\mb z}} \sup_{(\bo\theta^*_1, \bo \theta^*_2, \mb \Sigma_1, \mb \Sigma_2)\in \tilde{\mb \Theta}} 
    \bb E\Big[\Big(\bb P_{*,1, {\bo \eta}}[\hat z^{\mathsf{sym}}_1 =2 | \tilde{\mb Y}] + \bb P_{*,2, {\bo \eta}}[\hat z^{\mathsf{sym}}_1 =1 | \tilde{\mb Y}] \\ 
    &\qquad - \big(\bb P_{\mb y\sim \mc N(\bo \theta_1^*, \mb \Sigma_1)}[\tilde z(\mb y) = 2] + \bb P_{\mb y\sim \mathcal N(\bo \theta_2^*, \mb \Sigma_2)}[\tilde z(\mb y) = 1] \big)\Big)\Big].\label{eq: symmetrization argument 3}
    \end{align}

    Now we are left with proving the correctness of Claim \ref{claim: symmetrization argument}.
    \begin{proof}[Proof of Claim \ref{claim: symmetrization argument}]
    For every arbitrary $\hat{\mb z}$ and every permutation $\pi \in \Gamma_{\mc B}$, we have
    \begin{align}
        & \frac{1}{|\mc B^\complement|}\sum_{i\in \mc B^\complement}\bb P_{\mb z^*}[\hat z^\pi_i \neq z_i^*] =  \frac{1}{|\mc B^\complement|}\sum_{i\in \mc B^\complement}\bb P_{\mb z_\pi}[\hat z(\mb Y^\pi)_i \neq (\mb z^*_\pi)_i]\\ 
         =& \int \ind{\{\hat z(\mb Y^\pi)_i \neq (\mb z^*_\pi)_i\}}\mathrm d \bb P_{\mb z^*}(\mb Y)   
         \stackrel{\text{(i)}}{=} \int \ind{\{\hat z(\mb Y^\pi)_i \neq (\mb z^*_\pi)_i\}}\mathrm d \bb P_{\mb z^*_{\pi}}(\mb Y^\pi) \\ 
          = &\frac{1}{|\mc B^\complement|}\sum_{i\in \mc B^\complement}\bb P_{\mb z^*_\pi}[\hat z_i \neq (z_\pi^*)_i],
    \end{align} 
    where (i) holds 
    since $  \bb P_{\mb z^*}(\mb Y)  =\bb P_{\mb z^*_{\pi}}(\mb Y^\pi)$. It follows that
    \begin{equation}
          \frac{1}{|\mb Z_{\mc B}|}\sum_{z^*\in \mb Z_{\mc B}} \frac{1}{|\mc  B^{\complement}|} \sum_{i \in \mc B^\complement} \bb P_{\mb z^*}[\hat z^\pi_i \neq z_i^*]  = \frac{1}{|\mb Z_{\mc B}|}\sum_{z^*\in \mb Z_{\mc B}} \frac{1}{|\mc B^\complement|}\sum_{i\in \mc B^\complement}\bb P_{\mb z^*_\pi}[\hat z_i \neq (\mb z_\pi^*)_i] 
        =    \frac{1}{|\mb Z_{\mc B}|}\sum_{z^*\in \mb Z_{\mc B}} \frac{1}{|\mc B^\complement|}\sum_{i\in \mc B^\complement}\bb P_{\mb z^*}[\hat z_i \neq z^*_i] 
    \end{equation}
    which finally leads to \eqref{eq: symmetrization argument} and proves Claim~\ref{claim: symmetrization argument}.
    \end{proof}

\subsubsection{Step 3: Fano's Method}
\label{subsubsec: step 3}
The final step is an application of Fano's method to the right-hand side of \eqref{eq: symmetrization argument 3}, where the key ingredient lies in a variant of Fano's method established in \cite{azizyan2013minimax} and the specific construction of the subset.  We recall that $\tilde{\mb \Theta}_\alpha$ is defined as
\begin{align}
& \tilde{\mb \Theta}_{\alpha} \coloneqq \tilde{\mb\Theta}_\alpha(n, p, \tilde\sigma, \mb S_1^{\mod}, \mb S_2^{\mod}, \mathsf{SNR}_0^{\mod})= \Big\{( \bo \theta_1^*, \bo \theta_2^*, \mb \Sigma_1, \mb \Sigma_2):  \\
& \quad (\bo \theta_1^*, \bo \theta_2^*)=\mb V^* \mb R \text{ for some $\mb V^*\in O(p,2)$ and $\mb R \in \mathrm{GL}_2(\bb R)$}; ~~\max_{k\in[2]} \bignorm{\mb \Sigma_k} \leq \tilde\sigma^2  ; \\ 
& \quad  {\mb V^{*\top}} \mb \Sigma_k \mb V^* = \mb S_k^{\mod}, k \in [2];~~
\mathsf{SNR}^{\mathsf{mod}}(\{\bo \theta_k^*\}, \{\mb \Sigma_k\}) =   \mathsf{SNR}^{\mathsf{mod}}_0; \quad \frac{-\log(\rbayes) }{{\mathsf{SNR}_0^{\mod}}^2 /2} \geq \alpha^2\Big\}. 
\end{align}

We first show our reduction scheme in Step~3.1, then provide a sketch of our idea in the parameter subset construction. Following the sketch, we verify the required separation condition on the delicately designed parameter subset in Steps~3.2 and~3.3, and finally confirm the KL divergence condition in Step~3.4. 

\emph{Step 3.1: Reduction Scheme via Alternative Fano's Method. }
The traditional Fano's method is not directly applicable to the current problem since the form shown in \eqref{eq: symmetrization argument 3} does not possess a semi-distance. We introduce a variant of Fano's method whose spirit is parallel to Proposition 1 in \cite{azizyan2013minimax} that generalizes the semi-distance to the case of a function of the estimator and the parameters: 
\begin{lemma}\label{lemma: fano's method}
    Let $\{\bb P_j\}_{j\in [M]}$ be a collection of probability measures on $\mc D$ with $\max_{j_1\neq j_2}\mathrm{KL}(\bb P_{j_1}, \bb P_{j_2}) \leq c_0\log M$, and  $M \geq 3$ for some sufficiently small $c_0$. Given arbitrary functions $f_j:\mc D \rightarrow \bb R, j\in[M]$ satisfying that for every $\mb x \in \mc D$, $\min_{j_1\neq j_2}f_{j_1}(\mb x)+ f_{j_2}(\mb x) \geq \gamma$, then we have $\sup_{j\in[M]}\bb E_j[f_j(\mb X)] \geq c\gamma$ for some positive constant $c$. 
\end{lemma}
The proof of Lemma \ref{lemma: fano's method} is postponed to Section~\ref{sec: proof of fano's method}. Returning to our problem, define ${\bo \eta}^{(j)} = (\bo\theta^{*}_1, \bo\theta^{*}_2, \bo\Sigma^{(j)}_1, \bo\Sigma^{(j)}_2)$. We denote the submatrix $(\mb y_2,\cdots,\mb y_n)\t$ by $\tilde{\mb Y}\in \bb R^{(n-1)\times p}$ and the marginal distribution of $\tilde{\mb Y}$ under $\bb P_{*,1,{\bo \eta}^{(j)}}$ by $\bar{\bb P}_{*,{\bo \eta}^{(j)}}$ (this marginal distribution actually also coincides with the corresponding marginal distribution under $\bb P_{*,2,{\bo \eta}^{(j)}}$).
As summarized in Section~\ref{subsection: minimax lower bound for anisotropic gaussian mixtures}, we
let $L_{{\bo \eta}}(\hat{\mb z})$ be $\big(\bb P_{*,1, {\bo \eta}}[\hat z^{\mathsf{sym}}_1 =2 |\tilde{\mb Y }] + \bb P_{*,2, {\bo \eta}}[\hat z^{\mathsf{sym}}_1 =1|\tilde{\mb Y }] \big)- \big(\bb P_{\mb y\sim \mc N(\bo \theta_1^*, \mb \Sigma_1)}[\tilde z(\mb y) = 2] + \bb P_{\mb y\sim \mc N(\bo \theta_2^*, \mb \Sigma_2)}[\tilde z(\mb y)= 1] \big)$ for $\bo \eta = (\{\bo \theta_k^*\}_{k\in[2]}, \{\mb \Sigma_k\}_{k\in[2]})$, which depends on $\tilde{\mb Y}$ and $\hat {\mb z}$. 
To apply Lemma \ref{lemma: fano's method}, 
a carefully designed subset $\{{\bo \eta}^{(j)}\}_{j=0}^M$ in $\tilde{\mb \Theta}_\alpha$ is needed such that $L_{{\bo \eta}^{(j_1)}}(\hat{\mb z}) + L_{{\bo \eta}^{(j_2)}}(\hat{\mb z})$ reflects the discrepancy between the minimax rate and the Bayesian oracle lower bound. Moreover, the following proposition, whose proof is deferred to Section~\ref{subsec: proof of proposition: lower bound simplification}, offers us an approach to lower bounding this quantity.

\begin{proposition}\label{proposition: lower bound simplification}
     For an arbitrary pair of parameter ${\bo \eta}^{(j_1)} = (\bo \theta_1^{*}, \bo \theta_2^*, \mb \Sigma_1^{(j_1)}, \mb \Sigma_2^{(j_1)} ) $, ${\bo \eta}^{(j_2)} = (\bo \theta_1^*, \bo \theta_2^*, \mb \Sigma_1^{(j_2)}, \mb \Sigma_2^{(j_2)}) \in \tilde{\mb \Theta}_\alpha$ and any estimator $\hat{\mb z}$, we have
    \begin{align}
         L_{{\bo \eta}^{(j_1)}}( \hat{\mb z}) + L_{{\bo \eta}^{(j_2)}}( \hat{\mb z}) 
        \geq  & \int_{ \frac{\mathrm d \bb P_{\bo \theta_2^*, \mb \Sigma_2^{(j_1)} }}{\mathrm d \bb P_{\bo \theta_1^* , \mb \Sigma_1^{(j_1)} }}\leq \frac{1}{2},\frac{\mathrm d \bb P_{\bo \theta_1^*, \mb \Sigma_1^{(j_2)}}}{\mathrm d \bb P_{\bo \theta_2^*, \mb \Sigma_2^{(j_2)}}}\leq \frac{1}{2} } \min\{\phi_{\bo \theta_1^*, \mb \Sigma_1^{(j_1)}},~\phi_{\bo\theta_2^*, \mb \Sigma_2^{(j_2)}} \} \mathrm dx \\ 
        &  + \int_{ \frac{\mathrm d \bb P_{\bo \theta_1^*, \mb \Sigma_1^{(j_1)}}}{\mathrm d \bb P_{\bo \theta_2^*, \mb \Sigma_2^{(j_1)}}}\leq \frac{1}{2}, \frac{\mathrm d \bb P_{\bo \theta_2^*, \mb \Sigma_2^{(j_2)}}}{\mathrm d \bb P_{\bo \theta_1^*, \mb \Sigma_1^{(j_2)}}}\leq \frac{1}{2}} \min\{\phi_{\bo \theta_2^*, \mb \Sigma_2^{(j_1)}}, 
        ~\phi_{\bo \theta_1^*, \mb \Sigma_1^{(j_2)}} \} \mathrm dx.\label{eq: lower bound simplification 3}
    \end{align}
\end{proposition}

Let $\gamma = \exp\Big(-(1 + o(1))\frac{{\mathsf{SNR}^{\mathsf{mod}}}^2}{2} \Big)$. To lower bound $L_{{\bo \eta}^{(j_1)}}(\hat{\mb z}) + L_{{\bo \eta}^{(j_2)}}(\hat{\mb z})$, everything boils down to constructing a subset $\{{\bo \eta}^{(j)}\}_{j=0}^M \subset \tilde{\mb \Theta}_\alpha$ such that 
\begin{align}
    &\text{the RHS of \eqref{eq: lower bound simplification 3}} \geq \exp\Big(-(1 + o(1))\frac{{\mathsf{SNR}_0^{\mod}}^2}{2} \Big). \label{eq: lower bound simplification 4}
\end{align}
\bigskip

Regarding the inequality \eqref{eq: lower bound simplification 4}, it is clearly impossible to directly approximate the probability within the irregular regions $\Big\{ \frac{\phi_{\bo \theta^*_2, \mb \Sigma^{(j_1)}_2}}{\phi_{\bo \theta^*_1, \mb \Sigma^{(j_1)}_1}} \leq \frac{1}{2}, \frac{\phi_{\bo \theta^*_1, \mb \Sigma^{(j_2)}_1}}{\phi_{\bo \theta^*_2, \mb \Sigma^{(j_2)}_2}} \leq \frac{1}{2}\Big\}$ and  $\Big\{\frac{\phi_{\bo \theta_1^*, \mb \Sigma_1^{(j_1)}}}{\phi_{\bo \theta_2^*, \mb \Sigma_2^{(j_1)}}}\leq \frac{1}{2}, \frac{\phi_{\bo \theta_2^*, \mb \Sigma_2^{(j_2)}}}{\phi_{\bo \theta_1^*, \mb \Sigma_1^{(j_2)}}}\leq \frac{1}{2}\Big\}$. Instead of tackling these irregular regions directly, it is more practical to look for regions in regular shapes, satisfying that (i) they are contained within the integral region in the RHS of \eqref{eq: lower bound simplification 3}; (ii) the integral over this region is approximately equal to $\exp\left(-(1 + o(1))\frac{{\mathsf{SNR}_0^{\mod}}^2}{2} \right)$. These conditions are formalized as \textsc{Condition~1} and \textsc{Condtion~2} in Step~3.3. Before we dive into the intricate details of the construction, we would like to provide a high-level overview of the main idea and shed light on the necessities to meet the desired condition.  

\bigskip

\emph{An Illustrative Example of Dimension $3$. }
We get started from a case with $p = 3$ to develop some intuition of which region is critical in identifying the gap with $ \exp\left(-(1 + o(1))\frac{{\mathsf{SNR}_0^{\mod}}^2}{2} \right)$. Suppose that two possible Gaussian mixture models characterized by parameters $\{\bo \theta_k^*, \mb \Sigma_k^{(1)}\}_{k\in[2]}$ and  $\{\bo \theta_k^*, \mb \Sigma_k^{(2)}\}_{k\in[2]}$, where 
\begin{align}
    & \boldsymbol \theta_1^* = (x, 0, 0)\t , \quad \boldsymbol \theta_2^* = (0, x,0)\t , \\
    &\mb \Sigma_1^{(1)} = \mb \Sigma_2^{(1)} = \left(\begin{matrix}
    1& 0 & c \\ 
    0& 1 & -c \\ 
    c & -c &1 
\end{matrix}\right), \quad \mb \Sigma_1^{(2)} = \mb \Sigma_2^{(2)} = \left(\begin{matrix}
    1& 0 & -c \\ 
    0& 1 & c \\ 
    -c & c &1 
\end{matrix}\right)
\end{align}
with $0 < c < 1/  \sqrt{2}$. The decision boundaries for these two cases are depicted in Figure~\ref{fig: decision boundaries}. 

Letting the columns of $\mb V^*\in\bb R^{3\times 2}$ be the first two canonical bases of $\bb R^3$, it is immediate that $\mb S_k^{\mod} = {\mb V^*}\t\mb \Sigma_k^{(1)} \mb V^* =  {\mb V^*}\t\mb \Sigma_k^{(2)} \mb V^* = \mb I_2$ for $k\in[2]$, $\mb w_1^* = {\mb V^*}\t \bo \theta_1^*= (x, 0)\t $, and $\mb w_2^* = {\mb V^*}\t \bo \theta_2^* = (0,x)\t$. Then one has 
$$\mathsf{SNR}^{\mod}(\{\bo \theta_k^*\}_{k\in[2]}, \{\mb \Sigma_k^{(1)}\}_{k\in[2]})^2 = \mathsf{SNR}^{\mod}(\{\bo \theta_k^*\}_{k\in[2]}, \{\mb \Sigma_k^{(2)}\}_{k\in[2]})^2 = 2x^2.$$
Further, $(\frac{x}{\sqrt{2}}, \frac{x}{\sqrt{2}})\t $ is the minimizer of the function in the $\mathsf{SNR}$'s definition: 
\begin{align}
    & (\frac{x}{\sqrt{2}}, \frac{x}{\sqrt{2}})\t  = \argmin_{\mb y \in \bb R^3: \norm{{\mb S_1^{\mod}}^{-\frac12} (\mb y- \mb w_1^*)}_2 = \norm{{\mb S_2^{\mod}}^{-\frac12} (\mb y- \mb w_2^*)}_2 } \norm{{\mb S_1^{\mod}}^{-\frac12} (\mb y- \mb w_1^*)}_2.
\end{align}
Intuitively, after discarding the third entry of our observation, $(\frac{x}{\sqrt{2}}, \frac{x}{\sqrt{2}})\t$ is the location that aligns with the decision boundary of the Gaussian mixture model with the reduced dimension two and is most prone to misclustering. Specifically, the density function at $(\frac{x}{\sqrt{2}}, \frac{x}{\sqrt{2}})\t$ under $\mc N({\mb V^*}\t\bo \theta_k^*,{\mb V^*}\t \mb \Sigma_k^{(j)} \mb V^*)$ for all $j,k\in[2]$ has a magnitude of $\exp(-x^2) = \exp(-\frac{{\mathsf{SNR}^{\mathsf{mod}}}^2}{2})$ as $x\rightarrow \infty$. 

\begin{figure}
    \centering
    \begin{subfigure}[b]{0.3\textwidth}
        \centering
        \includegraphics[width=\textwidth]{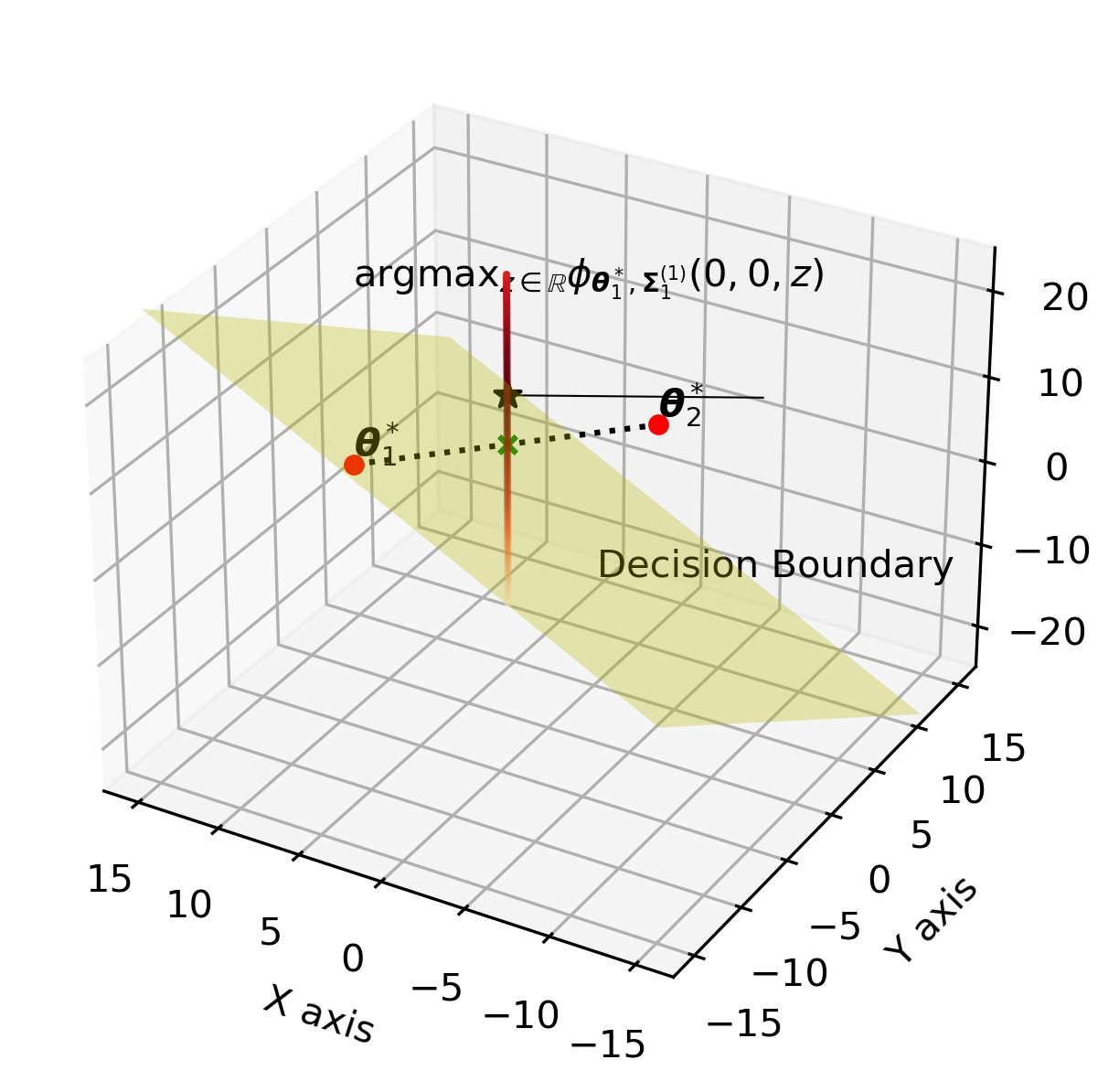}
        \caption{Case 1}
        \label{fig: boundary case 1}
    \end{subfigure}
    \begin{subfigure}[b]{0.3\textwidth} 
    \centering
    \includegraphics[width=\textwidth]{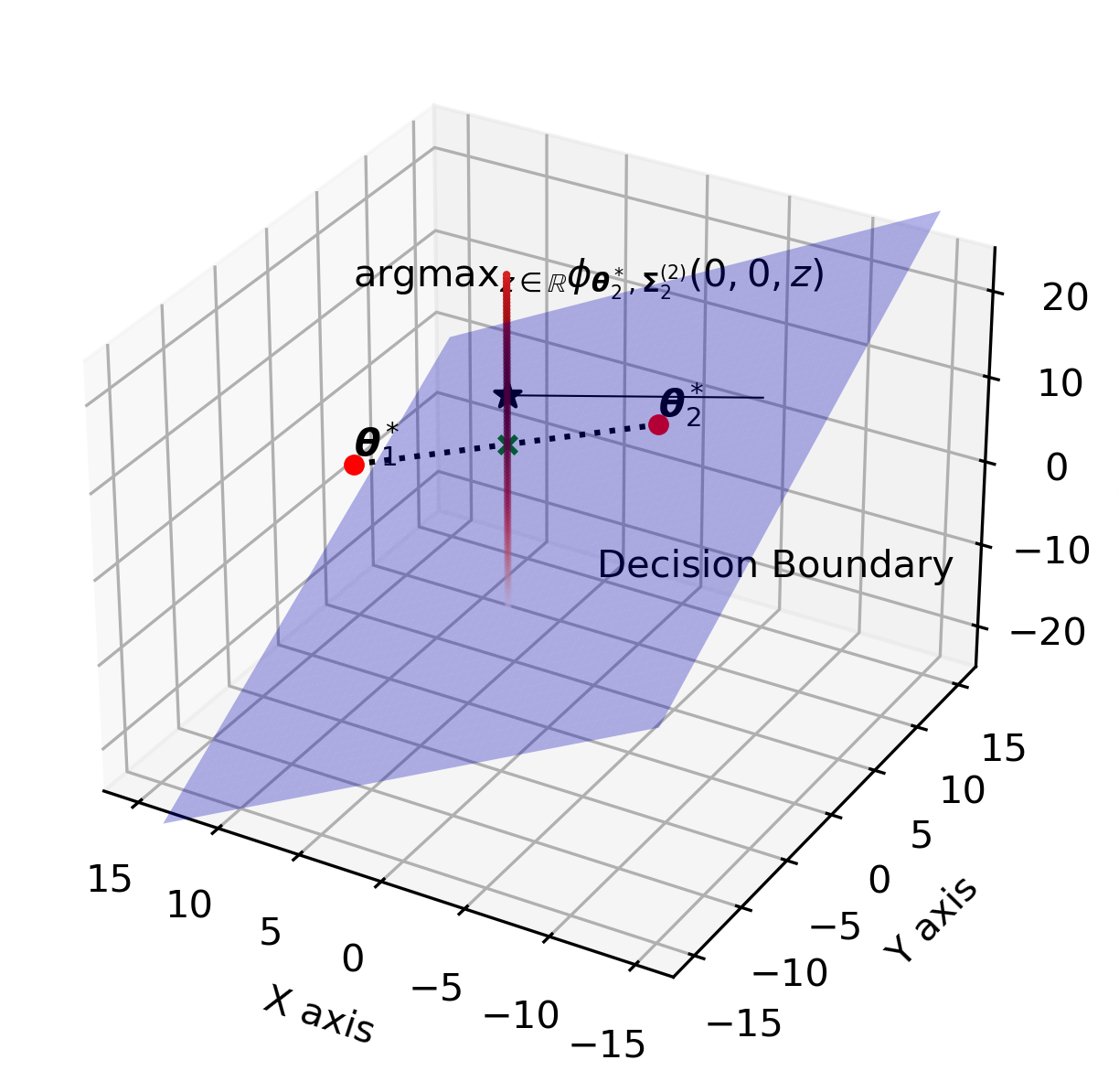}
    \caption{Case 2}
    \end{subfigure}
        \begin{subfigure}[b]{0.3\textwidth} 
    \centering
    \includegraphics[width = \textwidth]{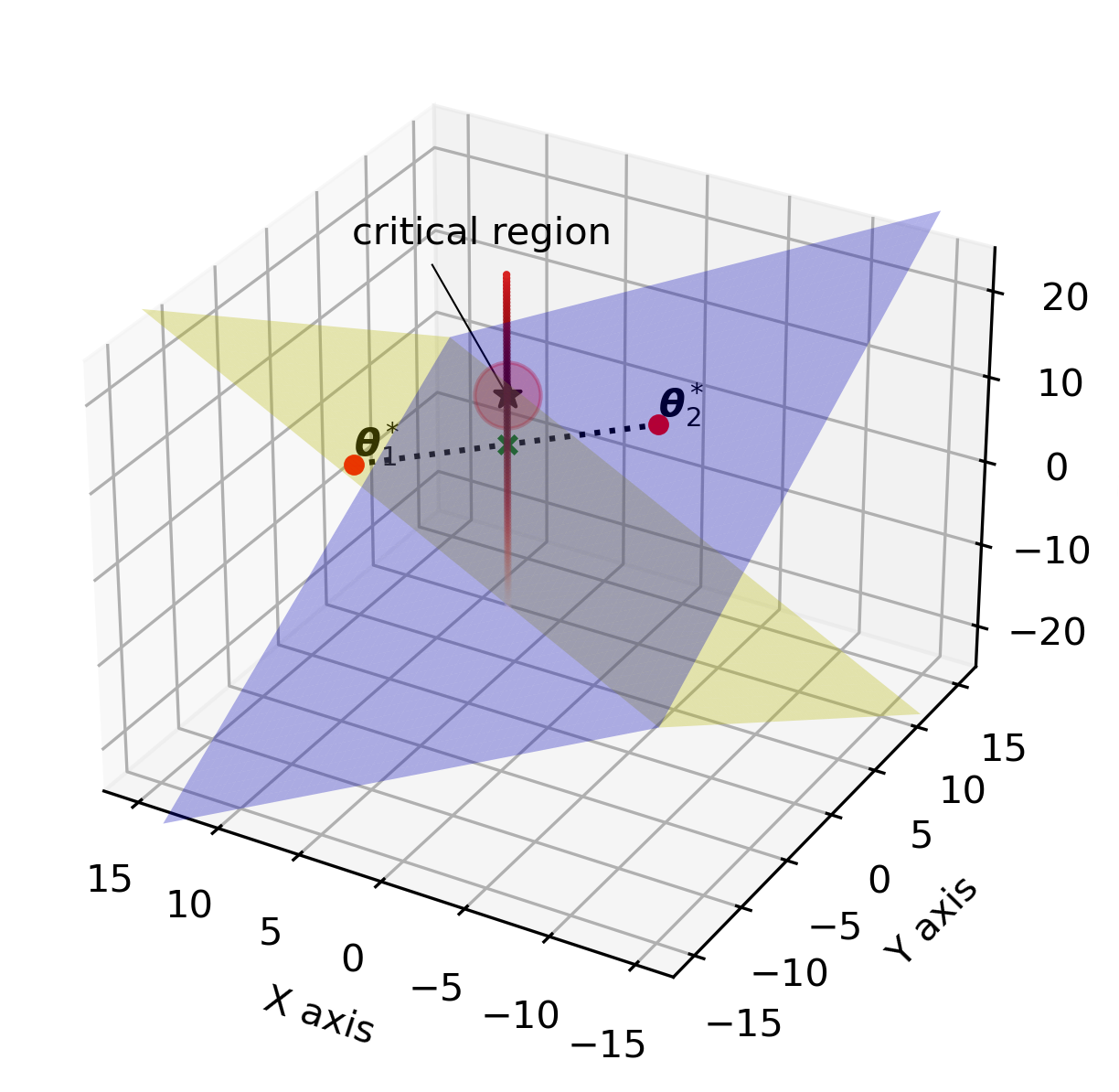}
    \caption{Critical Region}
    \label{fig: critical region}
    \end{subfigure}
    \caption{Two-Component Gaussian Mixture Example in $\bb R^3$. }
    \label{eq: decision boundaries}
\end{figure}

However, when we reversely embed $(\frac{x}{\sqrt{2}}, \frac{x}{\sqrt{2}})\t$ back into the original sample space $\bb R^3$ as $\mb V^*(\frac{x}{\sqrt{2}}, \frac{x}{\sqrt{2}})\t =(\frac{x}{\sqrt{2}}, \frac{x}{\sqrt{2}}, 0 )\t$, the density function at $(\frac{x}{\sqrt{2}}, \frac{x}{\sqrt{2}}, 0 )\t$ is written as 
$$c' \cdot \exp(- \frac{x^2}{(1 - 2c^2)}) = c' \cdot \exp(- \frac{{\mathsf{SNR}^{\mathsf{mod}}}^2}{2(1 - 2c^2)}), $$
where $c'$ is a constant related to $c$. Since we aim to identify a region where the density is at the order $\exp(-\frac{{\mathsf{SNR}^{\mathsf{mod}}}^2}{2})$, we search over the affine space perpendicular to $\mb V^*$ --specifically, along the $z$-axis -- extending from $\mb V^*(0,0)\t = (0,0,0)\t$. Basic algebra reveals that 
$$\max_{z\in \bb R}\phi_{\bo \theta_1^*, \mb \Sigma_1^{(1)}}((\frac{x}{\sqrt{2}}, \frac{x}{\sqrt{2}}, z )\t) = \max_{z\in \bb R}\phi_{\bo \theta_2^*, \mb \Sigma_2^{(2)}}((\frac{x}{\sqrt{2}}, \frac{x}{\sqrt{2}}, z )\t ) = c'\exp(-\frac{{\mathsf{SNR}^{\mathsf{mod}}}^2}{2})$$
with $z^* = cx$ being the optimizer. Note that $(x / \sqrt{2},x / \sqrt{2},z^*)$ does not align with the decision boundaries under either parameter tuple, which means each likelihood ratio estimator can confidently classify it into one cluster, as depicted in Figure~\ref{fig: critical region}. Reinterpreting the above in the context of \eqref{eq: lower bound simplification 3}, a neighborhood of $(x / \sqrt{2},x / \sqrt{2},z^*)$, the so-called critical region, will fall into the region $\Big\{ \frac{\phi_{\bo \theta^*_2, \mb \Sigma^{(1)}_2}}{\phi_{\bo \theta^*_1, \mb \Sigma^{(1)}_1}} \leq \frac{1}{2}, \frac{\phi_{\bo \theta^*_1, \mb \Sigma^{(2)}_1}}{\phi_{\bo \theta^*_2, \mb \Sigma^{(2)}_2}} \leq \frac{1}{2}\Big\}$ as $x \rightarrow \infty$; on the other hand, the quantity $\min\{\phi_{\bo \theta^*_1, \mb \Sigma_1^{(1)}} (\mb x), \phi_{\bo \theta^*_2, \mb \Sigma_2^{(2)}} (\mb x)\}$ for every $\mb x$ in the neighborhood of $(\frac{x}{\sqrt{2}}, \frac{x}{\sqrt{2}}, z^*)\t$ is of magnitude $\exp\Big(-(1 +o(1))\frac{{\mathsf{SNR}^{\mathsf{mod}}}^2}{2}\Big)$. Jointly using these two facts helps us deduce that for an arbitrary estimator $\hat{\mb z}$, 
\begin{align}
    & L_{(\{\bo\theta^*_k\}_{k\in[2]}, \{\mb \Sigma_k^{(1)} \}_{k\in[2]})}(\hat{\mb z}) + L_{(\{\bo\theta^*_k\}_{k\in[2]}, \{\mb \Sigma_k^{(2)} \}_{k\in[2]})} (\hat{\mb z})\\
    \stackrel{\text{by \eqref{eq: lower bound simplification 3}}}{\geq} & \int_{\{\frac{\phi_{\bo \theta^*_2, \mb \Sigma^{(1)}_2}}{\phi_{\bo \theta^*_1, \mb \Sigma^{(1)}_1}} \leq \frac{1}{2}, \frac{\phi_{\bo \theta^*_1, \mb \Sigma^{(2)}_1}}{\phi_{\bo \theta^*_2, \mb \Sigma^{(2)}_2}} \leq \frac{1}{2}\}} \min\{\phi_{\bo \theta^*_1, \mb \Sigma_1^{(1)}} (\mb x), \phi_{\bo \theta^*_2, \mb \Sigma_2^{(2)}} (\mb x)\} \mathrm d \mb x\\ 
    \geq &  \exp\Big(-(1 +o(1))\frac{{\mathsf{SNR}^{\mathsf{mod}}}^2}{2}\Big)
\end{align}
as $x \rightarrow \infty$, in this illustrative case. 
\bigskip

Reflecting on the above derivation in the illustrative example in $\bb R^3$, the fact that the optimizers of $\phi_{\bo \theta_1^*, \mb \Sigma_1^{(1)}}((\frac{x}{\sqrt{2}}, \frac{x}{\sqrt{2}}, z )\t )$ and $\phi_{\bo \theta_2^*, \mb \Sigma_2^{(2)}}((\frac{x}{\sqrt{2}}, \frac{x}{\sqrt{2}}, z )\t )$ coincide hinges critically on the condition ${\mb V^*}\t \mb \Sigma_1^{(1)}\mb V^*_\perp = -{\mb V^*}\t \mb \Sigma_2^{(2)}\mb V^*_\perp$, where $\mb V^*_\perp$ represents the vector $(0,0,1)\t$. However, when considering $M$ parameter tuples $\{\bo \eta^{(j)}\}_{j\in[M]}$, this condition is hard to be satisfied for each pair of parameters, even when $p >3$. 
To circumvent this issue, we shall leverage the high-dimensionality and the condition $\tilde \sigma  = \omega(  \bar \sigma)$ stated in Theorem~\ref{theorem: gaussian lower bound}. The approach is outlined as follows, continuing the discussion on $p$-dimensional Gaussian mixtures.

\emph{High-level Outline of the Parameter Construction Satisfying \eqref{eq: lower bound simplification 4}. }
Suppose that we are given two parameter tuples $\boldsymbol \eta^{(j_1)} = (\{\bo \theta_k^*\}_{k\in[2]}, \{\mb \Sigma_k^{(j_1)}\}_{k\in[2]})$ and $\boldsymbol \eta^{(j_2)} = (\{\bo \theta_k^*\}_{k\in[2]}, \{\mb \Sigma_k^{(j_2)}\}_{k\in[2]})$, whose structures will be specified as the discussion proceeds. 
We first focus on the $2$-dimensional subspace spanned by the centers and examine the minimizer in the definition of $\mathsf{SNR}$; formally,
we denote the point that reaches the minimum in the definition of $\mathsf{SNR}$ by 
\eq{
    \mb w_* \coloneqq \argmin_{\mb x \in\bb R^2}\big\{\ip{(\mb x - \mb w_1^*)\t{ \mb S_1^{\mod}}^{-1}(\mb x - \mb w_1^*)}:  \ip{(\mb x - \mb w_1^*)\t{ \mb S_1^{\mod}}^{-1}(\mb x - \mb w_1^*)} = \ip{(\mb x - \mb w_2^*)\t{ \mb S_2^{\mod}}^{-1}(\mb x - \mb w_2^*)}\big\}.\label{eq: definition of w_*}
}
We also denote its embedding in $\bb R^p$ by $\mb x_* \coloneqq \mb V^* \mb w_*$. Then the maximizer of $\phi_{\bo\theta_k^*, \mb \Sigma_k^{(j)}} (\mb x_* + \mb V^*_\perp \mb z)$ in terms of $\mb z\in \bb R^{p-2}$ for $k\in[2]$ and $j\in \{j_1, j_2\}$ is expressed as 
\begin{align}
    & \mb z^{k,(j)}_* \coloneqq  -\Big({\mb V^*_\perp}\t {\mb \Sigma_k^{(j)}}^{-1} \mb V_\perp^* \Big)^{-1} \Big({\mb V^*_\perp }\t {\mb \Sigma_k^{(j)}}^{-1} \mb V^*\Big){\mb V^*}\t \big(\mb x_* - \bo \theta_k^*\big)
\end{align}
by directly taking the first-order condition. 

We now describe the ``critical region'' in this case by examining the density functions. 
Instead of equating $\mb z_*^{1,(j_1)}$ with $\mb z_*^{2,(j_2)}$ or $\mb z_*^{2,(j_1)}$ with $\mb z_*^{1,(j_2)}$ as in the $3$-dimentional example, we shall exploit the density behavior at $\mb z_*^{1,(j_1)} + \mb z_*^{2,(j_2)}$, with the aid of some orthogonality across different parameter tuples.

Our construction proceeds as follows: on the one hand, we let ${\mb V^*_\perp }\t {\mb \Sigma_k^{(j)}}^{-1} \mb V^*$ be 
\eq{
 {\mb V^*_\perp }\t {\mb \Sigma_k^{(j)}}^{-1} \mb V^* = \text{const} \cdot \frac{1}{\underline\sigma^2 }\cdot  \mb v^{(j)} \tilde{\mb w}\t \label{eq: cov construction 1}
}
for $j \in \{j_1, j_2\}$, 
where $\mb v^{(j_1)}$, $ \mb v^{(j_2)}$ are unit vectors in $\bb R^{p-2}$, $\tilde{\mb w}$ is defined as 
\eq{\label{eq: definition of tilde w}
\tilde{\mb w}\coloneqq \frac{\mb w_2^* - \mb w_1^*}{\norm{\mb w_2^* - \mb w_1^*}_2}, 
}and $\mb v^{(j_1)}$ is ``almost orthogonal'' to $\mb v^{(j_2)}$ (see the later Step 3.2 for details); on the other hand, we let ${\mb V^*_\perp}\t {\mb \Sigma_k^{(j)}}^{-1} \mb V_\perp^*$ be \eq{
{\mb V^*_\perp}\t {\mb \Sigma_k^{(j)}}^{-1} \mb V_\perp^* = \frac{1}{\tilde\sigma^2}\big(\mb I_{p-2} - \mb v^{(j)}{\mb v^{(j)}}\t\big) +  \frac{1}{\underline\sigma^2}\cdot \mb v^{(j)} {\mb v^{(j)}}\t . \label{eq: cov construction 2}
}
Note that the value of ${\mb V^*}\t {\mb \Sigma_k^{(j)}}^{-1} \mb V^*$ has been uniquely determined by ${\mb V_\perp^*}\t {\mb \Sigma_k^{(j)}}^{-1} \mb V^*_\perp$ and ${\mb V_\perp^*}\t {\mb \Sigma_k^{(j)}}^{-1} \mb V^*$ according to the constraint ${\mb V^*}\t\mb \Sigma_k^{(j)} \mb V^*= \mb S_k^{\mod}$ for $k\in[2]$. Additionally, given $j \in \{j_1, j_2\}$, $\mb v\t \mb v^{(j)} = 0$ implies that $\mb v\t \mb \Sigma_k^{(j)} \mb v = \tilde \sigma^2$ for $k \in [2]$ by the formula of block matrix inverse. Intuitively, \(\mb \Sigma^{(j)}\) exhibits smaller variability along the directions of \(\mb V^*\) and \(\mb v^{(j)}\), while showing a larger variability, \(\tilde \sigma\), in the orthogonal directions.

Equipped with the above construction, we first notice that 
$$ 
\max_{\mb z \in \bb R^{p-2}}\phi_{\bo\theta_k^*, \mb \Sigma_k^{(j)}} (\mb x_* + \mb V^*_\perp \mb z) = \phi_{\bo\theta_k^*, \mb \Sigma_k^{(j)}} (\mb x_* + \mb V^*_\perp \mb z_*^{k,(j)}) = \frac{c}{\tilde\sigma^{p-2}}\cdot \exp(-{\mathsf{SNR}_0^{\mod}}^2 / 2)$$ for some constant $c$.  Moreover, according to the formula of block matrix inverse, $\mb z_*^{k,(j)}$ is expressed as 
\begin{equation}
    \mb z^{k,(j)}_* =   -\mathrm{const} \cdot \big(\tilde{\mb w}\t (\mb w_* - \mb w_k^*)\big) \mb v^{(j)},
\end{equation}
which aligns with the direction of $\mb v^{(j)}$. 

Moreover, it follows from the almost orthogonality ${\mb v^{(j_1)}}\t \mb v^{(j_2)} \approx 0$ that
$${\mb v^{(j_2)}}^{\top} \mb \Sigma^{(j_1)}_1 \mb v^{(j_2)}\approx{ \mb v^{(j_2)}}^{\top} \mb \Sigma^{(j_1)}_2 \mb v^{(j_2)} \approx {\mb v^{(j_1)}}^{\top} \mb \Sigma^{(j_2)}_1 \mb v^{(j_1)}\approx {\mb v^{(j_1)}}^{\top} \mb \Sigma^{(j_2)}_2 \mb v^{(j_1)} \approx \tilde\sigma^2 .  $$
Hereafter, the symbol "$\approx$" is used for intuitive illustration, with its explicit form to be clarified in the formal proof later (from Step 3.2 to Step 3.3).
Invoking the condition $\underline\sigma = o(1) \tilde\sigma$, one can tell that a translation along a direction with approximate variance $\tilde \sigma^2$ does not alter the density function much. We thus deduce that 
\begin{align}
& \phi_{\bo \theta_1^*, \mb \Sigma_1^{(j_1)}}(\mb x_* + \mb V_\perp^* \mb z_*^{1,(j_1)}) \approx \phi_{\bo \theta_1^*, \mb \Sigma_1^{(j_1)}}(\mb x_* + \mb V_\perp^* ( \mb z_*^{1,(j_1)} + \mb z_*^{2,(j_2)} )) \\ 
\approx & \phi_{\bo \theta_2^*, \mb \Sigma_2^{(j_2)}}(\mb x_* + \mb V_\perp^* ( \mb z_*^{1,(j_1)} + \mb z_*^{2,(j_2)} )) \approx \phi_{\bo \theta_2^*, \mb \Sigma_2^{(j_2)}}(\mb x_* + \mb V_\perp^* \mb z_*^{2,(j_2)})
\end{align}
according to \eqref{eq: cov construction 1} and \eqref{eq: cov construction 2}. 

Furthermore, for an orthonormal matrix $\mb V^{(j_1, j_2)}_\perp \in \bb R^{p\times( p-4)}$ whose column space is orthogonal to $\big(\mb V^*, \mb V_\perp^*(\mb v^{(j_1)}, \mb v^{(j_2)})\big)$, one has 
\begin{align}
&\phi_{\bo \theta_1^*, \mb \Sigma_1^{(j_1)}}(\mb x_* + \mb V_\perp^* \mb z_*^{1,(j_1)} + \mb V^{(j_1, j_2)}_\perp \mb z' ) 
\approx \phi_{\bo \theta_1^*, \mb \Sigma_1^{(j_1)}}(\mb x_* + \mb V_\perp^* ( \mb z_*^{1,(j_1)} + \mb z_*^{2,(j_2)} )+ \mb V^{(j_1, j_2)}_\perp \mb z' ) \\ 
\approx & \phi_{\bo \theta_2^*, \mb \Sigma_2^{(j_2)}}(\mb x_* + \mb V_\perp^* ( \mb z_*^{1,(j_1)} + \mb z_*^{2,(j_2)} )+ \mb V^{(j_1, j_2)}_\perp \mb z' ) 
\approx\phi_{\bo \theta_2^*, \mb \Sigma_2^{(j_2)}}(\mb x_* + \mb V_\perp^* \mb z_*^{2,(j_2)}+ \mb V^{(j_1, j_2)}_\perp \mb z' ) \\ 
\approx & \exp\Big(- (1 + o(1))\frac{{\mathsf{SNR}^{\mathsf{mod}}}^2}{2}\Big) \cdot \phi_{\tilde\sigma^2 \mb I_{p-4}}(\mb z'),\label{eq: pdf approximation}
\end{align}
leveraging the independence between $(\mb V^*, \mb V_\perp^*(\mb v^{(j_1)}, \mb v^{(j_2)}))$ and $\mb V_\perp^{(j_1,j_2)}$ under $\mc N(\bo \theta_k^*, \mb \Sigma_k^{(j)})$. 

Given $\mb x \in \bb R^m$ and $\rho> 0$, we let $B(\mb x, \rho)$ be $\{\mb y: \norm{\mb x - \mb y}_2 \leq \rho\}$. We also fix an orthonormal matrix $\mb V^{(j_1, j_2)} \in O(p-2, 2)$ whose column space is the one spanned by $(\mb v^{(j_1)}, \mb v^{(j_2)})$. 
Provided the above characterization of the density function, we focus on a region $R^{(j_1,j_2)}$ defined as follows: 
\begin{align}
&R^{(j_1, j_2)} \coloneqq \Big(\mb V^*,~ \mb V_\perp^*\mb V^{(j_1,j_2)}, ~\mb V^{(j_1,j_2)}_\perp \Big) \\ 
& \times \left[B\big(\mb w_*, \bar \rho_1\big)  \times B\big({\mb V^{(j_1, j_2)} }\t (\mb z_*^{1,(j_1)}+\mb z_*^{2,(j_2)}), \bar \rho_2\big) \times \bb R^{p-4}\right],
\end{align}
where $\bar \rho_1, \bar \rho_2$ are some constants representing the radius of the spherical region. 
Each point within this region is affirmatively classified into the first cluster according to the likelihood ratio estimator of $\boldsymbol \eta^{(j_1)}$ or into the second cluster according to the likelihood ratio estimator of $\bo \eta^{(j_2)}$. In other words, we can prove that 
$$
R^{(j_1,j_2)}\subseteq \Big\{ \frac{\phi_{\bo \theta^*_2, \mb \Sigma^{(j_1)}_2}}{\phi_{\bo \theta^*_1, \mb \Sigma^{(j_1)}_1}} \leq \frac{1}{2}, \ 
\frac{\phi_{\bo \theta^*_1, \mb \Sigma^{(j_2)}_1}}{\phi_{\bo \theta^*_2, \mb \Sigma^{(j_2)}_2}} \leq \frac{1}{2}\Big\}.
$$

Moreover, invoking Proposition~\ref{proposition: lower bound simplification} and the relation \eqref{eq: pdf approximation},
integrating with respect to the function 
$$\min\big\{\phi_{\bo \theta^*_1, \mb \Sigma^{(j_1)}_1}(\mb x), \ \phi_{\bo \theta^*_2, \mb \Sigma^{(j_2)}_2}(\mb x) \big\}$$
over $R^{(j_1,j_2)}$ yields the lower bound as follows: 
\begin{align}
    & L_{\bo \eta^{(j_1)}}(\hat{\mb z}) + L_{\bo \eta^{(j_2)}}(\hat{\mb z}) \geq  \pi^2 \bar \rho_1^2 \bar \rho_2^2 \exp\Big(- (1 + o(1))\frac{{\mathsf{SNR}_0^{\mod}}^2}{2}\Big) = \exp\Big(- (1 + o(1))\frac{{\mathsf{SNR}_0^{\mod}}^2}{2}\Big). 
\end{align}
holds for any estimator $\hat{\mb z}$. 

It is worth mentioning here that applying Lemma~\ref{lemma: fano's method} requires establishing a lower bound on the cardinality of $\{\bo \eta^{(j)}\}_{j=1}^M$ such that every pair in this set satisfies the above relationship. This requirement is met by leveraging the cardinality lower bound for the vectors $\mb v^{(j)}$'s involved in the construction of \eqref{eq: cov construction 1} and \eqref{eq: cov construction 2}, with the aid of high-dimensionality. 

The above construction is detailed in the following Step~3.2 and Step~3.3. Additionally, Step~3.4 addresses the control of KL divergence between two arbitrary parameters in the subset.

\bigskip

\emph{Step 3.2: Constructing the Parameter Subset. }
Here, we collectively summarize the notations used:

\begin{itemize}
    \item $\mb V^*$ is a $p$-by-$2$ orthonormal matrix representing the right singular space of $\mathbb{E}[\mb Y]$.
    \item $\mb w_k^* = {\mb V^*}^\top \bo \theta_k^*$, and $\mb S_k^{\mod} = {\mb V^*}^\top \mb \Sigma_k \mb V^*$.
    \item The minimizer in the definition of $\mathsf{SNR}^{\mod}$ is defined as:
    \begin{align*}
        \mb w_* \coloneqq \argmin_{\mb x \in \mathbb{R}^2} \Big\{ & \ip{(\mb x - \mb w_1^*)^\top { \mb S_1^{\mod}}^{-1} (\mb x - \mb w_1^*)} :  \ip{(\mb x - \mb w_1^*)^\top { \mb S_1^{\mod}}^{-1} (\mb x - \mb w_1^*)} = \ip{(\mb x - \mb w_2^*)^\top { \mb S_2^{\mod}}^{-1} (\mb x - \mb w_2^*)} \Big\}.
    \end{align*}
    \item We use $\bar \sigma$ and $\underline \sigma$ to denote the upper and lower bounds on the eigenvalues of $\mb S_k^{\mod}$, $k\in[K]$. 
\end{itemize}

\medskip 

\subparagraph*{Almost Mutually Orthogonal Vectors}
As outlined above, we first introduce a packing on a sphere $\bb S^{p-2}$ that stands for the possible correlation directions in the high-dimensional covariance matrices.  

In view of the Varshamov-Gilbert bound \cite[Lemma 4.7]{massart2007concentration}, there exists a subset $\{\tilde {\mb v}^{(j)}\}_{j=1}^M$ of $\{-1,1\}^{p-2}$ such that 
    \longeq{
    &\log M \geq \big((1 + \delta)\log(1 + \delta) + (1 - \delta)\log (1- \delta)\big) \frac{p-2 }{2}, \\ 
    &  \min\big\{\bignorm{\tilde{\mb v}^{(j_1)} + \tilde{\mb v}^{(j_2)}}_2 ,\bignorm{\tilde{\mb v}^{(j_1)} - \tilde{\mb v}^{(j_2)}}_2 \big\} \geq \sqrt{2p(1 - \delta)}\text{\quad for $j_1 \neq j_2 \in [M]$}. \label{eq: VG bound}
    }
For $\delta \in (-1, 1]$, the Taylor expansion gives that 
$$(1 + \delta)\log(1 + \delta) + (1 - \delta)\log (1- \delta) \geq \frac{\delta^2}{2}$$ 
since 
$$\big((1 + x)\log (1+ x)\big)''  = \frac{1}{1+ x}\geq  \frac{1}{2}, \text{ for }x \in (-1,1].$$ 
Letting $\delta $ be $(n^{\frac{1}{2}}\tilde\sigma^{1+ \epsilon})/({p^{\frac{1}{2}} \max_{k\in[2]}\|\mb S_k^{\mod}\|^{\frac{1+\epsilon}{2}}})$, we then have
\eq{\label{eq: log M upper bound}
     \log M \geq c n \frac{\tilde\sigma^{2(1 + \epsilon)}}{\max_{k\in [2]} \norm{\mb S_k^{\mod}}^{1 + \epsilon} }
     } for some constant $c$. At the end, we normalize $\{\tilde{\mb v}^{(j)}\}_{j=1}^M$ to be of unit norm and denote the normalized vectors by $\{\mb v^{(j)}\}_{j=1}^M\in \bb R^{p-2}$. By \eqref{eq: VG bound}, for two arbitrary $j_1\neq j_2\in[M]$ one has 
\begin{align}\label{eq: almost mutually orthogonal property}
    & |{\mb v^{(j_1)\top}} \mb v^{(j_2)}| \leq \delta \coloneqq  \frac{n^{\frac{1}{2}}\tilde\sigma^{1 + \epsilon}}{p^{\frac{1}{2}} \max_{k\in[2]}\|\mb S_k^{\mod}\|^{\frac{1}{2}(1 + \epsilon)}},
\end{align}
where the right-hand side decreases to zero in our setting. 
In other words, this subset enjoys an almost mutually orthogonal property, which plays a crucial role in constructing ${\mb V^*}\t {\mb\Sigma_k^{(j)}}^{-1} \mb V_\perp^*$. 

\medskip

\medskip
\subparagraph*{Covariance Construction}            
Equipped with the above preparation, we are ready to construct a covariance matrix subset that in a way represents the complexity of the decision problem. We start by fixing an arbitrary orthonormal matrix $\tilde {\mb V} \coloneqq (\mb V^*, \mb V^*_\perp) \in O(p)$ where $\mb V^* \in O(p,2)$ and projected centers $\mb w_1^*, \mb w^*_2$ such that 
$\mathsf{SNR}^{\mod}(\{\mb V^*\mb w_k^*\}_{k\in[2]}, \{ \mb V^*\mb S_k^{\mod}{\mb V^{*\top}} \}_{k\in[2]}) = \mathsf{SNR}^{\mod}_0$. Then 
    we define ${\bo \eta}^{(j)} =  (\bo \theta_1^{*}, \bo \theta_2^{*}, \mb \Sigma_1^{(j)}, \mb \Sigma_2^{(j)})$ as
    \longeq{
    &\bo \theta_k^{*} =  \mb V^* \mb w_k^*, \qquad \mb \Sigma_k^{(j)} = \big( \mb V^*, \mb V^*_{\perp}\big) (\mb \Omega_k^{(j)})^{-1} \big( \mb V^*, \mb V^*_{\perp}\big)\t, \label{eq: centers and covariances in the packing}
    }
    where \eq{\mb \Omega_k^{(j)} = \left( 
    \begin{matrix}
        {\mb S_k^{\mod}}^{-1} + \dfrac{{\alpha'}^2}{\underline\sigma^2} \tilde{\mb w}{\tilde{\mb w}}\t, & \dfrac{\alpha'}{\underline\sigma^2}\tilde{\mb w}{\mb v^{(j)\top}}  \\[3mm] 
        \dfrac{\alpha'}{\underline\sigma^2}\mb v^{(j)}\tilde{\mb w}\t  , & \dfrac{1}{\tilde\sigma^2}\big(\mb I_{p-2} - \mb v^{(j)}{\mb v^{(j)\top}} \big) + \dfrac{\mb v^{(j)} {\mb v^{(j)\top}} }{\underline\sigma^2} \label{eq: definition of Omega_k}
    \end{matrix}
    \right), } 
    for $k = 1,2$ with $\alpha' = 8\alpha \frac{\bar \sigma^2 }{\underline\sigma^2 }$ and $\tilde {\mb w}$ defined in \eqref{eq: definition of tilde w}. 
\medskip
    \emph{Verifying the conditions in $\tilde{\mb \Theta}_\alpha$. }
    Note that the above design ensures ${\mb V^{*\top}} \mb \Sigma^{(j)}_k \mb V^* = \mb S_k^{\mod}$ from basic linear algebra that 
    \longeq{\label{eq: inverse of block matrix}
        \begin{pmatrix}
        \mb A &  \quad \mb B \\ 
        \mb C &  \quad \mb D \end{pmatrix}^{-1} = 
        \begin{pmatrix}
            (\mb A - \mb B\mb D^{-1} \mb C)^{-1} & \quad * \\ 
            * & \quad *
        \end{pmatrix}. }
        Moreover, $\mb \Sigma_k^{(j)}$ is positive definite by the fact that $\mb y\t \mb \Omega_k^{(j)} \mb y \geq 0$ for all $\mb y \in \bb R^{p-2}$. Furthermore, for a sufficiently large $n$, the eigenvalues of $\mb \Sigma_k^{(j)}$ are upper bounded by $\tilde \sigma^2$  and the eigenvalues of $\mb \Omega_k^{(j)}$ are lower bounded by ${1}/{\tilde\sigma^2}$ since $\bar \sigma = o(1 )\tilde \sigma$. 

To verify that ${\bo \eta}^{(j)}$ is contained in $\tilde{\mb \Theta}_\alpha$ for each $j\in[M]$, we are left to show that
$ \frac{-\log(\rbayes)}{{\mathsf{SNR}^{\mod}_0}^2 /2} \geq \alpha^2 $ holds for $\bo \eta^{(j)}$. Notice that applying Proposition~\ref{proposition: SNR' and Bayesian oracle risk 2} to $\{\bo \eta^{(j)}\}_{j\in[M]}$ yields that $\log\big(\rbayes(\bo \eta^{(j)})\big) = - (1 + o(1))\frac{\snrfull(\bo \eta^{(j)})^2}{2}$. Thus it suffices for show that the $\snrfull$ of ${\bo \eta}^{(j)}$ is greater than or equal to $(1 + \delta) \alpha \mathsf{SNR}_0^{\mod}$ for all $j\in[M]$ for some $\delta>0$, which is stated in the following claim. 
\begin{claim}
    \label{claim: snrfull}
    With $\alpha' =\frac{ 12\alpha \bar\sigma^2}{\underline\sigma^2}$, we have  $\snrfull(\{\bo \theta_k^*\}_{k\in[2]}, \{\mb \Sigma_k^{(j)} \}_{k\in[2]})\geq 2\alpha \mathsf{SNR}^{\mod}_0$ for every $j \in [M]$. 
\end{claim}

\begin{proof}[Proof of Claim~\ref{claim: snrfull}]
    We start with an observation that $\snrfull(\{\bo \theta_k^*\}_{k\in[2]}, \{\mb \Sigma_k\}_{k\in[2]})$ is the same as the $\snrfull$ of the  Gaussian mixture 
    $$\mc N\Big((\mb V^*, \mb V_\perp^*\mb v^{(j)})\t \bo \theta_k^*, \quad(\mb V^*, \mb V_\perp^*\mb v^{(j)})\t\mb \Sigma_k (\mb V^*, \mb V_\perp^*\mb v^{(j)})\Big), \qquad k = \{1,2\} $$ of dimension $3$ for all $j\in[M]$. For ease of notation, we denote that 
\begin{align}
    & \mb w_k^{*,(j)} \coloneqq (\mb V^*, \mb V_\perp^*\mb v^{(j)})\t \bo \theta_k^*, \\ 
    &  \mb S_k^{\mod,(j)} \coloneqq (\mb V^*, \mb V_\perp^*\mb v^{(j)})\t\mb \Sigma_k (\mb V^*, \mb V_\perp^*\mb v^{(j)})\\
    = & \left(\begin{matrix}
       { \mb S_k^{\mod}}^{-1} + \dfrac{{\alpha'}^2}{\underline\sigma^2} \tilde{\mb w}{\tilde{\mb w}}\t, & \dfrac{\alpha'}{\underline\sigma^2}\tilde{\mb w}  \\[3mm] 
        \dfrac{\alpha'}{\underline\sigma^2}\tilde{\mb w}\t  , & \dfrac{1}{\underline\sigma^2}
    \end{matrix}\right)^{-1} \in \bb R^{3\times 3}.\label{eq: centers and covariances of equivalent Gaussian mixture}
    \end{align}
And we write the inverse of $\mb S_k^{\mod,(j)}$ as 
\begin{align}
    & (\mb S_k^{\mod,(j)})^{-1} = \left(\begin{matrix}
        {\mb S_k^{\mod}}^{-1} &~ \mb 0_{2\times 1} \\ 
         \mb 0_{1\times 2}   &~ 0
     \end{matrix}\right) + \frac{1}{\underline\sigma^2 }(\alpha'\tilde{\mb w}\t, 1)\t(\alpha'\tilde{\mb w}\t, 1).
\end{align} 
A consequence of the above decomposition is that 
\begin{align}
    &\bignorm{{\mb S_1^{\mod, (j)}}^{-1} - {\mb S_2^{\mod, (j)}}^{-1}}_2 \leq \frac{2}{\underline\sigma^2}, \label{eq: S_k^j diff} \\ 
    & \min_{k\in[2]}\sigma_{\min}({\mb S_k^{\mod,(j)}}^{-1}) \geq \frac{1}{\bar\sigma^2}.  \label{eq: lower bound on sigma min}
\end{align}

To justify the relation $\snrfull \geq 4\alpha^2 \mathsf{SNR}_0^{\mod}$ for some $\delta>0$, we turn to showing that for every $\mb x \in \bb R^3$ such that $(\mb x - \mb w_1^{*,(j)})\t {\mb S_1^{\mod, (j)}}^{-1}(\mb x - \mb w_1^{*,(j)}) \leq 4\alpha^2{\mathsf{SNR}^{\mod}_0}^2 $ the equality in the definition of $\snrfull$ is not satisfied. Firstly, we notice that $(\mb x - \mb w_1^{*,(j)})\t {\mb S_1^{\mod, (j)}}^{-1}(\mb x - \mb w_1^{*,(j)}) \leq  4 \alpha^2{\mathsf{SNR}_0^{\mod}}^2 $ implies that $\big\|\mb x - \mb w_1^{*, (j)}\big\|_2 \leq 2 \bar \sigma\alpha \mathsf{SNR}^{\mod}_0$ by \eqref{eq: lower bound on sigma min}.  
Then the expression in $\snrfull$ of the equivalent Gaussian mixture model with the means and covariance matrices defined in \eqref{eq: centers and covariances of equivalent Gaussian mixture}  gives that 
\begin{align} 
    & (\mb x - \mb w_1^{*,(j)})\t {\mb S_1^{\mod, (j)}}^{-1} (\mb x - \mb w_1^{*,(j)}) + \log|\mb S_1^{\mod, (j)}|  -  (\mb x - \mb w_2^{*,(j)})\t {\mb S_2^{\mod, (j)}}^{-1} (\mb x - \mb w_2^{*,(j)}) -\log|\mb S_2^{\mod, (j)}|  \\ 
     =&(\mb x - \mb w_1^{*,(j)})\t \big({\mb S_1^{\mod, (j)}}^{-1} - {\mb S_2^{\mod, (j)}}^{-1}\big)(\mb x - \mb w_1^{*,(j)})  + 2(\mb w_2^{*,(j)} - \mb w_1^{*,(j)})\t {\mb S_2^{\mod, (j)}}^{-1}(\mb x - \mb w_1^{*,(j)}) \\& - (\mb w_2^{*, (j)} - \mb w_1^{*, (j)})\t {\mb S_2^{\mod, (j)}}^{-1}(\mb w_2^{*, (j)} - \mb w_1^{*, (j)})  +  \log|\mb S_1^{\mod, (j)}| - \log |\mb S_2^{\mod, (j)}| \\ 
    \leq  &  (\mb x - \mb w_1^{*,(j)})\t \big({\mb S_1^{\mod, (j)}}^{-1} - {\mb S_2^{\mod, (j)}}^{-1}\big)(\mb x - \mb w_1^{*,(j)})  - 2 \bignorm{ {\mb S_2^{\mod, (j)}}^{-\frac{1}{2}} (\mb w_2^{*, (j)} - \mb w_1^{*, (j)}) }_2\\
    &\cdot  \big(- \bignorm{ {\mb S_2^{\mod, (j)}}^{-\frac{1}{2}} (\mb x - \mb w_1^{*, (j)}) }_2 + \frac{1}{2}\bignorm{ {\mb S_2^{\mod, (j)}}^{-\frac{1}{2}} (\mb w_2^{*, (j)} - \mb w_1^{*, (j)}) }_2\big)   + \log|\mb S_1^{\mod, (j)}| - \log |\mb S_2^{\mod, (j)}|. 
\label{eq: proof of the lower bound alpha SNR}
\end{align}

Looking further into the terms in \eqref{eq: proof of the lower bound alpha SNR} together with \eqref{eq: S_k^j diff} shows that 
\begin{align}
    &  (\mb x - \mb w_1^{*,(j)})\t \big({\mb S_1^{\mod, (j)}}^{-1} - {\mb S_2^{\mod, (j)}}^{-1}\big)(\mb x - \mb w_1^{*,(j)}) \leq  \frac{8\alpha^2 \bar \sigma^2{\mathsf{SNR}_0^{\mod}}^2}{\underline\sigma^2},  \label{eq: upper bound on diff of S}
    \\ 
    & \bignorm{ {\mb S_2^{\mod, (j)}}^{-\frac{1}{2}} (\mb x - \mb w_1^{*, (j)}) }_2  
    =  \big(\ip{(\mb x - \mb w_1^{*, (j)}), ({\mb S_2^{\mod, (j)}}^{-1} - {\mb S_1^{\mod,(j)}}^{-1}) (\mb x - \mb w_1^{*, (j)})}  
    +  \bignorm{ {\mb S_1^{\mod, (j)}}^{-\frac{1}{2}} (\mb x - \mb w_1^{*, (j)}) }_2^2 \big)^{\frac{1}{2}}\\ 
    \leq & \frac{2\sqrt{2} \alpha \bar \sigma\mathsf{SNR}^{\mod}_0}{\underline\sigma} + \alpha \mathsf{SNR}^{\mod}_0\leq  5 \alpha \frac{\bar \sigma}{\underline\sigma } \mathsf{SNR}^{\mod}_0, 
    \label{eq: upper bound on diff of S2}
    \\
    &  \bignorm{ {\mb S_2^{\mod, (j)}}^{-\frac{1}{2}} (\mb w_2^{*, (j)} - \mb w_1^{*, (j)}) }_2 = \ip{(\mb w_2^{*, (j)} - \mb w_1^{*, (j)}), {\mb S_2^{\mod, (j)}}^{-1}(\mb w_2^{*, (j)} - \mb w_1^{*, (j)})}^{\frac{1}{2}}\\
    = & \big(\ip{\mb w_2^* - \mb w_1^*, {\mb S_2^{\mod}}^{-1}(\mb w_2^* - \mb w_1^*)} +  \frac{{\alpha'}^2}{\underline\sigma^2}(\tilde{\mb w}\t (\mb w_2^* - \mb w_1^*))^2 \big)^{\frac{1}{2}} \\
    \geq & \left(\frac{\underline\sigma^2}{\bar \sigma^2} {\mathsf{SNR}^{\mathsf{mod}}_0}^2 + {\alpha'}^2 \frac{\underline\sigma^2}{\bar \sigma^2} {\mathsf{SNR}_0^{\mod}}^2\right)^{\frac{1}{2}}\geq \alpha'\frac{\underline\sigma }{\bar \sigma} \mathsf{SNR}_0^{\mod}, 
    \label{eq: upper bound on diff of S3}
\end{align}
where we make use of Lemma~\ref{lemma: SNR and distance} and \eqref{eq: definition of tilde w}. 

Taking the bounds \eqref{eq: upper bound on diff of S}, \eqref{eq: upper bound on diff of S2}, and \eqref{eq: upper bound on diff of S3} collectively into \eqref{eq: proof of the lower bound alpha SNR} yields that
\begin{align}
    & (\mb x - \mb w_1^{*,(j)})\t \big({\mb S_1^{\mod, (j)}}^{-1} - {\mb S_2^{\mod, (j)}}^{-1}\big)(\mb x - \mb w_1^{*,(j)}) \\
    & + 2(\mb w_2^{*,(j)} - \mb w_1^{*,(j)})\t {\mb S_2^{\mod, (j)}}^{-1}(\mb x - \mb w_1^{*,(j)}) \\& - (\mb w_2^{*, (j)} - \mb w_1^{*, (j)})\t {\mb S_2^{\mod, (j)}}^{-1}(\mb w_2^{*, (j)} - \mb w_1^{*, (j)}) + \log|\mb S_1^{\mod, (j)}| - \log |\mb S_2^{\mod, (j)}|\\ 
    \leq & \frac{8 \alpha^2 \bar \sigma^2{\mathsf{SNR}_0^{\mod}}^2}{\underline\sigma^2}  -  \alpha'\frac{\underline\sigma }{\bar \sigma } \mathsf{SNR}^{\mod}_0\left( 
        \alpha' \frac{\underline\sigma }{\bar \sigma} \mathsf{SNR}_0^{\mod} - 10  \alpha \frac{\bar \sigma}{\underline\sigma } \mathsf{SNR}_0^{\mod}
    \right) + \log|\mb S_1^{\mod, (j)}| - \log |\mb S_2^{\mod, (j)}| \\ 
    \stackrel{\alpha' = \frac{12\alpha\bar\sigma^2 }{\underline\sigma^2} }{=} &( 8 \alpha^2 - 24\alpha^2)\frac{\bar \sigma^2 {\mathsf{SNR}_0^{\mod}}^2}{\underline\sigma^2 } + \log|\mb S_1^{\mod, (j)}| - \log |\mb S_2^{\mod, (j)}|\\ 
    < &  0,
\end{align}
for every sufficiently large $n$ since $\mathsf{SNR}^{\mod}_0 \rightarrow \infty$, which leads to the conclusion.

\end{proof}

\emph{Step 3.3: Lower Bounding $L_{{\bo \eta}^{(j_1)}}(\hat{\mb z}) + L_{{\bo \eta}^{(j_2)}}(\hat{\mb z})$. }
We finally come to the most essential step of our proof. Proposition~\ref{proposition: lower bound simplification} has allowed us to reformulate a $\hat{\mb z}$-related problem into a quantity that only relies on the parameters themselves, as expressed by the RHS of \eqref{eq: lower bound simplification 3}. 
The main challenge in deriving a lower bound for our target is that we can not directly calculate the integral since the decision boundaries of $\tilde{\mb z}^{(j_1)}$ and $\tilde{\mb z}^{(j_2)}$ are both quadratic surfaces except for the special homogeneous covariance matrix case with $\bo\Sigma_1=\bo\Sigma_2$. We take a detour herein to find a critical region inside the set 
$$\left\{ \frac{\phi_{\bo \theta_2^*, \mb \Sigma_2^{(j_1)} }}{\phi_{\bo \theta_1^*, 
\mb \Sigma_1^{(j_1)} }}\leq \frac{1}{2},
\quad 
\frac{\phi_{\bo \theta_1^*,
\mb \Sigma_1^{(j_2)}}}{\phi_{\bo \theta_2^*, \mb \Sigma_2^{(j_2)}}}\leq \frac{1}{2}\right\}.$$ 

Recall that the maximizer of $\phi_{\bo\theta_k^*, \mb \Sigma_k^{(j)}} (\mb V^* \mb w_* + \mb V^*_\perp \mb z)$ in terms of $\mb z\in \bb R^{p-2}$ for $k\in[2]$ and $j \in [M]$ is written as 
\begin{align}
    & \mb z^{k,(j)}_* =  - \Big({\mb V^*_\perp}\t {\mb \Sigma_k^{(j)}}^{-1} \mb V_\perp^* \Big)^{-1} \Big({\mb V^*_\perp }\t {\mb \Sigma_k^{(j)}}^{-1} \mb V^*\Big){\mb V^*}\t \big(\mb x_* - \bo \theta_k^*\big).\label{eq: maximizer over perpendicular space}
\end{align}
Plugging \eqref{eq: cov construction 1} and \eqref{eq: cov construction 2} into \eqref{eq: maximizer over perpendicular space} yields that 
\eq{
\mb z^{k,(j)}_* =   -\alpha' \big(\tilde{\mb w}\t (\mb w_* - \mb w_k^*)\big) \mb v^{(j)}. \label{eq: expression of z_*^k(j)}
}
Given $j_1 \neq j_2 \in[M]$, we also introduced an orthonormal matrix denoted by $\mb V^{(j_1, j_2)} \in O(p-2, 2)$ whose column space is the one spanned by $(\mb v^{(j_1)}, \mb v^{(j_2)})$.

Now we let $\mb V^{(j_1, j_2)}_\perp \in O(p, p-4)$ be an orthonormal matrix perpendicular to $\big(\mb V^*, \mb V^*_\perp \mb V^{(j_1,j_2)}\big)$. Then the critical region $R^{(j_1, j_2)}$ is written as 
\begin{align}
     R^{(j_1,j_2)}
    = &\Big(\mb V^*, \mb V_\perp^*\mb V^{(j_1,j_2)}, \mb V^{(j_1,j_2)}_\perp \Big) \\ 
& \times \left[B\big(\mb w_*, \rho_1 \underline\sigma \big)  \times B\big({\mb V^{(j_1,j_2)}}\t  (\mb z_*^{1,(j_1)}+\mb z_*^{2,(j_2)}), \rho_2 \underline\sigma \big) \times \bb R^{p-4}\right] \\ 
     =  & \big\{\mb V^* (\mb w_* + \bo \triangle_1 ) + \mb V_\perp^*\mb V^{(j_1,j_2)}\big({{}\mb V^{(j_1,j_2)}}\t  (\mb z^{1,(j_1)}_* + \mb z^{2,(j_2)}_*) + \boldsymbol \triangle_2 \big) + \mb V^{(j_1, j_2)}_\perp \bo \triangle_3: \\ 
     & \qquad \norm{\bo\triangle_1 }_2 \leq \rho_1 \underline \sigma, 
     \ \bignorm{\bo\triangle_2 }_2 \leq \rho_2 \underline \sigma, 
     \ \bo \triangle_3 \in \bb R^{p-4} \big\}, \label{eq: set R}
\end{align}
where $\rho_1$ and $\rho_2$ are some fixed positive constants.

To lower bound $L_{\bo \eta^{(j_1)}}(\hat{\mb z}) + L_{\bo \eta^{(j_2)}}(\hat{\mb z})$ via integrating over $R^{(j_1,j_2)}$, the following two conditions are essential:
\begin{itemize}
    \item \textsc{Condition 1}: $$\frac{\phi_{\bo \theta_2^*, \mb \Sigma_2^{(j_1)} }(\mb x ) }{\phi_{\bo \theta_1^* , \mb \Sigma_1^{(j_1)} }(\mb x) }\leq \frac{1}{2}\text{\quad  and \quad }\frac{\phi_{\bo \theta_1^*, \mb \Sigma_1^{(j_2)}}(\mb x)}{\phi_{\bo \theta_2^*, \mb \Sigma_2^{(j_2)}}(\mb x) }\leq \frac{1}{2}.$$ 
    \item \textsc{Condition 2}: The minimum of $\phi_{\bo \theta_1^{*}, \mb \Sigma_1^{(j_1)}}(\mb x)$ and $\phi_{\bo \theta_2^{*}, \mb \Sigma_2^{(j_2)}}(\mb x)$ is lower bounded by 
    \begin{align}
       & f^{\mathsf{lower}}(\mb x) \\
 \coloneqq &\underbrace{\left[\frac{1}{(2\pi)^2 \bar\sigma^4}\exp\left(-\left(1 +  \frac{C^{\mathsf{density}}_1}{\mathsf{SNR}_0^{\mod}} + C^{\mathsf{density}}_2 \delta + C^{\mathsf{density}}_3 \frac{\bar\sigma^2}{\tilde\sigma^2 } +C_4^{\mathsf{density}} \frac{\log\big( \frac{\tilde\sigma}{\bar \sigma}\big)}{{\mathsf{SNR}_0^{\mod}}^2}\right)\frac{{\mathsf{SNR}^{\mathsf{mod}}}^2}{2}\right) \right]}_{\text{the pdf of a dim-4 Gaussian distribution}} \\ 
    & \cdot \underbrace{\left[\frac{1}{(2\pi)^{\frac{p-2}{2}}\tilde\sigma^{p-2}}\exp\Big(-{\norm{\mb x\t \mb V_\perp^{(j_1,j_2)}}_2^2}/({2\tilde\sigma^2}) \Big)\right]}_{{\text{the pdf of a dim-$(p-4)$ Gaussian distribution}}} 
    \end{align}
    for some constants $C^{\mathsf{density}}_i$, $i = 1,2,3, 4$. 
\end{itemize}
We certify \textsc{Condition~1} and \textsc{Condition~2} for each element in $R^{(j_1, j_2)}$ in the following claim. 
\begin{claim}\label{claim: condition 1/2}
    For the $\{\bo \eta^{(j)}\}_{j\in[M]}$ constructed in \eqref{eq: centers and covariances in the packing}, \textsc{Condition~1} and \textsc{Condition~2} hold for every $\mb x \in R^{(j_1,j_2)}$ and every sufficiently large $n$ with the constants $C^{\mathsf{density}}_i$, $i =1,2,3,4$ associated with $\bar\sigma / \underline\sigma, \rho_1, \rho_2, \alpha$. 
\end{claim}

\medskip

\begin{proof}[Proof of Claim~\ref{claim: condition 1/2}]
We first verify \textsc{Condition~1}. For each $\mb y\in R^{(j_1,j_2)}$, the difference of log-likelihood functions is given by 
\begin{align}
    & (\mb y - \bo \theta_1^*)\t (\mb \Sigma_1^{(j_1)})^{-1}(\mb y - \bo \theta_1^*) - (\mb y - \bo \theta_2^*)\t (\mb \Sigma_2^{(j_1)})^{-1}(\mb y - \bo \theta_2^*) + \log|\mb \Sigma_1^{(j_1)}| - \log|\mb \Sigma_2^{(j_1)}| \\ 
    = &  (\mb y - \bo \theta_1^*)\t \mb V^*\big({\mb S_1^{\mod}}^{-1} + \frac{{\alpha'}^2}{\underline\sigma^2}\tilde{\mb w}\tilde{\mb w}\t \big){\mb V^{*\top}} (\mb y - \bo \theta_1^*)  - (\mb y - \bo \theta_2^*)\t\mb V^*\big({\mb S_2^{\mod}}^{-1} + \frac{{\alpha'}^2}{\underline\sigma^2}\tilde{\mb w}\tilde{\mb w}\t \big){\mb V^{*\top}}  (\mb y - \bo \theta_2^*) \\ 
    & \qquad -2 \frac{\alpha' }{\underline\sigma^2}(\mb y - \bo \theta_1^*)\t \mb V^* \tilde{\mb w}{\mb v^{(j_1)}}\t {\mb V_\perp^*}\t (\mb y - \bo \theta_1^*)  + 2 \frac{\alpha' }{\underline\sigma^2}(\mb y - \bo \theta_2^*)\t \mb V^* \tilde{\mb w}{\mb v^{(j_1)}}\t {\mb V_\perp^*}\t (\mb y - \bo \theta_2^*) +  \log|\mb \Sigma_1^{(j_1)}| - \log|\mb \Sigma_2^{(j_1)}|  ,\label{eq: diff of pdf 0}
\end{align}
where we employ the fact that 
$$(\mb y - \bo \theta_1^*)\t \mb V_\perp^*  (\mb \Sigma_1^{(j_1)})^{-1} {\mb V_\perp^* }\t (\mb y - \bo \theta_1^*) =  (\mb y - \bo \theta_2^*)\t \mb V_\perp^* (\mb \Sigma_2^{(j_1)})^{-1} { \mb V_\perp^* } \t (\mb y - \bo \theta_2^*). $$

Plugging $\mb y = \mb V^* ( \mb w_* + \bo \triangle_1 ) + \mb V^*_\perp \mb V^{(j_1,j_2)} \big({\mb V^{(j_1,j_2)}}\t (\mb z_*^{1,(j_1)} + \mb z_*^{2,(j_2)} ) + \bo \triangle_2\big) +  \mb V^{(j_1, j_2)}_\perp \bo \triangle_3$ into \eqref{eq: diff of pdf 0} yields that 
\begin{align}
&(\mb y - \bo \theta_1^*)\t (\mb \Sigma_1^{(j_1)})^{-1}(\mb y - \bo \theta_1^*) - (\mb y - \bo \theta_2^*)\t (\mb \Sigma_2^{(j_1)})^{-1}(\mb y - \bo \theta_2^*)  + \log|\mb \Sigma_1^{(j_1)}| - \log|\mb \Sigma_2^{(j_1)}| \\
     = &   (\mb w_* - \mb w_1^* + \bo \triangle_1)\t \big({\mb S_1^{\mod}}^{-1} + \frac{{\alpha'}^2}{\underline\sigma^2}\tilde{\mb w}\tilde{\mb w}\t \big)(\mb w_* - \mb w_1^* + \bo \triangle_1)  - (\mb w_* - \mb w_2^* + \bo \triangle_1)\t\big({\mb S_2^{\mod}}^{-1} + \frac{{\alpha'}^2}{\underline\sigma^2}\tilde{\mb w}\tilde{\mb w}\t \big)(\mb w_* - \mb w_2^* + \bo \triangle_1)
    \\ 
    & -2 \frac{\alpha' }{\underline\sigma^2}(\mb w_* - \mb w_1^* + \bo \triangle_1)\t \tilde{\mb w}{\mb v^{(j_1)}}\t(\mb z_*^{1,(j_1)} + \mb z_*^{2,(j_2)} + \mb V^{(j_1,j_2)} \bo \triangle_2  )
    \\ 
    & + 2 \frac{\alpha' }{\underline\sigma^2}(\mb w_* - \mb w_2^* + \bo \triangle_1)\t  \tilde{\mb w}{\mb v^{(j_1)}}\t (\mb z_*^{1,(j_1)} + \mb z_*^{2,(j_2)} + \mb V^{(j_1,j_2)} \bo \triangle_2  )  +  \log|\mb \Sigma_1^{(j_1)}| - \log|\mb \Sigma_2^{(j_1)}|  
     \\
    =  & 
    \underbrace{\frac{\alpha'^2}{\underline\sigma^2}\big( (\mb w_* - \mb w_1^*) \t \tilde{\mb w}\big)^2  - \frac{\alpha'^2}{\underline\sigma^2}\big( (\mb w_* - \mb w_2^*) \t \tilde{\mb w}\big)^2 }_{\eqqcolon C^\mathsf{main}_1}
    \underbrace{
     - 2 \frac{\alpha' }{\underline\sigma^2} (\mb w_* - \mb w_1^*)\t \tilde{\mb w} {\mb v^{(j_1)}}\t \mb z_*^{1,(j_1)} +  2 \frac{\alpha' }{\underline\sigma^2} (\mb w_* - \mb w_2^*)\t \tilde{\mb w} {\mb v^{(j_1)}}\t \mb z_*^{1,(j_1)}
    }_{\eqqcolon C^\mathsf{main}_2}  
    \\
    & + \underbrace{ \bo\triangle_1\t \big( {\mb S_1^{\mod}}^{-1} - {\mb S_2^{\mod}}^{-1} \big) \bo \triangle_1 + 2 \bo \triangle_1\t {\mb S_1^{\mod}}^{-1}(\mb w_* - \mb w_1^*)  - 2 \bo \triangle_1\t {\mb S_2^{\mod}}^{-1}(\mb w_* - \mb w_2^*)}_{\eqqcolon C_1} \\
    & + \underbrace{ 2\frac{{\alpha'}^2}{\underline\sigma^2} \bo \triangle_1\t \tilde{\mb w}\tilde{\mb w}\t\big(\mb w_* - \mb w_1^* \big) -  2 \frac{{\alpha'}^2}{\underline\sigma^2}\bo \triangle_1\t \tilde{\mb w}\tilde{\mb w}\t\big(\mb w_* - \bo w_2^* \big) }_{\eqqcolon C_2}\\ 
    &   
    \underbrace{
       -2 \frac{\alpha'}{\underline\sigma^2} (\mb w_* - \mb w_1^*)\t \tilde{\mb w} {\mb v^{(j_1)}}\t \mb z_*^{2,(j_2)} + 2 \frac{\alpha'}{\underline\sigma^2} (\mb w_* - \mb w_2^*)\t \tilde{\mb w} {\mb v^{(j_1)}}\t \mb z_*^{2,(j_2)} 
    }_{\eqqcolon C_3}
    \\ 
    & +  \underbrace{2 \frac{{\alpha'}}{\underline\sigma^2}(\mb w_* - \mb w_1^*)\t \tilde{\mb w} {\mb v^{(j_1)}}\t \mb V^{(j_1,j_2)} \bo \triangle_2 - 2 \frac{{\alpha'}}{\underline\sigma^2}(\mb w_* - \mb w_2^*)\t \tilde{\mb w} {\mb v^{(j_1)}}\t \mb V^{(j_1,j_2)} \bo \triangle_2}_{\eqqcolon C_4} \\
    & +\underbrace{ 2 \frac{{\alpha'}}{\underline\sigma^2}\bo \triangle_1\t \tilde{\mb w} {\mb v^{(j_1)}}\t \mb V^{(j_1,j_2)} \bo \triangle_2 - 2 \frac{{\alpha'}}{\underline\sigma^2}\bo \triangle_1\t \tilde{\mb w} {\mb v^{(j_2)}}\t \mb V^{(j_1,j_2)} \bo \triangle_2 }_{\eqqcolon C_5}  + \underbrace{\log|\mb \Sigma_1^{(j_1)}| - \log|\mb \Sigma_2^{(j_1)}|}_{C_6}, \label{eq: C1 to C6}
\end{align}
where we make use of the property inferred from the definition of $\mb w_*$ in \eqref{eq: definition of w_*} that 
\begin{align}
    & (\mb w_* - \mb w_1^*)\t  {\mb S_1^{\mod}}^{-1} {\mb V^{*\top}} (\mb x_* - \bo \theta_1^*) - (\mb x_* - \bo \theta_2^*)\t \mb V^* {\mb S_2^{\mod}}^{-1} (\mb w_* - \mb w_2^*)  = 0. 
\end{align}

To facilitate understanding, $C_1^{\mathsf{main}}$ and $C_2^{\mathsf{main}}$ capture the substantial gap between two log-likelihood functions, while $C_1$ through $C_6$ collect the remnant effects influenced by $\rho_1$, $\rho_2$, ${\mb v^{(j_1)}}\t \mb v^{(j_2)}$, and $\log|\mb S_k^{\mod}|$. 

Denote 
$$
(\mb w_* - \mb w_1^*) \t \tilde{\mb w}=:\xi^{\mathsf{align}}.
$$ 
By the definition of $\tilde{\mb w}$, it is immediate that 
\eq{(\mb w_* - \mb w_2^*) \t \tilde{\mb w} = \xi^{\mathsf{align}} - \norm{\bo \theta_2^* - \bo \theta_1^*}_2. 
\label{eq: xi align relation}
}
We then analyze these terms in \eqref{eq: C1 to C6} separately: 
\begin{enumerate}
\item Regarding the sum of the first two terms in \eqref{eq: C1 to C6}, invoking \eqref{eq: expression of z_*^k(j)} and \eqref{eq: xi align relation} gives that 
\begin{align}
    C_1^{\mathsf{main}}+ C_2^{\mathsf{main}}  = & \frac{\alpha'^2}{\underline\sigma^2} {\xi^{\mathsf{align}}}^2 - \frac{\alpha'^2}{\underline\sigma^2}  ( \norm{\bo \theta_1^* - \bo \theta_2^*}_2 - \xi^{\mathsf{align}}) ^2  -2 \frac{\alpha'^2 }{\underline\sigma^2} {{}\xi^{\mathsf{align}}}^2  + 2 \frac{\alpha'^2 }{\underline\sigma^2} \xi^{\mathsf{align}} (\xi^{\mathsf{align}} -\norm{\bo \theta_1^* - \bo \theta_2^*}_2 )  \\ 
     = & - \frac{\alpha'^2}{\underline\sigma^2} \norm{\bo \theta_2^* - \bo \theta_1^*}_2^2 \leq - \alpha'^2 {\mathsf{SNR}_0^{\mod}}^2. 
\end{align}
\item
Employing the constraint on $\bo \triangle_1$ in $R^{(j_1,j_2)}$ as well as the relation between $\mathsf{SNR}$ and $\mb w_*$ yields that 
\begin{align}
    |C_1| \leq & \frac{2}{\underline\sigma^2}\rho_1^2 + 4\frac{\rho_1}{\underline\sigma} \mathsf{SNR}_0^{\mod}. 
\end{align}
\item 
With regard to $C_2$, it immediately follows by Lemma~\ref{lemma: SNR and distance} that 
\begin{align}
    & |C_2 |= 2 \frac{\alpha'^2}{\underline\sigma^2} \norm{\bo \theta_2^* - \bo \theta_1^*}_2 |\bo \triangle_1\t \tilde{\mb w}| \leq 4 \rho_1 \frac{\alpha'^2}{\underline\sigma^2} \bar\sigma \mathsf{SNR}_0^{\mod}. 
\end{align}
\item 
The term $C_3$ is related to the inner product between $\mb v^{(j_1)}$ and $\mb v^{(j_2)}$. Apply Lemma~\ref{lemma: SNR and distance}, \eqref{eq: almost mutually orthogonal property}, and \eqref{eq: expression of z_*^k(j)} to $C_3$ yields that 
\begin{align}
    | C_3 |\leq 2\frac{\alpha'^2}{\underline\sigma^2} \norm{\bo \theta_1^* - \bo \theta_2^* }_2 \delta \leq  4 \frac{\alpha'^2}{\underline\sigma^2} \bar\sigma \mathsf{SNR}_0^{\mod},
\end{align}
since $\delta = o(1)$.
\item 
As for $C_4$, it can be bounded by Lemma~\ref{lemma: SNR and distance} that 
\eq{
    |C_4 |= 2\frac{\alpha'}{\underline\sigma^2} \norm{\bo \theta_2^* - \bo \theta_1^*}_2 \norm{\bo\triangle_2}_2   \leq 4\frac{{\alpha'}}{\underline\sigma^2}\rho_2 \bar \sigma\mathsf{SNR}_0^{\mod}.
}
\item 
What we are left is to upper bound the term $C_6$. The elementary fact that 
\eq{\label{eq: determinant of block matrix}
    \mathrm{det}\left(\begin{matrix}
    \mb A & \mb B \\ 
    \mb C & \mb D
\end{matrix}\right) = \mathrm{det}(\mb A)\mathrm{det}(\mb B\mb D^{-1} \mb C)}
given an invertible $\mb D$ yields that 
\eq{
   \big|  \log|\mb \Sigma_1^{(j_1)}| - \log|\mb \Sigma_2^{(j_1)}| \big|  = \big|\log|{\mb S_1^{\mod}}^{-1}| - \log|{\mb S_2^{\mod}}^{-1}| \big| = \big|\log|{\mb S_1^{\mod}}| - \log|{\mb S_2^{\mod}}| \big| \leq 2 \log\left(\frac{\bar\sigma}{\underline\sigma}\right). 
}
\end{enumerate}

Plugging the above bounds on $C^{\mathsf{main}}_1$, $C^{\mathsf{main}}_2$, and $C_1$ through $C_6$ into \eqref{eq: C1 to C6} gives that 
\begin{align}
    & (\mb y - \bo \theta_1^*)\t (\mb \Sigma_1^{(j_1)})^{-1}(\mb y - \bo \theta_1^*) - (\mb y - \bo \theta_2^*)\t (\mb \Sigma_2^{(j_1)})^{-1}(\mb y - \bo \theta_2^*) + \log|\mb \Sigma_1^{(j_1)}| - \log|\mb \Sigma_2^{(j_1)}|  \\ 
    \leq & - \alpha'^2 {\mathsf{SNR}_0^{\mod}}^2 + \frac{2}{\underline\sigma^2}\rho_1^2 + 4 \frac{\rho_1}{\underline\sigma} \mathsf{SNR}_0^{\mod} + 4\rho_1 \frac{\alpha'^2}{\underline \sigma^2}\bar\sigma \mathsf{SNR}_0^{\mod} + 4 \frac{\alpha'^2}{\underline\sigma^2} \bar\sigma \mathsf{SNR}_0^{\mod} + 4 \frac{\alpha'}{\underline\sigma^2} \rho_2 \bar \sigma \mathsf{SNR}_0^{\mod} + 2 \log(\frac{\bar \sigma}{\underline\sigma}) \\ 
    < & -\log 2, \label{eq: lower bound on likelihood ratio}
\end{align}
holds for every sufficiently large $n$ since $\mathsf{SNR}_0^{\mod} \rightarrow \infty$. Referring back to the definition of $R^{(j_1,j_2)}$, \eqref{eq: lower bound on likelihood ratio} has already implied that 
$$ R^{(j_1, j_2)}\subseteq \left\{ \frac{\phi_{\bo \theta_2^*, \mb \Sigma_2^{(j_1)}}}{\phi_{\bo \theta_1^*, \mb \Sigma_1^{(j_1)}}}\leq \frac{1}{2}\right\}
$$
for every sufficiently large $n$. 
Following the same argument, we can similarly verify that $$R^{(j_1, j_2)}\subseteq \left\{ \frac{\phi_{\bo \theta_1^*, \mb \Sigma_1^{(j_2)}}}{\phi_{\bo \theta_2^*, \mb \Sigma_2^{(j_2)}}} \leq \frac12  \right\}$$
for every sufficiently large $n$. To conclude, we have proved that 
$$R^{(j_1, j_2)}\subseteq \left\{\frac{\phi_{\bo \theta_2^*, \mb \Sigma_2^{(j_1)}}}{\phi_{\bo \theta_1^*, \mb \Sigma_1^{(j_1)}}}\leq \frac{1}{2},  \frac{\phi_{\bo \theta_1^*, \mb \Sigma_1^{(j_2)}}}{\phi_{\bo \theta_2^*, \mb \Sigma_2^{(j_2)}}}\leq \frac{1}{2}\right\}.$$ 

\medskip
\subparagraph*{Verifying \textsc{Condition 2}}
We now turn to investigate the minimum of two probability density functions in the region $R^{(j_1, j_2)}$. Looking into the density functions separately, the spherical region in $R^{(j_1,j_2)}$ yields that for every $$\mb y = \underbrace{\mb V^* (\mb w_* + \bo \triangle_1 ) + \mb V^*_\perp \big(\mb V^{(j_1,j_2)}({\mb V^{(j_1,j_2)}}\t (\mb z_*^{1,(j_1)} + \mb z_*^{2,(j_2)}) + \bo \triangle_2 )\big)}_{\eqqcolon \mb y^{(j_1,j_2)}_{\mathsf{key}}} + \mb V_\perp^{(j_1,j_2)}\bo \triangle_3 \in R^{(j_1, j_2)}, $$ 
it holds that
\begin{align}
    & \phi_{\bo \theta_1^*, \mb \Sigma_1^{(j_1)}}(\mb y) =   \underbrace{\frac{1}{(2\pi)^2\big|{{}\tilde{\mb V}^{(j_1,j_2)}}\t \mb \Sigma_1^{(j_1)}\tilde{\mb V}^{(j_1,j_2)}\big|^{\frac{1}{2}}} \exp\big(-\frac{1}{2}({{}\mb y_{\mathsf{key}}^{(j_1,j_2)} } - \bo \theta_1^* )\t {\mb \Sigma^{(j_1)}}^{-1}(\mb y_{\mathsf{key}}^{(j_1,j_2)} - \bo \theta_1^*)\big)}_{\eqqcolon f^{(j_1, j_2)}_{1,\mathsf{essential}}} \\ 
    &\quad \cdot \frac{1}{(2\pi)^{\frac{p-4}{2}}\tilde\sigma^{p-4}}\exp\left(- \frac{\norm{\bo \triangle_3}_2^2}{2\tilde\sigma^2}\right). \label{eq: density of theta1}
\end{align}
where $\tilde{\mb V}^{(j_1,j_2)} \coloneqq (\mb V^*,\mb V_\perp^* \mb V^{(j_1, j_2)})\in O(p,4)$ is an orthonormal matrix. 
Now we set out to analyze the function $f_{1, \mathsf{essential}}^{(j_1,j_2)}$ defined in \eqref{eq: density of theta1}. First, the normalization factor ${1}/{\big|{{}\tilde{\mb V}^{(j_1,j_2)}}\t \mb \Sigma_1^{(j_1)}\tilde{\mb V}^{(j_1,j_2)}\big|^{\frac{1}{2}}}$ can be reduced as 
\begin{align}
    & \frac{1}{\big|{{}\tilde{\mb V}^{(j_1,j_2)}}\t \mb \Sigma_1^{(j_1)}\tilde{\mb V}^{(j_1,j_2)}\big|^{\frac{1}{2}}} = \big|\big({{}\tilde{\mb V}^{(j_1,j_2)}}\t \mb \Sigma_1^{(j_1)}\tilde{\mb V}^{(j_1,j_2)}\big)^{-1} \big|^{\frac{1}{2}} \geq \frac{1}{\tilde\sigma\underline\sigma \bar \sigma^2}\label{eq: lower bound on normalization factor}
\end{align} 
by the definition \eqref{eq: centers and covariances in the packing} and the fact \eqref{eq: determinant of block matrix}. 

Second, recalling the definition of $\mb \Sigma^{(j_1)}$ in \eqref{eq: centers and covariances in the packing}, the exponent of $f_{1, \mathsf{essential}}^{(j_1,j_2)}$ is decomposed as follows: 
\begin{align}
    & -\frac{1}{2}({{}\mb y_{\mathsf{key}}^{(j_1,j_2)} } - \bo \theta_1^* )\t {\mb \Sigma^{(j_1)}}^{-1}(\mb y_{\mathsf{key}}^{(j_1,j_2)} - \bo \theta_1^*) \\ 
    = & -\frac{1}{2}(\mb w_* + \bo \triangle_1 - \mb w_1^*) \t {\mb V^{*\top}}{ \mb \Sigma_1^{(j_1)}}^{-1}\mb V^* (\mb w_* + \bo \triangle_1 - \mb w_1^*) \\ 
     &- \big(\mb w_* + \bo \triangle_1  - \mb w_1^*\big) \t {\mb V^*}\t {\mb \Sigma_1^{(j_1)}}^{-1} \mb V^*_\perp  \big(\mb z_*^{1,(j_1)} + \mb z_*^{2,(j_2)} + \mb V^{(j_1,j_2)}\bo \triangle_2 \big)\\ 
     & - \frac{1}{2} \big(\mb z_*^{1,(j_1)} + \mb z_*^{2,(j_2)} + \mb V^{(j_1,j_2)}\bo \triangle_2 \big)\t  {\mb V^*_\perp}\t{ \mb \Sigma_1^{(j_1)}}^{-1}\mb V^*_\perp  \big(\mb z_*^{1,(j_1)} + \mb z_*^{2,(j_2)} + \mb V^{(j_1,j_2)}\bo \triangle_2 \big)\\ 
    =: &  D_1 + D_2 + D_3 + D_4, \label{eq: D1 to D4}
\end{align}
where $D_i$, $ i=1,2,3,4$ are defined as follows:
\begin{align}
    & D_1 \coloneqq -\frac{1}{2}(\mb w_* + \bo \triangle_1 - \mb w_1^*) \t \big({\mb S_1^{\mod}}^{-1} + \frac{{\alpha'}^2}{\underline\sigma^2}\tilde{\mb w}{\tilde{\mb w}}\t  \big) (\mb w_* + \bo \triangle_1 - \mb w_1^*), \\ 
    & D_2 \coloneqq  - \big(\mb w_* + \bo \triangle_1  - \mb w_1^*\big) \t {\mb V^*}\t {\mb \Sigma_1^{(j_1)}}^{-1} \mb V^*_\perp  \big(\mb z_*^{1,(j_1)} + \mb z_*^{2,(j_2)} + \mb V^{(j_1,j_2)}\bo \triangle_2 \big), \\ 
    & D_3 \coloneqq  -\frac{1}{2\tilde\sigma^2 } \big(\mb z_*^{1,(j_1)}  + \mb z_*^{2,(j_2)}  +  \mb V^{(j_1,j_2)}\bo \triangle_2 \big)\t\\ 
    & \cdot \big( 
        \mb I_{p-2} - \mb v^{(j_1)}\mb v^{(j_1)\top} \big)
   \big(\mb z_*^{1,(j_1)}  + \mb z_*^{2,(j_2)}  +  \mb V^{(j_1,j_2)}\bo \triangle_2 \big) ,\\ 
   & D_4 \coloneqq - \frac{1}{2\underline \sigma^2} \big(\mb z_*^{1,(j_1)} + \mb z_*^{2,(j_2)} + \mb V^{(j_1,j_2)}\bo \triangle_2 \big)\t  \mb v^{(j_1)} {\mb v^{(j_1)}}\t 
   \big(\mb z_*^{1,(j_1)} + \mb z_*^{2,(j_2)} + \mb V^{(j_1,j_2)}\bo \triangle_2 \big). 
\end{align}

In what follows, we shall bound $D_1$ to $D_4$ separately: 
\begin{enumerate}
    \item Notice that by Lemma~\ref{lemma: SNR and distance} one has 
    \longeq{\label{eq: xi_align upper bound}
   &  |\xi^{\mathsf{align}}| \leq \max\{|\tilde{\mb w}\t (\mb w_* - \mb w_1^*)|, |\tilde{\mb w}\t (\mb w_* - \mb w_2^*)|\} \\ 
    \leq & \max\{\norm{\mb w_* - \mb w_1^*}_2, \norm{\mb w_* - \mb w_2^*}_2\} \leq \bar \sigma \mathsf{SNR}_0 . }
    By the definition of $\mathsf{SNR}$, one has for some constant $C_1>0 $ and every sufficiently large $n$ that 
    \begin{align}
     D_1  
        =&   -\frac{1}{2}{\mathsf{SNR}_0^{\mod}}^2 - \frac{{\alpha'}^2}{2\underline\sigma^2}{\xi^{\mathsf{align}}}^2 - (\mb w_*-\mb w_1^*)\t \big({\mb S_1^{\mod}}^{-1} + \frac{{\alpha'}^2}{\underline\sigma^2}\tilde{\mb w}\tilde{\mb w}\t \big)\bo \triangle_1  -  \frac{1}{2}\bo \triangle_1\t  \big({\mb S_1^{\mod}}^{-1}  + \frac{{\alpha'}^2}{\underline\sigma^2}\tilde{\mb w}\tilde{\mb w}\t \big)\bo \triangle_1 \\ 
        \geq & -\frac{1}{2}{\mathsf{SNR}_0^{\mod}}^2 - \frac{{\alpha'}^2}{2\underline\sigma^2}{\xi^{\mathsf{align}}}^2 - 2\bar \sigma \mathsf{SNR}_0\big(\frac{1}{\underline\sigma^2} + \frac{{\alpha'}^2}{\underline\sigma^2}\big) \rho_1 \underline\sigma - \frac{1}{2}\big(\frac{1}{\underline\sigma^2} + \frac{{\alpha'}^2}{\underline\sigma^2}\big)\underline\sigma^2 \rho_1^2 \\ 
        \geq  & - \frac{1}{2}(1 + \frac{C_1}{\mathsf{SNR}_0}) {\mathsf{SNR}_0^{\mod}}^2 - \frac{{\alpha'}^2}{2\underline\sigma^2}{\xi^{\mathsf{align}}}^2
    \end{align}
    by \eqref{eq: xi_align upper bound} and the fact that $\mathsf{SNR}_0 \rightarrow \infty$. 
    \item The term $D_2$ can be lower-bounded as follows 
    \begin{align}
     D_2 = & \frac{\alpha'^2  }{\underline\sigma^2}\big(\tilde{\mb w}\t (\mb w_* - \mb w_1^*)\big)^2 - \frac{\alpha'}{\underline \sigma^2}(\mb w_* - \mb w_1^*)\t \tilde{\mb w} {\mb v^{(j_1)}}\t  \mb z_*^{2,(j_2)}  - \frac{\alpha' }{\underline\sigma^2} (\mb w_* - \mb w_1^*)\t \tilde{\mb w} {\mb v^{(j_1)}}\t \mb V^{(j_1, j_2)}\bo \triangle_2  \\ 
     & \qquad - \bo \triangle_1\t \mb  V^* {\mb \Sigma_1^{(j_1)}}^{-1} \mb V_\perp^* \big(\mb z_*^{1,(j_1)} + \mb z_*^{2,(j_2)} + \mb V^{(j_1,j_2)} \bo \triangle_2 \big) 
     \\  \geq  & \frac{\alpha'^2}{\underline\sigma^2} {\xi^{\mathsf{align}}}^2 - \frac{\alpha'^2}{\underline\sigma^2}|\tilde{\mb w}\t (\mb w_* - \mb w_1^* )| |\tilde{\mb w}\t (\mb w_* - \mb w_2^* )| \delta - \frac{\alpha'}{\underline\sigma^2} \xi^{\mathsf{align}}\underline\sigma  \rho_2  \\ 
     & \qquad - \underline\sigma \rho_1 \frac{\alpha'}{\underline \sigma^2} \big(\alpha'|\tilde{\mb w}\t (\mb w_* - \mb w_1^* )|+ \alpha' |\tilde{\mb w}\t (\mb w_* - \mb w_2^* )| + \underline\sigma    \rho_2  \big).\label{eq: lower bound for D2 1}
    \end{align} 
    We deduce from \eqref{eq: lower bound for D2 1} and \eqref{eq: xi_align upper bound} that 
    \begin{align}
        & D_2 \geq \frac{\alpha'^2}{\underline\sigma^2} {\xi^{\mathsf{align}}}^2 - \frac{\alpha'^2}{\underline\sigma^2}\bar\sigma^2 {\mathsf{SNR}_0^{\mod}}^2 \delta  - \frac{\alpha'}{\underline\sigma} \bar \sigma \mathsf{SNR}_0 \rho_2 - 2 \rho_1 \frac{\alpha'^2 }{\underline\sigma} \bar\sigma \mathsf{SNR}_0 - \alpha' \rho_1 \rho_2 \\ 
        \geq & \frac{\alpha'^2}{\underline\sigma^2} {\xi^{\mathsf{align}}}^2 - \big( C_2 \delta + \frac{C_3}{\mathsf{SNR}_0}\big){\mathsf{SNR}_0^{\mod}}^2
    \end{align}
    holds for every sufficiently large $n$ with some constants $C_2, C_3 > 0$.
        \item Regarding the third term $D_3$, employing the triangle inequality yields that 
    \begin{align}
        D_3 \geq & - \frac{1}{\tilde\sigma^2} \Big[\big(\mb z_*^{1,(j_1)}  + \mb z_*^{2,(j_2)} \big)\t \big( 
        \mb I_{p-2} - \mb v^{(j_1)}\mb v^{(j_1)\top} \big)
   \big(\mb z_*^{1,(j_1)}  + \mb z_*^{2,(j_2)}  \big)  +\big(\mb V^{(j_1,j_2)}\bo \triangle_2 \big)\t  \big( 
        \mb I_{p-2} - \mb v^{(j_1)}\mb v^{(j_1)\top} \big)
   \big(  \mb V^{(j_1,j_2)}\bo \triangle_2 \big)  \Big]
    \end{align}

    Invoking the fact that $\{\|\mb z_*^{1,(j_1)}\|_2, \|\mb z_*^{2,(j_2)}\|_2\} \leq \alpha' \bar \sigma \mathsf{SNR}_0$ yields that 
    \begin{align}
        & D_3 \geq  -4\frac{\alpha'^2 \bar \sigma^2}{\tilde\sigma^2} {\mathsf{SNR}_0^{\mod}}^2 -\frac{\underline\sigma^2}{\tilde\sigma^2} \rho_2^2 \geq - C_4 \frac{\bar\sigma^2}{\tilde{\sigma}^2} {\mathsf{SNR}_0^{\mod}}^2  
    \end{align}
    holds with some constant $C_4 >0$ for every sufficiently large $n$. 

    \item Finally, the term $D_4$ is lower bounded by 
    \begin{align}
         D_4 \geq & - \frac{1}{2\underline\sigma^2} {\mb z_*^{1,(j_1)}}\t \mb v^{(j_1)}{\mb v^{(j_1)}}\t \mb z_*^{1,(j_1)} -\frac{1}{2\underline\sigma^2} {\mb z_*^{2,(j_2)}}\t \mb v^{(j_1)}{\mb v^{(j_1)}}\t \mb z_*^{2,(j_2)} \\
        & - \frac{1}{2 \underline\sigma^2}(\mb V^{(j_1,j_2)}\bo\triangle_2)\t \mb v^{(j_1)}{\mb v^{(j_1)}}\t (\mb V^{(j_1,j_2)}\bo\triangle_2) \\ 
        & - \frac{1}{\underline\sigma^2} \norm{\bo \triangle_2}_2 \norm{\mb z_*^{1,(j_1)} + \mb z_*^{1,(j_1)}}_2  - \frac{1}{\underline\sigma^2}\big|{\mb z_*^{1,(j_1)}}\t \mb v^{(j_1)}{\mb v^{(j_1)}}\t \mb z_*^{2,(j_2)}  \big|. 
    \end{align}
    Invoking \eqref{eq: xi_align upper bound} yields that 
    \begin{align}
         D_4 \geq & -\frac{\alpha'^2}{2\underline\sigma^2} {\xi^{\mathsf{align}}}^2 - \frac{\alpha'^2}{2\underline\sigma^2} \bar \sigma^2 {\mathsf{SNR}_0^{\mod}}^2 \delta^2 - \frac{1}{2}\rho_2^2 - 2\rho_2 \alpha' \frac{\bar \sigma}{\underline\sigma}\mathsf{SNR}_0 - \frac{\alpha'^2}{\underline\sigma^2} \bar\sigma^2 {\mathsf{SNR}_0^{\mod}}^2 \delta  \\ 
         \geq & -\frac{\alpha'^2}{2\underline\sigma^2} {\xi^{\mathsf{align}}}^2 - \big(C_5 \delta + \frac{C_6}{\mathsf{SNR}_0}\big){\mathsf{SNR}_0^{\mod}}^2. 
    \end{align}
\end{enumerate}
Taking these bounds collectively into \eqref{eq: D1 to D4} gives that 
\begin{align}
    & -\frac{1}{2}({{}\mb y_{\mathsf{key}}^{(j_1,j_2)} } - \bo \theta_1^* )\t {\mb \Sigma^{(j_1)}}^{-1}(\mb y_{\mathsf{key}}^{(j_1,j_2)} - \bo \theta_1^*) \\ 
    \geq &  \Big( - \frac{1}{2}(1 + \frac{C_1}{\mathsf{SNR}_0}) {\mathsf{SNR}_0^{\mod}}^2 - \frac{{\alpha'}^2}{2\underline\sigma^2}{\xi^{\mathsf{align}}}^2\Big) + \Big(\frac{\alpha'^2}{\underline\sigma^2} {\xi^{\mathsf{align}}}^2 - \big( C_2 \delta + \frac{C_3}{\mathsf{SNR}_0}\big){\mathsf{SNR}_0^{\mod}}^2\Big) \\ 
    & + \Big( -C_4 \frac{\bar\sigma^2}{\tilde{\sigma}^2} {\mathsf{SNR}_0^{\mod}}^2 \Big) + \Big(-\frac{\alpha'^2}{2\underline\sigma^2} {\xi^{\mathsf{align}}}^2 - \big(C_5 \delta + \frac{C_6}{\mathsf{SNR}_0}\big){\mathsf{SNR}_0^{\mod}}^2 \Big)    \\ 
    \geq & -\-\big(1 +  \frac{C^{\mathsf{density}}_1}{\mathsf{SNR}_0} + C^{\mathsf{density}}_2 \delta + C^{\mathsf{density}}_3  \frac{\bar\sigma^2 }{\tilde\sigma^2 } \big)\frac{{\mathsf{SNR}_0^{\mod}}^2}{2}\label{eq: lower bound on exponent}
\end{align}
for some constants $C_i^{\mathsf{density}}$, $i = 1,2,3$ depending on $C_i$, $i \in [6]$. 

We then substitude \eqref{eq: lower bound on normalization factor} and \eqref{eq: lower bound on exponent} into \eqref{eq: density of theta1} to obtain that 
\begin{align}
    &  \phi_{\bo \theta_1^*, \mb \Sigma_1^{(j_1)}}(\mb y) \geq \frac{1}{\tilde\sigma \underline\sigma \bar \sigma^2} \exp(-\frac{1}{2}(1 + o(1)) {\mathsf{SNR}_0^{\mod}}^2) \cdot \frac{1}{(2\pi)^{\frac{p-4}{2}}\tilde\sigma^{p-4}}\exp\big(- \frac{\norm{\bo \triangle_3}_2^2}{2\tilde\sigma^2}\big)\\ 
     \geq  & \frac{1}{(2\pi)^2\bar \sigma^4}\exp\Big(-\frac{1}{2} \big(1 +  \frac{C^{\mathsf{density}}_1}{\mathsf{SNR}_0} + C^{\mathsf{density}}_2 \delta + C^{\mathsf{density}}_3 \frac{\bar\sigma^2 }{\tilde\sigma^2 } +2 \frac{\log\big( \frac{\tilde\sigma}{\bar \sigma}\big)}{{\mathsf{SNR}_0^{\mod}}^2}\big){\mathsf{SNR}_0^{\mod}}^2\Big)  \cdot \frac{1}{(2\pi)^{\frac{p-4}{2}}\tilde\sigma^{p-4}}\exp\big(- \frac{\norm{\bo \triangle_3}_2^2}{2\tilde\sigma^2}\big). 
\end{align}

Following the same argument, we can also prove that
\begin{align}
     &\phi_{\bo \theta_2^*, \mb \Sigma_2^{(j_2)}}(\mb y) \geq \frac{1}{\tilde\sigma \underline\sigma \bar \sigma^2} \exp(-\frac{1}{2}(1 + o(1)) {\mathsf{SNR}_0^{\mod}}^2) \cdot \frac{1}{(2\pi)^{\frac{p-4}{2}}\tilde\sigma^{p-4}}\exp\big(- \frac{\norm{\bo \triangle_3}_2^2}{2\tilde\sigma^2}\big)\\ 
     \geq  & \frac{1}{(2\pi)^2\bar \sigma^4}\exp\Big(-\frac{1}{2} \big(1 +  \frac{C^{\mathsf{density}}_1}{\mathsf{SNR}_0} + C^{\mathsf{density}}_2 \delta + C^{\mathsf{density}}_3 \frac{\bar\sigma^2 }{\tilde\sigma^2 } +2 \frac{\log\big( \frac{\tilde\sigma}{\bar \sigma}\big)}{{\mathsf{SNR}_0^{\mod}}^2}\big){\mathsf{SNR}_0^{\mod}}^2\Big)  \cdot \frac{1}{(2\pi)^{\frac{p-4}{2}}\tilde\sigma^{p-4}}\exp\big(- \frac{\norm{\bo \triangle_3}_2^2}{2\tilde\sigma^2}\big).
\end{align}
\end{proof}

\medskip
\subparagraph*{Control $L_{{\bo \eta}^{(j_1)}}(\hat{\mb z}) + L_{{\bo \eta}^{(j_2)}}(\hat{\mb z})$}
Now we are well prepared to lower-bound the ``separation degree'' $L_{{\bo \eta}^{(j_1)}}(\hat{\mb z}) + L_{{\bo \eta}^{(j_2)}}(\hat{\mb z})$ using $\textsc{Condition 1}$ and $\textsc{Condition 2}$. To begin with, Proposition \ref{proposition: lower bound simplification} gives that 
\begin{align}
    & L_{{\bo \eta}^{(j_1)}}(\hat{\mb z}) + L_{{\bo \eta}^{(j_2)}}(\hat{\mb z})\geq  \int_{ \frac{\mathrm d \bb P_{\bo \theta_2^*, \mb \Sigma_2^{(j_1)} }}{\mathrm d \bb P_{\bo \theta_1^* , \mb \Sigma_1^{(j_1)} }}\leq \dfrac{1}{2},\frac{\mathrm d \bb P_{\bo \theta_1^*, \mb \Sigma_1^{(j_2)}}}{\mathrm d \bb P_{\bo \theta_2^*, \mb \Sigma_2^{(j_2)}}}\leq \dfrac{1}{2} } \min\{p_{\theta_1^{(j_1)}, \mb \Sigma_1^{(j_1)}},p_{\bo \theta_2^*, \mb \Sigma_2^{(j_2)}} \} \mathrm dx \\ 
    &\qquad   + \int_{ \frac{\mathrm d \bb P_{\theta_1^{(j_1)}, \mb \Sigma_1^{(j_1)}}}{\mathrm d \bb P_{\bo \theta_2^*, \mb \Sigma_2^{(j_1)}}}\leq \dfrac{1}{2}, \frac{\mathrm d \bb P_{\bo \theta_2^*, \mb \Sigma_2^{(j_2)}}}{\mathrm d \bb P_{\bo \theta_1^*, \mb \Sigma_1^{(j_2)}}}\leq \dfrac{1}{2}} \min\{p_{\bo \theta_2^*, \mb \Sigma_2^{(j_1)}},p_{\bo \theta_1^*, \mb \Sigma_1^{(j_2)}} \} \mathrm dx \label{eq: lower bound on likelihood ratio (recall)}
\end{align}
holds for an arbitary $\hat{\mb z}$. Focusing on the first term on the right-hand side of \eqref{eq: lower bound on likelihood ratio (recall)}, we shrink the integral region to $R^{(j_1,j_2)}$ and apply $\textsc{Condition 1}$ and $\textsc{Condition 2}$ to obtain that 
\begin{align}
    & L_{{\bo \eta}^{(j_1)}}(\hat{\mb z}) + L_{{\bo \eta}^{(j_2)}}(\hat{\mb z})  
    \geq  \int_{R^{(j_1,j_2)}} f^{\mathsf{lower}}(\mb y) \mathrm d \mb y  
    =  \frac{\pi^2\rho_1^2 \rho_2^2}{(2\pi)^2 \bar \sigma^4 }\exp\big( -(1 + o(1))\frac{{\mathsf{SNR}_0^{\mod}}^2}{2}\big) \\ 
    = & \exp\big( -(1 + o(1))\frac{{\mathsf{SNR}_0^{\mod}}^2}{2}\big), \label{eq: lower bound on L}
\end{align}
where we leverage the conditions that 
\begin{align}
    & \mathsf{SNR}_0 \rightarrow \infty, \qquad  \mathsf{\delta} \rightarrow 0,\qquad  \frac{\bar \sigma^2 }{\tilde\sigma^2} \rightarrow 0, \qquad  \frac{\log\big( \frac{\tilde\sigma}{\bar \sigma}\big)}{{\mathsf{SNR}_0^{\mod}}^2} \rightarrow 0. 
\end{align}

\medskip
\emph{Step 3.4: Upper Bounding the KL Divergence. }
In the sequel, we need to upper bound the KL divergence between ${\bo \eta}^{(j_1)}$ and ${\bo \eta}^{(j_2)}$.  Invoking the conditional property of KL-divergence \cite[Theorem 7.5 (c)]{polyanskiy2024information}, we know that 
\eq{
	 \mathrm{KL}\Big(\frac{1}{2}\bb P_{\theta_1^*, \mb \Sigma_1^{(j_1)}} + \frac{1}{2}\bb P_{\bo \theta_2^*, \mb \Sigma_2^{(j_1)}},
     ~~\frac{1}{2}\bb P_{\bo \theta_1^*, \mb \Sigma_1^{(j_2)}} + \frac{1}{2}\bb P_{\bo \theta_2^*, \mb \Sigma_2^{(j_2)}}\Big) 
     \leq \frac{1}{2} \mathrm{KL}\Big( \bb P_{\theta_1^*, \mb \Sigma_1^{(j_1)}}, \bb P_{\bo \theta_1^*, \mb \Sigma_1^{(j_2)}}\Big) + \frac{1}{2} \mathrm{KL}\Big( \bb P_{\bo \theta_2^*, \mb \Sigma_2^{(j_1)}}, \bb P_{\bo \theta_2^*, \mb \Sigma_2^{(j_2)}}\Big).
}

For the KL divergence of a multivariate Gaussian distribution, one has 
\begin{equation}
 \mathrm{KL}\big( \bb P_{\theta_k^{(j_1)}, \mb \Sigma_k^{(j_1)}}, \bb P_{\theta_k^{(j_2)}, \mb \Sigma_k^{(j_2)}}\big) = \frac{1}{2}\log\frac{|\mb \Sigma_k^{(j_2)}|}{|\mb \Sigma_k^{(j_1)}|} +  \frac{1}{2}\mathrm{Tr}\big({\mb \Sigma_k^{(j_2)}}^{-1}\big( \mb \Sigma_k^{(j_1)} - \mb \Sigma_k^{(j_2)}\big)\big) = \frac{1}{2}\mathrm{Tr}\big({\mb \Sigma_k^{(j_2)}}^{-1}\big( \mb \Sigma_k^{(j_1)} - \mb \Sigma_k^{(j_2)}\big)\big) \label{eq: expression of KL},
\end{equation}
where $|\mb \Sigma_k^{(j_1)}| = |\mb \Sigma_k^{(j_2)}|$ holds by the fact that $\mathrm{det}\left(\begin{matrix}
    \mb A & \mb B \\ 
    \mb C & \mb D
\end{matrix} \right) = \mathrm{det}(\mb A) \mathrm{det}(\mb D - \mb C\mb A^{-1}\mb B)$ given an invertible block $\mb A$ and arbitrary blocks $\mb B, \mb C, \mb D$ in a block matrix. Recall the orthonormal matrices $\mb V_\perp^{(j_1,j_2)}$ and $\tilde{\mb V}^{(j_1,j_2)}$ appearing in \eqref{eq: set R} and \eqref{eq: density of theta1}, respectively, and the properties that 
\begin{align}
    {\mb V_\perp^{(j_1,j_2)}}\t \mb \Sigma_k^{(j_1)}\mb V_\perp^{(j_1,j_2)} = {\mb V_\perp^{(j_1,j_2)}}\t \mb \Sigma_k^{(j_2)}\mb V_\perp^{(j_1,j_2)}, \quad 
    {{}\tilde{\mb V}^{(j_1,j_2)}}\t \mb \Sigma_k^{(j_1)} \mb V_\perp^{(j_1,j_2)} = {{}\tilde{\mb V}^{(j_1,j_2)}}\t \mb \Sigma_k^{(j_2)} \mb V_\perp^{(j_1,j_2)} = \mb 0,\label{eq: orthogonality of V}
\end{align}
we then have 
\begin{align}
    \mathrm{KL}\big( \bb P_{\theta_k^{(j_1)}, \mb \Sigma_k^{(j_1)}}, \bb P_{\theta_k^{(j_2)}, \mb \Sigma_k^{(j_2)}}\big) 
   = \frac{1}{2}\mathrm{Tr}\big(({{}\tilde{\mb V}^{(j_1,j_2)}}\t \mb \Sigma_k^{(j_2)} \tilde{\mb V}^{(j_1,j_2)} )^{-1}{{}\tilde{\mb V}^{(j_1,j_2)}}\t\big( \mb \Sigma_k^{(j_1)} - \mb \Sigma_k^{(j_2)}\big)\tilde{\mb V}^{(j_1,j_2)}\big). 
\end{align}

Invoking the fact that $\mathrm{Tr}(\mb A \mb B) \leq \mathrm{Tr}(\mb A)\norm{\mb B}$, one has 
\longeq{
    &\mathrm{Tr}\big(({{}\tilde{\mb V}^{(j_1,j_2)}}\t \mb \Sigma_k^{(j_2)} \tilde{\mb V}^{(j_1,j_2)} )^{-1}{{}\tilde{\mb V}^{(j_1,j_2)}}\t\big( \mb \Sigma_k^{(j_1)} - \mb \Sigma_k^{(j_2)}\big)\tilde{\mb V}^{(j_1,j_2)}\big) \\ 
    \leq&   \mathrm{Tr}({{}\tilde{\mb V}^{(j_1,j_2)}}\t\big( \mb \Sigma_k^{(j_1)} - \mb \Sigma_k^{(j_2)}\big)\tilde{\mb V}^{(j_1,j_2)})\norm{({{}\tilde{\mb V}^{(j_1,j_2)}}\t \mb \Sigma_k^{(j_2)} \tilde{\mb V}^{(j_1,j_2)} )^{-1}} \\ 
    \leq & 4 \bignorm{{{}\tilde{\mb V}^{(j_1,j_2)}}\t\big( \mb \Sigma_k^{(j_1)} - \mb \Sigma_k^{(j_2)}\big)\tilde{\mb V}^{(j_1,j_2)}}\norm{({{}\tilde{\mb V}^{(j_1,j_2)}}\t \mb \Sigma_k^{(j_2)} \tilde{\mb V}^{(j_1,j_2)} )^{-1}}.
}

From the fact that $\bignorm{\mb \Sigma_k^{(j)}} \leq \tilde\sigma^2, k\in[2],j\in[M]$ for every sufficiently large $n$, it is immediate that $\bignorm{{{}\tilde{\mb V}^{(j_1,j_2)}}\t\big( \mb \Sigma_k^{(j_1)} - \mb \Sigma_k^{(j_2)}\big)\tilde{\mb V}^{(j_1,j_2)}}\leq 2\tilde\sigma^2$ for every sufficiently large $n$. Regarding $\bignorm{\big({{}\tilde{\mb V}^{(j_1,j_2)}}\t \mb \Sigma_k^{(j_2)} \tilde{\mb V}^{(j_1,j_2)}\big)^{-1}}$, one has 
\begin{align}
    \bignorm{({{}\tilde{\mb V}^{(j_1,j_2)}}\t \mb \Sigma_k^{(j_2)} \tilde{\mb V}^{(j_1,j_2)})^{-1}} = \bignorm{{{}\tilde{\mb V}^{(j_1,j_2)}}\t {\mb \Sigma_k^{(j_2)}}^{-1} \tilde{\mb V}^{(j_1,j_2)}} \leq \bignorm{ {\mb \Sigma_k^{(j_2)}}^{-1}} \leq \frac{2 + \alpha'^2}{\underline\sigma^2}
\end{align}
for $k\in[2]$ and every sufficiently large $n$, 
where the first equality holds from the facts in \eqref{eq: inverse of block matrix} and \eqref{eq: orthogonality of V}. Combining these relations gives that 
\begin{align}
    &\mathrm{KL}\big( \bb P_{\theta_k^{(j_1)}, \mb \Sigma_k^{(j_1)}}, \bb P_{\theta_k^{(j_2)}, \mb \Sigma_k^{(j_2)}}\big)\leq \frac{(4 + 2{\alpha'}^2)\tilde\sigma^2}{\underline\sigma^2}, \quad \text{for }k\in[2]. \label{eq: KL divergence between two Gaussian distributions}
\end{align}

We remind that in \eqref{eq: symmetrization argument 3} the minimax rate is lower bounded by 
$$
\frac{1}{4\beta}\inf_{\hat{\mb z}}\sup_{{\bo \eta} \in \{{\bo \eta}^{(j)}\}_{j\in[M]}}L_{{\bo \eta}}(\hat{\mb z}),
$$ 
while $\hat{\mb z }$ could be viewed as a random classifier determined by $\tilde{\mb Y}=(\mb y_2,\cdots, \mb y_n)\t$. We thus consider the KL divergence of the samples $\tilde{\mb Y}$ of size $n-1$. 
Again, the conditional property of KL-divergence allows us to upper bound
the KL divergence between $\bar{\bb P}_{*, {\bo \eta}^{(j_1)}}$ and $\bar{\bb P}_{*, {\bo \eta}^{(j_2)}}$ that 
\eq{
    \mathrm{KL}(\bar{\bb P}_{*, {\bo \eta}^{(j_1)}}, \bar{\bb P}_{*, {\bo \eta}^{(j_2)}})\leq \frac{(n - 1)(4 + 2{\alpha'}^2)\tilde\sigma^2}{\underline\sigma^2}\label{eq: upper bound on KL divergence}
}
 thanks to \eqref{eq: KL divergence between two Gaussian distributions}.

\subsubsection{Putting All the Pieces Together}
\label{subsubsec: wrap up the proof of lower bound}
We now summarize the preceding building blocks to derive the final minimax rate of the problem. We view the marginal distribution of $\tilde{\mb Y}$ under $\frac{1}{2}\bb P_{*,1, {\bo \eta}^{(j)}} + \frac{1}{2}\bb P_{*,2, {\bo \eta}^{(j)}}$ and $L_{{\bo \eta}^{(j)}}(\hat{\mb z})$ as the given distribution and the functions in Lemma~\ref{lemma: fano's method}, respectively. Further, combining \eqref{eq: log M upper bound} with \eqref{eq: upper bound on KL divergence} under the assumption $\tilde\sigma = \omega( \bar\sigma)$ implies that 
\eq{
    \frac{\max_{j_1\neq j_2 \in [M]}  \mathrm{KL}(\bar{\bb P}_{*, {\bo \eta}^{(j_1)}}, \bar{\bb P}_{*, {\bo \eta}^{(j_2)}})}{\log M} 
    \leq 
    \left[{\frac{(n - 1)(4 + 2{\alpha'}^2)\tilde\sigma^2}{\underline\sigma^2}}\right] \Big/
    \left[ c n \frac{\tilde\sigma^{2(1 + \epsilon)}}{\max_{k\in [2]} \norm{\mb S_k^{\mod}}^{1 + \epsilon} }\right] 
    \rightarrow 0
}
since $\tilde \sigma / \bar\sigma \rightarrow 0$. Finally, we apply Lemma~\ref{lemma: fano's method} on \eqref{eq: symmetrization argument 3} with the ``seperation degree'' condition \eqref{eq: lower bound on L} to obtain that 
\begin{equation}
     \inf_{\hat{\mb z}} \sup_{(\mb z^*, \bo \eta) \in \mb \Theta_\alpha} \big(\bb E h(\hat {\mb z}, \mb z^*) -  \frac{1}{4\beta} \big(\bb P_{\bo \theta_1^*, \mb \Sigma_1}[\tilde z = 2] + \bb P_{\bo \theta_2^*, \mb \Sigma_2}[\tilde z = 1] \big) \big)  
    \geq  \exp\left( -(1 + o(1))\frac{{\mathsf{SNR}_0^{\mod}}^2}{2}\right)
\end{equation}
using the condition that $\frac{\log \beta}{{\mathsf{SNR}_0^{\mod}}^2}\rightarrow 0$. 

\subsubsection{Proof of Lemma~\ref{lemma: fano's method}}
    \label{sec: proof of fano's method}
    Consider a uniform prior measure on $\{{\bo \eta}^{(j)}\}_{j=0}^M$ in $\tilde{\mb \Theta}$. By a standard argument, we have 
    \begin{align}
        & \sup_{j\in[M]}\bb E_j[f_j(\mb X)] \geq \frac{1}{M}\sum_{j_0\in[M]}\bb P_{j_0}[f_{j_0}(\mb X)\geq \gamma / 2] \\ 
        \geq & \frac{1}{M}\sum_{j_0\in[M]}\bb P_{j_0}\big[f_{j_0}(\mb X) \neq \min_{j\in[M]} f_j(\mb X)\big]  
        = \frac{1}{M}\sum_{j_0\in[M]}\bb P_{j_0}\big[\hat j(\mb X) \neq j_0\big],\label{eq: multiple testing error}
    \end{align}
    where $\hat j(\mb x)\coloneqq \argmin_{j\in[M]}f_j(\mb x)$ and the second inequality follows from the fact that \eq{
    \{f_j(\mb X)< \gamma / 2\} \subseteq \{ f_{j_0}(\mb X)=  \min_{j\in[M]} f_j(\mb X)\}.} 

    By Fano's lemma \cite[Corollary 2.6]{10.5555/1522486}, the multiple testing error \eqref{eq: multiple testing error} is lower bounded as follows: 
    \eq{
        \frac{1}{M}\sum_{j_0\in[M]}\bb P_{j_0}\big[\hat j(\mb X) \neq j_0\big] \geq \frac{\log M - \log 2}{\log(M-1)} \geq  c_0 \frac{(M - 1) \log M}{M\log(M - 1)}\geq  c > 0. 
    }
    for some sufficiently small $c_0$ and $c$.

\subsection{Proof of Proposition~\ref{proposition: lower bound simplification}} 
\label{subsec: proof of proposition: lower bound simplification}
The proof idea shares a spirit similar to \citet[Theorem 5]{bing2023optimal}. 
Since the marginal distribution of $\mb y_1$ under $\bb P_{*,1, {\bo \eta}}$ is exactly $\bb P_{\bo \theta_1,\mb \Sigma_1}$, it follows by looking into the event that $\tilde{\mb z}$ is not equal to $\hat z_1$ that
    \begin{align}
        & \bb P_{*,1, {\bo \eta}}[\hat z_1 =2 ] + \bb P_{*,2, {\bo \eta}}[\hat z_1 =1] -   \bb P_{\mb y\sim \mc N(\bo \theta_1^*, \mb \Sigma_1)}[\tilde z(\mb y) = 2] - \bb P_{\mb y\sim \mc N(\bo \theta_2^*, \mb \Sigma_2)}[\tilde z(\mb y) = 1] \\ 
        = &
        \int_{\{\hat z_1 = 2\}} 1 \mathrm d\bb P_{*,1, {\bo \eta}} +\int_{\{\hat z_1 = 1\}} 1 \mathrm d\bb P_{*,2, {\bo \eta}}  -  \int_{\{\tilde z = 2\}} 1 \mathrm d\bb P_{*,1, {\bo \eta}} -\int_{\{\tilde z = 1\}} 1 \mathrm d\bb P_{*,2, {\bo \eta}} 
        \\  
        = & \int_{\{\hat z_1 = 2, \tilde z = 1\}} 1 \mathrm d(\bb P_{*,1, {\bo \eta}}-\bb P_{*,2, {\bo \eta}})  + \int_{\{\hat z_1 = 1, \tilde z = 2\}} 1 \mathrm d(\bb P_{*,2, {\bo \eta}} - \bb P_{*,1, {\bo \eta}}) \\ 
        = & \int_{\{\hat z_1 = 2, \tilde z = 1\}} \big(1 - \frac{\phi_{\bo \theta^*_2, \mb \Sigma_2}}{\phi_{\bo \theta^*_1, \mb \Sigma_1}} \big)  \mathrm d\bb P_{*,1, {\bo \eta}}  + \int_{\{\hat z_1 = 1, \tilde z = 2\}} \big(1 - \frac{\phi_{\bo \theta^*_1, \mb \Sigma_1}}{\phi_{\bo \theta^*_2, \mb \Sigma_2}} \big)\mathrm d\bb P_{*,2, {\bo \eta}}  \\ 
        \geq & \frac{1}{2}\Big(\bb P_{*,1, {\bo \eta}}\big[\hat z_1 = 2, \tilde z = 1,  \frac{\phi_{\bo \theta^*_2, \mb \Sigma_2}}{\phi_{\bo \theta^*_1, \mb \Sigma_1}}\leq \frac{1}{2}\big] + \bb P_{*,2, {\bo \eta}}\big[\hat z_1 = 1, \tilde z = 2,  \frac{\phi_{\bo \theta^*_1, \mb \Sigma_1}}{\phi_{\bo \theta^*_2, \mb \Sigma_2}}\leq \frac{1}{2} \big] 
        \Big) \\ 
         & \quad = \frac{1}{2}\Big(\bb P_{*,1, {\bo \eta}}\big[\hat z_1 = 2,  \frac{\phi_{\bo \theta^*_2, \mb \Sigma_2}}{\phi_{\bo \theta^*_1, \mb \Sigma_1}}\leq \frac{1}{2}\big] + \bb P_{*,2, {\bo \eta}}\big[\hat z_1 = 1,  \frac{\phi_{\bo \theta^*_1, \mb \Sigma_1}}{\phi_{\bo \theta^*_2, \mb \Sigma_2}}\leq \frac{1}{2} \big] 
        \Big) . 
        \label{eq: lower bound simplification 1}
    \end{align}

    Now given ${\bo \eta}^{(j_1)} = (\bo \theta_1^*, \bo \theta_2^*, \mb \Sigma_1^{(j_1)}, \mb \Sigma_2^{(j_1)}),{\bo \eta}^{(j_2)} =  (\bo \theta_1^*, \bo \theta_2^*, \mb \Sigma_1^{(j_2)}, \mb \Sigma_2^{(j_2)}) \in \mb \Theta$, we denote the likelihood ratio estimator for ${\bo \eta}^i$ by $\tilde z^i$ for $i = 1,2$. Invoking the simple fact that 
    \begin{align}
        & \big\{\hat z_1 = 2,  \frac{\phi_{\bo \theta_2^*, \mb \Sigma^{(j_1)}_2}}{\phi_{\bo \theta^*_1, \mb \Sigma^{(j_1)}_1}}\leq \frac{1}{2}\big\}\cup  \big\{\hat z_1 = 1,  \frac{\phi_{\bo \theta^*_1, \mb \Sigma^{(j_2)}_1}}{\phi_{\bo \theta^*_2, \mb \Sigma^{(j_2)}_2}}\leq \frac{1}{2} \big\} \supseteq \big\{\frac{\phi_{\bo \theta_2^*, \mb \Sigma^{(j_1)}_2}}{\phi_{\bo \theta^*_1, \mb \Sigma^{(j_1)}_1}}\leq \frac{1}{2},  \frac{\phi_{\bo \theta^*_1, \mb \Sigma^{(j_2)}_1}}{\phi_{\bo \theta^*_2, \mb \Sigma^{(j_2)}_2}}\leq \frac{1}{2}\big\}, 
    \end{align}
    for $i\neq j\in[2]$, 
    it holds that 
    \begin{align}\label{eq: inequality of misspecification probability}
    & \bb P_{*,1, {\bo \eta}^{(j_1)}}\big[\hat z_1 = 2, \frac{\phi_{\bo \theta_2^*, \mb \Sigma_2^{(j_1)} }}{\phi_{\bo \theta_1^* , \mb \Sigma_1^{(j_1)} }}\leq \frac{1}{2}\big] 
    + \bb P_{*,2, {\bo \eta}^{(j_1)}}\big[\hat z_1 = 1, \frac{\phi_{\bo \theta_1^*, \mb \Sigma_1^{(j_1)}}}{\phi_{\bo \theta_2^*, \mb \Sigma_2^{(j_1)}}}\leq \frac{1}{2}\big] \\& +  \bb P_{*,1, {\bo \eta}^{(j_2)}}\big[\hat z_1 = 2, \frac{\phi_{\bo \theta_2^*, \mb \Sigma_2^{(j_2)}}}{\phi_{\bo \theta_1^*, \mb \Sigma_1^{(j_2)}}}\leq \frac{1}{2}\big] + \bb P_{*,2, {\bo \eta}^{(j_2)}}\big[\hat z_1 = 1, \frac{\phi_{\bo \theta_1^*, \mb \Sigma_1^{(j_2)}}}{\phi_{\bo \theta_2^*, \mb \Sigma_2^{(j_2)}}}\leq \frac{1}{2}\big] \\
     \geq & \int_{ \frac{\phi_{\bo \theta_2^*, \mb \Sigma_2^{(j_1)} }}{\phi_{\bo \theta_1^* , \mb \Sigma_1^{(j_1)} }}\leq \frac{1}{2},\frac{\phi_{\bo \theta_1^*, \mb \Sigma_1^{(j_2)}}}{\phi_{\bo \theta_2^*, \mb \Sigma_2^{(j_2)}}}\leq \frac{1}{2} } \min\{\phi_{\bo \theta_1^*, \mb \Sigma_1^{(j_1)}}(\mb x),\phi_{\bo \theta_2^*, \mb \Sigma_2^{(j_2)}}(\mb x) \} \mathrm d\mb x \\ 
     &+\int_{ \frac{\phi_{\bo \theta_1^*, \mb \Sigma_1^{(j_1)}}}{\phi_{\bo \theta_2^*, \mb \Sigma_2^{(j_1)}}}\leq \frac{1}{2}, \frac{\phi_{\bo \theta_2^*, \mb \Sigma_2^{(j_2)}}}{\phi_{\bo \theta_1^*, \mb \Sigma_1^{(j_2)}}}\leq \frac{1}{2}} \min\{\phi_{\bo \theta_2^*, \mb \Sigma_2^{(j_1)}}(\mb x),\phi_{\bo \theta_1^*, \mb \Sigma_1^{(j_2)}}(\mb x) \} \mathrm d\mb x 
    \end{align}
    which leads to the coclusion combined with \eqref{eq: lower bound simplification 1}.

    \subsection{Proof of Theorem~\ref{theorem: gaussian lower bound with K components}}
\label{subsec: proof of theorem: gaussian lower bound with K components}
    We provide a general version of Theorem~\ref{theorem: gaussian lower bound with K components}, while Theorem~\ref{theorem: gaussian lower bound with K components} is an immediate conclusion of the general one. 
\begin{theorem*}[Minimax Lower Bound for $K$-component Gaussian Mixtures]\label{theorem: gaussian lower bound with K components (general)}
    Consider the $K$-component Gaussian mixture model and the parameter space $\mb \Theta_{\alpha, K}$ with $1<\alpha<\frac43$. Given $\mathsf{SNR}_0 \rightarrow \infty$, 
    $\frac{K(\log \beta \vee 1)}{{\mathsf{SNR}_0^{\mod}}^2} \rightarrow 0$, one has 
	\eq{
	\inf_{\hat{\mb z}}\sup_{(\mb z^*, \bo \eta ) \in \mb\Theta_{\alpha,K}}\bb E[h(\hat{\mb z}, \mb z^*)] \geq \exp\left(-(1 + o(1))\frac{{\mathsf{SNR}_0^{\mod}}^2}{2}\right),
	}
    if $\tilde \sigma = \omega(\bar\sigma)$, $\bar\sigma / \underline\sigma = O(1)$, $\log(\tilde\sigma^2 / \underline\sigma^2) = o({\mathsf{SNR}_0^{\mod}}^2)$, and $n\tilde\sigma^{2(1+\epsilon)} = o(p \bar\sigma^{2(1+\epsilon)})$ for some constant $\epsilon > 0$. 
\end{theorem*}
    
    The basic idea of the proof is to focus on the most hard-to-distinguish pair of clusters among the $K$ clusters. Reducing the problem into distinguishing these two components, the remaining parts follow a similar route in the proof of Theorem~\ref{theorem: gaussian lower bound}. One subtle thing to note is that the treatment to lower bound the probability in a subregion is different from the proof of Theorem~\ref{theorem: gaussian lower bound}.

    To begin with, we fix an aribtrary $\mb z^{(0)} \in \mb \Theta_{z, K}$ and choose a subset $\mc B_m \subset \mc I_m(\mb z^{(0)})$ such that $|\mc B_m| = n_m^* - \lfloor \frac{n}{8\beta K}\rfloor$ for $m = 1,2$. With $\mc B \coloneqq \cup_{m=1}^2 \mc B_m\cup \big( \cup_{i = 3}^K \mc I_i(\mb z^{(0)})\big)$, we define $\mb Z_{\mc B} = \{\mb z\in \mb\Theta_{z, K}: z_i = z_i^* , \forall i \in \mc B\}$.
    For notational simplicity, denote $ \bb P_{\mb y\sim \mc N(\bo \theta_1^*, \mb \Sigma_1)}[\tilde z(\mb y) = 2] =: \bb P_{\bo \theta_1^*, \mb \Sigma_1}[\tilde z = 2]$ and denote $\bb P_{\bo \theta_2^*, \mb \Sigma_2}[\tilde z = 1]$ similarly.
    Following the procedure in \eqref{eq: reduction to subset of theta}, we have 
    \longeq{
        &\inf_{\hat{\mb z}}\sup_{(\mb z^*, \{\bo \theta_k^*\}_{k\in[K]}, \{\mb \Sigma_k\}_{k\in[K]} )\in \mb\Theta_{\alpha, K}}\bb Eh(\hat{\mb z}, \mb z^*)\\
         \geq &\frac{1}{4\beta K} \inf_{\hat{\mb z}} \sup_{(\{\bo \theta_k^*\}, \{\mb \Sigma_k\}) \in \tilde{\mb \Theta}_{0, K}} \frac{1}{|\mb Z_{\mc B}|}\sum_{\mb z^* \in \mb Z_{\mc B}}\Big(\frac{1}{|\mc B^\complement|}\sum_{\mb z^* \in \mb Z_{\mc B}} \big(\frac{1}{\mc B^\complement} \sum_{i\in\mc B^\complement} \bb P[\hat z_i \neq z_i^{(0)}]   - (\bb P_{\bo\theta_1^*, \mb \Sigma_1}[\tilde z = 1] + \bb P_{\bo \theta_2^*, \mb \Sigma_2}[\tilde z = 1])\big) \Big).
    }
    For now, the minimax lower bound has been reduced to a form only related to the first two clusters; that is, we are supposed to focus on the cases where this pair is the hardest pair to be distinguished. Provided the assignment subset $\mb Z_{\mc B}$, the symmetrization argument in Step 2 of the proof of Theorem~\ref{theorem: gaussian lower bound} can be applied to the above expression again. Hence, we have 
    \begin{align}
        &\inf_{\hat{\mb z}}\sup_{(\mb z^*, \{\bo \theta_k^*\}_{k\in[K]}, \{\mb \Sigma_k\}_{k\in[K]} )\in \mb\Theta_{\alpha, K}}\bb Eh(\hat{\mb z}, \mb z^*)\\ 
        \geq & \frac{1}{4 \beta K}\inf_{\hat{\mb z}} \sup_{{\bo \eta}\in \tilde{\mb \Theta}_{\alpha,K}} 
        \Big[\bb P_{*,1, {\bo \eta}}[\hat z^{\mathsf{sym}}_1 =2 ] + \bb P_{*,2, {\bo \eta}}[\hat z^{\mathsf{sym}}_1 =1]  - \big(\bb P_{\bo \theta_1^*, \mb \Sigma_1}[\tilde z = 2] + \bb P_{\bo \theta_2^*, \mb \Sigma_2}[\tilde z = 1] \big)\Big] \\ 
        = & \frac{1}{4 \beta K}\inf_{\hat{\mb z}} \sup_{{\bo \eta}\in \tilde{\mb \Theta}_{\alpha, K}} \bb E
        \Big[\bb P_{*,1, {\bo \eta}}[\hat z^{\mathsf{sym}}_1 =2|\tilde{\mb Y}] + \bb P_{*,2, {\bo \eta}}[\hat z^{\mathsf{sym}}_1 =1 |\tilde{\mb Y}]   - \big(\bb P_{\bo \theta_1^*, \mb \Sigma_1}[\tilde z = 2] + \bb P_{\bo \theta_2^*, \mb \Sigma_2}[\tilde z = 1] \big)\Big],  \label{eq: corollary lower bound simplification}
    \end{align}
    where $\tilde{\mb Y} = (\mb y_2, \cdots,\mb y_n)\t$ and we analogously denote by $\bb P_{*, i, {\bo \eta}}$ the marginal probability measure with a uniform prior over $\{\mb z\in \mc Z_{\mc B}: z_1 = i\}$ for $i = 1,2$. Analogous to the previous definition in Theorem~\ref{theorem: gaussian lower bound}, we define the function $L_{{\bo \eta}}(\hat{\mb z})$ as 
    \begin{align}
        & L_{{\bo \eta}}(\hat{\mb z}) \coloneqq \bb P_{*,1, {\bo \eta}}[\hat z^{\mathsf{sym}}_1 =2|\tilde{\mb Y}] + \bb P_{*,2, {\bo \eta}}[\hat z^{\mathsf{sym}}_1 =1 |\tilde{\mb Y}]    - \big(\bb P_{\bo \theta_1^*, \mb \Sigma_1}[\tilde z = 2] + \bb P_{\bo \theta_2^*, \mb \Sigma_2}[\tilde z = 1] \big)
    \end{align}

    In order to apply the reduction scheme in Step 3.1 in the proof of Theorem \eqref{theorem: gaussian lower bound} so as to lower bound the supremum of the  expectation on the right-hand side of \eqref{eq: corollary lower bound simplification}, the core components, which respectively correspond to Step 3.2, Step 3.3, and Step 3.4 in the proof of Theorem~\ref{theorem: gaussian lower bound}, are concisely listed as follows: 
    \begin{itemize}
        \item \emph{Step 3.2*.} We shall present a parameter subset that represents the hardness of this clustering task. 
        \item \emph{Step 3.3*.} Provided a well-designed parameter subset $\{{\bo \eta}^{(j)}\}_{j\in[M]}$ of $\tilde{\mb \Theta}_{\alpha,K}$ in \eqref{eq: parameter space for K-component cases}, we show that for an arbitrary estimator $\hat{\mb z}$ and $j_1\neq j_2\in[M]$, one has 
        \begin{align}
            & L_{{\bo \eta}^{(j_1)}}(\hat{\mb z}) + L_{{\bo \eta}^{(j_2)}}(\hat{\mb z}) \geq   \exp\left(-(1 + o(1))\frac{{\mathsf{SNR}_0^{\mod}}^2}{2}\right).\label{eq: lower bound for two components out of K components 0}
        \end{align}
        The main technique toward \eqref{eq: lower bound for two components out of K components 0} lies in Proposition~\ref{proposition: lower bound simplification} in combination with a region similar to \eqref{eq: set R}. 
        \item \emph{Step 3.4*.} Lastly, we will prove that the KL-divergence between two arbitrary components in the parameter subset is appropriately controlled as Step 3.4 in the previous proof. 
    \end{itemize}

    The following parts are devoted to presenting the details of these steps. 
    \emph{Step 3.2*. }
    We start by constructing a parameter subset, in which each component shares the same centers and covariance matrices except for the first two components, and the proposed signal-to-noise-ratio between the first two components achieves $\mathsf{SNR}_0$. Without loss of generality, we let $\mb S_k^{\mod} = \underline\sigma^2 \mb I_K$. 
    
    Since $p - K \geq \frac{p}{2}$ for every sufficiently large $n$, we can always obtain a packing on $\bb S^{p - K-1}$ by appending zeroes to a packing on $\bb S^{\frac{p}{2}}$ for every sufficiently large $n$. 
    Following the same way as in the proof of Theorem~\ref{theorem: gaussian lower bound} (especially \eqref{eq: VG bound} and \eqref{eq: log M upper bound}), an almost-orthogonal packing on $\bb S^{p-K -1 }$ is given as $\{\mb v^{(j)}\}$ for $j\in[M]$,
    which satisfies that (i):
         \eq{
            |{\mb v^{(j_1)}}\t \mb v^{(j_2)}| \leq \frac{n^\frac{1}{2}\tilde\sigma^{1+\epsilon}}{p^{\frac{1}{2}} \underline\sigma^{1 + \epsilon}} \quad\text{for } j_1\neq j_2\in[M],
          }
         and (ii):
    \eq{
        \log M \geq c n (\tilde\sigma/\underline\sigma)^{2(1+\epsilon)}
    \label{eq: log M lower bound with K components}
    } 
    for some constant $c>0$. 

    Fixing an arbitrary $p\times p$ orthonormal matrix $(\mb V^*, \mb V^*_\perp) \in O(p)$, the parameter subset is defined as 
    $\{{\bo \eta}^{(j)}\}_{j\in[M]} = \big\{(\{\bo \theta_k^{*}\}_{k\in[K]}, \{\mb \Sigma_k^{(j)}\}_{k\in[K]})\big\}_{j\in[M]}$ where 
    \begin{align}
        & \bo \theta_k^{*} \coloneqq \sqrt{2}\underline\sigma{\mathsf{SNR}^{\mod}_0}^{\frac{1}{2}}\mb V^* \mb e_k, \quad \mb \Sigma_k^{(j)} \coloneqq  (\mb V^*, \mb V^*_\perp)(\mb \Omega_k^{(j)})^{-1}(\mb V^*, \mb V^*_\perp)\t , 
    \end{align}
    with \begin{align}
        & \mb \Omega^{(j)}_k \coloneqq \left( \begin{matrix}
             \frac{1}{\underline\sigma^2}\mb I_{K} +  \frac{1}{\underline\sigma^2}\alpha'^2 \mb w_0 \mb w_0\t &  \frac{1}{\underline\sigma^2}\alpha'\mb w_0 {\mb v^{(j)}}\t\\ 
            \frac{1}{\underline\sigma^2}\alpha'\mb v^{(j)}\mb w_0\t & \frac{1}{\tilde\sigma^2}\big(\mb I_{p-K} - \mb v^{(j)}{\mb v^{(j)}}\t  \big) +  \frac{1}{\underline\sigma^2}\mb v^{(j)}{\mb v^{(j)}}\t
        \end{matrix}\right), \quad \text{ for $k = 1,2$}, \\ 
        &  \mb \Omega^{(j)}_k = \left(
            \begin{matrix}
                 \frac{1}{\underline\sigma^2}{\alpha''}^2\mb I_{K} & \mb 0 \\ 
                \mb 0 & \frac{1}{\tilde\sigma^2}\mb I_{p-K}
            \end{matrix}
        \right)    
        , \quad \text{ for $k = 3, \cdots, K$}. \label{eq: trivial omega}
    \end{align}
    Here $\alpha' \coloneqq  12\alpha, \alpha''\coloneqq \frac{8\alpha}{2 - \frac{3}{2}\alpha}$ are quantities related to $\alpha$, the vector $\mb e_k$ denotes the $k$-th canonical basis vector in $\bb R^K$ for $k\in[K]$, and $\mb w_0 = \frac{1}{\sqrt{2}}(-1, 1, 0, \cdots, 0) \in \bb R^K$. It is not hard to verify that ${\mb V^{*\top}} \mb \Sigma_k^{(j)}\mb V^* = \mb I_K$. 
    
    We are left with verifying that $\{{\bo \eta}^{(j)}\}\in \tilde{\mb \Theta}_{\alpha, K}$. 
    Before proceeding, we write 
    \begin{align}
        & \mathsf{SNR}^{\mod}(\{\bo \theta_k^*\}_{k\in[K]}, \{\mb \Sigma_k\}_{k\in[K]}) = \min_{a\neq b\in[K]}\mathsf{SNR}^{\mod}_{a,b}(\{\bo \theta_a^*, \bo \theta_b^*\}, \{\mb \Sigma_a, \mb \Sigma_b\}), \\ 
        & \snrfull(\{\bo \theta_k^*\}_{k\in[K]}, \{\mb \Sigma_k\}_{k\in[K]})= \min_{a\neq b\in[K]}\mathsf{SNR}^{\mathsf{full}}_{a,b}(\{\bo \theta_a^*, \bo \theta_b^*\}, \{\mb \Sigma_a, \mb \Sigma_b\}),\label{eq: pairwise definition of SNR's}
    \end{align}
    where the functions $\mathsf{SNR}_{a,b}$ and ${\snrfull}_{a,b}$ are naturally defined as 
    \begin{align}
        & \mathsf{SNR}^{\mod}_{a,b}(\{\bo \theta_a^*, \bo \theta_b^*\}, \{\mb \Sigma_a, \mb \Sigma_b\}) \coloneqq \min_{\mb x \in \bb R^2}\big\{\mb x\t {\mb S^*_a}^{-1}\mb x:\  \mb x\t\big({\mb S_b^*}^{-1} - {\mb S_a^*}^{-1}\big) \mb x \\ 
        & \qquad - 2 \mb x\t {\mb S^*_b}^{-1}\big( \mb w^*_b - \mb w^*_a\big)  +\big( \mb w^*_b - \mb w^*_a\big)\t {\mb S_b^*}^{-1}\big( \mb w^*_b - \mb w^*_a\big) = 0\big\} , 
        \\ 
        & \mathsf{SNR}^{\mathsf{full}}_{a,b}(\{\bo \theta_a^*, \bo \theta_b^*\}, \{\mb \Sigma_a, \mb \Sigma_b\})  
        \coloneqq  \min_{\mb x \in \bb R^p}\big\{\mb x\t \mb \Sigma_a^{-1}\mb x: \ \frac{1}{2}\mb x\t(\mb \Sigma_b^{-1} - \mb \Sigma_a^{-1})\mb x + \mb x\t \mb \Sigma_b^{-1}(\bo\theta^*_a - \bo \theta^*_b)
        \\
       &\qquad 
       + \frac{1}{2}(\bo \theta_a^* - \bo \theta_b^*)\t \mb \Sigma_b^{-1}(\bo \theta_a^* - \bo \theta_b^*) - \frac{1}{2}\log|\mb \Sigma_a^*| + \frac{1}{2}\log |\mb \Sigma_b^*| = 0\big\}.
    \end{align}

    We claim the following fact, whose proof is postponed to the end of the whole proof of this corollary. 
    \begin{claim} \label{claim: SNR relations for K components}
        Given the parameter subset $\{{\bo \eta}^{(j)}\}_{j\in[M]}$ defined above, it holds for every sufficiently large $n$ that
    \longeq{
        &     \mathsf{SNR}^{\mod}_{a,b}(\{\bo \theta_a^{*},\bo \theta_b^{*}\}, \{\mb \Sigma_a, \mb \Sigma_b\}) =\mathsf{SNR}_0^{\mod},\quad  \text{if $a =1$, $b= 2$ or $a= 2$, $b=1$},
        \\ 
        & \mathsf{SNR}_{a,b}^{\mod}(\{\bo \theta_a^{*},\bo \theta_b^{*}\}, \{\mb \Sigma_a, \mb \Sigma_b\})
            > \mathsf{SNR}_0^{\mod}, \quad  \text{otherwise}, \\ 
        &\text{and } -\log(\rbayes)  \geq \alpha^2 {\mathsf{SNR}_0^{\mod}}^2  .  \label{eq: claims in K components lower bound}
    }
\end{claim}

    In light of \eqref{eq: pairwise definition of SNR's} and \eqref{eq: claims in K components lower bound}, for each ${\bo \eta}^{(j)}$ with $j \in[M]$, we have  $$\mathsf{SNR}^{\mod}(\{\bo \theta^*_k\}_{k\in[K]},\{\mb \Sigma_k\}_{k\in[K]}) = \mathsf{SNR}_0^{\mod}$$ and 
    $$\snrfull(\{\bo \theta^*_k\}_{k\in[K]},\{\mb \Sigma_k\}_{k\in[K]}) \geq \alpha \mathsf{SNR}_0^{\mod},$$
    and ${\bo \eta}^{(j)}$ is therefore contained in $\tilde{\mb \Theta}_{\alpha, K}$. 

    \emph{Step 3.3*. } 
    As sketched above, applying Proposition~\ref{proposition: lower bound simplification} to $L_{{\bo \eta}^{(j_1)}} + L_{{\bo \eta}^{(j_2)}}$ for $j_1\neq j_2\in[M]$
    gives that 
    \longeq{
        & L_{{\bo \eta}^{(j_1)}} + L_{{\bo \eta}^{(j_2)}} \geq  \int_{ \frac{\mathrm d \bb P_{\bo\theta_2^*, \mb \Sigma_2^{(j_1)} }}{\mathrm d \bb P_{\bo\theta_1^* , \mb \Sigma_1^{(j_1)} }}\leq \frac{1}{2},\frac{\mathrm d \bb P_{\bo\theta_1^*, \mb \Sigma_1^{(j_2)}}}{\mathrm d \bb P_{\bo\theta_2^*, \mb \Sigma_2^{(j_2)}}}\leq \frac{1}{2} } \min\{p_{\bo\theta_1^*, \mb \Sigma_1^{(j_1)}},p_{\bo\theta_2^*, \mb \Sigma_2^{(j_2)}} \} \mathrm dx \\ 
        & \qquad  + \int_{ \frac{\mathrm d \bb P_{\bo\theta_1^*, \mb \Sigma_1^{(j_1)}}}{\mathrm d \bb P_{\bo\theta_2^*, \mb \Sigma_2^{(j_1)}}}\leq \frac{1}{2}, \frac{\mathrm d \bb P_{\bo\theta_2^*, \mb \Sigma_2^{(j_2)}}}{\mathrm d \bb P_{\bo\theta_1^*, \mb \Sigma_1^{(j_2)}}}\leq \frac{1}{2}} \min\{p_{\bo\theta_2^*, \mb \Sigma_2^{(j_1)}},p_{\bo\theta_1^*, \mb \Sigma_1^{(j_2)}} \} \mathrm dx.
        \label{eq: lower bound for two components out of K components}
    }
    
    Next, we shall parse the inequality \eqref{eq: lower bound for two components out of K components} by considering the following region: 
    \begin{align}
        & R^{(j_1,j_2)}_K \\ 
        \coloneqq  & \big\{ 
        \mb V^*(\mb w_* + \bo \triangle_1) 
        + \mb V_\perp^* \mb V^{(j_1,j_2)} \big({\mb V^{(j_1,j_2)}}\t (\mb z_*^{1,(j_1)} + \mb z_*^{2,(j_2)}) + \bo \triangle_2 \big) + \mb V_\perp^{(j_1,j_2)}\bo \triangle_3, \\ 
        & \quad \norm{ \mc P_{1:2}(\bo \triangle_1)}_2 \leq \rho_1 \underline\sigma , \norm{\bo\triangle_2}_2 \leq \rho_2\underline\sigma, \bo \triangle_3 \in \bb R^{p-K-2}
        \big\}, 
    \end{align}
    where $\mc P_{1:2}(\mb x)$ denotes the first two entries of a vector $\mb x$ and $\rho_1,\rho_2$ are some positive constants, $\mb V^{(j_1,j_2)}\in O(p-K, 2)$ denotes an orthonormal matrix whose column space aligns with the one of $(\mb v^{(j_1)}, \mb v^{(j_2)})\in \bb R^{(p-K)\times 2}$, and $\mb V_\perp^{(j_1,j_2)}\in O(p, p-K - 2)$ denotes an orthonomal matrix perpendicular to $\big(\mb V^*, \mb V^*_\perp\mb V^{(j_1,j_2)}\big)$. Provided the region $R^{(j_1,j_2)}_K$, the following conditions serve as analogs of \textsc{Condition 1} and \textsc{Condition 2} in the proof of Theorem~\ref{theorem: gaussian lower bound}. 
    \begin{itemize}
        \item \textsc{Condition 1*}: $$\frac{\phi_{\bo \theta_2^*, \mb \Sigma_2^{(j_1)} }}{\phi_{\bo \theta_1^* , \mb \Sigma_1^{(j_1)} }}\leq \frac{1}{2}
        \text{~ and~ }
        \frac{\phi_{\bo \theta_1^*, \mb \Sigma_1^{(j_2)}}}{\phi_{\bo \theta_2^*, \mb \Sigma_2^{(j_2)}}}\leq \frac{1}{2}.$$ 
        \item \textsc{Condition 2*}: the minimum of $\phi_{\bo \theta_1^*, \mb \Sigma_1^{(j_1)}}$ and $\phi_{\bo \theta_2^*, \mb \Sigma_2^{(j_2)}}$ is lower bounded by 
        \longeq{
            &  \min \{\phi_{\bo \theta_1^*, \mb \Sigma_1^{(j_1)}}\big( \mb x\big), \phi_{\bo \theta_2^*, \mb \Sigma_2^{(j_2)}}\big(\mb x\big)\}\\ 
        \geq & \frac{1}{(2\pi)^2\underline \sigma^4}\exp\Big(-\frac{1}{2} \big(1 +  \frac{C^{\mathsf{density}}_1}{\mathsf{SNR}_0^{\mod}} + C^{\mathsf{density}}_2 \delta + C^{\mathsf{density}}_3 \frac{\underline\sigma^2 }{\tilde\sigma^2 } +C^{\mathsf{density}}_4 \frac{\log\big(\tilde{\sigma} / \underline\sigma\big)}{{\mathsf{SNR}_0^{\mod}}^2}\big){\mathsf{SNR}_0^{\mod}}^2\Big) \\ 
     & \cdot \frac{1}{(2\pi)^{\frac{p-K-2}{2}}\tilde\sigma^{p-K-2}}\exp\Big(- \frac{\big\|{\mb V^{(j_1,j_2)}_\perp}\t \mb x\Big\|_2^2}{2\tilde\sigma^2}\Big) \cdot \frac{1}{(2\pi)^{\frac{K-2}{2}} \underline \sigma^{K-2}} \exp(- \frac{\bignorm{{\mb V_{-2}^*}\t \mb x}_2^2}{2\underline\sigma^2})
        }
        for some constants $C^{\mathsf{density}}_i> 0 $, $i \in [4]$. 
    \end{itemize}

    We aim to verify the conditions above for every $\mb x \in R^{(j_1,j_2)}_K$. We denote by $\mb V^*_{2}$ and $\mb V^*_{-2}$ the first two columns and the last $K-2$ columns of $\mb V^*$, respectively.

    \medskip
    \subparagraph*{Verifying \textsc{Condition 1*}} 

    An observation is that the weight in the subspace spanned by the last $K-2$ columns of $\mb V^*$ (denoted by $\mb V^*_{-2}$) does not contribute to the likelihood ratio; to be specific, from the construction of the subset of the covariance matrices one can infer that 
    \begin{align}
        \frac{\mathrm d \bb P_{\bo \theta_1^*, \mb \Sigma_1^{(j)}}(\mb x_1 + \mb V^*_{-2}\mb x_2 )}{\mathrm d \bb P_{\bo \theta_2^*, \mb \Sigma_2^{(j)}}(\mb x_1 + \mb V^*_{-2} \mb x_2)} = \frac{\mathrm d \bb P_{\bo \theta_1^*, \mb \Sigma_1^{(j)}}(\mb x_1  )}{\mathrm d \bb P_{\bo \theta_2^*, \mb \Sigma_2^{(j)}}(\mb x_1 )} 
    \end{align} for every $\mb x_1 \in \bb R^{p}$ and $\mb x_2 \in \bb R^{K-2}$. Therefore, for an $\mb x \in R^{(j_1,j_2)}_K$, \textsc{Condition 1*} is equivalent to $\big(\mb V^*_2, \mb V^*_\perp\big)\t \mb x$ satisfying the likelihood ratio conditions
    \longeq{& 
    \frac{\phi_{\theta_2^{*, \text{new}}, \mb \Sigma_2^{(j_1), \text{new}} }}{\phi_{\theta_1^{*, \text{new}} , \mb \Sigma_1^{(j_1), \text{new}} }}\leq \frac{1}{2},\quad \frac{\phi_{\theta_1^{*, \text{new}}, \mb \Sigma_1^{(j_2), \text{new}}}}{\phi_{\theta_2^{*, \text{new}}, \mb \Sigma_2^{(j_2), \text{new}}}}\leq \frac{1}{2}
    \label{eq: likelihood ratio conditions in Condition 1*}
    }
    with new centers $\bo \theta_1^{*,\text{new}} \coloneqq \sqrt{2} \mathsf{SNR}^{\frac{1}{2}} \cdot (1,0, 0,\cdots,0)\t \in \bb R^{p-K+ 2}$ and $\bo \theta_2^{*,\text{new}} \coloneqq \sqrt{2} \mathsf{SNR}^{\frac{1}{2}} \cdot (0,1, 0,\cdots,0)\t \in \bb R^{p - K+ 2}$, and new covariance matrices $\mb \Sigma_1^{(j_1),\text{new}}$ and $\mb \Sigma_2^{(j_2),\text{new}}$ are given by
    \begin{align}
        & \mb \Sigma_1^{(j),\text{new}} = \mb \Sigma_2^{(j),\text{new}} \coloneqq  {\mb \Omega^{(j),\text{new}}}^{-1}, \\
        & \mb \Omega^{(j),\text{new}} \coloneqq \left(
            \begin{matrix}
                \frac{1}{\underline\sigma^2}\mb I_2 + \frac{\alpha'^2}{\underline\sigma^2} \hat {\mb w}_0\hat{\mb w}_0\t & \frac{\alpha'}{\underline\sigma^2} \hat {\mb w}_0{\mb v^{(j)}}\t \\ 
                 \frac{\alpha'}{\underline\sigma^2}\mb v^{(j)}{{}\hat{\mb w}_0}\t  & \frac{1}{\tilde\sigma^2}\big(\mb I_{p-K} - \mb v^{(j)}{\mb v^{(j)}}\t  \big)+ \frac{\mb v^{(j)}{\mb v^{(j)}}\t}{\underline\sigma^2} 
            \end{matrix}
        \right)\label{eq: new covariance matrix}
    \end{align}
    for $j \in \{ j_1, j_2\}$, where $\hat{\mb w}_0 \coloneqq \frac{1}{\sqrt{2}}(-1,1)\t$. 

    On the other hand, regarding the region $R^{(j_1,j_2)}_K$, its projected version $\hat R^{(j_1,j_2)}_K\coloneqq \{\mb V_2^*{\mb V^*_2}\t\mb x + \mb V_\perp^*{\mb V_\perp^*}\t \mb x: \mb x \in \bb R^p\}\subset \bb R^{p-K+2}$ is of the form of $R^{(j_1,j_2)}_K$ (two-component cases) considered in \eqref{eq: set R}.

    Treating the new centers, new covariance matrices, and the projected region $\hat R^{(j_1,j_2)}_K$ as the corresponding ones of two-component Gaussian mixtures, \eqref{eq: likelihood ratio conditions in Condition 1*} have been verified by the proof of \textsc{Condition 1} in Section~\ref{sec: proof of gaussian lower bound}. 

    \medskip
    \subparagraph*{Verifying \textsc{Condition 2*}}
    Similar to the proof of \textsc{Condition 1*}, we would like to reuse the proof of \textsc{Condition 2} by reducing it to the two-component case. Toward this, we make note that, for a random vector $\mb x$ obeying $\mc N(\bo \theta_1^{*}, \mb \Sigma_1^{(j_1)})$ or $\mc N(\bo \theta_2^{*}, \mb \Sigma_2^{(j_2)})$,  the decomposition $\mb x = \big(\mb V^*_2, \mb V_\perp^*\big)\big(\mb V^*_2, \mb V_\perp^*\big)\t \mb x + \mb V_{-2}^*{\mb V_{-2}^*}\t \mb x$ satisfies that $\big(\mb V^*_2, \mb V_\perp^*\big)\t \mb x$ is independent of ${\mb V_{-2}^*}\t \mb x$. For the former one $\big(\mb V^*_2, \mb V_\perp^*\big)\t \mb x$, the proof of Claim~\ref{claim: condition 1/2} gives the minimum of two probability density functions regarding $\big(\mb V^*_2, \mb V_\perp^*\big)\t \mb x$ as 
    \longeq{
        &\min \{\phi_{\bo \theta_1^{*,(j_1),\text{new}}, \mb \Sigma_1^{(j_1), \text{new}}}\big(\big(\mb V^*_2, \mb V_\perp^*\big)\t \mb x\big), \phi_{\bo \theta_2^{*,(j_2),\text{new}}, \mb \Sigma_2^{(j_2), \text{new}}}\big(\big(\mb V^*_2, \mb V_\perp^*\big)\t \mb x\big)\} \\ 
        \geq & \frac{1}{(2\pi)^2\underline \sigma^4}\exp\Big(-\frac{1}{2} \big(1 +  \frac{C^{\mathsf{density}}_1}{\mathsf{SNR}_0^{\mod}} + C^{\mathsf{density}}_2 \delta + C^{\mathsf{density}}_3 \frac{\underline\sigma^2 }{\tilde\sigma^2 } +C^{\mathsf{density}}_4 \frac{\log\big( \frac{\tilde\sigma}{\underline \sigma}\big)}{{\mathsf{SNR}_0^{\mod}}^2}\big){\mathsf{SNR}_0^{\mod}}^2\Big) \\ 
     & \cdot \frac{1}{(2\pi)^{\frac{p-K-2}{2}}\tilde\sigma^{p-K - 2}}\exp\Big(- \frac{\big\|{\mb V^{(j_1,j_2)}_\perp}\t \mb x\Big\|_2^2}{2\tilde\sigma^2}\Big)
    }
    for some constants $C^{\mathsf{density}}_i$, $i= 1,2,3,4$ that are determined by $\bar \sigma/ \underline\sigma, \rho_1, \rho_2, \alpha$. 
    Here, $\mb x_\perp^{(j_1,j_2)}$ denotes the vector ${\mb V_\perp^{(j_1,j_2)}}\t   \mb x$. On the other hand, we can see that ${\mb V_{-2}^*}\t \mb X $ is a standard multivariate Gaussian vector from the fact that ${\mb V_{-2}^*}\t \mb \Sigma^{(j)}_k \mb V_{-2}^* = \mb I_{p-K}$ for $k = 1,2$ and $j \in[M]$. In this way, the minimum of the two (full) density function regarding $\mb x$ is given by 
    \begin{align}
        & \min \{\phi_{\bo \theta_1^*, \mb \Sigma_1^{(j_1)}}\big( \mb x\big), \phi_{\bo \theta_2^*, \mb \Sigma_2^{(j_2)}}\big(\mb x\big)\}\\ 
        \geq & \frac{1}{(2\pi)^2\underline \sigma^4}\exp\Big(-\frac{1}{2} \big(1 +  \frac{C^{\mathsf{density}}_1}{\mathsf{SNR}_0} + C^{\mathsf{density}}_2 \delta + C^{\mathsf{density}}_3 \frac{\underline\sigma^2 }{\tilde\sigma^2 } +C^{\mathsf{density}}_4 \frac{\log\big( \frac{\tilde\sigma}{\underline \sigma}\big)}{{\mathsf{SNR}_0^{\mod}}^2}\big){\mathsf{SNR}_0^{\mod}}^2\Big) \\ 
     & \cdot \frac{1}{(2\pi)^{\frac{p-K-2}{2}}\tilde\sigma^{p-K-2}}\exp\Big(- \frac{\big\|{\mb V^{(j_1,j_2)}_\perp}\t \mb x\Big\|_2^2}{2\tilde\sigma^2}\Big) \cdot \frac{1}{(2\pi)^{\frac{K-2}{2}} \underline \sigma^{K-2}} \exp(- \frac{\bignorm{{\mb V_{-2}^*}\t \mb x}_2^2}{2\underline\sigma^2}).
    \end{align}
    We have thus verified \textsc{Condition 2*}.

    \medskip

    By \textsc{Condition 1*} and \textsc{Condition 2*}, the right-hand side of \eqref{eq: lower bound for two components out of K components} is lower bounded by 
    \begin{align}
        & \int_{R^{(j_1,j_2)}_K} \frac{1}{(2\pi)^2\underline\sigma^4 }\exp(-(1 + o(1))\frac{{\mathsf{SNR}^{\mathsf{mod}}}^2}{2}) \cdot \frac{1}{(2\pi)^{\frac{p-2}{2}}\tilde\sigma^{p-2}}\exp\big(-\frac{\bignorm{{\mb V_\perp^{(j_1,j_2)}}\t {\mb V_\perp^*}\t  \mb x}_2^2}{2\tilde\sigma^2} \big)\\ 
        & \cdot \frac{1}{(2\pi)^{\frac{K-2}{2}}\underline\sigma^{K-2}} \exp(- \frac{\bignorm{{\mb V_{-2}^*}\t \mb x}_2^2}{2\underline\sigma^2}) \mathrm d\mb x\\ 
        = & \frac{\pi^2 \rho_1^2 \rho_2^2}{(2\pi)^2}\exp(-(1 + o(1))\frac{{\mathsf{SNR}_0^{\mod}}^2}{2}) =  \exp(-(1 + o(1))\frac{{\mathsf{SNR}_0^{\mod}}^2}{2}),
    \end{align}
    by marginalizing out the variables ${\mb V_\perp^{(j_1,j_2)}}\t {\mb V_\perp^*}\t  \mb x$ and ${\mb V_{-2}^*}\t \mb x$. 

    \medskip
    \emph{Step 3.4*. }
    This step is almost parallel to Step 3.4 in the proof of Theorem~\ref{theorem: gaussian lower bound}. Firstly, 
    
    We note that the distributions of the components other than the first two components are the same across different $j\in[M]$. The KL-divergence between $\bar{\bb P}_{*,(j_1)}$ and $\bar{\bb P}_{*,(j_2)}$ thus turns out to be upper bounded by 
    \longeq{
        &\mathrm{KL}(\bar{\bb P}_{*,(j_1)}, \bar{\bb P}_{*,(j_2)}) \\ 
        \leq & n\cdot \mathrm{KL}\big(\frac{1}{2}
        \bb P_{\theta_1^{*,(j_1)}, \mb \Sigma_1^{(j_1)}} + \frac{1}{2}\bb P_{\theta_2^{*,(j_1)}, \mb \Sigma_2^{(j_1)}}, \frac{1}{2}\bb P_{\theta_1^{*,(j_2)}, \mb \Sigma_1^{(j_2)}} + \frac{1}{2}\bb P_{\theta_2^{*,(j_2)}, \mb \Sigma_2^{(j_2)}}\big)\\ 
        \leq & \frac{n}{2}\cdot \big(\mathrm{KL}( \bb P_{\theta_1^{*,(j_1)}, \mb \Sigma_1^{(j_1)}}, \bb P_{\theta_1^{*,(j_2)}, \mb \Sigma_1^{(j_2)}}) + \mathrm{KL}( \bb P_{\theta_2^{*,(j_1)}, \mb \Sigma_2^{(j_1)}}, \bb P_{\theta_2^{*,(j_2)}, \mb \Sigma_2^{(j_2)}}) \big), 
        \label{eq: upper bound for KL divergence in K components 1}
    }
    by applying the conditional property of KL divergence on each assignment. 

    Noticing that $|\mb \Sigma_k^{(j_1)}| = |\mb \Sigma_k^{(j_2)}|$ for $k = 1,2$ and $j_1,j_2\in[M]$, we invoke \eqref{eq: expression of KL} to obtain that 
    \begin{align}
        & \mathrm{KL}( \bb P_{\theta_k^{*,(j_1)}, \mb \Sigma_1^{(j_1)}}, \bb P_{\theta_k^{*,(j_2)}, \mb \Sigma_1^{(j_2)}})\leq \frac{1}{2} \mathrm{Tr}\big({\mb \Sigma_k^{(j_1)}}^{-1}\big(\mb \Sigma_k^{(j_1)} - \mb \Sigma_k^{(j_2)} \big) \big) \label{eq: upper bound for KL divergence in K components 2}
    \end{align}
    for $k=1,2$. The upper bound on the right-hand side can be accomplished in the same way as in Step 3.4 in the proof of Theorem~\ref{theorem: gaussian lower bound}. We omit the details for conciseness and give the conclusion that 
    \begin{align}
        & \frac{1}{2} \mathrm{Tr}\big({\mb \Sigma_k^{(j_1)}}^{-1}\big(\mb \Sigma_k^{(j_1)} - \mb \Sigma_k^{(j_2)} \big) \big) \leq C\frac{\tilde\sigma^2}{\underline\sigma^2},
        \label{eq: upper bound for KL divergence in K components 3}
    \end{align} 
    for $k\in[2]$, where $C$ is a constant related to $\alpha$. 

    Putting \eqref{eq: upper bound for KL divergence in K components 1}, \eqref{eq: upper bound for KL divergence in K components 2}, \eqref{eq: upper bound for KL divergence in K components 3} together, we conclude that 
    \begin{align}
        & \mathrm{KL}( \bb P_{\bo \theta_k^{*}, \mb \Sigma_1^{(j_1)}}, \bb P_{\bo \theta_k^{*}, \mb \Sigma_1^{(j_2)}})\leq n \cdot \frac{1}{2} \mathrm{Tr}\big({\mb \Sigma_k^{(j_1)}}^{-1}\big(\mb \Sigma_k^{(j_1)} - \mb \Sigma_k^{(j_2)} \big) \big) \leq \frac{ C n \tilde\sigma^2}{2\underline\sigma^2}. \label{eq: upper bound for KL divergence in K components 4}
    \end{align}

    \emph{Putting All Pieces Together. }
    Again, invoking \eqref{eq: log M lower bound with K components} and \eqref{eq: upper bound for KL divergence in K components 4}, we control the ratio between the KL divergence and $\log M$ by 
    \eq{
    \frac{\max_{j_1\neq j_2 \in[M]}\mathrm{KL}(\bar{\bb P}_{*,{\bo \eta}^{(j_1)}}, \bar{\bb P}_{*,{\bo \eta}^{(j_2)}}
    )}{\log M} 
    \leq \frac{\frac{ C n \tilde\sigma^2}{2\underline\sigma^2}}{c n \frac{\tilde\sigma^{2(1+\epsilon)}}{\bar \sigma^{2(1+ \epsilon) }} } \rightarrow 0. 
    }
    since $\tilde\sigma \rightarrow \infty$. To arrive at our final conclusion, we invoke \eqref{eq: corollary lower bound simplification} together with Lemma~\ref{lemma: fano's method} to obtain that 
    \eq{
        \inf_{\hat{\mb z}}\sup_{\bo \theta \in \mb \Theta} \bb Eh(\hat{\mb z}, \mb z^*) \geq \frac{1}{4\beta K}\exp\Big(-(1 + o(1) \frac{{\mathsf{SNR}_0^{\mod}}^2}{2}\Big)  =  \exp\big(-(1 + o(1))\frac{{\mathsf{SNR}_0^{\mod}}^2}{2} \big)
    }
    since $\log\beta\vee  K = o({\mathsf{SNR}_0^{\mod}}^2)$.

    \emph{Proof of Claim \ref{claim: SNR relations for K components}. }
    For the first part, we note that ${\mb V^{*\top}} \mb \Sigma_k^{(j)}\mb V^* = \mb I_K$ for every $k\in[2]$ and $j\in[M]$. Then the condition that 
    $$\mathsf{SNR}^{\mod}_{1,2}(\{\bo \theta_k^{*}\}_{k\in[K]}, \{\mb \Sigma_k^{(j)}\}_{k\in[K]}) = \mathsf{SNR}^{\mod}_{2,1}(\{\bo \theta_k^{*}\}_{k\in[K]}, \{\mb \Sigma_k^{(j)}\}_{k\in[K]})= \mathsf{SNR}^{\mod}_0$$ follows from the definition of $\mathsf{SNR}_{j_1,j_2}$. 

    Moreover, it is obvious that 
    $$
    {\mb V^{*\top}} \mb \Sigma_k \mb V^* = \frac{\underline\sigma^2}{{\alpha''}^2 }\mb I_K,\quad \text{for } k = 3, \cdots, K.
    $$
    A direct calculation based on the definition of $\mathsf{SNR}^{\mod}_{a,b}$ gives that 
    \begin{align}
        & \mathsf{SNR}^{\mod}_{a,b}(\{\bo \theta_k^{*}\}_{k\in[K]}, \{\mb \Sigma_k^{(j)}\}_{k\in[K]}) \\ 
        = & 
        \begin{cases}
            \dfrac{2\alpha''}{1 + \alpha''}\mathsf{SNR}^{\mod}_0 & \text{ if $a \in \{1,2\}$ and $b \in \{3,\cdots, K\}$ or $b \in \{1,2\}$ and $a \in \{3,\cdots, K\}$},\\[3mm]
                \alpha'' \mathsf{SNR}^{\mod}_0&  \text{ if $a,b\in \{3, \cdots, K\}$},
        \end{cases}
    \end{align}
    and thus proves the second part.

    In what follows, we shall verify the condition that 
    \eq{
    \label{eq: verify rbayes condition for k component case}
    -\log(\rbayes) = -\min_{a\neq b \in [K]}\log(\rbayes(\{\bo \theta_a^*, \bo \theta_b^*\}, \{\mb \Sigma_a^{(j)}, \mb \Sigma_b^{(j)}\})) \geq \alpha \mathsf{SNR}_0^{\mod}.
    }
    To this end, we separately analyze the cases where $a = 1, b = 2$ (or equivalently $a = 2, b = 1$) and where $a \in \{3,\cdots, K\}$ ($b \in \{3,\cdots, K\}$).

    To begin with, we invoke Proposition~\ref{proposition: SNR' and Bayesian oracle risk 1} to relate $\rbayes$ and $\snrfull$ for the first case $a = 1, b=2$: 
    \eq{ 
    \label{eq: rbayes relation for a = 1, b = 2}
        \rbayes(\{\bo \theta_1^*, \bo \theta_2^*\}, \{\mb \Sigma_1, \mb \Sigma_2\}) = \exp\Big(- (1 + o(1)) \frac{{\mathsf{SNR}^{\mathsf{full}}_{1,2}}^2 }{2}\Big). 
    }

    Then we leverage the intermediate result from the proof of Theorem~\ref{theorem: gaussian lower bound}, in light of a reduction argument. Precisely, we notice that  
    \begin{align}
        & {\mathsf{SNR}^{\mathsf{full}}_{1,2}}(\{\bo \theta_k^{*}\}_{k\in[K]}, \{\mb \Sigma_k^{(j)}\}_{k\in[K]}) = {\mathsf{SNR}^{\mathsf{full}}_{2,1}}(\{\bo \theta_k^{*}\}_{k\in[K]}, \{\mb \Sigma_k^{(j)}\}_{k\in[K]}) \\ 
        =& \snrfull(\{\bo \theta_1^{*,\text{new}},\bo \theta_2^{*,\text{new}}\}, \{\mb \Sigma_1^{(j),\text{new}}, \mb \Sigma_2^{(j),\text{new}}\}), 
    \end{align}
    where the second equality holds by an observation that the minimizer $\mb x$ in the definition of $\snr^{\mathsf{full}}_{1,2}(\{\bo \theta_k^{*}\}_{k\in[K]}, \{\mb \Sigma_k^{(j)}\}_{k\in[K]})$ must satisfy ${\mb V_{-2}^*}\t \mb x =\mb 0$ since 
    \eq{\mb x\t {\mb \Sigma_k^{(j)}}^{-1} \mb x \geq \big((\mb I_p - \mb V_{-2}^*{\mb V_{-2}^*}\t )\mb x\big)\t {\mb \Sigma_k^{(j)}}^{-1}\big((\mb I_p - \mb V_{-2}^*{\mb V_{-2}^*}\t )\mb x\big)}
    holds for every $\mb x \in \bb R^{p}$. 
    Moreover, we can see that the forms of the new centers and covariance matrices are the same as the ones considered in the proof of Theorem \ref{theorem: gaussian lower bound} (\eqref{eq: centers and covariances in the packing} and \eqref{eq: definition of Omega_k}). Therefore, the derivation in the part ``Verifying the conditions in $\tilde{\mb \Theta}_\alpha$'' also implies that \eq{
        \snrfull(\{\bo \theta_1^{*,\text{new}},\bo \theta_2^{*,\text{new}}\}, \{\mb \Sigma_1^{(j),\text{new}}, \mb \Sigma_2^{(j),\text{new}}\})\geq 2\alpha \mathsf{SNR}_0^{\mod}.
        } 
    by invoking $\alpha' = 12\alpha$. 
    This together with \eqref{eq: rbayes relation for a = 1, b = 2} in turn verifies the condition that 
    $$
    -\log(\rbayes(\{\bo \theta_1^*, \bo \theta_2^*\}, \{\mb \Sigma_1, \mb \Sigma_2\})) \geq \alpha^2  \frac{{\mathsf{SNR}_0^{\mod}}^2}{2}
    $$
    for every sufficiently large $n$. 

    What remains to be solved is the cases involving at least one component with the covariance matrix \eqref{eq: trivial omega}. We start by discussing the relation between $\snr^{\mathsf{full}}_{a,b}$ and $\mathsf{SNR}_0^{\mod}$. We first assume that $a$ is equal to $3$ and $b$ is arbitrary in $[K]$. We shall verify that $\snr^{\mathsf{full}}_{3,b}(\{\bo \theta_k^{*}\}_{k\in[K]}, \{\mb \Sigma_k^{(j)}\}_{k\in[K]}) \geq \frac{3}{2} \alpha \mathsf{SNR}^{\mod}_0$ by showing that every $\mb x$ that is too close to $\bo \theta_3^*$
    is not able to satisfy the equation in the definition of $\snr^{\mathsf{full}}_{3,b}(\{\bo \theta_k^{*}\}_{k\in[K]}, \{\mb \Sigma_k^{(j)}\}_{k\in[K]})$. 
    ; to be more specific, we aim to show that 
    \begin{align}
        & (\mb x - \bo \theta_3^*)\t {\mb \Sigma_3^{(j)}}^{-1}(\mb x - \bo \theta_3^*) + \log|\mb \Sigma_3^{(j)}| - \log|\mb \Sigma_b^{(j)}|> \frac{9}{4}\alpha^2 {\mathsf{SNR}_0^{\mod}}^2\label{eq: contradicting condition for 3,b}
    \end{align}
    holds for every $\mb x$ with 
    \eq{
        (\mb x - \bo \theta_b^*)\t {\mb \Sigma_b^{(j)}}^{-1}(\mb x - \bo \theta_b^*) < \frac94\alpha^2{\mathsf{SNR}_0^{\mod}}^2.\label{eq: condition for 3,b}
    }
    By substituting $\alpha'' = \frac{8\alpha}{2 - \frac{3}{2}\alpha}\geq 8\alpha$ (recall that $1<\alpha<\frac{4}{3}$), we first note that \eqref{eq: condition for 3,b} implies that \begin{align}
        &\bignorm{{\mb V^{*\top}} \big(\mb x -\bo \theta_b^*\big)}_2 <\frac{3\alpha}{2\alpha''}\mathsf{SNR}_0^{\mod} \leq  \frac{1}{4}\mathsf{SNR}_0^{\mod}.\label{eq: upper bound for x - theta_3}
    \end{align}

    Now we discuss the following two cases: 
    \begin{itemize}
        \item{$b = 4, \cdots, K$.} In this case, it is straightforward to see that $\log|\mb \Sigma_b^{(j)}| = \log|\mb \Sigma_3^{(j)}|$. Moreover, the triangle inequality implies that $\bignorm{{\mb V^{*\top}}\big(\mb x - \bo \theta_3^*\big)}_2 \geq 2\mathsf{SNR}_0^{\mod} - \bignorm{{\mb V^{*\top}}\big(\mb x - \bo \theta_b^*\big)}_2 \geq \frac{7}{4}\mathsf{SNR}_0^{\mod}$. Then we conclude that  
        \eq{
            (\mb x - \bo \theta_3^*)\t {\mb \Sigma_3^{(j)}}^{-1}(\mb x - \bo \theta_3^*) + \log|\mb \Sigma_3^{(j)}| - \log|\mb \Sigma_b^{(j)}| \geq \frac{49}{16}{\alpha''}^2 {\mathsf{SNR}_0^{\mod}}^2 > \frac{9}{4}\alpha^2 {\mathsf{SNR}_0^{\mod}}^2.
        }
        and thus 
        \eq{\label{eq: snrfull 3 b}
        \snr^{\mathsf{full}}_{3,b} = \snr^{\mathsf{full}}_{b,3} \geq \frac{3}{2}\alpha\mathsf{SNR}_0. }
        \item{$b= 1,2$. } Without loss of generality, we let $b$ be $1$. Noticing that ${\mb \Sigma_1^{(j)}}^{-1} \succ \mb V^* {\mb V^{*\top}}$, \eqref{eq: condition for 3,b} implies that
        \eq{
            \bignorm{{\mb V^{*\top}}(\mb x - \bo \theta_1^{*, (j)})}_2 < \frac{3}{2}\alpha \mathsf{SNR}_0^{\mod},
        }
        which in turn yields that 
        \eq{
             \bignorm{{\mb V^{*\top}}(\mb x - \bo \theta_3^{*, (j)})}_2 \geq (2 - \frac{3}{2}\alpha)\mathsf{SNR}_0^{\mod}. 
        }

        By the definition of $\alpha''$, we then have 
        \eq{
            (\mb x - \bo \theta_3^*)\t {\mb \Sigma_3^{(j)}}^{-1}(\mb x - \bo \theta_3^*) \geq {\alpha''}^2(2-\alpha)^2 {\mathsf{SNR}_0^{\mod}}^2= 64\alpha^2 {\mathsf{SNR}_0^{\mod}}^2.\label{eq: lower bound for quadratic term with b = 1}
        }

        On the other hand, we compute the difference of the log determinants as follows:  
        \longeq{
            &\log|\mb \Sigma_3^{(j)}| - \log|\mb \Sigma_1^{(j)}| = - \log|\mb \Omega_3^{(j)}|+ \log|\mb \Omega_1^{(j)}| \\ 
            =&  - \log|(\mb V^*, \mb V^*_\perp\mb v^{(j)})\t \mb \Omega_3^{(j)}(\mb V^*, \mb V^*_\perp \mb v^{(j)})|+ \log|(\mb V^*, \mb V^*_\perp\mb v^{(j)})\t \mb \Omega_1^{(j)}(\mb V^*, \mb V^*_\perp \mb v^{(j)})| \\ 
            \geq&  -(K+1)\log\Big(\frac{\bignorm{(\mb V^*, \mb V^*_\perp \mb v^{(j)})\t \mb \Omega_3^{(j)}(\mb V^*, \mb V^*_\perp \mb v^{(j)})}}{\sigma_{\min}\big((\mb V^*, \mb V^*_\perp \mb v^{(j)})\t \mb \Omega_1^{(j)}(\mb V^*, \mb V^*_\perp\mb v^{(j)})\big)}\Big) \\ 
            \geq &- 2(K+1)\log(\alpha'')\\ 
            \geq &-\alpha^2 {\mathsf{SNR}_0^{\mod}}^2,
            \label{eq: lower bound for log det}
        }
        where the last inequality holds for every sufficiently large $n$, given $K  =o({\mathsf{SNR}_0^{\mod}}^2 )$. 
        Combining \eqref{eq: lower bound for quadratic term with b = 1} with \eqref{eq: lower bound for log det}, we verify the inequality \eqref{eq: contradicting condition for 3,b} and conclude that $\mathsf{SNR}^{\mathsf{full}}_{3,1} \geq \frac{3}{2} \alpha \mathsf{SNR}_0^{\mod}$ for every sufficiently large $n$. Further, it is obvious from the definition that ${\mathsf{SNR}^{\mathsf{full}}_{1,3}}^2 \geq {\mathsf{SNR}^{\mathsf{full}}_{1,3}}^2 - \big|\log|\mb \Sigma_3^{(j)}| - \log|\mb \Sigma_1^{(j)}|\big|  $. From the derivation of \eqref{eq: lower bound for log det}, we also have $\snr^{\mathsf{full}}_{1,3} \geq \frac{5}{4}\alpha \mathsf{SNR}_0^{\mod}$ for every sufficiently large $n$. 
    \end{itemize}

    Given the above characterization, we verify the condition $-\log(\rbayes(\{\bo \theta_a^*, \bo \theta_b^*\}, \{\mb \Sigma_a,\mb \Sigma_b\}))\geq \alpha^2 \frac{{\mathsf{SNR}_0^{\mod}}^2}{2}$ for $a$ or $b$ in $\{3, \cdots, K\}$ and every sufficiently large $n$ in the following:
    \begin{itemize}
        \item If $a\neq b\in\{3, \cdots, K\}$, combining \eqref{eq: snrfull 3 b} with Proposition~\ref{proposition: SNR' and Bayesian oracle risk 1} directly leads to the conclusion. 
        \item If $a \in \{1,2\}, b\in \{3, \cdots, K\}$, we look into the form of $\rbayes$ and have
        \begin{align}
        &\rbayes(\{\bo \theta_1^*, \bo \theta_3^*\}, \{\mb \Sigma_1, \mb \Sigma_3\}) \leq \bb P\Big[\norm{\bo \epsilon}_2 \geq  \min\{\snr^{\mathsf{full}}_{1,3}, \snr^{\mathsf{full}}_{3,1}\} \Big]\\ 
        \leq & \exp\Big(- (1+o(1))\frac{\min\{{\snr^{\mathsf{full}}_{1,3}}^2, {\snr^{\mathsf{full}}_{3,1}}^2\}}{2}\Big),
        \end{align}
        where $\epsilon$ is a $(K+1)$-dimensional standard Gaussian vector. Here the last inequality holds because of the Hanson-Wright inequality together with the condition that $K = o({\mathsf{SNR}_0^{\mod}}^2) $. Recall that we have proved $\min\{\snr^{\mathsf{full}}_{1,3}, \snr^{\mathsf{full}}_{3,1}\} \geq \frac{5}{4}\mathsf{SNR}_0$ holds for every sufficiently large $n$. Consequently, $-\log(\rbayes(\{\bo \theta_1^*, \bo \theta_3^*\}, \{\mb \Sigma_1,\mb \Sigma_3\}))\geq \alpha^2 \frac{{\mathsf{SNR}_0^{\mod}}^2}{2}$ holds for every sufficiently large $n$. 
    \end{itemize}

    \subsection{Proof of Corollary~\ref{corollary: simple gaussian lower bound}}

    Following the notations and the reduction scheme in \emph{Step 1} of Theorem~\ref{theorem: gaussian lower bound}'s proof and assuming $1 \neq \mc B$ without loss of generality, we can similarly obtain a relation bridging the minimax risk with the Bayesian risk $\rbayes$: 
    \begin{align}
        & \inf_{\hat{\mb z}} \sup_{(\mb z, \bo\eta)\in \mb \Theta} \bb E h(\hat {\mb z}, \mb z^*)  
    =  \inf_{\hat{\mb z}} \sup_{(\bo\theta^*_1, \bo \theta^*_2, \mb \Sigma_1, \mb \Sigma_2) \in \tilde{\mb \Theta}_\alpha} \sup_{\mb z^* \in \mb \Theta_z}\bb E h(\hat {\mb z}, \mb z^*) \\
    \geq &\inf_{\hat{\mb z}} \sup_{(\bo\theta^*_1, \bo \theta^*_2, \mb \Sigma_1, \mb \Sigma_2) \in \tilde{\mb \Theta}_\alpha} \frac{1}{|\mb Z_{\mc B}|} \sum_{\mb z^*\in \mb Z_{\mc B}} \Big( \frac{1}{n} \sum_{i \in \mc B^\complement} \bb P[\hat z_i \neq z_i^*] \Big)  \\
    \geq &\frac{1}{4\beta }\inf_{\hat{\mb z}} \sup_{(\bo\theta^*_1, \bo \theta^*_2, \mb \Sigma_1, \mb \Sigma_2) \in \tilde{\mb \Theta}_\alpha} \frac{1}{|\mb Z_{\mc B}|} \sum_{\mb z^*\in \mb Z_{\mc B}} \Big( \frac{1}{|\mc  B^{\complement}|}\sum_{i \in \mc B^\complement} \bb P[\hat z_i \neq z_i^*] \Big)  \\
    \geq &\frac{1}{4\beta } \sup_{(\bo\theta^*_1, \bo \theta^*_2, \mb \Sigma_1, \mb \Sigma_2) \in \tilde{\mb \Theta}_\alpha} \rbayes(\{\bo \theta_j^*\}_{j\in[2]}, \{ \mb \Sigma_j\}_{j\in[2]} ).\label{eq: simple lower bound simplification}
    \end{align}

    Then we focus on a group of easy-to-handle parameters $(\bo \theta_1^*, \bo \theta_2^*, \bar{ \mb \Sigma}_{1},\bar{\mb \Sigma}_2)$ to invoke Proposition~\ref{proposition: SNR' and Bayesian oracle risk 1}; it is easy to verify that $\tilde{\mb V} = \mb V_\perp^*$ satisfies the hypothesis in the second part of Proposition~\ref{proposition: SNR' and Bayesian oracle risk 1}. Combining \eqref{eq: simple lower bound simplification} with Proposition~\ref{proposition: SNR' and Bayesian oracle risk 1}, we therefore have 
    \begin{align}
        &  \inf_{\hat{\mb z}} \sup_{\theta \in \mb \Theta} \bb E h(\hat {\mb z}, \mb z^*) \geq  \frac{1}{4\beta }\rbayes (\{\bo \theta_j^*\}_{j\in[2]}, \{\bar{ \mb \Sigma}_{j}\}_{j\in[2]}) = \exp(-(1 + o(1))\frac{{\mathsf{SNR}_0^{\mod}}^2}{2}), 
    \end{align}
    since $\mathsf{SNR}^{\mod}_0\rightarrow \infty$ and $\frac{\log \beta}{{\mathsf{SNR}_0^{\mod}}^2} \rightarrow 0 $.

\subsection{Proof of Theorem~\ref{thm: lower bound regarding SNReh}}
\label{subsec: proof of thm lower bound regarding SNReh}
The spirit of the proof is aligned with the ones in \cite{ndaoud2018sharp,chen2021cutoff}, yet the concrete treatment differs from those due to the additional consideration of anisotropic covariance structures. To begin with, we introduce a uniform measure over $A_0\coloneqq \big\{\mb z^*: z_i^* \in \{1,2\},~ \forall 1 \leq i \leq \lceil \frac{n}{3}\rceil ; z_i^* = 1,~\forall~\lceil \frac{n}{3}\rceil +1 \leq  i \leq 2\lceil \frac{n}{3} \rceil ;~z_i^* = 2,~\forall~ 2 \lceil  \frac{n}{3}\rceil + 1\leq  i \leq n\big\}$, denoted by $\mc P_{\mb z}$. 
Moreover, we consider a measure 
$\mc P_{\bo \theta}$ of the centers $\{\bo \theta_k^*\}_{k\in[K]}$ such that $\bo \theta_1^*$ and $\bo \theta_2^*$ are independently drawn from $\mc N(0, \kappa_n^2\bo \Sigma^2)$. 
Building upon these priors, we
denote by $\pi_{\mc P_{\mb z} \otimes \mc P_{\bo \theta}| A}$ a conditional measure of $\mc P_{\mb z} \otimes \mc P_{\bo \theta}$ on the set 
\longeq{ 
& A \coloneqq \Big\{(\mb z^*, \{\bo \theta_k^*\}_{k\in[K]}): n_k(\mb z^*) \in \big[\big(1 - C\sqrt{\log n/ n}\big) \frac{n}{2}, \big(1 + C\sqrt{\log n /n }\big) \frac{n}{2}\big],~k=1,2; \\ 
& \qquad \qquad \mathsf{SNR}^{\mathsf{exc}}(\mb z^*, \{\bo \theta_k^*\}_{k\in[2]}, \{\bo \Sigma_k\}_{k\in[2]}) \geq \mathsf{SNR}^{\mathsf{exc}}_0\Big\}
\label{eq: event A}
} with some constant $C>0$, that is, $\bb P_{\pi_{\mc P_{\mb z} \otimes \mc P_{\bo \theta}| A}}[B] = \frac{\bb P_{\mc P_{\mb z} \otimes \mc P_{\bo \theta}}[A \cap B]}{ \bb P_{\mc P_{\mb z} \otimes \mc P_{\bo \theta}}[A]}$. Placing this prior on $(\mb z^*, \{\bo \theta^*_k\}_{k\in[2]})$, a straightforward calculation yields that 
\begin{align}
    & \inf_{\hat{\mb z}} \sup_{(\mb z^*, \{\bo \theta_k^*\}_{k\in[2]}, \{\bo \Sigma_k\}_{k\in[2]}) \in \bo \Theta^{\mathsf{exc}}} \bb E[h(\hat{\mb z},\mb z^*)] \geq \inf_{\hat{\mb z}} \bb E_{(\mb z^*, \{\bo \theta^*_k\}_{k\in[2]}) \sim \pi_{\mc P_{\mb z} \otimes \mc P_{\bo \theta}| A}}[h(\hat{\mb z}, \mb z^*)] \\ 
    \geq & \inf_{\hat{\mb z}} \bb E_{\mc P_{\mb z}\otimes \mc P_{\bo \theta}}\Big[h(\hat{\mb z}, \mb z^*)\cdot \ind\{(\mb z^*,\{\bo \theta^*_k\}_{k\in[2]})  \in A \} \Big] / \bb P_{\mc P_{\mb z}\otimes \mc P_{\bo \theta}}[A] \\ 
    \stackrel{(1)}{\geq} & \inf_{\hat{\mb z}} \bb E_{\mc P_{\mb z}\otimes \mc P_{\bo \theta}}\Big[\big(\frac{1}{n}\sum_{1\leq i \leq \lceil\frac{2n}{3K}\rceil} \ind\{\hat z_i \neq z_i^* \} \big)\cdot \ind\{(\mb z^*,\{\bo \theta^*_k\}_{k\in[2]})  \in A \} \Big] / \bb P_{\mc P_{\mb z}\otimes \mc P_{\bo \theta}}[A] \\ 
    \geq &  \frac{1}{\bb P_{\mc P_{\mb z}\otimes \mc P_{\bo \theta}}[A]} \inf_{\hat{\mb z}} \bb E_{\mc P_{\mb z}\otimes \mc P_{\bo \theta}}\Big[\frac{1}{n}\sum_{1\leq i \leq \lceil\frac{2n}{3K}\rceil} \ind\{\hat z_i \neq z_i^* \}  \Big]  -  \frac{\bb P_{\mc P_{\mb z}\otimes \mc P_{\bo \theta}}[A^{\complement} ]}{\bb P_{\mc P_{\mb z}\otimes \mc P_{\bo \theta}}[A]} \\ 
    \geq & \frac{1}{\bb P_{\mc P_{\mb z}\otimes \mc P_{\bo \theta}}[A]}  \cdot  \frac{1}{n}\cdot \sum_{1\leq i \leq \lceil\frac{n}{3K}\rceil} \inf_{\hat z_i}\bb E_{\mc P_{\mb z}\otimes \mc P_{\bo \theta}}\big[ \ind\{\hat z_i \neq z_i^* \}  \big]  -  \frac{\bb P_{\mc P_{\mb z}\otimes \mc P_{\bo \theta}}[A^{\complement} ]}{\bb P_{\mc P_{\mb z}\otimes \mc P_{\bo \theta}}[A]}  \\ 
    \geq &  \frac{1}{\bb P_{\mc P_{\mb z}\otimes \mc P_{\bo \theta}}[A]} \cdot  \frac{\lceil\frac{n}{3K}\rceil}{n} \cdot  \inf_{\hat z_1}\bb E_{\mc P_{\mb z}\otimes \mc P_{\bo \theta}}\big[ \ind\{\hat z_1 \neq z_1^* \}  \big]  -  \frac{\bb P_{\mc P_{\mb z}\otimes \mc P_{\bo \theta}}[A^{\complement} ]}{\bb P_{\mc P_{\mb z}\otimes \mc P_{\bo \theta}}[A]}, 
    \label{eq: lower bounding the minimax rate}
\end{align}
where the inequality (1) holds since we are able to fix the optimal permutation by always assigning the $\lceil \frac{n}{3}\rceil + 1$-th to $\lceil \frac{2n}{3} \rceil$-th through the first cluster and the remaining samples (from $\lceil \frac{2n}{3} \rceil + 1$-th to $n$-th) to the second cluster. 

To proceed, we have the following claim that lower bounds $\bb P_{\mc P_{\mb z} \otimes \mc P_{\bo \theta}}\big[A\big]$, whose proof is deferred to the end of this part.  
\begin{claim}
    \label{claim: lower bound on P[A]}
    When choosing $\kappa_n = \sqrt{2\mathsf{SNR}_0^{\mathsf{exc}}} \big(n\tr(\bo\Sigma^2)\big)^{-\frac14} \big(1 + C_\kappa \frac{\sqrt{\log n}\norm{\bo \Sigma}}{\sqrt{\tr(\bo \Sigma^2)}}\big)$ for some sufficiently large constant $C_\kappa$, one has for some constant $C_A$ that 
    $$\bb P_{\mc P_{\mb z} \otimes \mc P_{\bo \theta}}\big[A\big] \geq 1 - n^{-C_A}. $$
\end{claim} 
We note in passing that the choice of $\kappa_n$ together with the conditions in Theorem~\ref{thm: lower bound regarding SNReh} implies that $\kappa_n \ll \frac{1}{\sqrt{n}\norm{\bo \Sigma}^{\frac12}}$, which will be used later. 

As a consequence of Claim~\ref{claim: lower bound on P[A]}, \eqref{eq: lower bounding the minimax rate} is further written as 
\begin{align}
    & \inf_{\hat{\mb z}} \sup_{(\mb z^*, \{\bo \theta_k^*\}_{k\in[2]}, \{\bo \Sigma_k\}_{k\in[2]}) \in \bo \Theta^{\mathsf{exc}}} \bb E[h(\hat{\mb z},\mb z^*)] \geq \frac{1}{3K}\cdot \inf_{\hat z_1}\bb E_{\mc P_{\mb z}\otimes \mc P_{\bo \theta}}\big[ \ind\{\hat z_1 \neq z_1^* \}  \big]  - 2n^{-C_A}. 
\end{align}

In line with the approach in \cite{ndaoud2018sharp}, we analyze $\inf_{\hat z_1}\bb E_{\mc P_{\mb z}\otimes \mc P_{\bo \theta}}\big[ \ind\{\hat z_1 \neq z_1^* \}  \big] $ via the likelihood obtained by marginalizing out the centers $\bo \theta_1^*$ and $\bo \theta_2^*$. Specifically, given an assignment $\mb z$, the marginal likelihood of $\mb Y$ is written as 
\begin{align}
     & p_{\mb z}(\mb Y)  \propto \int \exp\bigg\{-\frac{1}{2}\sum_{i\in[n]} (\mb y_i - \bo \theta_{z_i} )\t \bo \Sigma^{-1} (\mb y_i - \bo \theta_{z_i} ) - \frac{{\bo \theta_1^*}\t {\bo \Sigma}^{-2} \bo \theta_1^* + {\bo \theta_2^*}\t {\bo \Sigma}^{-2} \bo \theta_2^*}{2 \kappa_n^2 } \bigg\} \mathrm d \bo \theta_1^* \mathrm d\bo \theta_2^*  \\ 
    \propto & \exp\bigg\{-\frac{1}{2} \sum_{i\in[n]}\mb y_i\t \bo \Sigma^{-1} \mb y_i + \frac{1}{2} \big(\sum_{i: z_i = 1} \mb y_i\big)\t \bo \Sigma^{-1}\Big(n_1(\mb z)\bo \Sigma^{-1} + \kappa_n^{-2} \bo \Sigma^{-2} \Big)^{-1} \bo \Sigma^{-1} \big(\sum_{i: z_i = 1} \mb y_i\big)\\ 
    & \qquad \qquad+ \frac{1}{2} \big(\sum_{i: z_i = 2} \mb y_i\big)\t \bo \Sigma^{-1}\Big(n_2(\mb z)\bo \Sigma^{-1} + \kappa_n^{-2} \bo \Sigma^{-2} \Big)^{-1} \bo \Sigma^{-1} \big(\sum_{i: z_i = 2} \mb y_i\big)  \bigg\}. 
\end{align}

Now fix an assignment $\mb z^0 \in \{1,2\}^{n-1}$ for the second through the $n$-th samples. We define $\mb z^1 \coloneqq (1, \mb z^0)$ and $\mb z^2 \coloneqq (2, \mb z^0)$. With these definitions, the log-likelihood difference can be decomposed as
\begin{align}
    & \log p_{\mb z^1}(\mb Y) - \log p_{z^2}(\mb Y) = \alpha_1 + \alpha_2 + \alpha_3, 
\end{align}
where $\alpha_i,~i =1,2,3$ are defined below (we use $\otimes$ to denote the Kronecker product): 
\begin{align}
    \alpha_1 \coloneqq &  \mb y_1\t \bo \Sigma^{-1}\big(\bo \Sigma^{-1} + \frac{2}{n}\kappa_n^{-2}\bo\Sigma^{-2} \big)^{-1}\bo \Sigma^{-1} \Big(\frac{\sum_{i\geq 2: z^1_i = 1} \mb y_i}{n/2} - \frac{\sum_{i\geq 2: z^1_i = 2} \mb y_i}{n/2}\Big), \\
     \alpha_2 \coloneqq &\Bigg\{ \frac{1}{2} \Big(\sum_{i: z^1_i = 1} \mb y_i\Big)\t \bo \Sigma^{-1}\bigg[\Big(\frac{n_1(\mb z^1)}{n/2}\bo \Sigma^{-1} + \frac{2}{n}\kappa_n^{-2} \bo \Sigma^{-2} \Big)^{-1} -\Big(\bo \Sigma^{-1} + \frac{2}{n}\kappa_n^{-2} \bo \Sigma^{-2} \Big)^{-1}\bigg] \bo \Sigma^{-1}\Big( \frac{\sum_{i: z^1_i = 1} \mb y_i}{n/2}\Big) \Bigg\} \\ 
     - & \Bigg\{ \frac{1}{2} \Big(\sum_{i: z^2_i = 1} \mb y_i\Big)\t \bo \Sigma^{-1}\bigg[\Big(\frac{n_1(\mb z^2)}{n/2}\bo \Sigma^{-1} + \frac{2}{n}\kappa_n^{-2} \bo \Sigma^{-2} \Big)^{-1} -\Big(\bo \Sigma^{-1} + \frac{2}{n}\kappa_n^{-2} \bo \Sigma^{-2} \Big)^{-1}\bigg] \bo \Sigma^{-1}\Big( \frac{\sum_{i: z^2_i = 1} \mb y_i}{n/2}\Big) \Bigg\} \\ 
     + &\Bigg\{ \frac{1}{2} \Big(\sum_{i: z^1_i = 2} \mb y_i\Big)\t \bo \Sigma^{-1}\bigg[\Big(\frac{n_2(\mb z^1)}{n/2}\bo \Sigma^{-1} + \frac{2}{n}\kappa_n^{-2} \bo \Sigma^{-2} \Big)^{-1} -\Big(\bo \Sigma^{-1} + \frac{2}{n}\kappa_n^{-2} \bo \Sigma^{-2} \Big)^{-1}\bigg] \bo \Sigma^{-1}\Big( \frac{\sum_{i: z^1_i = 2} \mb y_i}{n/2}\Big) \Bigg\} \\ 
     - &\Bigg\{ \frac{1}{2} \Big(\sum_{i: z^2_i = 2} \mb y_i\Big)\t \bo \Sigma^{-1}\bigg[\Big(\frac{n_2(\mb z^2)}{n/2}\bo \Sigma^{-1} + \frac{2}{n}\kappa_n^{-2} \bo \Sigma^{-2} \Big)^{-1} -\Big(\bo \Sigma^{-1} + \frac{2}{n}\kappa_n^{-2} \bo \Sigma^{-2} \Big)^{-1}\bigg] \bo \Sigma^{-1}\Big( \frac{\sum_{i: z^2_i = 2} \mb y_i}{n/2}\Big) \Bigg\}, \\ 
    \alpha_3 \coloneqq &  \log\left|\mb \Sigma^{-1} \otimes \mb I_{n_1(\bo z^1)} + \Big(\bo \Sigma^{-1}\Big(n_1(\mb z^1)\bo \Sigma^{-1} + \kappa_n^{-2} \bo \Sigma^{-2} \Big)^{-1} \bo \Sigma^{-1} \Big) \otimes (\mb 1_{n_1(\mb z^1)} \mb 1_{n_1(\mb z^1)}\t) \right| \\
    & + \log\left|\mb \Sigma^{-1} \otimes \mb I_{n_2(\bo z^1)} + \Big(\bo \Sigma^{-1}\Big(n_2(\mb z^1)\bo \Sigma^{-1} + \kappa_n^{-2} \bo \Sigma^{-2} \Big)^{-1} \bo \Sigma^{-1} \Big) \otimes (\mb 1_{n_2(\mb z^1)} \mb 1_{n_2(\mb z^1)}\t) \right| \\
    & - \log\left|\mb \Sigma^{-1} \otimes \mb I_{n_1(\bo z^1)} + \Big(\bo \Sigma^{-1}\Big(n_1(\mb z^1)\bo \Sigma^{-1} + \kappa_n^{-2} \bo \Sigma^{-2} \Big)^{-1} \bo \Sigma^{-1} \Big) \otimes (\mb 1_{n_1(\mb z^2)} \mb 1_{n_1(\mb z^2)}\t) \right| \\
    &-   \log\left|\mb \Sigma^{-1} \otimes \mb I_{n_2(\bo z^2)} + \Big(\bo \Sigma^{-1}\Big(n_2(\mb z^2)\bo \Sigma^{-1} + \kappa_n^{-2} \bo \Sigma^{-2} \Big)^{-1} \bo \Sigma^{-1} \Big) \otimes (\mb 1_{n_2(\mb z^2)} \mb 1_{n_2(\mb z^2)}\t) \right|. 
    \label{eq: decomposition of log-difference}
\end{align} 

In the analysis that follows, we show that the value of $\alpha_1$ primarily determines the sign of the log-likelihood difference $\log p_{\mb z^1}(\mb Y) - \log p_{\mb z^2}(\mb Y)$ for any fixed $\mb z^1$ that satisfies the conditions in the event $A$. A useful observation is that the conditions $(\snr^{\exc}_0)^{\frac12} (n / \tr(\bo \Sigma^2)^{\frac14} \norm{\bo \Sigma}^{\frac12} \ll 1$ and $\snr^{\exc}_0 \rightarrow \infty$ together implies $\sqrt{n}\norm{\bo \Sigma} = o ( \tr(\bo \Sigma^2)^{\frac12})$, a relation will be repeatedly invoked.

We begin by controlling the magnitude of $\alpha_2$. 
\bigskip

\emph{Upper Bound for $\alpha_2$. }
Suppose that $\mb Y = (\mb y_1, \cdots, \mb y_n)\t \in \bb R^{n \times p}$ obeys the distribution under the assignment $\mb z^1$. To upper bound the first two terms in $\alpha_2$, observe that
\begin{align}
    & \Bigg| \frac{1}{2} \Big(\sum_{i: z^1_i = 1} \mb y_i\Big)\t \bo \Sigma^{-1}\bigg[\Big(\frac{n_1(\mb z^1)}{n/2}\bo \Sigma^{-1} + \frac{2}{n}\kappa_n^{-2} \bo \Sigma^{-2} \Big)^{-1} -\Big(\bo \Sigma^{-1} + \frac{2}{n}\kappa_n^{-2} \bo \Sigma^{-2} \Big)^{-1}\bigg] \bo \Sigma^{-1}\Big( \frac{\sum_{i: z^1_i = 1} \mb y_i}{n/2}\Big)  \\ 
    & - \frac{1}{2} \Big(\sum_{i: z^2_i = 1} \mb y_i\Big)\t \bo \Sigma^{-1}\bigg[\Big(\frac{n_1(\mb z^2)}{n/2}\bo \Sigma^{-1} + \frac{2}{n}\kappa_n^{-2} \bo \Sigma^{-2} \Big)^{-1} -\Big(\bo \Sigma^{-1} + \frac{2}{n}\kappa_n^{-2} \bo \Sigma^{-2} \Big)^{-1}\bigg] \bo \Sigma^{-1}\Big( \frac{\sum_{i: z^2_i = 1} \mb y_i}{n/2}\Big)  \\ 
    & -  \frac{1}{2} \tr\bigg(\bo \Sigma^{-\frac12} \bigg[\Big(\frac{n_1(\mb z^1)}{n/2}\bo \Sigma^{-1} + \frac{2}{n}\kappa_n^{-2} \bo \Sigma^{-2} \Big)^{-1} - \Big(\frac{n_1(\mb z^2)}{n/2}\bo \Sigma^{-1} + \frac{2}{n}\kappa_n^{-2} \bo \Sigma^{-2} \Big)^{-1}\bigg]\bo \Sigma^{-\frac12} \bigg) 
    \Bigg|  \\ 
    \leq & \gamma_1 + \gamma_2 + \gamma_3 + \gamma_4, 
    \label{eq: decomposition of first two terms in alpha2}
\end{align}
where we define the quantities in the last line as follows:
\begin{align}
    & \gamma_1 \coloneqq \bigg| \frac{1}{2} \Big(\sum_{i: z^1_i = 1} \mb E_i\Big)\t \bo \Sigma^{-1}\bigg[\Big(\frac{n_1(\mb z^1)}{n/2}\bo \Sigma^{-1} + \frac{2}{n}\kappa_n^{-2} \bo \Sigma^{-2} \Big)^{-1} -\Big(\frac{n_1(\mb z^2)}{n/2}\bo \Sigma^{-1} + \frac{2}{n}\kappa_n^{-2} \bo \Sigma^{-2} \Big)^{-1}\bigg] \\ 
    &  \cdot  \bo \Sigma^{-1}\Big( \frac{\sum_{i: z^1_i = 1} \mb E_i}{n/2}\Big)  - \frac{1}{2} \tr\bigg(\bo \Sigma^{-\frac12} \bigg[\Big(\frac{n_1(\mb z^1)}{n/2}\bo \Sigma^{-1} + \frac{2}{n}\kappa_n^{-2} \bo \Sigma^{-2} \Big)^{-1} - \Big(\frac{n_1(\mb z^2)}{n/2}\bo \Sigma^{-1} + \frac{2}{n}\kappa_n^{-2} \bo \Sigma^{-2} \Big)^{-1}\bigg]\bo \Sigma^{-\frac12} \bigg)  \bigg| , \\ 
    & \gamma_2 \coloneqq  \bigg| \frac{n_1(\mb z^1)}{2} {\bo \theta_1^*}\t  \bo \Sigma^{-1}\bigg[\Big(\frac{n_1(\mb z^1)}{n/2}\bo \Sigma^{-1} + \frac{2}{n}\kappa_n^{-2} \bo \Sigma^{-2} \Big)^{-1} -\Big(\frac{n_1(\mb z^2)}{n/2}\bo \Sigma^{-1} + \frac{2}{n}\kappa_n^{-2} \bo \Sigma^{-2} \Big)^{-1}\bigg] \bo \Sigma^{-1}\bo \theta_1^* \bigg| ,\\ 
    & \gamma_3 \coloneqq \bigg| \Big(\sum_{i: z^1_i = 1} \mb E_i\Big)\t \bo \Sigma^{-1}\bigg[\Big(\frac{n_1(\mb z^1)}{n/2}\bo \Sigma^{-1} + \frac{2}{n}\kappa_n^{-2} \bo \Sigma^{-2} \Big)^{-1} -\Big(\frac{n_1(\mb z^2)}{n/2}\bo \Sigma^{-1} + \frac{2}{n}\kappa_n^{-2} \bo \Sigma^{-2} \Big)^{-1}\bigg] \bo \Sigma^{-1}\bo \theta_1^* \bigg|  \\ 
    & \gamma_4 \coloneqq \bigg| \frac{1}{2} \Big(\sum_{i: z^1_i = 1} \mb y_i\Big)\t \bo \Sigma^{-1}\bigg[\Big(\frac{n_1(\mb z^2)}{n/2}\bo \Sigma^{-1} + \frac{2}{n}\kappa_n^{-2} \bo \Sigma^{-2} \Big)^{-1} -\Big(\bo \Sigma^{-1} + \frac{2}{n}\kappa_n^{-2} \bo \Sigma^{-2} \Big)^{-1}\bigg] \bo \Sigma^{-1}\Big( \frac{\sum_{i: z^1_i = 1} \mb y_i}{n/2}\Big) \\
    & - \frac{1}{2} \Big(\sum_{i: z^2_i = 1} \mb y_i\Big)\t \bo \Sigma^{-1}\bigg[\Big(\frac{n_1(\mb z^2)}{n/2}\bo \Sigma^{-1} + \frac{2}{n}\kappa_n^{-2} \bo \Sigma^{-2} \Big)^{-1} -\Big(\bo \Sigma^{-1} + \frac{2}{n}\kappa_n^{-2} \bo \Sigma^{-2} \Big)^{-1}\bigg] \bo \Sigma^{-1}\Big( \frac{\sum_{i: z^2_i = 1} \mb y_i}{n/2}\Big)  
    \bigg|. 
\end{align}

Before proceeding to handle $\gamma_i$'s, we first apply the Hanson-Wright inequality \citep[Theorem 2.1]{hsu2012tail} to derive the following controls on the involved random vectors: 
\begin{align}
    & \norm{\mb E_1}_2 \vee \Bignorm{\sum_{i: z_i^1 = 1} \mb E_{i} / \sqrt{n}} \vee \Bignorm{\sum_{i: z_i^2 = 1} \mb E_{i} / \sqrt{n}} \\ 
    \lesssim & \sqrt{\tr(\bo \Sigma) + 2 \sqrt{\tr(\bo \Sigma^2)\cdot c\log n} + 2 c \norm{\bo \Sigma} \log n} \lesssim \sqrt{\tr(\bo \Sigma) +  \sqrt{\tr(\bo \Sigma^2)\cdot \log n} },  \label{eq: bound on random vector 1}
    \\ 
    & \kappa_n \bignorm{\sum_{i: z^1_i = 1}\bo \Sigma^{\frac12} \mb E_i / \sqrt{n}} \vee \norm{\bo \theta_1^*}_2 \lesssim \kappa_n \sqrt{\tr(\bo \Sigma^2) + 2 \sqrt{\tr(\bo \Sigma^4)\cdot c\log n} + 2 c \norm{\bo \Sigma^2} \log n} \lesssim \kappa_n \sqrt{\tr(\bo \Sigma^2)}, 
     \label{eq: bound on random vector 2}
    \\ 
    & \norm{\bo \Sigma^{\frac12} \mb y_1}_2 \vee \Big(\Bignorm{\sum_{i: z_i^2 = 1}\bo \Sigma^{\frac12}  \mb y_i} / \sqrt{n}\Big)  
    \leq  \sqrt{n}\norm{\bo \Sigma^{\frac12} \bo \theta_1^*}_2 + \norm{\bo \Sigma^{\frac12} \mb E_{1}}_2 \vee \Big(\Bignorm{\sum_{i: z_i^2 = 1} \bo \Sigma^{\frac12} \mb E_{i} / \sqrt{n}} \Big) \\ 
    \lesssim  & \kappa_n \sqrt{\tr(\bo \Sigma^3) + \sqrt{\tr(\bo \Sigma^6) \log n} + \norm{\bo \Sigma}^3 \log n}  +\sqrt{\tr(\bo \Sigma^2)}  
    \lesssim  \sqrt{\tr(\bo \Sigma^2)} 
     \label{eq: bound on random vector 3}
\end{align}
with probability at least $1 - O(n^{-c})$ for some constant $c$, invoking the conditions $\norm{\bo \Sigma} \log n \ll \sqrt{\tr(\bo \Sigma^2)}$ and $\kappa_n \ll \frac{1}{\sqrt{n}\norm{\bo \Sigma}^{\frac12}}$ and the fact $\tr(\mb A \mb B) \leq \tr(\mb A) \norm{\mb B}$ for any non-negative definite matrix $\mb A$. 

Moreover, we control the spectral norm using the relation $\mb A^{-1} - \mb B^{-1} = \mb A^{-1} (\mb B - \mb A) \mb B^{-1}$: 
\begin{align}
    & \norm{\bo \Sigma^{-a} \bigg[\Big(\frac{n_1(\mb z^1)}{n/2}\bo \Sigma^{-1} + \frac{2}{n}\kappa_n^{-2} \bo \Sigma^{-2} \Big)^{-1} - \Big(\frac{n_1(\mb z^2)}{n/2}\bo \Sigma^{-1} + \frac{2}{n}\kappa_n^{-2} \bo \Sigma^{-2} \Big)^{-1}\bigg] \bo \Sigma^{-b}}  \\ 
    \leq & \norm{\bo \Sigma^{-a} \Big(\frac{n_1(\mb z^1)}{n/2}\bo \Sigma^{-1} + \frac{2}{n}\kappa_n^{-2} \bo \Sigma^{-2} \Big)^{-1} \Big( \frac{2}{n} \bo \Sigma^{-1} \Big) \Big(\frac{n_1(\mb z^2)}{n/2}\bo \Sigma^{-1} + \frac{2}{n}\kappa_n^{-2} \bo \Sigma^{-2} \Big)^{-1} \bo \Sigma^{-b}} \\ 
    \leq & \frac{2}{n} \norm{\bo \Sigma^{-a} \Big(\frac{n_1(\mb z^1)}{n/2}\bo \Sigma^{-1} + \frac{2}{n}\kappa_n^{-2} \bo \Sigma^{-2} \Big)^{-1}  \bo \Sigma^{-\frac12} } \norm{\bo \Sigma^{-b} \Big(\frac{n_1(\mb z^2)}{n/2}\bo \Sigma^{-1} + \frac{2}{n}\kappa_n^{-2} \bo \Sigma^{-2} \Big)^{-1}  \bo \Sigma^{-\frac12} }  \\ 
    \leq &    \frac{n}{2}\kappa_n^4 \norm{\bo \Sigma}^{3-a-b} ,  \\ 
    & \tr\bigg(\bo \Sigma^{-a} \Big[\Big(\frac{n_1(\mb z^1)}{n/2}\bo \Sigma^{-1} + \frac{2}{n}\kappa_n^{-2} \bo \Sigma^{-2} \Big)^{-1} - \Big(\frac{n_1(\mb z^2)}{n/2}\bo \Sigma^{-1} + \frac{2}{n}\kappa_n^{-2} \bo \Sigma^{-2} \Big)^{-1}\Big] \bo \Sigma^{-b}\bigg)  \\
    \leq &    \frac{n}{2}\kappa_n^4 \tr(\bo \Sigma)^{3-a-b} , \qquad \qquad \text{for any $a,~b \leq \frac32$. }  \label{eq: bound on covariance difference 1}
\end{align}

Similarly, one has 
\begin{align}
    &  \norm{\bo \Sigma^{-1}\bigg[\Big(\frac{n_1(\mb z^1)}{n/2}\bo \Sigma^{-1} + \frac{2}{n}\kappa_n^{-2} \bo \Sigma^{-2} \Big)^{-1} -\Big(\bo \Sigma^{-1}+ \frac{2}{n}\kappa_n^{-2} \bo \Sigma^{-2} \Big)^{-1}\bigg] \bo \Sigma^{-1} } \leq \frac{n}{2}\big|n_1(\mb z^1) - \frac{n}{2}\big|\kappa_n^4 \norm{\bo \Sigma}. 
\end{align}

As a consequence, we invoke \cite[Theorem~6.2.1]{vershynin2018high} to upper bound $\gamma_1$:  
\begin{align}
    & \gamma_1 \lesssim 
    \sqrt{\tr(\bo \Sigma^2) \cdot  \log n} \cdot n \kappa_n^4 \norm{\bo \Sigma}
    + \frac{1}{|n_1(\mb z^1)|^{\frac12}}\cdot n\kappa_n^4 \norm{\bo \Sigma}^2 \cdot \log n \lesssim \sqrt{\tr(\bo \Sigma^2)  \log n} \cdot n \kappa_n^4 \norm{\bo \Sigma}
\end{align}
with probability at least $1 - O(n^{-c})$, given $n_1(\mb z^1) \in \big[\big(1 - C\sqrt{\log n/ n}\big) \frac{n}{2}, \big(1 + C\sqrt{\log n /n }\big) \frac{n}{2}\big]$ and $\sqrt{n} \norm{\bo \Sigma} = o(\tr(\bo \Sigma^2)^{\frac12})$. 

For $\gamma_2$ and $\gamma_3$, we invoke \eqref{eq: bound on random vector 2} together with \eqref{eq: bound on covariance difference 1} and to derive that
\begin{align}
    & \gamma_2 \lesssim n \kappa_n^4 \norm{\bo \Sigma} \cdot n \kappa_n^2 \tr(\bo \Sigma^2) \lesssim \big(n\kappa_n^2 \norm{\bo \Sigma}\big) \cdot \big(n\kappa_n^4 \tr(\bo \Sigma^2)\big), \\ 
    & \gamma_3  \lesssim  \sqrt{n} \cdot \sqrt{\tr(\bo \Sigma^2)} \cdot n \kappa_n^4 \norm{\bo \Sigma}^{\frac12} \cdot \kappa_n \sqrt{\tr(\bo \Sigma^2)}  
    \lesssim  \big(n \kappa_n \norm{\bo \Sigma}^{\frac12}\big) \cdot \big(n\kappa_n^4 \tr(\bo \Sigma^2)\big)
\end{align}
with probability at least $1 - O(n^{-c})$. 

In terms of $\gamma_4$, employing \eqref{eq: bound on random vector 3} yields that with probability at least $1 - O(n^{-c})$
\begin{align}
    & \gamma_4 \lesssim n^{\frac12} \kappa_n^4 \tr(\bo \Sigma^2). 
\end{align}

Plugging these inequalities into \eqref{eq: decomposition of first two terms in alpha2} yields that with probability at least $1 - 3 n^{-c}$
\begin{align}
    & \text{LHS of \eqref{eq: decomposition of first two terms in alpha2}} \lesssim \Big[\frac{(\log n)^{\frac12} \norm{\bo \Sigma}}{\sqrt{\tr(\bo \Sigma^2)}} + n\kappa_n^2 \norm{\bo \Sigma} + n\kappa_n \norm{\bo \Sigma}^{\frac12} + n^{-\frac12} \Big] n \kappa_n^4 \tr(\bo \Sigma^2). 
    \label{eq: decomposition of last two terms in alpha1 final}
\end{align}

Besides, a similar argument leads to
\begin{align}
    & \Bigg| \frac{1}{2} \Big(\sum_{i: z^1_i = 2} \mb y_i\Big)\t \bo \Sigma^{-1}\bigg[\Big(\frac{n_2(\mb z^1)}{n/2}\bo \Sigma^{-1} + \frac{2}{n}\kappa_n^{-2} \bo \Sigma^{-2} \Big)^{-1} -\Big(\bo \Sigma^{-1} + \frac{2}{n}\kappa_n^{-2} \bo \Sigma^{-2} \Big)^{-1}\bigg] \bo \Sigma^{-1}\Big( \frac{\sum_{i: z^1_i = 2} \mb y_i}{n/2}\Big)  \\ 
     &-  \frac{1}{2} \Big(\sum_{i: z^2_i = 2} \mb y_i\Big)\t \bo \Sigma^{-1}\bigg[\Big(\frac{n_2(\mb z^2)}{n/2}\bo \Sigma^{-1} + \frac{2}{n}\kappa_n^{-2} \bo \Sigma^{-2} \Big)^{-1} -\Big(\bo \Sigma^{-1} + \frac{2}{n}\kappa_n^{-2} \bo \Sigma^{-2} \Big)^{-1}\bigg] \bo \Sigma^{-1}\Big( \frac{\sum_{i: z^2_i = 2} \mb y_i}{n/2}\Big)\\ 
    & -  \frac{1}{2} \tr\bigg(\bo \Sigma^{-\frac12} \bigg[\Big(\frac{n_2(\mb z^1)}{n/2}\bo \Sigma^{-1} + \frac{2}{n}\kappa_n^{-2} \bo \Sigma^{-2} \Big)^{-1} - \Big(\frac{n_2(\mb z^2)}{n/2}\bo \Sigma^{-1} + \frac{2}{n}\kappa_n^{-2} \bo \Sigma^{-2} \Big)^{-1}\bigg]\bo \Sigma^{-\frac12} \bigg)  \Bigg|\\ 
     \lesssim & \Big[\frac{\sqrt{p \log n} \norm{\bo \Sigma}^2}{\tr(\bo \Sigma^2)} + n\kappa_n^2 \norm{\bo \Sigma} + n\kappa_n \norm{\bo \Sigma}^{\frac12} + n^{-\frac12}\Big] n \kappa_n^4 \tr(\bo \Sigma^2) 
     \label{eq: decomposition of last two terms in alpha2 final}
\end{align}
with probability at least $1- O(n^{-c})$.

As the final step, we control the difference between the deterministic terms appearing on the left-hand sides of \eqref{eq: decomposition of first two terms in alpha2} and \eqref{eq: decomposition of last two terms in alpha2 final}. To this end, we use the relation: 
$$\mb A^{-1}  - \mb B^{-1}= \mb C^{-1} (\mb B - \mb A) \mb C^{-1} + F(\mb A, \mb B, \mb C), $$
where $F(\mb A, \mb B, \mb C) \coloneqq \mb A^{-1} (\mb C - \mb A) \mb C^{-1} (\mb B - \mb A) \mb C^{-1} + \mb A^{-1} (\mb B - \mb A) \mb C^{-1} (\mb C - \mb B) \mb B^{-1}$. 
We choose \(\mathbf C = \tfrac{4}{n}\kappa_n^{-2}\boldsymbol\Sigma^{-1}\), which yields 
\begin{align}
    &   \bigg|  \tr\bigg(\bo \Sigma^{-\frac12} \bigg[\Big(\frac{n_2(\mb z^1)}{n/2}\bo \Sigma^{-1} + \frac{2}{n}\kappa_n^{-2} \bo \Sigma^{-2} \Big)^{-1}- \Big(\frac{n_2(\mb z^2)}{n/2}\bo \Sigma^{-1} + \frac{2}{n}\kappa_n^{-2} \bo \Sigma^{-2} \Big)^{-1}\bigg]\bo \Sigma^{-\frac12} \bigg)  & \\ 
    & \qquad +  \tr\bigg(\bo \Sigma^{-\frac12} \bigg[\Big(\frac{n_1(\mb z^1)}{n/2}\bo \Sigma^{-1} + \frac{2}{n}\kappa_n^{-2} \bo \Sigma^{-2} \Big)^{-1}- \Big(\frac{n_1(\mb z^2)}{n/2}\bo \Sigma^{-1} + \frac{2}{n}\kappa_n^{-2} \bo \Sigma^{-2} \Big)^{-1}\bigg]\bo \Sigma^{-\frac12} \bigg) \bigg|  \\ 
    \leq & \Bigg|  \tr\bigg(\bo \Sigma^2 \Big[\underbrace{\Big(\frac{n_2(\mb z^1)}{n/2} \bo \Sigma^2 + \frac{2}{n}\kappa_n^{-2} \bo \Sigma\Big)^{-1}}_{\eqqcolon \mb A_1^{-1}} - \underbrace{\Big(\frac{n_2(\mb z^2)}{n/2} \bo \Sigma^2 + \frac{2}{n}\kappa_n^{-2} \bo \Sigma \Big)^{-1}}_{\eqqcolon \mb B_1^{-1}} \\ 
    & \qquad + \underbrace{\Big(\frac{n_1(\mb z^1)}{n/2}\bo \Sigma^2 + \frac{2}{n}\kappa_n^{-2} \bo \Sigma  \Big)^{-1}}_{\eqqcolon \mb A_2^{-1}} - \underbrace{\Big(\frac{n_1(\mb z^2)}{n/2} \bo \Sigma^2 + \frac{2}{n}\kappa_n^{-2} \bo \Sigma \Big)^{-1}}_{\eqqcolon \mb B_2^{-1}} \Big] \bigg) \Bigg|  \\ 
    \leq & |\tr(\bo \Sigma^2 f(\mb A_1, \mb B_1, \mb C)| + |\tr(\bo \Sigma^2 f(\mb A_2, \mb B_2, \mb C)| \\ 
    \lesssim & \big(n^{\frac12} \kappa_n^2 \log n  \norm{\bo \Sigma} \big) \cdot \big(n \kappa^4\tr(\bo \Sigma^2)\big), 
    \label{eq: decomposition of deterministic terms in alpha2 final}
\end{align}
where we used the fact that 
\[
\mathbf B_1 - \mathbf A_1 + \mathbf B_2 - \mathbf A_2 = \mathbf 0,
\qquad 
\text{since } 
n_1(\mathbf z^1) + n_2(\mathbf z^1)
=
n_1(\mathbf z^2) + n_2(\mathbf z^2)
= n,
\]
and 
\[
\big|\tr(\boldsymbol\Sigma^2\, F(\mathbf A_1, \mathbf B_1, \mathbf C))\big|
\;\vee\;
\big|\tr(\boldsymbol\Sigma^2\, F(\mathbf A_2, \mathbf B_2, \mathbf C))\big|
\;\le\;
C_F n^{3}\kappa_n^{6}
\left( \frac{\sqrt{n\log n}}{n} \right) \frac{1}{n}
\tr(\boldsymbol\Sigma^2 )\,\|\boldsymbol\Sigma\|,
\]
for some constant \(C_F\) determined by the conditions on \(\mathbf z^1\) and \(\mathbf z^2\) (cf.\ \eqref{eq: event A}).

With the bounds \eqref{eq: decomposition of last two terms in alpha1 final}, \eqref{eq: decomposition of last two terms in alpha2 final}, and \eqref{eq: decomposition of deterministic terms in alpha2 final} in place, we thereby conclude that there exists some constant $C_3$ such that
\begin{align}
    & E_{1}
     \coloneqq \bigg\{|\alpha_2| \leq C_3 \Big[\frac{(\log n)^{\frac12} \norm{\bo \Sigma}}{\sqrt{\tr(\bo \Sigma^2)}} + n\kappa_n^2 \norm{\bo \Sigma} + n\kappa_n \norm{\bo \Sigma}^{\frac12} + n^{-\frac12}  \Big] \cdot \big( n\kappa_n^4 \tr(\bo \Sigma^2) \big) \bigg\} 
\end{align}
holds with probability at least $1 - C_3 n^{-c}$. 

\bigskip

\emph{Upper Bound for $\alpha_3$. }
To begin with, we simplify the form of $\alpha_3$ as follows: 
\begin{align}
    & \alpha_3 =  \log\left|\mb \Sigma^{-1} \otimes \mb I_{n_1(\bo z^1)} + \Big(\bo \Sigma^{-1}\Big(n_1(\mb z^1)\bo \Sigma^{-1} + \kappa_n^{-2} \bo \Sigma^{-2} \Big)^{-1} \bo \Sigma^{-1} \Big) \otimes (\mb 1_{n_1(\mb z^1)} \mb 1_{n_1(\mb z^1)}\t) \right| \\
    & + \log\left|\mb \Sigma^{-1} \otimes \mb I_{n_2(\bo z^1)} + \Big(\bo \Sigma^{-1}\Big(n_2(\mb z^1)\bo \Sigma^{-1} + \kappa_n^{-2} \bo \Sigma^{-2} \Big)^{-1} \bo \Sigma^{-1} \Big) \otimes (\mb 1_{n_2(\mb z^1)} \mb 1_{n_2(\mb z^1)}\t) \right| \\
    & - \log\left|\mb \Sigma^{-1} \otimes \mb I_{n_1(\bo z^1)} + \Big(\bo \Sigma^{-1}\Big(n_1(\mb z^1)\bo \Sigma^{-1} + \kappa_n^{-2} \bo \Sigma^{-2} \Big)^{-1} \bo \Sigma^{-1} \Big) \otimes (\mb 1_{n_1(\mb z^2)} \mb 1_{n_1(\mb z^2)}\t) \right| \\
    & - \log\left|\mb \Sigma^{-1} \otimes \mb I_{n_2(\bo z^2)} + \Big(\bo \Sigma^{-1}\Big(n_2(\mb z^2)\bo \Sigma^{-1} + \kappa_n^{-2} \bo \Sigma^{-2} \Big)^{-1} \bo \Sigma^{-1} \Big) \otimes (\mb 1_{n_2(\mb z^2)} \mb 1_{n_2(\mb z^2)}\t) \right| \\ 
    & + \log|\bo \Sigma \otimes \mb I_{n_1(\mb z^1)} |+ \log|\bo \Sigma \otimes \mb I_{n_2(\mb z^1)}|  - \log|\bo \Sigma \otimes \mb I_{n_1(\mb z^2)}| - \log|\bo \Sigma \otimes \mb I_{n_2(\mb z^2)}| \\ 
    = & \log\left|\mb I_{pn_1(\bo z^1)} + \Big(\big(n_1(\mb z^1)\bo \Sigma^{-1} + \kappa_n^{-2} \bo \Sigma^{-2} \Big)^{-1} \bo \Sigma^{-1} \Big) \otimes (\mb 1_{n_1(\mb z^1)} \mb 1_{n_1(\mb z^1)}\t) \right| \\ 
    & +  \log\left|\mb I_{pn_2(\bo z^1)} + \Big(\big(n_2(\mb z^1)\bo \Sigma^{-1} + \kappa_n^{-2} \bo \Sigma^{-2} \big)^{-1} \bo \Sigma^{-1} \Big) \otimes (\mb 1_{n_2(\mb z^1)} \mb 1_{n_2(\mb z^1)}\t) \right|  \\ 
    & -  \log\left|\mb I_{pn_1(\bo z^2)} + \Big( \big(n_1(\mb z^2)\bo \Sigma^{-1} + \kappa_n^{-2} \bo \Sigma^{-2} \big)^{-1} \bo \Sigma^{-1} \Big) \otimes (\mb 1_{n_1(\mb z^2)} \mb 1_{n_1(\mb z^2)}\t) \right| \\ 
    & - \log\left|\mb I_{pn_2(\bo z^2)} + \Big(\big(n_2(\mb z^2)\bo \Sigma^{-1} + \kappa_n^{-2} \bo \Sigma^{-2} \big)^{-1} \bo \Sigma^{-1} \Big) \otimes (\mb 1_{n_2(\mb z^2)} \mb 1_{n_2(\mb z^2)}\t) \right|. 
    \label{eq: simplification of alpha3}
\end{align}

Notice that, for each $i\in[p]$, the $i$-th largest eigenvalue with $i \in [p]$ of $\mb I _{pn_{k_1}(\mb z^{k_2})} +  \Big(\big(n_1(\mb z^{k_2})\bo \Sigma^{-1} + \kappa_n^{-2} \bo \Sigma^{-2} \Big)^{-1} \bo \Sigma^{-1} \Big) \otimes (\mb 1_{n_1(\mb z^{k_2})} \mb 1_{n_1(\mb z^{k_2})}\t) $ is expressed as $$
1 + n_{k_1}(\mb z^{k_2}) \cdot \lambda_i\Big(\big(n_{k_1}(\mb z^{k_2}) + \kappa_n^{-2} \bo \Sigma^{-1} \big)^{-1}  \Big) =  \frac{2 n_{k_1}(\mb z^{k_2}) + \frac{1}{\kappa_n^2 \lambda_i(\bo \Sigma)}}{n_{k_1}(\mb z^{k_2}) + \frac{1}{\kappa_n^2 \lambda_i(\bo \Sigma)}}.$$ And the rest of the eigenvalues all equal one. 
This allows us to focus on the ratios among the eigenvalues. For convenience, we define that $f(\mb x) \coloneqq \log(\frac{2x + 1  }{x+1})$ and $z_{i,k_1,k_2} \coloneqq n_{k_1}(\mb z^{k_2})\kappa_n^2 \lambda_i(\bo \Sigma) $. Then, we
rewrite $\alpha_3$ as 
\begin{align}
    & \alpha_3 = \log\Big(\prod_{i\in[p]} \frac{2 n_1(\mb z^1)\kappa_n^2 \lambda_i(\bo \Sigma) + 1 }{n_1(\mb z^1)\kappa_n^2 \lambda_i(\bo \Sigma) + 1} \cdot \frac{2 n_2(\mb z^1)\kappa_n^2 \lambda_i(\bo \Sigma) + 1}{n_2(\mb z^1)\kappa_n^2 \lambda_i(\bo \Sigma) + 1} \cdot \frac{ n_1(\mb z^2)\kappa_n^2 \lambda_i(\bo \Sigma) + 1}{2n_1(\mb z^2)\kappa_n^2 \lambda_i(\bo \Sigma) + 1} \cdot \frac{ n_1(\mb z^1)\kappa_n^2 \lambda_i(\bo \Sigma)+ 1}{2n_1(\mb z^1)\kappa_n^2 \lambda_i(\bo \Sigma) + 1}\Big) \\ &    =  \sum_{i\in[p]}\bigg(
    f(z_{i,1,1}) +  f(z_{i,2,1}) - f(z_{i,1,2}) -  f(z_{i,2,2})
    \bigg). 
\end{align}

Next, we notice that the second-order derivative of the function $f$ is bounded for $x \geq 0$. We thereby apply the mean-value theorem and derive that for some constant $C_f$, 
\begin{align}
     |\alpha_3| \leq&  \sum_{i\in[p]} \Big| 
    f'(z_{i,1,*})(z_{i,1,1} - z_{i,1,2}) + f'(z_{i,2,*})(z_{i,2,1} - z_{i,2,2})
    \Big| \\ 
    \leq & \sum_{i\in[p]} \Big[ \big| f'(z_{i,1,*})\big| \big|z_{i,1,1} - z_{i,1,2} + z_{i,2,1} - z_{i,2,2}\big| + \big| 
    f'(z_{i,1,*}) - f'(z_{i,2,*})\big|\big|z_{i,1,1} - z_{i,1,2}\big| \Big] \\ 
    \leq &  \sum_{i\in[p]} \big| 
    f'(z_{i,1,*}) - f'(z_{i,2,*})\big|\big|z_{i,1,1} - z_{i,1,2}\big| 
    \leq  C_f \sum_{i\in[p]} \kappa_n^4 \lambda_i(\bo \Sigma)^2 
    \leq C_f  \kappa_n^4 \tr(\bo \Sigma^2), 
\end{align} 
where $z_{i,1,*}$ and $z_{i,2,*}$ are given by the mean-value theorem and $C_f$ is some constant upper bounding the second derivative of $f$ for $x \geq 0$.

\bigskip

\emph{Analysis of $\alpha_1$. }
Now we are left with characterizing the behavior of $\alpha_1$. To begin with, we invoke the Gaussianity of the noise and derive that 
\begin{align}
    & \alpha_1  = {\bo \theta_1^*}\t \bo \Sigma^{-1} \big(\bo \Sigma^{-1} + \frac{2}{n}\kappa_n^{-2}\bo\Sigma^{-2} \big)^{-1}\bo \Sigma^{-1} \bo \theta_1^* + \mb E_1\t  \bo \Sigma^{-1} \big(\bo \Sigma^{-1} + \frac{2}{n}\kappa_n^{-2}\bo\Sigma^{-2} \big)^{-1}\bo \Sigma^{-1}\big( \mb E' + (\bo \theta_1^* + \bo \theta_2^*)\big) \\ 
    & - {\bo \theta_1^*}\t \bo \Sigma^{-1} \big(\bo \Sigma^{-1} + \frac{2}{n}\kappa_n^{-2}\bo\Sigma^{-2} \big)^{-1}\bo \Sigma^{-1}\bo \theta_2^* + {\bo \theta_1^*}\t \bo \Sigma^{-1} \big(\bo \Sigma^{-1} + \frac{2}{n}\kappa_n^{-2}\bo\Sigma^{-2} \big)^{-1}\bo \Sigma^{-1}\mb E'  ,
    \label{eq: decomposition of alpha1}
\end{align}
where $\mb E'$ obeys the centered Gaussian distribution with covariance $\frac{4}{n}\bo \Sigma$. In the sequel, we control each term on the right-hand side of \eqref{eq: decomposition of alpha1}. 

First of all, employing the Hanson-Wright inequality gives that 
\begin{align}
    &  {\bo \theta_1^*}\t \bo \Sigma^{-1} \big(\bo \Sigma^{-1} + \frac{2}{n}\kappa_n^{-2}\bo\Sigma^{-2} \big)^{-1}\bo \Sigma^{-1} \bo \theta_1^* \leq \frac{n}{2}\kappa_n^2 \norm{\bo \theta_1^*}_2^2 \leq \frac{n}{2}\kappa_n^4 \Big[\tr(\bo \Sigma^2) + 2 \sqrt{\tr(\bo \Sigma^4)\cdot c\log n} + 2 c \norm{\bo \Sigma^2} \log n\Big]
    \label{eq: E2 eq1}
\end{align}
with probability at least $1 - n^{-c}$. 

We further note that the second term $ \mb E_1\t  \bo \Sigma^{-1} \big(\bo \Sigma^{-1} + \frac{2}{n}\kappa_n^{-2}\bo\Sigma^{-2} \big)^{-1}\bo \Sigma^{-1}\big( \mb E' + (\bo \theta_1^* + \bo \theta_2^*)\big)$ is distributionally identical to $\epsilon_1 \norm{\bo \Sigma^{-\frac12} \big(\bo \Sigma^{-1} + \frac{2}{n}\kappa_n^{-2}\bo\Sigma^{-2} \big)^{-1} \big( \frac{2}{\sqrt{n}}\bo\Sigma^{-\frac12} \bo \epsilon_2 + \kappa_n \bo \epsilon_3 \big)}_2$ with $\epsilon_1 \sim \mc N(0,1)$ and $\bo \epsilon_2, \bo \epsilon_3 \sim \mc N(0,\mb I_p)$. Since $\epsilon_2$ is independent of $\epsilon_3$, then with probability $\frac{1}{2}$ the following holds: 
\begin{align}
    & \norm{\bo \Sigma^{-\frac12} \big(\bo \Sigma^{-1} + \frac{2}{n}\kappa_n^{-2}\bo\Sigma^{-2} \big)^{-1} \big( \frac{2}{\sqrt{n}}\bo\Sigma^{-\frac12} \bo \epsilon_2 + \kappa_n \bo \epsilon_3 \big)}_2^2 \\ 
    =  & \norm{\bo \Sigma^{-\frac12} \big(\bo \Sigma^{-1} + \frac{2}{n}\kappa_n^{-2}\bo\Sigma^{-2} \big)^{-1}  \frac{2}{\sqrt{n}}\bo\Sigma^{-\frac12} \bo \epsilon_2 }_2^2  \\ 
    & + \norm{\bo \Sigma^{-\frac12} \big(\bo \Sigma^{-1} + \frac{2}{n}\kappa_n^{-2}\bo\Sigma^{-2} \big)^{-1}\kappa_n \bo \epsilon_3}_2^2 \\ 
    & +\frac{2}{\sqrt{n}} \bo \epsilon_2 \t \bo\Sigma^{-\frac12}\big(\bo \Sigma^{-1} + \frac{2}{n}\kappa_n^{-2}\bo\Sigma^{-2} \big)^{-1} \bo \Sigma^{-1} \big(\bo \Sigma^{-1} + \frac{2}{n}\kappa_n^{-2}\bo\Sigma^{-2} \big)^{-1}\kappa_n \bo \epsilon_3 \\ 
    \geq &  \norm{\bo \Sigma^{-\frac12} \big(\bo \Sigma^{-1} + \frac{2}{n}\kappa_n^{-2}\bo\Sigma^{-2} \big)^{-1}  \frac{2}{\sqrt{n}}\bo\Sigma^{-\frac12} \bo \epsilon_2 }_2^2  \\ 
    & + \norm{\bo \Sigma^{-\frac12} \big(\bo \Sigma^{-1} + \frac{2}{n}\kappa_n^{-2}\bo\Sigma^{-2} \big)^{-1}\kappa_n \bo \epsilon_3}_2^2 \\ 
    \geq & \frac{4}{n}\norm{\bo \Sigma^{-\frac12} \big(\bo \Sigma^{-1} + \frac{2}{n}\kappa_n^{-2}\bo\Sigma^{-2} \big)^{-1}  \bo\Sigma^{-\frac12} \bo \epsilon_2 }_2^2
\end{align}
Moreover, applying \cite[Theorem~6.2.1]{vershynin2018high} yields, with probability at least $1 - C_4 n^{-c}$ for some constant $C_4$,  
\begin{align}
    & \norm{\bo \Sigma^{-\frac12} \big(\bo \Sigma^{-1} + \frac{2}{n}\kappa_n^{-2}\bo\Sigma^{-2} \big)^{-1}\bo \Sigma^{-\frac12} \bo \epsilon_2}_2 \\ 
    \geq & \norm{\bo \Sigma^{-\frac12}\big(\bo \Sigma^{-1} + \frac{2}{n}\kappa_n^{-2}\bo\Sigma^{-2} \big)^{-1} \bo \Sigma^{-\frac12}}_F - \Big(  C_4 n\kappa_n^2 \sqrt{\mathrm{Tr}(\bo \Sigma^2) \log n}+   C n\kappa_n^2 \norm{\bo \Sigma^2} \Big)^{\frac{1}{2}}  \\ 
    \geq & \frac{n}{2}\kappa_n^2 \tr(\bo \Sigma^2)^{\frac12} -\norm{\bo \Sigma^{-\frac12}\big(\bo \Sigma^{-1} + \frac{2}{n}\kappa_n^{-2}\bo\Sigma^{-2} \big)^{-1} \bo \Sigma^{-\frac12} - \frac{n}{2}\kappa_n^2 \bo \Sigma}_F  -\Big(  C_4 n\kappa_n^2 \sqrt{\mathrm{Tr}(\bo \Sigma^2) \log n}+  C n\kappa_n^2 \norm{\bo \Sigma^2} \Big)^{\frac{1}{2}} \\ 
    \geq &  \frac{n}{2}\kappa_n^2 \tr(\bo \Sigma^2)^{\frac12} - \frac{n^2}{4} \kappa_n^4 \mathrm{Tr}(\bo \Sigma^4)^{\frac12}  -\Big(  C_4 n\kappa_n^2 \sqrt{\mathrm{Tr}(\bo \Sigma^2) \log n}+   C_4 n\kappa_n^2 \norm{\bo \Sigma^2} \Big)^{\frac{1}{2}}   ,
\end{align}
 which leads to 
\begin{align}
    & \norm{\bo \Sigma^{-\frac12} \big(\bo \Sigma^{-1} + \frac{2}{n}\kappa_n^{-2}\bo\Sigma^{-2} \big)^{-1} \big( \frac{2}{\sqrt{n}}\bo\Sigma^{-\frac12} \bo \epsilon_2 + \kappa_n \bo \epsilon_3 \big)}_2 \\ 
    \geq & \sqrt{n}\kappa_n^2 \tr(\bo \Sigma^2)^{\frac12} - \frac{1}{2} n^{\frac32}\kappa_n^4 \mathrm{Tr}(\bo \Sigma^4)^{\frac12}  -2\Big(  C_4 \kappa_n^2 \sqrt{\mathrm{Tr}(\bo \Sigma^2) \log n}+   C_4 \kappa_n^2 \norm{\bo \Sigma^2} \Big)^{\frac{1}{2}}  
     \label{eq: E2 eq2} 
\end{align}
with probability at least $\frac{1}{2} - C_4 n^{-c}$. 

For the rest of the terms in \eqref{eq: decomposition of alpha1}, arguments parallel to those for the second term yield the bounds as follows: 
\begin{align}
    &  \left|{\bo \theta_1^*}\t \bo \Sigma^{-1} \big(\bo \Sigma^{-1} + \frac{2}{n}\kappa_n^{-2}\bo\Sigma^{-2} \big)^{-1}\bo \Sigma^{-1}\bo \theta_2^* \right| \leq C_5 n\kappa_n^4 \big[\tr(\bo \Sigma^4)^{\frac12} + \tr(\bo \Sigma^8)^{\frac14}(\log n)^{\frac14} + \norm{\bo \Sigma}^4 (\log n)^{\frac12} \big]\sqrt{\log n},  \label{eq: E2 eq3}
    \\ 
    & \left|{\bo \theta_1^*}\t \bo \Sigma^{-1} \big(\bo \Sigma^{-1} + \frac{2}{n}\kappa_n^{-2}\bo\Sigma^{-2} \big)^{-1}\bo \Sigma^{-1}\mb E' \right| \leq C_5 n^{\frac12}\kappa_n^3 \big[\tr(\bo \Sigma^3)^{\frac12} + \tr(\bo \Sigma^6)(\log n)^{\frac14} + \norm{\bo \Sigma}^3(\log n)^\frac12 \big] \sqrt{\log n}, \label{eq: E2 eq4},
\end{align}
each holding with probability at least $1 - C_5 n^{-c}$ for some constant $C_5$. For convenience, define the event:
\begin{align}
    & E_2 \coloneqq \left\{\eqref{eq: E2 eq1},~\eqref{eq: E2 eq2},~\eqref{eq: E2 eq3},~\eqref{eq: E2 eq4}\text{~hold} \right\},
\end{align}
which occurs with probability exceeding $\frac{1}{2} - C_6 n^{-c}$ for some constant $C_6$.

\bigskip

Equipped with the above building blocks, we are finally ready to lower bound the probability of the event $\big\{\log p_{\mb z^1}(\mb Y) - \log p_{z^2}(\mb Y) \leq 0, \mb z_1^* = 1\big\}$. Specifically, one has 
\begin{align}
    & \bb P_{\mc P_{\mb z}\otimes \mc P_{\bo \theta}}\Big[\log p_{\mb z^1}(\mb Y) - \log p_{\mb z^2}(\mb Y) \leq 0, \mb z_1^* = 1 \Big] \\ 
    \geq & \frac{1}{2}\Big[\big(1-\frac{1}{2}- (C_3+C_6)n^{-c}\big)\bb P_{\epsilon \sim \mc N(0,1)}\big[ \epsilon \geq \xi\big]  - (C_3+C_6)n^{-c}\Big],
\end{align}
where $\xi$ is defined as
\begin{align}
    &\xi \coloneqq  \frac{1}{\sqrt{n}\kappa_n^2 \tr(\bo \Sigma^2)^{\frac12} - \frac{1}{2} n^{\frac32} \kappa_n^4 \mathrm{Tr}(\bo \Sigma^4)^{\frac12}  -2 \Big(  C_4 \kappa_n^2 \sqrt{\mathrm{Tr}(\bo \Sigma^2) \log n}+   C_4 \kappa_n^2 \norm{\bo \Sigma^2} \Big)^{\frac{1}{2}}} \\ 
    & \cdot \bigg\{\frac{n}{2}\kappa_n^4 \Big[\tr(\bo \Sigma^2) + 2 \sqrt{\tr(\bo \Sigma^4)\cdot c\log n} + 2 c \norm{\bo \Sigma^2} \log n\Big] \\ 
    & \qquad + C_5 n\kappa_n^4 \big[\tr(\bo \Sigma^4)^{\frac12} + \tr(\bo \Sigma^8)^{\frac14}(\log n)^{\frac14} + \norm{\bo \Sigma}^4 (\log n)^{\frac12} \big]\sqrt{\log n}\\  
    & \qquad + C_5 n^{\frac12}\kappa_n^3 \big[\tr(\bo \Sigma^3)^{\frac12} + \tr(\bo \Sigma^6)(\log n)^{\frac14} + \norm{\bo \Sigma}^3(\log n)^\frac12 \big] \sqrt{\log n} \\ 
    & \qquad + C_3 \Big[\frac{(\log n)^{\frac12} \norm{\bo \Sigma}}{\sqrt{\tr(\bo \Sigma^2)}} + n\kappa_n^2 \norm{\bo \Sigma} + n\kappa_n \norm{\bo \Sigma}^{\frac12} + n^{-\frac12}  \Big] \cdot \big( n\kappa_n^4 \tr(\bo \Sigma^2) \big)  + C_f  \kappa_n^4 \tr(\bo \Sigma^2) \bigg\} \\ 
    = & (1 + o(1)) \mathsf{SNR}_0^{\mathsf{exc}} .
\end{align}
The last equality arises from the facts that 
\begin{align}
    & n \tr(\bo \Sigma^2) \kappa_n^4  / 4   = {\snr_0^{\exc}}^2 \Big(1 + C_\kappa \sqrt{\log n}\norm{\bo \Sigma} / \sqrt{\tr(\bo \Sigma^2)}\Big)^4 = (1 + o(1)) {\snr^{\exc}_0}^2 ,\\ 
    & n^{\frac32} \kappa_n^4 \tr(\bo \Sigma^4)^{\frac12} \leq n^{\frac32} \kappa_n^4 \Big(\tr(\bo \Sigma^2) \Big)^{\frac12} \norm{\bo \Sigma} \ll  n \kappa_n^4 \tr(\bo \Sigma^2), \\
    & \Big(\kappa_n \big( \tr(\bo \Sigma^2) \log n\big)^{\frac14}\Big) \vee \Big(\kappa_n \norm{\bo \Sigma}\Big) \ll n^{\frac12} \kappa_n^2 \tr(\bo \Sigma^2)^{\frac12}, \\ 
    & n^{\frac{2a - 4}{a}} \tr(\bo \Sigma^{a})^{\frac{2}{a}} \leq n^{\frac{2a - 4}{a}}\tr(\bo \Sigma^{2})^{\frac{2}{a}} \norm{\bo \Sigma}^{\frac{2a -4}{a}} = o(\tr(\bo \Sigma^2)), \quad  \text{for any $a > 2$},
\end{align}
where we used the condition $\sqrt{n} \norm{\bo \Sigma} = o( \sqrt{\tr(\bo \Sigma^2)})$.

Invoking the tail probability of the standard normal distribution together with \eqref{eq: lower bounding the minimax rate} at the beginning concludes the proof: 
\begin{align}
    & \inf_{\hat{\mb z}} \sup_{(\mb z^*, \{\bo \theta_k^*\}_{k\in[2]}, \{\bo \Sigma_k\}_{k\in[2]}) \in \bo \Theta^{\mathsf{exc}}} \bb E[h(\hat{\mb z},\mb z^*)] \geq \frac{1}{2}\bb P_{\mc P_{\mb z}\otimes \mc P_{\bo \theta}}\Big[\log p_{\mb z^1}(\mb Y) - \log p_{z^2}(\mb Y) \leq 0, \mb z_1^* = 1 \Big] - n^{-c} \\ 
    \geq & \exp\Big(-\big(1 + o(1)\big)\frac{{\mathsf{SNR}^{\mathsf{exc}}}^2}{2} \Big). 
\end{align}

\begin{proof}[Proof of Claim~\ref{claim: lower bound on P[A]}]
    Applying Hoeffding's inequality with some constant $C_2$ gives that the event 
\begin{align}
    & A_1 \coloneqq \Big\{\Big|\sum_{i\in[\lceil\frac{n}{3}\rceil]}\ind\{z_i^* = 1\} - \lceil\frac{n}{3}\rceil/2\Big|  \vee \Big|\sum_{i\in[\lceil\frac{n}{3}\rceil]}\ind\{z_i^* = 2\} - \lceil\frac{n}{3}\rceil/2\Big| \leq C_2 \sqrt{n \log n} \Big\}
    \label{eq: definition of A1}
\end{align}
holds with probability at least $1 - O(n^{-C_2})$ under $\mc P_{\mb z}$. Next, we shall study the low-rank structure of the matrix $\mb Z^*{\bo \Theta^*}\t$. Consider a variant of $\bo \Theta^*$ after orthogonalization: 
\begin{align}
    & \tilde{\bo \Theta}^* = \left(\kappa_n \sqrt{\tr(\bo \Sigma^2)} \mb v_1, \kappa_n \sqrt{\tr(\bo \Sigma^2)} \mb v_2\right), ~ \text{where $ \mb v_1 = \frac{\bo \theta_1^*}{\norm{\bo \theta_1^*}_2}, \mb v_2 = \frac{\bo \theta_2^* (\mb I_p - \mb v_1 \mb v_1\t ) }{\norm{\bo \theta_2^* (\mb I_p - \mb v_1 \mb v_1\t )}_2}$}. 
\end{align}
Elementary algebra gives that the left top-$2$ singular vectors of $\mb Z^*{{}\tilde{\bo \Theta}^*}\t$ can be written as $\tilde{\mb U}^* \coloneqq \mb Z^*\diag(\frac{1}{\sqrt{n_k(\mb z^*)}})_{k\in[2]}$, with the corresponding eigenvalues being $\sqrt{n_k(\mb z^*)} \norm{\tilde{\bo \Theta}^*_{\cdot, k}}_2$ for $k \in[K]$. 
Denote the left top-$2$ singular vectors of $\mb Z^*{\bo \Theta^*}\t$ by $\mb U^*$ and the diagonal matrix of the eigenvalues of $\mb Z^*{\bo \Theta^*}\t$ by $\bo \Lambda^*$. 
We note that the identity $\mb U^* = \tilde{\mb U}^* {{}\tilde{\mb U}^*}\t \mb U^*$ holds, since $\mb U^*$ and $ \tilde{\mb U}^*$ span the same subspace (the column space of $\mb Z^*$). 

Moreover, by the Hanson-Wright inequality, one has 
\begin{align}
    & \max_{k\in[2]}\Big| \kappa_n^2 \tr(\bo \Sigma^2)  - \norm{\bo \theta_k^*}_2^2\Big| \leq \frac{1}{2}C_6 \kappa_n^2 \sqrt{\tr(\bo \Sigma^4)\log n} + \frac{1}{2} C_6 \kappa_n^2\norm{\bo \Sigma^2} \log n \leq C_6 \kappa_n^2 \sqrt{\tr(\bo \Sigma^4)\log n} , \\ 
    & \Big| {\bo \theta_2^*}\t \bo \theta_1^* \Big| = \Big| {\bo \theta_2^*}\t \bo \Sigma^{-1}   \frac{\bo \Sigma \bo \theta_1^*}{\norm{\bo \Sigma \bo \theta_1^*}_2} \Big|\cdot \norm{\bo \Sigma \bo \theta_1^*}_2 \leq C_6\kappa_n^2 \sqrt{\tr(\bo \Sigma^4)\log n}
\end{align}
for some constant $C_6$ with probability at least $1 - O(n^{-c})$, provided that $\log n \ll p$. By the definitions of $\mb v_1$ and $\mb v_2$, these inequalities lead to 
\begin{align} 
     &\frac{1}{C_5}\kappa_n \sqrt{n \tr(\bo \Sigma^2)}\leq \sigma_{2}(\mb Z^*{\bo \Theta^*}\t) \leq \norm{\mb Z^*{\bo \Theta^*}\t} \leq C_7 \kappa_n \sqrt{n\tr(\bo \Sigma^2)},\label{eq: concentration on Theta and tilde Theta 1}
     \\
     &\norm{\bo \Theta - \tilde{\bo \Theta}^*}\leq  \big|\norm{\bo \theta_1^*}_2 - \kappa_n \sqrt{\tr(\bo \Sigma^2)}\big|+ 2\bignorm{{\bo \theta_2^*}\t \mb v_1}_2 + \Big|\norm{\bo \theta_2^*} - \kappa_n \sqrt{\tr(\bo \Sigma^2)} \Big|\leq C_7 \kappa_n \sqrt{\log n} \norm{\bo\Sigma},\label{eq: concentration on Theta and tilde Theta 2}
     \\ 
    &\norm{\tilde{\bo \Lambda}^* - \bo \Lambda^*}\leq  \norm{\mb Z^*(\bo \Theta^* - \tilde{\bo \Theta}^*)\t } \leq C_7 \kappa_n \sqrt{n  \log n} \norm{\bo \Sigma}
    \label{eq: concentration on Theta and tilde Theta 3}
\end{align}
for some constant $C_7$ with probability at least $1 - O(n^{-c})$ conditional on the event $ A_1$ defined in \eqref{eq: definition of A1}. Provided these concentrations above, we define an event $A_2$ with probability exceeding $1- O(d^{-c})$ as 
\begin{equation}
    A_2 \coloneqq \{\text{ \eqref{eq: concentration on Theta and tilde Theta 1}, \eqref{eq: concentration on Theta and tilde Theta 2}, and \eqref{eq: concentration on Theta and tilde Theta 3} hold}\}. 
\end{equation}

Now we turn to controlling the magnitude of $\mathsf{SNR}^{\exc}$. Recall the revelant quantities $\mb w_k^* =  {\mb V^*}\t \bo \theta_k^*$ and $\mb S_k^{\exc} \coloneqq {\bo \Lambda^*}^{-1}{\mb U^*}\t \bb E\big[\mc H(\mb E \mb E\t)_{\cdot, i} \mc H(\mb E \mb E\t)_{i,\cdot}\big] \mb U^* {\bo \Lambda^*}^{-1}$ for $k\in[2]$ and any $i \in [n]$ with $z_i^* = k$ in the definition of $\snr^{\exc}$. 
One useful observation from the fact $\mb S^{\exc}_k  \preceq \tr(\bo \Sigma^2) \cdot \bo {\Lambda^*}^{-2} \eqqcolon  \mb S^*$ for $k \in[2]$ is that 
\begin{align}
    & \mathsf{SNR}^{\exc}(\mb z^*, \{ \bo \theta_k^*\}_{k\in[2]}, \{\bo \Sigma_k\}_{k\in[2]})^2 \geq  \frac{1}{4}\norm{(\mb w_1^*- \mb w_2^*){\mb S^*}^{-\frac12}}_2^2\\ 
\geq & \Big[1 - \frac{ \norm{{{}\tilde{\mb S}^*}^{-1} - {\mb S^*}^{-1}}}{ \sigma_{2}({{}\tilde{\mb S}^*}^{-1})} \Big] \cdot\frac{1}{4}\Bignorm{(\mb w_1^*- \mb w_2^*){{}\tilde{\mb S}^*}^{-\frac12}}_2^2 \\ 
\geq & \Big[1 - \frac{\sigma_1(\tilde{\mb S}^*)\norm{\tilde{\mb S}^* - \mb S^*}}{ \sigma_{2}(\tilde{\mb S}^*) \sigma_{2}(\mb S^*)} \Big] \cdot\frac{1}{4}\Bignorm{(\mb w_1^*- \mb w_2^*){{}\tilde{\mb S}^*}^{-\frac12}}_2^2
\label{eq: SNReh alternative}
\end{align}
where $
    \tilde{\mb S}^* \coloneqq \tr(\bo \Sigma^2) \bo \Lambda^* ({{}\tilde{\mb U}^*}\t \mb U^*)\t {{}\tilde{\bo \Lambda}^*}^{-4}  ({{}\tilde{\mb U}^*}\t \mb U^*) \bo \Lambda^* $. 
 Here $\tilde{\bo \Lambda}^*$ represents the diagonal matrix of the eigenvalues of $\mb Z{{}\tilde{\bo \Theta}^*}\t$. By the definitions of $\tilde{\mb S}^*$ and $\mb S^*$, one has 
\begin{align}
    & \norm{\tilde{\mb S}^* - \mb S^*} \leq 2  \tr(\bo \Sigma^2) \Bignorm{ ({{}\tilde{\mb U}^*}\t \mb U^*) \t {{}\tilde{\bo \Lambda}^*}^{-2} ({{}\tilde{\mb U}^*}\t \mb U^*) - {\bo \Lambda^*}^{-2}  } \Bignorm{{\bo \Lambda^*}^{-1}} \bignorm{\bo \Lambda^* }  \\ 
    & + \tr(\bo \Sigma^2) \Bignorm{ ({{}\tilde{\mb U}^*}\t \mb U^*) \t {{}\tilde{\bo \Lambda}^*}^{-2} ({{}\tilde{\mb U}^*}\t \mb U^*) - {\bo \Lambda^*}^{-2}  }^2 \norm{\bo \Lambda^*}^2. 
    \label{eq: bound tilde Dk - Dk}
\end{align}
Regarding the term $\Bignorm{ ({{}\tilde{\mb U}^*}\t \mb U^*) \t {{}\tilde{\bo \Lambda}^*}^{-2} ({{}\tilde{\mb U}^*}\t \mb U^*) - {\bo \Lambda^*}^{-2}}$ in \eqref{eq: bound tilde Dk - Dk}, we break it down as follows:  
\begin{equation}
     \Bignorm{ ({{}\tilde{\mb U}^*}\t \mb U^*) \t {{}\tilde{\bo \Lambda}^*}^{-2} ({{}\tilde{\mb U}^*}\t \mb U^*) - {\bo \Lambda^*}^{-2}  }  
    \leq    \frac{\Bignorm{ ({{}\tilde{\mb U}^*}\t \mb U^*)\t  {{}\tilde{\bo \Lambda}^*}^{2} ({{}\tilde{\mb U}^*}\t \mb U^*) -  {\bo \Lambda^*}^{2}  }}{\sigma_2({{}\tilde{\bo \Lambda}^*})^2  \sigma_2(\bo\Lambda^*)^2}.
    \label{eq: tilde U tilde Lambda sgn - U Lambda}
\end{equation}

Further, for $\bignorm{ ({{}\tilde{\mb U}^*}\t \mb U^*)\t  {{}\tilde{\bo \Lambda}^*}^{2} ({{}\tilde{\mb U}^*}\t \mb U^*) -  {\bo \Lambda^*}^{2}  }$ one has the following:
\begin{align}
    & \Bignorm{ ({{}\tilde{\mb U}^*}\t \mb U^*)\t  {{}\tilde{\bo \Lambda}^*}^{2} ({{}\tilde{\mb U}^*}\t \mb U^*) -  {\bo \Lambda^*}^{2}  } 
    \leq   \Bignorm{\mb Z^*{{}\tilde{\bo \Theta}^*}\t \tilde{\bo \Theta}^* {\mb Z^*}\t - \mb Z^*{\bo \Theta^*}\t \bo \Theta^*{\mb Z^*}\t } \\ 
    \leq & C_8  \kappa_n^2 n \sqrt{\tr(\bo \Sigma^2) \log n} \norm{\bo \Sigma}
    \label{eq: sgnt tilde Lambda sgn - Lambda}
\end{align}
with some constants $C_8$, invoking \eqref{eq: concentration on Theta and tilde Theta 1} and \eqref{eq: concentration on Theta and tilde Theta 3}. 

Therefore, plugging \eqref{eq: tilde U tilde Lambda sgn - U Lambda}, and \eqref{eq: sgnt tilde Lambda sgn - Lambda} into \eqref{eq: bound tilde Dk - Dk} yields that for some constant $C_9$,
\begin{align}
    & \norm{\tilde{\mb S}^* - \mb S^*} \leq C_9  \frac{ \sqrt{\log n} \norm{\bo \Sigma}}{\kappa_n^2 n \tr(\bo \Sigma^2)^{\frac12} }. 
\end{align}
This also leads to:
\begin{equation}
    \frac{\sigma_1(\tilde{\mb S}^*)\norm{\tilde{\mb S}^* - \mb S^*}}{ \sigma_{2}(\tilde{\mb S}^*) \sigma_{2}(\mb S^*)} \leq 
     C_{10} \frac{ \sqrt{\log n} \norm{\bo \Sigma}}{ \tr(\bo \Sigma^2)^{\frac12}}\leq C_{10} \frac{\sqrt{\log n}\norm{\bo \Sigma}}{\sqrt{\tr(\bo \Sigma^2)}}
     \label{eq: upper bound on factor 1 in snr lower bound}
\end{equation}
for some constant $C_{10}$. 

For the alternative to $\mathsf{SNR}^{\mathsf{exc}}$ appearing in \eqref{eq: SNReh alternative}, it is further simplified as follows: 
\begin{align}
    & \Bignorm{(\mb w_1^*- \mb w_2^*){{}\tilde{\mb S}^*}^{-\frac12}}_2^2 \\ 
    =&  \tr(\bo \Sigma^2)^{-1} \big(\frac{1}{\sqrt{n_1(\mb z^*)}} , - \frac{1}{\sqrt{n_2(\mb z^*)}} \big)\t \diag\big(\kappa_n^4 (n_1(\mb z^*)\tr(\bo \Sigma^2))^2, \kappa_n^4 (n_2(\mb z^*)\tr(\bo \Sigma^2))^2\big)
    \big(\frac{1}{\sqrt{n_1(\mb z^*)}} , - \frac{1}{\sqrt{n_2(\mb z^*)}} \big) \\ 
    = & \kappa_n^4  n \tr(\bo \Sigma^2). 
    \label{eq: upper bound on factor 2 in snr lower bound}
\end{align}

Finally, we choose $\kappa_n = \sqrt{2\mathsf{SNR}_0^{\mathsf{exc}}} \big(n\tr(\bo\Sigma^2)\big)^{-\frac14} \Big[(1 + C_\kappa \frac{\sqrt{\log n}\norm{\bo \Sigma}}{\sqrt{\tr(\bo \Sigma^2)}})\Big]$
with some sufficiently large constant $C_\kappa$. 
Taking \eqref{eq: upper bound on factor 1 in snr lower bound}, \eqref{eq: upper bound on factor 2 in snr lower bound}, and \eqref{eq: SNReh alternative} collectively implies that 
\begin{align}
    &  \mathsf{SNR}^{\exc}(\mb z^*, \{ \bo \theta_k^*\}_{k\in[2]}, \{\bo \Sigma_k\}_{k\in[2]})^2 \geq {\mathsf{SNR}^{\mathsf{exc}}_0}^2
\end{align} 
under the event $ A_1 \cap A_2 \subseteq A$ with probability exceeding $1 -O(n^{-c})$, where we used the condition that $\sqrt{\log n}\norm{\bo \Sigma} = o( \sqrt{\tr(\bo \Sigma^2)})$. 
\end{proof}

\section{Linear Approximation for Singular Subspace Estimation}
\label{sec: proof of linear approximation for singular subspace estimation}
This section develops the proof of our row-wise singular subspace perturbation theory (Theorem~\ref{thm: singular subspace perturbation theory}), which accommodates the weak-signal regime. As elaborated in Section~\ref{subsec: subspace perturbation}, although row-wise subspace perturbation theory has been extensively developed in a variety of settings, most existing results are derived by taking a maximum of $\ell_2$ fluctuation norms over all rows. Such an argument is prone to yield an extra $\sqrt{\log d}$ factor in the residual term's upper bound --- an effect that becomes non-negligible relative to the linear term in the weak-consistency regime --- along with a polynomial probability tail. 

An notable exception is the $\ell_p$ theory developed in \cite{abbe2022}, which used a refined analysis to avoid logarithmic factors under weak signals. In contrast to our result, their analysis is tailored to sub-Gaussian noise, whereas our approximation applies to (high-probability) bounded noise, which requires additional techniques to handle block dependencies.

Recall that we consider a general rank-$r$ low-rank matrix $\mb M^*\in \bb R^{n\times p}$ and its noisy version $\mb M = \mb M^* + \mb E \in \bb R^{n \times p}$ and perform the eigen decomposition on $\mc H(\mb M\mb M\t)$ to obtain $\mb U \in \bb R^{n \times r}$ as an estimate for the top-$K$ left singular vectors of $\mb M^*$. We let $d \coloneqq n \vee p$.

Before proceeding, we introduce a noise matrix $\mb W$ and variants of $\mb W$ that decouple the dependency when needed:   
\begin{align}
    & \mb W = \mc H\big(\mb E\mb E\t + \mb E {\mb M^*}\t + \mb M^* \mb E\t\big) -\diag(\mb M^*{\mb M^*}\t), 
\\
    & \mb W^{(i)} \coloneqq \mc H(\mc P_{-i,\cdot}(\mb E) \mc P_{-i,\cdot}(\mb E)\t + \mc P_{-i,\cdot}(\mb E){\mb M^*}\t + \mb M^* \mc P_{-i,\cdot}(\mb E)\t  ) - \diag(\mb M^*{\mb M^*}\t), 
        \label{eq: leave-one-row-out W}
    \\ 
    & \mb W^{(-c)} \coloneqq \mc H(\mc P_{\cdot,-S_{c}}(\mb E) \mc P_{\cdot,-S_{c}}(\mb E)\t + \mc P_{\cdot,-S_{c}}(\mb E){\mb M^*}\t + \mb M^* \mc P_{\cdot,-S_{c}}(\mb E)\t  ) - \diag(\mb M^*{\mb M^*}\t) 
    \label{eq: leave-one-block-out W}
\end{align}
for $i\in[n], ~ c\in [l]$. Here,   \(\mathcal P_{-i,\cdot}(\cdot)\) zeroes out the \(i\)-th row of its matrix argument,   \(\mathcal P_{\cdot,-j}(\cdot)\) zeroes out the \(j\)-th column, and  \(\mathcal P_{-S,\cdot}(\cdot)\) and \(\mathcal P_{\cdot,-S}(\cdot)\) zero out all rows or columns indexed by a set \(S\), respectively.

At a high level, our strategy is as follows: We first develop a decomposition for $ \bignorm{(\mb U\mb U\t \mb U^* - \mb U^*)_{i,\cdot} - \mb W_{i,\cdot} \mb U^* {\bo \Lambda^*}^{-2}}_2$, which is closely connected to the ultimate target $\bignorm{\mb U_{i,\cdot} (\mb U\t \mb U^*) - \mb U^*_{i,\cdot}\big(\mb I - (\mb M^*{\mb M^*}\t)_{i,i} {\bo \Lambda^*}^{-2}\big) - \mk L_i }_2$. This decomposition expresses the error as a partial sum of products involving the noise matrix and deterministic matrices, as formalized in Theorem~\ref{thm: linear proximity for singular subspace perturbation}. Next, in Lemmas~\ref{lemma: bound for quantities in expansion (Gaussian case)} and \ref{lemma: bound for quantities in expansion (bounded case)}, we separately bound these product terms under Gaussian noise and under bounded noise. Finally, we assemble all components to obtain the final bounds in \eqref{eq: final decomposition for linear approximation of singular subspace perturbation 1} and \eqref{eq: final decomposition for linear approximation of singular subspace perturbation 2}.

\begin{theorem}
    Instate the assumptions in Theorem~\ref{thm: singular subspace perturbation theory}. Then, with probability at least $1 - O(d^{-c_p})$, one has 
    \begin{align}
        &     \norm{(\mb U\mb U\t \mb U^* - \mb U^*)_{i,\cdot} - \big((\mb I - \mb U^*{\mb U^*}\t)\mb W \mb U^* {\bo \Lambda^*}^{-2} \big)_{i,\cdot}}_2 \\ 
        \lesssim &   \norm{\mb U^*}\ti \cdot \deltaop^2 / (\sigma_r^*)^4  +  \Big[ \sigma_1^*\norm{\mb E_{i,\cdot} \mb V^*}_2 + \norm{\mb E_{i,\cdot} \mc P_{-i,\cdot}(\mb E)\t \mb U^*}_2 \Big] \cdot \deltaop / (\sigma_r^*)^4 \\ 
    + &  \sum_{h=0}^{\log n -1}  \norm{\mb E_{i, \cdot}\mc P_{-i,\cdot}(\mb E)\t \mc P^0 \big(\mc P^0 \mb W^{(i)} \mc P^0 \big)^h \mb W^{(i)} \mb U^*}_2 \cdot (\sigma_r^*)^{-2(h+2)} . 
    \end{align} 
    \label{thm: linear proximity for singular subspace perturbation}
\end{theorem}
We make note of the RHS of the bound presented in Theorem~\ref{thm: linear proximity for singular subspace perturbation}: in the Gassuain case, a direct application of the Hanson-Wright inequality, conditional on $\mb E_{-i,\cdot}$, already provides the desired concentration as stutied in \cite{agterberg2022entrywise} (see Lemma~\ref{lemma: bound for quantities in expansion (Gaussian case)}). On the other hand, under bounded noise, which was not treated in \cite{abbe2022,agterberg2022entrywise}, one must carefully handle the terms $\bignorm{\mc P_{-i,\cdot}(\mb E)\t \mc P^0 \big(\mc P^0 \mb W^{(i)} \mc P^0 \big)^h \mb W^{(i)} \mb U^*}\ti$ arising from the matrix Bernstein inequality. To address this, we employ a delicate inductive argument together with the leave-one-block-out version $\mb W^{(-c)}$ defined 
as developed in the proof of Lemma~\ref{lemma: bound for quantities in expansion (bounded case)}. Compared to the leave-two-out strategy adopted in \cite{cai2021subspace,yan2024inference}, our leave-one-block-out approach is more concise and has a weaker requirement on signal strength.

To interpret the terms involved in Theorem~\ref{thm: linear proximity for singular subspace perturbation}, we have the following linear approximation error guarantees for Gaussian noise and bounded noise with local dependence, respectively, with the proofs deferred to Section~\ref{subsec: proof of bounds for quantities in expansion}. 

\begin{lemma}
    Instate the Gaussian noise assumption in Theorem~\ref{thm: singular subspace perturbation theory}. Then we have for any $k \geq 0$ and some constant $C_0$ that 
    \begin{align}
        & \norm{\mb E_{i,\cdot} \mc P_{-i,\cdot}(\mb E)\t \mc P^0 (\mc P^0 \mb W^{(i)} \mc P^0 )^k \mc P^0 \mb W^{(i)} \mb U^*}_2 \leq C_0 \sqrt{\frac{\mu_1 r }{n}} \deltaop^{t+2} (\sqrt{\log d} \wedge t_0)
    \end{align}
    with probabiltiy at least $1 -  O(e^{-t_0^2 / 2} \vee d^{-c_p -1 })$. 
    \label{lemma: bound for quantities in expansion (Gaussian case)}
\end{lemma}
\begin{lemma}
    Instate the assumption for bounded noise with local dependence in Theorem~\ref{thm: singular subspace perturbation theory}. Then we have simultaneuously for every $k$ with $0 \leq k \leq k' \leq d$ that 
    \begin{align}
        & \norm{\mb E_{i,\cdot} \mc P_{-i,\cdot}(\mb E)\t \mc P^0 (\mc P^0 \mb W^{(i)} \mc P^0 )^k \mc P^0 \mb W^{(i)} \mb U^*}_2 \leq  C_0 \sqrt{\frac{\mu_1 r }{n}} \big[ C_0\deltaop\big]^{t+2} (\sqrt{\log d} \wedge t_0 +  \sqrt{\log k'})
    \end{align}
    with probabiltiy at least $1 -   O(e^{-t_0^2 / 2} \vee d^{-c_p - 1})$. 
    \label{lemma: bound for quantities in expansion (bounded case)}
\end{lemma}

As a consequence of the above lemmas, we have the proof of Theorem~\ref{thm: singular subspace perturbation theory} below: 
\begin{proof}[Proof of Theorem~\ref{thm: singular subspace perturbation theory}]
    Substituting the bounds from Lemma~\ref{lemma: bound for quantities in expansion (Gaussian case)} or Lemma~\ref{lemma: bound for quantities in expansion (bounded case)}, together with the estimates for $\norm{\mb E_{i,\cdot} \mb V^*}_2$ and $\norm{\mb E_{i,\cdot} \mc P_{-i,\cdot}(\mb E)\t \mb U^*}_2$ in Lemma~\ref{lemma: noise matrix concentrations using the universality (Gaussian)} or Lemma~\ref{lemma: noise matrix concentrations using the universality (bounded)} (stated lalter), into the decomposition in Theorem~\ref{thm: linear proximity for singular subspace perturbation}, we obtain that  
    \begin{align}
        & \norm{(\mb U\mb U\t \mb U^* - \mb U^*)_{i,\cdot} - \big((\mb I - \mb U^*{\mb U^*}\t)\mb W \mb U^* {\bo \Lambda^*}^{-2} \big)_{i,\cdot}}_2 \lesssim \sqrt{\frac{\mu_1 r}{n}} \frac{\deltaop^2}{(\sigma_r^*)^4} + \frac{\deltaop}{(\sigma_r^*)^4} \Big[\bar \sigma \sqrt{r} \sigma_1^* + \sigma \tilde\sigma \sqrt{pr} \Big]t_0\\ 
        & \qquad + \sum_{h =0}^{\log n-1} C_0 \sqrt{\frac{\mu_1 r }{n}} \big[ C_0\deltaop\big]^{t+2} ( t_0 + \sqrt{\log \log n}) \cdot (\sigma_r^*)^{-2(h+2)} \\ 
        \lesssim & \sqrt{\frac{\mu_1 r}{n}}\cdot \frac{\deltaop^2 (t_0 + \sqrt{\log \log n})}{(\sigma_r^*)^4}
    \end{align}
    holds with probability at least $1 - O(e^{-t^2_0/2}\vee d^{-c_p})$ since $\deltaop \ll (\sigma_r^*)^2$. Finally, we define the bias term as  
    \eq{
     \bias_{z_i^*}\t \coloneqq - \mb U^*_{i,\cdot} {\mb U^*}\t \big(\mc H(\mb E\mb E\t) + \mb E {\mb M^*}\t  \big) \mb U^*{\bo \Lambda^*}^{-2} - (\mb I - \mb U^*{\mb U^*}\t )_{i,\cdot} \diag(\mb M^*{\mb M^*}\t) \mb U^* {\bo \Lambda^*}^{-2},\text{~for $i\in[n]$}. 
     \label{eq: definition of bias}
    }
    Then, one has from the definition of $\mb W$ that 
    \begin{align}
        & \bignorm{\mb U_{i,\cdot} (\mb U\t \mb U^*) - \mb U^*_{i,\cdot} - \mk L_i\t - \bias_{z_i^*}\t  }_2   \\ 
        \leq & \bignorm{(\mb U\mb U\t \mb U^* - \mb U^*)_{i,\cdot} - \big((\mb I - \mb U^*{\mb U^*}\t)\mb W \mb U^* {\bo \Lambda^*}^{-2} \big)_{i,\cdot}}_2 + 2\norm{(\mb I - \mb U^*{\mb U^*}\t )_{i,\cdot} \diag(\mb E{\mb M^*}\t) \mb U^*{\bo \Lambda^*}^{-2}}_2 \\ 
        \leq &  \sqrt{\frac{\mu_1 r}{n}}\cdot \Big[ \frac{\deltaop^2 (t_0 + \sqrt{\log \log n})}{(\sigma_r^*)^4} +  \frac{\mu_1^{\frac12} r \sigma_1^* \bar \sigma \sqrt{ \log d / n} }{(\sigma_K^*)^2 } \Big] \lesssim  \sqrt{\frac{\mu_1 r}{n}}\cdot \Big[ \frac{\deltaop^2 t_0}{(\sigma_r^*)^4} +  \frac{\mu_1^{\frac12} r \sigma_1^* \bar \sigma \sqrt{ \log d / n} }{(\sigma_K^*)^2 } \Big],
        \label{eq: final decomposition for linear approximation of singular subspace perturbation 1} 
        \\ 
         & \hspace{-1cm} \norm{\bias_{z_i^*}}_2    
        \leq   \norm{\mb U^*}\ti \bignorm{{\mb U^*}\t \big(\mc H(\mb E\mb E\t) + \mb E {\mb M^*}\t - \diag(\mb M^*{\mb M^*}\t) \big) \mb U^*{\bo \Lambda^*}^{-2}} + \bignorm{\text{diag}(\mb M^*{\mb M^*}\t) \mb U^*{\bo \Lambda^*}^{-2}}\ti\\ 
        \lesssim & \sqrt{\frac{\mu_1 r}{n}}\cdot  \frac{r \tilde\sigma \sigma \sqrt{p \log d} + \sigma_1^* \bar \sigma r \sqrt{\log d} + \mu_1 r (\sigma_1^*)^2 / n}{(\sigma_r^*)^2}  \lesssim \sqrt{\frac{\mu_1 r}{n}}\cdot  \frac{r \deltaop \sqrt{\log d / n} + \mu_1 r (\sigma_1^*)^2 / n}{(\sigma_r^*)^2}
        \label{eq: final decomposition for linear approximation of singular subspace perturbation 2}, 
    \end{align}
    holds with probability at least $1 - O(e^{-t^2_0/2}\vee d^{-c_p})$ for $t \gg \sqrt{\log \log n}$. Here, we used the following facts by Lemmas~\ref{lemma: noise matrix concentrations using the universality (Gaussian)}~and~\ref{lemma: noise matrix concentrations using the universality (bounded)} that, with probability at least $1- O(d^{-c_p})$, 
    \begin{align}
    & |\mb E_{i,\cdot} \mb M_{i,\cdot}\t | \leq \norm{\mb E_{i,\cdot} \mb V^*}_2 \sigma_1^* \norm{\mb U^*}\ti \lesssim \bar \sigma \sqrt{r \log d} \sigma_1^* \sqrt{\frac{\mu_1 r}{n}},  \\ 
    & \norm{{\mb U^*}\t \mb E \mb V^*} \leq r^{\frac12} \max_{j \in[r]}\bignorm{\frac{1}{\sqrt{n_k}}\sum_{i\in \mc I_k(\mb z^*)}  \mb E_{i,\cdot} \mb V^*}_2 \lesssim r \bar \sigma \sqrt{\log n}, \\ 
    &  \bignorm{{\mb U^*}\t \mc H(\mb E \mb E\t) \mb U^*} \lesssim r \tilde\sigma \sigma \sqrt{p \log d}. 
    \end{align}
\end{proof}

\subsection{Proof of Theorem~\ref{thm: linear proximity for singular subspace perturbation}}

In what follows, we shall present the necessary components for proving Theorem~\ref{thm: linear proximity for singular subspace perturbation}. 
We start with the control on the spectral norm of $\mb W$ and its variants, with its proof deferred to Section~\ref{subsubsec: proof of W concentration}. Its proof for bounded cases involves a combination of two universality results in \cite{brailovskaya2022universality}, which sharpens \cite[Lemma~1]{cai2021subspace} by removing logarithmic factors --- a refinement that is crucial for handling the weak-signal regime.
\begin{lemma}
    Instate the assumptions in Theorem~\ref{thm: singular subspace perturbation theory}. Under Gaussian noise one has 
    \begin{equation}
        \max\Big\{ \bignorm{\mb W}, ~\max_{i\in[n]} \bignorm{\mb W^{(i)}} \Big\}  \leq  C_W \Big( \tilde \sigma^2 n + \sigma \tilde \sigma\sqrt{np}   +  \bar \sigma  \sqrt{n}  \sigma_1^* + \frac{\mu_1 r }{n} {\sigma_1^*}^2 \Big) \eqqcolon \deltaop 
    \end{equation}
    for some constant $C_W$ with probability at least $1-O(d^{-c_p - 2})$ with some constant $c_p >0$. Moreover,
    under bounded noise, one has that 
\begin{align}
    & \max\Big\{ \bignorm{\mb W}, ~\max_{i\in[n]} \bignorm{\mb W^{(i)}}, ~\max_{c\in[l]} \bignorm{\mb W^{(-c)}} 
    \Big\}  
    \leq \deltaop , \qquad {\max_{c \in [l]} \norm{\mb W^{(-c)} \mb U^*}\ti \leq C_H \sqrt{\frac{\mu_1 r}{n}} \deltaop \sqrt{\log d}}
\end{align}
for some sufficiently large constant $C_W$ with probability at least $1-O(d^{-c_p - 2})$. 
\label{lemma: concentration for W and its variants}
\end{lemma}

An important tool in our analysis is the expansion formula developed in \cite{xia2021normal}, as stated below. 
\begin{lemma}[Theorem~1 in \cite{xia2021normal}]
\label{lemma: expansion from xia2021normal}
Given $\norm{\mb W} \ll (\sigma_r^*)^2$, then it holds that
    $$
         \mb U \mb U\t - \mb U^* {\mb U^*}\t  = \sum_{k\geq 1} \mc S_k,
    $$
    where the terms $\mc S_k$, $k \geq 1$ are defined as 
    $$
        \mc S_k =  \mc S_{\mb R^*{\mb R^*}\t,k }(\mb W)\coloneqq \sum_{\mb s: s_1 + \cdots + s_{k+1} = k} (-1)^{1 + \tau(\mb s)} \mc P^{-s_1} \mb W \mc P^{-s_2} \mb W\cdots \mc P^{-s_p} \mb W \mc P^{-s_{p+1}}
    $$
    with $\mc P^0 = \mb U^*_{\perp}{\mb U^*_\perp}\t$ and $\mc P^{-k} = \mb U^*{\bo \Lambda^*}^{-2k} {\mb U^*}\t$. 
\end{lemma}

To justify the first-order approximation provided by Lemma~\ref{lemma: expansion from xia2021normal}, it is sufficient to show that the higher-order terms are negligible. The following lemma, aligned the spirit of Lemma~6 in \cite{agterberg2022entrywise}, establishes a sharp upper bound that holds in the weak-signal regime. Its proof is postponed to Section~\ref{subsubsec: proof of lemma: upper bound for SkU}. 
\begin{lemma}
    Instate the assumptions in Theorem~\ref{thm: singular subspace perturbation theory}.  Then given $c_p >0$, one has for any $i \in [n]$ that 
    \begin{align}
        & \norm{\big[ \mc S_k\big]_{i,:} \mb U^*}_2 \leq   C_0  \bigg\{ \Big[4C_0 \deltaop\Big]^{k-1}  \Big[ \sigma_1^*\norm{\mb E_{i,\cdot} \mb V^*}_2 + \norm{\mb E_{i,\cdot} \mc P_{-i,\cdot}(\mb E)\t \mb U^*}_2 \Big]  (\sigma_r^*)^{-2k} \\ 
    + &  \sum_{h=0}^{k-2}  \Big[4C_0  \deltaop\Big]^{k-h-2}  \norm{\mb E_{i, \cdot}\mc P_{-i,\cdot}(\mb E)\t \mc P^0 \big(\mc P^0 \mb W^{(i)} \mc P^0 \big)^h \mc P^0 \mb W^{(i)} \mb U^*}_2 (\sigma_r^*)^{-2k}  +  \norm{\mb U^*}\ti  \cdot \Big[4C_0\deltaop / (\sigma_r^*)^2\Big]^{k} \bigg\} 
    \end{align}
    holds with probabability at least $1- O(d^{-c_p-1})$ for some constant $C_0$. 
        \label{lemma: upper bound for SkU}
\end{lemma}

Building upon Lemma~\ref{lemma: expansion from xia2021normal} and Lemma~\ref{lemma: upper bound for SkU}, we demonstrate Theorem~\ref{thm: linear proximity for singular subspace perturbation} as follows. 
\begin{proof}[Proof of Theorem~\ref{thm: linear proximity for singular subspace perturbation}]
    First, one has $\deltaop / (\sigma_r^*)^2 \leq \frac{1}{e}$ by assumption. Then, applying Lemma~\ref{lemma: upper bound for SkU} then yields the following upper bound on the partial sum: 
    \begin{align}
            & \sum_{h=2}^{\log n+1}\norm{(\mc S_k)_{i,\cdot} \mb U^*}_2 \leq  \frac{C_0}{1 - 1/e} \cdot \bigg\{ \Big[4C_0 \deltaop\Big] \cdot \Big[ \sigma_1^*\norm{\mb E_{i,\cdot} \mb V^*}_2 + \norm{\mb E_{i,\cdot} \mc P_{-i,\cdot}(\mb E)\t \mb U^*}_2 \Big] \cdot (\sigma_r^*)^{-4} \\ 
        + &  \sum_{h=0}^{t-2}  \norm{\mb E_{i, \cdot}\mc P_{-i,\cdot}(\mb E)\t \mc P^0 \big(\mc P^0 \mb W^{(i)} \mc P^0 \big)^h \mc P^0 \mb W^{(i)} \mb U^*}_2 \cdot (\sigma_r^*)^{-2(h+2)}   +  \norm{\mb U^*}\ti \cdot \Big[4C_0\deltaop / (\sigma_r^*)^2\Big]^{2} \bigg\},
    \end{align}
    with probability at least $1 - O(d^{-c_p-1} \log d) = 1 - O(d^{-c_p})$. 
    Next, consider the tail sum. We have
    \begin{align}
         & \sum_{h = \log n + 2 }^{+\infty}\norm{(\mc S_h)_{i,\cdot} \mb U^*}_2 \leq \sum_{h = \log n + 2 }^{+\infty}\norm{\mc S_h \mb U^*} \leq \sum_{k=\log n +2}^{+\infty} \norm{\mb W^h} \cdot (\sigma_r^*)^{-2h} \\ 
         \leq &  \sum_{h =\log n+2}^{+\infty} e^{-\log n} \norm{\mb W}^{h - \log n} \cdot (\sigma_k^*)^{-2( h - \log n )} \leq \frac{1}{n} \Big[\deltaop / (\sigma_r^*)^2\Big]^{2} \leq \norm{\mb U^*}\ti \Big[\deltaop / (\sigma_r^*)^2\Big]^{2}. 
    \end{align}
    
    Combining the above bounds with Lemma~\ref{lemma: expansion from xia2021normal} yields with probability at least $1- O(d^{-c_p})$ that
    \begin{align}
        & \norm{(\mb U\mb U\t \mb U^* - \mb U^*)_{i,\cdot} - \mb W \mb U^* {\bo \Lambda^*}^{-2}}_2 \leq \sum_{h=2}^{\log n+1}\norm{(\mc S_h)_{i,\cdot} \mb U^*}_2 + \sum_{h = \log n + 2 }^{+\infty}\norm{(\mc S_h)_{i,\cdot} \mb U^*}_2
        \\ 
        \lesssim & \norm{\mb U^*}\ti \cdot \deltaop^2 / (\sigma_r^*)^4  +  \Big[ \sigma_1^*\norm{\mb E_{i,\cdot} \mb V^*}_2 + \norm{\mb E_{i,\cdot} \mc P_{-i,\cdot}(\mb E)\t \mb U^*}_2 \Big] \cdot \deltaop / (\sigma_r^*)^4 \\ 
    + &  \sum_{h=0}^{\log n-1}  \norm{\mb E_{i, \cdot}\mc P_{-i,\cdot}(\mb E)\t \mc P^0 \big(\mc P^0 \mb W^{(i)} \mc P^0 \big)^h \mb W^{(i)} \mb U^*}_2 \cdot (\sigma_r^*)^{-2(h+2)}  . 
    \end{align}
\end{proof}

\subsubsection{Proof of Lemma~\ref{lemma: concentration for W and its variants}}
\label{subsubsec: proof of W concentration}
    For the Gaussian case, it comes from the Hanson-Wright inequality for an arbitrary $\mb v\in \bb S^{n-1}$:  
    \longeq{
        & \big|\mb v\t \mb E \mb E\t \mb v - \bb E[\mb v\t \mb E \mb E\t \mb v]\big| \lesssim \bignorm{\sum_{i\in[n]} v_i^2 \bo \Sigma_{z_i^*}}_F \sqrt{n + \log d} +  \bignorm{\sum_{i\in[n]} v_i^2 \bo \Sigma_{z_i^*}} \big( n + \log d \big) 
         \lesssim  \tilde\sigma \sigma \sqrt{np} + \tilde\sigma^2 n 
    }
    with probability at least $1 - O(d^{-c_p-2} 9^{-n})$. Let $\mc N$ be a $\frac14$-net of the $n$-sphere with $|\mc N| \leq 9^n$. Hence, one has $\norm{\mb E \mb E\t - \bb E[\mb E \mb E\t]} \leq 2 \max_{v\in \mc N}\big|\mb v\t \mb E \mb E\t \mb v - \bb E[\mb v\t \mb E \mb E\t \mb v]\big|$, which leads to 
    \begin{equation}
        \norm{\mb E \mb E\t - \bb E[\mb E \mb E\t]}  \lesssim \tilde\sigma \sigma \sqrt{np} + \tilde\sigma^2 n 
    \end{equation}
    with probability at least $1 - O(d^{-c_p-2})$. Since $\norm{\diag(\mb E \mb E\t) - \diag(\bb E[\mb E \mb E\t]) } \leq \norm{\mb E\mb E\t - \bb E[\mb E \mb E\t]}$, one obtains
    \begin{equation}
        \bignorm{\mc H(\mb E\mb E\t)}  \leq \bignorm{\mb E\mb E\t - \bb E[\mb E \mb E\t]} + \bignorm{\diag(\mb E \mb E\t) - \diag(\bb E[\mb E \mb E\t]) } \lesssim \tilde\sigma \sigma \sqrt{np} + \tilde\sigma^2 n
    \end{equation}
    with probability at least $1 - O(d^{-c_p-2})$. 

    Moreover, one has 
    \eq{
    \bignorm{\mc H(\mb E{\mb M^*}\t)} \leq \bignorm{\mb E {\mb M^*}\t} + \bignorm{\diag(\mb E {\mb M^*}\t)} \leq 2\bignorm{\mb E {\mb M^*}\t} \leq 2 \max_{\mb v\in \mc N} \bignorm{\mb v\t \mb E{\mb M^*}\t}_2. 
    \label{eq: H(EM*t) relation}
    }
    To upper bound $\norm{\mb v\t \mb E\mb M^*}_2$ with an arbitrary $\mb v$, an application of the Hanson-Wright inequality yields that 
    \longeq{
        & \norm{\mb v\t \mb E\mb M^*}_2^2 \leq (\sigma_1^*)^2 \norm{\mb v\t \mb E \mb V^*}_2^2 \\ 
        \lesssim &  \sigma_1^* \bignorm{\sum_{i\in[n]}v_i^2 ({\mb V^*}\t \bo \Sigma_{z_i^*} \mb V^*)}_F \sqrt{ (n + \log d)} + \sigma_1^* \bignorm{\sum_{i\in[n]}v_i^2 ({\mb V^*}\t \bo \Sigma_{z_i^*} \mb V^*)} (n + \log d ) \\ 
         \lesssim &  (\sigma_1^*)^2 \sqrt{Kn} \bar\sigma^2 +( \sigma_1^*)^2 \bar \sigma^2 n \lesssim  (\sigma_1^*)^2 \bar \sigma^2 n 
    }
    with probability at least $1- O(d^{-c_p-2}9^{-n})$. Taking a union bound for the RHS of \eqref{eq: H(EM*t) relation} yields that 
    \begin{align}
        & \bignorm{\mc H(\mb E{\mb M^*}\t)} \lesssim \sigma_1^* \bar \sigma \sqrt{n}
    \end{align}
    with probability at least $1- O(d^{-c_p-2})$.

    These bounds taken collectively with the fact that $\bignorm{\diag(\mb M^* {\mb M^*}\t)} \leq \mu_1^2 K / n$ yields that 
    \begin{align}
        & \norm{\mb W} \leq \bignorm{\mc H(\mb E {\mb M^*}\t + \mb M^* \mb E\t + \mb E \mb E\t )} + \bignorm{\diag(\mb M^* {\mb M^*}\t)} \leq 
        C_W \Big(\tilde\sigma \sqrt{np} + \tilde\sigma^2 n  + \bar \sigma \sqrt{n} \sigma_1^*  +   \frac{\mu_1 r }{n} {\sigma_1^*}^2 \Big)
    \end{align}
    holds for some constant $C_W$ with probability at least $1-O(d^{-c_p-2})$.

    For the bounded noise with local dependence, we separately parse each component in $\mb W$ hereafter. 
    \begin{enumerate}
        \item For the quadratic noise term $\mb E \mb E\t - \bb E[\mb E\mb E\t]$, our approach splits into two regimes depending on whether $mn (\log d)^6 \gtrsim p$ or $mn (\log d)^6 \lesssim p$, respectively. 
        
        \emph{Low-dimensional regime. }
        For the relatively low dimensional regime with $mn (\log d)^6 \gtrsim p$, we resort to the universality results (Theorem~2.5 in \cite{brailovskaya2022universality}) on a linerized version of $\mb E\mb E - \bb E[\mb E \mb E\t ]$ so as to relate it to its Gaussian analog, whose concentration is resolved by the result in \cite{amini2021concentration}. 

        The linearization argument, following \cite[Section~3.3]{bandeira2023matrix}, seeks to tie the quadratic form in interest to the squared singular values of a matrix associated with $\mb E$. To be specific, consider 
        \begin{align}
            & \check{\mb E} = \left(\begin{matrix}
                \mb 0 & \mb E & (\norm{\bb E[\mb E \mb E\t ]} \mb I_n - \bb E[\mb E \mb E\t ] )^{\frac12} \\ 
                \mb E\t & \mb 0 & \mb 0 \\ 
                (\norm{\bb E[\mb E \mb E\t ]} \mb I_n - \bb E[\mb E \mb E\t ] )^{\frac12} & \mb 0 & \mb 0
            \end{matrix}\right),
        \end{align}
        and let $\check{\mb G}$ be its Gaussian counterpart, namely, a matrix with Gaussian entries having the same means and covariance structure. 

        Simple algebra gives that the eigenvalues of $\tilde{\mb E}$ coincides with those of $\mb E\mb E\t - \bb E[\mb E\mb E\t] + \norm{\bb E[\mb E\mb E\t]} \mb I_n$, each with multiplicity $2$. Since $|a^{\frac12} - b^{\frac12}|(a^{\frac12} + b^{\frac12}) = |a -b|$ for $a, b\geq 0$, the relation
        $$ \mathrm d_H(\mathrm{sp}(\check{\mb E}), \mathrm{sp}(\check{\mb G})) \leq \epsilon,
        $$
        where $\mathrm{sp}(\mb X)$ denotes the set of the eigenvalues of a matrix $\mb X$ and $\mathrm d_H(A, B) = \inf\{\epsilon >0: A \subseteq B + [-\epsilon, \epsilon]\text{~and~} B \subseteq A + [-\epsilon, \epsilon]\}$, 
        implies that 
        \begin{align}
             & \Big|\lambda_{\max}\big(\mb E\mb E\t - \bb E[\mb E\mb E\t] + \norm{\bb E[\mb E\mb E\t]} \mb I_n \big) -  \lambda_{\max}\big(\mb G\mb G\t - \bb E[\mb E\mb E\t] + \norm{\bb E[\mb E\mb E\t]} \mb I_n \big) \Big|\\ 
             \vee & \Big|\lambda_{\min}\big(\mb E\mb E\t - \bb E[\mb E\mb E\t] + \norm{\bb E[\mb E\mb E\t]} \mb I_n \big) -  \lambda_{\min}\big(\mb G\mb G\t - \bb E[\mb E\mb E\t] + \norm{\bb E[\mb E\mb E\t]} \mb I_n \big) \Big|\\ 
             \leq & \big(\norm{\mb E} + \norm{\mb G} + 2 \norm{\bb E[\mb E\mb E\t]}^{\frac12}   \big)\epsilon. 
        \end{align}
        
        To apply Theorem~2.5 in \cite{brailovskaya2022universality}, we need to control the quantities $\sigma_*(\check{\mb E})$, $v(\check{\mb E})$, and $R(\check{\mb E})$ defined therein. Notice that $\sigma_*(\check{\mb E}) = \sigma_*(\mb E)$, $v(\check{\mb E}) = v(\mb E)$, $R(\check{\mb E}) = R(\mb E)$. Then we plug the controls \eqref{eq: bound on v(E)}, \eqref{eq: bound on sigma_*(E)}, and \eqref{eq: bound on R(E)} into \cite[Theorem~2.5]{brailovskaya2022universality} to yield that 
        \begin{align}
            & \big|\lambda_{\max}(\check{\mb E})- \lambda_{\max}(\check{\mb G})\big| \\ 
             \lesssim & \big(\tilde\sigma \sqrt{\log d} + m^{\frac16}B^{\frac13}\big(p^{\frac13}\sigma^{\frac23} + n^{\frac13}  \tilde\sigma^{\frac23} \big)(\log d)^{\frac23} + \sqrt{m}B \log d \big)\\ 
             \lesssim & m^{\frac16}B^{\frac13}\big(p^{\frac13}\sigma^{\frac23}+ n^{\frac13}  \tilde\sigma^{\frac23} \big)(\log d)^{\frac23}
        \end{align}
        with probability at least $1- O(d^{-c_p-2})$, 
        which leads to 
        \begin{align}
            &  \bigg|\lambda_{\max}\Big(\mb E\mb E\t - \bb E[\mb E\mb E\t] + \norm{\bb E[\mb E\mb E\t]} \mb I_n \Big) -  \lambda_{\max}\Big(\mb G\mb G\t - \bb E[\mb E\mb E\t] + \norm{\bb E[\mb E\mb E\t]} \mb I_n \Big) \bigg| \\ 
             \lesssim & \big(\tilde \sigma \sqrt{n} + \sigma \sqrt{p} \big)\Big( m^{\frac16}B^{\frac13}\big(p^{\frac13}\sigma^{\frac23}+ n^{\frac13}  \tilde\sigma^{\frac23} \big)(\log d)^{\frac23}\Big) \\ 
             \lesssim &  \tilde \sigma^2 n + \sigma \tilde \sigma\sqrt{np} 
        \end{align}
        with probability at least $1- O(d^{-c_p-2})$. Here we used the facts that 
        \begin{align}
            & m^{\frac16} B^{\frac13} p^{\frac13} \sigma^{\frac23} (\log d)^{\frac23} \lesssim m^{\frac16} B^{\frac13} (mn(\log d)^6)^{\frac13} \sigma^{\frac23} (\log d)^{\frac23} = m^{\frac12}B^{\frac13} n^{\frac13} (\log d)^{\frac83} \sigma^{\frac23} \lesssim \tilde \sigma \sqrt{n}, \\ 
            & m^{\frac16} B^{\frac13} (\log d)^{\frac23} \lesssim n^{\frac16} \tilde \sigma^{\frac13}
        \end{align}
        from the conditions $mn(\log d)^6 \lesssim p$ and $m^3 B^2 (\log d)^{16} \lesssim \sigma^2 n$. 

        \bigskip

        \emph{High-dimensional regime. }
        To proceed, as a complement to the above discussion, we examine the behaviour of $\mb E \mb E\t - \bb E[\mb E\mb E\t]$ under the condition $mn (\log d)^6 \lesssim p$, interpreting it as the sum of independent components $\mb E_{\cdot, S_c} \mb E_{\cdot, S_c}\t - \bb E[\mb E_{\cdot, S_c} \mb E_{\cdot, S_c}\t ], ~c\in[l]$. 
        Specifically, invoking \cite[Theorem~2.8]{brailovskaya2022universality} together with \cite[Theorem~2.7 and Lemma~2.5]{bandeira2023matrix} yields that 
        \begin{align}
            & \norm{\mb E\mb E\t  - \bb E[\mb E \mb E\t]} \leq e \bb E[\mathrm{Tr}(\mb E \mb E\t - \bb E[\mb E \mb E\t])^{2t}]^{\frac{1}{2t}} \lesssim \sigma(\mb E \mb E\t - \bb E[\mb E \mb E\t]) \\ 
            & \qquad \qquad + v(\mb E \mb E\t - \bb E[\mb E \mb E\t])^{\frac12} \sigma(\mb E \mb E\t - \bb E[\mb E \mb E\t])^{\frac12} t^{\frac34} + R_{2t}(\mb E \mb E\t - \bb E[\mb E \mb E\t])t^2 
            \label{eq: application of universality to EEt}
        \end{align}
        with probability at least $1 - O(d^{-c_p-2})$, where we let $t = \lceil C_t \log d\rceil$ with some sufficiently large constant $C_t$. 

         Notice that $\Big(\bb E\big[ \mb E\mb E\t \mb E\mb E\t\big]\Big)_{i,j} = 0$ for $i \neq j \in[n]$ and $\Big(\bb E\big[ \mb E\mb E\t \mb E\mb E\t\big]\Big)_{i,i} = \bb E[\norm{\mb E_{i,\cdot}}_2^2] \Big( \sum_{i' \neq i} \bb E[\norm{\mb E_{i',\cdot}}_2^2] \Big) + \bb E[\norm{\mb E_{i,\cdot}}_2^4] \leq np \sigma^2 \tilde\sigma^2 + 2 p \sigma^2 \tilde \sigma^2 $. Here we used the Bernstein inequality on $\norm{\mb E_{i,\cdot}}_2^2 = \sum_{c \in[l]} \norm{\mb E_{i,S_c}}_2^2$ so as to upper bound its first and second moments invoking the condition that $B \sqrt{m \log d} \ll \sigma \sqrt{p}$, which leads to 
        \eq{
        \bb E[\norm{\mb E_{i,\cdot}}_2^4] \leq 2 p \sigma^2 \tilde\sigma^2 . 
        \label{eq: upper bound for fourth moment of Ei}
        }
        As a consequence, we have
        \begin{equation}
            \sigma(\mb E \mb E\t - \bb E[\mb E \mb E\t])^2 \lesssim np\sigma^2 \tilde\sigma^2. 
            \label{eq: upper bound for sigma term}
        \end{equation}

                Moreover, for the covariance of the vectorization of $\mb E \mb E\t - \bb E[\mb E \mb E\t] $, we begin with examining the covariance between two entries in $\text{vec}(\mb E \mb E\t - \bb E[\mb E\mb E\t])$ to derive that
        \begin{align}
            & \text{cov}\big( \big(\mb E \mb E\t  - \bb E[ \mb E \mb E\t] \big)_{i,j}, \big(\mb E \mb E\t - \bb E[ \mb E \mb E\t] \big)_{i',j'}\big) \\ 
            = & \begin{cases}
                \tr\big(\text{Cov}(\mb E_i)\text{Cov}(\mb E_j)\big), & \text{ if }i\neq j,~ i = i',~j=j'\text{ or }i = j', j = i', \\ 
                \bb E[\norm{\mb E_i}_2^4] - \bb E[\norm{\mb E_i}_2^2]^2, & \text{ if } i=j=i'=j', \\
                0, & \text{ otherwise. }
            \end{cases}
        \end{align} 
        This allows us to derive that 
        \eq{        
        v\big( \mb E \mb E\t - \bb E[\mb E] \bb E[\mb E]\t \big)^2 \leq 4 p\sigma^2 \tilde \sigma^2
        \label{eq: upper bound for variance of EEt}
         } by exploiting the block-diagonal structure and \eqref{eq: upper bound for fourth moment of Ei}. 

         To finish up, one has for $R_{2t}(\mb E\mb E\t - \bb E[\mb E\mb E\t])$ that 
         \begin{align}
             & R_{2t}(\mb E\mb E\t - \bb E[\mb E\mb E\t]) \leq l^{\frac{1}{2t}} \max_{c \in[l]} \bb E[\mathrm{Tr}\big(|\mb E_{\cdot,S_c} \mb E_{\cdot,S_c}\t - \bb E[\mb E_{\cdot,S_c} \mb E_{\cdot,S_c}\t] |^{2t} \big)]^{\frac{1}{2t}} \\ 
             \leq & (l / n)^{\frac{1}{2t}} \max_{c \in[l]} \bb E\big[\mathrm{Tr}\big(|\mb E_{\cdot,S_c} \mb E_{\cdot,S_c}\t - \bb E[\mb E_{\cdot,S_c} \mb E_{\cdot,S_c}\t] |^{2t} \big)\big]^{\frac{1}{2t}} \\ 
             \leq &  (ml / n)^{\frac{1}{2t}} \max_{c \in[l]} \bb E\big[\norm{\mb E_{\cdot,S_c}\t \mb E_{\cdot,S_c}}^{2t} \big]^{\frac{1}{2t}} + l^{\frac{1}{2t}} \tilde\sigma^2. 
             \label{eq: upper bound for R2t}
         \end{align}
         To handle $\bb E\big[\norm{\mb E_{\cdot,S_c}\t \mb E_{\cdot,S_c}}^{2t} \big]^{\frac{1}{2t}}$, we view $\mb E_{\cdot,S_c}\t \mb E_{\cdot,S_c}$ as the sum of independent matrices $\sum_{i\in[n]}\mb E_{i,S_c}\t \mb E_{i,S_c}$ and invoke the tail probability control from the Bernstein inequality \citep[Theorem~6.1.1]{tropp2015introduction} to derive that 
         \begin{align}
             & \bb E\big[\norm{\mb E_{\cdot,S_c}\t \mb E_{\cdot,S_c}}^{2t} \big] = \int_{x \geq 0} 4t x^{4t-1} \bb P\big[\norm{ \mb E_{\cdot,S_c}} \geq x \big] \mathrm d x  \\ 
             \leq & \int_{x \geq 0} 4t x^{4t-1} (m+n) \Big[\exp\Big(-\frac{x^2 / 2}{2 (n \tilde\sigma^2 + m \sigma^2 ) } \Big) + \exp\Big(-\frac{x/2}{2 m B^2 /3} \Big) \Big] \mathrm d x \\ 
             \lesssim  &  (4n\tilde\sigma^2 + 4 m\sigma^2 )^{2t} (m+n)\Gamma(2t+1)  + (4 mB^2 / 3)^{4t} (m+n)\Gamma(4t+1),  
         \end{align}
        whcih gives that 
        \begin{equation}
            \bb E\big[\norm{\mb E_{\cdot,S_c}\t \mb E_{\cdot,S_c}}^{2t} \big]^{\frac{1}{2t}}  \lesssim n  \tilde \sigma^2  \log d+ mB^2(\log d)^2. 
            \label{eq: 4t moment of E_cdot, S_c}
        \end{equation}
        Plugging \eqref{eq: 4t moment of E_cdot, S_c} into \eqref{eq: upper bound for R2t} gives that 
        \begin{equation}
            R_{2t}(\mb E\mb E\t - \bb E[\mb E\mb E\t]) \lesssim n  \tilde \sigma^2  \log d+ mB^2(\log d)^2. 
            \label{eq: final upper bound for R2t}
        \end{equation}

        Plugging \eqref{eq: upper bound for sigma term}, \eqref{eq: upper bound for variance of EEt}, and \eqref{eq: final upper bound for R2t} into \eqref{eq: application of universality to EEt} gives that 
        \begin{align}
            & \norm{\mb E\mb E\t  - \bb E[\mb E \mb E\t]} \lesssim \sqrt{np} \sigma \tilde \sigma + \sqrt{p}n^{\frac14} \sigma \tilde\sigma (\log d)^{\frac34} + n\tilde\sigma^2 (\log d)^3 + mB^2 (\log d)^4 \lesssim \sqrt{np} \sigma \tilde \sigma,
        \end{align}
        where in the last inequality we used the conditions $(\log d)^3 \lesssim n$, $m n (\log d)^6 \lesssim p$, and $m^3 B^2 (\log d)^{12} \lesssim \sigma^2 n$.

        \item For the diagonal part $\diag(\mb E\mb E\t) - \bb E[\mb E \mb E\t]$, one has for each of its diagonal entry that 
        \begin{align}
            & \Big|\norm{\mb E_i}_2^2 - \bb E[\norm{\mb E_i}_2^2 ] \Big| \lesssim \sqrt{mB^2 p \sigma^2 \log d} + mB^2 \log d \lesssim \sqrt{np} \sigma^2 
        \end{align}
        with probability at least $1 - O(d^{-c-3})$ invoking the Bernstein inequality along with the block dependence structure. 

        \item Regarding the term $\mc H(\mb E{\mb M^*}\t )$, one  has 
        \begin{align}
        &\bignorm{\mc H(\mb E{\mb M^*}\t )} \leq \bignorm{\mb E {\mb M^*}\t} + \bignorm{\mathrm{diag}(\mb E{\mb M^*}\t )} \leq  2\bignorm{\mb E {\mb M^*}\t} \leq 2 \sigma_1^* \norm{\mb E \mb V^*}  \lesssim \sigma_1^*\sqrt{n}\bar \sigma
        \end{align}
        by Lemma~\ref{lemma: noise matrix concentrations using the universality (bounded)} with probability at least $1- O(d^{-c_p-2})$. 
    \end{enumerate}

    In the end, taking these bounds collectively gives that 
    \begin{align}
        & \norm{\mb W} \leq \norm{\mb E \mb E\t - \bb E[\mb E\mb E\t]} + \norm{\mathrm{diag}(\mb E\mb E\t) - \bb E[\mb E \mb E\t]} + 2\bignorm{\mc H(\mb E {\mb M^*}\t)} + \bignorm{\mathrm{diag}(\mb M^*{\mb M^*}\t)}  \\ 
        \leq & C_W \Big( \tilde \sigma^2 n + \sigma \tilde \sigma\sqrt{np}   +  \bar \sigma \sqrt{n}  \sigma_1^* + \frac{\mu_1 r }{n} {\sigma_1^*}^2 \Big)
    \end{align}
    with probability at least $1 - O(d^{-c_p-2})$ for some constant $C_W$.
    
    \bigskip
    
    As for the leave-one-out version $\mb W^{(i)}$ with an arbitrary $i\in[n]$, we notice 
    \begin{align}
    \mb W^{(i)} =&  \mc H(\mc P_{-i,\cdot}(\mb E) \mc P_{-i,\cdot}(\mb E)\t + \mc P_{-i,\cdot}(\mb E) \mc P_{-i,\cdot}(\mb M^*)\t +  \mc P_{-i,\cdot}(\mb M^*) \mc P_{-i,\cdot}(\mb E)\t  )  \\ 
    & +  \mc P_{-i,\cdot}(\mb E) \mc P_{i,\cdot}(\mb M^*)\t +  \mc P_{i,\cdot}(\mb M^*) \mc P_{-i,\cdot}(\mb E)\t - \diag(\mb M^*{\mb M^*}\t), 
    \end{align}
    where the first term $\mc H(\mc P_{-i,\cdot}(\mb E) \mc P_{-i,\cdot}(\mb E)\t + \mc P_{-i,\cdot}(\mb E) \mc P_{-i,\cdot}(\mb M^*)\t +  \mc P_{-i,\cdot}(\mb M^*) \mc P_{-i,\cdot}(\mb E)\t  )$
    can be bounded similarly by viewing it as the preceding noise $\mb W$ with $n-1$ samples, and the second and third terms are both bounded by $\norm{\mb E{\mb M^*}\t}$. Consequently, under Gaussian or bounded noise, one has 
    $$
         \bignorm{\mb W^{(i)}} \leq  C_W \Big( \tilde \sigma^2 n + \sigma \tilde \sigma\sqrt{np}   +  \bar \sigma \sqrt{n}  \sigma_1^* + \frac{\mu_1 r }{n} {\sigma_1^*}^2 \Big)
    $$
    with probability at least $1 - O(d^{-c_p-2})$ for some sufficiently large constant $C_W$.

    Moreover, the term $\mc H(\mc P_{\cdot,-S_{c}}(\mb E) \mc P_{\cdot,-S_{c}}(\mb E)\t + \mc P_{\cdot,-S_{c}}(\mb E){\mb M^*}\t + \mb M^* \mc P_{\cdot,-S_{c}}(\mb E)\t  )$ in $\mb W^{(-c)}$ appearing in $\mb W^{(-c)}$ can be viewed as a noise matrix $\mb W$ with an effective dimension $p' = p-|S_c|$. Consequently, with probability at least $1 - O(d^{-c_p-2})$, we have
    \begin{align}
        & \norm{\mb W^{(-c)}} \leq  C_W \Big( \tilde \sigma^2 n + \sigma \tilde \sigma\sqrt{np}   +  \bar \sigma \sqrt{n}  \sigma_1^* + \frac{\mu_1 r }{n} {\sigma_1^*}^2 \Big)
    \end{align}
    with probability at least $1 - O(d^{-c_p-2})$ for some sufficiently large constant $C_W$. Likewise, for the quantity $\max_{c\in[l]}\norm{\mb W^{(-c)}\mb U^*}_{2,\infty}$, Lemma~\ref{lemma: noise matrix concentrations using the universality (bounded)} implies that, with probability exceeding $1 - O(d^{-c_p-2})$,
    \begin{align}
        & \max_{c \in [l]} \bignorm{\mb W^{(-c)} \mb U^*}\ti \lesssim \sqrt{\frac{\mu_1 r}{n}} \deltaop \sqrt{\log d}. 
    \end{align}

\subsubsection{Proof of Lemma~\ref{lemma: upper bound for SkU}}
\label{subsubsec: proof of lemma: upper bound for SkU}
First, one has from the definitions in Lemma~\ref{lemma: expansion from xia2021normal} that
\begin{align}
    & \mc S_k \mb U^* = \sum_{\mb s: s_1 + \cdots + s_{k+1} = k } (-1)^{1+ \tau(\mb s)} \mc P^{-s_1} \mb W \mc P^{-s_2} \mb W \cdots \mc P^{-s_p} \mb W \mc P^{-s_{p+1}} \mb U^*. 
\end{align}
This motivates us to focus on the terms that begin with $(\mc P^0 \mb W \mc P^0)^t \mc P^0 \mb W \mb U^*$. For convenience, we introduce the following event:
\begin{align}
    & \mc E_{W} \coloneqq \Big\{ \max\big\{ \norm{\mb W},~\max_{i\in[n]} \bignorm{\mb W^{(i)}},~ \norm{\mc H(\mb E \mb E\t)}, ~ \bignorm{\mc H(\mb E{\mb M^*}\t)} \}   \big\}\leq \deltaop \Big\}, \\ 
    & \mc E_{H} \coloneqq \mc E_{W} \cap \Big\{ \norm{\mb E\mb V^*} \leq C_H \bar \sigma \sqrt{n},~ \max_{i\in[n]}\big| \mb E_{i,\cdot}{\mb M^*_{i,\cdot}}\t \big| \leq C_H \bar \sigma \sqrt{r \log d } \sigma_1^* \sqrt{\frac{\mu_1 r}{n}} \Big\},
    \label{eq: definition of E_H}
\end{align}
where the exceptional probability is controlled by $\bb P\big[\mc E_H^{\complement} \big] \leq 1 - O(d^{- c_p - 1})$ given a sufficiently large constant $C_H$ due to Lemmas~\ref{lemma: concentration for W and its variants},~\ref{lemma: noise matrix concentrations using the universality (Gaussian)},~and~\ref{lemma: noise matrix concentrations using the universality (bounded)}.

We prove the conclusion by induction, whose statements are as follows: under the event $\mc E_H$, we have for $t \geq 0$ that
\begin{align}
    &\textsc{Hypothesis $H_{t,1}$:\qquad } \norm{\big((\mc P^0 \mb W^{(i)} \mc P^0)^t \mc P^0 \mb W^{(i)} \mb U^*\big)_{i,\cdot}}_2 \leq C_0  \sqrt{\frac{\mu_1 r}{n}}  \Big[C_0\deltaop\Big]^{t+1};\label{eq: hypothesis 1 for subspace perturbation}
    \\ 
    &
    \textsc{Hypothesis $H_{t,2}$:\qquad } \\ 
    & 
    \quad \norm{\big((\mc P^0 \mb W \mc P^0)^t \mc P^0 \mb W \mb U^*\big)_{i,\cdot}}_2 \leq C_0  \Big[C_0 \deltaop\Big]^{t-1}  \Big[ \sigma_1^*\norm{\mb E_{i,\cdot} \mb V^*}_2 + \norm{\mb E_{i,\cdot} \mc P_{-i,\cdot}(\mb E)\t \mb U^*}_2 \Big] \\ 
    &\quad  +  \sum_{h=0}^{t-2} C_0  \Big[C_0 \deltaop\Big]^{t-h-2}  \norm{\mb E_{i, \cdot}\mc P_{-i,\cdot}(\mb E)\t \mc P^0 \big(\mc P^0 \mb W^{(i)} \mc P^0 \big)^h \mb W^{(i)} \mb U^*}_2  + C_0  \norm{\mb U_{i,\cdot}^*}_2  \Big[C_0\deltaop\Big]^{t} \eqqcolon \xi_t
    \label{eq: hypothesis 2 for subspace perturbation}
\end{align}
for some constant $C_0$. Hereinafter, the summation $\sum_{i = a}^b a_i$ with $b < a$ simply equals zero. The first inductive hypothesis will serve as an auxiliary tool in the proof of the latter one. 

\paragraph{Base case. }
For the base case $t= 0$, it naturally holds that 
\begin{align}
    & \norm{(\mc P^0 \mb W^{(i)} \mb U^*)_{i,\cdot}} \leq \norm{(\mb W^{(i)} \mb U^*)_{i,\cdot}}_2 + \norm{\mb U^*}\ti \bignorm{\mb W^{(i)}} \\ 
    \leq&  \norm{\mb M^*_{i,\cdot} \mc P_{-i,\cdot}(\mb E)\t \mb U^*}_2 + \big|(\mb M^*{\mb M^*}\t)_{i,i}\big| \norm{\mb U^*}\ti + \norm{\mb U^*}\ti \bignorm{\mb W^{(i)}}  \\ 
    \leq & \sqrt{\frac{\mu_1 r}{n}}\sigma_1^* \norm{\mc P_{-i,\cdot}(\mb E) \mb V^*} + \sigma_1^* \Big(\frac{\mu_1 r}{n}\Big)^{\frac32} + \sqrt{\frac{\mu_1 r}{n}} \deltaop \leq  C_0 \sqrt{\frac{\mu_1 r}{n}} \deltaop, \\ 
    & \norm{(\mc P^0 \mb W \mb U^*)_{i,\cdot}} \leq \norm{(\mb W \mb U^*)_{i,\cdot}}_2 + \norm{\mb U^*}\ti \bignorm{\mb W^{(i)}}\\ 
    \leq & \norm{\mb E_{i,\cdot} \mc P_{-i,\cdot}(\mb E)\t \mb U^*}_2 + \norm{\mb E_{i,\cdot} \mb V^*}_2 \sigma_1^*+ \norm{\mb M^*_{i,\cdot} \mc P_{-i,\cdot}(\mb E)\t \mb U^*}_2 \\ 
    & \qquad + \big|(\mb M^*{\mb M^*}\t)_{i,i}\big| \norm{\mb U^*}\ti + \norm{\mb U^*}\ti \bignorm{\mb W^{(i)}} \\ 
    \leq & \norm{\mb E_{i,\cdot} \mc P_{-i,\cdot}(\mb E)\t \mb U^*}_2 + \norm{\mb E_{i,\cdot} \mb V^*}_2 \sigma_1^* + C_0 \sqrt{\frac{\mu_1 r}{n}} \deltaop
\end{align}
under the event $\mc E_H$ for some sufficiently large constant $C_0$, where we invoke the relation 
\begin{equation}
    \norm{\mb M^*_{i,\cdot} \mc P_{-i,\cdot}(\mb E)\t \mb U^*}_2 \leq \norm{\mb M^*_{i,\cdot} \mc P_{-i,\cdot}(\mb E)\t }_2 \leq \sigma_1^*\norm{\mc P_{-i,\cdot}(\mb E) \mb V^*}\norm{\mb U^*_{i,\cdot}}_2 \leq \sigma_1^* \norm{\mb E \mb V^*}\norm{\mb U^*}\ti \leq \deltaop \sqrt{\frac{\mu_1 r}{n}}. 
    \label{eq: M_i,cdot Et relation}
\end{equation} 

\paragraph{Proof of the first inductive hypothesis. }
Then we move on to the cases with $t \geq 1$. By definition, it follows from the inductive hypotheses for $t' \leq t$ that 
\begin{align}
    & \norm{\big((\mc P^0 \mb W^{(i)} \mc P^0)^t \mc P^0 \mb W^{(i)} \mb U^*\big)_{i,\cdot}}_2 
    \\ 
    \leq & \norm{\mb M^*_{i,\cdot} \mc P_{-i,\cdot}(\mb E)\t }_2 \bignorm{\mb W^{(i)}}^{t} + \big|(\mb M^*{\mb M^*}\t)_{i,i}\big| \norm{\big((\mc P^0 \mb W^{(i)} \mc P^0)^{t-1} \mb W^{(i)} \mb U^*\big)_{i,\cdot}}_2  + \norm{\mb U^*}\ti \bignorm{\mb W^{(-i)}}^{t+1}  \\ 
    \leq & C_0  \sqrt{\frac{\mu_1 r}{n}} \big[C_0 \deltaop \big]^{t+1}
\end{align}
under $\mc E_H$, where we invoke \eqref{eq: M_i,cdot Et relation} again. 

\paragraph{Proof of the second inductive hypothesis. }
Similar to the approach of \cite{agterberg2022entrywise}, for each $t \geq 1$ one obtains the telescoping decomposition: 
\begin{align}
    & \norm{\big((\mc P^0 \mb W \mc P^0)^t \mb W \mb U^*\big)_{i,\cdot}}_2 \leq 
    \underbrace{\norm{(\mc P^0 \mb W \mc P^0)_{i,\cdot} \big(\mc P^0 
    \mb W^{(i)} \mc P^0 \big)^{t-1} \mc P^0 \mb W^{(i)} \mb U^*}_2}_{\gamma_1}
    \\
    & \qquad \qquad  + \sum_{l=0}^{t-2}\underbrace{  \norm{(\mc P^0 \mb W \mc P^0)_{i,\cdot}  \big(\mc P^0 \mb W^{(i)} \mc P^0\big)^l \mc P^0 \big(\mb W^{(i)} - \mb W \big)  \mc P^0 \big(\mc P^0 \mb W \mc P^0\big)^{t-l - 2} \mc P^0 \mb W \mb U^*}_2}_{\gamma_{2,l}} \\ 
    &\qquad \qquad  + \underbrace{\norm{ (\mc P^0 \mb W \mc P^0)_{i,\cdot}  \big(\mc P^0 \mb W^{(i)} \mc P^0\big)^{t-1}\mc P^0 \big(\mb W - \mb W^{(i)} \big)\mb U^*}_2}_{\gamma_3}. 
    \label{eq: telescoping decomposition}
\end{align}

Before continuing, we first notice by definition that
\begin{align}
    & \mb W - \mb W^{(i)} = \mc P_{i,\cdot}(\mb E){\mc P_{-i,\cdot}(\mb E)}\t +  \mc P_{-i,\cdot}(\mb E){\mc P_{i,\cdot}(\mb E)}\t + \mc P_{i,\cdot}(\mb E) {\mc P_{-i,\cdot}(\mb M^*)}\t + \mc P_{-i,\cdot}(\mb M^*){\mc P_{i,\cdot}(\mb E)}\t. 
    \label{eq: decomposition of W - W-i}
\end{align}
Besides, a useful observation from the relation $\mc P^0 \mb M^* = \mb 0$ is that 
\begin{align}
    & \mc P^0 \mb W \mc P^0 = \mc P^0 \mc H(\mb E \mb E\t) \mc P^0 - 2\mc P^0 \diag\big(\mb E{\mb M^*}\t \big) \mc P^0 - \mc P^0 \diag\big(\mb M^*{\mb M^*}\t \big) \mc P^0, 
    \label{eq: identity of P0 W P0}
    \\ 
    &\mc P^0 \big( \mb W - \mb W^{(i)}\big) \mc P^0 =  \mc P^0 \big[ \mc P_{i,\cdot}(\mb E){\mc P_{-i,\cdot}(\mb E)}\t +  \mc P_{-i,\cdot}(\mb E){\mc P_{i,\cdot}(\mb E)}\t \big] \mc P^0  - 2 \mc P^0  \mc P_{i,\cdot}(\mb E) {\mc P_{i,\cdot}(\mb M^*)}\t \mc P^0 .
    \label{eq: identity of P0 (W - Wi) P0}
\end{align}

\bigskip

\emph{Bounding $\gamma_1$. }
By the triangle inequality and \eqref{eq: identity of P0 W P0}, the first term $\gamma_1$ in \eqref{eq: telescoping decomposition} is rewritten as 
\begin{align}
    & \norm{(\mc P^0 \mb W \mc P^0)_{i,\cdot} \big(\mc P^0 
    \mb W^{(i)} \mc P^0 \big)^{t-1} \mc P^0 \mb W^{(i)} \mb U^*}_2 \\ 
    =  &\norm{\big(\mc P^0 \big(\mc H(\mb E \mb E\t)  -\diag( \mb E{\mb M^*}\t + \mb M^*\mb E\t - \mb M^*{\mb M^*}\t ) \big) \mc P^0\big)_{i,\cdot} \big(\mc P^0 
    \mb W^{(i)} \mc P^0 \big)^{t-1} \mc P^0 \mb W^{(i)} \mb U^*}_2  \\ 
    \leq & 
      \norm{\mb E_{i,\cdot} \mc P_{-i,\cdot}(\mb E)\t \mc P^0 \big(\mc P^0 
    \mb W^{(i)} \mc P^0 \big)^{t-1} \mc P^0 \mb W^{(i)} \mb U^*}_2 +  \norm{\mb U^*}\ti  \norm{\mc H(\mb E \mb E \t)} \norm{\big(\mc P^0 
    \mb W^{(i)} \mc P^0 \big)^{t-1} \mc P^0 \mb W^{(i)} \mb U^*}_2
    \\ 
    & \qquad \qquad +\underbrace{ \norm{ \Big[\mc P^0 \diag( \mb E{\mb M^*}\t + \mb M^*\mb E\t - \mb M^*{\mb M^*}\t ) \mc P^0  \big(\mc P^0 
    \mb W^{(i)} \mc P^0 \big)^{t-1} \mc P^0 \mb W^{(i)} \mb U^*\Big]_{i,\cdot}}_2}_{\gamma_{1}'} 
    \label{eq: decomposition of gamma 1}
 \end{align}

To upper bound $\gamma_1'$ in \eqref{eq: decomposition of gamma 1}, we reason as follows: we first note that 
\begin{align}
    & \big|\big(\mb E{\mb M^*}\t + \mb M^*\mb E\t - \mb M^*{\mb M^*}\t \big)_{i',i'} \big| \leq  2 \big|\mb E_{i',\cdot}{\mb M^*_{i',\cdot}}\t \big| + \norm{\mb M^*_{i',\cdot}}_2^2 \\ 
    \leq &  2C_H \bar\sigma \sqrt{r \log d} \sigma_1^* \sqrt{\frac{\mu_1 r}{n}} + (\sigma_1^*)^2 \frac{\mu_1 r}{n}
    \label{eq: diagonal residual control}
\end{align}
for every $i' \in[n]$ under the event $\mc E_H$. 
Invoking the inductive hypotheses $H_{t',1}$ with $0\leq t' \leq t-1$, it then turns out to be 
\begin{align}
    &  \norm{ \mc P^0 \diag( \mb E{\mb M^*}\t + \mb M^*\mb E\t - \mb M^*{\mb M^*}\t ) \mc P^0  \big(\mc P^0 
    \mb W^{(i)} \mc P^0 \big)^{t-1} \mc P^0 \mb W^{(i)} \mb U^*} \\ 
    \leq &  \big|\big(\mb E{\mb M^*}\t + \mb M^*\mb E\t - \mb M^*{\mb M^*}\t \big)_{i,i} \big| \bignorm{ \big(\big(\mc P^0 
    \mb W^{(i)} \mc P^0 \big)^{t-1} \mc P^0 \mb W^{(i)} \mb U^*\big)_{i,\cdot} }_2 \\ 
    & + \big|\big(\mb E{\mb M^*}\t + \mb M^*\mb E\t - \mb M^*{\mb M^*}\t \big)_{i,i} \big| \norm{\mb U^*}\ti \bignorm{    \big(\mc P^0 
    \mb W^{(i)} \mc P^0 \big)^{t-1} \mc P^0 \mb W^{(i)} \mb U^*}  \\ 
    & + \norm{\mb U^*}\ti \bignorm{\diag( \mb E{\mb M^*}\t + \mb M^*\mb E\t - \mb M^*{\mb M^*}\t ) \mc P^0  \big(\mc P^0 
    \mb W^{(i)} \mc P^0 \big)^{t-1} \mc P^0 \mb W^{(i)} \mb U^*} \\ 
    \stackrel{\text{by \eqref{eq: diagonal residual control}}}{\leq}  & C_1\Big( \bar\sigma \sqrt{r \log d} \sigma_1^* \sqrt{\frac{\mu_1 r}{n}} + (\sigma_1^*)^2 \frac{\mu_1 r}{n}\Big) \Big(C_0\sqrt{\frac{\mu_1 r}{n}} \big[C_0\deltaop\big]^t + \sqrt{\frac{\mu_1 r}{n}} \deltaop^t \Big)\\ 
    & \qquad + C_1 \sqrt{\frac{\mu_1 r}{n}} \Big( \bar\sigma \sqrt{r \log d} \sigma_1^* \sqrt{\frac{\mu_1 r}{n}} + (\sigma_1^*)^2 \frac{\mu_1 r}{n}\Big) \deltaop^t 
    \\ 
    \leq &  C_2 \sqrt{\frac{\mu_1 r}{n}} \deltaop^{t+1} 
    \label{eq: bound for gamma1' 1}
\end{align}
under $\mc E_H$ for some constants $C_1$ and $C_2$, where we used the condition $r \log d \mu_1 r^2 \lesssim n$
to derive $ \bar\sigma \sqrt{r \log d} \sigma_1^* \sqrt{\frac{\mu_1 r}{n}} + (\sigma_1^*)^2 \frac{\mu_1 r}{n} \lesssim \deltaop$.

Plugging \eqref{eq: bound for gamma1' 1} into \eqref{eq: decomposition of gamma 1} gives that 
\begin{equation}
     \norm{(\mc P^0 \mb W \mc P^0)_{i,\cdot} \big(\mc P^0 
    \mb W^{(i)} \mc P^0 \big)^{t-1} \mc P^0 \mb W^{(i)} \mb U^*}_2 \leq \norm{\mb E_{i,\cdot} \mc P_{-i,\cdot}(\mb E)\t \mc P^0 \big(\mc P^0 
    \mb W^{(i)} \mc P^0 \big)^{t-1} \mc P^0 \mb W^{(i)} \mb U^*}_2  + C_3 \sqrt{\frac{\mu_1 r}{n}} [C_0\deltaop]^{t+1} 
    \label{eq: bound for gamma1}
\end{equation}
holds under $\mc E_H$ for some constant $C_3$ . 

\bigskip

\emph{Bounding $\gamma_{2,l}$. }
It follows from the triangle inequality and \eqref{eq: identity of P0 (W - Wi) P0} that
\begin{align}
    & \gamma_{2,l}  =  \norm{(\mc P^0 \mb W \mc P^0)_{i,\cdot}  \big(\mc P^0 \mb W^{(i)} \mc P^0\big)^l \mc P^0 \big(\mb W^{(i)} - \mb W \big)  \mc P^0 \big(\mc P^0 \mb W \mc P^0\big)^{t-l - 2}\mb W \mb U^*}_2 \\ 
    \leq & \norm{\mc P^0 \mb W \mc P^0 \big(\mc P^0 \mb W^{(i)} \mc P^0\big)^l \mc P^0} \norm{\mb E_{i,\cdot} \mc P_{-i,\cdot}(\mb E)\t \mc P^0 \big(\mc P^0 \mb W \mc P^0\big)^{t-l - 2} \mc P^0 \mb W \mb U^*}_2 \\ 
    & + \norm{\mc P^0 \mb W \mc P^0  \big(\mc P^0 \mb W^{(i)} \mc P^0\big)^l \mc P^0} \norm{\mc H(\mb E \mb E\t) } \norm{\big( \big(\mc P^0 \mb W \mc P^0\big)^{t-l - 2} \mc P^0 \mb W \mb U^*\big)_{i,\cdot}}_2 \\
    & + 2 \big|\mb E_{i,\cdot} {\mb M^*_{i,\cdot}}\t\big| \norm{\mc P^0 \mb W \mc P^0  \big(\mc P^0 \mb W^{(i)} \mc P^0\big)^l \mc P^0} \norm{( \big(\mc P^0 \mb W \mc P^0\big)^{t-l - 2} \mc P^0 \mb W \mb U^*\big)_{i,\cdot}}_2. 
    \label{eq: upper bound for gamma2t}
\end{align}
for every $t \in [k - 1]$. 

Notice that $\mb W_{i,\cdot} = \mb E_{i,\cdot} \mc P_{-i,\cdot}(\mb E)\t +  \mb E_{i,\cdot} \mc P_{-i,\cdot}(\mb M^*)\t + \mb M_{i,\cdot}^* \mc P_{-i,\cdot}(\mb E)\t - \norm{\mb M^*_{i,\cdot}}_2^2 \mb e_i\t $, where $\mb e_i$ denotes the $i$-th canonical basis in $\bb R^n$. This decomposition allows us to handle the term $\norm{\mb E_{i,\cdot} \mc P_{-i,\cdot}(\mb E)\t \mc P^0 \big(\mc P^0 \mb W \mc P^0\big)^{t-l - 2} \mc P^0 \mb W \mb U^*}_2 $ using the inductive hypotheses $H_{t',2}$ for $0\leq t' \leq t-1$ with $\xi_{t'}$ defined in \eqref{eq: hypothesis 2 for subspace perturbation}: it holds for some constant $C_4$ that
\begin{align}
    & \norm{\mb E_{i,\cdot} \mc P_{-i,\cdot}(\mb E)\t \mc P^0 \big(\mc P^0 \mb W \mc P^0\big)^{t-l - 2} \mc P^0 \mb W \mb U^*}_2  \\ 
    \leq & \norm{\mb W_{i,\cdot} \mc P^0\big(\big(\mc P^0 \mb W \mc P^0\big)^{t-l - 2} \mc P^0 \mb W \mb U^* \big)_{i,\cdot}}_2 + \Big[\norm{\mb E_{i,\cdot} \mc P_{-i,\cdot}(\mb M^*)\t}_2  \\
    & \qquad + \norm{\mb M_{i,\cdot}^* \mc P_{-i,\cdot}(\mb E)\t}_2 \Big] \norm{\big(\mc P^0 \mb W \mc P^0\big)^{t-l - 2} \mc P^0 \mb W \mb U^*} \\ 
    &  + \norm{\mb M^*_{i,\cdot}}_2^2 \norm{\big(\big(\mc P^0 \mb W \mc P^0\big)^{t-l - 2} \mc P^0 \mb W \mb U^* \big)_{i,\cdot}}_2 \\ 
    \leq & \norm{\big(\big(\mc P^0 \mb W \mc P^0\big)^{t-l - 1} \mc P^0 \mb W \mb U^* \big)_{i,\cdot}}_2 + \norm{\mb U^*}\ti \norm{\mb W}^{t-l} \\ 
    & + \Big[ \norm{\mb E_{i,\cdot}\mb V^*}_2 \sigma_1^*   + \norm{\mb U^*}\ti \sigma_1^* \norm{\mc P_{-i,\cdot}(\mb E) \mb V^*} \Big] \norm{\big(\mc P^0 \mb W \mc P^0\big)^{t-l - 2} \mc P^0 \mb W \mb U^*} \\ 
    & + (\sigma_1^*)^2 \frac{\mu_1 r}{n} \norm{\big(\big(\mc P^0 \mb W \mc P^0\big)^{t-l - 2} \mc P^0 \mb W \mb U^* \big)_{i,\cdot}}_2 \\ 
    \leq & \xi_{t-l-1} + \sqrt{\frac{\mu_1 r}{n}} \deltaop^{t-l} + C_4\Big[\norm{\mb E_{i,\cdot} \mb V^*}_2 \sigma_1^*  + \sqrt{\frac{\mu_1 r}{n}} \sigma_1^* \bar \sigma \sqrt{n} \Big] \deltaop^{t-l-1} + \deltaop \xi_{t-l-2} \leq 2\xi_{t-l-1} 
    \label{eq: upper bound for gamma2t 1}
\end{align}
for some constant $C_4$. Here we note that, provided a sufficiently large $C_0$, the last inequality arises from the definition of $\xi_t$. 

Returning to the upper bound for $\gamma_{2,t}$, substituting the conditions in $\mc E_H$ and \eqref{eq: upper bound for gamma2t 1} into \eqref{eq: upper bound for gamma2t} gives the following. 
\begin{align}
    \gamma_{2,l} \leq 2 \deltaop^{l+1} \xi_{t- l -1} + \deltaop^{l+2} \xi_{t - l -2} + 2 C_1 \bar\sigma \sqrt{r \log d} \sigma_1^* \sqrt{\frac{\mu_1 r}{n}} \deltaop^{l+1} \xi_{t-l-2} \leq 4 \deltaop^{l+1} \xi_{t- l -1}
\end{align}
holds under the event $\mc E_H$.  

\bigskip

\emph{Bounding $\gamma_3$. }
Next, with respect to $\gamma_3$, employing \eqref{eq: decomposition of W - W-i} together with the conditions in $\mc E_H$ gives that 
\begin{align}
    & \gamma_3 \leq \norm{ \mc P^0 \mb W \mc P^0 \big(\mc P^0 \mb W^{(i)} \mc P^0\big)^{t-1}} \\ 
    & \cdot \Big[\norm{ \mc P_{i,\cdot}(\mb E) \mc P_{-i,\cdot}(\mb E)\t \mb U^*} + 
    \norm{\mc H(\mb E\mb E\t)} \norm{\mb U^*}\ti  + \norm{\mb E_{i,\cdot} {\mb M^*}\t \mb U^*} + 2|\mb E_{i,\cdot}{ \mb M^*_{i,\cdot} }\t|  \Big] \\ 
    \leq & \deltaop^t \cdot \Big[ \norm{ \mb E_{i,\cdot} \mc P_{-i,\cdot}(\mb E)\t \mb U^*}_2  + \deltaop \sqrt{\frac{\mu_1 r}{n}} + \norm{\mb E_{i,\cdot} \mb V^*}_2 \sigma_1^* + 2 C_H \bar \sigma \sqrt{r \log d } \sigma_1^* \sqrt{\frac{\mu_1 r}{n}}\Big]\\ 
    \leq & \deltaop^t \cdot \Big[ \norm{ \mb E_{i,\cdot} \mc P_{-i,\cdot}(\mb E)\t \mb U^*}_2  + \norm{\mb E_{i,\cdot} \mb V^*}_2 \sigma_1^* + C_4\deltaop \sqrt{\frac{\mu_1 r}{n}}\Big]
    \label{eq: upper bound for gamma3}
\end{align}
holds under the event $\mc E_H$ for some constant $C_4$, where we used the condition $r \log d \lesssim n$ and the facts that $\mc P^0 \mb M^* \mc P_{i,\cdot}(\mb E)\t = \mb 0$ and $\mb M^* = \mc P_{i,\cdot}(\mb M^*) + \mc P_{-i,\cdot}(\mb M^*)$. 

\bigskip

\emph{Taking the bounds together. }
Plugging these bounds collectively into \eqref{eq: telescoping decomposition} reveals that 
\begin{align}
    & \norm{\big((\mc P^0 \mb W \mc P^0)^t \mb W \mb U^*\big)_{i,\cdot}}_2 \leq \norm{\mb E_{i,\cdot} \mc P_{-i,\cdot}(\mb E)\t \mc P^0 \big(\mc P^0 
    \mb W^{(i)} \mc P^0 \big)^{t-1} \mc P^0 \mb W^{(i)} \mb U^*}_2  + C_3 \sqrt{\frac{\mu_1 r}{n}} [C_0\deltaop]^{t+1}  \\ 
    & \qquad +  \sum_{j\in[t]} 4 \deltaop^{l+1} \xi_{t- l -1} + \deltaop^t \cdot \Big[ \norm{ \mb E_{i,\cdot} \mc P_{-i,\cdot}(\mb E)\t \mb U^*}_2  + \norm{ \mb E_{i,\cdot} \mb V^*}_2  \sigma_1^*\Big] + C_4 \sqrt{\frac{\mu_1 r}{n}} \deltaop^{t+1} \leq \xi_{t}
    \label{eq: upper bound on Ei E-it (W)t WU}
\end{align}
given a sufficiently large $C_0$. 

\bigskip

Notice that there are at most $4^k$ terms in each $\mc S_k$. Hence, \eqref{eq: hypothesis 2 for subspace perturbation} leads to the following control on $\norm{(\mc S_k)_{i,\cdot} \mb U^*}_2$: 
    \begin{align}
        & \bignorm{(\mc S_k)_{i,\cdot} \mb U^*}_2 \leq 4^k \max\Big\{ \max_{0 \leq t \leq k-1} \xi_{t} \deltaop^{k-t-1} , \norm{\mb U^*}\ti \deltaop^{k} \Big\} \\ 
        \leq & C_0 \cdot \bigg\{ \Big[4C_0 \deltaop\Big]^{k-1} \cdot \Big[ \sigma_1^*\norm{\mb E_{i,\cdot} \mb V^*}_2 + \norm{\mb E_{i,\cdot} \mc P_{-i,\cdot}(\mb E)\t \mb U^*}_2 \Big] \cdot (\sigma_r^*)^{-2k} \\ 
    + &  \sum_{h=0}^{t-2} \cdot \Big[4C_0  \deltaop\Big]^{t-h-2} \cdot \norm{\mb E_{i, \cdot}\mc P_{-i,\cdot}(\mb E)\t \mc P^0 \big(\mc P^0 \mb W^{(i)} \mc P^0 \big)^h \mb W^{(i)} \mb U^*}_2 (\sigma_r^*)^{-2k}  + \cdot \norm{\mb U_{i,\cdot}^*}_2 \cdot \Big[4C_0\deltaop / (\sigma_r^*)^2\Big]^{k} \bigg\}, 
    \end{align}
    where we make use of the fact that $\norm{\mc P^{-s_1} \mb W \mc P^{-s_2} \mb W \cdots \mc P^{-s_k} \mb W \mc P^{-s_{k+1}} \mb U^*} \leq \norm{
    \mb U^*}\ti \Big[\deltaop / (\sigma_r^*)^2\Big]^k$ for every $\mb s$ with $s_1 \geq 1$ under $\mc E_H$.

\qed 

\subsection{Proofs of Lemmas~\ref{lemma: bound for quantities in expansion (Gaussian case)} and~\ref{lemma: bound for quantities in expansion (bounded case)}}

\label{subsec: proof of bounds for quantities in expansion}

\begin{proof}[Proof of Lemma~\ref{lemma: bound for quantities in expansion (Gaussian case)}]
    To deal with the Gaussian case, it suffices to condition on $\mc P_{-i,\cdot}(\mb E)$ and apply the Hanson-Wright inequality to derive that 
    \begin{align}
        & \norm{\mb E_{i,\cdot} \mc P_{-i,\cdot}(\mb E)\t \mc P^0 (\mc P^0 \mb W^{(i)} \mc P^0 )^k \mc P^0 \mb W^{(i)} \mb U^*}_2^2 - \norm{\bo \Sigma^{\frac12} \mc P_{-i,\cdot}(\mb E)\t \mc P^0 (\mc P^0 \mb W^{(i)} \mc P^0 )^k \mc P^0 \mb W^{(i)} \mb U^*}_F^2 \\ 
        \lesssim & \norm{{\mb U^*}\t  \mb W^{(i)} \mc P^0 (\mc P^0 \mb W^{(i)} \mc P^0 )^k \mc P^0 \mc P_{-i,\cdot}(\mb E) \bo \Sigma \mc P_{-i,\cdot}(\mb E)\t \mc P^0 (\mc P^0 \mb W^{(i)} \mc P^0 )^k \mc P^0 \mb W^{(i)} \mb U^*}_F (\sqrt{\log d} \wedge t_0) \\ 
        &  + \norm{{\mb U^*}\t  \mb W^{(i)} \mc P^0 (\mc P^0 \mb W^{(i)} \mc P^0 )^k \mc P^0 \mc P_{-i,\cdot}(\mb E) \bo \Sigma \mc P_{-i,\cdot}(\mb E)\t \mc P^0 (\mc P^0 \mb W^{(i)} \mc P^0 )^k \mc P^0 \mb W^{(i)} \mb U^*}(\log d \wedge t_0^2)
    \end{align}
    with probability at least $1-O(d^{-c_p-1}\vee e^{-t_0^2 / 2})$ conditional on $\mc P_{-i,\cdot}(\mb E)$. Notice that 
    \begin{align}
        & \norm{\bo \Sigma^{\frac12} \mc P_{-i,\cdot}(\mb E)\t \mc P^0 (\mc P^0 \mb W^{(i)} \mc P^0 )^k \mc P^0 \mb W^{(i)} \mb U^*}_F^2 \lesssim r  \tilde\sigma^2\deltaop^{2k+2}\big(\sigma^2 p + \tilde\sigma^2 n \big) \\ 
        & \norm{{\mb U^*}\t  \mb W^{(i)} \mc P^0 (\mc P^0 \mb W^{(i)} \mc P^0 )^k \mc P^0 \mc P_{-i,\cdot}(\mb E) \bo \Sigma \mc P_{-i,\cdot}(\mb E)\t \mc P^0 (\mc P^0 \mb W^{(i)} \mc P^0 )^k \mc P^0 \mb W^{(i)} \mb U^*}_F \\ 
        \lesssim & \sqrt{r}\norm{{\mb U^*}\t  \mb W^{(i)} \mc P^0 (\mc P^0 \mb W^{(i)} \mc P^0 )^k \mc P^0 \mc P_{-i,\cdot}(\mb E) \bo \Sigma \mc P_{-i,\cdot}(\mb E)\t \mc P^0 (\mc P^0 \mb W^{(i)} \mc P^0 )^k \mc P^0 \mb W^{(i)} \mb U^*},  \\ 
        & \norm{{\mb U^*}\t  \mb W^{(i)} \mc P^0 (\mc P^0 \mb W^{(i)} \mc P^0 )^k \mc P^0 \mc P_{-i,\cdot}(\mb E) \bo \Sigma \mc P_{-i,\cdot}(\mb E)\t \mc P^0 (\mc P^0 \mb W^{(i)} \mc P^0 )^k \mc P^0 \mb W^{(i)} \mb U^*} \\ 
        \lesssim & \tilde\sigma^2\deltaop^{2k+2}\big(\sigma^2 p + \tilde\sigma^2 n \big) 
    \end{align}
    with probability at least $1- O(d^{-c_p-1})$ by Lemma~\ref{lemma: concentration for W and its variants} and Lemma~\ref{lemma: noise matrix concentrations using the universality (Gaussian)}. This leads to the conclusion that 
    \begin{equation}
         \norm{\mb E_{i,\cdot} \mc P_{-i,\cdot}(\mb E)\t \mc P^0 (\mc P^0 \mb W^{(i)} \mc P^0 )^k \mc P^0 \mb W^{(i)} \mb U^*}_2  \lesssim \tilde\sigma \deltaop^{k+1} (\sigma \sqrt{p} + \tilde\sigma \sqrt{n}) \sqrt{r}(t_0 \wedge \sqrt{\log d}) \leq \sqrt{\frac{\mu_1 r }{n}} \deltaop^{t+2} (t_0 \wedge \sqrt{\log d})
    \end{equation}
    holds with probability at least $1 - O(d^{-c_p-1} \vee e^{-t_0^2 / 2})$. 
\end{proof}

\begin{proof}[Proof of Lemma~\ref{lemma: bound for quantities in expansion (bounded case)}]
        To address dependence within a block, we have recourse to a leave-one-block-out approach in combination with an inductive argument. We first introduce the following events that characterize the behavior of the involved quantities: 
        \begin{align}
        & \mc E_{H_{-1}^{(i)}} \coloneqq \Big\{\norm{\mb E_{i,\cdot} \mb V^*}_2 \leq C_H \bar \sigma\sqrt{r}(t_0 \wedge \sqrt{\log d}); \qquad \norm{\mb E_{i,\cdot}\mc P_{-i,\cdot}(\mb E)\t \mb U^*}_2 \leq C_H \tilde\sigma \sigma  \sqrt{pr}(t_0 \wedge \sqrt{\log d}) \big]
        \Big\} ; \\
            & \mc E_{H_t^{(i)}} \coloneqq \bigg\{ 
           \norm{\mb E_{i,\cdot} \mc P_{-i,\cdot}(\mb E)\t \mc P^0 (\mc P^0 \mb W^{(i)} \mc P^0 )^{t} \mc P^0 \mb W^{(i)} \mb U^*}_2 \\ 
           & \leq C_H \tilde \sigma \norm{\mc P_{-i,\cdot}(\mb E)\t \mc P^0 (\mc P^0 \mb W^{(i)} \mc P^0 )^{t} \mc P^0 \mb W^{(i)} \mb U^*}_2  \sqrt{r} (t_0 \wedge \sqrt{\log d}) \\ 
            & \qquad  + C_H mB \norm{\mc P_{-i,\cdot}(\mb E)\t \mc P^0 (\mc P^0 \mb W^{(i)} \mc P^0 )^{t} \mc P^0 \mb W^{(i)} \mb U^*}\ti   (t_0 \wedge \sqrt{\log d})^2   
            \bigg\}; \\ 
            & \mc E_{H_{-1}} \coloneqq \bigg\{\norm{\mb E \mb V^*}\ti \leq C_H\bar \sigma \sqrt{r \log d};~  \norm{\mc H(\mb E\mb E\t) \mb U^*}\ti \leq C_H \tilde\sigma \norm{\mc P_{-i,\cdot}(\mb E)\t \mb U^*} \sqrt{r\log d} \\ 
        & \qquad + C_H mB\norm{\mc P_{-i,\cdot}(\mb E)\t \mb U^*}\ti \log d \bigg\} ;
            \\ 
            & \mc E_{H_t} \coloneqq \bigg\{ \norm{\mb E_{i,\cdot} \mc P_{-i,\cdot}(\mb E)\t \mc P^0 (\mc P^0 \mb W^{(i)} \mc P^0 )^{t} \mc P^0 \mb W^{(i)} \mb U^*}_2 \\ 
           & \leq C_H \tilde \sigma \norm{\mc P_{-i,\cdot}(\mb E)\t \mc P^0 (\mc P^0 \mb W^{(i)} \mc P^0 )^{t} \mc P^0 \mb W^{(i)} \mb U^*}_2 \sqrt{\log d} \\ 
            & \qquad  + C_H \sqrt{m}B \max_{c\in[l]}\norm{\mc P_{-i,S_c}(\mb E)\t \mc P^0 (\mc P^0 \mb W^{(i)} \mc P^0 )^{t} \mc P^0 \mb W^{(i)} \mb U^*} \log d ~\text{for}~i\in[n];\\ 
            & \norm{\mb E_{\cdot, S_c}\t \mc P^0  (\mc P^0 \mb W^{(-c)} \mc P^0 )^{t} \mc P^0 \mb W^{(-c)} \mb U^*} \leq C_H \sqrt{m}\tilde \sigma \norm{ (\mc P^0 \mb W^{(-c)} \mc P^0 )^{t} \mc P^0 \mb W^{(-c)} \mb U^*}\sqrt{r \log d} \\ 
            & \qquad + C_H \sqrt{m}B\norm{ (\mc P^0 \mb W^{(-c)} \mc P^0 )^{t } \mc P^0 \mb W^{(-c)} \mb U^*}\ti \log d  ~\text{for}~j\in[p]
            \bigg\}; \\ 
            & \mc E_E \coloneqq \bigg\{ 
            \norm{\mb E } \leq C_H (\tilde\sigma \sqrt{n} + \sigma \sqrt{p}),~
            \max_{c \in[l]}\norm{\mb E_{\cdot,S_c}} \leq C_H \tilde\sigma \sqrt{n}, ~ \max_{c\in[l]} \norm{\mb E_{\cdot,S_c}\t \mb U^*} \leq C_H \tilde \sigma \sqrt{mr \log d}, \\ 
            & \qquad \max_{c \in [l]} \bignorm{\mb W^{(-c)} \mb U^*}\ti \leq C_H \sqrt{\frac{\mu_1 r}{n}} \deltaop \sqrt{\log d}, ~ \max_{i\in[n]}\norm{\mb E_{i,\cdot} \mb V^*} \leq C_H \bar \sigma \sqrt{r \log d}
            \bigg\}
        \end{align}
        According to the matrix Bernstein inequality \citep[Theorem~6.1.1]{tropp2015introduction} together with Lemma~\ref{lemma: concentration for W and its variants} and Lemma~\ref{lemma: noise matrix concentrations using the universality (bounded)}, we can prove with an appropriately chosen $C_H$ that 
        \eq{
        \max_{t\geq -1}\bb P\big[\mc E_{H_t^{(i)}}^{\complement}\big] \leq O\Big(d^{-c_p-2} \vee e^{-t_0^2/2}\Big) \text{\quad and \quad } \bb P[\mc E_E^\complement] \vee \max_{t \geq -1}\bb P[{\mc E_{H_t}}^{\complement}] \leq O(d^{-c_p-2}) .
        \label{eq: exceptional probability used in induction}
        }

        We state the inductive hypotheses as follows: for each $t \geq 0$ and $i\in[n]$, 
        \begin{align}
         & \textsc{Hypothesis $H_{t}$: }\norm{(\mc P^0 \mb W \mc P^0 )^{t} \mc P^0 \mb W \mb U^*}\ti \leq C_0' \sqrt{\frac{\mu_1 r}{n}}\big[ C_0' \deltaop\big]^{t+1} \sqrt{\log d} \\ 
         & \qquad \qquad \qquad \qquad \qquad \qquad \qquad \qquad \qquad \qquad \text{~under~}\big(\cap_{k=-1}^{t}\mc E_{H_{k}}\big) \cap \mc E_W \cap \mc E_E; \label{eq: inductive hypothesis 1} \\ 
            & \textsc{Hypothesis $H_{t}'$: } \max_{c \in[l]} \norm{ \mb E_{\cdot, S_c}\t \mc P^0  (\mc P^0 \mb W \mc P^0 )^{t} \mc P^0 \mb W \mb U^*} \leq C_0'  \big[ C_0' \deltaop\big]^{t+1} \Big[ \tilde \sigma  \sqrt{m r\log d} +  \frac{\tilde \sigma^2 }{\bar \sigma } \sqrt{mn}\sqrt{\frac{ \mu_2 r}{p}}\Big]   \\
            & \qquad \qquad \qquad \qquad \qquad \qquad \qquad \qquad \qquad \qquad \text{~under~}\big(\cap_{k=-1}^{t}\mc E_{H_{k}}\big) \cap \mc E_W \cap \mc E_E; 
            \label{eq: inductive hypothesis 2}
            \\ 
            & \textsc{Hypothesis $H_{t}^{(i)}$: } \norm{\mb E_{i,\cdot} \mc P_{-i,\cdot}(\mb E)\t \mc P^0 (\mc P^0 \mb W^{(i)} \mc P^0 )^t \mc P^0 \mb W^{(i)} \mb U^*} \leq  \frac{1}{3}C_0' \sqrt{\frac{\mu_1 r}{n}}\big[ C_0' \deltaop\big]^{t+2}(t_0 \wedge \sqrt{\log d}), \\ 
            &   \norm{\mb E_{i,\cdot} \mc P_{-i,\cdot}(\mb E)\t \mc P^0 (\mc P^0 \mb W \mc P^0 )^t \mc P^0 \mb W \mb U^*}_2 \leq  \frac{2}{3}C_0' \sqrt{\frac{\mu_1 r}{n}}\big[ C_0' \deltaop\big]^{t+2}(t_0 \wedge \sqrt{\log d}),  \\ 
            & \text{\qquad and \quad } \norm{ \big((\mc P^0 \mb W \mc P^0 )^{t+1} \mc P^0 \mb W \mb U^*\big)_{i,\cdot}}_2 \leq C_0' \sqrt{\frac{\mu_1 r}{n}}\big[ C_0' \deltaop\big]^{t+2}(t_0 \wedge \sqrt{\log d}),  \\ 
            &  \qquad \qquad \qquad \qquad \qquad \qquad \qquad \qquad \qquad  \text{~under~}\big(\cap_{k=-1}^{t}\mc E_{H_{k}^{(i)}}\big) \cap \big(\cap_{k=-1}^{t}\mc E_{H_{k}}\big) \cap \mc E_W \cap \mc E_E, \label{eq: inductive hypothesis 3}
            \end{align}
    where the event $\mc E_W$ was previously defined in \eqref{eq: definition of E_H}. 
The desired conclusion follows immediately from \textsc{Hypothesis} $H_t^{(i)}$ together with the exceptional probability bound in \eqref{eq: exceptional probability used in induction}.
Thus, it remains only to establish the inductive hypotheses.
        
        \paragraph{Base Case. }
        \emph{Verifying $H_0$. }
        We start by verifying hypothesis $H_0$ in \eqref{eq: inductive hypothesis 1} in the base case $t=0$. It is straightforward that under $\mc E_{H_{-1}} \cap \mc E_W \cap \mc E_E$, 
        \begin{equation}
             \norm{\mc P^0 \mb W \mb U^*}\ti \leq \norm{\mb U^*}\ti \norm{\mb W \mb U^*} + \norm{\mb W\mb U^*}\ti \leq \sqrt{\frac{\mu_1 r}{n}} \deltaop + C_H \sqrt{\frac{\mu_1 r}{n}} \deltaop \sqrt{\log d} \leq C_0' \sqrt{\frac{\mu_1 r}{n}} \deltaop \sqrt{\log d}.  
            \label{eq: proof of inductive hypothesis 1 base case}
        \end{equation}
        for some constant $C_0'$.

        \bigskip

        \emph{Verifying $H_0'$. }
        Then we move on to the second hypothesis $H_0'$ in \eqref{eq: inductive hypothesis 2} by evaluating $\norm{\mb E_{\cdot, S_c}\t \mc P^0 \mb W \mb U^*}$ for each block $S_c$. For $c \in[l]$, replacing $\mb W$ with $\mb W^{(-c)}$ gives that 
        \begin{align}
            & \norm{\mb E_{\cdot, S_c}\t \mc P^0 \mb W \mb U^*}_2 \leq \norm{\mb E_{\cdot, S_c}\t \mc P^0 \mb W^{(-c)} \mb U^*}_2 + \norm{\mb E_{\cdot, S_c}\t \mc P^0 \big( \mb W - \mb W^{(-c)} \big)\mb U^*}_2  \\ 
            \leq & C_H \sqrt{m}\tilde \sigma \norm{\mc P^0 \mb W^{(-c)} \mb U^*} \sqrt{r \log d} + \sqrt{m}B \norm{\mc P^0 \mb W^{(-c)} \mb U^*}\ti \\ 
            & + \norm{\mb E_{\cdot, S_c}} \norm{\mc P^0 \mc H(\mb E_{\cdot, S_{c}} \mb E_{\cdot, S_{c}}\t + \mb M^*_{\cdot, S_{c}} \mb E_{\cdot, S_{c}}\t + \mb E_{\cdot, S_{c}} {\mb M^*_{\cdot, S_{c}}}\t  ) \mb U^* }, 
            \label{eq: matrix bernstein for H_0'}
        \end{align}
        where we used the conditions in $\mc E_{H_0}$. Regarding the quantities on the RHS of \eqref{eq: matrix bernstein for H_0'}, we invoke the conditions in $\mc E_H \cap \mc E_E$ to yield that 
        \begin{align}
            & \norm{\mc P^0 \mb W^{(-c)} \mb U^*}\leq \norm{\mb W^{(-c)}} \leq \deltaop,\quad \norm{\mb E_{\cdot, S_c}}_2 \leq C_H \tilde \sigma \sqrt{n}, \\ 
            & \norm{\mc P^0 \mb W^{(-c)} \mb U^*}\ti\leq \norm{\mb W^{(-c)} \mb U^*}\ti + \norm{\mb U^*}\ti \norm{\mb W^{(-c)} } \leq 2 C_H \sqrt{\frac{\mu_1 r}{n}} \deltaop \sqrt{\log d},  \\ 
            & \norm{\mc P^0 \mc H(\mb E_{\cdot, S_{c}} \mb E_{\cdot, S_{c}}\t + \mb M^*_{\cdot, S_{c}} \mb E_{\cdot, S_{c}}\t + \mb E_{\cdot, S_{c}} {\mb M^*_{\cdot, S_{c}}}\t  ) \mb U^* } \leq \norm{\mb E_{\cdot, S_{c}}} \norm{\mb E_{\cdot, S_{c}}\t \mb U^*} +2 \sqrt{m} \norm{\mb E_{\cdot, S_c}} \bignorm{{\mb M^*}\t }\ti \\ 
            & \qquad \qquad + \Bignorm{\diag(\mb E_{\cdot, S_{c}} \mb E_{\cdot, S_{c}}\t)} \norm{\mb U^*}\ti + 2\Bignorm{\diag(\mb M^*_{\cdot, S_{c}} \mb E_{\cdot, S_{c}}\t)} \norm{\mb U^*}\ti \\ 
            & \qquad \leq C_H^2 \tilde \sigma^2 \sqrt{nmr\log d} + 2 C_H \tilde \sigma \sqrt{mn} \sigma_1^* \sqrt{\frac{\mu_2 r }{p}} + mB^2 \sqrt{\frac{\mu_1 r}{n}} + 2 \frac{\mu_1 r}{n} \sigma_1^* \sqrt{m}B \\
            & \qquad \leq   C_1' \big(\tilde\sigma^2 \sqrt{nmr \log d}  + \tilde \sigma \sqrt{mn} \sigma_1^* \sqrt{\frac{\mu_2 r }{p}}  + \bar \sigma \sigma_1^*  \sqrt{r}\big)
        \end{align}
        for some constant $C_1'$, where the last line results from $mB^2 \sqrt{\mu_1 r / n} \leq (\sqrt{m\mu_1}B)^2 \sqrt{r/n} \lesssim \bar \sigma^2 n \sqrt{r / n}$ and $\frac{\mu_1 r}{n} \sigma_1^*\sqrt{m}B \leq \sqrt{\mu_1 r / n} \sigma_1^* \sqrt{m}B \lesssim \sigma_1^* \bar \sigma \sqrt{r}$, given the condition that $\sqrt{m\mu_1}B \lesssim \sqrt{n}\bar \sigma$. 

        Plugging these bounds into \eqref{eq: matrix bernstein for H_0'} yields that 
        \begin{align}
            & \norm{\mb E_{\cdot, S_c}\t \mc P^0 \mb W \mb U^*} \leq C_H \sigma \deltaop \sqrt{mr\log d} + \sqrt{m}B \big(2 C_H \sqrt{\frac{\mu_1 r}{n}} \deltaop \sqrt{\log d}\big) \\ 
            & \qquad + (C_H \tilde \sigma \sqrt{n}) C_1' \Big[ \tilde\sigma^2 \sqrt{mnr \log d}  + \tilde \sigma \sqrt{mn} \sigma_1^* \sqrt{\frac{\mu_2 r }{p}}  + \bar \sigma \sigma_1^*  \sqrt{r} \Big]  \\ 
            \leq  & C_0' \deltaop  \Big[\tilde \sigma  \sqrt{m r\log d} +  \frac{\tilde\sigma^2 }{\bar \sigma} \sqrt{mn}\sqrt{\frac{ \mu_2 r}{p}}\Big]
        \end{align}
         under $\mc E_{H_0} \cap \mc E_W \cap \mc E_E$, provided a sufficiently large $C_0'$. Here we use the facts that $\sigma \sigma_1^* \sqrt{n} = \frac{\sigma}{\bar \sigma} \bar\sigma \sigma_1^* \sqrt{n} \leq \frac{\sigma}{\bar \sigma} \deltaop$ and $\sqrt{m}B \sqrt{\mu_1 r/n}\sqrt{\log d} \lesssim \sigma $. 
         Since the above arguments hold for every $c \in[l]$ under $\mc E_{H_0} \cap \mc E_W \cap \mc E_E$, thus $H_0'$ has been proved. 

        \bigskip

        \emph{Verifying $H_0^{(i)}$. }
        Next, we shall verify the correctness of $H_{0}^{(i)}$. 
        Under the event $\mc E_{H_0^{(i)}}$, one has 
        \begin{align}
            & \norm{\mb E_{i,\cdot} \mc P_{-i,\cdot}(\mb E)\t \mc P^0 \mb W^{(i)} \mb U^*} \leq C_H \tilde\sigma \norm{ \mc P_{-i,\cdot}(\mb E)\t \mc P^0 \mb W^{(i)} \mb U^*} \sqrt{r}(t_0 \wedge \sqrt{\log d})\\ 
            & \qquad \qquad \qquad + C_H \sqrt{m}B \max_{c\in[l]}\norm{ \mc P_{-i,S_c}(\mb E)\t \mc P^0 \mb W^{(i)} \mb U^*} (t_0 \wedge \sqrt{\log d})^2.
            \label{eq: bernstein inequality for the base case}
        \end{align} 
        Under $\mc E_W \cap \mc E_E$, the first term of the RHS above is bounded by 
        \begin{align}
            & \norm{ \mc P_{-i,\cdot}(\mb E)\t \mc P^0 \mb W^{(i)} \mb U^*} \leq C_H(\tilde\sigma \sqrt{n} + \sigma\sqrt{p}) \deltaop. 
            \label{eq: term 1 in the bernstein inequality for the base case}
        \end{align}
        
        Regarding $\max_{c\in[l]}\norm{ \mc P_{-i,S_c}(\mb E)\t \mc P^0 \mb W^{(i)} \mb U^*}$ in the second term of the RHS above, $\mb W^{(i)}$ can be replaced by the alternative in \eqref{eq: inductive hypothesis 1} to gives that 
        \begin{align} 
        & \max_{c\in[l]}\norm{ \mc P_{-i,S_c}(\mb E)\t \mc P^0 \mb W^{(i)} \mb U^*} \leq \max_{c\in[l]}\norm{\mb E_{\cdot, S_c}\t \mc P^0 \mb W \mb U^*} + \max_{c\in[l]}\norm{\mc P_{i,S_c}(\mb E)\t \mc P^0 \mb W^{(i)} \mb U^*} \\ 
        & + \max_{c\in[l]}\norm{\mb E_{\cdot, S_c}\t \mc P^0 (\mb W - \mb W^{(i)}) \mb U^*} 
            \label{eq: leave-two-out replacement for the base case}
        \end{align}
        with $j \in S_{c(j)}$. 
        For the first term in the RHS of \eqref{eq: leave-two-out replacement for the base case}, the preceding hypothesis $H_0'$ gives that
        \begin{align}
            & \max_{c\in[l]}\norm{\mb E_{\cdot, S_c}\t \mc P^0 \mb W \mb U^*}\leq C_0' \deltaop \Big[\tilde \sigma  \sqrt{m r\log d} +  \frac{\tilde\sigma^2}{\bar \sigma} \sqrt{mn}\sqrt{\frac{ \mu_2 r}{p}}\Big]
            \label{eq: term 1 for leave-two-out matrix bernstein (base case)}
        \end{align}
        under the event $\mc E_{H_{-1}} \cap \mc E_{H_0} \cap \mc E_W \cap \mc E_E$.

        Moreover, the second term in \eqref{eq: leave-two-out replacement for the base case} is controlled by 
        \begin{align}
            & \max_{c\in[l]}\norm{\mc P_{i,S_c}(\mb E)\t \mc P^0 \mb W^{(i)} \mb U^*}  \leq \sqrt{m}B \Big[\norm{\mb U^*}\ti \bignorm{\mb W^{(i)}} + \bignorm{ \mb W^{(i)} \mb U^*}\ti\Big] \leq C_2' \sqrt{m}B \sqrt{\frac{\mu_1 r}{n}} \deltaop \sqrt{\log d}
            \label{eq: term 2 for leave-two-out matrix bernstein (base case)}
        \end{align}
        for some constant $C_2'$ under $\mc E_W \cap \mc E_{E}$. 
        
        With respect to the thrid term in \eqref{eq: leave-two-out replacement for the base case}, we recall the decomposition in \eqref{eq: decomposition of W - W-i} and have under $\mc E_W \cap \mc E_E \cap \mc E_{H_{-1}^{(i)}}$ that
        \begin{align} 
            &\norm{\mc P^0\big( \mb W - \mb W^{(i)}\big)  \mb U^*}  \\ 
            \leq  & \norm{\mb E_{i,\cdot}\mc P_{-i,\cdot}(\mb E)\t \mb U^*} + \norm{\mc P_{-i,\cdot}(\mb E) \mb E_{i,\cdot}\t }_2  \norm{\mb U^*}\ti + \bignorm{\mb E_{i,\cdot} {\mb M^*}\t }_2 + 2 \big|\mb E_{i,\cdot}{\mb M^*_{i,\cdot}}\t\big|   \\ 
            \leq & \Big[(C_H \tilde\sigma \sigma \sqrt{p r}(t_0 \wedge \sqrt{\log d} ) + \deltaop \sqrt{\frac{\mu_1 r}{n}} + 3 C_H\bar \sigma \sqrt{r} (t_0 \wedge \sqrt{\log d} ) \sigma_1^* \Big] \\ 
            \leq&  C_3' \deltaop \sqrt{\frac{\mu_1 r}{n}}(t_0 \wedge \sqrt{\log d} ),
            \label{eq: upper bound for P0W - Wi U*}
            \\
            & \max_{c\in[l]}\norm{\mb E_{\cdot, S_c}\t \mc P^0 (\mb W - \mb W^{(i)}) \mb U^*} \leq 
            \max_{c\in[l]} \norm{\mb E_{\cdot, S_c} }  \norm{\mc P^0\big( \mb W - \mb W^{(i)}\big)  \mb U^*}
            \\ 
            \leq &  (C_H \tilde \sigma \sqrt{n})(C_3' \deltaop \sqrt{\frac{\mu_1 r}{n}})(t_0 \wedge \sqrt{\log d} )
            \label{eq: term 3 for leave-two-out matrix bernstein (base case)}
        \end{align}
        for some constant $C_3'$, where we used the fact that $\norm{\mc P_{-i,\cdot}(\mb E) \mb E_{i,\cdot}\t }_2 \leq \norm{\mc H(\mb E\mb E\t)} \leq \deltaop $ under $\mc E_W$. 

        Plugging \eqref{eq: term 1 for leave-two-out matrix bernstein (base case)}, \eqref{eq: term 2 for leave-two-out matrix bernstein (base case)}, and \eqref{eq: term 3 for leave-two-out matrix bernstein (base case)} into \eqref{eq: leave-two-out replacement for the base case} yields that
        \begin{align}
            & \norm{ \mc P_{-i,S_c}(\mb E)\t \mc P^0 \mb W^{(i)} \mb U^*}_2 \leq C_0' \deltaop \Big[ \tilde \sigma  \sqrt{m r\log d} +  \frac{\tilde\sigma^2}{\bar \sigma}\sqrt{mn}\sqrt{\frac{ \mu_2 r}{p}}\Big]  + C_2' \sqrt{m} B \sqrt{\frac{\mu_1 r}{n}} \deltaop \sqrt{\log d} \\ 
            & \qquad  + (C_H \tilde \sigma \sqrt{n})(C_3' \deltaop \sqrt{\frac{\mu_1 r}{n}}) (t_0 \wedge \sqrt{\log d} )
            \leq C_4' \deltaop  \Big[ \tilde \sigma  \sqrt{\mu_1 m r\log d} + \frac{\tilde\sigma^2}{\bar \sigma} \sqrt{mn} \sqrt{\frac{\mu_2 r}{p}}\Big]
            \label{eq: final bound in the leave-two-out for the base case}
        \end{align}
        holds under $\big(\mc E_{H_{-1}^{(i)}} \cap \mc E_{H_0^{(i)}}\big) \cap \big(\mc E_{H_{-1}} \cap \mc E_{H_0} \big) \cap \mc E_W \cap \mc E_E$ for some constant $C_4'$, where the last inequality holds since $ B\sqrt{\mu_1 r /n} \lesssim \sigma $. 
        Then, provided a sufficiently large constant $C_H$, substituting \eqref{eq: term 1 in the bernstein inequality for the base case} and \eqref{eq: final bound in the leave-two-out for the base case} into \eqref{eq: bernstein inequality for the base case} yields that 
        \begin{align}
           & \norm{\mb E_{i,\cdot} \mc P_{-i,\cdot}(\mb E)\t \mc P^0 \mb W^{(i)} \mb U^*} \leq C_H \tilde\sigma (\tilde\sigma \sqrt{n} + \sigma\sqrt{p}) \deltaop \sqrt{r} (t_0 \wedge \sqrt{\log d}) \\ 
           & \qquad \qquad + \sqrt{m}B\bigg\{  C_4' \deltaop  \Big[ \tilde \sigma  \sqrt{\mu_1 m r\log d} + \frac{\tilde\sigma^2}{\bar \sigma} \sqrt{mn} \sqrt{\frac{\mu_2 r}{p}}\Big]\bigg\} (t_0 \wedge \sqrt{\log d})^2  \\ 
           \leq & \frac13 C_0' \sqrt{\frac{\mu_1 r}{n}} \deltaop^2 (t_0 \wedge \sqrt{\log d})
           \label{eq: first hypothesis for base case}
        \end{align}
        holds under $\mc E_{H_0} \cap \mc E_W \cap \mc E_E$ for some constant $C_5'$ under $\big(\mc E_{H_{-1}^{(i)}} \cap \mc E_{H_0^{(i)}}\big) \cap \big(\mc E_{H_{-1}} \cap \mc E_{H_0} \big) \cap \mc E_W \cap \mc E_E$, where we made use of the facts that 
        \begin{align}
            & m  \sqrt{\mu_1 }\log d B \lesssim \sigma \sqrt{n} \lesssim \tilde\sigma \sqrt{n}, \quad  \frac{\tilde \sigma }{\bar \sigma} m B \log d \sqrt{\mu_2  / p} \lesssim \tilde\sigma. 
        \end{align}

        Toward the second inequality in $H_t^{(i)}$, applying the triangle inequality implies that 
        \begin{align}
            & \norm{\mb E_{i,\cdot} \mc P_{-i,\cdot}(\mb E)\t \mc P^0 \mb W\mb U^*} \leq \norm{\mb E_{i,\cdot} \mc P_{-i,\cdot}(\mb E)\t \mc P^0 \mb W^{(i)}  \mb U^*} \\ 
            & \qquad + \norm{\mb E_{i,\cdot} \mc P_{-i,\cdot}(\mb E)\t \mc P^0 \mc H(\mb E_{i,\cdot }\mb E_{-i,\cdot }\t + \mb E_{-i,\cdot } \mb E_{i,\cdot }\t +\mb E_{i,\cdot} {\mb M^*}\t + \mb M^*\mb E_{i,\cdot}\t )  \mb U^*}\\ 
            \leq &  \norm{\mb E_{i,\cdot} \mc P_{-i,\cdot}(\mb E)\t \mc P^0 \mb W^{(i)}  \mb U^*} + \norm{\mc H(\mb E\mb E)} \Big[\norm{\mb E_{i,\cdot }\mb E_{-i,\cdot }\t  \mb U^*} + \norm{\mb E}\norm{\mb E_{i,\cdot}}_2 \norm{\mb U^*}\ti \\ 
            & \qquad  + 2\norm{\mb U^*}\ti^2 \sigma_1^* \norm{\mb E_{i,\cdot} \mb V^*}_2 + \sigma_1^* \norm{\mb E_{i,\cdot} \mb V^*}_2  \Big] \\ 
            \leq & \frac13 C_0' \sqrt{\frac{\mu_1 r}{n}} \deltaop^2 (t_0 \wedge \sqrt{\log d})  + \deltaop \Big[C_H \tilde\sigma \sigma \sqrt{pr}(t_0 \wedge \sqrt{\log d}) \\ 
        & \qquad+  \big(C_H (\tilde \sigma \sqrt{n} + \sigma\sqrt{p})\big)(C_H\sigma \sqrt{n})\sqrt{\frac{\mu_1 r}{n}}  + 2 \frac{\mu_1 r}{n} \sigma_1^* (C_H \bar \sigma \sqrt{n}) + \sigma_1^* (C_H \bar \sigma \sqrt{n}) \Big] \\ 
            \leq & \frac{2}{3}C_0' \sqrt{\frac{\mu_1 r}{n}} \deltaop^2 (t_0 \wedge \sqrt{\log d})
              \label{eq: second hypothesis for base case}
        \end{align}
        for some constant $C_6'$ under the event $\big(\mc E_{H_{-1}^{(i)}} \cap \mc E_{H_0^{(i)}}\big) \cap \big(\mc E_{H_{-1}} \cap \mc E_{H_0} \big) \cap \mc E_W \cap \mc E_E$. 

        For the last inequality, provided a sufficiently large $C_0'$, one has under $\big(\mc E_{H_{-1}^{(i)}} \cap \mc E_{H_0^{(i)}}\big) \cap \big(\mc E_{H_{-1}} \cap \mc E_{H_0} \big) \cap \mc E_W \cap \mc E_E$ that 
        \begin{align}
            & \norm{\big( \mc P^0 \mb W \mc P^0 \mb W \mb U^*\big)_{i,\cdot}}_2 \leq \norm{\mb E_{i,\cdot} \mc P_{-i,\cdot}(\mb E)\t \mc P^0 \mb W\mb U^*}  \\ 
            & \qquad  + \big(2|\mb E_{i,\cdot}{\mb M^*_{i,\cdot}}\t | + \bignorm{\diag(\mb M^*{\mb M^*}\t)} \big) \norm{(\mc P^0 \mb W \mb U^*)_{i,\cdot} }_2 + \norm{\mb U^*}\ti \norm{\mb W}^2  \\ 
            \leq &  \frac{2}{3}C_0' \sqrt{\frac{\mu_1 r}{n}} \deltaop^2 (t_0 \wedge \sqrt{\log d}) \\ 
            & + \big(2C_H \bar \sigma \sqrt{r \log d }\sigma_1^* \sqrt{\mu_1 r /n } + \frac{\mu_1 r}{n} {\sigma_1^*}^2\big) \cdot \Big[(C_H \tilde\sigma \sigma \sqrt{p r}(t_0 \wedge \sqrt{\log d}) + 3C_H \sigma_1^* \bar \sigma \sqrt{r}(t_0 \wedge \sqrt{\log d}) \\ 
            & \qquad + \frac{\mu_1 r}{n} {\sigma_1^*}^2 \sqrt{\frac{\mu_1 r}{n}} + 2 C_H\bar \sigma \sqrt{r} (t_0 \wedge \sqrt{\log d} ) \sigma_1^*  \Big] + \sqrt{\frac{\mu_1 r}{n}} \deltaop^2 \\ 
            \leq & C_0' \sqrt{\frac{\mu_1 r}{n}} \deltaop^2 (t_0 \wedge \sqrt{\log d}),
        \end{align}
        according to \eqref{eq: second hypothesis for base case}, the conditions under the specificed event, and the fact that 
        \longeq{
        \norm{(\mc P^0 \mb W \mb U^*)_{i,\cdot} }_2 \leq  &\norm{\mb E_{i,\cdot} \mc P_{-i,\cdot}(\mb E)\t \mb U^*}_2 + \norm{\mb E_{i,\cdot} \mb V^*}_2 \sigma_1^* + 2|\mb E_{i,\cdot}{\mb M^*_{i,\cdot}}\t | \\ 
        & \qquad + \bignorm{\diag(\mb M^*{\mb M^*}\t)}\norm{\mb U^*}\ti + \norm{\mb U^*}\ti\norm{\mb W} \\ 
        \leq & \norm{\mb E_{i,\cdot} \mc P_{-i,\cdot}(\mb E)\t \mb U^*}_2 + 3\norm{\mb E_{i,\cdot} \mb V^*}_2 \sigma_1^* + \bignorm{\diag(\mb M^*{\mb M^*}\t)}\norm{\mb U^*}\ti + \norm{\mb U^*}\ti\norm{\mb W}. 
        }
        
        \bigskip

        \paragraph{Induction Step with $t \geq 1$. }
        Now we move on to the case with $t \geq 1$ assuming that the hypothesis holds for $ k= 0, \cdots, t - 1$. 

        \emph{Verifying $H_t$ in \eqref{eq: inductive hypothesis 1}. }
        Invoking the fact that $\mc P^0 = \mb I - \mb U^*{\mb U^*}\t$ along with \eqref{eq: identity of P0 W P0}, one has 
        \begin{align}
            & \norm{(\mc P^0 \mb W \mc P^0 )^{t} \mc P^0 \mb W \mb U^*}\ti \leq \norm{\mc H(\mb E\mb E\t ) \mc P^0 (\mc P^0 \mb W \mc P^0 )^{t-1} \mc P^0 \mb W \mb U^* }\ti \\ 
            & \quad + 2 \max_{i \in[n]}\big|\mb E_{i,\cdot}\mb M^*_{i,\cdot} \big| \norm{(\mc P^0 \mb W \mc P^0 )^{t-1} \mc P^0 \mb W \mb U^* }\ti  
            + \norm{\mb U^*}\ti  \norm{\mb W}^{t+1}
            \\ 
            \leq & C_0' \sqrt{\frac{\mu_1 r}{n}} \deltaop^{t+1} \sqrt{\log d}
        \end{align}
        under $\big(\cap_{k=-1}^{t}\mc E_{H_{t}}\big) \cap \mc E_W \cap \mc E_E$. 

           \emph{Verifying $H_t'$ in \eqref{eq: inductive hypothesis 3}. }
         For $\norm{ \mb E_{\cdot, S_c}\t \mc P^0  (\mc P^0 \mb W \mc P^0 )^{t} \mc P^0 \mb W \mb U^*}_2$ with $c \in[l]$, we employ the telescoping decomposition to derive that 
        \begin{align}
            & \norm{\mb E_{\cdot, S_c}\t \mc P^0  (\mc P^0 \mb W \mc P^0 )^{t } \mc P^0 \mb W \mb U^*} \leq \underbrace{\norm{\mb E_{\cdot, S_c}\t \mc P^0  (\mc P^0 \mb W^{(-c)} \mc P^0 )^{t } \mc P^0 \mb W^{(-c)} \mb U^*}}_{\alpha_{1}} \\ 
            & \qquad \qquad  + \underbrace{\sum_{h=0}^{t -1} \norm{ \mb E_{\cdot, S_c}\t \mc P^0  (\mc P^0 \mb W^{(-c)} \mc P^0 )^h \mc P^0 (\mb W^{(-c)} - \mb W )\mc P^0  (\mc P^0 \mb W \mc P^0 )^{t - h -1} \mc P^0 \mb W \mb U^*} }_{\alpha_{2}}\\ 
            & \qquad \qquad +\underbrace{ \norm{\mb E_{\cdot, S_c}\t   \mc P^0  (\mc P^0 \mb W^{(-c)} \mc P^0 )^{t } \mc P^0 (\mb W^{(-c)} - \mb W)\mb U^*}}_{\alpha_{3}}.
            \label{eq: decomposition regarding alpha123 in leave-two-out}
        \end{align}

        We shall focus on $\alpha_1$ first. The conditions in $\mc E_{H_t}$ implies that 
        \begin{align}
            & \alpha_1 \leq C_H \sqrt{m} \tilde \sigma \norm{ (\mc P^0 \mb W^{(-c)} \mc P^0 )^{t } \mc P^0 \mb W^{(-c)} \mb U^*}\sqrt{r \log d} + C_H \sqrt{m}B\norm{ (\mc P^0 \mb W^{(-c)} \mc P^0 )^{t } \mc P^0 \mb W^{(-c)} \mb U^*}\ti \log d. 
            \label{eq: matrix Bernstein in leave-two-out}
        \end{align} 
        With respect to $\norm{ (\mc P^0 \mb W^{(-c)} \mc P^0 )^{t } \mc P^0 \mb W^{(-c)} \mb U^*}\ti$ in \eqref{eq: matrix Bernstein in leave-two-out}, a telescoping decomposition again gives that 
        \begin{align}
            & \norm{ (\mc P^0 \mb W^{(-c)} \mc P^0 )^{t } \mc P^0 \mb W^{(-c)} \mb U^*}\ti \leq \norm{ (\mc P^0 \mb W \mc P^0 )^{t } \mc P^0 \mb W \mb U^*}\ti \\ 
            & \qquad \qquad  + \sum_{h=0}^{t -1} \norm{  (\mc P^0 \mb W^{(-c)} \mc P^0 )^h \mc P^0 (\mb W^{(-c)} - \mb W )\mc P^0  (\mc P^0 \mb W \mc P^0 )^{t - h -1} \mc P^0 \mb W \mb U^*}  \\ 
            & \qquad \qquad + \norm{    (\mc P^0 \mb W^{(-c)} \mc P^0 )^{t } \mc P^0 (\mb W^{(-c)} - \mb W)\mb U^*}. 
            \label{eq: decomposition of the upper bound term matrix Bernstein in leave-two-out}
        \end{align}

        We notice the relation that 
        $$\mb W - \mb W^{(-c)} = \mc H\big(\mc P_{:,S_{c(j)}}(\mb E)\mc P_{:,S_{c(j)}}(\mb E)\t  + \mc P_{:,S_{c(j)}}(\mb E) {\mb M^*_{:, S_{c(j)}}}\t +  \mb M^*_{:, S_{c(j)}} \mc P_{:,S_{c(j)}}(\mb E)\t \big).$$ 
        It consequently allows us to control the summands appearing in \eqref{eq: decomposition of the upper bound term matrix Bernstein in leave-two-out} (as well as $\alpha_2$'s summands in \eqref{eq: decomposition regarding alpha123 in leave-two-out}) as follows: 
        \begin{align}
            & \norm{ (\mc P^0 \mb W^{(-c)} \mc P^0 )^h \mc P^0 (\mb W^{(-c)} - \mb W )\mc P^0  (\mc P^0 \mb W \mc P^0 )^{t - h -1} \mc P^0 \mb W \mb U^*} \\ 
            \leq & \norm{\mb W^{(-c)}}^h  \cdot \Big[  \norm{\mb E_{\cdot, S_c}}_2\norm{\mb E_{\cdot, S_c}\t \mc P^0  (\mc P^0 \mb W \mc P^0 )^{t - h -1} \mc P^0 \mb W \mb U^*}_2\\ 
            & \qquad \qquad \qquad +  \max_{i'\in[n]} \big( 2 \big| \mb E_{i',S_{c(j)}}{\mb M^*_{i',S_{c(j)}}}\t \big| + \big| \mb E_{i',S_{c(j)}}\mb E_{i',S_{c(j)}}\t \big|\big)\norm{(\mc P^0 \mb W \mc P^0 )^{t - h -1} \mc P^0 \mb W \mb U^*}\ti   \Big] \\ 
            \leq  & \deltaop^h \Big[  (C_H \tilde \sigma\sqrt{n}) \Big( C_0' \big[ C_0' \deltaop\big]^{t-h} \Big[ \tilde \sigma  \sqrt{mr\log d} + \frac{\tilde\sigma^2}{\bar \sigma}  \sqrt{m n} \sqrt{\frac{\mu_2 r}{p}}\Big] \Big) \\ 
            & \qquad + 2 m B \sigma_1^* \sqrt{\frac{\mu_1 \mu_2 r^2 }{np}} \deltaop^{t - h} + mB^2 \big[C_0' \sqrt{\frac{\mu_1 r}{n}} (C_0' \deltaop)^{t-h} \sqrt{\log d} \big]
            \Big] \\ 
            \leq &  C_7' \Big[\big[ C_0' \deltaop\big]^{t}  \tilde\sigma^2 \sqrt{mnr \log d}+ \big[ C_0' \deltaop\big]^t \frac{\tilde \sigma}{\bar \sigma} \sqrt{m} (\tilde \sigma^2 n) \sqrt{\mu_2 r/ p} +\sigma\sqrt{mr} \sigma_1^* \Big], 
            \label{eq: summand control in leave-two-out 1}
            \\ 
            & \norm{ \mb E_{\cdot, S_c}\t \mc P^0  (\mc P^0 \mb W^{(-c)} \mc P^0 )^h \mc P^0 (\mb W^{(-c)} - \mb W )\mc P^0  (\mc P^0 \mb W \mc P^0 )^{t - h -1} \mc P^0 \mb W \mb U^*}_2 \\ 
            \leq & \norm{\mb E_{\cdot, S_c}}_2 \norm{ (\mc P^0 \mb W^{(-c)} \mc P^0 )^h \mc P^0 (\mb W^{(-c)} - \mb W )\mc P^0  (\mc P^0 \mb W \mc P^0 )^{t - h -1} \mc P^0 \mb W \mb U^*} \\
            \leq  &   C_7' \tilde \sigma \sqrt{n} \Big[\deltaop^h \big[ C_0' \deltaop\big]^{t-h}  \tilde\sigma^2  \sqrt{mnr \log d}+ \deltaop^h \big[ C_0' \deltaop\big]^{t-h} \frac{\tilde \sigma}{\bar \sigma} \sqrt{m} (\tilde \sigma^2 n) \sqrt{\mu_2 r/ p} +\sigma\sqrt{mr} \sigma_1^* \Big]
            \label{eq: summand control in alpha 2}
        \end{align}
        hold for some constant $C_7'$ under $\big(\cap_{k=-1}^{t}\mc E_{H_{t}}\big) \cap \mc E_W \cap \mc E_E$. In the last inequality of \eqref{eq: summand control in leave-two-out 1}, we used the facts that 
        \begin{align}
            &   m B \sigma_1^* \sqrt{\frac{\mu_1 r}{n}} \leq \sigma\sqrt{mr} \sigma_1^*, 
            \label{eq: condition 1 in summand control of alpha 2}
            \\ 
            & mB^2 \sqrt{\mu_1 r / n} \sqrt{\log d} \lesssim \sigma^2 \sqrt{nr\log d}. \label{eq: condition 2 in summand control of alpha 2}
        \end{align}
        
        Additionally, we control the last term on the RHS of \eqref{eq: decomposition of the upper bound term matrix Bernstein in leave-two-out} (as well as $\alpha_3$ in \eqref{eq: matrix Bernstein in leave-two-out}) to derive that 
        \begin{align}
            & \norm{ (\mc P^0 \mb W^{(-c)} \mc P^0 )^{t} \mc P^0 (\mb W^{(-c)} - \mb W)\mb U^*} \\ 
            \leq & \norm{\mb W^{(-c)}}^{t } \Big[\norm{\mb E_{\cdot, S_c}}_2\big( \norm{\mb E_{\cdot, S_c}\t \mb U^*}_2 + \sqrt{m} \sigma_1^* \sqrt{\frac{\mu_2 r}{p}}\big) \\ 
            & \qquad \qquad + \max_{i'\in[n]} \big( 2 \big| \mb E_{i',S_{c(j)}}{\mb M^*_{i',S_{c(j)}}}\t \big| + \big| \mb E_{i',S_{c(j)}}\mb E_{i',S_{c(j)}}\t \big|\big) \norm{\mb U^*}\ti\Big]  \\ 
            \leq & \deltaop^{t } \Big[ (C_H \tilde \sigma  \sqrt{n}) \sigma (C_H \tilde\sigma \sqrt{mr\log d} + \sqrt{m}\sigma_1^* \sqrt{\frac{\mu_2 r}{p}}) + 2 m B \sigma_1^* \sqrt{\frac{\mu_1 \mu_2 r^2 }{np}}  + mB^2 \sqrt{\frac{\mu_1 r}{n}} \Big] \\ 
            \leq &  C_8' \deltaop^{t }\Big[  \tilde\sigma^2 \sqrt{m nr \log d} + \sigma_1^* \tilde \sigma \sqrt{mn} \sqrt{\mu_2 r /p}\Big],
            \label{eq: summand control in leave-two-out 2}
            \\ 
            & \alpha_3 \leq \norm{\mb E_{\cdot, S_c}}_2 \norm{\mc P^0  (\mc P^0 \mb W^{(-c)} \mc P^0 )^{t} \mc P^0 (\mb W^{(-c)} - \mb W)\mb U^*}  \leq C_8' \tilde \sigma \sqrt{n}\deltaop^{t }\Big[  \tilde\sigma^2 \sqrt{m nr \log d} + \sigma_1^* \tilde \sigma \sqrt{mn} \sqrt{\mu_2 r /p}\Big]
            \label{eq: alpha 3 control}
        \end{align}
        hold for some constant $C_8'$ under $\big(\cap_{k=-1}^{t}\mc E_{H_{t}}\big) \cap \mc E_W \cap \mc E_E$, where we used the facts \eqref{eq: condition 1 in summand control of alpha 2} and \eqref{eq: condition 2 in summand control of alpha 2} in the last line of \eqref{eq: summand control in leave-two-out 2}

        Taking \eqref{eq: inductive hypothesis 1}, \eqref{eq: matrix Bernstein in leave-two-out}, \eqref{eq: summand control in leave-two-out 1}, and \eqref{eq: summand control in leave-two-out 2} collectively yields that
        \begin{align}
            & \alpha_1  
            \leq  C_H \deltaop^{t+1} \tilde \sigma \sqrt{mr \log d}  + C_H \sqrt{m} B \bigg\{ \sum_{h=0}^{t-1}\Big[ C_7' \big(\deltaop^h \big[ C_0' \deltaop\big]^{t-h} \tilde \sigma^2 \sqrt{mnr \log d} \\ 
            & \qquad  + \deltaop^h \big[ C_0' \deltaop\big]^{t-h} \frac{\tilde \sigma}{\bar \sigma} \sqrt{m} (\tilde \sigma^2 n) \sqrt{\frac{\mu_2 r}{p}} + \sigma \sqrt{mr }\sigma_1^*\big)\Big]  + C_8'\deltaop^{t }\Big[  \tilde\sigma^2 \sqrt{m nr \log d} + \sigma_1^* \tilde \sigma \sqrt{mn} \sqrt{\mu_2 r /p}\Big]  \bigg\} \log d \\ 
            & \quad \leq  C_9' \big[ C_0' \deltaop\big]^{t+1} \tilde \sigma \sqrt{m r \log d}
            \label{eq: alpha 1 control}
        \end{align}
        for some constant $C_9'$ under $\big(\cap_{k=-1}^{t}\mc E_{H_{t}}\big) \cap \mc E_W \cap \mc E_E$ invoking the conditions that $\frac{m\sqrt{\mu_1 }B \log d}{\bar \sigma \sqrt{n}} \vee \frac{m\sqrt{ \mu_2 }B \log d}{\bar \sigma \sqrt{p}} \ll 1$. In the end, plugging \eqref{eq: summand control in alpha 2}, \eqref{eq: alpha 3 control}, and \eqref{eq: alpha 1 control} into \eqref{eq: decomposition regarding alpha123 in leave-two-out} leads us to the conclusion that 
        \begin{align}
            & \max_{c\in[l]} \norm{\mb E_{\cdot, S_c}\t \mc P^0  (\mc P^0 \mb W \mc P^0 )^{t } \mc P^0 \mb W \mb U^*}  \leq  C_9' \big[ C_0' \deltaop\big]^{t+1} \tilde \sigma \sqrt{m r \log d}\\ 
            & \qquad \qquad +  C_7' \tilde\sigma \sqrt{n} \Big[\big[ C_0' \deltaop\big]^{t}  \tilde\sigma^2 \sqrt{mnr \log d}+ \big[ C_0' \deltaop\big]^t \frac{\tilde \sigma}{\bar \sigma} \sqrt{m} (\tilde \sigma^2 n) \sqrt{\mu_2 r/ p} +\sigma\sqrt{mr} \sigma_1^* \Big]\\ 
            & \qquad \qquad + C_8' \tilde \sigma \sqrt{n}\deltaop^{t }\Big[  \tilde\sigma^2 \sqrt{m nr \log d} + \sigma_1^* \tilde \sigma \sqrt{mn} \sqrt{\mu_2 r /p}\Big]
            \\ 
            & \qquad \leq C_0'  \big[ C_0' \deltaop\big]^{t+1} \Big[ \tilde \sigma  \sqrt{m r\log d} +  \frac{\tilde \sigma^2}{\bar \sigma } \sqrt{mn}\sqrt{\frac{ \mu_2 r}{p}}\Big] 
        \end{align}
        holds, provided a sufficiently large $C_0'$.

        \emph{Verifying $H_{t}^{(i)}$ in \eqref{eq: inductive hypothesis 3}. }
        Under the event $\mc E_{H^{(i)}_{t}}$, one has 
        \begin{align}
            & \norm{\mb E_{i,\cdot} \mc P_{-i,\cdot}(\mb E)\t \mc P^0 (\mc P^0 \mb W^{(i)} \mc P^0 )^{t} \mc P^0 \mb W^{(i)} \mb U^*}_2\\ 
             \leq & C_H \tilde \sigma \norm{\mc P_{-i,\cdot}(\mb E)\t \mc P^0 (\mc P^0 \mb W^{(i)} \mc P^0 )^{t} \mc P^0 \mb W^{(i)} \mb U^*}\sqrt{r}(t_0 \wedge \sqrt{\log d}) \\ 
            &   + C_H \sqrt{m}B \max_{c\in[l]}\norm{\mc P_{-i,S_c}(\mb E)\t \mc P^0 (\mc P^0 \mb W^{(i)} \mc P^0 )^{t} \mc P^0 \mb W^{(i)} \mb U^*}(t_0^2 \wedge \log d). 
            \label{eq: matrix bernstein for H_t^(i)}
        \end{align}

        First of all, we deduce from the conditions of $\mc E_{W} \cap \mc E_E$ that 
        \begin{align}
            & \norm{\mc P_{-i,\cdot}(\mb E)\t \mc P^0 (\mc P^0 \mb W^{(i)} \mc P^0 )^{t} \mc P^0 \mb W^{(i)} \mb U^*} \leq (C_H \tilde\sigma \sqrt{n} + \sigma \sqrt{p}) \deltaop^{t+1}.
            \label{eq: first term in matrix bernstein for H_t^(i)}
        \end{align}

        In order to control the second term on the RHS of \eqref{eq: matrix bernstein for H_t^(i)}, we employ a telescoping decomposition that 
        \begin{align}
            & \norm{\mc P_{-i,S_c}(\mb E)\t \mc P^0 (\mc P^0 \mb W^{(i)} \mc P^0 )^{t} \mc P^0 \mb W^{(i)} \mb U^*}  \\  
            \leq & \norm{\mc P_{-i,S_c}(\mb E)\t \mc P^0 (\mc P^0 \mb W \mc P^0 )^{t} \mc P^0 \mb W \mb U^*}  \\
            &  + \sum_{l = 0}^{t-1}\norm{\mc P_{-i,S_c}(\mb E)} \norm{ (\mc P^0 \mb W^{(i)}\mc P^0 )^{l} } \norm{\mc P^0 \big( \mb W^{(i)} - \mb W \big) \mc P^0  (\mc P^0 \mb W \mc P^0 )^{t -1 - l} \mc P^0 \mb W \mb U^*} \\ 
            &  +\norm{\mc P_{-i,S_c}(\mb E)}  \norm{ (\mc P^0 \mb W^{(i)}\mc P^0 )^{t} } \norm{\mc P^0 \big( \mb W^{(i)} - \mb W \big) \mb U^*}. 
        \end{align}

        Recall the identity \eqref{eq: identity of P0 (W - Wi) P0} associated with the difference $\mc P^0\mb W - \mb W^{(i)} \mc P^0$. 
        We therefore seek to control $$\norm{\mc P^0 \big( \mb W^{(i)} - \mb W \big) \mc P^0  (\mc P^0 \mb W \mc P^0 )^{t -1 - l} \mc P^0 \mb W \mb U^*} $$ by incorporating the preceding hypotheses with the decomposition below: 
        \begin{align}
            & \norm{\mc P^0 \big( \mb W^{(i)} - \mb W \big) \mc P^0  (\mc P^0 \mb W \mc P^0 )^{t -1 - l} \mc P^0 \mb W \mb U^*} \leq  \underbrace{\norm{ \mb E_{i,\cdot } \mc P_{-i,\cdot }(\mb E)\t \mc P^0  (\mc P^0 \mb W \mc P^0 )^{t -1 - l} \mc P^0 \mb W \mb U^*}}_{\gamma_1} \\
               +& \underbrace{ \big[\norm{\mc P_{-i,\cdot}(\mb E)\mb E_{i,\cdot}\t }_2 + 2\big| \mb E_{i,\cdot} {\mb M^*_{i,\cdot}}\t\big|\big] \norm{ \big((\mc P^0 \mb W \mc P^0 )^{t -1 - l} \mc P^0 \mb W \mb U^*\big)_{i,\cdot}}_2 
               }_{\gamma_2} 
            \label{eq: definition of gamma 12}
        \end{align}

        According to the inductive hypothesis $H_{t-1-l}^{(i)}$, the quantity $\gamma_1$ can be directly bounded by 
        \begin{align}
            & \gamma_1 \leq C_0' \sqrt{\frac{\mu_1 r}{n}}\big[ C_0' \deltaop\big]^{t+1 - l} (t_0 \wedge \sqrt{\log d})
            \label{eq: gamma 1 in leave-two-out}
        \end{align}
        under $\big(\cap_{k=0}^{t+1 - l}\mc E_{H_{k}^{(i)}}\big) \cap \big(\cap_{k=0}^{t+1 -l }\mc E_{H_{t}}\big) \cap \mc E_W \cap \mc E_E$.

        For $\gamma_2$, 
        leveraging the conditions under $\mc E_W \cap \mc E_E$ together with the inductive hypothesis $H^{(i)}_{t-2-l}$ gives that 
        \begin{align}
            & \gamma_2 \leq \Big[\deltaop + 2 C_H \bar \sigma \sigma_1^*\sqrt{r \log d}\sqrt{\mu_1 r /n} \Big] \cdot \Big[C_0' \sqrt{\mu_1 r/n} [C_0 \deltaop]^{t-l} (t_0 \wedge \sqrt{\log d}) \Big]  \\ 
            & \qquad \leq C_{10}' \sqrt{\mu_1 r/n} [C_0 \deltaop]^{t-l+1} (t_0 \wedge \sqrt{\log d})
            \label{eq: gamma 2 in leave-two-out}
        \end{align}
        holds for some constant $C_{10}'$ under the event $\big(\cap_{k=0}^{t+1 - l}\mc E_{H_{k}^{(i)}}\big) \cap \big(\cup_{k=0}^{t-1-l} \mc E_{H_k}\big)\cap \mc E_W \cap \mc E_E$, invoking $\sqrt{\frac{\mu_1 r}{n}} \leq 1$ and $\sqrt{r \log d} \lesssim \sqrt{n}$. 

        Putting \eqref{eq: gamma 1 in leave-two-out} and \eqref{eq: gamma 2 in leave-two-out} into \eqref{eq: definition of gamma 12} yields that 
        \begin{align}
            &  \norm{\mc P^0 \big( \mb W^{(i)} - \mb W \big) \mc P^0  (\mc P^0 \mb W \mc P^0 )^{t -1 - l} \mc P^0 \mb W \mb U^*} \leq C_{11}' \sqrt{\mu_1 r/n} [C_0 \deltaop]^{t-l+1}( t_0 \wedge \sqrt{\log d}) 
            \label{eq: upper bound for intermediate term in telescoping decomposition of leave-one-out}
        \end{align}
        for some constant $C_{11}'$. 

        Next, substituting the bound in the inductive hypothesis $H_t'$, \eqref{eq: upper bound for P0W - Wi U*}, and \eqref{eq: upper bound for intermediate term in telescoping decomposition of leave-one-out} together with the conditions in $\big(\cap_{k=0}^{t+1 - l}\mc E_{H_{k}^{(i)}}\big) \cap \big(\cup_{k=0}^{t-1-l} \mc E_{H_k}\big) \cap \mc E_W\cap \mc E_E$ implies that 
        \begin{align}
            & \norm{\mc P_{-i,S_c}(\mb E)\t \mc P^0 (\mc P^0 \mb W^{(i)} \mc P^0 )^{t} \mc P^0 \mb W^{(i)} \mb U^*} \leq C_0'  \big[ C_0' \deltaop\big]^{t+1} \Big[ \tilde \sigma  \sqrt{m r\log d} +  \frac{\tilde\sigma^2 }{\bar \sigma } \sqrt{mn}\sqrt{\frac{ \mu_2 r}{p}}\Big] \\ 
            & \qquad  + \sum_{l=0}^{t-1} (C_H\tilde\sigma \sqrt{n} ) \deltaop^l \Big[C_{11}' \sqrt{\mu_1 r/n} [C_0' \deltaop]^{t-l+1}(t_0 \wedge  \sqrt{\log d} )\Big] \\ 
            & \qquad  + (C_H\tilde\sigma \sqrt{n} ) \deltaop^t \big( C_3' \deltaop \sqrt{\frac{\mu_1 r}{n}}(t_0 \wedge \sqrt{\log d} ) \big) \leq C_{12}'\big[ C_0' \deltaop\big]^{t+1} \Big[ \tilde \sigma  \sqrt{\mu_1 m r\log d} +  \frac{\tilde\sigma^2 }{\bar \sigma } \sqrt{mn}\sqrt{\frac{ \mu_2 r}{p}}\Big] 
            \label{eq: final upper bound for the second term in matrix bernstein for H_t^(i)}
        \end{align} 
        for some constant $C_{12}'$ under the joint event $(\cap_{k=-1}^{t-1-l} \mc E_{H_k^{(i)}}) \cap (\cap_{k=-1}^{t-1-l} \mc E_{H_k})\cap \mc E_W \cap \mc E_E$. 

        Incorporating \eqref{eq: first term in matrix bernstein for H_t^(i)} with \eqref{eq: final upper bound for the second term in matrix bernstein for H_t^(i)} gives that 
        \begin{align}
            & \norm{\mb E_{i,\cdot} \mc P_{-i,\cdot}(\mb E)\t \mc P^0 (\mc P^0 \mb W^{(i)} \mc P^0 )^{t} \mc P^0 \mb W^{(i)} \mb U^*}_2\\ 
             \leq & C_H \tilde \sigma \Big[(C_H \tilde\sigma \sqrt{n}+ \sigma \sqrt{p}) \deltaop^{t+1} \Big]\sqrt{r}(t_0 \wedge \sqrt{\log d}) \\ 
            &   + C_H \sqrt{m}B \Big[C_{12}'\big[ C_0' \deltaop\big]^{t+1} \Big( \tilde \sigma  \sqrt{\mu_1 m r\log d} +  \frac{\tilde \sigma^2 }{\bar \sigma } \sqrt{mn}\sqrt{\frac{ \mu_2 r}{p}}\Big) \Big](t_0^2 \wedge \log d) \\ 
            \leq & \frac{C_0'}{3} \sqrt{\frac{\mu_1 r}{n}}[C_0'\deltaop]^{t+2} (t_0 \wedge \sqrt{\log d}) 
            \label{eq: final upper bound for the first statement of H_t^(i)}
        \end{align}
        for some sufficiently large $C_0'$ under the event $(\cap_{k=-1}^{t-1-l} \mc E_{H_k^{(i)}}) \cap (\cap_{k=-1}^{t-1-l} \mc E_{H_k})\cap \mc E_W \cap \mc E_E$, where we used the conditions $m\sqrt{\mu_1} \log d B \lesssim \sigma \sqrt{n}$ and $ \frac{\tilde \sigma }{\bar \sigma} m B \log d \sqrt{\mu_2  / p} \lesssim  \tilde\sigma$. 

        We are left with verifying the second statement in the hypothesis $H_t^{(i)}$. We resort to the telescoping decomposition again to deduce that 
        \begin{align}
            & \norm{\mb E_{i,\cdot} \mc P_{-i,\cdot}(\mb E)\t \mc P^0 (\mc P^0 \mb W \mc P^0 )^t \mc P^0 \mb W \mb U^*}_2  \\ 
            \leq & \norm{\mb E_{i,\cdot} \mc P_{-i,\cdot}(\mb E)\t \mc P^0 (\mc P^0 \mb W^{(i)} \mc P^0 )^{t} \mc P^0 \mb W^{(i)} \mb U^*}_2 \\ 
            & \qquad + \sum_{l = 0}^{t-1}\norm{\mb E_{i,\cdot} \mc P_{-i,S_c}(\mb E)} \norm{ (\mc P^0 \mb W^{(i)}\mc P^0 )^{l} } \norm{\mc P^0 \big( \mb W^{(i)} - \mb W \big) \mc P^0  (\mc P^0 \mb W \mc P^0 )^{t -1 - l} \mc P^0 \mb W \mb U^*} \\ 
            & \qquad   +\norm{\mb E_{i,\cdot} \mc P_{-i,S_c}(\mb E)}  \norm{ (\mc P^0 \mb W^{(i)}\mc P^0 )^{t} } \norm{\mc P^0 \big( \mb W^{(i)} - \mb W \big) \mb U^*}. 
            \label{eq: telescoping decomposition for the second statement of H_t^(i)}
        \end{align}

        Similar to the derivation of \eqref{eq: final upper bound for the second term in matrix bernstein for H_t^(i)}, we plug \eqref{eq: final upper bound for the first statement of H_t^(i)}, \eqref{eq: upper bound for intermediate term in telescoping decomposition of leave-one-out}, \eqref{eq: upper bound for P0W - Wi U*}, and the conditions in $\mc E_W$ collectively into \eqref{eq: telescoping decomposition for the second statement of H_t^(i)} to arrive at the conclusion that 
        \begin{align}
            & \norm{\mb E_{i,\cdot} \mc P_{-i,\cdot}(\mb E)\t \mc P^0 (\mc P^0 \mb W \mc P^0 )^t \mc P^0 \mb W \mb U^*}_2 \leq \frac{C_0}{3} \sqrt{\frac{\mu_1 r}{n}}[C_0'\deltaop]^{t+2} (t_0 \wedge \sqrt{\log d}) \\ 
            & \qquad + \sum_{l=0}^{t-1}  \deltaop^{l+1} \Big[C_{11}' \sqrt{\mu_1 r/n} [C_0 \deltaop]^{t-l+1}( t_0 \wedge \sqrt{\log d}) \Big] + \deltaop^{t+1} \big( C_3' \deltaop \sqrt{\frac{\mu_1 r}{n}}(t_0 \wedge \sqrt{\log d}) \big) \\ 
            \leq & \frac{2}{3}C_0' \sqrt{\frac{\mu_1 r}{n}}\big[ C_0' \deltaop\big]^{t+2}(t_0 \wedge \sqrt{\log d}),
        \end{align}
        holds under $(\cap_{k=-1}^{t} \mc E_{H_k^{(i)}}) \cap (\cap_{k=-1}^{t} \mc E_{H_k})\cap \mc E_W \cap \mc E_E$, given a sufficiently large $C_0'$. 

        Finally, we notice that 
                \begin{align}
            & \norm{ \big((\mc P^0 \mb W \mc P^0 )^{t } \mc P^0 \mb W \mb U^*\big)_{i,\cdot}}_2 \leq \norm{\mb E_{i,\cdot}\mc P_{-i,\cdot}(\mb E)\t (\mc P^0 \mb W \mc P^0 )^{t -1 } \mc P^0 \mb W \mb U^*}  \\ 
            & \qquad  + \big( 2|\mb E_{i,\cdot} {\mb M^*_{i,\cdot}}\t| + \bignorm{\diag(\mb M^*{\mb M^*}\t)}  \big)  \norm{\big((\mc P^0 \mb W \mc P^0 )^{t -1 } \mc P^0 \mb W \mb U^*\big)_{i,\cdot}} + 
             \norm{\mb U^*}\ti \norm{\mb W}^{t+1}
             \\ 
            \leq & \frac{2}{3}C_0' \sqrt{\frac{\mu_1 r}{n}}\big[ C_0' \deltaop\big]^{t+1} (t_0 \wedge \sqrt{\log d}) \\ 
            & + \Big[2 C_H \bar \sigma \sigma_1^* \sqrt{\mu_1 r /n} + \frac{\mu_1 r}{n }{\sigma_1^*}^2  \Big] \Big[C_0' \sqrt{\mu_1 r/n} [C_0 \deltaop]^{t} (t_0 \wedge \sqrt{\log d})  \Big] + \sqrt{\frac{\mu_1 r}{n}} \deltaop^{t+1} \\ 
            \leq & C_0' \sqrt{\frac{\mu_1 r}{n}}\big[ C_0' \deltaop\big]^{t +1 } (t_0 \wedge \sqrt{\log d}). 
        \end{align}
        under $(\cap_{k=-1}^{t} \mc E_{H_k^{(i)}}) \cap (\cap_{k=-1}^{t} \mc E_{H_k})\cap \mc E_W \cap \mc E_E$, provided a sufficiently large $C_0'$.

\end{proof}

\section{Proof of the Upper Bounds for COPO}\label{sec-proof of upper bounds}

\label{sec: proof of upper bound for algorithm}
In this section, we will present the proofs of Theorem~\ref{theorem: upper bound for algorithm (gaussian)} and Theorem~\ref{theorem: upper bound for algorithm (bounded)} collectively, by establishing the convergence rate for the iterates in Algorithm~\ref{algorithm: CPSC}. 
For clarity, we summarize below the shorthand notations and the key quantities that are used throughout: For each $k \in [K]$, we denote by $\mathbf c_k^*$ the common value of the rows ${\mathbf U^*_{i,\cdot}}^\top$ corresponding to all the indices $i$ such that $z_i^* = k$, since all rows within the same cluster take identical values in $\mathbf U^*$. Also recall that $\mb w_k^* \coloneqq \bo \Lambda^* \mb c_k^* = {\mb V^*}\t \bo \theta_k^*$. Then we define $\tilde{\mb c}_k^* \coloneqq \mb c_k^* + \bias_k$ and $\tilde{\mb w}_k^* \coloneqq \bo \Lambda^* \tilde{\mb c}_k^*$, where the bias term $\bias$ is defined in \eqref{eq: definition of bias}. 
We further introduce: 
\begin{align}
    & \bar \sigma_{\mathsf{cov}} \coloneqq \max_{k\in[K]} \norm{\mb S_k^*},~ \underline\sigma_{\mathsf{cov}} \coloneqq \min_{k\in[K]} \sigma_K^*(\mb S_k^*). 
\end{align}

Since we aim to establish the same upper bound for the clustering error under two different settings (Gaussian mixtures with general dependence and general mixtures with local dependence), in what follows the arguments will be presented in a unified manner, and we will specify the difference between the two settings when necessary. Moreover, we will treat the second case in Assumption~\ref{assumption: bounded noise}.\ref{item: bounded noise assumption 2} as the first case (bounded r.v.'s) since $\bb P[\mb E \neq \mb E']$ does not alter the exceptional probability in all of the following arguments. 

\subsection{Key Steps of Error Control}
Recap that to quantify the clustering error combined with the adjusted distances, we recall the variant of $h(\mb z, \mb z^*)$ as
\eq{
 l(\mb z, \mb z^*) \coloneqq \sum_{i \in [n]} \omega_{z_i,\pi(z_i^*)} \ind\{z_i\neq \pi^*(z_i^*)\}, \text{ where }\pi
^*\coloneqq \argmin_{\pi \in \Pi_K} \sum_{i\in[n]}\ind\{z_i \neq \pi(z_i^*)\}}
Note that its form is similar to the ones in \cite{chen2024optimal,gao2022iterative}, but a slight difference lies in its weighting of the misspecification error by the adjusted distances related to the groundtruth projected covariance matrices.  
Our goal is to show, with high probability, that the clustering error contracts across iterations of the algorithm: 
\begin{equation}\label{eq: proof sketch - iterative decay}
    l(\hat{\mathbf z}^{(s)}, \mathbf z^*) 
     \le 
    \xi_{\mathsf{oracle}}(\delta) + \xi_{\mathsf{approx}}(\delta)
     +  \frac{1}{4}\,
    l(\hat{\mathbf z}^{(s-1)}, \mathbf z^*),
\end{equation}
where $\xi_{\mathsf{oracle}}(\delta)$ and $\xi_{\mathsf{approx}}(\delta)$ will be defined later.

\paragraph*{Step 1: Error Decomposition via a One-Step Analysis. } 
To begin with, a simple calculation tells us that a geometric decay of the alternative sequence $\{l(\hat{\mb z}^{(t)}, \mb z^*)\}_{t=0}^{T}$ can exhibit a geometric decay through the relation \eqref{eq: proof sketch - iterative decay}. However, each step of the sequence of cluster labels $\{\hat{\mb z}^{(t)}\}$ is interdependent. 
To address this dependency, we employ a one-step analysis to justify that there exists a high-probability event under which the decay relation \eq{
l(\hat{\mb z}, \mb z^*) \leq \xi_{\oracle}(\delta) + \xi_{\mathsf{approx}}(\delta) + \frac{1}{4}l(\mb z, \mb z^*)\label{eq: decay relation 2}
}
uniformly holds for all possible $\mb z$ with a small enough $l(\mb z, \mb z^*)$; here, the updated estimate $\hat{\mb z}$ is computed based on the last-step estimate $\mb z$ following Algorithm~\ref{algorithm: CPSC}'s mechanism.

 With this idea in mind, we set out to look at a break-down form of $l(\hat{\mb z}, \mb z^*)$ given the last-step estimate $\mb z$ with $l(\mb z, \mb z^*) \leq C\frac{\beta K(\log d)^4}{n}$ for some sufficiently small constant $C$. 
Without loss of generality, we assume that $ \argmin_{\pi \in \mathsf{Perm}[K]}\sum_{i\in[n]} \ind{\{z_i^* \neq  \pi( z_i )\}} = \mathbf{id} = \pi^*$. Then
\begin{align}
      & l(\hat{\mb z}, \mb z^*)
     \leq \sum_{i\in[n]}\sum_{k\in[K]\backslash \{z_i^*\}} \omega_{k,z_i^*} \ind{\{\hat z_i =  k\}}\\ 
	\leq  & \sum_{i\in[n]}\sum_{k\in[K]\backslash \{z_i^*\}} \omega_{k,z_i^*}   \ind{\big\{ \ip{\big(\mb U_i - \hat{\mb c}_k(\mb z)\big), \hat{\bo \Omega}_k(\mb z)^{-1} \big(\mb U_i - \hat{\mb c}_k(\mb z)\big)}  \leq \ip{\big(\mb U_i - \hat{\mb c}_{z_i^*}(\mb z)\big), \hat{\bo \Omega}_{z_i^*}(\mb z)^{-1} \big(\mb U_i - \hat{\mb c}_{z_i^*}(\mb z)\big)}
	\big\}}.\label{eq: upper bound on l(hat z, z^*)}
\end{align}

Define that $\mb O \coloneqq \mb L \mb R\t$, where the SVD of $\mb U\t \mb U^*$ is written as $\mb U\t \mb U^* = \mb L \mathrm{diag}(\sigma_i(\mb U\t \mb U^*)) \mb R\t$. For the misspecifying event for the $i$-th sample
$$\ind{
\left\{
\ip{\big(\mb U_i - \hat{\mb c}_k(\mb z)\big), \hat{\bo \Omega}_k(\mb z)^{-1} \big(\mb U_i - \hat{\mb c}_k(\mb z)\big)} 
\leq
\ip{\big(\mb U_i - \hat{\mb c}_{z_i^*}(\mb z)\big), \hat{\bo \Omega}_{z_i^*}(\mb z)^{-1} \big(\mb U_i - \hat{\mb c}_{z_i^*}(\mb z)\big)}
	\right\}
    }
    $$ in \eqref{eq: upper bound on l(hat z, z^*)}, we take a difference between the left hand side and the right-hand side and decompose it as follows:   
\begin{align}
    & \ip{\big(\mb U_i - \hat{\mb c}_k(\mb z)\big), \hat{\bo \Omega}_k(\mb z)^{-1} \big(\mb U_i - \hat{\mb c}_k(\mb z)\big)} - 
\ip{\big(\mb U_i - \hat{\mb c}_{z_i^*}(\mb z)\big), \hat{\bo \Omega}_{z_i^*}(\mb z)^{-1} \big(\mb U_i - \hat{\mb c}_{z_i^*}(\mb z)\big)}\\ 
    =&\ip{\bo \Lambda^* \mb O\t \big(\mb U_i - \hat{\mb c}_k(\mb z)\big), \big( \bo \Lambda^* \mb O\t \hat{\bo \Omega}_k(\mb z) \mb O \bo \Lambda^* \big)^{-1}\bo \Lambda^* \mb O\t \big(\mb U_i - \hat{\mb c}_k(\mb z)\big)} 
\\
& -  \ip{ \bo \Lambda^*\mb O\t  \big(\mb U_i - \hat{\mb c}_{z_i^*}(\mb z)\big),\big(\bo \Lambda^* \mb O\t  \hat{\bo \Omega}_{z_i^*}(\mb z) \mb O\bo \Lambda^*\big)^{-1} \bo \Lambda^* \mb O\t \big(\mb U_i - \hat{\mb c}_{z_i^*}(\mb z)\big)}\\ 
= &  \underbrace{\zeta_{\oracle, i}(k)}_{\text{the oracle error} } + \underbrace{\zeta_{\mathsf{approx, i}}(k,\mb z)}_{\text{linear approximation error}} -\underbrace{ \big( F_i(k, \mb z) + G_i(k,\mb z) + H_i(k, \mb z) \big) }_{\text{the misspecification effect of $\mb z$}} \label{eq: upper bound on l(hat z, z^*) 2}
\end{align}
holds for every $k\neq z_i^*, k\in[K]$, where $\zeta_{\oracle, i}(k)$, $\zeta_{\mathsf{approx}, i}(k, \mb z)$, $F_i(k, \mb z)$, $G_i(k,\mb z)$, and $H_i(k,\mb z)$ are defined as 
\begin{align} 
& \zeta_{\oracle,i}(k) \coloneqq \ip{\bo \Lambda^*\mathfrak L_i,  (\bo \Lambda^* \mb O\t \hat{\bo \Omega}_k(\mb z^*) \mb O \bo \Lambda^*)^{-1} \bo \Lambda^* \mb O\t \big( \hat {\mb c}_{z_i^*}(\mb z^*) -  \hat{\mb c}_k(\mb z^*) \big)} \\ 
& \qquad + \frac{1}{2}\ip{\bo \Lambda^*\mathfrak L_i, \big[(\bo \Lambda^* \mb O\t \hat{\bo\Omega}_k(\mb z^*) \mb O \bo \Lambda^*)^{-1} - (\bo \Lambda^* \mb O\t \hat{\bo\Omega}_{z_i^*}(\mb z^*) \mb O \bo \Lambda^*)^{-1}\big]\bo \Lambda^* \mathfrak L_i}  \\ 
&\qquad + \frac{1}{2}\ip{\bo \Lambda^*\big(\tilde{\mb c}_{z_i^*}^*- \mb O\t\hat{\mb c}_k(\mb z^*) \big),  (\bo \Lambda^* \mb O\t \hat{\bo \Omega}_k(\mb z^*) \mb O \bo \Lambda^*)^{-1} \bo \Lambda^*\big(\tilde{\mb c}_{z_i^*}^*- \mb O\t\hat{\mb c}_k(\mb z^*) \big)} \\ 
& \qquad - \frac{1}{2}\ip{\bo \Lambda^*\big(\tilde{\mb c}_{z_i^*}^*-\mb O\t \hat{\mb c}_{z_i^*}(\mb z^*) \big), (\bo \Lambda^* \mb O\t \hat{\bo \Omega}_{z_i^*}(\mb z^*) \mb O \bo \Lambda^*)^{-1} \bo \Lambda^*\big(\tilde{\mb c}_{z_i^*}^*-\mb O\t \hat{\mb c}_{z_i^*}(\mb z^*) \big)}, \label{eq: zeta_orackle,i(k)}
\\
& \zeta_{\mathsf{approx},i}(k,\mb z) \coloneqq \frac12 \ip{\bo \Lambda^* \mb O\t \big(\mb U_i - \hat{\mb c}_k(\mb z)\big), \big( \bo \Lambda^* \mb O\t \hat{\bo \Omega}_k(\mb z) \mb O \bo \Lambda^* \big)^{-1}\bo \Lambda^* \mb O\t \big(\mb U_i - \hat{\mb c}_k(\mb z)\big)} 
\\
& \qquad -\frac12 \ip{ \bo \Lambda^*\mb O\t  \big(\mb U_i - \hat{\mb c}_{z_i^*}(\mb z)\big),\big(\bo \Lambda^* \mb O\t  \hat{\bo \Omega}_{z_i^*}(\mb z) \mb O\bo \Lambda^*\big)^{-1}  \bo \Lambda^* \mb O\t \big(\mb U_i - \hat{\mb c}_{z_i^*}(\mb z)\big)} \\ 
& \qquad -\frac12 \ip{\bo \Lambda^* \big(\mathfrak L_i + (\tilde{\mb c}_{z_i^*}^* - \mb O\t  \hat{\mb c}_k(\mb z))\big), \big( \bo \Lambda^* \mb O\t \hat{\bo \Omega}_k(\mb z) \mb O \bo \Lambda^* \big)^{-1} \bo \Lambda^* \big( \mathfrak L_i + (\tilde{\mb c}_{z_i^*}^* - \mb O\t  \hat{\mb c}_k(\mb z) )\big)} 
\\
& \qquad + \frac12\ip{\bo \Lambda^* \big(\mathfrak L_i + (\tilde{\mb c}_{z_i^*}^* -  \mb O\t  \hat{\mb c}_{z_i^*}(\mb z))\big),\big(\bo \Lambda^* \mb O\t  \hat{\bo \Omega}_{z_i^*}(\mb z) \mb O\bo \Lambda^*\big)^{-1} \bo \Lambda^* \big( \mathfrak L_i + (\tilde{\mb c}_{z_i^*}^* -  \mb O\t  \hat{\mb c}_{z_i^*}(\mb z))\big)},
\label{eq: zeta_approx,i(k)}
\\ 
        & F_i(k,\mb z)
     \coloneqq -\ip{
        \bo \Lambda^* \mathfrak L_i , (\bo \Lambda^* \mb O\t \hat{\bo \Omega}_k(\mb z) \mb O \bo \Lambda^*)^{-1} \bo \Lambda^* \mb O\t \big(\hat{\mb c}_k(\mb z^*) - \hat{\mb c}_k(\mb z)\big)} \\ 
        & \qquad + \ip{
            \bo \Lambda^* \mathfrak L_i , (\bo \Lambda^* \mb O\t \hat{\bo \Omega}_{z_i^*}(\mb z) \mb O \bo \Lambda^*)^{-1} \bo \Lambda^*\big( \hat{\mb c}_{z_i^*}(\mb z^*) - \hat{\mb c}_{z_i^*}(\mb z) \big)
    } 
    \\
     &\qquad  - \ip{
       \bo \Lambda^* \mathfrak L_i,  \big[(\bo \Lambda^* \mb O\t \hat{\bo\Omega}_k(\mb z) \mb O \bo \Lambda^*)^{-1} - (\bo \Lambda^* \mb O\t \hat{\bo\Omega}_k(\mb z^*) \mb O \bo \Lambda^*)^{-1}\big] \bo \Lambda^*\big(\tilde{\mb c}_{z_i^*}^*- \mb O\t\hat{\mb c}_k(\mb z^*) \big)}  \\ 
       & \qquad + \ip{ \bo \Lambda^* \mathfrak L_i 
        ,\big[(\bo \Lambda^* \mb O\t \hat{\bo\Omega}_{z_i^*}(\mb z) \mb O \bo \Lambda^*)^{-1} - (\bo \Lambda^* \mb O\t \hat{\bo\Omega}_{z_i^*}(\mb z^*) \mb O \bo \Lambda^*)^{-1}\big] \bo \Lambda^*\big(\tilde{\mb c}_{z_i^*}^*-\mb O\t \hat{\mb c}_{z_i^*}(\mb z^*) \big) },\label{eq: F term}
        \\         
     &
     G_i(k,\mb z) \coloneqq 
     -\frac{1}{2}\ip{\bo \Lambda^* \mathfrak L_i, \big( (\bo \Lambda^* \mb O\t\hat{\bo \Omega}_k(\mb z) \mb O \bo \Lambda^*)^{-1} - (\bo \Lambda^* \mb O\t\hat{\bo \Omega}_k(\mb z^*) \mb O \bo \Lambda^*)^{-1}\big) \bo \Lambda^*\mathfrak L_i} \\ 
     & \qquad + \frac{1}{2}\ip{\bo \Lambda^* \mathfrak L_i, \big((\bo \Lambda^* \mb O\t \hat{\bo \Omega}_{z_i^*}(\mb z) \mb O \bo \Lambda^*)^{-1} -(\bo \Lambda^* \mb O\t \hat{\bo \Omega}_{z_i^*}(\mb z^*) \mb O \bo \Lambda^*)^{-1}\big) \bo \Lambda^* \mathfrak L_i} \label{eq: G term},
    \\ 
    & H_i(k,\mb z) \coloneqq -
    \frac{1}{2}\ip{\bo \Lambda^*\big(\tilde{\mb c}_{z_i^*}^*- \mb O\t \hat{\mb c}_{z_i^*}(\mb z) \big), (\bo \Lambda^* \mb O\t \hat{\bo \Omega}_{z_i^*}(\mb z) \mb O \bo \Lambda^*)^{-1} \bo \Lambda^*\big(\tilde{\mb c}_{z_i^*}^*- \mb O\t \hat{\mb c}_{z_i^*}(\mb z) \big) } \\ 
    & \qquad + \frac{1}{2} \ip{\bo \Lambda^*\big(\tilde{\mb c}_{z_i^*}^*- \mb O\t \hat{\mb c}_{z_i^*}(\mb z^*) \big), (\bo \Lambda^* \mb O\t \hat{\bo \Omega}_{z_i^*}(\mb z^*) \mb O \bo \Lambda^*)^{-1} \bo \Lambda^*\big(\tilde{\mb c}_{z_i^*}^*- \mb O\t \hat{\mb c}_{z_i^*}(\mb z^*) \big) }
    \\
     & \qquad + \frac{1}{2}\ip{\bo \Lambda^*\big(\tilde{\mb c}_{z_i^*}^*- \hat{\mb c}_k(\mb z) \big), (\bo \Lambda^* \mb O\t \hat{\bo \Omega}_k(\mb z) \mb O \bo \Lambda^*)^{-1} \bo \Lambda^*\big(\tilde{\mb c}_{z_i^*}^*- \hat{\mb c}_k(\mb z) \big) } \\ 
    & \qquad - \ip{\bo \Lambda^*\big(\tilde{\mb c}_{z_i^*}^*- \hat{\mb c}_k(\mb z^*) \big), (\bo \Lambda^* \mb O\t \hat{\bo \Omega}_k(\mb z^*) \mb O \bo \Lambda^*)^{-1} \bo \Lambda^*\big(\tilde{\mb c}_{z_i^*}^*- \hat{\mb c}_k(\mb z^*) \big) }. \label{eq: H term}
\end{align}
The quantity $\zeta_{\mathsf{oracle},i}(k)$ characterizes the intrinsic difficulty of clustering under oracle label information. The term $\zeta_{\mathsf{approx},i}(k,\mb z)$ measures the approximation error incurred by replacing $\mb O\t \mb U_i - \tilde{\mb c}^*_{z_i^*}$ with $\mk L_i$. The remaining components $F_i(k,\mb z)$, $G_i(k,\mb z)$, and $H_i(k,\mb z)$ reflect the coupled misspecification effects carried over from the previous iteration: $F_i$ and $G_i$ correspond to the linear and quadratic interactions with the $i$-th noise vector, while $H_i$ captures the error in estimating the centers. Recall that $\mathfrak L_i$ represents the linear approximation ${\mb V^*}\t \mb E_i + {\bo \Lambda^*}^{-1}{\mb U^*}\t \mc H_{-i,\cdot}(\mb E) \mb E_i$ of $\big(\mb U\mb U\t  \mb U^* - \mb U^*\big)_{i,\cdot}\t $ for $i \in[n]$.

With the decomposition \eqref{eq: upper bound on l(hat z, z^*) 2} in place, we are able to separately parse the one-step clustering barrier: 
\begin{align}
& l(\hat{\mb z}, \mb z^*)\leq  \sum_{i\in[n]} \sum_{k\in[K] \backslash\{z_i^*\}} \omega_{k, z_i^*}\ind{\big\{\zeta_{\oracle, i}(k) \leq \frac{\delta}{2}\omega_{k, z_i^*}
	 \big\}}  + \sum_{i\in[n]} \sum_{k\in[K] \backslash\{z_i^*\}} \omega_{k, z_i^*} \ind{\big\{
	 \zeta_{\mathsf{approx},i}(k, \mb z) \geq \frac{\delta}{8}\omega_{k, z_i^*}
	  \big\} } \\ 
	  &+\sum_{i\in[n]} \max_{k\in[K] \backslash\{z_i^*\}} \omega_{k, z_i^*}\Big[  \ind{\big\{
	 F_i(k, \mb z) \geq \frac{\delta}{8}\omega_{k, z_i^*}
	  \big\} }  + \ind{\big\{
	 G_i(k, \mb z) \geq \frac{\delta}{8}\omega_{k, z_i^*}
	  \big\} }   + \ind{\big\{
	 H_i(k, \mb z) \geq \frac{\delta}{8}\omega_{k, z_i^*}
	  \big\} } \Big].\label{eq: upper bound on l(hat z, z^*) 3}
\end{align}
For further simplicity, we write the first two terms on the right-hand side of \eqref{eq: upper bound on l(hat z, z^*) 3} as
\begin{align}
    &\xi_{\oracle}(\delta) \coloneqq \sum_{i\in[n]} \sum_{k\in[K] \backslash\{z_i^*\}}\omega_{k, z_i^*}\ind{\big\{\zeta_{\oracle, i}(k) \leq \frac{\delta}{2}\omega_{k, z_i^*}
	 \big\}}, \\ 
     & \xi_{\mathsf{approx}}(\delta) \coloneqq \sum_{i\in[n]} \sum_{k\in[K] \backslash\{z_i^*\}} \omega_{k, z_i^*} \ind{\big\{
	 \zeta_{\mathsf{approx},i}(k,\mb z) \geq \frac{\delta}{8}\omega_{k, z_i^*}
	  \big\} },
\end{align}
where $\zeta_{\oracle,i}(k)$ and $\zeta_{\mathsf{approx},i}(k,\mb z)$ were defined in \eqref{eq: zeta_orackle,i(k)} and \eqref{eq: zeta_approx,i(k)}, respectively. 
In what follows, we justify each component separately: 
\begin{enumerate} 
    \item The expectations of the first and second terms, related to $\zeta_{\oracle,i}(k)$ and $\zeta_{\mathsf{approx},i}(k,\mb z)$, respectively, are of the order $\exp\big(-(1 + o(1))\frac{{\mathsf{SNR}}^2}{2}\big)$; this will then be the dominant magnitude in the upper bounds.
    \item The remaining terms are controlled by $l(\mb z,\mb z^*)$ uniformly over all possible $\mb z$ with high probability. This control finally leads to the second term on the RHS of \eqref{eq: decay relation 2}. 
\end{enumerate}

\paragraph*{Step 2: Error Analysis Given True $\mb z^*$. } 

We have the following lemma, whose proof is presented in Section~\ref{sec: proof of lemma for tilde zeta oracle}.
\begin{lemma}
\label{lemma: Bound for the first three terms in zeta_oracle}
Instate either the assumptions in Theorem~\ref{theorem: upper bound for algorithm (gaussian)} or the assumptions in Theorem~\ref{theorem: upper bound for algorithm (bounded)}. Then it holds for an arbitrary vanishing sequence $\delta$ with $\kappa_{\mathsf{cov}}^4 \delta = o(1)$ and some constant $c$ that 
$$
     \bb P\big[\zeta_{\oracle, i}(k) \leq  \delta \omega_{k,z_i^*} \big] \lesssim \exp\Big(-(1 - \tilde \delta) \frac{\snr^2}{2}\Big)  \vee O(d^{-c}),
$$
where $\tilde\delta$ is also a vanishing sequence as $n \to \infty$.
\end{lemma}

We now control $\xi_{\oracle}$. For any vanishing sequence $\delta \ge 0$ satisfying $\kappa_{\mathsf{cov}}^4 \delta = o(1)$,
\begin{align}
    & \bb E\big[\xi_{\oracle} \big]\lesssim  n (K-1) \max_{i \in [n], k\in[K]\backslash \{z_i^*\}}\bb E\Big[ \omega_{k, z_i^*} \ind{\big\{\zeta_{\oracle, i}(k) \leq \frac{\delta}{2} \omega_{k,z_i^*} \}}\Big]
    \\ 
    \lesssim & n K\nu^2\underline \omega \max_{i \in [n], k\in[K]\backslash \{z_i^*\}} \bb P\Big[\zeta_{\oracle, i}(k) \leq \frac{\delta}{2}\omega_{k,z_i^*} \Big]
    \lesssim   n K\nu^2 \underline \omega \Big[ \exp(-(1 - \tilde\delta) \frac{{\mathsf{SNR}}^2}{2}) + O(d^{-c})\Big] \label{eq: upper bound for expectation of xi_oracle},
\end{align}
where in the penultimate inequality we used the fact that $\omega_{a,b} \lesssim \nu^2 \underline \omega $ for every $a \neq b \in [K]$. 

Moreover, the following lemma bounds the expectation of $\xi_{\mathsf{approx}}$, thereby justifying the accuracy of the linear approximation $\mk L_i$. The proof is provided in Section~\ref{subsubsec: proof of xi_approx,i}. 
\begin{lemma}
    Instate either the assumptions of Theorem~\ref{theorem: upper bound for algorithm (gaussian)} or those of Theorem~\ref{theorem: upper bound for algorithm (bounded)}. Then it hold for $\delta = C\frac{\beta^{\frac12} \kappa K \nu \kappa_{\mathsf{cov}}^2}{\underline\omega^{\frac12}}$ with some constant $C$ that 
    \begin{align}
        &  \bb E\Big[ \xi_{\mathsf{approx}}\big(\delta \big) \Big]  \lesssim nK \nu \underline\omega \Big[ \exp(-(1 + o(1))\frac{\snr^2}{2}) \vee O(d^{-c})\Big]. 
    \end{align}
    \label{lemma: Bound for xi_approx,i}
\end{lemma}

\paragraph*{Step 3: Error Analysis Regarding $\mb z$. } 
In order to decouple the interdependence between $\{\hat{\mb z}^{(t)}\}_{t=1}^{T}$ and $\hat{\mb z}^{(0)}$, we adopt a one-step analysis on the alternative clustering error quantities $ l(\mb z, \mb z^*)$, given any last-step estimate $\mb z$ whose alternative misspecification error falls in an appropriate range.

\begin{lemma}\label{lemma: Bound for F_i}
	Instate either the assumptions of Theorem~\ref{theorem: upper bound for algorithm (gaussian)} or those of Theorem~\ref{theorem: upper bound for algorithm (bounded)}. With probability at least $1 -  O(d^{-c}) \vee e^{- (1 + o(1)) \frac{\snr^2}{2}}$ it holds for some sufficiently small $c_1$ that 
	\longeq{
&\max_{\mb z: l(\mb z, \mb z^*) \leq c_1\frac{n}{\beta K(\log d)^4}}\frac{\sum_{i\in[n]}\max_{b\in[K]\backslash \{z_i^*\}} \frac{F_i( b, \mb z)^2}{\omega_{z_i^*,b } }}{l(\mb z, \mb z^*)}\lesssim \frac{1}{\underline\omega}\nu \kappa^4 \kappa_{\mathrm{cov}}^8 \beta^6 K^6 = o(\frac{1}{\kappa_{\mathsf{cov}}^4}) .
	}
\end{lemma}
 \begin{lemma}
 \label{lemma: Bound for G_i}
    Instate either the assumptions of Theorem~\ref{theorem: upper bound for algorithm (gaussian)} or those of Theorem~\ref{theorem: upper bound for algorithm (bounded)}. With probability at least $1 - O(d^{-c}) \vee e^{- (1 + o(1)) \frac{\snr^2}{2}}$ it holds for some sufficiently small $c_1$ that 
    \longeq{
    & \max_{\mb z:l(\mb z, \mb z^*) \leq c_1\frac{n}{\beta K(\log d)^4}}\frac{\sum_{i\in [n]} \max_{b\in[K], b\neq z_i^*} \frac{G_i(b, \mb z)^4}{\omega_{z_i^*, b}^3}}{l(\mb z,\mb z^*)} \lesssim \frac{\kappa^8 \kappa_{\mathsf{cov}}^{16} \beta^{11} K^{13} }{\underline \omega^5}  = o(\frac{1}{\kappa_{\mathsf{cov}}^4}) . 
    }
\end{lemma}

\begin{lemma}\label{lemma: Bound for H_i}
	Instate either the assumptions of Theorem~\ref{theorem: upper bound for algorithm (gaussian)} or those of Theorem~\ref{theorem: upper bound for algorithm (bounded)}. With probability at least $1 - O(d^{-c}) \vee e^{- (1 + o(1)) \frac{\snr^2}{2}}$ it holds for some sufficiently small $c_1$  that 
	\longeq{
	&\max_{\mb z: l(\mb z, \mb z^*) \leq c_1\frac{n}{\beta K(\log d)^4}} \max_{i\in[n]} \max_{b\in[K], b\neq z_i^*}H_i(b, \mb z)\lesssim  \frac{\kappa \kappa_{\mathrm{cov}}^2 \beta^2 K^2  }{\sqrt{\underline\omega} } + \frac{\kappa^2 \kappa_{\mathsf{cov}}^4 \beta^4 K^4 }{\underline\omega^{\frac32}} + \nu \kappa_{\mathrm{cov}}^2 \beta^{\frac12} K^{\frac12}  + \nu^2 \kappa_{\mathrm{cov}}^4 \kappa^2\beta^{\frac52} K^2 \underline \omega^{\frac12}   = o( \frac{\underline \omega}{\kappa_{\mathsf{cov}}^4}).
	}
\end{lemma}
The proofs of Lemmas~\ref{lemma: Bound for F_i},~\ref{lemma: Bound for G_i}, and~\ref{lemma: Bound for H_i} is postponed to Section~\ref{subsection: Misspecification Effect Analysis}, while their direct implication is that the third term in \eqref{eq: upper bound on l(hat z, z^*) 3} could be separately bounded as follows for some vanishing sequence $\delta$ satisfying $\kappa_{\mathsf{cov}}^4 \delta = o(1)$ and every large enough $n$: 
\begin{align}
    &  \sum_{i\in[n]} \max_{k\in[K] \backslash\{z_i^*\}} \omega_{k,z_i^*} \ind{\big\{
    F_i(k, \mb z) \geq \frac{\delta}{8} \omega_{k,z_i^*} \big\}}\leq \frac{1}{8}l(\mb z, \mb z^*), \\
    &\sum_{i\in[n]} \max_{k\in[K] \backslash\{z_i^*\}} \omega_{k,z_i^*} \ind{\big\{
        G_i(k, \mb z) \geq \frac{\delta}{8}\omega_{k,z_i^*}
         \big\} } \leq \frac{1}{8}l(\mb z, \mb z^*),\\
         & \sum_{i\in[n]} \max_{k\in[K] \backslash\{z_i^*\}} \omega_{k,z_i^*}  \ind{\big\{
        H_i(k, \mb z) \geq \frac{\delta}{8}\omega_{k,z_i^*}
         \big\} } = 0
\end{align}
simultaneously hold with probability at least  $ 1- O(d^{-c}) \vee e^{- (1 + o(1)) \frac{\snr^2}{2}}$. As a consequence, \eqref{eq: upper bound on l(hat z, z^*) 3} turns out to be 
\begin{align}
    & l(\hat{\mb z}, \mb z^*) = \xi_{\mathsf{oracle}}(\delta) + \frac{1}{4}l(\mb z,\mb z^*)
\end{align}
with probability at least  $1 - O(d^{-c}) \vee e^{- (1 + o(1)) \frac{\snr^2}{2}}$ for all $\mb z$ with $l(\mb z, \mb z^*) \leq c_1 \frac{n}{\beta K (\log d)^4}$. This one-step analysis serves as the groundwork for the upcoming analysis of geometric decay.

\medskip
\paragraph*{Step 4: Iterative Error Decay.}
Finally, armed with the upper bound \eqref{eq: upper bound for expectation of xi_oracle} on $\bb E\big[\xi_{\mathsf{oracle}}\big]$, along with Lemmas~\ref{lemma: Bound for xi_approx,i},~\ref{lemma: Bound for F_i},~\ref{lemma: Bound for G_i},~and~\ref{lemma: Bound for H_i}, we are ready to establish the iterative error decay of the alternative sequence $\{l(\hat{\mb z}^{(t)}, \mb z^*)\}_{t=0}^{T}$ via the one-step relation   \eqref{eq: proof sketch - iterative decay}. We let $\delta$ in \eqref{eq: upper bound on l(hat z, z^*) 3} be 
\eq{
\max\left\{\text{$\delta$ in Lemma~\ref{lemma: Bound for xi_approx,i}}, \text{upp. bounds in Lemmas~\ref{lemma: Bound for F_i},~\ref{lemma: Bound for G_i}}, \frac{\text{upp. bound in Lemma~\ref{lemma: Bound for H_i}}}{\underline \omega} \right\}
}
multiplied by some sufficiently large constant. 
Moreover, we define the following event:
\begin{align}
&\mc F_{\mathsf{good}} \coloneqq \Big\{ \text{inequalities in Lemmas~\ref{lemma: Bound for F_i},~\ref{lemma: Bound for G_i},~and~\ref{lemma: Bound for H_i} hold}, ~\text{and}~l(\hat{\mb z}^{(0)}, \mb z^*) \leq \frac{c_1 n}{\beta K(\log d)^4} \Big\},\\ 
& \mc F_{\mathsf{oracle}} \coloneqq \Big\{  \text{$\xi_{\mathsf{oracle}} \vee  \xi_{\mathsf{approx}} < \frac{1}{4}\frac{c_1 n}{\beta K(\log d)^4}$ holds } \Big\}. 
\end{align}
Note that 
\eq{
\bb P\big[\mc F_{\mathsf{good}}^{\complement}\big] \leq  O(d^{-c}) + O(n^{-2}) = O(n^{-2})
}
by Lemmas~\ref{lemma: Bound for F_i},~\ref{lemma: Bound for G_i},~and~\ref{lemma: Bound for H_i}. 

Then, we employ an induction argument under the event $\mc F_{\mathsf{good}} \cap \mc F_{\mathsf{oracle}}$. Invoking
\eqref{eq: upper bound on l(hat z, z^*) 3} yields:  
\begin{align}
    & l(\hat{\mb z}^{(1)}, \mb z^*) \leq \xi_{\oracle} + \xi_{\mathsf{approx}} + \frac{1}{4} l(\hat{{\mb z}}^{(0)}, \mb z^*) \leq \frac{n}{\beta K(\log d)^4}.
\end{align}
For each $k \in \bb N^+$, given the hypothesis that $l(\hat{\mb z}^{(k)}, \mb z^*) \leq \frac{n}{\beta K(\log d)^4}$, a similar argument gives: 
\begin{align}
    & l(\hat{\mb z}^{(k+1)}, \mb z^*) \leq \xi_{\oracle} + \xi_{\mathsf{approx}} + \frac{1}{4} l(\hat{{\mb z}}^{(k)}, \mb z^*) \leq \frac{n}{\beta K(\log d)^4}.
\end{align}

Therefore, by induction, we have 
$$l(\hat{\mb z}^{(t+1 )}, \mb z^*) \leq \xi_{\oracle}+ \xi_{\mathsf{approx}} + \frac{1}{4}l(\hat{\mb z}^{(t)}, \mb z^*)$$
for all $t\in \bb N$. 
We let $T = c_T\lceil \log n \rceil$ with some constant $c_T>0$ and apply the above relationship to derive that 
\begin{align}
    & l(\hat{\mb z}^{(t)}, \mb z^*) \leq\frac{4}{3}\xi_{\oracle} + \frac{4}{3}\xi_{\mathsf{approx}}+ \underbrace{4^{-c_T\lceil \log n \rceil} \cdot \frac{cn}{\beta K (\log d)^4}}_{< n^{-5}} \label{eq: l(hat z, z*) control under Fgood and Foracle}
\end{align}
holds for every $t \geq T$ under the event $\mc F_{\mathsf{good}}$. 

Moreover, invoking the relation that $h(\mb z, \mb z^*) \leq  \frac{1}{n\underline \omega }l(\mb z, \mb z^*)$ for every $\mb z$ together with \eqref{eq: l(hat z, z*) control under Fgood and Foracle} yields the desired upper bound on the expectation of $h(\hat{\mb z}^{(t)}, \mb z^*)$ for every $t \geq T$: 
\begin{align}
    & \bb E\big[ h(\hat{\mb z}^{(t)}, \mb z^*)\big] \\
    \leq & \bb E\big[\ind\{\mc F_{\mathsf{good}} \cap \mc F_{\mathsf{oracle}}\} \cdot \frac{1}{n\underline \omega}  \cdot \big(\frac{4}{3}\xi_{\mathsf{oracle}}+ \frac{4}{3}\xi_{\mathsf{approx}} + n^{-5}  \big) \big] + \bb E\big[\ind\{\mc F_{\mathsf{oracle}}^\complement\} \big]+ \bb E\big[\ind\{\mc F_{\mathsf{good}}^\complement\} \big]
    \\
    \lesssim & \frac{1}{n} \bb E\big[\ind\{\mc F_{\mathsf{good}} \cap \mc F_{\mathsf{oracle}}\}   (\xi_{\mathsf{oracle}} +  \xi_{\mathsf{approx}}) + n^{-5} \big]\\ 
    & +  \bb E\big[\ind\{\mc F_{\mathsf{oracle}}^\complement\} \big(\frac{\beta K(\log d)^4}{n }\big) \big(\frac{1}{2} \frac{cn}{\beta K (\log d)^4}\big)\big]+ \bb E\big[\ind\{\mc F_{\mathsf{good}}^\complement\} \big]\\ 
    \lesssim &  \frac{1}{n}\bb E\big[\ind\{\mc F_{\mathsf{good}} \cap \mc F_{\mathsf{oracle}}\}   (\xi_{\mathsf{oracle}} +  \xi_{\mathsf{approx}})   \big]  
     +  \frac{\beta K(\log d)^4}{n}\bb E\big[\ind\{\mc F_{\mathsf{oracle}}^\complement\}  \frac{cn}{\beta K (\log d)^4} \big]+ \bb E\big[\ind\{\mc F_{\mathsf{good}}^\complement\} \big] + n^{-5}\\ 
    \lesssim & \frac{\beta K(\log d)^4}{n}\bb E\big[ \xi_{\mathsf{oracle}} +  \xi_{\mathsf{approx}} \big] + O(n^{-2}) + e^{-(1+o(1)) \frac{\snr^2}{2}} \\
    \leq &\beta K^2\nu^2 \underline \omega (\log d)^4\Big[ \exp\big(-(1 + o(1)) \frac{{\mathsf{SNR}}^2}{2}\big) + O(d^{-c})\Big]+  O(n^{-2}) + e^{-(1+o(1)) \frac{\snr^2}{2}} \\ 
    \lesssim & \kappa_{\mathsf{cov}}^4 \mathsf{SNR}^4 (\log d)^4 \exp\big(-(1 + o(1)) \frac{{\mathsf{SNR}}^2}{2}\big) + O(\kappa_{\mathsf{cov}}^4\mathsf{SNR}^4\frac{(\log d)^4 }{d^{c}}) + O(n^{-2}) + e^{-(1+o(1)) \frac{\snr^2}{2}} \\ 
    \lesssim & \exp\big(- (1+ o(1)) \frac{{\mathsf{SNR}}^2}{2} \big) +   O(\kappa_{\mathsf{cov}}^4\mathsf{SNR}^4 n^{-5}) + O(n^{-2}) ,\label{eq: upper bound on misclustering expection}
\end{align} 
where we used \eqref{eq: upper bound for expectation of xi_oracle} and Lemma~\ref{lemma: Bound for xi_approx,i} in the third-to-last inequality.  
The penultimate inequality follows from $n\le d$, 
$\beta\nu^2 K^2 = o(\underline\omega)$, 
$\underline\omega \lesssim \kappa_{\mathsf{cov}}^2 \mathsf{SNR}^2$,  
and the last inequality uses 
$\mathsf{SNR} = \omega(\sqrt{\log\log d})$ and 
$\mathsf{SNR} = \omega(\kappa_{\mathsf{cov}}^4)$. 

To arrive at the conclusions, we analyze the misclustering rate under the following two regimes of $\mathsf{SNR}$: 
\begin{enumerate} 
    \item First, if $\mathsf{SNR} \leq  \sqrt{(2 + \epsilon)\log n}$, then \eqref{eq: upper bound on misclustering expection} yields that 
    \begin{equation}
        \bb E\big[ h(\hat{\mb z}^{(t)}, \mb z^*)] \leq \exp\big(- (1+ o(1)) \frac{{\mathsf{SNR}}^2}{2} \big)
    \end{equation}
    for every $t \geq T$, 
    where we use the fact that $ O(\kappa_{\mathsf{cov}}^4\mathsf{SNR}^4 n^{-5}) = \exp\big(- (1+ o(1)) \frac{{\mathsf{SNR}}^2}{2} \big)$, since  $n^{-5} \leq \exp\big(- \frac{10}{2+\epsilon}\cdot \frac{{\mathsf{SNR}}^2}{2} \big) \leq \exp\big(- \frac{10}{3}\cdot \frac{{\mathsf{SNR}}^2}{2} \big)$ and $\omega(\kappa_{\mathsf{cov}}^4 \mathsf{SNR}^4)   = \mathsf{SNR}^5  = o\big( \exp\big(\frac{7{\mathsf{SNR}}^2}{3}\big)\big)$. 
    \item Second, if $\mathsf{SNR} \geq \sqrt{(2 + \epsilon)\log n}$ for some $\epsilon > 0$, then it follows from \eqref{eq: l(hat z, z*) control under Fgood and Foracle} that for every $t \geq T$
    \begin{align}
        & \bb P\big[\hat{\mb z}^{(t)} \neq \mb z^* \big] \leq \bb P\big[\mc F_{\mathsf{good}}^\complement\big] + \bb P\big[\mc F_{\mathsf{good}} \cap \{h(\hat{\mb z}^{(t)}, \mb z^*) \geq \frac{1}{n} \} \big]\\ 
        \stackrel{\text{(1)}}{\leq } &\bb P\big[\mc F_{\mathsf{good}}^\complement\big] + \bb P\big[\mc F_{\mathsf{good}} \cap \big(\{\xi_{\mathsf{oracle}} \geq \frac{1}{4}\underline \omega  \} \cup  \{\xi_{\mathsf{approx}} \geq \frac{1}{4}\underline \omega  \}\big)\big]  \\
        \lesssim &  O(n^{-2}) + e^{-(1+o(1))\frac{\snr^2}{2}} + \frac{4}{\underline \omega} \big(\bb E[\xi_{\mathsf{oracle}}]+\bb E[\xi_{\mathsf{approx}}] \big) \\
         \stackrel{\text{by \eqref{eq: upper bound for expectation of xi_oracle}}}{\lesssim } &O(n^{-2}) + nK\nu^2 \Big[\exp\big(- (1 + o(1))\frac{{\mathsf{SNR}}^2}{2} \big) + O(d^{-c}) \Big]  
         \stackrel{\text{(2)}}{=}  o(1),
    \end{align}
    where (1) holds by \eqref{eq: l(hat z, z*) control under Fgood and Foracle} and the fact 
    $$ \mc F_{\mathsf{good}} \cap \big( \{\xi_{\mathsf{oracle}} \geq  \frac{1}{4}\underline \omega \} \cup  \{\xi_{\mathsf{approx}} \geq  \frac{1}{4}\underline \omega \}\big) \stackrel{n^{-5} \leq \frac{1}{2} \underline \omega}{\supseteq} \big(\mc F_{\mathsf{good}} \cap \{l(\hat{\mb z}^{(t)}, \mb z^*)\geq \underline \omega  \}\big) \supseteq \big(\mc F_{\mathsf{good}} \cap \{h(\hat{\mb z}^{(t)}, \mb z^*)\geq \frac{1}{n}  \}\big)$$
    for every $t \geq T$ and every sufficiently large $n$, and (2) holds since $nK\nu^2 \exp\big(-(1 + o(1))\frac{{\mathsf{SNR}}^2}{2} \big) \lesssim \exp\big(-(1 + o(1))\frac{{\mathsf{SNR}}^2}{2} \big) = o(1)$ and $O(\frac{nK\nu^2}{d^{c}}) = O(\frac{\nu^2}{d^{c-2}}) = o(1)$ by the assumption $\nu = o(d)$. 
\end{enumerate}

\subsection{Proof of the Lemmas in Section~\ref{sec: proof of upper bound for algorithm}}
Before we embark on the proofs, we first digress to present some instrumental lemmas. 

\subsubsection{Some Bounds on \texorpdfstring{$\mathsf{SNR}$}{SNR}}
We define the signal-to-noise ratio between two different clusters $a$ and $b$ by 
\begin{align}
    & \mathsf{SNR}_{a,b} \coloneqq \min_{\mb x\in \mc B_{a,b}} \norm{\mb x}_2, 
\end{align}
where we define that 
\begin{align*}
    \mc B_{a,b} = \Big\{
    & \mb x\in \bb R^K:  \mb x\t\big(\mb I - {\mb S_{a}^*}^{\frac{1}{2}}{\mb S_{b}^*}^{-1} {\mb S_{a}^*}^{\frac{1}{2}}\big)\mb x + \\
    & 2\mb x\t{\mb S_{a}^*}^{\frac{1}{2}}{\mb S_{b}^*}^{-1}\big({\mb w}_{b}^* - {\mb w}_{a}^*\big) -  \big(\mb w_{b}^* - {\mb w}_{a}^*\big)\t \mb V^*{\mb S_{b}^*}^{-1}{\mb V^{*\top}}\big({\mb w}_{b}^* - {\mb w}_{a}^*\big) = 0 \Big\}.
\end{align*} 
We introduce a lemma that relates $\mathsf{SNR}_{a,b}$ with the distance between ${\mb w}_a^*$ and ${\mb w}_b^*$.

\begin{lemma}
Assume that there exist constants $\lambda_{\min}, \lambda_{\max} >0$ such that $\lambda_{\min} \leq \lambda_K(\mb S^*_a) \leq \lambda_1(\mb S^*_a) \leq \lambda_{\max}$ for any $a\in[K]$. Then
\begin{align}
\frac{-\sqrt{\lambda_{\max}} + \sqrt{\lambda_{\max} + \frac{\lambda_{\min}(\lambda_{\min} + \lambda_{\max})}{\lambda_{\max}}}}{\lambda_{\min} + \lambda_{\max}} \bignorm{\mb w_a^* - \mb w_b^*} \leq \mathsf{SNR}_{a,b} \leq \lambda_{\min}^{-\frac{1}{2}} \bignorm{{\mb w}_a^* - {\mb w}_b^*}.
\end{align}
Moreover, with $\tau\coloneqq {\lambda_{\max}^{\frac12}}/{\lambda_{\min}^{\frac12}} \geq 1$ and $\triangle \coloneqq \min_{k_1\neq k_2\in[K]} \norm{\bo \theta_{k_1}^* - \bo \theta_{k_2}^*}_2$, we have
\begin{align}
    &\lambda_{\min}^{-\frac{1}{2}}\tau^{-1}\bignorm{{\mb w}_a^* - {\mb w}_b^*}_2\lesssim \mathsf{SNR}_{a,b} \leq \lambda_{\min}^{-\frac{1}{2}}\bignorm{{\mb w}_a^* - {\mb w}_b^*}_2,~  \frac{1}{2}\tau^{-1} \underline \omega^{\frac12} \leq \frac{\triangle}{2\bar\sigma_{\mathsf{cov}}} \leq  \mathsf{SNR} \leq  \underline \omega^{\frac12}. 
\end{align}
\label{lemma: SNR and distance}
\end{lemma}

\begin{remark}
We compare the exponent $-\snr^2/2$ in the upper bounds of 
Theorems~\ref{theorem: upper bound for algorithm (gaussian)} and 
\ref{theorem: upper bound for algorithm (bounded)} with existing results in 
\cite{abbe2022,zhang2024leave}.

First, consider the regime studied in \cite{zhang2024leave}, where
\begin{equation}
    \triangle \gg \sqrt{\beta}\, K\,(1+\sqrt{p/n})\,\tilde\sigma
    \qquad \text{(Theorem~3.1 in \cite{zhang2024leave})}.
\end{equation}
In this case, $\bar\sigma_{\mathsf{cov}}\approx \bar\sigma \le \tilde\sigma$. 
By Lemma~\ref{lemma: SNR and distance}, it follows that
\begin{align}
    \log\!\big(\text{upper bound in \cite[Theorem~3.1]{zhang2024leave}}\big)
    \asymp -\frac{\triangle^2}{8\tilde\sigma^2}
    \ge -\frac{\triangle^2}{8\bar\sigma^2}
    \gtrsim -\frac{\snr^2}{2}.
\end{align}

Moreover, \cite[Theorem~3.1]{abbe2022} establishes an upper bound of the form
\[
\exp\!\left(
-c\min\!\left\{
\frac{\triangle^2}{\tilde\sigma^2},
\frac{\triangle^2}{(p/n)\sigma^4/\triangle^2}
\right\}
\right).
\]
Since $\bar\sigma_{\mathsf{cov}}^2 \lesssim 
\max\{\tilde\sigma^2,(p/n)\sigma^4/\triangle^2\}$, we obtain
\[
-\min\!\left\{
\frac{\triangle^2}{\tilde\sigma^2},
\frac{\triangle^2}{(p/n)\sigma^4/\triangle^2}
\right\}
\gtrsim
-\frac{\triangle^2}{8\bar\sigma_{\mathsf{cov}}^2}
\ge
-\frac{\snr^2}{2}.
\]
\end{remark}

Besides, the condition $\mathsf{SNR} = \omega( \kappa^2 \kappa_{\mathsf{cov}}^{8} K^3 \beta^3  \nu^2  \sqrt{\log \log d})$ in Assumption~\ref{assumption: algorithm} implies the following consequence, which will be repeatedly invoked later:
\begin{align}
    & \underline \omega = \omega(\kappa^4 \kappa_{\mathsf{cov}}^{16} K^6 \beta^6  \nu^4  \log \log d). 
    \label{eq: alg condition}
\end{align}

\subsubsection{Concentrations on Noise Matrices} The following part comprises the concentration results for some linear forms of the noise matrix $\mb E$, under the Gaussian case (Lemma~\ref{lemma: noise matrix concentrations using the universality (Gaussian)}) and the bounded noise case (Lemma~\ref{lemma: noise matrix concentrations using the universality (bounded)}), respectively, as well as an upper bound (Lemma~\ref{lemma: 2k moment of At Ei}) on the moments of $\norm{\mb A\t \mb E_i}_2$ for a deterministic matrix $\mb A$. The proofs are postponed to Section~\ref{subsubsec: proof of noise matrix concentration}. 

\begin{lemma}\label{lemma: noise matrix concentrations using the universality (Gaussian)}
    Suppose that the noise matrix $\mb E$ satisfies the assumptions for the Gaussian case in Theorem~\ref{theorem: upper bound for algorithm (gaussian)}. Then with probability at least $1 - O(d^{-c-2})$, we have
    \begin{align}
        & \norm{\mb E} \lesssim \sigma \sqrt{p} + \tilde \sigma \sqrt{n}, \qquad \norm{\mb E_{\mc I_k(\mb z^*), \cdot}} \lesssim \sigma \sqrt{p} + \tilde\sigma \sqrt{n_k}, \qquad  \norm{\mb E_{i}}_2 \lesssim \sigma \sqrt{p},\qquad \norm{\mb E \mb V^*} \lesssim \bar \sigma \sqrt{n},\\
        &   \Bignorm{\sum_{i\in[n], z_i^* = k}\mb E_i\t \mb V^*}_2 \lesssim  \bar \sigma \sqrt{n_k K\log d},~\norm{\mb E \mb V^*}\ti \lesssim \bar\sigma \sqrt{K \log d},\\ 
        &  \norm{\mc H(\mb E \mb E\t) \mb U^*}\ti \lesssim \tilde \sigma \sigma \sqrt{p K \log d},~\norm{{\mb U^*}^\top \mc H(\mb E \mb E^\top)\mb U^*}
 \lesssim  K\tilde\sigma \sigma \sqrt{p\log d},  \\ 
        & \max_{i \in[n], k\in[K]} \Bignorm{{\mb U^*}\t \mc P_{-i,\cdot}(\mb E) \bo \Sigma_k \mc P_{-i,\cdot}(\mb E)\t \mb U^* - \sum_{i'\in[n]\backslash\{i\}} \tr(\bo \Sigma_{z_{i'}^*} \bo \Sigma_{z_i^*}){\mb U^*_{i',\cdot}}\t \mb U_{i',\cdot}^*}  
        \lesssim  K\sqrt{p}\sigma\tilde\sigma^3 \log d
        \label{eq:simplified spectral norm of sum EiV* (gaussian)}
    \end{align}
    with probability at least $1 - O(d^{-c-2})$. 
    \end{lemma}

  \begin{lemma}[Bounded Noise Matrix Concentrations]\label{lemma: noise matrix concentrations using the universality (bounded)}
Suppose the noise matrix $\mb E \in \bb R^{n\times p}$ obeys Assumption~\ref{assumption: bounded noise}.\ref{item: bounded noise assumption 1} and \ref{assumption: bounded noise}.\ref{item: bounded noise assumption 2}. Then we have
    \begin{align}
        & \norm{\mb E} \lesssim \sigma \sqrt{p} + \tilde \sigma \sqrt{n},  \qquad \bignorm{\mb E_{\mc I_k(\mb z^*), \cdot}}\lesssim \sigma \sqrt{p} + \tilde \sigma \sqrt{n_k}, \qquad \norm{\mb E_{i,\cdot}}_2 \lesssim \sigma\sqrt{p}, \qquad \norm{\mb E \mb V^*} \lesssim \tilde \sigma \sqrt{n}, \\  
        &     \Bignorm{\sum_{i\in[n], z_i^* = k}\mb E_i\t \mb V^*}_2 \lesssim \bar \sigma \sqrt{n_k K\log d},~ \norm{\mb E \mb V^*}\ti \lesssim \bar\sigma\sqrt{K \log d},~\norm{\mc H(\mb E \mb E\t) \mb U^*}\ti \lesssim \tilde \sigma \sigma \sqrt{p K \log d}, \\ 
        &  \max_{c\in[l]} \norm{\mb E_{\cdot, S_c} \t \mb U^*} \lesssim \sigma \sqrt{m K \log d}, ~ \norm{{\mb U^*}^\top \mc H(\mb E \mb E^\top)\mb U^*}
 \lesssim  K\tilde\sigma \sigma \sqrt{p\log d}, 
        \\
        & \max_{i \in[n], k\in[K]} \Bignorm{{\mb U^*}\t \mc P_{-i,\cdot}(\mb E) \bo \Sigma_k \mc P_{-i,\cdot}(\mb E)\t \mb U^* - \sum_{i'\in[n]\backslash\{i\}} \tr(\bo \Sigma_{z_{i'}^*} \bo \Sigma_{z_i^*}){\mb U^*_{i',\cdot}}\t \mb U_{i',\cdot}^*}  
        \lesssim  K\sqrt{p}\sigma\tilde\sigma^3\sqrt{\log d}
    \end{align}
    with probability at least $1 - O(d^{-c-2})$. 
\end{lemma}

\begin{lemma}
\label{lemma: 2k moment of At Ei}
    Instate Assumption~\ref{assumption: bounded noise} for the bounded noise cases. Then for every deterministic matrix $\mb A \in \bb R^{p \times K}$, one has 
    \begin{align}
        &  \bb E\big[\bignorm{\mb A\t \mb E_i}_2^{2k} \big]^{\frac{1}{2k}}\leq  \bb E\big[\bignorm{\mb A\t \mb G_i}_2^{2k} \big]^{\frac{1}{2k}} + C \big(\tr(\mb A\t \bo \Sigma_{z_i^*} \mb A)/K \big)^{\frac{1}{2k}} \big(\sqrt{m}B \max_{c\in[l]}\norm{\mb A_{S_c,\cdot}}\big)^{\frac{k-1}{k}} k^2,  \label{eq: 2k moment of At Ei}
    \end{align}
    where $\mb G_i$ denotes the Gaussian analog of $\mb E_i$ equipped with the same mean and covariance. 
\end{lemma}

\subsubsection{Additional Notations and Facts}
For notational simplicity, define 
\begin{align}
& \deltau \coloneqq \sqrt{\frac{\mu_1 K}{n}}\cdot \Big[ \frac{\deltaop^2 \snr  }{(\sigma_K^*)^4} + \frac{\mu_1^{\frac12} K \sigma_1^* \bar \sigma \sqrt{\log d / n} }{(\sigma_K^*)^2} \Big],  
\label{eq: definition of delta_u}
\\ 
& \deltau' \coloneqq \sqrt{\frac{\mu_1 K}{n}}\cdot \Big[ \frac{\deltaop^2 \sqrt{\log d}  }{(\sigma_K^*)^4} + \frac{\mu_1^{\frac12} K \sigma_1^* \bar \sigma \sqrt{\log d / n} }{(\sigma_K^*)^2} \Big],  
\label{eq: definition of delta_u'}
\end{align}
which correspond to the bounds in Theorem~\ref{thm: singular subspace perturbation theory} with $t_0=\mathsf{SNR}$ and $t_0=\sqrt{c\log d}$, respectively, taking $r=K$ in both cases. 
 The following facts will be used in the subsequent proof:
\begin{align}
    &\delta_u \lesssim \sqrt{\mu_1 }K \frac{\deltaop}{\sqrt{n} (\sigma_K^*)^2} \Big(\sqrt{\frac{\log d}{n}} +  \frac{\deltaop \snr}{(\sigma_K^*)^2} \Big),  
    \label{eq: alternative upper bound for delta_u}
    \\   
    &  \underline\sigma_{\mathsf{cov}}^2 \leq \bar \sigma_{\mathsf{cov}}^2 \lesssim  \frac{\deltaop^2}{n(\sigma_K^*)^2} \lesssim \kappa_{\mathsf{cov}}^2\underline \sigma_{\mathsf{cov}}^2, \label{eq: fact A} \\ 
    & n \snr^2 \leq n \underline \omega \lesssim \frac{(\sigma_1(\mb Z^*))^2 (\sigma_K(\bo \Theta^*))^2 }{ (\underline \sigma_{\mathsf{cov}})^2} \cdot \frac{K}{\beta} \lesssim \frac{K(\sigma_1^*)^2}{(\underline\sigma_{\mathsf{cov}})^2 } , 
    \label{eq: fact B}
            \\ 
            & \frac{\deltaop \snr}{(\sigma_K^*)^2} \lesssim  \frac{\deltaop^2}{n(\sigma_K^*)^2} \cdot \frac{n \snr^2 }{ \deltaop} \cdot \frac{1}{\snr}  \lesssim \bar \sigma_{\mathsf{cov}}^2 \cdot \frac{K(\sigma_1^*)^2}{\deltaop (\underline\sigma_{\mathsf{cov}})^2 } \cdot \frac{1}{\snr} \lesssim \frac{K \kappa^2 \kappa_{\mathsf{cov}}^3 }{\underline\omega^{\frac12}}  = o( 1) , \label{eq: fact C}\\ 
            & \mu_1 \leq \beta \label{eq: fact D}, \\ 
            & \delta_u \lesssim \sqrt{\mu_1 }K \frac{\deltaop \sqrt{\log d}}{\sqrt{n} (\sigma_K^*)^2} , \label{eq: alternative upper bound for delta_u'} \\ 
            & \underbrace{\sqrt{\frac{\beta K}{n}} \frac{K \deltaop \sqrt{\log d/n} + \beta K (\sigma_1^*)^2 / n}{(\sigma_K^*)^2} }_{\text{the bound for the bias terms in Theorem~\ref{thm: singular subspace perturbation theory}}} \lesssim \frac{ \underline \sigma_{\mathsf{cov}} \underline \omega^{\frac12}}{\kappa_{\mathsf{cov}}^5 \sigma_1^*} \Big[\frac{\kappa_{\mathsf{cov}}^5\kappa \beta^{\frac12} K^{\frac32}\deltaop \sqrt{\log d } /n }{\sigma_K^* \underline \sigma_{\mathsf{cov}} \underline \omega^{\frac12}} + \frac{ \sigma_1^* \beta^{\frac32} K^{\frac32} \kappa^2 \kappa_{\mathsf{cov}}^5 }{ \underline \sigma_{\mathsf{cov}} n^{\frac32} \underline \omega^{\frac12}} \Big] \\ 
            \lesssim & \frac{ \underline \sigma_{\mathsf{cov}} \underline \omega^{\frac12}}{\kappa_{\mathsf{cov}}^5 \sigma_1^*} \Big[\frac{\kappa_{\mathsf{cov}}^6\kappa \beta^{\frac12} K^{\frac32} \sqrt{\log d/n }}{\underline \omega^{\frac12}} + \frac{ \beta^{\frac12} \nu  \beta^{\frac32} K\kappa^3 \kappa_{\mathsf{cov}}^6 }{  n } \Big] = o(\frac{ \underline \sigma_{\mathsf{cov}} \underline \omega^{\frac12}}{\kappa_{\mathsf{cov}}^5 \sigma_1^*}) .  \label{eq: bound for the bias term}
\end{align}

The following corollary is a straightforward consequence of Theorem~\ref{thm: singular subspace perturbation theory}, which will be used repeatedly.
\begin{corollary}
\label{corollary: singular subspace perturbation theory in clustering problem}
    Instate either the assumptions in Theorem~\ref{theorem: upper bound for algorithm (gaussian)} or the assumptions in Theorem~\ref{theorem: upper bound for algorithm (bounded)}. Then for the top-$K$ left eigenvector matrix $\mb U$ of $\mc H(\mb M \mb M\t)$, with probability at least $1- O(e^{-\frac{\snr^2}{2}} \vee d^{-c})$ it holds: 
    \begin{align}
        & \norm{\mb U_{i,\cdot} \mb O  - \mb U^*_{i,\cdot} - \mk L_i - \mathsf{bias}_{z_i^*}\t }_2 \lesssim \deltau.
    \end{align} 
    Moreover, $\max_{k \in[K]}\norm{\bias_k}_2 \lesssim \sqrt{\frac{\mu_1 K}{n}}\cdot  \frac{K \deltaop \sqrt{\log d / n} + \mu_1 K (\sigma_1^*)^2 / n}{(\sigma_r^*)^2}$ with probability at least $1 - O(d^{-c})$. 
\end{corollary}
\begin{proof}[Proof of Corollary~\ref{corollary: singular subspace perturbation theory in clustering problem}]
    Invoking Theorem~\ref{thm: singular subspace perturbation theory} together with \eqref{eq: alternative upper bound for delta_u}, \eqref{eq: fact A}, and Lemmas~\ref{lemma: noise matrix concentrations using the universality (Gaussian)}-\ref{lemma: noise matrix concentrations using the universality (bounded)} yields that, with probability at least $1- O(e^{-\frac{\snr^2}{2}} \vee d^{-c})$,
    \begin{align}
        & \norm{\mb U_{i,\cdot}}_2  \lesssim \sqrt{\frac{\mu_1 K }{n}} \snr. 
    \end{align}

    Moreover, applying \cite[Lemma~2.5]{chen2021spectral} together with Wedin's theorem implies that 
    \begin{align}
        & \norm{\mb O - \mb U\t \mb U^*}\lesssim \frac{\deltaop^2}{(\sigma_K^*)^4}
    \end{align}
    with probability exceeding $1 - O(d^{-c})$. 

    As a result, one has with probability exceeding $1 - O(e^{-\frac{\snr^2}{2}} \vee d^{-c})$ that 
    \begin{align}
        & \norm{\mb U_{i,\cdot} \mb O  - \mb U^*_{i,\cdot} - \mk L_i - \mathsf{bias}_{z_i^*}\t }_2 \leq \norm{\mb U_{i,\cdot}}_2 \norm{\mb O - \mb U\t \mb U^*} +  \norm{\mb U_{i,\cdot} \mb U\t \mb U^* - \mb U^*_{i,\cdot} - \mk L_i - \mathsf{bias}_{z_i^*}\t }_2\lesssim \deltau.
    \end{align}

    The upper bound for the bias term directly follows from Theorem~\ref{thm: singular subspace perturbation theory}. 
\end{proof}

Lastly, we define: 
\begin{align}
        &\tilde{\omega}_{a,b} \coloneqq  \ip{ \tilde{\mb w}_a^* - \tilde{\mb w}_b^*, {\mb S_a^*}^{-1}( \tilde{\mb w}_a^* - \tilde{\mb w}_b^*)}. 
\end{align}
The distinction between the “tilde” quantities and the previously defined $\omega_{a,b}$ lies in the presence of bias. Nevertheless, by invoking Corollary~\ref{corollary: singular subspace perturbation theory in clustering problem}, we can still establish a tight proximity relationship: 
\begin{lemma}
\label{lemma: relation between tilde omega and omega}
    Instate either the assumptions in Theorem~\ref{theorem: upper bound for algorithm (gaussian)} or the assumptions in Theorem~\ref{theorem: upper bound for algorithm (bounded)}. Then with probability at least $1- O(d^{-c})$, $c_{\omega}\leq \max_{a\neq b\in[K]} \frac{\tilde \omega_{a,b}}{\omega_{a,b}} \leq  C_\omega$ for some constants $c_{\omega}$ and $C_{\omega}$. 
\end{lemma}
\begin{proof}[Proof of Lemma~\ref{lemma: relation between tilde omega and omega}]
The desired inequality follows directly from the following: 
\begin{align}
    & \frac{\tilde \omega_{a,b}}{\underline\omega} \leq \frac{2 \max_{k\in[K]} \norm{\mathsf{bias}_k}_2 \sigma_1^* \bar \omega^{\frac12}}{\underline \sigma_{\mathsf{cov}}\underline \omega} + \frac{\max_{k\in[K]} \norm{\mathsf{bias}_k}_2^2  (\sigma_1^*)^2 }{(\underline \sigma_{\mathsf{cov}})^2 \underline \omega } \\ 
    \stackrel{\text{by \eqref{eq: fact A}}}{\leq} & \frac{2 \max_{k\in[K]} \norm{\mathsf{bias}_k}_2 \sigma_1^* \bar \omega^{\frac12} \sqrt{n} \sigma_K^* \kappa_{\mathsf{cov}}}{ \deltaop \underline \omega} + \frac{\max_{k\in[K]} \norm{\mathsf{bias}_k}_2^2  (\sigma_1^*)^2 n (\sigma_K^*)^2 \kappa_{\mathsf{cov}}^2}{\deltaop^2 \underline \omega } \\
    \stackrel{\text{by Corollary~\ref{corollary: singular subspace perturbation theory in clustering problem} and \eqref{eq: bound for the bias term}}}{= } & o(1)
\end{align}
holds with probability at least $1- O(d^{-c})$ by Assumption~\ref{assumption: algorithm}. 
    
\end{proof}
\subsubsection{Center / Covariance Estimation Characterization}
The following lemma provides upper bounds on the fluctuations of the embedding centers given the true assignment $\mb z^*$, and given the estimated assignment $\mb z$, whose proofs are postponed to Section~\ref{subsec: proof of lemma misspecified center estimation}. 
\begin{lemma}
    \label{lemma: center estimation error}
        Instate either the assumptions in Theorem~\ref{theorem: upper bound for algorithm (gaussian)} or the assumptions in Theorem~\ref{theorem: upper bound for algorithm (bounded)}. Then it uniformly holds for all possible $\mb z$ that 
    \begin{align}
        & \norm{\bo \Lambda^*\mb O\t \hat{\mb c}_k(\mb z^*) -  \tilde{\mb w}_k^*}_2 \lesssim \kappa \mu_1 \sqrt{ \beta } K \frac{\deltaop}{\sqrt{n} \sigma_K^*} \Big( \sqrt{\frac{\log d}{n}}  + \frac{\deltaop \snr }{(\sigma_K^*)^2} \Big), 
        \label{eq: center estimation without misspecification}
        \\ 
        & \norm{\bo \Lambda^*\mb O\t \big(\hat{\mb c}_k(\mb z^*) - \hat{\mb c}_k(\mb z) \big)}_2 
         \lesssim  \beta K \sqrt{\frac{l(\mb z, \mb z^*)}{n\underline\omega}} \frac{\deltaop}{\sqrt{n} \sigma_K^*} 
        \label{eq: center estimation with misspecification}
    \end{align}
    with probability at least $1 - O(d^{-c}) \vee e^{- (1 + o(1)) \frac{\snr^2}{2}}$. 
    \end{lemma}

In order to obtain a uniform control on the fluctuations of the projected covariance matrix with misspecification, we present the following lemma, whose proof is presented in Section~\ref{sec: proof of projected covariance estimation error}.
\begin{lemma}
    \label{lemma: projected covariance matrix estimation error}
    Instate either the assumptions in Theorem~\ref{theorem: upper bound for algorithm (gaussian)} or the assumptions in Theorem~\ref{theorem: upper bound for algorithm (bounded)}. Then it uniformly holds for every $a \in [K]$ and every $\mb z$ with $l(\mb z, \mb z^*)\leq c\frac{n}{\beta K(\log d)^4}$ that
    \begin{align}
    & \norm{\bo \Lambda^*\mb O\t \big(\hat{\bo \Omega}_a(\mb z) - \hat{\bo \Omega}_a(\mb z^*)\big)\mb O\bo \Lambda^*} \lesssim   \underbrace{\kappa^2 \beta^3 K^2 \sqrt{\frac{l(\mb z, \mb z^*) K }{n}} \bar \sigma_{\mathsf{cov}}^2 / \underline \omega^{\frac12}  }_{\xicov} = o(\underline\sigma_{\mathsf{cov}}^2)\label{eq: definition of xi_cov}, \\ 
    &  \norm{{\bo \Lambda^*}^{-1}\mb O\t \big(\hat{\bo \Omega}_a^{-1}(\mb z) - \hat{\bo \Omega}_a^{-1}(\mb z^*)\big)\mb O{\bo \Lambda^*}^{-1}} 
            \lesssim   \frac{\xicov}{\underline \sigma_{\mathrm{cov}}^4} = o( \frac{1}{\underline\sigma_{\mathrm{cov}}^2})
    \end{align}
    with probability at least $1 - O(d^{-c}) \vee e^{- (1 + o(1)) \frac{\snr^2}{2}}$.
           \end{lemma}

To characterize the projected covariance matrix $\hat{\mb S}_k(\mb z^*)$ and its inverse given the ground truth $\mb z^*$, we present the following lemma, with its proof provided in Section~\ref{sec: proof of lemma inverse of projected covariance matrix estimation error 2}.
The main difficulty lies in establishing a concentration inequality for $\sum_{i: z_i^* = k}\bo \Lambda^* \mk L_i \mk L_i\t \bo \Lambda^*  / n_k - \mb S_k^*$, which arises for two reasons: (i) the expression involves controlling terms associated with the form $\mc H(\mb E \mb E\t) \mb E$, and $\mc H(\mb E \mb E\t) \mc H(\mb E \mb E\t)$; (ii) for locally dependent noise, the standard results of the literature do not apply.  
For the first challenge, we decompose the relevant terms into sums of U-statistics and then apply the decoupling technique from \cite{de1995decoupling} (cf.~Lemma~\ref{lemma: concentration on second part of covariance estimation}). 
To address the dependency issue, we resort to the universality result in \cite{brailovskaya2022universality}, which allows us to obtain more refined concentration inequalities (cf.~Lemma~\ref{lemma: S-universality}~and~Lemma~\ref{lemma: concentration on second part of covariance estimation}) for the projected covariance matrix.
\begin{lemma}
\label{lemma: projected covariance matrix estimation error with true z}
    Instate either the assumptions in Theorem~\ref{theorem: upper bound for algorithm (gaussian)} or the assumptions in Theorem~\ref{theorem: upper bound for algorithm (bounded)}. Then it holds with probability at least $1 - O(d^{-c})$ that 
    \begin{align}
        &\Bignorm{\bo \Lambda^* \mb O\t \hat{\bo \Omega}_k(\mb z^*)\mb O \bo \Lambda^*- \mb S_k^*} 
        \lesssim  \underbrace{\kappa^2 \beta^2 K^{\frac52} \frac{\deltaop^2}{n(\sigma_K^*)^2} \Big( \frac{\log d}{\sqrt{n}} +  \omega^{-\frac12} + \big( \frac{\tilde\sigma}{\sqrt{p} \sigma}\big)^{\frac16} \log d \Big) }_{\eqqcolon \delta_{\mathsf{cov}}} = o(\underline\sigma_{\mathsf{cov}}^2), \\
        &\Bignorm{\big(\bo \Lambda^* \mb O\t \hat{\bo \Omega}_k(\mb z^*)\mb O \bo \Lambda^*\big)^{-1} - {\mb S_k^*}^{-1}}  
        \lesssim  \frac{\delta_{\mathsf{cov}}}{\underline \sigma_{\mathsf{cov}}^4 }  = o(\frac{1}{\underline\sigma_{\mathsf{cov}}^2})\label{eq: inverse of projected covariance matrix estimation error 2}. 
    \end{align}
\end{lemma}

\subsubsection{Proof of Lemma~\ref{lemma: Bound for the first three terms in zeta_oracle}}
\label{sec: proof of lemma for tilde zeta oracle}
Recall the definition of $ \zeta_{\oracle, i}(k)$ in \eqref{eq: zeta_orackle,i(k)} that 
\longeq{
    & \zeta_{\oracle, i}(k) \coloneqq \ip{\bo \Lambda^*\mathfrak L_i,  (\bo \Lambda^* \mb O\t \hat{\bo \Omega}_k(\mb z^*) \mb O \bo \Lambda^*)^{-1} \bo \Lambda^* \mb O\t \big( \hat {\mb c}_{z_i^*}(\mb z^*) -  \hat{\mb c}_k(\mb z^*) \big)} \\ 
& \qquad + \frac{1}{2}\ip{\bo \Lambda^*\mathfrak L_i, \big[(\bo \Lambda^* \mb O\t \hat{\bo\Omega}_k(\mb z^*) \mb O \bo \Lambda^*)^{-1} - (\bo \Lambda^* \mb O\t \hat{\bo\Omega}_{z_i^*}(\mb z^*) \mb O \bo \Lambda^*)^{-1}\big]\bo \Lambda^* \mathfrak L_i}  \\ 
&\qquad + \frac{1}{2}\ip{\bo \Lambda^*\big(\tilde{\mb c}_{z_i^*}^*- \mb O\t\hat{\mb c}_k(\mb z^*) \big),  (\bo \Lambda^* \mb O\t \hat{\bo \Omega}_k(\mb z^*) \mb O \bo \Lambda^*)^{-1} \bo \Lambda^*\big(\tilde{\mb c}_{z_i^*}^*- \mb O\t\hat{\mb c}_k(\mb z^*) \big)} \\ 
& \qquad - \frac{1}{2}\ip{\bo \Lambda^*\big(\tilde{\mb c}_{z_i^*}^*-\mb O\t \hat{\mb c}_{z_i^*}(\mb z^*) \big), (\bo \Lambda^* \mb O\t \hat{\bo \Omega}_{z_i^*}(\mb z^*) \mb O \bo \Lambda^*)^{-1} \bo \Lambda^*\big(\tilde{\mb c}_{z_i^*}^*-\mb O\t \hat{\mb c}_{z_i^*}(\mb z^*) \big)}
}

For the inverse of the empirical quantity $\bo \Lambda^* \mb O\t \hat{\bo \Omega}_k(\mb z^*) \mb O \bo \Lambda^*$ given the oracle labels, we recognize it as a perturbed version of $\mb S_k^*$ up to a rotation and thus we can apply Lemmas~\ref{lemma: center estimation error},~\ref{lemma: projected covariance matrix estimation error},~and~\ref{lemma: projected covariance matrix estimation error with true z} to obtain that 
\begin{align}
    & \bignorm{\big(\bo \Lambda^* \mb O\t \hat{\bo \Omega}_k(\mb z^*) \mb O \bo \Lambda^*\big)^{-1} \bo \Lambda^* \mb O\t \big(\hat{\mb c}_{z_i^*}(\mb z^*) - \hat{\mb c}_k(\mb z^*)\big) - {\mb S_k^*}^{-1} \bo \Lambda^*  \big(\mb c_{z_i^*}^* - \mb c_{k}^* \big)}_2 \\ 
    \leq & \bignorm{\big(\bo \Lambda^* \mb O\t \hat{\bo \Omega}_k(\mb z^*) \mb O \bo \Lambda^*\big)^{-1} \bo \Lambda^* \mb O\t \big(\hat{\mb c}_{z_i^*}(\mb z^*) - \hat{\mb c}_k(\mb z^*)\big) - {\mb S_k^*}^{-1} \bo \Lambda^*  \big(\tilde{\mb c}_{z_i^*}^* - \tilde{\mb c}_{k}^* \big)}_2 \\ 
    & + \norm{{\mb S_k^*}^{-1} \bo \Lambda^*  \Big[\big(\tilde{\mb c}_{z_i^*}^* - \tilde{\mb c}_{z_i^*}^* \big) - \big(\mb c_{k}^* - \mb c_{k}^* \big)\Big]} \\ 
    \lesssim &\bignorm{\big(\bo \Lambda^* \mb O\t \hat{\bo \Omega}_k(\mb z^*) \mb O \bo \Lambda^*\big)^{-1}}\Big[\bignorm{ \bo \Lambda^* \big( \mb O\t\hat{\mb c}_{z_i^*}(\mb z^*) - \tilde{\mb c}^*_{z_i^*}\big)}_2 + \bignorm{\bo \Lambda^*\big(\mb O\t\hat{\mb c}_k(\mb z^*) - \tilde{\mb c}^*_{k}\big)} \Big] \\
    & \quad + \bignorm{\big(\bo \Lambda^* \mb O\t \hat{\bo \Omega}_k(\mb z^*) \mb O \bo \Lambda^*\big)^{-1} - {\mb S_k^*}^{-1} } \bignorm{\bo \Lambda^* \big( \tilde{\mb c}_{z_i^*}^* - \tilde{\mb c}_{k}^*\big)}_2 +  \max_{k'\in[K]}\norm{{\mb S_{k'}^*}^{-1} \bo \Lambda^*  \big(\tilde{\mb c}_{k'}^* - \mb c_{k'}^* \big)} \\
    \lesssim  & \frac{1}{\underline\sigma_{\mathrm{cov}}^2} \kappa \sqrt{\mu_1 \beta } K \frac{\deltaop}{\sqrt{n}\sigma_K^*}\Big[ \sqrt{\frac{\log d}{n}} + \frac{\deltaop}{(\sigma_K^*)^2} \snr \Big] + \frac{\delta_{\mathsf{cov}}}{\underline\sigma_{\mathrm{cov}}^4}  \nu \bar \sigma_{\mathrm{cov}} \underline\omega^{\frac12} + \frac{\sigma_1^*}{\underline \sigma_{\mathsf{cov}}^2}  \sqrt{\frac{\beta K}{n}}  \frac{K \deltaop \sqrt{\log d /n } + \beta K (\sigma_1^*)^2 / n}{(\sigma_K^*)^2}
    \label{eq: decomposition the first three term in zeta_oracle}
\end{align}
holds with probability at least $1 - O(d^{-c})$. It then follows from \eqref{eq: fact A}, \eqref{eq: fact C}, and \eqref{eq: bound for the bias term} that with probability at least $1 - O(d^{-c})$ 
\begin{align}
    & \bignorm{(\mb S^*_{z_i^*})^{\frac12}\big(\bo \Lambda^* \mb O\t \hat{\bo \Omega}_k(\mb z^*) \mb O \bo \Lambda^*\big)^{-1} \bo \Lambda^* \mb O\t \big(\hat{\mb c}_{z_i^*}(\mb z^*) - \hat{\mb c}_k(\mb z^*)\big) - (\mb S^*_{z_i^*})^{\frac12}{\mb S_k^*}^{-1} \bo \Lambda^*  \big(\tilde{\mb c}_{z_i^*}^* - \tilde{\mb c}_{z_i^*}^* \big)}_2  \\
    \lesssim& \frac{\kappa_{\mathrm{cov}}}{\underline\sigma_{\mathrm{cov}}} \kappa \sqrt{\mu_1 \beta } K \frac{\deltaop}{\sqrt{n}\sigma_K^*}\Big[ \sqrt{\frac{\log d}{n}} + \frac{\deltaop}{(\sigma_K^*)^2} \snr \Big] + \frac{\kappa_{\mathrm{cov}}^2\delta_{\mathsf{cov}}}{\underline\sigma_{\mathrm{cov}}^2}  \nu \underline\omega^{\frac12}  \\
    & + \frac{\kappa_{\mathsf{cov}} \sigma_1^*}{\underline \sigma_{\mathsf{cov}}}  \sqrt{\frac{\beta K}{n}}  \frac{K \deltaop \sqrt{\log d /n } + \beta K (\sigma_1^*)^2 / n}{(\sigma_K^*)^2} \\ 
    \eqqcolon & \delta_{\mathsf{linear}}  =  o(\frac{\underline\omega^{\frac12}}{\kappa_{\mathsf{cov}}^4}). 
    \label{eq: linear term in the perturbed SNR 1}
\end{align}

On the other hand, by \eqref{eq: inverse of projected covariance matrix estimation error 2} we have 
\begin{align}
    & \max_{b \in [K]} \bignorm{\big(\bo \Lambda^* \mb O\t \hat{\bo \Omega}_k(\mb z^*) \mb O \bo \Lambda^*\big)^{-1}  - {\mb S_k^*}^{-1}} \lesssim \frac{\delta_{\mathsf{cov}}}{\underline\sigma_{\mathrm{cov}}^4} \eqqcolon \delta_{\mathsf{quad}} = o(\frac{1}{\underline\sigma_{\mathsf{cov}}^2}) \label{eq: quadratic term in the perturbed SNR}
\end{align}
with probability at least $1- O(d^{-c})$. 
Then it follows from \eqref{eq: quadratic term in the perturbed SNR} that 
\begin{align}
    & \Bignorm{\mb S_{z_i^*}^{\frac{1}{2}} \Big[\big(\bo \Lambda^* \mb O\t \hat{\bo \Omega}_{z_i^*}(\mb z^*) \mb O \bo \Lambda^*\big)^{-1} - \big(\bo \Lambda^* \mb O\t \hat{\bo \Omega}_k(\mb z^*) \mb O \bo \Lambda^*\big)^{-1} \Big]\mb S_{z_i^*}^{\frac{1}{2}} - \big(\mb I - {\mb S_{z_i^*}^*}^{\frac{1}{2}}{\mb S_k^*}^{-1}{\mb S_{z_i^*}^*}^{\frac{1}{2}}\big)} =   o(\frac{1}{\kappa_{\mathsf{cov}}^4}) 
    \label{eq: quadratic term in the perturbed SNR 1}
\end{align}
holds with probability at least $1- O(d^{-c})$ for every $i\in[n]$ and $k\in[K]$.

For the last two terms in $\zeta_{\mathsf{oracle},i}(k)$, we follow the same decomposition as in \eqref{eq: decomposition the first three term in zeta_oracle} and obtain that, with probability at least $1 - O(d^{-c})$,
\begin{align}
    \Big\| \bo \Lambda^* \mb O^\top \big(\hat{\mb c}_{z_i^*}(\mb z^*) - \hat{\mb c}_k(\mb z^*)\big) 
    - \bo \Lambda^* \big(\mb c_{z_i^*}^* - \mb c_k^*\big) \Big\|_2
    = o\!\big(\frac{\bar \sigma_{\mathsf{cov}} \underline\omega^{1/2}}{\kappa_{\mathsf{cov}}^5}\big).
\end{align}
Combining this bound with \eqref{eq: quadratic term in the perturbed SNR} yields the existence of a term $\delta_{\mathsf{int}}$ such that, with probability at least $1 - O(d^{-c})$,
\begin{align}
\Bigg| 
&\frac{1}{2}\ip{\bo \Lambda^*\big(\tilde{\mb c}_{z_i^*}^*- \mb O^\top\hat{\mb c}_k(\mb z^*) \big),  
(\bo \Lambda^* \mb O^\top \hat{\bo \Omega}_k(\mb z^*) \mb O \bo \Lambda^*)^{-1} 
\bo \Lambda^*\big(\tilde{\mb c}_{z_i^*}^*- \mb O^\top\hat{\mb c}_k(\mb z^*) \big)} \\
&-
\frac{1}{2}\ip{\bo \Lambda^*\big(\tilde{\mb c}_{z_i^*}^*-\mb O^\top \hat{\mb c}_{z_i^*}(\mb z^*) \big), 
(\bo \Lambda^* \mb O^\top \hat{\bo \Omega}_{z_i^*}(\mb z^*) \mb O \bo \Lambda^*)^{-1} 
\bo \Lambda^*\big(\tilde{\mb c}_{z_i^*}^*-\mb O^\top \hat{\mb c}_{z_i^*}(\mb z^*) \big)} 
-\frac{1}{2}\omega_{k,z_i^*}
\Bigg|
\le \delta_{\mathsf{int}},
\label{eq: intercept term in the perturbed SNR 1}
\end{align}
where $\delta_{\mathsf{int}} = o\!\left(\omega / \kappa_{\mathsf{cov}}^4\right)$.

We note that, the pursued rates $o(\frac{\underline\omega^{\frac12}}{\kappa_{\mathsf{cov}}^4})$ in \eqref{eq: linear term in the perturbed SNR 1}, $o(\frac{1}{\kappa_{\mathsf{cov}}^4})$ in \eqref{eq: quadratic term in the perturbed SNR 1}, and $o(\frac{\underline \omega}{\kappa_{\mathsf{cov}}^4})$ in \eqref{eq: intercept term in the perturbed SNR 1}, are precisely those required for analyzing the stability of the perturbed decision boundary in the following. Specifically, we focus on ${\mb S_{z_i^*}^*}^{-\frac{1}{2}} \bo \Lambda^* \mk L_i \eqqcolon \tilde{\mb E}_i \in \bb R^K$ and define the perturbed signal noise ratio as follows: 
\begin{align}
    &     \mathsf{SNR}^{\mathsf{perturbed}}_{z_i^*, k}(\delta) \coloneqq  \argmin_{\mb x\in \bb R^K}\Big\{
    \norm{\mb x}_2:     \ip{\mb x, \mb S_{z_i^*}^{\frac{1}{2}} {\mb S_k^*}^{-1}\big(\mb w_{z_i^*}^* - \mb w_k^* \big)} \\ 
    & + \frac{1}{2}\ip{\mb x,  \big(\mb I - {\mb S_{z_i^*}^*}^{\frac{1}{2}}{\mb S_k^*}^{-1}{\mb S_{z_i^*}^*}^{\frac{1}{2}}\big)\mb x } + \big(\frac{1}{2} - \delta) \ip{
         \mb w_k^* - \mb w_{z_i^*}^*, {\mb S_k^*}^{-1} (\mb w_k^* - \mb w_{z_i^*}^*)} \leq 0 
    \Big\}, \quad \text{ for $\delta \geq 0$}.
\end{align}
For a sufficiently small $\delta$, one can tell that $\mathsf{SNR}^{\mathsf{perturbed}}_{z_i^*, k}(\delta) \leq \omega^{\frac12}_{k, z_i^*}$ following a similar argument to the ones in the proof of Lemma~\ref{lemma: SNR and distance}. Recall that $\delta_0$ denotes an vanishing sequence satisfying that $\kappa_{\mathsf{cov}}^4 \delta_0 = o(1)$. Let $\delta_1 = C_\delta \max\{\delta_0,\delta_{\mathsf{linear}} \underline \omega^{-\frac12}, \delta_{\mathsf{quad}}, \delta_{\mathsf{int}} \underline \omega^{-1}\} $ with some sufficiently large constant $C_\delta$, which again satisfies $\kappa_{\mathsf{cov}}^4 \delta_1 = o(1)$. 
Notice that 
\begin{align}
    & \mc E_{i,k,1} \bigcap \mc E_{i,k,2} \bigcap \mc E_{i,k,3} \bigcap \mc E_{i,k,4} \subseteq \left\{{\zeta}_{\mathsf{oracle},i}(k) >    \delta_0 \omega_{k, z_i^*}^{\frac12}\right\}
\end{align}
holds for every sufficiently small $\delta$, where $\mc E_{i,k,1}, \mc E_{i,k,2}, \mc E_{i,k,3}$ are defined as:
\begin{align}
    & \mc E_{i,k,1} \coloneqq \left\{ \bignorm{\tilde{\mb E}_i}_2 < \mathsf{SNR}^{\mathsf{perturbed}}_{k}(4\delta_1)\right\},
    \\ 
     & \mc E_{i,k,2} \coloneqq  \left\{  \bignorm{{\mb S^*_{z_i^*}}^{\frac12}\big(\bo \Lambda^* \mb O\t \hat{\bo \Omega}_k(\mb z^*) \mb O \bo \Lambda^*\big)^{-1} \bo \Lambda^* \mb O\t \big(\hat{\mb c}_{z_i^*}(\mb z^*) - \hat{\mb c}_k(\mb z^*)\big) - {\mb S^*_{z_i^*}}^{\frac12}{\mb S_k^*}^{-1} \bo \Lambda^*  \big(\tilde{\mb c}_{z_i^*}^* - \tilde{\mb c}_{z_i^*}^* \big)}_2 <   \delta_1\underline \omega^{\frac{1}{2}} \right\} , 
     \\ 
     & \mc E_{i,k,3} \coloneqq \left\{ \bignorm{\mb S_{z_i^*}^{\frac{1}{2}}\mb O\t \big(\hat{\mb S}_{z_i^*}(\mb z^*)^{-1} - \hat{\mb S}_k(\mb z^*)^{-1}\big) \mb O\mb S_{z_i^*}^{\frac{1}{2}} - \big(\mb I - {\mb S_{z_i^*}^*}^{\frac{1}{2}}{\mb S_k^*}^{-1}{\mb S_{z_i^*}^*}^{\frac{1}{2}}\big)} < \delta_1   \right\} , \\ 
     & \mc E_{i,k,4} \coloneqq \left\{\text{the LHS of \eqref{eq: intercept term in the perturbed SNR 1}  is upper bounded by $\delta_1 \underline \omega$}\right\} . 
\end{align}
It follows that 
\begin{align}
    & \bb P\left[ \zeta_{\oracle, i}(k)  \leq    \delta_0 \ip{\mb w_k^* -\mb w_{z_i^*}^* , {\mb S_k^*}^{-1}\big(\mb w_k^* - \mb w_{z_i^*}^* \big)} \right] \\
     \leq &  \bb P\left[\mc E_{i,k,1}^\complement \right] +  \bb P\left[\mc E_{i,k,2}^\complement \right] + \bb P\left[ \mc E_{i,k,3}^\complement\right] + \bb P\left[ \mc E_{i,k,4}^\complement\right] = \bb P\left[ \bignorm{\tilde{\mb E}_i}_2 \geq   \mathsf{SNR}^{\mathsf{perturbed}}_{z_i^*, k}(4\delta_1)\right] + O(d^{-c}) \label{eq: tilde zeta event probability decomposition}
\end{align}
by invoking \eqref{eq: linear term in the perturbed SNR 1} and \eqref{eq: quadratic term in the perturbed SNR 1}.

To handle the term $\bb P\left[ \bignorm{\tilde{\mb E}_i}_2 \geq   \mathsf{SNR}^{\mathsf{perturbed}}_{z_i^*, k}(4\delta_1)\right]$, we require a bound on how the perturbed decision boundary shifts in the projected space. The following result, taken directly from \cite[Lemma~C.9]{chen2024optimal}, summarizes this stability: 
\begin{lemma}
\label{lemma: SNR' and SNR}
    Consider the notations defined above. Then it holds that 
    \begin{align}
        & \mathsf{SNR}^{\mathsf{perturbed}}_{a,b}(\delta) \geq (1 - c_{\mathsf{SNR}}\kappa_{\mathsf{cov}}^4 \delta) \mathsf{SNR}_{a,b}
    \end{align}
    for some universal constant $c_{\mathsf{SNR}} > 0$ and every $\delta$ satisfying $\kappa_{\mathsf{cov}}^4 \delta \leq c_0$ for some sufficiently small constant $c_0$. 
\end{lemma}
As a consequence of Lemma~\ref{lemma: SNR' and SNR}, we deduce that 
\begin{equation}
    \bb P\left[ \bignorm{\tilde{\mb E}_i}_2 \geq   \mathsf{SNR}^{\mathsf{perturbed}}_{a,b}(4\delta_1)\right] \leq \bb P\left[ \bignorm{\tilde{\mb E}_i}_2 \geq  (1 - 4c_{\mathsf{SNR}} \kappa_{\mathsf{cov}}^4 \delta_1) \mathsf{SNR}_{a,b}\right] 
    \label{eq: probability tail regarding perturbed boundary}
\end{equation}
for every sufficiently large $n$. 

To control the $\ell_2$ norm of $\tilde{\mb E}_2$, we separately discuss two noise cases: for the Gaussian case, it is straightforward that
\eq{
    \bb P\left[ \bignorm{\tilde{\mb E}_i}_2 \geq  (1 - 4c_{\mathsf{SNR}} \kappa_{\mathsf{cov}}^4 \delta_1) \mathsf{SNR}_{a,b} \big| \mb E_{-i,\cdot}\right] \leq \bb P_{\epsilon \sim \mc N(\mb 0, \mb I_K)}\left[ \bignorm{\hat{\mb I}^{\frac12}\epsilon}_2(1 - 4c_{\mathsf{SNR}} \kappa_{\mathsf{cov}}^4 \delta_1) \mathsf{SNR}_{a,b}\right] , 
    \label{eq: hat I}
}
where $\epsilon \sim \mc I(\mb 0, \mb I_K)$ and $\hat{\mb I} \coloneqq {\mb S^*_{z_i^*}}^{-\frac12}  \big({\mb V^*}\t + {\bo \Lambda^*}^{-1} {\mb U^*}\t \mc P_{-i,\cdot}(\mb E)\big) \bo \Sigma_{z_i^*} \big( \mb V^* + \mc P_{-i,\cdot}(\mb E)\t \mb U^* {\bo \Lambda^*}^{-1}\big) {\mb S^*_{z_i^*}}^{-\frac12}$. Invoking Lemma~\ref{lemma: noise matrix concentrations using the universality (Gaussian)} gives us the following with probability at least $1- O(d^{-c})$: 
\begin{align}
    & \bignorm{\hat{\mb I} - \mb I_K} \lesssim \frac{K\sqrt{p} \sigma \tilde\sigma^3 \log d}{\underline{\sigma}_{\mathsf{cov}}^2 {\sigma^*_r}^2} + \frac{\tilde\sigma^3 K\sqrt{ \log d}}{
    \sigma_K^* \underline\sigma^2_{\mathsf{cov}}} \lesssim \frac{\tilde\sigma^2 K \log d }{\sqrt{n}\underline\sigma_{\mathsf{cov}}^2} \ll 1, 
     \label{eq: error of hat I}
\end{align}
which together with the Hanson-Wright inequality gives rise to  
\eq{
    \bb P\left[ \bignorm{\hat{\mb I}^{\frac12}\epsilon}_2(1 - 4c_{\mathsf{SNR}} \kappa_{\mathsf{cov}}^4 \delta_1) \mathsf{SNR}_{a,b}\right]  \leq \bb P\left[ \bignorm{\bo \epsilon}_2 \geq (1 - c\frac{\tilde\sigma^2 K\log d}{n^{\frac12} \underline\sigma_{\mathsf{cov}}^2})(1 - 4c_{\mathsf{SNR}} \kappa_{\mathsf{cov}}^4 \delta_1) \mathsf{SNR}_{a,b}\right] 
     \leq \exp\left(- t\right) + O(d^{-c})
}
for some constant $c$ and an arbitrary $t$ satisfying $K + 2 \sqrt{K t} + 2 t \leq  (1 - c\frac{\tilde\sigma^2 K\log d}{n^{\frac12} \underline\sigma_{\mathsf{cov}}^2})(1 - 4c_{\mathsf{SNR}} \kappa_{\mathsf{cov}}^4 \delta_1)^2\mathsf{SNR}_{a,b}^2$. Note that $K + 2\sqrt{Kt} + 2t \leq (2 + 4\sqrt{\frac{K}{t}})t$. Letting $t = (1 - \psi)\snr^2 /2$ with $0< \psi < \frac12$, one has $(2 + 4\sqrt{\frac{K}{t}})t \leq (1 + 4\sqrt{\frac{K}{\snr^2}})(1 - \psi)\snr^2 \eqqcolon \phi^2 \snr^2$ and thus $\psi = 1-  \frac{\phi}{1+4\sqrt{\frac{K}{\snr^2}}}$. Consequently, we have: 
\begin{align}
    & \bb P\left[ \bignorm{\bo \epsilon}_2 \geq (1 - c \frac{\tilde\sigma^2 K\log d}{n^{\frac12} \underline\sigma_{\mathsf{cov}}^2})(1 - 4c_{\mathsf{SNR}} \kappa_{\mathsf{cov}}^4 \delta_1) \mathsf{SNR}_{a,b}\right] 
     \\ 
     \leq & \exp\bigg[-\Big(1-\frac{ (1 - c \frac{\tilde\sigma^2 K\log d}{n^{\frac12} \underline\sigma_{\mathsf{cov}}^2})(1 - 4c_{\mathsf{SNR}} \kappa_{\mathsf{cov}}^4 \delta_1) }{1+4\sqrt{\frac{K}{\snr^2}}}\Big)\frac{\snr^2}{2}\bigg] + O(d^{-c}) = \exp\big(-(1 + o(1))\frac{\snr^2}{2}\big) + O(d^{-c}). 
\end{align}

Now we move on to the bounded noise case, where we shall separately discuss two regimes where $\snr \leq C_0\sqrt{K\log d}$ and $\snr > C_0 \sqrt{K\log d}$ provided a sufficiently large constant $C_0 >0$.

\emph{Weak Signal Case: $\snr \leq C_0\sqrt{K\log d}$. }
We employ Lemma~\ref{lemma: 2k moment of At Ei} conditional on $\mc P_{-i,\cdot}(\mb E)$ to derive that 
\begin{align}
    & \bb E\big[\bignorm{\tilde{\mb E}_i}_2^{2k} | \mc P_{-i,\cdot}(\mb E) \big]^{\frac{1}{2k}} \leq \bb E\Big[\bignorm{\mb G_i\big( \mb V^* + \mc P_{-i,\cdot}(\mb E)\t\mb U^* {\bo \Lambda^*}^{-1}\big){\mb S_{z_i^*}^*}^{-\frac12}}_2^{2k}| \mc P_{-i,\cdot}(\mb E) \Big]^{\frac{1}{2k}}  \\ 
    & + \big(\frac{1}{K} \tr(\hat{\mb I})\big)^{\frac{1}{2k}} \Big[ \sqrt{m} B \big(\sqrt{m} \norm{ \mb V^*}\ti  + \max_{c \in[l]}\bignorm{\mc P_{-i, S_c}(\mb E)\t\mb U^* {\bo \Lambda^*}^{-1}} \big) \norm{{\mb S_{z_i^*}^*}^{-\frac12}}\Big]^{\frac{k-1}{k}} k^2. 
    \label{eq: universality on the moments of tilde Ei}
\end{align}

For the first term on the RHS of \eqref{eq: universality on the moments of tilde Ei}, we note that $ \bb E\Big[\bignorm{\mb G_i\big( \mb V^* + \mc P_{-i,\cdot}(\mb E)\t\mb U^* {\bo \Lambda^*}^{-1}\big){\mb S_{z_i^*}^*}^{-\frac12}}_2^{2k}| \mc P_{-i,\cdot}(\mb E) \Big]^{\frac{1}{2k}}  = \bb E[\bignorm{\hat{\mb I} \bo \epsilon}_2^{2k}]^{\frac{1}{2k}}
$, 
where $\hat{\mb I}$ is defined below \eqref{eq: hat I} and $\epsilon \sim \mc N(\mb 0, \mb I_K)$. According to Lemma~\ref{lemma: noise matrix concentrations using the universality (bounded)}, one can prove that \eqref{eq: error of hat I} also holds for bounded noise. We therefore invoke the moment expression of a chi-squared distribution along with the Stirling formula to have 
\begin{align}
    & \bb E\big[\bignorm{\tilde{\mb E}_i}_2^{2k} | \mc P_{-i,\cdot}(\mb E) \big]^{\frac{1}{2k}} \leq  \big(1 + c \frac{\tilde\sigma^2 K\log d}{n^{\frac12} \underline\sigma_{\mathsf{cov}}^2}\big) \bb E\big[ \norm{\bo \epsilon}_2^{2k}\big]^{\frac{1}{2k}} \leq (1 + c \frac{\tilde\sigma^2 K\log d}{n^{\frac12} \underline\sigma_{\mathsf{cov}}^2})\Big[ 2^k \frac{\Gamma(k + \frac{K}{2} ) }{\Gamma(\frac{K}{n})}\Big]^{\frac{1}{2k}}, 
    \label{eq: upper bound on 2k moments}
    \\ 
      & \Big[ 2^k \frac{\Gamma(k + \frac{K}{2} ) }{\Gamma(\frac{K}{n})}\Big]^{\frac{1}{2k}} \leq 
    \begin{cases}
         \sqrt{2} \Big[\sqrt{2\pi k}(\frac{k}{e})^k \Big]^{\frac{1}{2k}}, & K =2\\ 
        \sqrt{2} \bigg[C \frac{\sqrt{2\pi(k + \frac{K}{2} - 1) } \big( \frac{k+ \frac{K}{2} - 1}{e}\big)^{k+\frac{K}{2} - 1}}{\sqrt{2\pi(\frac{K}{2} -1) } \big( \frac{\frac{K}{2} -1 }{e}\big)^{\frac{K}{2} - 1}} \bigg]^{\frac{1}{2k}} , & K \geq 3 
    \end{cases}\\ 
    \leq & \begin{cases}
            \sqrt{\frac{2}{e}} (2\pi k)^{\frac{1}{4k}}\sqrt{k} , & K = 2 \\ 
            \sqrt{\frac{2}{e}}C_1^{\frac{1}{2k}} (1 + \varrho^{-1})^{\varrho / 2 + 1 /4k}(1 + \varrho)^{\frac12}  \sqrt{k}    , & K \geq 3       
        \end{cases}
\end{align}
for some constants $c$ and $C$, where we write $\frac{K/2 - 1}{k}$ as $\varrho$. To proceed, we denote $\sup\limits_{k \geq \snr^2 /4}(2\pi k)^{\frac{1}{4k}}$ and $ \sup\limits_{k \geq \snr^2 /4}C_1^{\frac{1}{2k}} (1 + \varrho^{-1})^{\varrho / 2 + 1 /4k}(1 + \varrho)^{\frac12}$ by $a_n$ and $b_n$, respectively (we recall that $\snr$ can be viewed as a sequence with respect to $n$). 
According to the assumption on $\snr$, one has $ \lim\limits_{n \rightarrow \infty} a_n = 1$ and $\lim\limits_{n \rightarrow \infty} b_n =1$.

Moreover, we have with probability at least $1- O(d^{-c})$ that 
\begin{align}
    & \max_{c \in[l]} \bignorm{\mc P_{-i,S_c}(\mb E)\t\mb U^* {\bo \Lambda^*}^{-1}} \lesssim \frac{\tilde\sigma \sqrt{K \log d}}{\sigma_K^*},
    \label{eq: concentration on P-i(E)t U Lambda-1}
\end{align}
which leads to 
\begin{align}
    &c_n \coloneqq \sup\limits_{ \snr^2 /4 \leq k \leq C_0^2 K\log d} \big(\frac{1}{K} \tr(\hat{\mb I})\big)^{\frac{1}{2k}} \Big[ \sqrt{m} B \big(\sqrt{m} \norm{ \mb V^*}\ti  + \max_{c \in[l]}\bignorm{\mc P_{-i, S_c}(\mb E)\t\mb U^* {\bo \Lambda^*}^{-1}} \big) \bignorm{{\mb S_{z_i^*}^*}^{-\frac12}}\Big]^{\frac{k-1}{k}} k^2 \\ 
    \lesssim &  \Big[\frac{mB \sqrt{\mu_2/ p}K^{\frac92}(\log d)^4}{\underline\sigma_{\mathsf{cov}}} + \sqrt{m}B \frac{\tilde\sigma K^{\frac92} (\log d)^{\frac92}}{\underline\sigma_{\mathsf{cov}} \sigma_K^*}\Big]^{\frac{1}{2}} \lesssim 1.
\end{align}
with probability at least $1- O(d^{-c})$ by the concentration on $\norm{\hat{\mb I} - \mb I}$, \eqref{eq: concentration on P-i(E)t U Lambda-1}, and $\frac{C_0^2 K \log d  -1}{C_0^2 K \log d} \geq \frac12$. As a consequence, we let  
\begin{align}
k = \begin{cases}
\lfloor (1 - 4c_{\mathsf{SNR}} \kappa_{\mathsf{cov}}^4 \delta_1)(1 + c \frac{2\tilde\sigma^2 \log d}{n^{\frac12} \underline\sigma_{\mathsf{cov}}^2})^{-1} \big(\frac{\snr^2}{2a_n^2} - \sqrt{e}\sup_{n}c_n\big)\rfloor, & K=2 \\ 
\lfloor(1 - 4c_{\mathsf{SNR}} \kappa_{\mathsf{cov}}^4 \delta_1)(1 + c \frac{\tilde\sigma^2 K\log d}{n^{\frac12} \underline\sigma_{\mathsf{cov}}^2})^{-1}\big( \frac{\snr^2}{2b_n^2} - \sqrt{e} \sup_{n}c_n \big)\rfloor, & K \geq 3
\end{cases}
\end{align}
and apply the Markov inequality together with \eqref{eq: probability tail regarding perturbed boundary} and \eqref{eq: upper bound on 2k moments} to deduce that 
\begin{align}
    & \bb P\left[ \bignorm{\tilde{\mb E}_i}_2 \geq   \mathsf{SNR}^{\mathsf{perturbed}}_{k}(4\delta_1)\right] \leq \exp(-k) + O(d^{-c}) = \exp\big( - (1 + o(1))\frac{\snr^2}{2}\big) + O(d^{-c}).  
\end{align}

\emph{Strong Signal Case: $\snr > C_0\sqrt{\log d}$. } 

Conditional on $\mc E_{-i,\cdot}$, we apply the matrix Bernstein inequailty to yield that 
\begin{equation}
    \norm{\tilde{\mb E}_i}_2 \lesssim \tr(\hat{\mb I})^{\frac12} \sqrt{\log d} + \sqrt{m}B \Big[\sqrt{m} \norm{ \mb V^*}\ti  + \max_{c \in[l]}\bignorm{\mc P_{-i, S_c}(\mb E)\t\mb U^* {\bo \Lambda^*}^{-1}} \Big]\bignorm{{\mb S_{z_i^*}^*}^{-\frac12}} \log d  
    \lesssim  \sqrt{K \log d}
\end{equation}
holds with probability at least $1-  O(d^{-c})$, 
since \eqref{eq: error of hat I} also holds for bounded noise because of Lemma~\ref{lemma: noise matrix concentrations using the universality (bounded)}. Hence, given an appropriate $C_0$, we derive that
\begin{equation}
    \bb P\left[ \bignorm{\tilde{\mb E}_i}_2 \geq  (1 - 4c_{\mathsf{SNR}} \kappa_{\mathsf{cov}}^4 \delta_1) \mathsf{SNR}_{a,b}\right]  \leq \bb P\left[ \bignorm{\tilde{\mb E}_i}_2 \geq  \frac12 C_0\sqrt{K\log d}\right] = O(d^{-c}). 
\end{equation}

\bigskip

To conclude, from the preceding arguments and \eqref{eq: tilde zeta event probability decomposition}, there exists a vanishing sequence $\delta'$ such that 
\begin{align}
    & \bb P\left[ \zeta_{\oracle, i}(k)  \leq    \delta_0 \omega_{k,z_i^*} \right]\leq   \exp\big(-(1 - \delta')\frac{\snr^2}{2}\big) + O(d^{-c}). 
\end{align}

\subsubsection{Proof of Lemma~\ref{lemma: Bound for xi_approx,i}}
\label{subsubsec: proof of xi_approx,i}
Recall that for every $i \in[n]$ and $b\in[K]$, 
\begin{align}
    & \zeta_{\mathsf{approx},i}(b, \mb z)= \ip{\bo \Lambda^* \mb O\t \big(\mb U_i - \hat{\mb c}_b(\mb z)\big), \big( \bo \Lambda^* \mb O\t \hat{\bo \Omega}_b(\mb z) \mb O \bo \Lambda^* \big)^{-1}\bo \Lambda^* \mb O\t \big(\mb U_i - \hat{\mb c}_b(\mb z)\big)} 
\\
& \qquad - \ip{ \bo \Lambda^*\mb O\t  \big(\mb U_i - \hat{\mb c}_{z_i^*}(\mb z)\big),\big(\bo \Lambda^* \mb O\t  \hat{\bo \Omega}_k(\mb z) \mb O\bo \Lambda^*\big)^{-1}  \bo \Lambda^* \mb O\t \big(\mb U_i - \hat{\mb c}_{z_i^*}(\mb z)\big)} \\ 
& \qquad - \ip{\bo \Lambda^* \big(\mathfrak L_i + (\tilde{\mb c}_{z_i^*}^* - \mb O\t  \hat{\mb c}_b(\mb z))\big), \big( \bo \Lambda^* \mb O\t \hat{\bo \Omega}_b(\mb z) \mb O \bo \Lambda^* \big)^{-1} \bo \Lambda^* \big( \mathfrak L_i + (\tilde{\mb c}_{z_i^*}^* - \mb O\t  \hat{\mb c}_b(\mb z) )\big)} 
\\
& \qquad + \ip{\bo \Lambda^* \big(\mathfrak L_i + (\tilde{\mb c}_{z_i^*}^* -  \mb O\t  \hat{\mb c}_{z_i^*}(\mb z))\big),\big(\bo \Lambda^* \mb O\t  \hat{\bo \Omega}_{z_i^*}(\mb z) \mb O\bo \Lambda^*\big)^{-1} \bo \Lambda^* \big( \mathfrak L_i + (\tilde{\mb c}_{z_i^*}^* -  \mb O\t  \hat{\mb c}_{z_i^*}(\mb z))\big)}.
\end{align}
It is straightforward from its definition to derive that 
\longeq{
    & \zeta_{\mathsf{approx},i}(b, \mb z) \leq \norm{\bo \Lambda^*\big(\mb O\t\mb U_i - \mk L_i - \tilde{\mb c}_{z_i^*}^*\big)}_2 / \min_{k\in[K]} \sigma_{\min}(\bo \Lambda^* \mb O\t \hat{\bo \Omega}_k(\mb z) \mb O \bo \Lambda^*)\\ 
    & \cdot \Big[\norm{\bo \Lambda^*\big(\mb O\t\mb U_i - \mk L_i - \tilde{\mb c}_{z_i^*}^*\big)}_2 + 2\max_{k\in[K]} \norm{\bo \Lambda^* \big( \mb c_{k}^* - \mb O\t \hat{\mb c}_k(\mb z) \big)}_2 + \nu \bar \sigma_{\mathrm{cov}} \underline\omega^{\frac12} \Big] \\ 
    \lesssim & \frac{\nu \bar \sigma_{\mathrm{cov}} \underline\omega^{\frac12}}{\underline \sigma_{\mathrm{cov}}^2} \cdot \sigma_1^* \deltau = \frac{\kappa \mu_1^{\frac12} K \nu \bar \sigma_{\mathrm{cov}} \underline\omega^{\frac12}}{\underline \sigma_{\mathrm{cov}}^2} \frac{\deltaop}{\sqrt{n}\sigma_K^*} \Big( \frac{\log d}{n} + \frac{\deltaop \snr}{(\sigma_K^*)^2}\Big)  
    \lesssim \beta^{\frac12} \kappa K \nu  \underline\omega^{\frac12} \kappa_{\mathsf{cov}}^2  = o(\frac{\underline \omega}{\kappa_{\mathsf{cov}}^4}) 
}
with probability at least $1- O(d^{-c}) \vee e^{-(1 + o(1))\frac{\snr^2}{2}}$, where we used Lemmas~\ref{lemma: relation between tilde omega and omega} and \ref{lemma: center estimation error}, Corollary~\ref{corollary: singular subspace perturbation theory in clustering problem}, Eqs.~\eqref{eq: definition of delta_u},~\eqref{eq: fact A},~\eqref{eq: fact B},~\eqref{eq: fact C}, and the condition \eqref{eq: alg condition}.     
    As a consequence, we have for some constant $C$ that
      \begin{align}
          & \bb E\Big[ \xi_{\mathsf{approx}}\big(C\frac{\beta^{\frac12} \kappa K \nu \kappa_{\mathsf{cov}}^2}{\underline\omega^{\frac12}}\big) \Big]  \lesssim nK \nu \underline\omega \Big[ e^{-(1 + o(1))\frac{\snr^2}{2}} \vee d^{-c}\Big]. 
      \end{align}

\subsection{Misspecification Effect Analysis}\label{subsection: Misspecification Effect Analysis}
In what follows, we shall work on upper bounding the effect of misspecification of the cluster labels in the last step.

\subsubsection{Proof of Lemma~\ref{lemma: Bound for F_i}}
    Recall that 
    \longeq{
    &F_i(b,\mb z)
     = -\ip{
        \bo \Lambda^* \mathfrak L_i , (\bo \Lambda^* \mb O\t \hat{\bo \Omega}_k(\mb z) \mb O \bo \Lambda^*)^{-1} \bo \Lambda^* \mb O\t \big(\hat{\mb c}_k(\mb z^*) - \hat{\mb c}_k(\mb z)\big)} \\ 
        & \qquad + \ip{
            \bo \Lambda^* \mathfrak L_i , (\bo \Lambda^* \mb O\t \hat{\bo \Omega}_{z_i^*}(\mb z) \mb O \bo \Lambda^*)^{-1} \bo \Lambda^*\big( \hat{\mb c}_{z_i^*}(\mb z^*) - \hat{\mb c}_{z_i^*}(\mb z) \big)
    } 
    \\
     &\qquad  - \ip{
       \bo \Lambda^* \mathfrak L_i,  \big[(\bo \Lambda^* \mb O\t \hat{\bo\Omega}_k(\mb z) \mb O \bo \Lambda^*)^{-1} - (\bo \Lambda^* \mb O\t \hat{\bo\Omega}_k(\mb z^*) \mb O \bo \Lambda^*)^{-1}\big] \bo \Lambda^*\big(\tilde{\mb c}_{z_i^*}^*- \mb O\t\hat{\mb c}_k(\mb z^*) \big)}  \\ 
       & \qquad + \ip{ \bo \Lambda^* \mathfrak L_i 
        ,\big[(\bo \Lambda^* \mb O\t \hat{\bo\Omega}_{z_i^*}(\mb z) \mb O \bo \Lambda^*)^{-1} - (\bo \Lambda^* \mb O\t \hat{\bo\Omega}_{z_i^*}(\mb z^*) \mb O \bo \Lambda^*)^{-1}\big] \bo \Lambda^*\big(\tilde{\mb c}_{z_i^*}^*-\mb O\t \hat{\mb c}_{z_i^*}(\mb z^*) \big) }.
    }
    Using Cauchy's inequality, we shall control $\sum_{i\in[n]} \max_{b\in[K] \backslash \{z_i^* \}}\frac{F_i(b,\mb z)^2}{\omega_{z_i^*, b}}$ via upper bounding the summation of  each term's square appearing above over all the samples.
    \begin{itemize} 
	\item For the first term $-\ip{
        \bo \Lambda^* \mathfrak L_i , (\bo \Lambda^* \mb O\t \hat{\bo \Omega}_k(\mb z) \mb O \bo \Lambda^*)^{-1} \bo \Lambda^* \mb O\t \big(\hat{\mb c}_k(\mb z^*) - \hat{\mb c}_k(\mb z)\big)}$ of \eqref{eq: F term}, one has 
    \begin{align}
    & \sum_{i\in[n]} \max_{b\in [K]\backslash \{z_i^*\}}  \frac{\big\langle 
        \bo \Lambda^* \mathfrak L_i , (\bo \Lambda^* \mb O\t \hat{\bo \Omega}_b(\mb z) \mb O \bo \Lambda^*)^{-1} \bo \Lambda^* \mb O\t \big(\hat{\mb c}_b(\mb z^*) - \hat{\mb c}_b(\mb z)\big) \big\rangle^2 }{\omega_{z_i^*, b} }\\  
      \leq  & \sum_{i\in[n]} \sum_{b \in[K], b\in[K]\backslash \{z_i^*\} } \frac{\big\langle 
        \bo \Lambda^* \mathfrak L_i , (\bo \Lambda^* \mb O\t \hat{\bo \Omega}_b(\mb z) \mb O \bo \Lambda^*)^{-1} \bo \Lambda^* \mb O\t \big(\hat{\mb c}_b(\mb z^*) - \hat{\mb c}_b(\mb z)\big) \big\rangle^2}{\omega_{z_i^*, b} }
        \\
         \leq & \frac{ K^2}{\underline \sigma_{\mathrm{cov}}^4 \underline \omega } \max_{b \in [K]}\bignorm{\bo \Lambda^* \mb O\t \big(\hat{\mb c}_b(\mb z^*) - \hat{\mb c}_b(\mb z)\big) }_2^2\Bignorm{\sum_{i\in[n]} \bo \Lambda^* \mk L_i \mk L_i\t \bo \Lambda^*}
         \label{eq: first term of F_i}
    \end{align}
    with probability at least $1- O(d^{-c})$ by Lemma~\ref{lemma: projected covariance matrix estimation error},
    where the second inequality follows from a useful fact that 
    \eq{\label{eq: fact of trace operation}
     \sum_{i\in[n]} \big \langle \mb a_i, \mb B \mb c\big\rangle^2  = \tr\big(\mb c\t \mb B\t\sum_{i\in[n]}\mb a_i \mb a_i\t \mb B \mb c \big) \leq \norm{\mb B\mb c}_2^2 \tr\big(\sum_{i\in[n]} \mb a_i \mb a_i\t \big) \leq    K\big\lv\mb B \mb c\big\lv_2^2 \Big\lv\sum_{i\in[n]}\mb a_i \mb a_i\t \Big\rv_2
    }
    holds for arbitrary vectors $\mb a_1, \cdots, \mb a_n, \mb  c\in \mathbb R^K$ and an arbitrary matrix $\mb B \in \bb R^{K\times K}$. 

    \begin{itemize} 
        \item For $ \max\limits_{a,b \in [K], a\neq b}\bignorm{\bo \Lambda^* \mb O\t \big(\hat{\mb c}_b(\mb z^*) - \hat{\mb c}_b(\mb z)\big) }_2^2$, Lemma~\ref{lemma: center estimation error} guarantees that, with probability at least $1- O(d^{-c})$, 
    \begin{align}
        & \max_{a,b \in [K], a\neq b}\bignorm{\bo \Lambda^* \mb O\t \big(\hat{\mb c}_b(\mb z^*) - \hat{\mb c}_b(\mb z)\big) }_2^2 \lesssim \beta^2 K^2 \frac{l(\mb z, \mb z^*)}{n\underline\omega} \frac{\deltaop^2}{n (\sigma_K^*)^2}  .
        \label{eq: proof of lemma 1 first term decomposition 1}
    \end{align}

    \item For $\Bignorm{\sum_{i\in[n]} \bo \Lambda^* \mk L_i \mk L_i\t \bo \Lambda^*}$, invoking the triangle inequality together with Eqs.~\eqref{eq: concentration for Lambda sum Li Lit Lambda / nk - S_k} and \eqref{eq: fact A} implies that with probability at least $1 - O(d^{-c})$
    \begin{align}
        & \Bignorm{\sum_{i\in[n]} \bo \Lambda^* \mk L_i \mk L_i\t \bo \Lambda^*}  \lesssim  \bar \sigma_{\mathrm{cov}}^2 n.
        \label{eq: proof of lemma 1 first term decomposition 2}
    \end{align}

\end{itemize}
    
     Directly plugging \eqref{eq: proof of lemma 1 first term decomposition 1} and \eqref{eq: proof of lemma 1 first term decomposition 2} into \eqref{eq: first term of F_i} together with Eqs. \eqref{eq: alternative upper bound for delta_u},~\ref{eq: fact A},~\ref{eq: fact D} gives that  
	\eq{
	 \text{the LHS of \eqref{eq: first term of F_i}} 
        \lesssim  \frac{K^2 }{\underline\sigma_{\mathrm{cov}}^4 \underline\omega} \big[\beta^2 K^2 \frac{l(\mb z,\mb z^*)}{n \underline\omega } \frac{\deltaop^2}{n (\sigma_K^*)^2}  \big] \bar \sigma_{\mathrm{cov}}^2 n \lesssim \frac{\beta^2 K^4 \kappa_{\mathrm{cov}}^4 }{ \underline \omega^2} l(\mb z, \mb z^*) 
	}
holds with probability at least $1 - O(d^{-c}) \vee e^{-(1+o(1))\frac{\snr^2}{2}}$. 
 
\item 	For the second term on the right-hand side of \eqref{eq: F term}, a similar derivation to the above gives:
	\longeq{
	&     \sum_{i\in[n]} \max_{b\in[K]\backslash \{z_i^*\}}
           \frac{\ip{
            \bo \Lambda^* \mathfrak L_i , (\bo \Lambda^* \mb O\t \hat{\bo \Omega}_{z_i^*}(\mb z) \mb O \bo \Lambda^*)^{-1} \bo \Lambda^*\big( \hat{\mb c}_{z_i^*}(\mb z^*) - \hat{\mb c}_{z_i^*}(\mb z) \big)}^2 }{\omega_{z_i^*, b} }\\
     \lesssim & \sum_{i\in[n]} \sum_{b\in[K], b\in[K]\backslash \{z_i^*\}}\frac{\ip{
            \bo \Lambda^* \mathfrak L_i , (\bo \Lambda^* \mb O\t \hat{\bo \Omega}_{z_i^*}(\mb z) \mb O \bo \Lambda^*)^{-1} \bo \Lambda^*\big( \hat{\mb c}_{z_i^*}(\mb z^*) - \hat{\mb c}_{z_i^*}(\mb z) \big)}^2 }{\omega_{z_i^*, b} }
     \\
\stackrel{\eqref{eq: fact of trace operation}}{\leq} &  \frac{ K^2}{\underline \sigma_{\mathrm{cov}}^2 \underline \omega}\max_{b \in [K]}\bignorm{\bo \Lambda^* \mb O\t \big(\hat{\mb c}_b(\mb z^*) - \hat{\mb c}_b(\mb z)\big) }_2^2\Bignorm{\sum_{i\in[n]} \bo \Lambda^* \mk L_i \mk L_i\t \bo \Lambda^*} \lesssim \frac{\beta^2 K^4 \kappa_{\mathrm{cov}}^4 }{ \underline \omega^2} l(\mb z, \mb z^*) 
        }
        holds with probability at least $1 - O(d^{-c}) \vee e^{-(1+o(1))\frac{\snr^2}{2}}$. 

    \item We are left with controlling the third term and the fourth term in   \eqref{eq: F term}. Toward this, we use \eqref{eq: fact of trace operation} again to derive that 
    \begin{align}
        & \sum_{i\in[n]} \max_{b\in[K]\backslash \{z_i^*\} } \frac{\ip{
       \bo \Lambda^* \mathfrak L_i,  \big[(\bo \Lambda^* \mb O\t \hat{\bo\Omega}_k(\mb z^*) \mb O \bo \Lambda^*)^{-1} - (\bo \Lambda^* \mb O\t \hat{\bo\Omega}_k(\mb z^*) \mb O \bo \Lambda^*)^{-1}\big] \bo \Lambda^*\big(\tilde{\mb c}_{z_i^*}^*- \hat{\mb c}_b(\mb z^*) \big)}^2}{\omega_{z_i^*, b} } \\ 
       & + \sum_{i\in[n]}\max_{b\in[K]\backslash\{z_i^*\}} \frac{\ip{ \bo \Lambda^* \mathfrak L_i 
        ,\big[(\bo \Lambda^* \mb O\t \hat{\bo\Omega}_{z_i^*}(\mb z) \mb O \bo \Lambda^*)^{-1} - (\bo \Lambda^* \mb O\t \hat{\bo\Omega}_{z_i^*}(\mb z^*) \mb O \bo \Lambda^*)^{-1}\big] \bo \Lambda^*\big(\tilde{\mb c}_{z_i^*}^*- \mb O\t \hat{\mb c}_{z_i^*}(\mb z^*) \big) }^2}{\omega_{z_i^*, b}} \\ 
       \lesssim & \frac{ K}{\underline \omega}\underbrace{\max_{k\in[K]}\norm{(\bo \Lambda^* \mb O\t \hat{\bo \Omega}_k(\mb z) \mb O \bo \Lambda^*)^{-1} - (\bo \Lambda^* \mb O\t \hat{\bo \Omega}_k(\mb z^*) \mb O \bo \Lambda^*)^{-1}}^2}_{\text{controlled by Lemma~\ref{lemma: projected covariance matrix estimation error}}} \\ 
       & \qquad\cdot \Big( \underbrace{\max_{k\in[K]}\norm{\bo \Lambda^*\big(\tilde{\mb c}_{z_i^*}^* -  \mb O\t \hat{\mb c}_k(\mb z^*)\big)}_2^2}_{\text{controlled by Lemma~\ref{lemma: center estimation error}}} + \underbrace{\bar \sigma_{\mathrm{cov}}^2  \nu \underline \omega}_{\text{by Lemma~\ref{lemma: relation between tilde omega and omega}}}   \Big)  \cdot \underbrace{\Bignorm{\sum_{i\in[n]} \bo \Lambda^* \mk L_i \mk L_i\t \bo \Lambda^*} }_{\text{controlled by \eqref{eq: proof of lemma 1 first term decomposition 2}}}. 
    \end{align}

    Substitution of the upper bounds in Lemma~\ref{lemma: center estimation error},  Lemma~\ref{lemma: projected covariance matrix estimation error}, and \eqref{eq: proof of lemma 1 first term decomposition 2} into the above gives that 
 \begin{align}
     & \sum_{i\in[n]} \max_{b\in[K]\backslash \{z_i^*\} } \frac{\ip{
       \bo \Lambda^* \mathfrak L_i,  \big[(\bo \Lambda^* \mb O\t \hat{\bo\Omega}_k(\mb z^*) \mb O \bo \Lambda^*)^{-1} - (\bo \Lambda^* \mb O\t \hat{\bo\Omega}_k(\mb z^*) \mb O \bo \Lambda^*)^{-1}\big] \bo \Lambda^*\big(\tilde{\mb c}_{z_i^*}^*- \hat{\mb c}_b(\mb z^*) \big)}^2}{\omega_{z_i^*, b} } \\ 
        +& \sum_{i\in[n]}\max_{b\in[K]\backslash\{z_i^*\}} \frac{\ip{ \bo \Lambda^* \mathfrak L_i 
        ,\big[(\bo \Lambda^* \mb O\t \hat{\bo\Omega}_{z_i^*}(\mb z) \mb O \bo \Lambda^*)^{-1} - (\bo \Lambda^* \mb O\t \hat{\bo\Omega}_{z_i^*}(\mb z^*) \mb O \bo \Lambda^*)^{-1}\big] \bo \Lambda^*\big(\tilde{\mb c}_{z_i^*}^*- \mb O\t \hat{\mb c}_{z_i^*}(\mb z^*) \big) }^2}{\omega_{z_i^*, b}}  \\ 
       \lesssim & \frac{K}{\underline \omega}\Big[ \frac{1}{\underline\sigma_{\mathrm{cov}}^8} \kappa^4 \beta^6 K^4 \frac{l(\mb z, \mb z^*) K }{n} \bar \sigma_{\mathsf{cov}}^4 / \underline \omega\Big](\bar \sigma_{\mathsf{cov}}^2 \nu \underline \omega)( \bar \sigma_{\mathrm{cov}}^2 n ) \\ 
       \lesssim & \frac{1}{\underline\omega}\nu \kappa^4 \kappa_{\mathrm{cov}}^8 \beta^6 K^6 l(\mb z, \mb z^*) 
 \end{align}
with probability at least $1 - O(d^{-c}) \vee e^{-(1+o(1))\frac{\snr^2}{2}}$. 
    \end{itemize}

	In the end, collecting the above upper bounds together leads to the conclusion that 
 \longeq{
& \sum_{i\in[n]}\max_{b\in[K]\backslash \{z_i^*\}} \frac{F_i( b, \mb z)^2}{\omega_{z_i^*,b }^2 } \lesssim  \frac{1}{\underline\omega}\nu \kappa^4 \kappa_{\mathrm{cov}}^8 \beta^6 K^6 l(\mb z, \mb z^*)  = o(\frac{1}{\kappa_{\mathsf{cov}}^4})l(\mb z,\mb z^*)
 }
uniformly holds for all qualified $\mb z$ with probability at least $1 - O(d^{-c}) \vee e^{-(1+o(1))\frac{\snr^2}{2}}$, where we use \eqref{eq: alg condition} to derive the last equality.  
\subsubsection{Proof of Lemma~\ref{lemma: Bound for G_i}} 
We recap that $G_i(b,\mb z)$ is defined as 
\longeq{
& G_i(k,\mb z) =
     -\frac{1}{2}\ip{\bo \Lambda^* \mathfrak L_i, \big( (\bo \Lambda^* \mb O\t\hat{\bo \Omega}_k(\mb z) \mb O \bo \Lambda^*)^{-1} - (\bo \Lambda^* \mb O\t\hat{\bo \Omega}_k(\mb z^*) \mb O \bo \Lambda^*)^{-1}\big) \bo \Lambda^*\mathfrak L_i} \\ 
     & \qquad + \frac{1}{2}\ip{\bo \Lambda^* \mathfrak L_i, \big((\bo \Lambda^* \mb O\t \hat{\bo \Omega}_{z_i^*}(\mb z) \mb O \bo \Lambda^*)^{-1} -(\bo \Lambda^* \mb O\t \hat{\bo \Omega}_{z_i^*}(\mb z^*) \mb O \bo \Lambda^*)^{-1}\big) \bo \Lambda^* \mathfrak L_i}.
}
    We consider the upper bounds for the summation over the fourth moment of each term of $G_i(b,\mb z)$ separately. 
    \begin{itemize} 
      \item 
      Regarding the first term, by Lemma~\ref{lemma: noise matrix concentrations using the universality (Gaussian)} (or Lemma~\ref{lemma: noise matrix concentrations using the universality (bounded)}) and Lemma~\ref{lemma: projected covariance matrix estimation error}, we have 
      \begin{align}
        &  \sum_{i\in[n]} \max_{b\in[K], b \neq z_i^*}  \frac{\ip{\bo \Lambda^* \mathfrak L_i, \big( (\bo \Lambda^* \mb O\t\hat{\bo \Omega}_k(\mb z) \mb O \bo \Lambda^*)^{-1} - (\bo \Lambda^* \mb O\t\hat{\bo \Omega}_k(\mb z^*) \mb O \bo \Lambda^*)^{-1}\big) \bo \Lambda^*\mathfrak L_i}^4 }{\omega_{z_i^*, b}^3 } \\ 
        \lesssim & \frac{1 }{\underline \omega^3}\max_{b\in[K]} \bignorm{(\bo \Lambda^* \mb O\t\hat{\bo \Omega}_b(\mb z) \mb O \bo \Lambda^*)^{-1} - (\bo \Lambda^* \mb O\t\hat{\bo \Omega}_b(\mb z^*) \mb O \bo \Lambda^*)^{-1}}^4 \sum_{i\in[n]}\bignorm{\bo \Lambda^* \mathfrak L_i}_2^8 \\ 
        \lesssim &\frac{1}{\underline \sigma_{\mathrm{cov}}^{16}\underline \omega^3} \Big[\kappa^8 \beta^{12} K^{10} \frac{l(\mb z, \mb z^*)^2}{n^2} \bar \sigma_{\mathsf{cov}}^8 / \underline \omega^{2}   \Big] \cdot 
        \Big[ \frac{K^4 (\log d)^4}{n^3 (\sigma_K^*)^8}\deltaop^8 \Big]   \\ 
        \lesssim & \frac{\kappa^8 \kappa_{\mathsf{cov}}^{16} \beta^{11} K^{13} }{\underline \omega^5}l(\mb z, \mb z^*),
      \end{align}
      where we used the assumption $\frac{l(\mb z, \mb z^*) \beta K  (\log d)^4}{n} = o( 1)$ and \eqref{eq: fact A}.
      \item For the second term, a similar argument yields that
      \begin{align}
        &  \sum_{i\in[n]} \max_{b\in[K], b \neq z_i^*} \frac{\frac{1}{2}\ip{\bo \Lambda^* \mathfrak L_i, \big((\bo \Lambda^* \mb O\t \hat{\bo \Omega}_{z_i^*}(\mb z) \mb O \bo \Lambda^*)^{-1} -(\bo \Lambda^* \mb O\t \hat{\bo \Omega}_{z_i^*}(\mb z^*) \mb O \bo \Lambda^*)^{-1}\big) \bo \Lambda^* \mathfrak L_i}^4}{\omega_{z_i^*, b}^3}\\ 
        \lesssim &  \frac{1}{\underline \omega^3}\max_{b\in[K]} \bignorm{(\bo \Lambda^* \mb O\t\hat{\bo \Omega}_b(\mb z) \mb O \bo \Lambda^*)^{-1} - (\bo \Lambda^* \mb O\t\hat{\bo \Omega}_b(\mb z^*) \mb O \bo \Lambda^*)^{-1}}^4 \sum_{i\in[n]}\bignorm{\bo \Lambda^* \mathfrak L_i}_2^8\\ 
        \lesssim & \frac{\kappa^8 \kappa_{\mathsf{cov}}^{16} \beta^{11} K^{13} }{\underline \omega^5}l(\mb z, \mb z^*) 
      \end{align}
holds with probability at least $1 - O(d^{-c}) \vee e^{-(1+o(1))\frac{\snr^2}{2}}$. 
    \end{itemize}

    Taking the upper bounds collectively yields that 
    \longeq{
    & \sum_{i\in [n]} \max_{b \in[K] \backslash \{z_i^* \}} \frac{G_i(k, \mb z)^4}{\omega_{z_i^*, k}^3}\lesssim \frac{\kappa^8 \kappa_{\mathsf{cov}}^{16} \beta^{11} K^{13} }{\underline \omega^5}l(\mb z, \mb z^*)  = o(\frac{1}{\kappa_{\mathsf{cov}}^4}) l(\mb z,\mb z^*)
    }
    holds uniformly for all eligible $\mb z$ with probability at least $1 - O(d^{-c}) \vee e^{-(1+o(1))\frac{\snr^2}{2}}$ because of \eqref{eq: alg condition}. 
    
\subsubsection{Proof of Lemma~\ref{lemma: Bound for H_i}}

In what follows, we jointly parse the first two terms and the last two terms, respectively. 
    \begin{align}
      &H_i( b, \mb z) =  \alpha_1 + \alpha_2,\label{eq: proof of lemma 3 decomposition 1}
\end{align}
where $\alpha_1, \alpha_2$ are defined as 
\begin{align}
     \alpha_1 \coloneqq& -
    \frac{1}{2}\ip{\bo \Lambda^*\big(\tilde{\mb c}_{z_i^*}^*- \mb O\t \hat{\mb c}_{z_i^*}(\mb z) \big), (\bo \Lambda^* \mb O\t \hat{\bo \Omega}_{z_i^*}(\mb z) \mb O \bo \Lambda^*)^{-1} \bo \Lambda^*\big(\tilde{\mb c}_{z_i^*}^*- \mb O\t \hat{\mb c}_{z_i^*}(\mb z) \big) } \\ 
    & \qquad + \frac{1}{2} \ip{\bo \Lambda^*\big(\tilde{\mb c}_{z_i^*}^*- \mb O\t \hat{\mb c}_{z_i^*}(\mb z^*) \big), (\bo \Lambda^* \mb O\t \hat{\bo \Omega}_{z_i^*}(\mb z^*) \mb O \bo \Lambda^*)^{-1} \bo \Lambda^*\big(\tilde{\mb c}_{z_i^*}^*- \mb O\t \hat{\mb c}_{z_i^*}(\mb z^*) \big) },  \\ 
     \alpha_2 \coloneqq& \frac{1}{2}\ip{\bo \Lambda^*\big(\tilde{\mb c}_{z_i^*}^*- \hat{\mb c}_k(\mb z) \big), (\bo \Lambda^* \mb O\t \hat{\bo \Omega}_k(\mb z) \mb O \bo \Lambda^*)^{-1} \bo \Lambda^*\big(\tilde{\mb c}_{z_i^*}^*- \hat{\mb c}_k(\mb z) \big) } \\ 
    & \qquad - \ip{\bo \Lambda^*\big(\tilde{\mb c}_{z_i^*}^*- \hat{\mb c}_k(\mb z^*) \big), (\bo \Lambda^* \mb O\t \hat{\bo \Omega}_k(\mb z^*) \mb O \bo \Lambda^*)^{-1} \bo \Lambda^*\big(\tilde{\mb c}_{z_i^*}^*- \hat{\mb c}_k(\mb z^*) \big) }. 
\end{align}
We jointly parse the first two terms and the last two terms, respectively. 
\begin{itemize} 
    \item For the term $\alpha_1$ defined in \eqref{eq: proof of lemma 3 decomposition 1}, we  further break it down as follows
    \begin{align}
        &  \alpha_1 = \alpha_{1,1} + \alpha_{1,2}, \label{eq: H term decomposition 0}
        \end{align}
        where $\alpha_{1,1}$ and $\alpha_{1,2}$ are defined as
        \begin{align}
            \alpha_{1,1} \coloneqq &-
    \frac{1}{2}\ip{\bo \Lambda^*\big(\tilde{\mb c}_{z_i^*}^*- \mb O\t \hat{\mb c}_{z_i^*}(\mb z) \big), (\bo \Lambda^* \mb O\t \hat{\bo \Omega}_{z_i^*}(\mb z) \mb O \bo \Lambda^*)^{-1} \bo \Lambda^*\big(\tilde{\mb c}_{z_i^*}^*- \mb O\t \hat{\mb c}_{z_i^*}(\mb z) \big) } \\ 
    & \qquad + \frac{1}{2} \ip{\bo \Lambda^*\big(\tilde{\mb c}_{z_i^*}^*- \mb O\t \hat{\mb c}_{z_i^*}(\mb z^*) \big), (\bo \Lambda^* \mb O\t \hat{\bo \Omega}_{z_i^*}(\mb z) \mb O \bo \Lambda^*)^{-1} \bo \Lambda^*\big(\tilde{\mb c}_{z_i^*}^*- \mb O\t \hat{\mb c}_{z_i^*}(\mb z^*) \big) },\\
            \alpha_{1,2} \coloneqq & -
    \frac{1}{2}\ip{\bo \Lambda^*\big(\tilde{\mb c}_{z_i^*}^*- \mb O\t \hat{\mb c}_{z_i^*}(\mb z^*) \big), (\bo \Lambda^* \mb O\t \hat{\bo \Omega}_{z_i^*}(\mb z) \mb O \bo \Lambda^*)^{-1} \bo \Lambda^*\big(\tilde{\mb c}_{z_i^*}^*- \mb O\t \hat{\mb c}_{z_i^*}(\mb z^*) \big) } \\ 
    & \qquad + \frac{1}{2} \ip{\bo \Lambda^*\big(\tilde{\mb c}_{z_i^*}^*- \mb O\t \hat{\mb c}_{z_i^*}(\mb z^*) \big), (\bo \Lambda^* \mb O\t \hat{\bo \Omega}_{z_i^*}(\mb z^*) \mb O \bo \Lambda^*)^{-1} \bo \Lambda^*\big(\tilde{\mb c}_{z_i^*}^*- \mb O\t \hat{\mb c}_{z_i^*}(\mb z^*) \big) }.
        \end{align}

To upper bound $\alpha_{1,1}$, we make the observation that
\begin{align}
& \alpha_{1,1} 
    \leq  \norm{(\bo \Lambda^* \mb O\t \hat{\bo \Omega}_{z_i^*}(\mb z) \mb O \bo \Lambda^*)^{-1}}\bignorm{ \bo \Lambda^* \mb O\t \big(\hat{\mb c}_{z_i^*}(\mb z^*) - \hat{\mb c}_{z_i^*}(\mb z) \big)}_2 
    \\ 
    \cdot & \Big[\frac{1}{2}\bignorm{ \bo \Lambda^* \mb O\t \big(\hat{\mb c}_{z_i^*}(\mb z^*) - \hat{\mb c}_{z_i^*}(\mb z) \big)}_2  + \bignorm{\bo \Lambda^* \mb O\t \big(\mb O\t \hat{\mb c}_{z_i^*}(\mb z^*) - \tilde{\mb c}_{z_i^*}^* \big)}_2 \Big].\label{eq: H term decomposition 1}
\end{align}
Substitution of the results from Lemma~\ref{lemma: center estimation error} into \eqref{eq: H term decomposition 1} gives that 
\begin{align}
 \alpha_{1,1}\lesssim&  \frac{1}{\underline \sigma_{\mathrm{cov}}^2} \Big[\beta K\sqrt{\frac{l(\mb z,\mb z^*)}{n\underline\omega}} \frac{\deltaop}{\sqrt{n}\sigma_K^*} \Big] \cdot  \Big\{\beta K\sqrt{\frac{l(\mb z,\mb z^*)}{n\underline\omega}} \frac{\deltaop}{\sqrt{n}\sigma_K^*}+  \kappa \mu_1 \sqrt{ \beta } K \frac{\deltaop}{\sqrt{n} \sigma_K^*} \Big( \sqrt{\frac{\log d}{n}}  + \frac{\deltaop \snr }{(\sigma_K^*)^2} \Big) \Big\} 
        \\  
         \lesssim & \frac{\kappa \kappa_{\mathrm{cov}}^2 \beta^2 K^2  }{\sqrt{\underline\omega} }
        \label{eq: proof of lemma 3 decomposition 2}
    \end{align}
    uniformly holds with probability at least $1 - O(e^{-(1 + o(1))\frac{\snr^2}{2}} \vee d^{-c})$, provided \eqref{eq: fact A}, \eqref{eq: fact C}, \eqref{eq: fact D}, and 
    \begin{align}
        & \frac{l(\mb z, \mb z^*)\beta K}{n } = o(1), \qquad \frac{\beta K^2 \log d}{n }\lesssim 1 .
    \end{align}
    
  Regarding $\alpha_{1,2}$ in \eqref{eq: H term decomposition 0}, it is obvious by Lemmas~\ref{lemma: center estimation error}~and~\ref{lemma: projected covariance matrix estimation error} that
\begin{align}
    \alpha_{1,2} 
    \leq & \frac{1}{2}\bignorm{(\bo \Lambda^* \mb O\t \hat{\bo \Omega}_{z_i^*}(\mb z^*) \mb O \bo \Lambda^*)^{-1} - (\bo \Lambda^* \mb O\t \hat{\bo \Omega}_{z_i^*}(\mb z) \mb O \bo \Lambda^*)^{-1}} \bignorm{\bo \Lambda^* \big(\tilde{\mb c}_{z_i^*}^* - \mb O\t \hat{\mb c}_{z_i^*}(\mb z^*)\big)}_2^2 \\ 
    \lesssim & \frac{1}{\underline\sigma_{\mathrm{cov}}^4} \Big[\kappa^2 \beta^3 K^2 \sqrt{\frac{l(\mb z, \mb z^*) K }{n}} \bar \sigma_{\mathsf{cov}}^2 / \underline \omega^{\frac12} \Big]
    \Big[\beta^2 K^2 \frac{l(\mb z, \mb z^*)}{n\underline\omega} \frac{\deltaop^2}{n (\sigma_K^*)^2}  \Big]  \\ 
    \lesssim &  \frac{\kappa^2 \kappa_{\mathsf{cov}}^4 \beta^4 K^4 }{\underline\omega^{\frac32}}
    \label{eq: H term decomposition 2}
\end{align}
with probability at least $1 - O(e^{-(1 + o(1))\frac{\snr^2}{2}} \vee d^{-c})$.

Taking \eqref{eq: proof of lemma 3 decomposition 2} and \eqref{eq: H term decomposition 2} together into \eqref{eq: H term decomposition 0} yields that 
\begin{align}
    & \alpha_1 \lesssim \frac{\kappa \kappa_{\mathrm{cov}}^2 \beta^2 K^2  }{\sqrt{\underline\omega} } + \frac{\kappa^2 \kappa_{\mathsf{cov}}^4 \beta^4 K^4 }{\underline\omega^{\frac32}}
\end{align}
uniformly holds with probability at least $1 - O(e^{-(1 + o(1))\frac{\snr^2}{2}} \vee d^{-c})$. 

    \item 
    Now it remains to upper bound the third term and the fourth term. A similar decomposition to \eqref{eq: H term decomposition 0} gives that 
    \begin{align}
     \alpha_2 = & \alpha_{2,1} + \alpha_{2,2},
    \end{align}
    where $\alpha_{2,1}$ and $\alpha_{2,2}$ are defined as 
    \begin{align}
        \alpha_{2,1} \coloneqq & \frac{1}{2}\ip{\bo \Lambda^*\big(\tilde{\mb c}_{z_i^*}^*- \mb O\t \hat{\mb c}_k(\mb z) \big), (\bo \Lambda^* \mb O\t \hat{\bo \Omega}_k(\mb z) \mb O \bo \Lambda^*)^{-1} \bo \Lambda^*\big(\tilde{\mb c}_{z_i^*}^*- \mb O\t \hat{\mb c}_k(\mb z) \big) } \\ 
    & \qquad - \ip{\bo \Lambda^*\big(\tilde{\mb c}_{z_i^*}^*- \mb O\t \hat{\mb c}_k(\mb z^*) \big), (\bo \Lambda^* \mb O\t \hat{\bo \Omega}_k(\mb z) \mb O \bo \Lambda^*)^{-1} \bo \Lambda^*\big(\tilde{\mb c}_{z_i^*}^*- \mb O\t \hat{\mb c}_k(\mb z^*) \big) } \\ 
         \alpha_{2,2} \coloneqq & \frac{1}{2}\ip{\bo \Lambda^*\big(\tilde{\mb c}_{z_i^*}^*- \mb O\t \hat{\mb c}_k(\mb z^*) \big), (\bo \Lambda^* \mb O\t \hat{\bo \Omega}_k(\mb z) \mb O \bo \Lambda^*)^{-1} \bo \Lambda^*\big(\tilde{\mb c}_{z_i^*}^*- \mb O\t \hat{\mb c}_k(\mb z^*) \big) } \\ 
    & \qquad - \ip{\bo \Lambda^*\big(\tilde{\mb c}_{z_i^*}^*- \mb O\t \hat{\mb c}_k(\mb z^*) \big), (\bo \Lambda^* \mb O\t \hat{\bo \Omega}_k(\mb z^*) \mb O \bo \Lambda^*)^{-1} \bo \Lambda^*\big(\tilde{\mb c}_{z_i^*}^*- \mb O\t \hat{\mb c}_k(\mb z^*) \big) }
    \end{align}

    For $\alpha_{2,1}$, it can be bounded as follows with probability at least $1 - O(e^{-(1 + o(1))\frac{\snr^2}{2}} \vee d^{-c})$.  
    \begin{align}
         \alpha_{2,1} 
         \lesssim  &  \frac{1}{\underline\sigma_{\mathrm{cov}}^2 } \bignorm{\bo \Lambda^* \mb O\t\big(\hat{\mb c}_k(\mb z) -   \hat{\mb c}_k(\mb z^*) \big)}_2 \Big[\bignorm{\bo \Lambda^* \mb O\t\big(\hat{\mb c}_k(\mb z) -   \hat{\mb c}_k(\mb z^*) \big)}_2 +\bignorm{\bo \Lambda^* \big(\tilde{\mb c}_{z_i^*}^*- \mb O\t\hat{\mb c}_k(\mb z^*) \big)}_2\\ 
         & \qquad + \bignorm{\bo \Lambda^* \big(\tilde{\mb c}_{z_i^*}^*-   \tilde{\mb c}_{z_i^*}^*\big)}_2\Big] \\
         \lesssim& \frac{1}{\underline\sigma_{\mathrm{cov}}^2} \beta K\sqrt{\frac{l(\mb z,\mb z^*)}{n\underline\omega}} \frac{\deltaop}{\sqrt{n}\sigma_K^*} \nu \bar \sigma_{\mathrm{cov}}\underline\omega^{\frac12}  
         \lesssim  \frac{\nu \kappa_{\mathrm{cov}} \beta^{\frac12} K^{\frac12} \deltaop}{\underline\sigma_{\mathrm{cov}} \sqrt{n}\sigma_K^*} \lesssim \nu \kappa_{\mathrm{cov}}^2 \beta^{\frac12} K^{\frac12}  ,  \label{eq: alpha 2 1 decomposition}
    \end{align}
    where we used an argument similar to \eqref{eq: H term decomposition 1} together with \eqref{eq: fact A} and Lemma~\ref{lemma: relation between tilde omega and omega}.

    Similarly, we could upper bound $\alpha_{2,2}$ using Lemmas~\ref{lemma: center estimation error}~and~\ref{lemma: projected covariance matrix estimation error}, that
    \begin{align}
        \alpha_{2,2} =&  \frac{1}{2}\ip{\bo \Lambda^*\big(\tilde{\mb c}_{z_i^*}^*- \mb O\t \hat{\mb c}_k(\mb z^*) \big), (\bo \Lambda^* \mb O\t \hat{\bo \Omega}_k(\mb z) \mb O \bo \Lambda^*)^{-1} \bo \Lambda^*\big(\tilde{\mb c}_{z_i^*}^*- \mb O\t \hat{\mb c}_k(\mb z^*) \big) } \\ 
    & \qquad - \ip{\bo \Lambda^*\big(\tilde{\mb c}_{z_i^*}^*- \mb O\t \hat{\mb c}_k(\mb z^*) \big), (\bo \Lambda^* \mb O\t \hat{\bo \Omega}_k(\mb z^*) \mb O \bo \Lambda^*)^{-1} \bo \Lambda^*\big(\tilde{\mb c}_{z_i^*}^*- \mb O\t \hat{\mb c}_k(\mb z^*) \big) } \\ 
        \lesssim & \bignorm{\bo \Lambda^*\big(\tilde{\mb c}_{z_i^*}^*- \mb O\t \hat{\mb c}_k(\mb z^*) \big)}_2^2 \bignorm{(\bo \Lambda^* \mb O\t \hat{\bo \Omega}_k(\mb z) \mb O \bo \Lambda^*)^{-1} - (\bo \Lambda^* \mb O\t \hat{\bo \Omega}_k(\mb z^*) \mb O \bo \Lambda^*)^{-1}} \\ 
        \lesssim & \frac{1}{\underline \sigma_{\mathrm{cov}}^4} \Big[\kappa^2 \beta^3 K^2 \sqrt{\frac{l(\mb z, \mb z^*) K }{n}} \bar \sigma_{\mathsf{cov}}^2 / \underline \omega^{\frac12} \Big]\nu^2 \bar \sigma_{\mathrm{cov}}^2 \underline \omega  \\ 
        \lesssim & \nu^2 \kappa_{\mathrm{cov}}^4 \kappa^2\beta^{\frac52} K^2 \omega^{\frac12} 
    \end{align}
    holds with probability at least $1 - O(e^{-(1 + o(1))\frac{\snr^2}{2}} \vee d^{-c})$.  
    \end{itemize}

    Combining these pieces together leads to the conclusion that 
    \begin{align}
        &\max_{i\in[n]} \max_{b\in[K], b\neq z_i^*}H_i(z_i^*, b, \mb z) \lesssim \frac{\kappa \kappa_{\mathrm{cov}}^2 \beta^2 K^2  }{\sqrt{\underline\omega} } + \frac{\kappa^2 \kappa_{\mathsf{cov}}^4 \beta^4 K^4 }{\underline\omega^{\frac32}} + \nu \kappa_{\mathrm{cov}}^2 \beta^{\frac12} K^{\frac12}  + \nu^2 \kappa_{\mathrm{cov}}^4 \kappa^2\beta^{\frac52} K^2 \omega^{\frac12}  = o(\frac{\underline\omega}{\kappa_{\mathsf{cov}}^4})
    \end{align}
    uniformly holds for all qualified $\mb z$ with probability at least $1 - O(e^{-(1 + o(1))\frac{\snr^2}{2}} \vee d^{-c})$, where we use the condition \eqref{eq: alg condition}.

\subsection{Proofs of Auxiliary Lemmas in the Iterative Charaterization}
In this section, we present the proofs for the concentration results (Lemmas~\ref{lemma: center estimation error}~and~\ref{lemma: projected covariance matrix estimation error}). These results are instrumental in establishing the covariance and center estimations consistency, which in turn support proving Theorems~\ref{theorem: upper bound for algorithm (gaussian)}~and~\ref{theorem: upper bound for algorithm (bounded)}.

\subsubsection{Proof of Lemma~\ref{lemma: center estimation error}}
\label{subsec: proof of lemma misspecified center estimation}
    We obtain from Lemmas~\ref{lemma: noise matrix concentrations using the universality (Gaussian)}~and~\ref{lemma: noise matrix concentrations using the universality (bounded)} together with Corollary~\ref{corollary: singular subspace perturbation theory in clustering problem} that 
    \begin{align}
        & \Bignorm{\sum_{i\in \mc I_k(\mb z^*)}\bo \Lambda^* \mk L_i}_2 / n_k \leq \frac{1}{n_k}\norm{\mb 1_{n_k}\t \mb E_{\mc I_k(\mb z^*), \cdot} \mb V^*}_2 + \frac{1}{n_k}\norm{\mb 1_{n_k}\t \mc H(\mb E\mb E\t)_{\mc I_k(\mb z^*), \cdot} \mb U^*{\bo \Lambda^*}^{-1}} \\ 
         \lesssim & (n_k)^{-\frac12} \bar \sigma \sqrt{K\log d} + (n_k)^{-\frac12} (\sigma_K^*)^{-1} \tilde\sigma \sigma \sqrt{pK\log d} , 
        \label{eq: sum of linear terms in center estimation proof}
        \\ 
        & \norm{\bo \Lambda^*\mb O\t \hat{\mb c}_k(\mb z^*) - \tilde{\mb w}_k^*}_2 \lesssim \frac{1}{n_k}\norm{\mb 1_{n_k}\t \mb E_{\mc I_k(\mb z^*), \cdot} \mb V^*}_2 + \frac{1}{n_k}\norm{\mb 1_{n_k}\t \mc H(\mb E\mb E\t)_{\mc I_k(\mb z^*), \cdot} \mb U^*{\bo \Lambda^*}^{-1}} \\ 
        & \qquad + \frac{1}{n_k} \sum_{i\in \mc I_k(\mb z^*)} \norm{\big(\mb U_{i,\cdot} \mb O - \tilde{\mb c}_k^* - \mk L_i\t \big) \bo \Lambda^*}_2 
        \\ 
        \lesssim & (n_k)^{-\frac12} \bar \sigma \sqrt{K\log d} + (n_k)^{-\frac12} (\sigma_K^*)^{-1} \tilde\sigma \sigma \sqrt{pK\log d} + \sigma_1^* \deltau
        \\ 
        \lesssim &  \kappa \mu_1 \sqrt{ \beta } K \frac{\deltaop}{\sqrt{n} \sigma_K^*} \Big( \sqrt{\frac{\log d}{n}}  + \frac{\deltaop \snr }{(\sigma_K^*)^2} \Big)
    \end{align}
    with probability at least $1 - O\big(d^{-c} \vee \exp(-(1 + o(1))\frac{\snr^2}{2})\big)$. 

    For the second part, for each $k\in[K]$, we begin by decomposing it into five terms which are controlled in the following: 
    \begin{align}
        & \norm{\bo \Lambda^* \mb O\t \big(\hat{\mb c}_k(\mb z) - \hat{\mb c}_k(\mb z^*) \big)}_2 \leq \Bignorm{\mb{V^*}^{\top} \big(\sum_{i\in \mc I_k(\mb z)}\mb{E}_i/ n_k(\mb z)- \sum_{i\in \mc I_k(\mb z^*)} \mb{E}_i / n_k \big)}_2 \\ 
        & \qquad + \Bignorm{{\bo \Lambda^*}^{-1}{\mb U^*}\t \big( \sum_{i \in \mc I_k(\mb z)}\mc H(\mb E\mb E\t)_{\cdot, i} / n_k(\mb z) -  \sum_{i \in \mc I_k(\mb z^*)}\mc H(\mb E\mb E\t)_{\cdot, i} / n_k\big)}_2 \\ 
        & \qquad + \Bignorm{\sum_{i\in\mc I_k(\mb z^*)}\bo \Lambda^*  \big(\mb O\t \mb U_i - \mk L_i - \mb c_{z_i^*}\big) / n_k - \sum_{i\in\mc I_k(\mb z)} \bo \Lambda^*  \big(\mb O\t \mb U_i - \mk L_i - \mb c_{z_i^*}\big) / n_k(\mb z) }_2
        \\ 
        & \qquad + \Bignorm{\sum_{i\in \mc I_k(\mb z)}\tilde{\mb w}_k / n_k(\mb z) - \tilde{\mb w}_k^* }_2 .
        \label{eq: decomposition of center estimation error with misspecification}
     \end{align}

    \begin{itemize} 
        \item For the first term $\Bignorm{\mb{V^*}^{\top} \big(\sum_{i\in \mc I_k(\mb z)}\mb{E}_i/ n_k(\mb z)- \sum_{i\in \mc I_k(\mb z^*)} \mb{E}_i / n_k \big)}_2$, we first further decompose it into two terms: 
            \begin{align}
                &\Bignorm{\mb{V^*}^{\top} \big(\sum_{i\in \mc I_k(\mb z)}\mb{E}_i/ n_k(\mb z)- \sum_{i\in \mc I_k(\mb z^*)} \mb{E}_i / n_k \big)}_2 \\ 
                \leq & 
            \Bignorm{\mb{V^*}^{\top} \big(\sum_{i\in \mc I_k(\mb z)}\mb{E}_i- \sum_{i\in \mc I_k(\mb z^*)} \mb{E}_i \big)}_2 / n_k(\mb z) +  \Bignorm{\mb{V^*}^{\top} \big(\sum_{i\in \mc I_k(\mb z^*)} \mb{E}_i / n_k(\mb z) - \sum_{i\in \mc I_k(\mb z^*)} \mb{E}_i / n_k \big)}_2. \label{eq: first term in difference between thetakz and thetakz* decomposition} 
            \end{align}

            We bound the above two terms separately, where we have: 
            \begin{align}
            &\Bignorm{\mb{V^*}^{\top} \big(\sum_{i\in \mc I_k(\mb z)}\mb{E}_i/ n_k(\mb z)- \sum_{i\in \mc I_k(\mb z^*)} \mb{E}_i / n_k(\mb z) \big)}_2 
            \overset{\text{(i)}}{\lesssim} \frac{\beta K}{n}\cdot\big(\bar \sigma  \sqrt{n } \big) \cdot \sqrt{nh(\mb z, \mb z^*)} 
             \overset{\text{(ii)}}{\lesssim}  \frac{\bar \sigma \beta K}{\sqrt{n \underline \omega }}\sqrt{l(\mb z,\mb z^*)}, 
             \label{eq: first term of center difference}
             \\   
             & \Big\lv{\mb V^{*\top}} \Big(\frac{\sum_{i\in[n]}\ind{\{z_i^*=k\}}\mb{E}_i}{\sum_{i\in[n]}\ind{\{z_i=k\}}} - \frac{\sum_{i\in[n]}\ind{\{z^*_i = k\}}\mb E_i}{\sum_{i\in[n]}\ind{\{z_i^*=k\}} }\Big)\Big\rv_2  
             \lesssim  \frac{nh(\mb z,\mb z^*)}{n_k^2}\cdot \big(\bar \sigma \sqrt{n_k}\big) = \frac{\beta^{\frac32}K^{\frac32}l(\mb z,\mb z^*)}{n^{\frac32}\underline \omega}\label{eq: second term of center difference}
            \end{align}
            uniformly hold with probability at least $1 - O(d^{-c})$ for all $\mb z^*$ satisfying the condition that 
            where (i) arises since 
            $\sum_{i\in[n]}\ind{\{z_i=k\}} \geq n_k -n h(\mb z, \mb z^*) \gtrsim  \frac{n}{K\beta}$
             given $\omega(nh(\mb z, \mb z^*)) = l(\mb z, \mb z^*) \lesssim \frac{n}{\beta K}$  and $
            \Big\lv{\mb V^{*\top}} \sum_{i\in[n]}\big(\ind{\{z_i =k, z_i^* \neq k\}} +\ind{\{z_i \neq k, z_i^* = k\}}\big)\mb E_i\Big\rv_2\leq  \sqrt{nh(\mb z, \mb z^*) }\norm{{\mb V^{*\top}} \mb E} \lesssim \big(\bar \sigma \sqrt{n}\big) \cdot  \sqrt{nh(\mb z, \mb z^*)}
            $, 
            and (ii) holds since $nh(\mb z, \mb z^*) \leq \frac{l(\mb z, \mb z^*)}{\underline \omega}$. 

            Taking \eqref{eq: first term of center difference} and \eqref{eq: second term of center difference} collectively into \eqref{eq: first term in difference between thetakz and thetakz* decomposition} implies that 
            \begin{align}
                & \Bignorm{\mb{V^*}^{\top} \big(\sum_{i\in \mc I_k(\mb z)}\mb{E}_i/ n_k(\mb z)- \sum_{i\in \mc I_k(\mb z^*)} \mb{E}_i / n_k \big)}_2 
                \lesssim   \frac{\bar \sigma \beta K}{\sqrt{n \underline \omega }}\sqrt{l(\mb z,\mb z^*)} + \frac{\beta^{\frac32}K^{\frac32}l(\mb z,\mb z^*)}{n^{\frac32}\underline \omega} \lesssim \frac{\bar \sigma \beta K}{\sqrt{n \underline \omega }}\sqrt{l(\mb z,\mb z^*)}
            \end{align}
            holds with probability at least $1- O(d^{-c})$, 
            where the last inequality follows from the fact that $\frac{\beta^{\frac32}K^{\frac32}l(\mb z,\mb z^*)}{n^{\frac32}\underline \omega} \lesssim \frac{\beta^{\frac32}K^{\frac32}l(\mb z,\mb z^*)}{n^{\frac32}\sqrt{\underline \omega}} \lesssim \frac{\bar \sigma \beta K}{\sqrt{n \underline \omega }}\sqrt{l(\mb z,\mb z^*)}$ noticing $l(\mb z, \mb z^*) = o( \frac{n}{\beta K})$.

            \item Regarding the second term $ \Bignorm{\bo \Lambda^*{\mb U^*}\t \big( \sum_{i \in \mc I_k(\mb z)}\mc H(\mb E\mb E\t)_{\cdot, i} / n_k(\mb z) -  \sum_{i \in \mc I_k(\mb z^*)}\mc H(\mb E\mb E\t)_{\cdot, i} / n_k\big)}_2$, we deduce from similar argument to the one for the first term together with the concentrations for $\norm{{\bo \Lambda^*}^{-1} {\mb U^*}\t \mc H(\mb E \mb E\t)_{\cdot, \mc I_k(\mb z^*)} \mb 1_{n_k}}$ and $\norm{{\bo \Lambda^*}^{-1} {\mb U^*}\t \mc H(\mb E \mb E\t) }$ that 
            \begin{align}
                & \Bignorm{{\bo \Lambda^*}^{-1}{\mb U^*}\t \big( \sum_{i \in \mc I_k(\mb z)}\mc H(\mb E\mb E\t)_{\cdot, i} / n_k(\mb z) -  \sum_{i \in \mc I_k(\mb z^*)}\mc H(\mb E\mb E\t)_{\cdot, i} / n_k\big)}_2 \\ 
                \lesssim & \frac{1}{n_k(\mb z)}\Bignorm{{\bo \Lambda^*}^{-1}{\mb U^*}\t \big( \sum_{i \in \mc I_k(\mb z)}\mc H(\mb E\mb E\t)_{\cdot, i} -  \sum_{i \in \mc I_k(\mb z^*)}\mc H(\mb E\mb E\t)_{\cdot, i} \big)}_2 \\ 
                & \qquad + \Bignorm{{\bo \Lambda^*}^{-1}{\mb U^*}\t \big( \frac{1}{n_k(\mb z)}\sum_{i \in \mc I_k(\mb z^*)}\mc H(\mb E\mb E\t)_{\cdot, i}  - \frac{1}{n_k} \sum_{i \in \mc I_k(\mb z^*)}\mc H(\mb E\mb E\t)_{\cdot, i} \big)}_2 \\ 
                \lesssim & \frac{\beta K}{n\sigma_K^*} \big(\tilde\sigma^2 n + \tilde\sigma \sigma \sqrt{pn} \big) \sqrt{nh(\mb z^*, \mb z)} + \frac{nh(\mb z, \mb z^*)}{n_k^2} \frac{1}{\sqrt{n_k} \sigma_K^*} \tilde\sigma \sqrt{p K\log d}   \\ 
                \lesssim &   \frac{\beta K}{\underline 
                \omega^{\frac12} \sigma_K^*} \big(\tilde\sigma^2  + \tilde\sigma \sigma \sqrt{p/n} \big) \sqrt{nl(\mb z^*, \mb z)}
            \end{align}
            with probability at least $1- O(d^{-c})$. 

            \item For the third term in \eqref{eq: decomposition of center estimation error with misspecification}, it follows from applying Corollary~\ref{corollary: singular subspace perturbation theory in clustering problem} to all rows together with \eqref{eq: alternative upper bound for delta_u'} that with probability at least $1 - O(d^{-c})$ that 
            \begin{align}
                  & \Bignorm{\sum_{i\in\mc I_k(\mb z^*)}\bo \Lambda^*  \big(\mb O\t \mb U_i - \mk L_i - \mb c_{z_i^*}\big) / n_k - \sum_{i\in\mc I_k(\mb z)} \bo \Lambda^*  \big(\mb O\t \mb U_i - \mk L_i - \mb c_{z_i^*}\big) / n_k(\mb z) }_2 \\ 
                  \leq &  \Bignorm{\sum_{i\in\mc I_k(\mb z^*)}\bo \Lambda^*  \big(\mb O\t \mb U_i - \mk L_i - \mb c_{z_i^*}\big) / n_k(\mb z) - \sum_{i\in\mc I_k(\mb z)} \bo \Lambda^*  \big(\mb O\t \mb U_i - \mk L_i - \mb c_{z_i^*}\big) / n_k(\mb z) }_2 \\ 
                  & + \Bignorm{\sum_{i\in\mc I_k(\mb z^*)}\bo \Lambda^*  \big(\mb O\t \mb U_i - \mk L_i - \mb c_{z_i^*}\big)} \Big( \frac{1}{n_k(\mb z)} -\frac{1}{n_k} \Big)
                    \\ 
                    \lesssim & \frac{\beta K l(\mb z, \mb z^*)}{n \underline\omega} \mu_1^{\frac12} K \frac{\deltaop \sqrt{\log d}}{\sqrt{n}(\sigma_K^*)} \lesssim \beta K\sqrt{\frac{ l(\mb z, \mb z^*)}{n \underline\omega} } \frac{\deltaop}{\sqrt{n}(\sigma_K^*)}. 
            \end{align}

            \item
            For the deterministic part $\Bignorm{\sum_{i\in \mc I_k(\mb z)}\tilde{\mb w}_k / n_k(\mb z) - \tilde{\mb w}_k^* }_2 $, applying the H\"{o}lder's inequality yields that 
            \longeq{
            &\Bignorm{\sum_{i\in \mc I_k(\mb z)}\tilde{\mb w}_k / n_k(\mb z) - \tilde{\mb w}_k^* }_2 
            \stackrel{\text{by Lemma~\ref{lemma: relation between tilde omega and omega}}}{\lesssim} \frac{1}{n_k(\mb z)} \Bignorm{\sum_{i\in \mc I_k(\mb z)}\big(\mb w_{z_i^*}^* - \mb w_k^*\big) }_2 \\ 
            \leq & \frac{\bar \sigma_{\mathsf{cov}}\beta K}{n} \Bignorm{{\mb S_k^*}^{-\frac12}\sum_{i\in[n]}\ind{\{z_i=k\}}\big(\mb w_{z_i^*}^* - \mb w_k^*\big) }_2
            \\
            \stackrel{(\text{i})}{\leq}& \frac{\bar \sigma_{\mathsf{cov}}\beta K}{n}\Big(\sum_{i\in[n]}\ind{\{z_i^* =k, z_i \neq k\}}\Big)^{\frac{1}{2}}  \Big( \sum_{i\in[n]}\ind{\{z_i^* \neq k, z_i = k\}}\ip{\big(\mb w^*_{z_i} - \mb w^*_{z_i^*}\big), {\mb S_k^*}^{-1}\big(\mb w^*_{z_i} - \mb w^*_{z_i^*}\big)}\Big)^{\frac{1}{2}}\\ 
            \stackrel{\text{(ii)}}{\lesssim} & \frac{\bar\sigma_{\mathsf{cov}} \beta K}{n\underline \omega^{\frac{1}{2}}} \big( \sum_{i\in[n]}\ind{\{z_i^* =k, z_i \neq k\}}\ip{\big(\mb w^*_{z_i} - \mb w^*_{z_i^*}\big), {\mb S_k^*}^{-1} \big(\mb w^*_{z_i} - \mb w^*_{z_i^*}\big)}\big)
            \lesssim  \frac{\bar \sigma_{\mathsf{cov}} \beta K}{n \underline \omega^{\frac{1}{2}}}l(\mb z, \mb z^*),
            \label{eq: second term of center difference deterministic}
            }
            where (i) holds by the H\"{o}lder's inequality and (ii) holds since $\underline \omega \leq \ip{\big(\mb w^*_{z_i} - \mb w^*_{z_i^*}\big), {\mb S_k^*}^{-1} \big(\mb w^*_{z_i} - \mb w^*_{z_i^*}\big)}$ for $z_i^* =k, z_i \neq k$ by definition. 
    \end{itemize}

    Combining these pieces together with \eqref{eq: decomposition of center estimation error with misspecification} gives  with probability exceeding $1 - O(e^{-(1 + o(1))\frac{\snr^2}{2}} \vee d^{-c})$ that 
    \longeq{
         \norm{\bo \Lambda^* \mb O\t \big(\hat{\mb c}_k(\mb z) - \hat{\mb c}_k(\mb z^*) \big)}_2 \lesssim &  \frac{\beta K\sqrt{l(\mb z,\mb z^*)} }{\sqrt{ \underline \omega }}\Big(\bar \sigma n^{-\frac12}  + (\tilde\sigma^2 + \tilde\sigma \sigma \sqrt{p/n}) (\sigma_K^*)^{-1} \Big) + \frac{\bar \sigma_{\mathsf{cov}} \beta K}{n \underline \omega^{\frac{1}{2}}}l(\mb z, \mb z^*)\\ 
         \lesssim & \beta K \sqrt{\frac{l(\mb z, \mb z^*)}{n\underline\omega}} \frac{\deltaop}{\sqrt{n} \sigma_K^*} .
    }

        \subsubsection{Proof of Lemma~\ref{lemma: projected covariance matrix estimation error}}
        \label{sec: proof of projected covariance estimation error}
               
        Invoking the definition of $\hat{\mb S}_a(\mb z)$ and $\hat{\mb S}_a(\mb z^*)$, it first follows by the triangle inequality that 
               \longeq{\label{eq: covariance matrix decomposition 2}
                &\norm{\bo \Lambda^*\mb O\t \big(\hat{\bo \Omega}_a(\mb z) - \hat{\bo \Omega}_a(\mb z^*)\big)\mb O\bo \Lambda^*} \\ 
                = & \Big\| \sum_{i\in \mc I_a(\mb z)} \bo \Lambda^* \mb O\t\big(\mb U_i - \hat{\mb c}_a(\mb z)\big)\t \big(\mb U_i - \hat{\mb c}_a(\mb z)\big)\mb O\bo \Lambda^* / n_a(\mb z)  \\ 
                & \qquad -  \sum_{i\in \mc I_a(\mb z^*)} \bo \Lambda^* \mb O\t\big(\mb U_i - \hat{\mb c}_a(\mb z^*)\big)\t \big(\mb U_i - \hat{\mb c}_a(\mb z^*)\big)\mb O\bo \Lambda^* / n_a \Big\| \\
                \leq & \frac{1}{n_a(\mb z)} \Bignorm{ \sum_{i\in \mc I_a(\mb z)} \bo \Lambda^* \mb O\t\big(\mb U_i - \hat{\mb c}_a(\mb z)\big) \big(\mb U_i - \hat{\mb c}_a(\mb z)\big)\t \mb O\bo \Lambda^* \\ 
                & \qquad -  \sum_{i\in \mc I_a(\mb z^*)} \bo \Lambda^* \mb O\t\big(\mb U_i - \hat{\mb c}_a(\mb z^*)\big)\big(\mb U_i - \hat{\mb c}_a(\mb z^*)\big)\t \mb O\bo \Lambda^* }  \\ 
                &  +  \Big| \frac{1}{n_a(\mb z)} - \frac{1}{n_a} \Big| \Bignorm{\sum_{i\in \mc I_a(\mb z^*)} \bo \Lambda^* \mb O\t\big(\mb U_i - \hat{\mb c}_a(\mb z^*)\big)\t \big(\mb U_i - \hat{\mb c}_a(\mb z^*)\big)\mb O\bo \Lambda^*}
                 \\
                 \leq & \alpha_1 + \alpha_2 + \alpha_3 + \alpha_4,
               }
               where $\alpha_1, \alpha_2, \alpha_3$, and $\alpha_4$ are defined as 
               \begin{align}
                \alpha_1  \coloneqq & \Big\|\sum\limits_{i\in[n], \mc I_a(\mb z) \cap \mc I_a(\mb z^*)} \Big[ \bo \Lambda^* \mb O\t \big( \mb U_i  - \hat{\mb c}_a(\mb z)\big) \big( \mb U_i  - \hat{\mb c}_a(\mb z)\big)\t \mb O \bo \Lambda^*
                \\ 
                & \qquad 
                -  \bo \Lambda^* \mb O\t \big( \mb U_i  - \hat{\mb c}_a(\mb z^*)\big) \big( \mb U_i  - \hat{\mb c}_a(\mb z^*)\big)\t \mb O \bo \Lambda^* \Big] \Big\| / n_a(\mb z), \\ 
                \alpha_2 \coloneqq & \Bignorm{\sum\limits_{\mc I_k(\mb z^*) \backslash \mc I_k(\mb z) } \bo \Lambda^* \mb O\t \big( \mb U_i  - \hat{\mb c}_a(\mb z^*)\big) \big( \mb U_i  - \hat{\mb c}_a(\mb z^*)\big)\t \mb O \bo \Lambda^*   } / n_a(\mb z) ,\\ 
                 \alpha_3 \coloneqq & \Bignorm{\sum\limits_{\mc I_k(\mb z) \backslash \mc I_k(\mb z^*) } \bo \Lambda^* \mb O\t \big( \mb U_i  - \hat{\mb c}_a(\mb z)\big) \big( \mb U_i  - \hat{\mb c}_a(\mb z)\big)\t \mb O \bo \Lambda^*   } / n_a(\mb z), \\ 
                 \alpha_4 \coloneqq &  \Big| \frac{1}{n_a(\mb z)} - \frac{1}{n_a} \Big|  \Bignorm{\sum_{i\in \mc I_a(\mb z^*)} \bo \Lambda^* \mb O\t\big(\mb U_i - \hat{\mb c}_a(\mb z^*)\big)\t \big(\mb U_i - \hat{\mb c}_a(\mb z^*)\big)\mb O\bo \Lambda^*}. 
                   \end{align}
        
                With the above decomposition, we then turn to bounding $\alpha_1, \alpha_2, \alpha_3$, and $\alpha_4$ separately: 
                \begin{itemize} 
               \item
               Regarding $\alpha_1$, a direct application of the triangle inequality yields that 
               \longeq{
                & \alpha_1 =\frac{1}{ n_a(\mb z) }  \bigg\|\sum\limits_{i\in  \mc I_a(\mb z) \cap \mc I_a(\mb z^*)} \Big\{ \bo \Lambda^* \mb O\t \big[ (\mb U_i - \hat{\mb c}_a(\mb z^*)) + (\hat{\mb c}_a(\mb z^*) - \hat{\mb c}_a(\mb z))\big] 
                \\ 
                &   \cdot \big[ (\mb U_i - \hat{\mb c}_a(\mb z^*)) + (\hat{\mb c}_a(\mb z^*) - \hat{\mb c}_a(\mb z))\big]\t \mb O \bo \Lambda^*
                -  \bo \Lambda^* \mb O\t \big( \mb U_i  - \hat{\mb c}_a(\mb z^*)\big) \big( \mb U_i  - \hat{\mb c}_a(\mb z^*)\big)\t \mb O \bo \Lambda^* \Big\} \bigg\|
                \\
                \leq & \frac{|\mc I_a(\mb z) \cap \mc I_a(\mb z^*)| }{ n_k(\mb z)}\bignorm{\bo \Lambda^* \mb O\t \big(\hat{\mb c}_a(\mb z) - \hat{\mb c}_a(\mb z^*)\big)}_2^2\\ 
                & +  \frac{2}{n_a(\mb z)}\Bignorm{\sum\limits_{i\in  \mc I_a(\mb z) \cap \mc I_a(\mb z^*)}
                \bo \Lambda^* \mb O\t  \big( \mb U_i - \hat{\mb c}_a(\mb z^*) \big)\big(\hat{\mb c}_a(\mb z^*) - \hat{\mb c}_a(\mb z) \big)\t \mb O\bo \Lambda^* 
                }
               \\ 
               \lesssim & \bignorm{\bo \Lambda^* \mb O\t \big(\hat{\mb c}_a(\mb z) - \hat{\mb c}_a(\mb z^*)\big)}_2 \Big(\bignorm{\bo \Lambda^* \mb O\t \big(\hat{\mb c}_a(\mb z) - \hat{\mb c}_a(\mb z^*)\big)}_2 \\ 
               &  
               + \bignorm{\sum_{i\in \mc I_a(\mb z) \cap \mc I_a(\mb z^*)} \bo \Lambda^* \big(\mb O\t \mb U_i - \tilde{\mb c}_a^* \big)}_2 / n_a(\mb z) +\bignorm{\bo \Lambda^* \mb O\t  \big(\tilde{\mb c}_a^* - \hat{\mb c}_a (\mb z^*)\big)}_2\Big), 
               \label{eq: covariance matrix decomposition 1}
                }
               where the last inequality follows from the simple facts: (i) $\frac{|\mc I_a(\mb z) \cap \mc I_a(\mb z^*) | }{n_k(\mb z)}\leq 1$; (ii) $\Bignorm{\sum\limits_{i\in  \mc I_a(\mb z) \cap \mc I_a(\mb z^*)}
                \bo \Lambda^* \mb O\t  \big( \mb U_i - \hat{\mb c}_a(\mb z^*) \big)}_2  \leq \Bignorm{\sum\limits_{i\in  \mc I_a(\mb z) \cap \mc I_a(\mb z^*)}
                \bo \Lambda^* \mb O\t  \big( \mb U_i - \tilde{\mb c}_a^* \big)}_2+ |\mc I_a(\mb z) \cap \mc I_a(\mb z^*) | \Bignorm{\bo \Lambda^* (\tilde{\mb c}_a^* - \mb O\t \hat{\mb c}_a(\mb z^*))}_2$.

               We next decompose the term $\bignorm{\sum_{i\in \mc I_a(\mb z) \cap \mc I_a(\mb z^*)} \bo \Lambda^* \big(\mb O\t \mb U_i - \tilde{\mb c}_a^* \big)}_2 / n_a(\mb z) $:  
               \begin{align}
               &\bignorm{\sum_{i\in \mc I_a(\mb z) \cap \mc I_a(\mb z^*)} \bo \Lambda^* \big(\mb O\t \mb U_i - \tilde{\mb c}_a^* \big)}_2 / n_a(\mb z)  \\
               \lesssim  & \frac{1}{n_a(\mb z)}\bigg[ \sum\limits_{i\in  \mc I_a(\mb z) \cap \mc I_a(\mb z^*)}\Bignorm{
                \bo \Lambda^*  \big(\mb O\t  \mb U_i - \tilde{\mb c}_a^* - \mk L_i \big)}_2  + \Bignorm{\sum_{i\in \mc I_a(\mb z^*)}\bo \Lambda^* \mk L_i}_2  + \Bignorm{\sum_{i\in \mc I_a(\mb z^*) \backslash \mc I_a(\mb z)}\bo \Lambda^* \mk L_i}_2 \bigg]  \\ 
                \lesssim & \sigma_1^* \deltau + (n_k)^{-\frac12} \bar \sigma \sqrt{K\log d} + (n_k)^{-\frac12} (\sigma_K^*)^{-1} \tilde\sigma \sigma \sqrt{pK\log d}\\ 
                & + \sqrt{l(\mb z, \mb z^*) / \underline\omega} \Big(\bar \sigma \sqrt{n_k} +  (\sigma_K^*)^{-1} \big(\tilde\sigma^2 n + \tilde\sigma \sigma \sqrt{np} \big) \Big) \frac{\beta K}{n} \\ 
                \lesssim  &  \kappa \sqrt{\mu_1 \beta } K \frac{\deltaop}{\sqrt{n}\sigma_K^*} \sqrt{\frac{\log d}{n}}  + \beta K \sqrt{\frac{l(\mb z,\mb z^*)}{n \underline\omega}} \frac{\deltaop}{\sqrt{n} \sigma_K^*} + \sigma_1^* \deltau 
               \label{eq: covariance matrix estimation 1.1}
               \end{align}
               uniformly holds for all possible $\mb z$ with probability at least $1 - O(e^{-(1+o(1))\frac{\snr^2}{2}} \vee  d^{-c})$. Here, the second inequality follows from the facts
               \begin{align}
               & \sum_{i\in[n]}\ind{\{z_i^* = a, z_i \neq a\}} \leq \sum_{i\in[n]}\ind{\{z_i^* = a, z_i \neq a\}} + \sum_{i\in[n]}\ind{\{z_i^* \neq  a, z_i = a\}} \leq \frac{l(\mb z, \mb z^*)}{\underline \omega} = o( \frac{n}{\beta K}),\label{eq: a fact on negligible misspecification effect}
               \\ 
                   & \sum_{i\in [n]}\ind{\{z_i = z_i^* = a\}} \asymp \sum_{i\in[n]}\ind{\{z_i= a\}} \asymp  n_a, \\ 
                   & \bignorm{\sum_{i\in[n],z_i = z_i^* = a}  \mb A_{i,\cdot }}_2\leq \sqrt{\sum_{i\in [n]}\ind{\{z_i = z_i^* = a\}}}\bignorm{\mb A}, \qquad \text{for any matrix $\mb A \in \bb R^{n_k \times K}$},
               \end{align}
               together with the concentration inequalities in Lemmas~\ref{lemma: noise matrix concentrations using the universality (Gaussian)} and~\ref{lemma: noise matrix concentrations using the universality (bounded)}. Regarding the upper bound in the second inequality
                \eq{
                \frac{1}{n_a(\mb z)} \Bignorm{\sum\limits_{i\in  \mc I_a(\mb z) \cap \mc I_a(\mb z^*)}
                \bo \Lambda^*  \big(\mb O\t  \mb U_i - \tilde{\mb c}_a^* - \mk L_i \big)}_2 \lesssim  \frac{2}{n_a} \sum\limits_{i\in \mc I_a(\mb z^*)}\Bignorm{
                \bo \Lambda^*  \big(\mb O\t  \mb U_i - \tilde{\mb c}_a^* - \mk L_i \big)}_2 \lesssim \sigma_1^* \deltau ,\label{eq: example of application of average upper bound lemma}
                }
                we justify it case by case.  
If $\snr \ll \sqrt{(c+1)\log d}$ for some positive constant $c$, we apply Lemma~\ref{lemma: average upper bound} with  
$X_i = \|\mathbf O^{\top}\mb U_i - \mk L_i - \tilde{\mathbf c}_{z_i^*}\|_2$,  
$p_A = \exp(-\snr^{2}/2)$,  
$p_B = O(d^{-c-1})$,  
and high-probability bounds $A$ and $B$ obtained from Corollary~\ref{corollary: singular subspace perturbation theory in clustering problem} using $t_0 = \snr$ and $t_0 = \sqrt{(c +1) \log d}$, respectively.  
We take $c = 1/\sqrt{\log d}$.  
Consequently,  
$1 - n p_A - p_B/c = 1 - O\!\big(e^{-(1+o(1))\snr^{2}/2}\big)$  
and  
$cA + (1-c)B \lesssim \deltau$. 
If $\snr \gtrsim \sqrt{(c + 1) \log d}$, a direct application of Corollary~\ref{corollary: singular subspace perturbation theory in clustering problem} with $t_0 = \sqrt{(c + 1) \log d}$ yields the desired bound. 
               
               Plugging the above together with Lemma~\ref{lemma: center estimation error} into \eqref{eq: covariance matrix decomposition 1} yields with probability at least $1 - O(d^{-c}) \vee e^{-(1+o(1))\frac{\snr^2}{2}}$ that
               \longeq{
                & \alpha_1 \lesssim  \underbrace{\Big(\beta K \sqrt{\frac{l(\mb z, \mb z^*)}{n\underline\omega}} \frac{\deltaop}{\sqrt{n} \sigma_K^*} \Big) }_{\text{by \eqref{eq: center estimation with misspecification}}} \cdot  \bigg\{\underbrace{ \Big(\beta K \sqrt{\frac{l(\mb z, \mb z^*)}{n\underline\omega}} \frac{\deltaop}{\sqrt{n} \sigma_K^*}  \Big) }_{\text{by \eqref{eq: center estimation with misspecification}}} \\ 
                + &   \underbrace{ \kappa \sqrt{\mu_1 \beta } K \frac{\deltaop}{\sqrt{n}\sigma_K^*} \sqrt{\frac{\log d}{n}}  + \beta K \sqrt{\frac{l(\mb z,\mb z^*)}{n \underline\omega}} \frac{\deltaop}{\sqrt{n} \sigma_K^*} + \sigma_1^* \deltau }_{\text{by \eqref{eq: covariance matrix estimation 1.1}}} + \underbrace{\kappa \mu_1 \sqrt{ \beta } K \frac{\deltaop}{\sqrt{n} \sigma_K^*} \Big( \sqrt{\frac{\log d}{n}}  + \frac{\deltaop \snr }{(\sigma_K^*)^2} \Big)}_{\text{by \eqref{eq: center estimation without misspecification}}}\bigg\}
                \\
                 \lesssim & \kappa \beta^2 K^2 \sqrt{\frac{l(\mb z,\mb z^*)}{n \underline \omega}} \bar \sigma_{\mathsf{cov}}^2 ,
                \label{eq: alpha1 decomposition}
                }
                 where we invoked Eq.~\eqref{eq: fact A} and the condition $\frac{\log d}{n } \ll 1$. 
               \item 
               Then we move on to bound the second term $\alpha_2$ in \eqref{eq: covariance matrix decomposition 2}. We deduce that  
               \longeq{
               & \alpha_2 =  \Bignorm{\sum\limits_{\mc I_k(\mb z^*) \backslash \mc I_k(\mb z) } \bo \Lambda^* \mb O\t ( \mb U_i  - \hat{\mb c}_a(\mb z^*)) ( \mb U_i  - \hat{\mb c}_a(\mb z^*))\t \mb O \bo \Lambda^*   } / n_a(\mb z) \\
               \lesssim &  \frac{1}{n_a(\mb z)}\Bignorm{\sum\limits_{\mc I_k(\mb z^*) \backslash \mc I_k(\mb z) } \bo \Lambda^*  \big( \mb O\t\mb U_i  - \tilde{\mb c}_{z_i^*}^*\big) \big(\mb O\t \mb U_i  - \tilde{\mb c}_{z_i^*}^*\big)\t \bo \Lambda^*   }+ \frac{\beta K nh(\mb z,\mb z^*) }{n}\Big\lv \bo \Lambda^* \big(\tilde{\mb c}_{z_i^*}^* - \mb O\t\hat{\mb c}_k(\mb z^*) \big)\Big\rv_2^2\\ 
               {\lesssim} & \frac{\beta K}{n} |\mc I_k(\mb z^*) \backslash \mc I_k(\mb z)| \underbrace{\max_{i\in[n]}\big\lv \bo \Lambda^* \mb O\t (\mb U_i - \tilde{\mb c}_{z_i^*}^*) \big\rv_2^2}_{\text{apply Lemmas~\ref{lemma: noise matrix concentrations using the universality (Gaussian)},~\ref{lemma: noise matrix concentrations using the universality (bounded)} and Thm~\ref{thm: singular subspace perturbation theory}}} + \underbrace{\frac{\beta K nh(\mb z,\mb z^*) }{n}\Big\lv \bo \Lambda^* \big(\tilde{\mb c}_{z_i^*}^* - \mb O\t\hat{\mb c}_k(\mb z^*) \big)\Big\rv_2^2}_{\text{apply Lemma~\ref{lemma: center estimation error}}}\\ 
                \lesssim &  \frac{\beta Kl(\mb z, \mb z^*)}{n\underline\omega}\cdot \Big[ \kappa^2 \mu_1 \beta  K^2 \frac{\deltaop^2 \log d}{n(\sigma_K^*)^2} \Big]  +  
                \frac{\beta Kl(\mb z, \mb z^*)}{n\underline\omega}\cdot \Big[ \kappa^2 \mu_1 \beta  K^2 \frac{\deltaop^2}{n(\sigma_K^*)^2} \frac{\log d}{n} + (\sigma_1^*)^2 \deltau^2\Big] \\ 
               \lesssim & \frac{\beta Kl(\mb z, \mb z^*)}{n\underline\omega}\cdot \Big[ \kappa^2 \mu_1 \beta  K^2 \frac{\deltaop^2 \log d}{n(\sigma_K^*)^2} + (\sigma_1^*)^2 \deltau^2\Big]
               \label{eq: cov. est. decom. term 2}
                }
               uniformly holds for all possible $\mb z$ with probability at least $1- O(e^{-(1+ o(1))\frac{\snr^2}{2}} \vee d^{-c})$. Here we make use of the facts that $\frac{l(\mb z, \mb z^*)\beta K (\log d)^2}{n}\lesssim  1$ and $\frac{\beta K^2 \log d}{n} \lesssim 1$. 
        \item Similarly, for the third term $\alpha_3$ one has 
        \begin{align}
               & \alpha_3 = \Bignorm{\sum\limits_{\mc I_k(\mb z) \backslash \mc I_k(\mb z^*) } \bo \Lambda^* \mb O\t \big( \mb U_i  - \hat{\mb c}_a(\mb z^*)\big) \big( \mb U_i  - \hat{\mb c}_a(\mb z^*)\big)\t \mb O \bo \Lambda^*   } / n_a(\mb z)
               \\ 
               \lesssim & \frac{1}{n_a(\mb z)}\Bignorm{\sum\limits_{\mc I_k(\mb z) \backslash \mc I_k(\mb z^*) } \bo \Lambda^* \mb O\t ( \mb U_i  - \tilde{\mb c}_{z_i^*}^*) ( \mb U_i  - \tilde{\mb c}_{z_i^*}^*)\t \mb O \bo \Lambda^*   }  + \frac{\beta K nh(\mb z,\mb z^*) }{n}\Big\lv \bo \Lambda^* \big(\tilde{\mb c}_{z_i^*}^* - \mb O\t\hat{\mb c}_k(\mb z^*) \big)\Big\rv_2^2  \\ 
               \lesssim & \frac{\beta Kl(\mb z, \mb z^*)}{n\underline\omega}\cdot \Big[ \kappa^2 \mu_1 \beta  K^2 \frac{\deltaop^2 \log d}{n(\sigma_K^*)^2} + (\sigma_1^*)^2 \deltau^2\Big]
               \label{eq: cov. est. decom. term 3}  
               \end{align}
              uniformly holds with probability at least $1 - O(d^{-c})$ by Lemma~\ref{lemma: noise matrix concentrations using the universality (Gaussian)}~or~\ref{lemma: noise matrix concentrations using the universality (bounded)}, Lemma~\ref{lemma: center estimation error}, and the assumptions that $\frac{l(\mb z, \mb z^*)\beta K (\log d)^2}{n}\lesssim  1$ and  $\frac{\beta K^2 \log d}{n} \lesssim 1$.  

        \item Lastly, we upper bound $\alpha_4$ by Lemma~\ref{lemma: projected covariance matrix estimation error with true z} as follows: 
        \begin{align}
            \alpha_4 \lesssim  & \frac{1}{n_a^2 }\cdot nh(\mb z,\mb z^*)\cdot \Bignorm{\sum\limits_{\mc I_k(\mb z^*)} \bo \Lambda^* \mb O\t \big( \mb U_i  - \hat{\mb c}_a(\mb z^*)\big) \big( \mb U_i  - \hat{\mb c}_a(\mb z^*)\big)\t \mb O \bo \Lambda^*   } \\ 
            \lesssim & \frac{1}{n_a^2} \cdot \frac{l(\mb z,\mb z^*)}{\underline \omega }\cdot n_a \bar \sigma_{\mathsf{cov}}^2 \lesssim  \frac{\beta K l(\mb z, \mb z^*)}{n \underline\omega } \bar \sigma_{\mathsf{cov}}^2
        \end{align}
        holds with probability at least $1 - O(d^{-c}) \vee e^{-(1+o(1))\frac{\snr^2}{2}}$, where we again use the assumption that $\frac{l(\mb z, \mb z^*)\beta K (\log d)^2}{n}\lesssim  1$. 
               \end{itemize}
               
              Combining the above pieces together, we finally arrive at the conclusion that 
              \begin{align}
              & \norm{\hat{\mb S}_a(\mb z) - \hat{\mb S}_a(\mb z^*)}\\ 
              \lesssim & \kappa \beta^2 K^2 \sqrt{\frac{l(\mb z,\mb z^*)}{n \underline \omega}} \bar \sigma_{\mathsf{cov}}^2   +  \frac{\beta Kl(\mb z, \mb z^*)}{n\underline\omega}\cdot \Big[ \kappa^2 \mu_1 \beta  K^2 \frac{\deltaop^2 \log d}{n(\sigma_K^*)^2} + (\sigma_1^*)^2 \deltau^2 + \bar \sigma_{\mathsf{cov}}^2\Big] 
        \\
        \lesssim &  \kappa^2 \beta^3 K^2 \sqrt{\frac{l(\mb z, \mb z^*) K }{n}} \bar \sigma_{\mathsf{cov}}^2 / \underline \omega^{\frac12}  
        \label{eq: final bound on cov est}
              \end{align}
         holds uniformly for all admissible $\mb z$ with probability at least $1 - O(d^{-c}) \vee e^{-(1+o(1))\frac{\snr^2}{2}}$. Here, we used the conditions $\beta K (\log d)^2 l(\mb z, \mb z^*) \ll n$ and $\frac{ K \log d}{n} \lesssim 1$, along with Eqs.~\eqref{eq: fact A}, \eqref{eq: fact C}, and~\eqref{eq: fact D}. 

         For the matrix inverse case, the condition \eqref{eq: alg condition} ensures that the right-hand side of \eqref{eq: final bound on cov est} is negligible relative to $\underline{\sigma}_{\mathsf{cov}}^2$, yielding the result directly.

            \subsubsection{Proof of Lemma~\ref{lemma: projected covariance matrix estimation error with true z}}
            \label{sec: proof of lemma inverse of projected covariance matrix estimation error 2}

We will start with several lemmas that control the terms ${\mb V^*}\t \mb E_{\mc I_k(\mb z^*),\cdot}\t \mb E_{\mc I_k(\mb z^*),\cdot} \mb V^*$, ${\mb U^*}\t  \mc H(\mb E \mb E\t)_{\cdot, \mc I_k(\mb z^*)} \mc H(\mb E \mb E\t)_{\mc I_k(\mb z^*), \cdot } \mb U^*$, and ${\mb U^*}\t  \mc H(\mb E \mb E\t)_{\cdot, \mc I_k(\mb z^*)} \mb E_{\mc I_k(\mb z^*), \cdot } \mb V^*$ that arise from sum of linear terms $\sum_{i\in \mc I_k(\mb z^*)} \mk L_i \mk L_i\t$. 
\begin{lemma} \label{lemma: Gaussian projected covariance}
    Suppose the noise matrix $\mb E$ follows the Gaussian assumption in Theorem~\ref{theorem: upper bound for algorithm (gaussian)}. Then for every $k\in[K]$, it holds with probability at least $1- O(d^{-c})$ that 
    \eq{
        \norm{\frac{{\mb V^{*\top}} \mb E_{\mc I_k(\mb z^*), \cdot}\t\mb E_{\mc I_k(\mb z^*), \cdot} \mb V^*}{\sum_{i\in[n]} \ind{\{z_i^*=k\}}} - {\mb V^*}\t \bo \Sigma_k \mb V^*} \lesssim \sqrt{\frac{\beta K + \beta \log d }{n}} \norm{\mb S_k^*} \lesssim \sqrt{\frac{\beta K^2 \log d}{n}}\norm{\mb S_k^*}.
    }
\end{lemma}
\begin{proof}
    It is clear by definition that ${\mb V^*}\mb E_{i,\cdot}$ is a centered Gaussian random vector with covariance ${\mb V^*}\t \bo \Sigma \mb V^*$. Then the conclusion immediately follows by Theorem 6.5 in \cite{wainwright2019high}.
\end{proof}

The proof of Lemma~\ref{lemma: S-universality} invokes the universality result of \cite{brailovskaya2022universality}, which is deferred to Section~\ref{subsubsec: proof of lemma: S-universality}. 
	\begin{lemma}\label{lemma: S-universality}
       		Consider the noise environment in Assumption~\ref{assumption: bounded noise}. Then given an arbitrary deterministic matrix $\mb A\in \bb R^{p \times K}$, it holds with probability at least $1 - e^{-t}$ that 
       		\longeq{
       		\label{eq: general concentration of projected covariance estimation}
       		& \norm{\mb A\t \mb E_{\mc I_k(\mb z^*), \cdot}\t \mb E_{\mc I_k(\mb z^*), \cdot} \mb A - \mb A\t \bb E[\mb E_{\mc I_k(\mb z^*), \cdot}\t \mb E_{\mc I_k(\mb z^*), \cdot} ] \mb A} 
            \lesssim \big(\frac{n}{K}\mathrm{Tr}(\mb A\t \mb \Sigma_k \mb A)\norm{\mb A\t \mb \Sigma_k \mb A}  +  \frac	{n}{K}lm^4B^4\norm{\mb A}\ti^4\big)^{\frac{1}{2}}
            \\ & +  t^{\frac{3}{4}}m^{\frac{1}{2}}B \norm{\mb A}\big(\frac{n}{K}\mathrm{Tr}(\mb A\t \mb \Sigma_k \mb A)\norm{\mb A\t \mb \Sigma_k \mb A}  + \frac	{n}{K}lm^4B^4\norm{\mb A}\ti^4\big)^{\frac{1}{4}}  + \big(n^{\frac{1}{2t}}mB^2t \norm{\mb A}^2 + n^{\frac{1}{2t}}\mathrm{Tr}(\mb A\t \mb \Sigma_k \mb A)\big).
       		} 
         
         Moreover, if we additionally assume that $\mb A = \mb V^*$ with $K = o( d), \sqrt{n}\bar\sigma^2 \gtrsim \sqrt{\beta K} mB^2 \log d$, then it holds that
       		\longeq{
       		\label{eq: concentration of projected covariance estimation}
       		 &\Bignorm{\frac{{\mb V^{*\top}} \mb E_{\mc I_k(\mb z^*), \cdot}\t \mb E_{\mc I_k(\mb z^*), \cdot} \mb V^*}{\sum_{i\in[n]}\ind{\{z_i^*=k\}}} - {\mb V^*}\t \bo \Sigma_k \mb V^*}\lesssim
          \sqrt{\frac{\beta K^2}{n}}\bar \sigma^2(\log d)
       		}
       		with probability at least $1 - O(d^{-c})$. 
         
       	\end{lemma}

The next lemma provides bounds for the summations involving the third- and fourth-order terms of the entries of $\mb E$. Its proof relies on a combination of a decoupling argument and a universality result, which we defer to Section~\ref{subsubsec: proof of lemma: concentration on second part of covariance estimation}.

\begin{lemma}
\label{lemma: concentration on second part of covariance estimation}
Instate either the assumptions in Theorem~\ref{theorem: upper bound for algorithm (gaussian)} or the assumptions in Theorem~\ref{theorem: upper bound for algorithm (bounded)}, then it holds with probability at least $1 - O(d^{-c})$ that 
\begin{align}
    & \Bignorm{{\mb U^*}\t  \mc H(\mb E \mb E\t)_{\cdot, \mc I_k(\mb z^*)} \mc H(\mb E \mb E\t)_{\mc I_k(\mb z^*), \cdot } \mb U^* - {\mb U^*}\t \text{diag}\big(\tr(\sum_{i'\in[n], i\neq i'} \bo \Sigma_{z_i^*} \bo \Sigma_{z_{i'}^*})\big)_{i\in[n]} \mb U^*}_F \\
    \lesssim &  \beta^{\frac14} K \sqrt{n_k} p \tilde\sigma^2 \sigma^2 \log d+ K n_k \sqrt{p} \tilde\sigma^3 \sigma (\log d)^2 + K n_k p^{\frac56} \tilde\sigma^{\frac73} \sigma^{\frac53} \log d, 
     \\ 
    & \Bignorm{{\mb U^*}\t  \mc H(\mb E \mb E\t)_{\cdot, \mc I_k(\mb z^*)} \mb E_{\mc I_k(\mb z^*), \cdot} \mb V^* }\lesssim K n_k \tilde\sigma^2 \bar \sigma \sqrt{\log d} + \sqrt{ n_kp} K\tilde\sigma \sigma \bar \sigma\sqrt{\log d}. 
\end{align}
\end{lemma}

With these lemmas in place, we are ready to control the term $\norm{\bo \Lambda^* \mb O\t \hat{\bo \Omega}_k(\mb z^*)\mb O \bo \Lambda^*- \mb S_k^*}$. By the definition of $\hat{\mb S}_k(\mb z^*)$, one has from Cauchy's inequality that
\begin{align}
                    & \norm{\bo \Lambda^* \mb O\t \hat{\bo \Omega}_k(\mb z^*)\mb O \bo \Lambda^*- \mb S_k^*}\\ 
                    = & \bigg\| \bo \Lambda^* \bigg\{\frac{\sum\limits_{i \in \mc I_k(\mb z^*)} \big(\mb O\t\mb U_i  - \tilde{\mb c}_{z_i^*}^*\big)\big(\mb O\t\mb U_i  - \tilde{\mb c}_{z_i^*}^*\big)\t}{  n_k} - \Big[ \frac{\sum\limits_{i \in \mc I_k(\mb z^*)} \big(\mb O\t\mb U_i  - \tilde{\mb c}_{z_i^*}^*\big)}{ n_k} \Big]\Big[ \frac{\sum\limits_{i \in \mc I_k(\mb z^*)} \big(\mb O\t\mb U_i  - \tilde{\mb c}_{z_i^*}^*\big)}{ n_k} \Big]\t  \bigg\} \bo \Lambda^* - \mb S_k^* \bigg\| \\ 
                \leq & \bigg\| \bo \Lambda^* \Big[\frac{\sum_{i \in \mc I_k(\mb z^*)} \mathfrak L_i \mathfrak L_i\t}{  n_k}  \Big] \bo \Lambda^* - \mb S_k^*  \bigg\| + \norm{\bo \Lambda^* \frac{\sum_{i \in \mc I_k(\mb z^*)} \mk L_i}{ n_k} }^2 \\ 
                & + \frac{2}{n_k}\sum_{i \in \mc I_k(\mb z^*)} {\sigma_1^*} \norm{\bo \Lambda^* \mk L_i}_2 \norm{\mb O\t \mb U_i  - \tilde{\mb c}_{z_i^*}^* - \mk L_i}_2  +  \frac{1}{n_k}\sum_{i \in \mc I_k(\mb z^*)} {\sigma_1^*}^2 \norm{\mb O\t \mb U_i  - \tilde{\mb c}_{z_i^*}^* - \mk L_i}_2^2 \\ 
                 &+ 2\sigma_1^*   \Bignorm{  \frac{\sum_{i \in \mc I_k(\mb z^*)} \bo \Lambda^* \mk L_i}{ n_k} }_2 \frac{\sum_{i \in \mc I_k(\mb z^*)} \norm{\mb O\t \mb U_i  - \tilde{\mb c}_{z_i^*}^* - \mk L_i}_2}{n_k}  + {\sigma_1^*}^2 \Big[ \frac{\sum_{i \in \mc I_k(\mb z^*)} \norm{\mb O\t \mb U_i  - \tilde{\mb c}_{z_i^*}^* - \mk L_i}_2}{n_k}\Big]^2 \\ 
            {\lesssim} &  \bigg\| \bo \Lambda^* \Big[\frac{\sum_{i \in \mc I_k(\mb z^*)} \mathfrak L_i \mathfrak L_i\t}{  n_k}  \Big] \bo \Lambda^* - \mb S_k^*  \bigg\| + \norm{\bo \Lambda^* \frac{\sum_{i \in \mc I_k(\mb z^*)} \mk L_i}{ n_k} }^2 \\ 
                & + \frac{2}{n_k} \sigma_1^* \Big(\sum_{i \in \mc I_k(\mb z^*)}  \norm{\bo \Lambda^* \mk L_i}_2^2\Big)^{\frac12} \Big(\sum_{i \in \mc I_k(\mb z^*)}\norm{\mb O\t \mb U_i  - \tilde{\mb c}_{z_i^*}^* - \mk L_i}_2^2\Big)^{\frac12}  +  \frac{1}{n_k}\sum_{i \in \mc I_k(\mb z^*)} {\sigma_1^*}^2 \norm{\mb O\t \mb U_i  - \tilde{\mb c}_{z_i^*}^* - \mk L_i}_2^2. 
                 \label{eq: decomposition of covariance estimation error given true z}
                \end{align}

        We start with bounding the first term on the RHS of \eqref{eq: decomposition of covariance estimation error given true z}. Since the linear approximation $\mk L_i = {\mb V^*}\t \mb E_i + {\mb U^*}\t \mc P_{-i,\cdot}(\mb E) \mb E_i$ comes from two difference sources of covariances $\mb S^{\sf{mod}}_k$ and $\mb S^{\sf{exc}}_k$, respectively, we separately parse the covariance estimation error as follows: 
        \begin{align}
            & \bigg\| \bo \Lambda^* \Big[\frac{\sum_{i \in \mc I_k(\mb z^*)} \mathfrak L_i \mathfrak L_i\t}{  n_k}  \Big] \bo \Lambda^* - \mb S_k^*  \bigg\| \leq \Bignorm{\frac{1}{n_k}\sum_{i\in \mc I_k(\mb z^*)}{\mb V^*}\t \mb E_i \mb E_i\t \mb V^* - \mb S^{\sf{mod}}_k}  \\ 
            & \qquad + \Bignorm{\frac{1}{n_k}\sum_{i\in \mc I_k(\mb z^*)}{\bo \Lambda^*}^{-1}{\mb U^*}\t \mc P_{-i,\cdot}(\mb E)\mb E_i \mb E_i\t \mc P_{-i,\cdot}(\mb E)\t  \mb U^* {\bo \Lambda^*}^{-1} - \mb S^{\sf{exc}}_k} \\ 
            & \qquad + \frac{2}{n_k}\Bignorm{\sum_{i\in \mc I_k(\mb z^*)}{\bo \Lambda^*}^{-1}{\mb U^*}\t \mc P_{-i,\cdot}(\mb E)\mb E_i \mb E_i\t \mb V^* }. 
            \label{eq: eq: decomposition of covariance estimation error given true z 1}
        \end{align}

        Regarding the first term of \eqref{eq: decomposition of covariance estimation error given true z}, Lemma~\ref{lemma: S-universality} reveals that
                \begin{align}
                    & \Bignorm{\frac{1}{n_k}\sum_{i\in \mc I_k(\mb z^*)}{\mb V^*}\t \mb E_i \mb E_i\t \mb V^* - \mb S^{\sf{mod}}_k}  \\
                    \lesssim  &
                    \left\{ \begin{matrix} \sqrt{\frac{\beta K^2}{n}}\bar \sigma^2(\log d)^{\frac{3}{4}} + \sqrt{\frac{\beta p Km^3 B^4}{n}} \norm{\mb V^*}\ti ^2(\log d)^{\frac{3}{4}} \text{\quad (bounded case)}\\ 
                    \sqrt{ \frac{\beta K}{n}}\bar \sigma^2\sqrt{\log d}    \text{\quad (Gaussian case)}
                    \end{matrix} \right. \lesssim \sqrt{\frac{\beta K^2}{n}}\bar\sigma^2 \log d, \label{eq: projected covariance matrix estimation error 1}
                \end{align}
                where we use Assumption~\ref{assumption: bounded noise}.\ref{item: bounded noise assumption 3} that $\sqrt{\mu_1 K} mB(\log d)^2 \lesssim \sigma(np)^{\frac14}$
                for the bounded case. 

        With regards to the second and the third terms in \eqref{eq: eq: decomposition of covariance estimation error given true z 1}, we invoke Lemma~\ref{lemma: concentration on second part of covariance estimation} to derive that 
        \begin{align}
            & \Bignorm{\frac{1}{n_k}\sum_{i\in \mc I_k(\mb z^*)}{\bo \Lambda^*}^{-1}{\mb U^*}\t \mc P_{-i,\cdot}(\mb E)\mb E_i \mb E_i\t \mc P_{-i,\cdot}(\mb E)\t  \mb U^* {\bo \Lambda^*}^{-1} - \mb S^{\sf{exc}}_k}  \\ 
            \lesssim & {\sigma_K^*}^{-2}\beta^{\frac34} K^{\frac32} n^{-\frac12} p \tilde\sigma^2 \sigma^2 \log d+ {\sigma_K^*}^{-2}K \sqrt{p} \tilde\sigma^3 \sigma (\log d)^2 +{\sigma_K^*}^{-2} K p^{\frac56} \tilde\sigma^{\frac73} \sigma^{\frac53} \log d, 
            \label{eq: projected covariance matrix estimation error 2}
            \\ 
            & \frac{2}{n_k}\Bignorm{\sum_{i\in \mc I_k(\mb z^*)}{\bo \Lambda^*}^{-1}{\mb U^*}\t \mc P_{-i,\cdot}(\mb E)\mb E_i \mb E_i\t \mb V^* } \lesssim {\sigma_K^*}^{-1} K \tilde\sigma^2 \bar \sigma \sqrt{\log d} +  {\sigma_K^*}^{-1} \sqrt{\frac{\beta p}{n}} K^{\frac32}\tilde\sigma \sigma \bar \sigma\sqrt{\log d}
            \label{eq: projected covariance matrix estimation error 3}
        \end{align}
        hold with probability at least $1- O(d^{-c_0})$. Combine \eqref{eq: projected covariance matrix estimation error 1}, \eqref{eq: projected covariance matrix estimation error 2}, and \eqref{eq: projected covariance matrix estimation error 3} to arrive at 
        \begin{align}
            & \bigg\| \bo \Lambda^* \Big[\frac{\sum_{i \in \mc I_k(\mb z^*)} \mathfrak L_i \mathfrak L_i\t}{  n_k}  \Big] \bo \Lambda^* - \mb S_k^*  \bigg\| \\ 
            \lesssim  & {\sigma_K^*}^{-2}\beta^{\frac34} K^{\frac32} n^{-\frac12} p \tilde\sigma^2 \sigma^2 \log d+ {\sigma_K^*}^{-2}K \sqrt{p} \tilde\sigma^3 \sigma (\log d)^2 +{\sigma_K^*}^{-2} K p^{\frac56} \tilde\sigma^{\frac73} \sigma^{\frac53} \log d \\ 
            & \qquad  + {\sigma_K^*}^{-1} K \tilde\sigma^2 \bar \sigma \sqrt{\log d} +  {\sigma_K^*}^{-1} \sqrt{\frac{\beta p}{n}} K^{\frac32}\tilde\sigma \sigma \bar \sigma\sqrt{\log d}  \\ 
            \lesssim &  \frac{\deltaop^2}{n (\sigma_K^*)^2} \Big[\frac{\beta K^{\frac32} \log d}{\sqrt{n}} +  K \frac{\tilde\sigma }{\sqrt{p}\sigma } (\log d)^2 + K (\frac{\tilde\sigma}{\sqrt{p} \sigma})^{\frac13} \log d + K \frac{\tilde \sigma}{\sqrt{p} \sigma}\sqrt{\log d}+ \frac{\beta^{\frac12} K^{\frac32}  \sqrt{\log d}}{\sqrt{n}}\Big] \\ 
            \lesssim &  \frac{\deltaop^2}{n (\sigma_K^*)^2} \Big[\frac{\beta K^{\frac32} \log d}{\sqrt{n}}  + K (\frac{\tilde\sigma}{\sqrt{p} \sigma})^{\frac13} (\log d)^2\Big] 
            \label{eq: concentration for Lambda sum Li Lit Lambda / nk - S_k}
        \end{align}
        with probability at least $1 -O(d^{-c})$, where we used the definition of $\deltaop$ and the relation $\tilde \sigma \leq \sqrt{p }\sigma$. 
        
                For the second term in \eqref{eq: decomposition of covariance estimation error given true z}, it immediately follows by \eqref{eq: sum of linear terms in center estimation proof} that 
                \begin{align}
                    &  \norm{\bo \Lambda^* \frac{\sum_{i \in \mc I_k(\mb z^*)} \mk L_i}{ n_k} }^2  \lesssim \frac{\beta K \bar \sigma^2 \log d }{n} + \frac{\beta \tilde\sigma^2 \sigma^2 pK^2 \log d}{n{\sigma_K^*}^2} \lesssim \beta K^2 \frac{\deltaop^2}{n(\sigma_K^*)^2}\frac{\log d}{n} 
                \end{align}
                with probability at least $1 -O(d^{-c})$. 

        For the third term in \eqref{eq: decomposition of covariance estimation error given true z}, Corollary~\ref{corollary: singular subspace perturbation theory in clustering problem}, \eqref{eq: alternative upper bound for delta_u}, and Lemmas~\ref{lemma: noise matrix concentrations using the universality (Gaussian)}-\ref{lemma: noise matrix concentrations using the universality (bounded)} together imply that, with probability at least $1- O(d^{-c}) \vee e^{-(1+o(1))\frac{\snr^2}{2}}$, 
        \begin{align}
            &  \frac{2}{n_k} \sigma_1^* \Big(\sum_{i \in \mc I_k(\mb z^*)}  \norm{\bo \Lambda^* \mk L_i}_2^2\Big)^{\frac12} \Big(\sum_{i \in \mc I_k(\mb z^*)}\norm{\mb O\t \mb U_i  - \tilde{\mb c}_{z_i^*}^* - \mk L_i}_2^2\Big)^{\frac12}  \\ 
            \lesssim & \kappa K^{\frac32} \mu_1^{\frac12} \frac{\deltaop^2}{n(\sigma_K^*)^2} \Big[ \frac{\log d}{n} + \frac{\deltaop^2 \snr^2}{(\sigma_K^*)^4}\Big]^{\frac12} \Big[\frac{\beta K^{\frac32} \log d}{\sqrt{n}}  + K (\frac{\tilde\sigma}{\sqrt{p} \sigma})^{\frac13} (\log d)^2\Big]^{\frac12}, 
        \end{align}
         where we used 
         \eq{
         \sum_{i \in \mc I_k(\mb z^*)}  \norm{\bo \Lambda^* \mk L_i}_2^2 = \tr(\sum_{i \in \mc I_k(\mb z^*)} \bo \Lambda^* \mk L_i \mk L_i\t \bo \Lambda^*) \leq K \Bignorm{\sum_{i \in \mc I_k(\mb z^*)} \bo \Lambda^* \mk L_i \mk L_i\t \bo \Lambda^*} \leq n_k K\bar \sigma_{\mathsf{cov}}^2 + n_k K \cdot \text{(the RHS of \eqref{eq: concentration for Lambda sum Li Lit Lambda / nk - S_k})}
         } with probability at least $1- O(d^{-c})$
         , and to upper bound $\sum_{i \in \mc I_k(\mb z^*)}\norm{\mb O\t \mb U_i  - \tilde{\mb c}_{z_i^*}^* - \mk L_i}_2^2$ we used an argument similar to \eqref{eq: example of application of average upper bound lemma}. 

        Regarding the fourth term in \eqref{eq: decomposition of covariance estimation error given true z}, invoking Corollary~\ref{corollary: singular subspace perturbation theory in clustering problem} again gives with probability at least $1- O(d^{-c}) \vee e^{-(1+o(1))\frac{\snr^2}{2}}$ and $ \snr \lesssim \frac{\sqrt{\beta K} \sigma_K^*}{\underline\sigma_{\mathsf{cov}} \sqrt{n}}$ that 
        \begin{align}
            &\frac{1}{n_k}\sum_{i \in \mc I_k(\mb z^*)} {\sigma_1^*}^2 \norm{\mb O\t \mb U_i  - \tilde{\mb c}_{z_i^*}^* - \mk L_i}_2^2 \lesssim \kappa^2 \mu_1 K^2 \frac{\deltaop^2}{n (\sigma_K^*)^2} \Big( \frac{\log d}{n} + \frac{\deltaop^2 \snr^2}{(\sigma_K^*)^4}\Big). 
        \end{align}

        Substituting these inequalities together into \eqref{eq: decomposition of covariance estimation error given true z} yields that  
        \begin{align}
            & \norm{\bo \Lambda^* \mb O\t \hat{\bo \Omega}_k(\mb z^*)\mb O \bo \Lambda^*- \mb S_k^*}
            \lesssim  \kappa^2 \beta^2 K^{\frac52} \frac{\deltaop^2}{n(\sigma_K^*)^2} \Big( \frac{\log d}{\sqrt{n}} +  \frac{\deltaop \snr}{(\sigma_K^*)^2} + \big( \frac{\tilde\sigma}{\sqrt{p} \sigma}\big)^{\frac16} \log d \Big) \\ 
            \lesssim & \kappa^2 \beta^2 K^{\frac52} \frac{\deltaop^2}{n(\sigma_K^*)^2} \Big( \frac{\log d}{\sqrt{n}} +  \omega^{-\frac12} + \big( \frac{\tilde\sigma}{\sqrt{p} \sigma}\big)^{\frac16} \log d \Big) \leq \kappa^2 \beta^2 K^{\frac52} \kappa_{\mathsf{cov}}^2 \underline\sigma_{\mathsf{cov}}^2 \Big( \frac{\log d}{\sqrt{n}} +  \underline \omega^{-\frac12} + \big( \frac{\tilde\sigma}{\sqrt{p} \sigma}\big)^{\frac16} \log d \Big)
        \end{align}
        with probability at least $1- O(d^{-c}) \vee e^{-(1+o(1))\frac{\snr^2}{2}}$. Here, we used \eqref{eq: fact A}, \eqref{eq: fact B}, \eqref{eq: fact C}, and the condtions $ \frac{\log d}{\sqrt{n}} \ll 1$ and $ \big( \frac{\tilde\sigma }{\sqrt{p} \sigma }\big)^{\frac16} \log d \ll 1$. 

        Finally, from the condition $\min\{\frac{n}{(\log d)^2}, \underline\omega, \frac{p^{\frac16} \sigma^{\frac13}}{\tilde\sigma^{\frac13} (\log d)^2} \} = \omega( \kappa^4 \kappa_{\mathsf{cov}}^4 \beta^4 K^5)$, we complete the proof of Lemma~\ref{lemma: projected covariance matrix estimation error with true z}.

    \subsection{Proofs of Concentration Inequalities}
    This subsection collects some concentration inequalities that are used in the proof of the main upper bound. 
    
    \subsubsection{Proof of Lemma~\ref{lemma: S-universality}}
    \label{subsubsec: proof of lemma: S-universality}

    \emph{Universality on Concentration for Projected Covariance Matrices. }
    For ease of presenting the matrix concentration universality, we first introduce some shorthand quantities following \cite{brailovskaya2022universality}: Given a $n_1$-by-$n_2$ matrix $\mb Y = \sum_{i\in[n]} \mb Z_i$ where $\mb Z_i, i = [n]$ are independent random matrices with $\bb E[\mb Z_i] = 0$, we then denote that 
    \begin{align*}
         \sigma(\mb Y) &\coloneqq \big(\max\big\{\norm{\bb E[\mb Y \mb Y\t ]}, \norm{\bb E[\mb Y\t \mb Y]}\big\}\big)^{\frac{1}{2}}, \quad 
         \sigma_*(\mb Y) \coloneqq \sup\limits_{\norm{\mb v} = \norm{\mb w} = 1} \bb E\big[\big|\langle\mb v, \mb Y\mb w\rangle\big|^2\big]^{\frac{1}{2}},\\
         v(\mb Y) &\coloneqq \norm{\mathrm{Cov}(\mb Y)}^{\frac{1}{2}},\quad
         R_p(\mb Y) \coloneqq \bb E[\sum_{i\in[n]} \bb E[\mathrm{tr}|\mb Z_i|^p]]^{\frac{1}{p}}.
    \end{align*}

       	\begin{proof}
            Define $\mb E^{[k]} \coloneqq \mb E_{\mc I_k(\mb z^*), \cdot}$. 
       		The core idea is to make use of the so-called \emph{S-Universality} in \cite{brailovskaya2022universality} to derive an upper bound on $\mathbb E\Big[\textrm{tr}\big((\mb A\t\mb E_{\mc I_k(\mb z^*), \cdot}\t \mb E^{[k]}\mb A - \mb A\t \bb E[\mb E_{\mc I_k(\mb z^*), \cdot}\t \mb E^{[k]}] \mb A)^p\big)\Big]^{\frac{1}{p}}$ for some sufficiently large $p$ where $\mathrm{tr}(\mb X)\coloneqq \frac{\mathrm{Tr}(\mb X)}{n_0}$ denotes the normalized trace operator for $\mb X \in \bb R^{n_0\times n_0}$. 

       		We denote $\mb A\t \mb E_{\mc I_k(\mb z^*), \cdot}\t \mb E^{[k]}\mb A - \mb A\t \mathbb E\big[\mb E_{\mc I_k(\mb z^*), \cdot}\t \mb E^{[k]}\big]\mb A $ by $\mb S$. 
       		Combining Theorem 2.7, Lemma 2.5 in \cite{bandeira2023matrix}, and Lemma 2.8 in \cite{brailovskaya2022universality} for an positive integer $p$ gives a control on the $(2p)$-th moment of the normalized trace of $\mb S$ that 
       		\begin{align}
       		& \mathbb E\big[\textrm{tr}\big(\mb S^{2p}\big)\big]^{\frac{1}{2p}} 
         \stackrel{\text{\cite[Theorem 2.8]{brailovskaya2022universality}}}{\lesssim}  \bb E[\mathrm{tr}(\mb G^{2p})]^{\frac{1}{2p}} + R_{2p}(\mb S)p^2 
         \\
         \stackrel{\text{\cite[Theorem 2.7]{bandeira2023matrix}}}{\lesssim}  & 
       		(\mathrm{tr}\otimes \tau)(\vert \mb S_{\textrm{free}}\vert^{2p})^{\frac{1}{2p}} + p^{\frac{3}{4}} v\big(\mb S\big)^{\frac{1}{2}}\sigma\big(\mb S\big)^{\frac{1}{2}} + R_{2p}(\mb S)p^2   
       		\lesssim  \norm{\mb S_{\textrm{free}}}  + p^{\frac{3}{4}}  v\big(\mb S\big)^{\frac{1}{2}}\sigma\big(\mb S\big)^{\frac{1}{2}} +  R_{2p}(\mb S)p^2\\
   		\stackrel{\text{\cite[Lemma 2.5]{bandeira2023matrix}}}{\lesssim} & 
   			\sigma(\mb S) + p^{\frac{3}{4}}  v\big(\mb S\big)^{\frac{1}{2}}\sigma\big(\mb S\big)^{\frac{1}{2}} +  R_{2p}(\mb S)p^2,
      \label{eq: universality and free prob for covariance estimation}
       		\end{align}
       		where  we use a fact in the third inequality that $\mathrm{tr}\otimes \tau(|X_{\textrm{free}}|^{2p})^{\frac{1}{2p}} \leq \norm{X_{\textrm{free}}}$ when we consider the $C^*$-algebra $M_d(\bb C)_{\textrm{sa}} \otimes \mc A$ where $(\mc A,\tau)$ is a semicircle family; see the details in \cite[Section 4]{bandeira2023matrix} and \cite[Lecture 3]{nica2006lectures}. 
       	
       	The next step is to separately control the quantities appearing in \eqref{eq: universality and free prob for covariance estimation}. 
        \begin{enumerate} 
            \item Toward bounding $\sigma(\mb S)^2 = \norm{\bb E\big[ \mb S^2\big]}$, we first observe that $\mb S = \sum_{i\in[n_k]}\big( \mb A\t \mb E^{[k]}_i {{}\mb E^{[k]}_i}\t \mb A - \mb A\t \mb \Sigma_k \mb A\big)$ and make use of the independence of $\mb E^{[k]}_i$ to rewrite $\sigma(\mb S)^2$ as
        \longeq{
    \sigma(\mb S)^2 =& n_k\Bignorm{\bb E\big[ \big(\mb A\t \mb E^{[k]}_1 {\mb E^{[k]}_1}\t \mb A - \mb A\t \mb \Sigma_k \mb A\big)^2\big]} \\ 
    \leq  & n_k\Bignorm{\bb E\big[ \mb A\t \mb E^{[k]}_1 {\mb E^{[k]}_1}\t \mb A \mb A\t \mb E^{[k]}_1 {\mb E^{[k]}_1}\t \mb A\big]}  + n_k  \norm{\mb A\t \mb \Sigma_k \mb A}^2\\ 
= &  n_k\Bignorm{\bb E\big[ \sum_{j_1,j_2,j_3,j_4\in[p]}\mb A_{j_1}\t E^{(k)}_{1,j_1}E^{(k)}_{1,j_2}\mb A_{j_2}\mb A_{j_3}\t E^{(k)}_{1, j_3} E^{(k)}_{1, j_4} \mb A_{j_4} \big]}  + n_k \norm{\mb A\t \mb \Sigma_k \mb A}^2.\label{eq: decomposition of the square of S}
     }
     
    Regarding the first term in \eqref{eq: decomposition of the square of S}, it is related to its Gaussian analog  that for every $j_1,j_2,j_3,j_4 \in [p]$:
       	\longeq{\label{eq: sigma term of projected covariance}
       	& \mathbb E\left[\mb A_{j_1}\t E^{(k)}_{1,j_1}E^{(k)}_{1,j_2}\mb A_{j_2}\mb A_{j_3}\t E^{(k)}_{i, j_3} E^{(k)}_{i, j_4} \mb A_{j_4} \right]\\ 
        =&  \mathbb E\left[\mb A_{j_1}\t g_{j_1}g_{j_2}\mb A_{j_2}\mb A_{j_3}\t g_{ j_3} g_{j_4} \mb A_{j_4} \right] \\ 
        &+ \Big(\mathbb E\left[\mb A_{j_1}\t E^{(k)}_{1,j_1}E^{(k)}_{1,j_2}\mb A_{j_2}\mb A_{j_3}\t E^{(k)}_{i, j_3} E^{(k)}_{i, j_4} \mb A_{j_4} \right] -\underbrace{\mathbb E\left[\mb A_{j_1}\t g_{j_1}g_{j_2}\mb A_{j_2}\mb A_{j_3}\t g_{ j_3} g_{j_4} \mb A_{j_4} \right]}_{\text{apply Lemma~\ref{lemma: Leonov-Shiryaev} to this term}}\Big)\\ 
        = &  \mathbb E\left[\mb A_{j_1}\t g_{j_1}g_{j_2}\mb A_{j_2}\mb A_{j_3}\t g_{ j_3} g_{j_4} \mb A_{j_4} \right]
        \\ &
        \quad +  \mb A_{j_1}\t \mb A_{j_2} \mb A_{j_3}\t \mb A_{j_4}\big(\bb E[E^{(k)}_{1,j_1}E^{(k)}_{1,j_2}E^{(k)}_{1,j_3}E^{(k)}_{1,j_4}] - \bb E[E^{(k)}_{1,j_1}E^{(k)}_{1,j_2}]\bb E[E^{(k)}_{1,j_3}E^{(k)}_{1,j_4}] \\ 
        &\quad -  \bb E[E^{(k)}_{1,j_1}E^{(k)}_{1,j_3}]\bb E[E^{(k)}_{1,j_2}E^{(k)}_{1,j_4}] - \bb E[E^{(k)}_{1,j_1}E^{(k)}_{1,j_4}]\bb E[E^{(k)}_{1,j_2}E^{(k)}_{1,j_4}]\big)\\ 
        \stackrel{\text{(a)}}{= } & \mathbb E\left[\mb A_{j_1}\t g_{j_1}g_{j_2}\mb A_{j_2}\mb A_{j_3}\t g_{ j_3} g_{j_4} \mb A_{j_4} \right]
        \\
        &\quad +  \mb A_{j_1}\t \mb A_{j_2} \mb A_{j_3}\t \mb A_{j_4}\big(\bb E[E^{(k)}_{1,j_1}E^{(k)}_{1,j_2}E^{(k)}_{1,j_3}E^{(k)}_{1,j_4}] - \bb E[E^{(k)}_{1,j_1}E^{(k)}_{1,j_2}]\bb E[E^{(k)}_{1,j_3}E^{(k)}_{1,j_4}] \\ 
        &\quad -  \bb E[E^{(k)}_{1,j_1}E^{(k)}_{1,j_3}]\bb E[E^{(k)}_{1,j_2}E^{(k)}_{1,j_4}] - \bb E[E^{(k)}_{1,j_1}E^{(k)}_{1,j_4}]\bb E[E^{(k)}_{1,j_2}E^{(k)}_{1,j_4}]\big)\\ 
        & \cdot \ind{\{\text{$j_1, j_2,j_3,j_4$ are in the same block $S_s$ for some $s\in[l]$}\}},
       	}
       	where $\mb g = (g_1, \cdots, g_p)\t $ is a centered Gaussian analog of $\mb E_1^{(k)}$ with the covariance matrix $\mb \Sigma_k$. To be more precise, here Lemma~\ref{lemma: Leonov-Shiryaev} comes into play by 
        \begin{align}
            & \mathbb E\left[\mb A_{j_1}\t g_{j_1}g_{j_2}\mb A_{j_2}\mb A_{j_3}\t g_{ j_3} g_{j_4} \mb A_{j_4} \right] = \mb A_{j_1}\t \mb A_{j_2} \mb A_{j_3}\t \mb A_{j_4} \bb E\big[g_{i_1}g_{i_2}g_{i_3}g_{i_4} \big] \\ 
            = & \mb A_{j_1}\t \mb A_{j_2} \mb A_{j_3}\t \mb A_{j_4} \bb \sum_{\pi \in \mc P([4])} \prod_{p\in \pi} \kappa(\mb g^{(j_1,j_2,j_3,j_4)}_p) \\ 
            = & \mb A_{j_1}\t \mb A_{j_2} \mb A_{j_3}\t \mb A_{j_4} \big(\bb E[g_{j_1}g_{j_2}] \bb E[g_{j_3}g_{j_4}] + \bb E[g_{j_1}g_{j_3}]\bb E[g_{j_2}g_{j_4}] + \bb E[g_{j_1}g_{j_4}]\bb E[g_{j_2}g_{j_3}]\big) \\ 
            = & \mb A_{j_1}\t \mb A_{j_2} \mb A_{j_3}\t \mb A_{j_4}  \big(\bb E[E^{(k)}_{1,j_1}E^{(k)}_{1,j_2}] \bb E[E^{(k)}_{1,j_3}E^{(k)}_{1,j_4}] +\bb E[E^{(k)}_{1,j_1}E^{(k)}_{1,j_3}]\bb E[E^{(k)}_{1,j_2}E^{(k)}_{1,j_4}] + \bb E[E^{(k)}_{1,j_1}E^{(k)}_{1,j_4}]\bb E[E^{(k)}_{1,j_2}E^{(k)}_{1,j_3}]\big),
        \end{align}
        where $p$ is an index set in a partition $\pi$ of $[4]$, $\mb g^{(j_1,j_2,j_3,j_4)} \coloneqq (g_{j_1}, g_{j_2},g_{j_3},g_{j_4})\t \in \bb R^4$, and the cumulant $\kappa(\mb g^{(j_1,j_2,j_3,j_4)}_p)$ is defined as  the coefficients of $\prod_{j \in p} t_j$ multiplied by $|p|!$ in the Taylor expansion of $\log\bb E[\exp(\mb t\t \mb g^{(j_1,j_2,j_3,j_4)})]$. Further, (a) arises since
        (i) if one of $j_1,j_2,j_3,j_4$ does not share a block with the rest, then
        \eq{
E[E^{(k)}_{1,j_1}E^{(k)}_{1,j_2}E^{(k)}_{1,j_3}E^{(k)}_{1,j_4}] - \bb E[E^{(k)}_{1,j_1}E^{(k)}_{1,j_2}] \bb E[E^{(k)}_{1,j_3}E^{(k)}_{1,j_4}]  
        - \bb E[E^{(k)}_{1,j_1}E^{(k)}_{1,j_3}]\bb E[E^{(k)}_{1,j_2}E^{(k)}_{1,j_4}] - \bb E[E^{(k)}_{1,j_1}E^{(k)}_{1,j_4}]\bb E[E^{(k)}_{1,j_2}E^{(k)}_{1,j_3}] =  0.
        }
        (ii) if two of $j_1,j_2,j_3,j_4$ are in a block $S_{s_1}$, say, $j_1,j_2 \in S_{s_1}$,  and the rest of them are in another block $S_{s_2}$, then
        \longeq{
&E[E^{(k)}_{1,j_1}E^{(k)}_{1,j_2}E^{(k)}_{1,j_3}E^{(k)}_{1,j_4}] - \bb E[E^{(k)}_{1,j_1}E^{(k)}_{1,j_2}] \bb E[E^{(k)}_{1,j_3}E^{(k)}_{1,j_4}]
         \\ 
         & \qquad - \bb E[E^{(k)}_{1,j_1}E^{(k)}_{1,j_3}]\bb E[E^{(k)}_{1,j_2}E^{(k)}_{1,j_4}] - \bb E[E^{(k)}_{1,j_1}E^{(k)}_{1,j_4}]\bb E[E^{(k)}_{1,j_2}E^{(k)}_{1,j_3}]  
        \\ 
        = &   \bb E[E^{(k)}_{1,j_1}E^{(k)}_{1,j_2}] \bb E[E^{(k)}_{1,j_3}E^{(k)}_{1,j_4}] - \bb E[E^{(k)}_{1,j_1}E^{(k)}_{1,j_2}] \bb E[E^{(k)}_{1,j_3}E^{(k)}_{1,j_4}] = 0.
        }
        
       	We now turn to analyze the Gaussian analog $ \sum_{j_1,j_2,j_3,j_4\in[p]}\mb A_{j_1}\t g_{j_1}g_{j_2}\mb A_{j_2}\mb A_{j_3}\t g_{ j_3} g_{j_4} \mb A_{j_4}$. By Wick's formula, for every Gaussian random vector $\mb v$ we have $\mathbb E\big[\mb v\mb v\t \mb v\mb v\t\big] = \mathrm{Tr}(\textrm{Cov}(\mb v))\textrm{Cov}(\mb v) + 2 \textrm{Cov}(\mb v)^2$. Therefore, we have for the Gaussian analog 
        \begin{align}
        &  \Bignorm{\bb E\Big[\sum_{j_1,j_2,j_3,j_4\in[p]}\mb A_{j_1}\t g_{j_1}g_{j_2}\mb A_{j_2}\mb A_{j_3}\t g_{ j_3} g_{j_4} \mb A_{j_4}\Big] } 
        =  \Bignorm{\bb E\big[\mb A\t \mb g \mb g\t \mb A\mb A\t \mb g\mb g\t \mb A \big]}
        \\ 
        = & \text{Tr}(\mb A\t \mb \Sigma_k \mb A)\bignorm{\mb A\t \mb \Sigma_k \mb A} + 2\bignorm{\mb A\t \mb \Sigma_k \mb A}^2 
        \lesssim  K\bignorm{\mb A\t \mb \Sigma_k \mb A}^2.
        \end{align}
        
        Substituting this into \eqref{eq: sigma term of projected covariance} yields that 
       	\begin{align} 
        &\Bignorm{\bb E\big[ \mb A\t \mb E^{[k]}_1 {\mb E^{[k]}_1}\t \mb A \mb A\t \mb E^{[k]}_1 {\mb E^{[k]}_1}\t \mb A\big]}= \Bignorm{\bb E\big[ \sum_{j_1,j_2,j_3,j_4\in[p]}\mb A_{j_1}\t E^{(k)}_{1,j_1}E^{(k)}_{1,j_2}\mb A_{j_2}\mb A_{j_3}\t E^{(k)}_{1, j_3} E^{(k)}_{1, j_4} \mb A_{j_4} \big]}  
        \\
        \lesssim &K\bignorm{\mb A\t \mb \Sigma_k \mb A}^2 + lm^4 \max_{j\in[p]}\norm{\mb A_{j}}_2 \big(\max_{j_1,j_2,j_3,j_4}|\bb E[E^{(k)}_{1,j_1}E^{(k)}_{1,j_2}E^{(k)}_{1,j_3}E^{(k)}_{1,j_4}]| + \max_{j_5,j_6} |\bb E[E^{(k)}_{1,j_5}E^{(k)}_{1,j_6}] |\big)\\ 
        \lesssim &K \bignorm{\mb A\t \mb \Sigma_k \mb A}^2 + lm^4B^4 \norm{\mb A}\ti^4  
        \asymp  \mathrm{Tr}(\mb A\t \mb \Sigma_k \mb A)\bignorm{\mb A\t \mb \Sigma_k \mb A} + pm^3B^4 \norm{\mb A}\ti^4, \label{eq: upper bound on S squared}
        \end{align}
        since $ml \asymp p$, which combines \eqref{eq: decomposition of the square of S} leading to 
        \longeq{
        & \sigma(\mb S)^2 \lesssim n_k K\bignorm{\mb A\t \mb \Sigma_k \mb A}^2  + n_k pm^3B^4\norm{\mb A}\ti^4 \label{eq: upper bound on sigma(S)}.
        }
       	
        \item
       	Moving forward, upper bounding the parameter $v(\mb S)^2$ amounts to a variational characterization of the spectral norm: 
       	\longeq{
       	& v(\mb S)^2 = \norm{\mathrm{cov}(\text{vec}(\mb S))} = \sup_{\mb o\in \bb R^{K^2}: \norm{\mb o}_2=1} \mb o\t \mathrm{cov}(\text{vec}(\mb S)) \mb o \\ 
        \leq &  n_k\sup_{\mb O: \norm{\mb O}_F = 1}\Big[\bb E\big[ \mathrm{Tr}(\mb O \mb A\t {\mb E^{[k]}_1}\t \mb E^{[k]}_1 \mb A )^2\big] -  \mathrm{Tr}(\mb O\mb A\t\mb\Sigma_k \mb A)^2 \Big] \\
        \stackrel{\text{(a)}}{\leq} & n_k \bb E\big[\mathrm{Tr}( \mb A\t {\mb E^{[k]}_1}\t \mb E^{[k]}_1 \mb A )^2 \big] + n_k \mathrm{Tr}(\mb A\t\mb\Sigma_k \mb A)^2 
        \\ \stackrel{\text{(b)}}{\leq} & n_k K \bb E\big[\mathrm{Tr}( \mb A\t {\mb E^{[k]}_1}\t \mb E^{[k]}_1 \mb A \mb A\t {\mb E^{[k]}_1}\t \mb E^{[k]}_1 \mb A)  \big] + n_k \mathrm{Tr}(\mb A\t\mb\Sigma_k \mb A)^2 \\ 
        \leq & n_k K^2 \Bignorm{\bb E\big[  \mb A\t {\mb E^{[k]}_1}\t \mb E^{[k]}_1 \mb A \mb A\t {\mb E^{[k]}_1}\t \mb E^{[k]}_1 \mb A\big]} + n_k \mathrm{Tr}(\mb A\t\mb\Sigma_k \mb A)^2, \label{eq: upper bound on v(S) 0}
        }
        where (a) holds since 
        \eq{
        \sup_{\mb O: \norm{\mb O}_F = 1}\bb E\big[ \mathrm{Tr}(\mb O \mb A\t {\mb E^{(k)\top}_1} \mb E^{[k]}_1 \mb A )^2\big]\leq  \sup_{\mb O: \norm{\mb O}_F = 1} \norm{\mb O}^2 \bb E\big[ \mathrm{Tr}( \mb A\t {\mb E^{(k)\top}_1} \mb E^{[k]}_1 \mb A )^2\big] 
        \leq  \bb E\big[ \mathrm{Tr}( \mb A\t {\mb E^{(k)\top}_1} \mb E^{[k]}_1 \mb A )^2\big].
        }
        and (b) holds by the fact that $\mathrm{Tr}(\mb X)^2 \leq K \mathrm{Tr}(\mb X^2)$ for a symmetric matrix $\mb X \in \bb R^{K\times K}$. 

        To finish up, we invoke \eqref{eq: upper bound on S squared} again together with \eqref{eq: upper bound on v(S) 0} to derive that 
        \eq{
        v(\mb S)^2 \lesssim K^2
        \big( n_k K\bignorm{\mb A\t \mb \Sigma_k \mb A}^2  + n_k pm^3B^4\norm{\mb A}\ti^4\big)
        }
        
       	\item 
       	Finally, we make use of the modified logarithmic Sobolev inequality (Lemma~\ref{lemma: Generalized Modified Logarithmic Sobolev Inequality I}) to upper bound $R_{2p}(\mb S)$ that 
       	\begin{align}
       	& R_q(\mb S) \leq \max_{i}(n_k)^{\frac{1}{2p}}\big(\mathbb E\big[\bignorm{\mb A\t {\mb E^{[k]}_{i}}}_2^{4p}\big]^{\frac{1}{q}} + \bb E\big[\bignorm{\mb A\t \mb E_i^{(k)}}_2^2\big]\big)\\
       	 \leq & (n_k)^{\frac{1}{2p}}\max_i\Big( \mathbb E\big[\big(\bignorm{\mb A\t {\mb E^{[k]}_{i}}\t}_2 - \mathbb E\bignorm{{\mb A\t \mb E_{i}^{(k)}}}_2\big)^{4p}\big]^{\frac{1}{2p}}   + \big(\mathbb E\bignorm{\mb A\t {\mb E^{[k]}_{i}} }_2\big)^2 + \bb E\big[\bignorm{\mb A\t \mb E_i^{(k)}}_2^2\big] \Big)\\ 
         \leq & (n_k)^{\frac{1}{2p}}\max_i\Big(\mathbb E\big[\big(\bignorm{\mb A\t {\mb E^{[k]}_{i}}}_2 - \mathbb E\bignorm{{\mb A\t \mb E_{i}^{(k)}}}_2\big)^{4p}\big]^{\frac{1}{2p}} +  \big(\mathbb E\bignorm{\mb A\t {\mb E^{[k]}_{i}}}_2^2 \big) \Big)
       	\\
       	 \lesssim & (n_k)^{\frac{1}{2p}}\big(mB^2p \big)\bignorm{\mb A}^2  + (n_k)^{\frac{1}{2p}}\mathrm{Tr}(\mb A\t \mb \Sigma_k \mb A). 
        \end{align}
       	where the second inequality above follows from Cauchy-Schwarz and the last line follows by Lemma~\ref{lemma: Generalized Modified Logarithmic Sobolev Inequality I} provided the fact that $\bignorm{\mb x\t \mb A}_2$ is a $\bignorm{\mb A}$-Lipschitz convex function of $\mb x$. 
        \end{enumerate}
       	
       	With these pieces in place, we plug the above upper bounds for $\sigma(\mb S), v(\mb S), R_{2p}(\mb S)$ into \eqref{eq: universality and free prob for covariance estimation} and derive that
       	\longeq{
       	& \bb E\big[\mathrm{Tr}(\mb S^{2p})\big]^{\frac{1}{2p}} \leq \bb E\big[K\mathrm{tr}(\mb S^{2p})\big]^{\frac{1}{2p}} \\
        \lesssim & K^{\frac{1}{p}}p^{\frac{3}{4}} K^{\frac{1}{2}}\big(n_k K \bignorm{\mb A\t \mb \Sigma_k \mb A}^2 + n_kp m^3B^4\bignorm{\mb A}\ti^4\big)^{\frac{1}{2}}+ 
        K^{\frac{1}{p}}\big((n_k)^{\frac{1}{2p}}\big(mB^2p \big)\bignorm{\mb A}^2 + (n_k)^{\frac{1}{2p}}\mathrm{Tr}(\mb A\t \mb \Sigma_k \mb A)\big).\label{eq: upper bound on 2p moment of S}
       	}
       	
       	Letting $p = \lceil t \rceil$, applying Markov's inequality to \eqref{eq: upper bound on 2p moment of S} gives that  
       	\longeq{
       	& \bignorm{\mb S} \lesssim  K^{\frac{1}{t}}t^{\frac{3}{4}} K^{\frac{1}{2}} n_k^{\frac{1}{2}}\big(\mathrm{Tr}(\mb A\t \mb \Sigma_k \mb A)\bignorm{\mb A\t \mb \Sigma_k \mb A}  + pm^3B^4\bignorm{\mb A}\ti^4\big)^{\frac{1}{2}}\\ 
        & + K^{\frac{1}{t}}\big((n_k)^{\frac{1}{2t}}\big(mB^2t \big)\bignorm{\mb A}^2 + (n_k)^{\frac{1}{2t}}\mathrm{Tr}(\mb A\t \mb \Sigma_k \mb A)\big) 
        \label{eq: genereal universality}
       	}
       	with probability at least $1- e^{-t}$. This concludes the first part of this lemma. 

        In the end, substitution of $\mb A = \mb V^*$ and $t = c\log d$ for a sufficiently large constant $c$ into \eqref{eq: genereal universality}  yields that 
        \begin{align}
             &\Bignorm{\frac{{\mb V^{*\top}} \mb E_{\mc I_k(\mb z^*), \cdot}\t \mb E^{[k]} \mb V^*}{\sum_{i\in[n]}\ind{\{z_i^*=k\}}} - \mb S_k^*} =  \frac{1}{n_k} \norm{\mb S} \\ 
            \lesssim & \sqrt{\frac{\beta K^2 }{n}} e^{\frac{\log K}{c\log d}}\big(\log d\big)^{\frac{3}{4}} \Big(\norm{\mb S_k^*}^2 + pm^3B^4\norm{\mb V^*}\ti ^4 / K\Big)^{\frac{1}{2}}  
           + e^{\frac{\log K}{\log d}}\frac{\beta K\big(mB^2\log d + K\norm{\mb S_k^*}\big)}{n}\\
       		\lesssim &\sqrt{\frac{\beta K^2}{n}}\bar \sigma^2(\log d)^{\frac{3}{4}} + \sqrt{\frac{\beta p Km^3 B^4}{n}} \norm{\mb V^*}\ti ^2(\log d)^{\frac{3}{4}}
         \lesssim  \sqrt{\frac{\beta K^2}{n}}\bar \sigma^2\log d
        \end{align}
        holds with probability at least $1 - O(d^{-c})$, where we invoke the fact $\log K\lesssim \log d$ and the conditions
        \begin{align} 
            &  n \bar \sigma^2 \gg \mu_1 K mB^2, \quad p \bar \sigma^2 \gg \mu_2 K m^2 B^2. 
    \end{align}
       	\end{proof}
        \begin{remark}
       		\emph{Improvement upon Bernstein's inequality. } We additionally remark that, compared with the \emph{S-Universality} result, the Bernstein inequality could only provide us the upper bound 
       		\longeq{ \norm{\frac{{\mb V^{*\top}} \mb E_{\mc I_k(\mb z^*), \cdot}\t \mb E^{[k]} \mb V^*}{\sum_{i\in[n]}\ind{\{z_i^*=k\}}} - \mb S_k}\lesssim &\frac{1}{n_k}\cdot\underbrace{\sigma(\mb S)}_{\text{defined in \eqref{eq: decomposition of the square of S}}} \cdot \sqrt{\log d} + \underbrace{\sqrt{\frac{\beta pK}{n}}B \log d}_{\text{troublesome}} \\ 
         \lesssim & \sqrt{\frac{\beta K}{n}}\bar \sigma^2(\log d)^{\frac{1}{2}} + \sqrt{\frac{\beta p m^3 B^4}{n}} \norm{\mb V^*}\ti ^2+\sqrt{\frac{\beta pK}{n}}B \log d, 
         }
         with probability at least $1 -O(d^{-c})$,
       		where the last term would be unsatisfactory when $p$ is large, although $(\log d)^{\frac{3}{4}}$ in our current upper bound is slightly looser compared with the first term above. 
       	\end{remark}

\subsubsection{Proof of Lemma~\ref{lemma: concentration on second part of covariance estimation}}
\label{subsubsec: proof of lemma: concentration on second part of covariance estimation}
    To deal with the targeted fourth-order term with respect to $\{\mb E_i\}$, we employ the decoupling inequality for U-statistics from \cite{de1995decoupling}. In particular, for any $k_1,k_2 \in[K]$, consider the $(k_1,k_2)$-th entry of the difference of interest: 
    \begin{align}
        & {\mb U_{\cdot, k_1}^*}\t \mc H(\mb E\mb E\t)_{\cdot, \mb I_k(\mb z^*)}\mc H(\mb E\mb E\t)_{\mb I_k(\mb z^*),\cdot } \mb U_{\cdot, k_2}^* \\ 
        = & \kappa_{k_1,k_2,k,k}\Big[\sum_{i_1 , i_2 \in \mc I_k(\mb z^*),i_1 \neq i_2}\ip{\mb E_{i_1}, \mb E_{i_2}}^2 + \sum_{i_1 , i_2 , i_3 \in \mc I_k(\mb z^*), i_1\neq i_2 \neq i_3} \ip{\mb E_{i_1}, \mb E_{i_2}}\ip{\mb E_{i_2}, \mb E_{i_3}} \Big] 
        \\ 
        &  + \sum_{k_3 \in [K] \backslash\{k\}}  \kappa_{k_1,k_2,k_3,k_3} \sum_{i_1\in \mc I_{k_3}(\mb z^*),i_2 \in \mc I_{k}(\mb z^*) }\ip{\mb E_{i_1}, \mb E_{i_2}}^2 \\ 
        &  +  \sum_{k_3, k_4 \in [K] \backslash\{k\}, k_3 \neq k_4}  \kappa_{k_1,k_2,k_3,k_4} \sum_{i_1\in \mc I_{k_3}(\mb z^*), i_2 \in \mc I_k(\mb z^*),i_3 \in \mc I_{k_4}(\mb z^*) }\ip{\mb E_{i_1}, \mb E_{i_2}}\ip{\mb E_{i_2}, \mb E_{i_3}} \\ 
        & +  2\sum_{k_3 \in [K]\backslash \{k\}} \sum_{i_1,i_2 \in \mc I_k(\mb z^*), i_1 \neq i_2,  i_3 \in \mc I_{k_3}(\mb z^*)} \kappa_{k_1,k_2,k,k_3}\ip{\mb E_{i_1}, \mb E_{i_2}}\ip{\mb E_{i_2}, \mb E_{i_3}} , 
        \label{eq: decomposition of Ut H(EEt) H(EEt) U}
    \end{align}
    where $\kappa_{k_1,k_2,k_3,k_4} \coloneqq (\mb U^*)_{i_1,k_1}(\mb U^*)_{i_1,k_2} $ for any $i_1\in \mc I_{k_3}(\mb z^*)$ and $i_2 \in \mc I_{k_4}(\mb z^*)$. This definition is well-posed since $\mb U^*_{i,k}$ is identical across $i \in \mc I_k(\mb z^*)$, owing to the fact that $\mb Z_i^*$ is the same for all such $i$.
    
    Before preceeding, we introduce the following auxiliary lemma to establish the subsequent concentration inequalities whose proof is presented in Section~\ref{subsubsec: proof of lemma: concentration for decoupled quantities}. 
    \begin{lemma}
    Consider three Gaussian distributed matrices $\mb G^{(1)} \in \bb R^{N_1 \times p}$, $\mb G^{(2)} \in \bb R^{N_2 \times p}$, and $\mb G^{(1)} \in \bb R^{N_1 \times p}$, whose rows independently follow $\mc N(0, \bo \Sigma_1)$, $\mc N(0, \bo \Sigma_2)$, and $\mc N(0, \bo \Sigma_3)$, respectively. Then it holds with probability at least $1- O(d^{-c})$ that 
    \begin{align}
        & \Big|\norm{\mb G^{(1)} {{}\mb G^{(2)}}\t}_F^2 - N_1 N_2 \tr(\bo \Sigma_1 \bo \Sigma_2) \Big|\\ 
        \lesssim & \sqrt{N_1 N_2}\tr(\bo \Sigma_1 \bo \Sigma_2) \sqrt{\log(p\vee N)} +N_1 \sqrt{N_2} \tr(\bo \Sigma_1 \bo \Sigma_2)^{\frac12}\norm{\bo \Sigma_1\bo \Sigma_2}^{\frac12} \sqrt{\log(p\vee N)} , \\ 
        & \Big| \tr\big(\mb 1_{N_3\times N_1}\mb G^{(1)} {{}\mb G^{(2)}}\t \mb G^{(2)} {{}\mb G^{(3)}}\t\big) \Big| \\ 
        \lesssim & \sqrt{N_1 N_3} \Big(N_2\tr(\bo \Sigma_1 \bo \Sigma_2)^{\frac12} \norm{\bo \Sigma_2\bo \Sigma_3}^{\frac12} + \sqrt{N_2}\tr(\bo \Sigma_1 \bo \Sigma_2)^{\frac12}\tr(\bo \Sigma_3 \bo \Sigma_2)^{\frac12}  \Big) \sqrt{\log(p\vee N)}, \\ 
        & \norm{\mb 1_{N_1}\t \mb G^{(1)} {\mb G^{(2)}}\t \mb G^{(2)} \mb T}_2 \lesssim (\sqrt{N_2}\norm{\bo \Sigma_1 \bo \Sigma_2}^{\frac12} + \tr(\bo \Sigma_1 \bo \Sigma_2)^{\frac12}) (\norm{\mb T\t \bo \Sigma_2 \mb T}^{\frac12} \sqrt{N_2 } ) \sqrt{rN_1\log(p\vee N)}
        \end{align}
        where $\mb T \in \bb O(p,r)$.  
        
    Moreover, for bounded noise matrices $\mb E^{(1)}$, $\mb E^{(2)}$, and $\mb E^{(3)}$ obeying Assumption~\ref{assumption: bounded noise} whose rows are of zero mean and covariance matrices $\bo \Sigma_1,~\bo \Sigma_2$, and $\bo \Sigma_3$, respectively, we additionally assume that $\sqrt{m}B \Big(\frac{\max_{i\in[3]}\norm{\bo \Sigma_i}^{\frac12}}{\max_{i,i' \in [3]}\sqrt{N_i}\norm{\bo \Sigma_i \bo \Sigma_{i'}}^{\frac12}} \vee \frac{\max_{i\in[3]}\norm{\bo \Sigma_i}^{\frac12}}{\max_{i,i' \in [3]}\tr(\bo \Sigma_i \bo \Sigma_{i'})^{\frac12}} \Big) \ll 1$, $\mathrm{rk}(\bo \Sigma_1 \bo \Sigma_2) \coloneqq \frac{\tr(\bo \Sigma_1\bo \Sigma_2)}{\norm{\bo \Sigma_1 \bo \Sigma_2}} \gg (\log(p\vee N))^6$, and $\min_{i\in[3]} N_i \gg (\log(p\vee N))^2$. Then
    it holds with probability at least $1- O(d^{-c})$ that  
    \begin{align}
        & \Big|\norm{\mb E^{(1)} {{}\mb E^{(2)}}\t}_F^2 - N_1 N_2 \tr(\bo \Sigma_1 \bo \Sigma_2) \Big| \\ 
        \lesssim & N_1^{\frac34} N_2^{\frac34} \tr(\bo\Sigma_1 \bo \Sigma_2) \log(p\vee N) +  N_1 N_2\tr(\bo \Sigma_1 \bo \Sigma_2)^{\frac56} \norm{\bo \Sigma_1 \bo \Sigma_2}^{\frac16} \log(p\vee N)  , \\ 
        & \Big| \tr(\mb 1_{n\times n} \mb E^{(1)} \mb E^{(2)\top} \mb E^{(2)} \mb E^{(3)\top} \Big| \\ 
        \lesssim  & \sqrt{N_1 N_3} \Big(N_2\tr(\bo \Sigma_1 \bo \Sigma_2)^{\frac12} \norm{\bo \Sigma_2\bo \Sigma_3}^{\frac12} + \sqrt{N_2}\tr(\bo \Sigma_1 \bo \Sigma_2)^{\frac12}\tr(\bo \Sigma_3 \bo \Sigma_2)^{\frac12}  \Big) (\log(p\vee N))^2, \\ 
        &  \Big|\mb 1_{N_1}\t \mb E_1{ \mb E^{(2)}}\t \mb E^{(2)} \mb T\Big| \lesssim (\sqrt{N_2}\norm{\bo \Sigma_1 \bo \Sigma_2}^{\frac12} + \tr(\bo \Sigma_1 \bo \Sigma_2)^{\frac12}) (\norm{\mb T\t \bo \Sigma_2 \mb T}^{\frac12} \sqrt{N_2 } ) \sqrt{rN_1\log(p\vee N)}. 
    \end{align}
    \label{lemma: concentration for decoupled quantities}
\end{lemma}

     In what follows, we seperately parse the four terms on the RHS of \eqref{eq: decomposition of Ut H(EEt) H(EEt) U}.

    \begin{itemize}
        \item     The first terms on the RHS above allow us to evaluate each term by replacing them with independence copies: by virtue of \cite[Theorem~1]{de1995decoupling}, there exists a universal constant $C$ such that 
    \begin{equation}
        \bb P \Big[\big|\sum_{i_1 ,i_2 \in \mc I_k(\mb z^*), i_1\neq i_2} \ip{\mb E_{i_1}, \mb E_{i_2}}^2 - n_k(n_k - 1) \tr(\bo \Sigma_k\bo \Sigma_k) \big|  \geq t\Big] \leq  C\bb P \Big[C\cdot \big|\sum_{i_1  i_2 \in [n],i_1\neq i_2} \big(\ip{\mb E^{(1)}_{i_1}, \mb E^{(2)}_{i_2}}^2 - n_k(n_k - 1)\tr(\bo \Sigma_k\bo \Sigma_k) \big) \big|  \geq t\Big],
    \end{equation}
    where $\mb E^{(1)}$ and $\mb E^{(2)}$ are independent copies of $\mb E$. For the decoupled counterpart, we further break it down as 
    \begin{align}
        & \big|\sum_{i_1 \neq i_2 \in \mc I_k(\mb z^*)} \Big(\ip{\mb E^{(1)}_{i_1}, \mb E^{(2)}_{i_2}}^2 - \tr\big(\bo \Sigma_k^2\big) \Big) \big| \\ 
        \leq & \underbrace{\bigg| \Bignorm{ \mb E^{(1)}_{\mc I_k(\mb z^*)} {\mb E^{(2)}_{\mc I_k(\mb z^*)}}\t }_F^2 - n_k^2 \tr\big(\bo \Sigma_k^2\big) \bigg|}_{A_1}  + \underbrace{\sum_{i \in \mc I_k(\mb z^*)} \Big|\ip{\mb E^{(1)}_{i}, \mb E^{(2)}_{i}}^2 - \tr\big(\bo \Sigma_k^2\big)\Big|}_{A_2}.     \end{align}
        Then it boilds down to bounding $A_1$ and $A_2$, respectively. By Lemma~\ref{lemma: concentration for decoupled quantities}, one has with probability at least $1 - O(d^{-c})$ that 
        \begin{align}
            & A_1 \lesssim n_k^{\frac32} \norm{\bo \Sigma_k}_F^2 \log d + n_k^{\frac32} \norm{\bo \Sigma_k}_F^{\frac53}\norm{\bo \Sigma_k^2}^{\frac13} \log d
        \end{align}
        under either noise setting. 
        For the term $A_2$, one easily infer from the triangle inequality that $
             A_2 \leq \sum_{i\in \mc I_k(\mb z^*)} \ip{\mb E^{(1)}_i, \mb E^{(2)}_i}^2 + n_k \tr(\bo \Sigma_k^2)$. For the first term on the RHS, one has 
             \begin{align}
                 & \ip{\mb E^{(1)}_i, \mb E^{(2)}_i}
                 \stackrel{\text{by Lemma~6.2.2 in \cite{vershynin2010introduction}}}{\lesssim} \norm{\bo \Sigma_k}_F \sqrt{\log d} + \norm{\bo \Sigma_k} \log d \lesssim \norm{\bo \Sigma_k}_F \sqrt{\log d} ~ \text{ for Gaussian noise}, \\ 
                 & \ip{\mb E^{(1)}_i, \mb E^{(2)}_i} \stackrel{\text{by matrix Bernstein ineq.}}{\lesssim} \norm{\bo \Sigma_k}_F \sqrt{\log d} + mB^2 \log d \lesssim \norm{\bo \Sigma_k}_F \sqrt{\log d} ~ \text{ for bounded noise}
             \end{align}
            with probability at least $1 - O(d^{-c-1})$, which leads to the result that $A_2 \lesssim n_k \norm{\bo \Sigma_k}_F^2 \log d$ with probability at least $1 - O(d^{-c})$.

            Combining these two pieces together with the decoupling technique, we thus have 
            \begin{align}
                &\sum_{i_1 , i_2 \in \mc I_k(\mb z^*),i_1 \neq i_2}\ip{\mb E_{i_1}, \mb E_{i_2}}^2   \lesssim n_k^{\frac32} \norm{\bo \Sigma_k}_F^2 \log d + n_k^{\frac32} \norm{\bo \Sigma_k}_F^{\frac53}\norm{\bo \Sigma_k^2}^{\frac13} \sqrt{\log d}
            \end{align}
            with probability at least $1 - O(d^{-c})$. 

            On the other hand, regarding $\sum_{i_1 , i_2 , i_3 \in \mc I_k(\mb z^*), i_1\neq i_2 \neq i_3} \ip{\mb E_{i_1}, \mb E_{i_2}}\ip{\mb E_{i_2}, \mb E_{i_3}} $ we again invoke the decoupling result \cite[Theorem~1]{de1995decoupling} to obtain that 
            \begin{equation}
                \bb P \Big[\big|\sum_{i_1 , i_2 , i_3 \in \mc I_k(\mb z^*), i_1\neq i_2 \neq i_3} \ip{\mb E_{i_1}, \mb E_{i_2}}\ip{\mb E_{i_2}, \mb E_{i_3}} \big|  \geq t\Big] \leq  C\bb P \Big[C\cdot \big|\sum_{i_1 , i_2 , i_3 \in \mc I_k(\mb z^*), i_1\neq i_2 \neq i_3} \ip{\mb E^{(1)}_{i_1}, \mb E^{(2)}_{i_2}}\ip{\mb E^{(2)}_{i_2}, \mb E^{(3)}_{i_3}}  \big|  \geq t\Big]
            \end{equation}
            for some constant $C$, where $\mb E^{(1)},~\mb E^{(2)}$, and $\mb E^{(3)}$ are independent copies of $\mb E$. 
            Lemma~\ref{lemma: concentration for decoupled quantities} allows us to upper bound $\sum_{i_1 , i_2 , i_3 \in \mc I_k(\mb z^*), i_1\neq i_2 \neq i_3} \ip{\mb E^{(1)}_{i_1}, \mb E^{(2)}_{i_2}}\ip{\mb E^{(2)}_{i_2}, \mb E^{(3)}_{i_3}}$ as follows: 
            \begin{align}
                & \sum_{i_1 , i_2 , i_3 \in \mc I_k(\mb z^*), i_1\neq i_2 \neq i_3} \ip{\mb E_{i_1}, \mb E_{i_2}}\ip{\mb E_{i_2}, \mb E_{i_3}}  
                \leq  \Big|\tr\big(\mb 1_{n_k \times n_k} \mb E_{\mc I_k(\mb z^*),\cdot}^{(1)} \mb E_{\mc I_k(\mb z^*),\cdot}^{(2)\top}\mb E_{\mc I_k(\mb z^*),\cdot}^{(2)}\mb E_{\mc I_k(\mb z^*),\cdot}^{(3)\top}\big) \Big|\\ 
                & + \Big| \tr\big(\mb 1_{n_k \times n_k} \text{diag}(\mb E_{\mc I_k(\mb z^*),\cdot}^{(1)} \mb E_{\mc I_k(\mb z^*),\cdot}^{(2)\top})\mb E_{\mc I_k(\mb z^*),\cdot}^{(2)}\mb E_{\mc I_k(\mb z^*),\cdot}^{(3)\top}\big) \Big| \\ 
                & + \Big|\tr\big(\mb 1_{n_k \times n_k} \mb E_{\mc I_k(\mb z^*),\cdot}^{(1)} \mb E_{\mc I_k(\mb z^*),\cdot}^{(2)\top}\text{diag}(\mb E_{\mc I_k(\mb z^*),\cdot}^{(2)}\mb E_{\mc I_k(\mb z^*),\cdot}^{(3)\top})\big) \Big|  \\ 
                & + \Big|\tr\big(\mb 1_{n_k \times n_k} \text{diag}(\mb E_{\mc I_k(\mb z^*),\cdot}^{(1)} \mb E_{\mc I_k(\mb z^*),\cdot}^{(2)\top})\text{diag}(\mb E_{\mc I_k(\mb z^*),\cdot}^{(2)}\mb E_{\mc I_k(\mb z^*),\cdot}^{(3)\top})\big) \Big|
                 \\ 
                \lesssim &  n_k^2 \norm{\bo \Sigma_k}_F \norm{\bo \Sigma_k} (\log d)^2 + n_k^{\frac32}\norm{\bo \Sigma_k}_F^2 (\log d )^2
            \end{align}
            holds with probability at least $1 - O(d^{-c})$ under either noise setting, where we used the facts that 
            \begin{align}
                & \norm{\text{diag}(\mb E_{\mc I_k(\mb z^*),\cdot}^{(1)}\mb E_{\mc I_k(\mb z^*),\cdot}^{(2)\top})\big)} \vee \norm{\text{diag}(\mb E_{\mc I_k(\mb z^*),\cdot}^{(2)}\mb E_{\mc I_k(\mb z^*),\cdot}^{(3)\top})\big)} 
                \lesssim   \norm{\bo \Sigma_k}_F\sqrt{\log d}, \\ 
                & \tr\big(\mb 1_{n_k \times n_k} \mb E_{\mc I_k(\mb z^*),\cdot}^{(1)} \mb E_{\mc I_k(\mb z^*),\cdot}^{(2)\top}\big) \vee  \tr\big(\mb 1_{n_k \times n_k} \mb E_{\mc I_k(\mb z^*),\cdot}^{(3)} \mb E_{\mc I_k(\mb z^*),\cdot}^{(2)\top}\big) \lesssim n_k \norm{\bo \Sigma_k}^{\frac12}_F \norm{\bo \Sigma_k}^{\frac12} 
            \end{align}
            with probability at least $1- O(d^{-c})$ by the Hanson-Wright inequality or the matrix Bernstein inequality, along with Assumptions~\ref{assumption: algorithm}~and~\ref{assumption: bounded noise}. 

            Taking these parts collectively yields that 
            \begin{align}
                & \kappa_{k_1,k_2,k,k}\Big[\sum_{i_1 , i_2 \in \mc I_k(\mb z^*),i_1 \neq i_2}\ip{\mb E_{i_1}, \mb E_{i_2}}^2 + \sum_{i_1 , i_2 , i_3 \in \mc I_k(\mb z^*), i_1\neq i_2 \neq i_3} \ip{\mb E_{i_1}, \mb E_{i_2}}\ip{\mb E_{i_2}, \mb E_{i_3}} \Big]  \\ \lesssim &  \kappa_{k_1,k_2,k,k} \Big[n_k^2 \norm{\bo \Sigma_k}_F \norm{\bo \Sigma_k} \log d + n_k^{\frac32}\norm{\bo \Sigma_k}_F^2 \log d \Big] \\ 
                \lesssim & n_k \tilde\sigma^3 \sigma\sqrt{p} \log d+ \sqrt{n_k} \tilde\sigma^2 \sigma^2 p \log d
            \end{align}
            with probability at least $1 - O(d^{-c})$. 
            \item For the second term on the RHS of \eqref{eq: decomposition of Ut H(EEt) H(EEt) U}, invoking Lemma~\ref{lemma: concentration for decoupled quantities} gives that 
            \begin{align}
                & \Big|\sum_{k_3 \in [K] \backslash\{k\}}  \kappa_{k_1,k_2,k_3,k_3} \sum_{i_1\in \mc I_{k_3}(\mb z^*),i_2 \in \mc I_{k}(\mb z^*) }\ip{\mb E_{i_1}, \mb E_{i_2}}^2 - \sum_{k_3 \in [K] \backslash\{k\}} \kappa_{k_1,k_2,k_3,k_3} n_k n_{k_3} \tr(\bo \Sigma_k \bo \Sigma_{k_3}) \Big| \\ 
                \lesssim &   \sum_{k_3 \in [K] \backslash\{k\}} \kappa_{k_1,k_2,k_3,k_3} \Big[(n_k n_{k_3})^{\frac34} \tr(\bo \Sigma_k \bo \Sigma_{k_3}) \log d + n_{k_3} n_k \tr(\bo \Sigma_k \bo \Sigma_{k_3})^{\frac56} \norm{\bo \Sigma_k \bo \Sigma_{k_3}}^{\frac16} \log d\Big] \\ 
                \lesssim & n_k^{\frac34}n_{k_3}^{-\frac14} \tilde\sigma^2 \sigma^2 \log d+ n_k p^{\frac56}\tilde\sigma^{\frac73} \sigma^{\frac53}\log d \lesssim \beta^{\frac14} n_k^{\frac12} p\tilde\sigma^2 \sigma^2 \log d+ n_kp^{\frac56} \tilde\sigma^{\frac73} \sigma^{\frac53}\log d
            \end{align}
            with probability at least $1 - O(d^{-c})$ under either noise. 
            \item With respect to the third term on the RHS of \eqref{eq: decomposition of Ut H(EEt) H(EEt) U}, we again invoke Lemma~\ref{lemma: concentration for decoupled quantities} to arrive at 
            \begin{align}
                & \Big| \sum_{k_3, k_4 \in [K] \backslash\{k\}, k_3 \neq k_4}  \kappa_{k_1,k_2,k_3,k_4} \sum_{i_1\in \mc I_{k_3}(\mb z^*), i_2 \in \mc I_k(\mb z^*),i_3 \in \mc I_{k_4}(\mb z^*) }\ip{\mb E_{i_1}, \mb E_{i_2}}\ip{\mb E_{i_2}, \mb E_{i_3}}\Big| \\ 
                \lesssim & \sum_{k_3, k_4 \in [K] \backslash\{k\}, k_3 \neq k_4} \kappa_{k_1,k_2,k_3,k_4} \Big[ \sqrt{n_{k_3}n_{k_4}}n_k \tr(\bo \Sigma_{k_3} \bo \Sigma_{k})^{\frac12} \norm{\bo \Sigma_k \bo \Sigma_{k_4}}^{\frac12} (\log d )^2 \\ 
                & \qquad + \sqrt{n_{k_3}n_{k_4}n_k} \tr(\bo \Sigma_{k_3} \bo \Sigma_k )^{\frac12} \tr(\bo \Sigma_{k_4}\bo \Sigma_k)^{\frac12} (\log d)^2\Big] \\ 
                \lesssim & \tilde\sigma^3 \sigma \sqrt{p}n_k (\log d)^2 + \sqrt{n_k} \tilde\sigma^2 \sigma^2 p(\log d)^2
            \end{align}
            with probability at least $1 - O(d^{-c})$, since $p \gtrsim (\log d)^3$.  
            \item Finally, regarding $\sum_{k_3 \in [K]\backslash \{k\}} \sum_{i_1,i_2 \in \mc I_k(\mb z^*), i_3 \in \mc I_{k_3}(\mb z^*)} \kappa_{k_1,k_2,k,k_3}\ip{\mb E_{i_1}, \mb E_{i_2}}\ip{\mb E_{i_2}, \mb E_{i_3}}$, one has 
            \begin{align}
                & \Big|\sum_{k_3 \in [K]\backslash \{k\}} \sum_{i_1,i_2 \in \mc I_k(\mb z^*),i_1\neq i_2, i_3 \in \mc I_{k_3}(\mb z^*)}\kappa_{k_1,k_2,k,k_3} \ip{\mb E_{i_1}, \mb E_{i_2}}\ip{\mb E_{i_2}, \mb E_{i_3}} \Big| \\ 
                \lesssim & \sum_{k_3 \in [K]\backslash \{k\}} \kappa_{k_1,k_2, k,k_3} \norm{\mc H(\mb E_{\mc I_k(\mb z^*), \cdot}\mb E\t_{\mc I_k(\mb z^*), \cdot})} \tr\big(\mb 1_{n_k}\t \mb E_{\mc I_k(\mb z^*), \cdot}\mb E\t_{\mc I_{k_3}(\mb z^*), \cdot} \mb 1_{n_{k_3}} \big) \\ 
                \lesssim &\sum_{k_3 \in [K]\backslash \{k\}} \kappa_{k_1,k_2, k,k_3} \Big(\tilde\sigma^2 n_k + \tilde\sigma  \sigma \sqrt{n_k p}    \Big) \Big(\sqrt{n_k n_{k_3}} \tr(\bo \Sigma_k \bo\Sigma_{k_3})^{\frac12} \sqrt{\log d}\Big) \\ 
                \lesssim & \tilde\sigma^3 \sigma n_k\sqrt{p} + \tilde\sigma^2 \sigma^2 \sqrt{n_k} p \sqrt{\log d} 
            \end{align}
            with probability at least $1 - O(d^{-c})$.  
    \end{itemize}
    Provided these pieces, \eqref{eq: decomposition of Ut H(EEt) H(EEt) U} turns out to be 
    \begin{align}
        & \Big| {\mb U_{\cdot, k_1}^*}\t \mc H(\mb E\mb E\t)_{\cdot, \mb I_k(\mb z^*)}\mc H(\mb E\mb E\t)_{\mb I_k(\mb z^*),\cdot } \mb U_{\cdot, k_2}^* - {\mb U_{k_1,\cdot}^*}\t \text{diag}\big(\tr(\sum_{i'\in[n], i\neq i'} \bo \Sigma_{z_i^*} \bo \Sigma_{z_{i'}^*})\big)_{i\in[n]} \mb U_{k_2,\cdot}^*\Big|\\ \lesssim &\beta^{\frac14} n_k^{\frac12} p\tilde\sigma^2 \sigma^2 \log d + n_k \sqrt{p}\tilde\sigma^3 \sigma (\log d)^2+ n_kp^{\frac56} \tilde\sigma^{\frac73} \sigma^{\frac53}\log d
    \end{align}
    with probability at least $1- O(d^{-c})$. As a consequence, the norm of $ {\mb U^*}\t \mc H(\mb E\mb E\t)_{\cdot, \mb I_k(\mb z^*)}\mc H(\mb E\mb E\t)_{\mb I_k(\mb z^*),\cdot } \mb U^*$ is controlled with probability at least $1- O(d^{-c})$ as follows: 
    \begin{align}
        & \Bignorm{ {\mb U^*}\t \mc H(\mb E\mb E\t)_{\cdot, \mb I_k(\mb z^*)}\mc H(\mb E\mb E\t)_{\mb I_k(\mb z^*),\cdot } \mb U^* -{\mb U^*}\t \text{diag}\big(\tr(\sum_{i'\in[n], i\neq i'} \bo \Sigma_{z_i^*} \bo \Sigma_{z_{i'}^*})\big)_{i\in[n]} \mb U^*}_F  \\ 
        \lesssim & \beta^{\frac14} K \sqrt{n_k} p \tilde\sigma^2 \sigma^2 \log d+ K n_k \sqrt{p} \tilde\sigma^3 \sigma (\log d)^2 + K n_k p^{\frac56} \tilde\sigma^{\frac73} \sigma^{\frac53} \log d. 
    \end{align}

    To prove the second upper bound in Lemma~\ref{lemma: concentration on second part of covariance estimation}, we decompose the target quantity for an $k_1 \in[K]$ as follows: 
    \begin{align}
        & {\mb U_{\cdot, k_1}^*}\t \mc H(\mb E \mb E\t)_{\cdot, \mc I_k(\mb z^*)} \mb E_{\mc I_k(\mb z^*), \cdot}\mb V^* = \kappa'_{k_1, k} \sum_{i_1 \neq i_2\in \mc I_k(\mb z^*)} \ip{\mb E_{i_1}, \mb E_{i_2}} \mb E_{i_2,\cdot} \mb V^* \\
        & \qquad  + \sum_{k'\in[K] \backslash\{k\}}  \sum_{i_1\in \mc I_{k'}(\mb z^*), i_2 \in \mc I_k(\mb z^*)} \kappa'_{k_1, k'} \ip{\mb E_{i_1}, \mb E_{i_2}}  \mb E_{i_2,\cdot} \mb V^*,
        \label{eq: first term in the decomposition of the corr term}
    \end{align}
    where $\kappa'_{k,k'} \coloneqq \mb U^*_{i_1,k}$ for any $i \in \mc I_{k'}(\mb z^*)$. For the first term on the RHS of \eqref{eq: first term in the decomposition of the corr term}, 
    we invoke the decoupling result \cite[Theorem~1]{de1995decoupling} again to derive that 
    \begin{align}
        & \bb P\Big[\bignorm{ \sum_{i_1 \neq i_2 \in \mc I_k(\mb z^*)} \ip{\mb E_{i_1}, \mb E_{i_2}} \mb E_{i_2,\cdot} \mb V^*}_2 \Big] \leq C\bb P\Big[\bignorm{ \sum_{i_1\neq i_2 \in \mc I_k(\mb z^*)}  \ip{\mb E_{i_1}^{(1)}, \mb E^{(2)}_{i_2}} \mb E^{(2)}_{i_2,\cdot} \mb V^* }_2 \Big]
    \end{align}
    for some constant $C$. 

    To deal with the decoupled version, we write: 
    \begin{equation}
        \bignorm{ \sum_{i_1\neq i_2 \in \mc I_k(\mb z^*)}  \ip{\mb E_{i_1}^{(1)}, \mb E^{(2)}_{i_2}} \mb E^{(2)}_{i_2,\cdot} \mb V^* }_2 \leq \Bignorm{ \mb 1_{n_k}\t \mb E_{\mc I_k(\mb z^*), \cdot}^{(1)} \mb E_{\mc I_k(\mb z^*), \cdot}^{(2)\top}  \mb E^{(2)}_{i_2,\cdot} \mb V^*}_2  + \norm{ \mb 1_{n_k }\t \diag(\mb E_{\mc I_k(\mb z^*), \cdot}^{(1)} \mb E_{\mc I_k(\mb z^*), \cdot}^{(2)\top}) \mb E^{(2)}_{i_2,\cdot} \mb V^*}_2. 
    \end{equation}
    Regarding the first term, 
    we apply Lemma~\ref{lemma: concentration for decoupled quantities} to derive that: 
    \begin{align}
        & \Bignorm{ \mb 1_{n_k}\t \mb E_{\mc I_k(\mb z^*), \cdot}^{(1)} \mb E_{\mc I_k(\mb z^*), \cdot}^{(2)\top}  \mb E^{(2)}_{i_2,\cdot} \mb V^*}_2  \lesssim n_k^{\frac32} \norm{\bo \Sigma_k} \norm{{\mb V^*}\t \bo \Sigma_k \mb V^*_{}}^{\frac12} \sqrt{K\log d} + n_k \norm{\bo \Sigma_k}_F \norm{{\mb V^*}\t \bo \Sigma_k \mb V^*_{}}^{\frac12} \sqrt{K\log d}
    \end{align}
    with probability at least $1- O(d^{-c})$ under either noise assumption. 
    For the second term, one has by Lemma~\ref{lemma: noise matrix concentrations using the universality (Gaussian)} and Lemma~\ref{lemma: noise matrix concentrations using the universality (bounded)} that 
    \begin{equation}
        \norm{ \mb 1_{n_k }\t \diag(\mb E_{\mc I_k(\mb z^*), \cdot}^{(1)} \mb E_{\mc I_k(\mb z^*), \cdot}^{(2)\top}) \mb E^{(2)}_{i_2,\cdot} \mb V^*}_2 \leq \sqrt{n_k} \norm{\diag(\mb E_{\mc I_k(\mb z^*), \cdot}^{(1)} \mb E_{\mc I_k(\mb z^*), \cdot}^{(2)\top})} \norm{\mb E^{(2)}_{i_2,\cdot} \mb V^*} \lesssim n_k \norm{\bo \Sigma_k}_F\norm{{\mb V^*}\t \bo \Sigma_k \mb V^*_{}}^{\frac12} \sqrt{\log d} 
    \end{equation}
    with probability at least $1 - O(d^{-c})$. Taking these inequalities together gives with probability at least $1 - O(d^{-c})$ that 
    \begin{align}
        & \Bignorm{\kappa'_{k_1, k} \sum_{i_1 \neq i_2\in \mc I_k(\mb z^*)} \ip{\mb E_{i_1}, \mb E_{i_2}} \mb E_{i_2,\cdot} \mb V^* }_2 \lesssim n_k \norm{\bo \Sigma_k} \norm{{\mb V^*}\t \bo \Sigma_k \mb V^*_{}}^{\frac12} \sqrt{K\log d} + n_k^{\frac12} \norm{\bo \Sigma_k}_F \norm{{\mb V^*}\t \bo \Sigma_k \mb V^*_{}}^{\frac12} \sqrt{K\log d}. 
    \end{align}

    Moreover, applying Lemma~\ref{lemma: concentration for decoupled quantities} to the second term on the RHS of \eqref{eq: first term in the decomposition of the corr term} yields that 
    \begin{align}
        & \Bignorm{\sum_{k'\in[K] \backslash\{k\}}  \sum_{i_1\in \mc I_{k'}(\mb z^*), i_2 \in \mc I_k(\mb z^*)} \kappa'_{k_1, k'} \ip{\mb E_{i_1}, \mb E_{i_2}}  \mb E_{i_2,\cdot} \mb V^*}_2 \\ 
        \lesssim & \sum_{k' \in [K]\backslash\{k\}}\Big[ n_k \norm{\bo \Sigma_k \bo \Sigma_{k'}}^{\frac12} \bignorm{{\mb V^*}\t \bo \Sigma_k \mb V^*_{}}^{\frac12} \sqrt{K\log d} + n_k^{\frac12} \tr(\bo \Sigma_k \bo \Sigma_{k'})^{\frac12} \norm{{\mb V^*}\t \bo \Sigma_k \mb V^*_{}}^{\frac12} \sqrt{K\log d}\Big]
    \end{align}
     with probability at least $1 - O(d^{-c})$. This together with the last inequality implies that 
    \eq{
          \norm{{\mb U_{\cdot, k_1}^*}\t \mc H(\mb E \mb E\t)_{\cdot, \mc I_k(\mb z^*)} \mb E_{\mc I_k(\mb z^*), \cdot}\mb V^*}_2  \lesssim K n_k \tilde\sigma^2 \bar \sigma \sqrt{\log d} + \sqrt{ n_kp} K\tilde\sigma \sigma \bar \sigma\sqrt{\log d}
    }
    with probability at least $1 - O(d^{-c})$.

      \qed

\subsubsection{Proof of Lemma~\ref{lemma: concentration for decoupled quantities}}  \label{subsubsec: proof of lemma: concentration for decoupled quantities}
    To analyze $\Big|\norm{\mb G^{(1)} {{}\mb G^{(2)}}\t}_F^2 - N_1 N_2 \tr(\bo \Sigma_1 \bo \Sigma_2) \Big|$, we first condition on $\mb E^{(2)}$ and then apply concentration results to $\Bignorm{\mb G^{(1)} {\mb G^{(2)}}\t }_F = \norm{\overrightarrow{\mb G^{(1)}}\t \mb M_{\mb G^{(2)}}}_2$ by vectorizing $\mb G^{(1)}$, where 
            $\text{vec}\big(\mb G^{(1)}\big)$ denotes the column vector formed from stacking the rows of $\mb G^{(1)}$, and 
            $ \mb M_{\mb G^{(2)}} \coloneqq \left(
            \begin{matrix}
                ( {\mb G^{(2)}}\t & & \\ 
                & \ddots & \\ 
                & &  {\mb G^{(2)}}\t
            \end{matrix}
            \right) \in \bb R^{ N_2 p \times N_2^2} 
            $. Since $\mb G^{(1)}$ and $\mb G^{(2)}$ play symmetric roles, without loss of generality we assume that $N_1 \geq N_2$. 

            For the Gaussian noise case, the Hanson-Wright inequality (Theorem~6.2.1 in \cite{vershynin2018high}) yields, with probability at least $1- O(d^{-c})$ conditional on an arbitrary $\mb G^{(2)}$, that 
            \begin{align}
                & \bigg|\norm{\overrightarrow{\mb G^{(1)}}\t \mb M_{\mb G^{(2)}}}_2^2 - \bb E\Big[ \bignorm{\overrightarrow{\mb G^{(1)}}\t \mb M_{\mb G^{(2)}}}_2^2  \big| \mb G^{(2)}\Big] \bigg|\\ 
                = &\Big|\bignorm{\overrightarrow{\mb G^{(1)}}\t \mb M_{\mb G^{(2)}}}_2^2 - N_1\tr(\bo \Sigma_1^{\frac12}{\mb G^{(2)}}\t \mb G^{(2)} \bo \Sigma_1^{\frac12}) \Big|\\ 
                \lesssim & \Big[N_1 \tr\big(\bo \Sigma_1^{\frac12}{\mb G^{(2)}}\t \mb G^{(2)} \bo \Sigma_1 {\mb G^{(2)}}\t \mb G^{(2)} \bo \Sigma_1^{\frac12}\big) \Big]^{\frac{1}{2}} \sqrt{\log(p\vee N)}  + \norm{\bo \Sigma_1^{\frac12}{\mb G^{(2)}}\t \mb G^{(2)} \bo \Sigma_1^{\frac12}} \log(p\vee N) \\ 
                \lesssim & \bigg[N_1  \tr\big(\bo \Sigma_1^{\frac12}{\mb G^{(2)}}\t \mb G^{(2)} \bo \Sigma_1^{\frac12} \big) \Bignorm{\bo \Sigma_1^{\frac12}{\mb G^{(2)}}\t \mb G^{(2)} \bo \Sigma_1^{\frac12}}\bigg]^{\frac12} \sqrt{\log(p\vee N)}  +  \Bignorm{\bo \Sigma_1^{\frac12}{\mb G^{(2)}}\t \mb G^{(2)} \bo \Sigma_1^{\frac12}} \log(p\vee N). 
                \label{eq: hanson-wright for A1}
            \end{align}
            Thus we obtain 
            \begin{align}
                &\Big| \bignorm{\overrightarrow{\mb G^{(1)}}\t \mb M_{\mb G^{(2)}}}_2^2 - N_1 N_2 \tr(\bo \Sigma_1 \bo \Sigma_2)\Big| \leq \Big| \bignorm{\overrightarrow{\mb G^{(1)}}\t \mb M_{\mb G^{(2)}}}_2^2 - \bb E\Big[ \bignorm{\overrightarrow{\mb G^{(1)}}\t \mb M_{\mb G^{(2)}}}_2^2  \big| \mb G^{(2)}\Big] \Big|\\ 
                & \qquad \qquad + \Big| N_1 \tr\Big(\bo \Sigma_1^{\frac12} \big({\mb G^{(2)}}\t \mb G^{(2)} - N_2 \bo \Sigma_1 \big) \bo \Sigma_1^{\frac12}\Big)\Big|
                \\ 
                & \qquad \leq    \bigg[N_1  \tr\big(\bo \Sigma_1^{\frac12}{\mb G^{(2)}}\t \mb G^{(2)} \bo \Sigma_1^{\frac12} \big) \Bignorm{\bo \Sigma_1^{\frac12}{\mb G^{(2)}}\t \mb G^{(2)} \bo \Sigma_1^{\frac12}}\bigg]^{\frac12} \sqrt{\log(p\vee N)}  \\ 
                & \qquad \qquad +  \Bignorm{\bo \Sigma_1^{\frac12}{\mb G^{(2)}}\t \mb G^{(2)} \bo \Sigma_1^{\frac12}} \log(p\vee N)  + \Big| N_1 \tr\Big(\bo \Sigma_1^{\frac12} \big({\mb G^{(2)}}\t \mb G^{(2)} - N_2 \bo \Sigma_2 \big) \bo \Sigma_1^{\frac12}\Big)\Big|
                \label{eq: decomposition of A1}
            \end{align}
            with probability at least $1 - O(d^{-c})$ conditional on an arbitrary $\mb G^{(2)}$. 

            To parse the quantities related to $\mb G^{(2)}$ on the RHS of \eqref{eq: decomposition of A1}, we turn to bounding the term 
            $$
            \tr\Big(\bo \Sigma_1^{\frac12} \big({\mb G^{(2)}}\t \mb G^{(2)} - N_2\bo \Sigma_2 \big) \bo \Sigma_1^{\frac12}\Big)= \bignorm{\text{vec}\big(\mb G^{(2)} \big)\t \mb M_{\bo \Sigma_1}}_2^2 - \bb E\Big[\bignorm{\text{vec}\big(\mb G^{(2)} \big)\t \mb M_{\bo \Sigma_1}}_2^2 \Big],
            $$
            where $\text{vec}\big(\mb G^{(2)} \big)$ is similarly defined as the column vector from stacking the rows of $\mb G^{(2)}$ and $\mb M_{\bo \Sigma_1}$ is defined as
            $
             \mb M_{\bo \Sigma_1} \coloneqq \left(
            \begin{matrix}
                 \bo \Sigma_1^{\frac12} & & \\ 
                & \ddots & \\ 
                & & \bo\Sigma_1^{\frac12} 
            \end{matrix}
            \right)\in \bb R^{N_1 p \times N_1 p}. 
            $
            Again, the Hansan-Wright inequality (Theorem~6.2.1 in \cite{vershynin2018high}) implies that 
            \begin{align}
                & \Big|\tr\big(\bo \Sigma_1^{\frac12} \big({\mb G^{(2)}}\t \mb G^{(2)} - N_2\bo \Sigma_2 \big) \bo \Sigma_1^{\frac12}\big) \Big| \lesssim \sqrt{N_2 \tr(\bo \Sigma_1^2 \bo \Sigma_2^2) \log(p\vee N)} + \norm{\bo \Sigma_1\bo \Sigma_2} \log(p\vee N)                 \label{eq: hanson-wright for E(2)}
            \end{align}
            with probability exceeding $1 - O(d^{-c})$, which leads to 
            \begin{align}
                & \tr\big(\bo \Sigma_1^{\frac12}{\mb G^{(2)}}\t \mb G^{(2)} \bo \Sigma_1^{\frac12} \big) \lesssim  N_2\tr(\bo \Sigma_1 \bo \Sigma_2)
                \label{eq: upper bound on trace of gram matrix wrt G2}
            \end{align}
            with probability at least $1 - O(d^{-c})$, where in $(i)$ and $(ii)$ we used the facts that $\sqrt{N_2 \tr(\bo \Sigma_1^2 \bo \Sigma_2^2) \log(p\vee N)} \leq \sqrt{N_2} \tr(\bo \Sigma_1 \bo \Sigma_2) \sqrt{\log(p\vee N)} \ll  N_2 \tr(\bo \Sigma_1 \bo \Sigma_2)$ and $\norm{\bo \Sigma_1\bo \Sigma_2} \log(p\vee N)\ll N_2 \tr(\bo \Sigma_1 \bo \Sigma_2)$. Besides, the concentration for a Gaussian distributed matrix gives that 
            \begin{align}
                & \Bignorm{\bo \Sigma_1^{\frac12}{\mb G^{(2)}}\t \mb G^{(2)} \bo \Sigma_1^{\frac12}} = \norm{\mb G^{(2)} \bo \Sigma_1^{\frac12}}^2 \lesssim \norm{\bo \Sigma_1 \bo \Sigma_2}N_2 + \tr(\bo \Sigma_1 \bo \Sigma_2)
                \label{eq: upper bound on spectral norm of gram matrix wrt G2}
            \end{align}
            with probability at least $1- O(d^{-c})$. 

            Plugging \eqref{eq: upper bound on trace of gram matrix wrt G2} and \eqref{eq: upper bound on spectral norm of gram matrix wrt G2} into \eqref{eq: decomposition of A1}, one reaches the conclusion that
            \longeq{
                & \Big| \bignorm{\overrightarrow{\mb G^{(1)}}\t \mb M_{\mb G^{(2)}}}_2^2 - N_1 N_2  \tr(\bo \Sigma_k^2 )\Big| 
                \lesssim \sqrt{N_1N_2} \tr(\bo \Sigma_1 \bo \Sigma_2)^{\frac12} \Big(\norm{\bo \Sigma_1\bo \Sigma_2}^{\frac12} \sqrt{N_2} + \tr(\bo \Sigma_1 \bo \Sigma_2)^{\frac12} \Big) \sqrt{\log(p\vee N)}\\ 
                & \qquad + \big(\norm{\bo \Sigma_1 \bo \Sigma_2}N_2 + \tr(\bo \Sigma_1 \bo \Sigma_2)\big) \log(p\vee N)+ N_1 \Big(  \sqrt{N_2 \tr(\bo \Sigma_1^2 \bo \Sigma_2^2) \log(p\vee N)} + \norm{\bo \Sigma_1\bo \Sigma_2} \log(p\vee N)\Big) \\ 
                \lesssim & (N_1 \sqrt{N_2} +  \sqrt{N_1}N_2 )\tr(\bo \Sigma_1 \bo \Sigma_2)^{\frac12}\norm{\bo \Sigma_1\bo \Sigma_2}^{\frac12} \sqrt{\log(p\vee N)} +  \sqrt{N_1 N_2}\tr(\bo \Sigma_1 \bo \Sigma_2) \sqrt{\log(p\vee N)} 
                \label{eq: hanson-wright for G2}
            }
            with probability at least $1 - O(d^{-c})$, where we used the facts 
            $
            \log(p\vee N)\lesssim  \sqrt{N_1 N_2} \sqrt{\log(p\vee N)}
            $
            and 
            \longeq{
            &N_1 \vee N_2\norm{\bo \Sigma_1\bo \Sigma_2}\log(p\vee N) \lesssim \sqrt{N_1N_2(N_1 \vee N_2)}\norm{\bo \Sigma_1\bo \Sigma_2}\sqrt{\log(p\vee N)} \\ 
            \lesssim& \sqrt{N_1 N_2(N_1 \vee N_2)} \tr(\bo \Sigma_1 \bo \Sigma_2)^{\frac12}\norm{\bo \Sigma_1\bo \Sigma_2}^{\frac12} \sqrt{\log(p\vee N)}
            } since $\log(p\vee N) \ll N_1 \wedge N_2$. 

            Now we turn to the bounded noise matrices $\mb E^{(1)}$ and $\mb E^{(2)}$ with local dependence.  By combining the Gaussian bound in \eqref{eq: hanson-wright for A1} with Theorem 2.14 of \cite{brailovskaya2022universality}, we can transfer control from the Gaussian one to the bounded setting:
            \begin{align}
                &  \Big|\bignorm{\overrightarrow{\mb E^{(1)}}\t \mb M_{\mb E^{(2)}}}_2 - \bb E\Big[ \bignorm{\overrightarrow{\mb E^{(1)}}\t \mb M_{\mb E^{(2)}} \big| \mb E^{(2)}}_2^2 \Big]^{\frac12} \Big|   
                \\ 
                \leq &   \Big|\bignorm{\overrightarrow{\mb G^{(1)}}\t \mb M_{\mb E^{(2)}}}_2 - \bb E\Big[ \bignorm{\overrightarrow{\mb G^{(1)}}\t \mb M_{\mb E^{(2)}} \big| \mb E^{(2)}}_2^2 \Big]^{\frac12} \Big|  + \epsilon_{\mb E^{(1)}}
                \\ 
                \leq &  \bb E\Big[ \bignorm{\overrightarrow{\mb G^{(1)}}\t \mb M_{\mb E^{(2)}} \big| \mb E^{(2)}}_2^2 \Big]^{-\frac12} \Big|\bignorm{\overrightarrow{\mb G^{(1)}}\t \mb M_{\mb E^{(2)}}}_2^2 - \bb E\Big[ \bignorm{\overrightarrow{\mb G^{(1)}}\t \mb M_{\mb E^{(2)}} \big| \mb E^{(2)}}_2^2 \Big] \Big| +  \epsilon_{\mb E^{(1)}}
                \\ 
                \leq & N_1^{-\frac12} \tr\big(\bo \Sigma_1^{\frac12}\mb E^{(2)\top}\mb E^{(2)} \bo \Sigma_1^{\frac12}\big)^{-\frac12} \bigg\{ \Big[N_1  \tr\big(\bo \Sigma_1^{\frac12}{\mb E^{(2)}}\t \mb E^{(2)} \bo \Sigma_1^{\frac12} \big) \Bignorm{\bo \Sigma_1^{\frac12}{\mb E^{(2)}}\t \mb E^{(2)} \bo \Sigma_1^{\frac12}}\bigg]^{\frac12} \sqrt{\log(p\vee N)}  \\ 
                & \qquad +  \Bignorm{\bo \Sigma_1^{\frac12}{\mb E^{(2)}}\t \mb E^{(2)} \bo \Sigma_1^{\frac12}} \log(p\vee N) \bigg\}  + \epsilon_{\mb E^{(1)}} 
                \label{eq: concentration for decoupled quadratic form (bounded case)}
            \end{align}
            with probability at least $1- O(d^{-c})$ conditional on $\mb E^{(2)}$, 
            where we used the identity $|a^2 - b^2| = (a + b) |a -b|$ for $a,~ b \geq 0$. The universality error $\epsilon_{\mb E^{(1)}}$ is defined, following the notions of \cite{brailovskaya2022universality}, with some constant $C_{\text{univ}}$:  
            \longeq{
                &\epsilon_{\mb E^{(1)}} \coloneqq  C_{\text{univ}} \Big[v(\overrightarrow{\mb E^{(1)}}\t \mb M_{\mb E^{(2)}})^{\frac12} \sigma(\overrightarrow{\mb E^{(1)}}\t \mb M_{\mb E^{(2)}})^{\frac12} (\log(p\vee N))^{\frac34} + \sigma_*(\overrightarrow{\mb E^{(1)}}\t \mb M_{\mb E^{(2)}}) \sqrt{\log(p\vee N)} \\ 
                & \qquad + R(\overrightarrow{\mb E^{(1)}}\t \mb M_{\mb E^{(2)}})^{\frac13} \sigma(\overrightarrow{\mb E^{(1)}}\t \mb M_{\mb E^{(2)}})^{\frac23}(\log(p\vee N))^{\frac23} + R(\overrightarrow{\mb E^{(1)}}\t \mb M_{\mb E^{(2)}}) \log(p\vee N)\Big]. 
            }
            The relevant quantities satisfy the following bounds:
            \begin{align}
                & v(\overrightarrow{\mb E^{(1)}}\t \mb M_{\mb E^{(2)}}) \vee \sigma_*(\overrightarrow{\mb E^{(1)}}\t \mb M_{\mb E^{(2)}}) \leq \bignorm{ \mb E^{(2)} \bo \Sigma_1^{\frac12}}\lesssim  \norm{\bo \Sigma_1 \bo \Sigma_2}^{\frac12} \sqrt{N_2} + \tr(\bo \Sigma_1 \bo \Sigma_2)^{\frac12} , \\ 
                & \sigma(\overrightarrow{\mb E^{(1)}}\t \mb M_{\mb E^{(2)}}) \lesssim \sqrt{N_1}\tr\big(\bo \Sigma_1^{\frac12}\mb E^{(2)\top}\mb E^{(2)} \bo \Sigma_1^{\frac12}\big)^{\frac12}, \qquad R(\overrightarrow{\mb E^{(1)}}\t \mb M_{\mb E^{(2)}}) \lesssim \sqrt{m}B \norm{\bo \Sigma_2}^{\frac12} \sqrt{N_2}, 
            \end{align}
            conditional on the event $\left\{\max_{c\in[l]} \norm{\mb E^{(2)}_{\cdot, S_c}} \leq C \norm{\bo \Sigma_2}^{\frac12} \sqrt{N_2},~\norm{\mb E^{(2)}} \lesssim \norm{\bo \Sigma_1 \bo \Sigma_2}^{\frac12} \sqrt{N_2} + \tr(\bo \Sigma_1 \bo \Sigma_2)^{\frac12}\right\}$ for some constant $C$, whose probabilty exceeds $1 - O(d^{-c})$. This gives the upper bound for $\epsilon_{\mb E^{(1)}}$: 
            \begin{align}
                & \epsilon_{\mb E^{(1)}} \lesssim  \big(\norm{\bo \Sigma_1 \bo \Sigma_2}^{\frac14}N_2^{\frac14} + \tr(\bo \Sigma_1 \bo \Sigma_2)^{\frac14}\big) \cdot N_1^{\frac14} \tr\big(\bo \Sigma_k^{\frac12}\mb E^{(2)\top}\mb E^{(2)} \bo \Sigma_k^{\frac12}\big)^{\frac14}(\log(p\vee N))^{\frac34} \\ 
                & + \big( \norm{\bo \Sigma_1 \bo \Sigma_2}^{\frac12} \sqrt{N_2} + \tr(\bo \Sigma_1 \bo \Sigma_2)^{\frac12}\big)  \sqrt{\log(p\vee N)}  \\ 
                & + m^{\frac16} B^{\frac13} \norm{\bo \Sigma_2}^{\frac16} N_2^{\frac16}\cdot  N_1^{\frac13}\tr\big(\bo \Sigma_1^{\frac12}\mb E^{(2)\top}\mb E^{(2)} \bo \Sigma_1^{\frac12}\big)^{\frac13}(\log(p\vee N))^{\frac23} + \sqrt{m}B \norm{\bo \Sigma_2}^{\frac12} \sqrt{N_2}\log(p\vee N).  
                \label{eq: upper bound on epsilonE1}
            \end{align}
            This calls for a control on $\tr\big(\bo \Sigma_1^{\frac12}\mb E^{(2)\top}\mb E^{(2)} \bo \Sigma_1^{\frac12}\big) = \bignorm{\overrightarrow{\mb E^{(2)}}\t \mb M_{\bo \Sigma_1}}_2^2$, which can be addressed by applying the universality again. Specifically, one combine \eqref{eq: hanson-wright for E(2)} with \cite[Theorem~2.14]{brailovskaya2022universality} to obtain that 
            \begin{align}
            & \Big| \tr\big(\bo \Sigma_1^{\frac12}\mb E^{(2)\top}\mb E^{(2)} \bo \Sigma_1^{\frac12}\big)^{\frac12} - \sqrt{N_2} \tr(\bo \Sigma_1 \bo \Sigma_2)^{\frac12}\Big| = \Big|\bignorm{\overrightarrow{\mb E^{(2)}}\t \mb M_{\bo \Sigma_1}}_2 - \sqrt{N_2} \tr(\bo \Sigma_1 \bo \Sigma_2)^{\frac12} \Big| \\ 
            \lesssim &  \big(\sqrt{N_2} \tr(\bo \Sigma_1 \bo \Sigma_2 )^{\frac12}\big)^{-1}\Big[\sqrt{N_2 \tr(\bo \Sigma_1^2 \bo \Sigma_2^2 ) \log(p\vee N)} + \norm{\bo \Sigma_1 \bo \Sigma_2}^2 \log(p\vee N)\Big] \\ 
            & \qquad  + \Big[\norm{\bo \Sigma_1 \bo \Sigma_2}^{\frac14} N_2^{\frac14} \tr(\bo \Sigma_1 \bo \Sigma_2)^{\frac14} (\log(p\vee N))^{\frac34} + \norm{\bo \Sigma_1 \bo \Sigma_2}^{\frac12}\sqrt{\log(p\vee N)}\\
            & \qquad + m^{\frac16}B^{\frac13} \norm{\bo \Sigma_1}^{\frac16} N_2^{\frac13} \tr(\bo \Sigma_1 \bo \Sigma_2)^{\frac13} (\log(p\vee N))^{\frac23} + \sqrt{m}B \norm{\bo \Sigma_1}^{\frac12} \log(p\vee N)\Big]  \\ 
            \lesssim & \sqrt{m}B\norm{\bo \Sigma_2}^{\frac12} \log(p\vee N) + m^{\frac16} N_2^{\frac13} B^{\frac13} \norm{\bo \Sigma_1}^{\frac16} \tr(\bo \Sigma_1 \bo \Sigma_2)^{\frac13} (\log(p\vee N))^{\frac34} \ll \sqrt{N_2} \tr(\bo \Sigma_1 \bo \Sigma_2)^{\frac12}
            \label{eq: universality for E(2)}
            \end{align}
            with probability at least $1- O(d^{-c})$, where the involved quantities are controlled by: $v(\overrightarrow{\mb E^{(2)}}\t \mb M_{\bo \Sigma_1}) \vee \sigma_*(\overrightarrow{\mb E^{(2)}}\t \mb M_{\bo \Sigma_1})\lesssim \norm{\bo \Sigma_1 \bo \Sigma_2}^{\frac12}$, $\sigma(\overrightarrow{\mb E^{(2)}}\t \mb M_{\bo \Sigma_1}) \lesssim \sqrt{N_2}\tr(\bo \Sigma_1 \bo \Sigma_2)^{\frac12}$, and $R(\overrightarrow{\mb E^{(2)}}\t \mb M_{\bo \Sigma_1}) \lesssim \sqrt{m}B \norm{\bo \Sigma_1}^{\frac12}$. Here we simplify the expression in the penultimate inequaltiy of \eqref{eq: universality for E(2)} using the fact that $
                 \sqrt{m}B \geq \norm{\bo \Sigma_2}^{\frac12}  \geq \frac{\norm{\bo \Sigma_1 \bo \Sigma_2}^{\frac12}}{\norm{\bo \Sigma_1}^{\frac12}}$ and the last inequality follows from the condition that $\sqrt{m}B \norm{\bo \Sigma_2}^{\frac12} ( \log(p\vee N))^{\frac94} \lesssim \tr(\bo \Sigma_1 \bo \Sigma_2)^{\frac12}$.

            Taking \eqref{eq: universality for E(2)} and \eqref{eq: upper bound on epsilonE1} collectively  yields that 
            \begin{align}
                \epsilon_{\mb E^{(1)}} \lesssim & \big(\norm{\bo \Sigma_1 \bo \Sigma_2}^{\frac14}N_2^{\frac14} + \tr(\bo \Sigma_1 \bo \Sigma_2)^{\frac14}\big) \cdot N_1^{\frac14} N_2^{\frac14}\tr\big(\bo \Sigma_1\bo \Sigma_2\big)^{\frac14}(\log(p\vee N))^{\frac34} \\ 
                & + \big( \norm{\bo \Sigma_1 \bo \Sigma_2}^{\frac12} \sqrt{N_2} + \tr(\bo \Sigma_1 \bo \Sigma_2)^{\frac12}\big)  \sqrt{\log(p\vee N)}  \\ 
                & + m^{\frac16} B^{\frac13} \norm{\bo \Sigma_2}^{\frac16} N_2^{\frac16}\cdot  N_1^{\frac13} N_2^{\frac13} \tr\big(\bo \Sigma_1 \bo \Sigma_2\big)^{\frac13}(\log(p\vee N))^{\frac23} + \sqrt{m}B \norm{\bo \Sigma_2}^{\frac12} \sqrt{N_2}\log(p\vee N) \\ 
                \lesssim & \big(\norm{\bo \Sigma_1 \bo \Sigma_2}^{\frac14}N_2^{\frac14} + \tr(\bo \Sigma_1 \bo \Sigma_2)^{\frac14}\big) \cdot N_1^{\frac14} N_2^{\frac14}\tr\big(\bo \Sigma_1\bo \Sigma_2\big)^{\frac14}(\log(p\vee N))^{\frac34} \\ 
                & + m^{\frac16} B^{\frac13} \norm{\bo \Sigma_2}^{\frac16}   N_1^{\frac13} N_2^{\frac12} \tr\big(\bo \Sigma_1 \bo \Sigma_2\big)^{\frac13}(\log(p\vee N))^{\frac23} + \sqrt{m}B \norm{\bo \Sigma_2}^{\frac12} \sqrt{N_2}\log(p\vee N)
                \label{eq: upper bound for epsilon E(1) (bounded case)}
            \end{align}
            holds with probability at least $1- O(d^{-c})$. 

            Finally, invoking \eqref{eq: concentration for decoupled quadratic form (bounded case)} and \eqref{eq: universality for E(2)} yields with probability at least $1- O(d^{-c})$ that 
            \begin{align}
                &  \Big|\bignorm{\overrightarrow{\mb E^{(1)}}\t \mb M_{\mb E^{(2)}}}_2 - \bb E\Big[ \bignorm{\overrightarrow{\mb E^{(1)}}\t \mb M_{\mb E^{(2)}} \big| \mb E^{(2)}}_2^2 \Big]^{\frac12} \Big| \\  
                \lesssim & \big(N_1 N_2 \tr(\bo \Sigma_1 \bo \Sigma_2)\big)^{-\frac12} \Big[ 
                \big(N_1 N_2 \tr(\bo \Sigma_1 \bo \Sigma_2)\big)^{\frac12}\big( \norm{\bo \Sigma_1 \bo \Sigma_2}^{\frac12} \sqrt{N_2} + \tr(\bo \Sigma_1 \bo \Sigma_2)^{\frac12}\big) \sqrt{\log (p\vee N)} 
                \\ 
                & \qquad + \big( \norm{\bo \Sigma_1 \bo \Sigma_2} N_2 + \tr(\bo \Sigma_1 \bo \Sigma_2)\big)\log (p\vee N) 
                \Big] \\ 
                & + \Big[\big(\norm{\bo \Sigma_1 \bo \Sigma_2}^{\frac14}N_2^{\frac14} + \tr(\bo \Sigma_1 \bo \Sigma_2)^{\frac14}\big) \cdot N_1^{\frac14} N_2^{\frac14}\tr\big(\bo \Sigma_1\bo \Sigma_2\big)^{\frac14}(\log(p\vee N))^{\frac34} \\ 
                & \qquad + m^{\frac16} B^{\frac13} \norm{\bo \Sigma_2}^{\frac16}   N_1^{\frac13} N_2^{\frac12} \tr\big(\bo \Sigma_1 \bo \Sigma_2\big)^{\frac13}(\log(p\vee N))^{\frac23} + \sqrt{m}B \norm{\bo \Sigma_2}^{\frac12} \sqrt{N_2}\log(p\vee N) \Big] \\ 
                \lesssim  & N_1^{\frac12} N_2^{\frac12} \norm{\bo \Sigma_1 \bo \Sigma_2}^{\frac16}\tr(\bo \Sigma_1 \bo \Sigma_2)^{\frac13} \log(p\vee N)+ N_1^{\frac14} N_2^{\frac14} \tr(\bo\Sigma_1 \bo \Sigma_2)^{\frac12} \log(p\vee N) \\ 
                \ll & \sqrt{N_1 N_2} \tr(\bo \Sigma_1 \bo \Sigma_2)^{\frac12} ,
            \end{align}
            where the second-to-last inequality follows from the facts that 
            \begin{align}
                & \big(N_1 N_2 \tr(\bo \Sigma_1 \bo \Sigma_2)\big)^{-\frac12} \norm{\bo \Sigma_1\bo \Sigma_2} N_2 \leq \sqrt{\frac{N_2}{N_1}} \norm{\bo \Sigma_1 \bo \Sigma_2}^{\frac12} \lesssim N_1^{\frac12} N_2^{\frac12} \norm{\bo \Sigma_1 \bo \Sigma_2}^{\frac16}\tr(\bo \Sigma_1 \bo \Sigma_2)^{\frac13} \log(p\vee N), \\ 
                & \big(N_1 N_2 \tr(\bo \Sigma_1 \bo \Sigma_2)\big)^{-\frac12} \tr(\bo \Sigma_1 \bo \Sigma_2) \log(p\vee N) \lesssim N_1^{\frac14} N_2^{\frac14} \tr(\bo\Sigma_1 \bo \Sigma_2)^{\frac12} \log(p\vee N), \\ 
                & m^{\frac16} B^{\frac13} \norm{\bo \Sigma_2}^{\frac16}   N_1^{\frac13} N_2^{\frac12} \tr\big(\bo \Sigma_1 \bo \Sigma_2\big)^{\frac13}(\log(p\vee N))^{\frac23} \lesssim N_1^{\frac12} N_2^{\frac12} \norm{\bo \Sigma_1\bo \Sigma_2}^{\frac16} \tr(\bo \Sigma_1 \bo \Sigma_2)^{\frac13} \log(p\vee N), 
                \label{eq: fact 1 regarding B (auxiliary lemma)}
                \\ 
                & \sqrt{m}B \norm{\bo \Sigma_2}^{\frac12} \sqrt{N_2}\log(p\vee N) \lesssim  (\sqrt{m}B \norm{\bo \Sigma_2}^{\frac12} )^{\frac16} (\sqrt{m}B \norm{\bo \Sigma_2}^{\frac12} )^{\frac13}\sqrt{N_2}\log(p\vee N)\\
                \lesssim & N_1^{\frac{1}{12}} N_2^{\frac12} \norm{\bo \Sigma_1 \bo \Sigma_2}^{\frac16}\tr(\bo \Sigma_1 \bo \Sigma_2)^{\frac13} \log(p\vee N), 
                \label{eq: fact 2 regarding B (auxiliary lemma)}
            \end{align}
            since 
            $ \sqrt{m} B \big(\frac{\norm{\bo \Sigma_2}^{\frac12}}{\sqrt{N_1 \wedge N_2}\norm{\bo \Sigma_1 \bo \Sigma_2}^{\frac12}}  \vee \frac{\norm{\bo \Sigma_2}^{\frac12}}{\tr(\bo \Sigma_1 \bo \Sigma_2)^{\frac12}}\big) \lesssim 1$,
            and the last inequality follows from the effective-rank condition $\mathrm{rk}(\bo \Sigma_1 \bo \Sigma_2) \coloneqq \frac{\tr(\bo \Sigma_1\bo \Sigma_2)}{\norm{\bo \Sigma_1 \bo \Sigma_2}} \gg (\log(p\vee N))^6$ and $N_1 \wedge N_2 \gg (\log (p\vee N))^2$. This combined with \eqref{eq: universality for E(2)} in turn leads to 
            \begin{equation}
                 \Big|\bignorm{\overrightarrow{\mb E^{(1)}}\t \mb M_{\mb E^{(2)}}}_2^2 - \bb E\Big[ \bignorm{\overrightarrow{\mb E^{(1)}}\t \mb M_{\mb E^{(2)}} \big| \mb E^{(2)}}_2^2 \Big]\Big| \lesssim N_1 N_2 \norm{\bo \Sigma_1 \bo \Sigma_2}^{\frac16}\tr(\bo \Sigma_1 \bo \Sigma_2)^{\frac56} \log(p\vee N)+ N_1^{\frac34} N_2^{\frac34} \tr(\bo\Sigma_1 \bo \Sigma_2) \log(p\vee N) 
                 \label{eq: eq: upper bound for hanson wright ineq regarding E1 (bounded case)}
            \end{equation}
            with probability at least $1 - O(d^{-c})$. 
            
            To conclude, combining \eqref{eq: universality for E(2)} with \eqref{eq: eq: upper bound for hanson wright ineq regarding E1 (bounded case)} gives that
            \begin{align}
                & \Big| \bignorm{\overrightarrow{\mb E^{(1)}}\t \mb M_{\mb E^{(2)}}}_2^2 - N_1 N_2  \tr(\bo \Sigma_1 \bo \Sigma_2 )\Big| \\ 
                \leq & \Big| \bignorm{\overrightarrow{\mb E^{(1)}}\t \mb M_{\mb E^{(2)}}}_2^2 - \bb E\Big[ \bignorm{\overrightarrow{\mb E^{(1)}}\t \mb M_{\mb E^{(2)}} \big| \mb E^{(2)}}_2^2 \Big]\Big| + N_1\Big|\tr\big(\bo \Sigma_k^{\frac12} \big({\mb E^{(2)}}\t \mb E^{(2)} - N_2\bo \Sigma_2 \big) \bo \Sigma_k^{\frac12}\big) \Big| \\ 
                \lesssim &N_1 N_2 \norm{\bo \Sigma_1 \bo \Sigma_2}^{\frac16}\tr(\bo \Sigma_1 \bo \Sigma_2)^{\frac56} \log(p\vee N)+ N_1^{\frac34} N_2^{\frac34} \tr(\bo\Sigma_1 \bo \Sigma_2) \log(p\vee N)  \\ 
                & + N_1 \sqrt{m}B\norm{\bo \Sigma_2}^{\frac12} \log(p\vee N) + m^{\frac16} N_2^{\frac13} B^{\frac13} \norm{\bo \Sigma_1}^{\frac16} \tr(\bo \Sigma_1 \bo \Sigma_2)^{\frac13} (\log(p\vee N))^{\frac34} \\ 
                \lesssim & N_1 N_2 \norm{\bo \Sigma_1 \bo \Sigma_2}^{\frac16}\tr(\bo \Sigma_1 \bo \Sigma_2)^{\frac56} \log(p\vee N)+ N_1^{\frac34} N_2^{\frac34} \tr(\bo\Sigma_1 \bo \Sigma_2) \log(p\vee N) 
            \end{align}
            with probability at least $1- O(d^{-c})$, where in the last inequality we reason analogously as in \eqref{eq: fact 1 regarding B (auxiliary lemma)} and \eqref{eq: fact 2 regarding B (auxiliary lemma)} for upper bounding the terms related to $B$.   

        \bigskip

        Now we move on to the second part corresponding to $\Big| \tr\big(\mb 1_{N_3\times N_1}\mb G^{(1)} {{}\mb G^{(2)}}\t \mb G^{(2)} {{}\mb G^{(3)}}\t\big) \Big|$ and $\Big| \tr\big(\mb 1_{N_3\times N_1}\mb E^{(1)} {{}\mb E^{(2)}}\t \mb E^{(2)} {{}\mb E^{(3)}}\t\big) \Big|$. We first embark on the Gaussian version. Notice that 
        $$\Big| \tr\big(\mb 1_{N_3\times N_1}\mb G^{(1)} {{}\mb G^{(2)}}\t \mb G^{(2)} {{}\mb G^{(3)}}\t\big) \Big| = \sqrt{N_1 N_3 }\Big|(\mb 1_{N_1} / \sqrt{N_1})\t \mb G^{(1)} {{}\mb G^{(2)}}\t \mb G^{(2)} {{}\mb G^{(3)}}\t (\mb 1_{N_3} / \sqrt{N_3}) \Big|$$ 
        is identically distributed as $\sqrt{N_1 N_3}\Big|\mb G_{1,\cdot}^{(1)} {{}\mb G^{(2)}}\t \mb G^{(2)} {{}\mb G_{1,\cdot}^{(3)}}\t \Big|$. Then conditional on $\mb G^{(2)}$, Lemma~6.2.2 in \cite{vershynin2018high} yields that 
        \begin{align}
            & \Big|\mb G_{1,\cdot}^{(1)} {{}\mb G^{(2)}}\t \mb G^{(2)} {{}\mb G_{1,\cdot}^{(3)}}\t \Big|  \lesssim \norm{\bo \Sigma_1^{\frac12} \mb G^{(2)\top} \mb G^{(2)} \bo \Sigma_3^{\frac12}}_F \sqrt{\log (p\vee N)} + \norm{\bo \Sigma_1^{\frac12} \mb G^{(2)\top} \mb G^{(2)} \bo \Sigma_3^{\frac12}} \log(p\vee N)
        \end{align}
        with probability at least $1 - O(d^{-c})$. Given the relations $\norm{\bo \Sigma_1^{\frac12} \mb G^{(2)\top} \mb G^{(2)} \bo \Sigma_3^{\frac12}}_F \leq \norm{\bo \Sigma_1^{\frac12} \mb G^{(2)\top}}_F \norm{\mb G^{(2)} \bo \Sigma_3^{\frac12}} $ and $\norm{\bo \Sigma_1^{\frac12} \mb G^{(2)\top} \mb G^{(2)} \bo \Sigma_3^{\frac12}} \leq \norm{\bo \Sigma_1^{\frac12} \mb G^{(2)\top}} \norm{\mb G^{(2)} \bo \Sigma_3^{\frac12}}$, we invoke \eqref{eq: upper bound on spectral norm of gram matrix wrt G2} and \eqref{eq: hanson-wright for G2} to obtain that 
        \begin{align}
            &  \Big|\mb G_{1,\cdot}^{(1)} {{}\mb G^{(2)}}\t \mb G^{(2)} {{}\mb G_{1,\cdot}^{(3)}}\t \Big| \lesssim \sqrt{N_2}\tr(\bo \Sigma_1 \bo \Sigma_2)^{\frac12}\big( \norm{\bo \Sigma_3 \bo \Sigma_2}^{\frac12} \sqrt{N_2 } + \tr(\bo \Sigma_2 \bo \Sigma_3)^{\frac12} \big)\sqrt{\log d} \\ 
            & \qquad +  \big( \norm{\bo \Sigma_1 \bo \Sigma_2}^{\frac12} \sqrt{N_2 } + \tr(\bo \Sigma_1 \bo \Sigma_2)^{\frac12} \big)\big( \norm{\bo \Sigma_2 \bo \Sigma_3}^{\frac12} \sqrt{N_2 } + \tr(\bo \Sigma_2 \bo \Sigma_3)^{\frac12} \big) \log d \\ 
            \lesssim & \sqrt{N_2}\tr(\bo \Sigma_1 \bo \Sigma_2)^{\frac12}\big( \norm{\bo \Sigma_3 \bo \Sigma_2}^{\frac12} \sqrt{N_2 } + \tr(\bo \Sigma_2 \bo \Sigma_3)^{\frac12} \big)\sqrt{\log d} 
        \end{align}
        holds with probability at least $1 -O(d^{-c})$. This leads to the conclusion that 
        \begin{align}
            & \Big| \tr\big(\mb 1_{N_3\times N_1}\mb G^{(1)} {{}\mb G^{(2)}}\t \mb G^{(2)} {{}\mb G^{(3)}}\t\big) \Big|  \\
            \lesssim &  \sqrt{N_1 N_3} \Big(N_2\tr(\bo \Sigma_1 \bo \Sigma_2)^{\frac12} \norm{\bo \Sigma_2\bo \Sigma_3}^{\frac12} + \sqrt{N_2}\tr(\bo \Sigma_1 \bo \Sigma_2)^{\frac12}\tr(\bo \Sigma_3 \bo \Sigma_2)^{\frac12}  \Big) \sqrt{\log(p\vee N)}
        \end{align}
        holds with probability at least $1 -O(d^{-c})$.

        \bigskip 
        
        With respect to the bounded noise with local dependence, we leverage the matrix Bernstein inequality to derive that 
        \begin{align}
            & \Big| \tr\big(\mb 1_{N_3\times N_1 }\mb E^{(1)} {{}\mb E^{(2)}}\t \mb E^{(2)} {{}\mb E^{(3)}}\t\big) \Big| \lesssim \sqrt{N_1}\norm{\mb 1_{N_3}\t \mb E^{(3)} \mb E^{(2)\top} \mb E^{(2)} \bo \Sigma_1^{\frac12}}_2 \sqrt{\log(p \vee N)} \\ 
            & \qquad \qquad + \sqrt{mB}\max_{c\in[l]} \norm{\mb E_{\cdot, S_c}^{(2)\top} \mb E^{(2)} \mb E^{(3)\top} \mb 1_{N_3}}_2 \log(p\vee N) 
            \label{eq: decoupled matrix product concentration (bounded matrix)}
        \end{align}
        with probability at least $1- O(d^{-c})$ conditional on $\mb E^{(2)}$ and $\mb E^{(3)}$. 

        Applying the matrix Bernstein inequality to $\norm{\mb 1_{N_3}\t \mb E^{(3)} \mb E^{(2)\top} \mb E^{(2)} \bo \Sigma_1^{\frac12}}_2$ and $\norm{\mb E_{\cdot, S_c}^{(2)\top} \mb E^{(2)} \mb E^{(3)\top}}$ yields that 
        \begin{align}
            & \norm{\mb 1_{N_3}\t \mb E^{(3)} \mb E^{(2)\top} \mb E^{(2)} \bo \Sigma_1^{\frac12}}_2 \lesssim \sqrt{N_3} \norm{\bo \Sigma_1^{\frac12}\mb E^{(2)\top} \mb E^{(2)} \bo \Sigma_3^{\frac12}}_F \sqrt{\log(p\vee N)} + \sqrt{m} B\max_{c\in[l]}\norm{\mb E_{\cdot, S_c}^{(2)\top} \mb E^{(2)} \bo \Sigma_1^{\frac12}} \log(p\vee N), 
            \\ 
            & \norm{\mb E_{\cdot, S_c}^{(2)\top} \mb E^{(2)} \mb E^{(3)\top} \mb 1_{N_3}} \lesssim \sqrt{N_3}\norm{\mb E_{\cdot, S_c}^{(2)\top} \mb E^{(2)} \bo \Sigma_3^{\frac12}}_F \sqrt{\log(p\vee N)} + \sqrt{m}B \max_{c'\in[l]} \norm{\mb E_{\cdot, S_c}^{(2)\top} \mb E^{(2)}_{\cdot, S_{c'}}} \log(p\vee N)
        \end{align}
        with probability at least $1- O(d^{-c})$ conditional on $\mb E^{(2)}$. Provided the inequality \eqref{eq: universality for E(2)} and the assumption on $\mb E^{(2)}$, one has for the quantities involved above that
        \begin{align}
            & \norm{\bo \Sigma_1^{\frac12}\mb E^{(2)\top} \mb E^{(2)} \bo \Sigma_3^{\frac12}}_F \leq \norm{\bo \Sigma_1^{\frac12}\mb E^{(2)\top}}_F \norm{\mb E^{(2)} \bo \Sigma_3^{\frac12}} \lesssim \sqrt{N_2} \tr(\bo \Sigma_1 \bo \Sigma_2)^{\frac12} \Big( \norm{\bo \Sigma_3\bo \Sigma_2}^{\frac12} \sqrt{N_2} + \tr(\bo \Sigma_3 \bo\Sigma_2)^{\frac12}\Big), \\ 
            &  \max_{c\in[l]}\norm{\mb E_{\cdot, S_c}^{(2)\top} \mb E^{(2)} \bo \Sigma_1^{\frac12}} \leq \norm{\mb E_{\cdot, S_c}^{(2)\top}}\norm{ \mb E^{(2)} \bo \Sigma_1^{\frac12}} \lesssim \norm{\bo \Sigma_2}^{\frac12} \sqrt{N_2}\Big( \norm{\bo \Sigma_1\bo \Sigma_2}^{\frac12} \sqrt{N_2} + \tr(\bo \Sigma_1 \bo\Sigma_2)^{\frac12}\Big), \\ 
            & \max_{c\in[l]} \norm{\mb E_{\cdot, S_c}^{(2)\top} \mb E^{(2)} \bo \Sigma_3^{\frac12}}_F \leq \norm{\mb E_{\cdot, S_c}^{(2)}}\norm{\mb E^{(2)} \bo \Sigma_3^{\frac12}}_F \lesssim \norm{\bo \Sigma_2}^{\frac12} \sqrt{N_2} \Big( \norm{\bo \Sigma_3\bo \Sigma_2}^{\frac12} \sqrt{N_2} + \tr(\bo \Sigma_3 \bo\Sigma_2)^{\frac12}\Big), \\ 
            & \max_{c,c'\in[l]} \norm{\mb E_{\cdot, S_c}^{(2)\top} \mb E^{(2)}_{\cdot, S_{c'}}}  \lesssim \norm{\bo \Sigma_2} N_2
        \end{align}
        with probability at least $1 -O(d^{-c})$. As a consequence, one has 
        \begin{align}
            & \norm{\mb 1_{N_3}\t \mb E^{(3)} \mb E^{(2)\top} \mb E^{(2)} \bo \Sigma_1^{\frac12}}_2 \lesssim \sqrt{N_2N_3} \tr(\bo \Sigma_1 \bo \Sigma_2)^{\frac12} \Big( \norm{\bo \Sigma_3\bo \Sigma_2}^{\frac12} \sqrt{N_2} + \tr(\bo \Sigma_3 \bo\Sigma_2)^{\frac12}\Big) \sqrt{\log(p \vee N)} \\ 
            &\qquad  \qquad + \sqrt{m}B \norm{\bo \Sigma_2}^{\frac12} \sqrt{N_2}\Big( \norm{\bo \Sigma_1\bo \Sigma_2}^{\frac12} \sqrt{N_2} + \tr(\bo \Sigma_1 \bo\Sigma_2)^{\frac12}\Big)\log(p \vee N) \\ 
            & \qquad \lesssim N_2 \sqrt{N_3} \tr(\bo \Sigma_1 \bo \Sigma_2) \norm{\bo \Sigma_2 \bo \Sigma_3}^{\frac12} \log(p \vee N) + \sqrt{N_2 N_3} \tr(\bo \Sigma_1 \bo \Sigma_2)^{\frac12} \tr(\bo \Sigma_3 \bo \Sigma_2)^{\frac12} \log(p \vee N), 
            \label{eq: decoupled matrix product concentration 1 (bounded matrix)}
            \\ 
            & \norm{\mb E_{\cdot, S_c}^{(2)\top} \mb E^{(2)} \mb E^{(3)\top} \mb 1_{N_3}} \lesssim \norm{\bo \Sigma_2}^{\frac12} \sqrt{N_2 N_3} \Big( \norm{\bo \Sigma_3\bo \Sigma_2}^{\frac12} \sqrt{N_2} + \tr(\bo \Sigma_3 \bo\Sigma_2)^{\frac12}\Big) \sqrt{\log (p \vee N)}  \\ 
            & \qquad \qquad + \sqrt{m} B \norm{\bo \Sigma_2 } N_2 \log(p \vee N) \\ 
            & \qquad \lesssim  \norm{\bo \Sigma_2}^{\frac12} \sqrt{N_2 N_3} \Big( \norm{\bo \Sigma_3\bo \Sigma_2}^{\frac12} \sqrt{N_2} + \tr(\bo \Sigma_3 \bo\Sigma_2)^{\frac12}\Big) \log (p \vee N)
            \label{eq: decoupled matrix product concentration 2 (bounded matrix)}
        \end{align}
        with probability at least $1- O(d^{-c})$, where we used $\sqrt{m} B \norm{\bo \Sigma_2}^{\frac12} \lesssim \sqrt{N_3} (\norm{\bo \Sigma_1 \bo \Sigma_2}^{\frac12} \wedge \norm{\bo \Sigma_3 \bo \Sigma_2}^{\frac12})$ and $\sqrt{m}B \norm{\bo \Sigma_2}^{\frac12} \lesssim \tr(\bo \Sigma_1 \bo \Sigma_2)^{\frac12} \wedge \tr(\bo \Sigma_3 \bo \Sigma_2)^{\frac12} $. 
        
        Plugging those upper bounds together with \eqref{eq: decoupled matrix product concentration 1 (bounded matrix)} and \eqref{eq: decoupled matrix product concentration 2 (bounded matrix)} into \eqref{eq: decoupled matrix product concentration (bounded matrix)} yields that 
        \begin{align}
            & \Big| \tr\big(\mb 1_{N_3\times N_1 }\mb E^{(1)} {{}\mb E^{(2)}}\t \mb E^{(2)} {{}\mb E^{(3)}}\t\big) \Big| \lesssim N_2 \sqrt{N_1 N_3} \tr(\bo \Sigma_1 \bo \Sigma_2) \norm{\bo \Sigma_2 \bo \Sigma_3}^{\frac12} \log(p \vee N) \\ 
            & \qquad + \sqrt{N_1 N_2 N_3} \tr(\bo \Sigma_1 \bo \Sigma_2)^{\frac12} \tr(\bo \Sigma_3 \bo \Sigma_2)^{\frac12} \log(p \vee N)  \\ 
            & \qquad+\sqrt{m} B \norm{\bo \Sigma_2}^{\frac12} \sqrt{N_2 N_3} \Big( \norm{\bo \Sigma_3\bo \Sigma_2}^{\frac12} \sqrt{N_2} + \tr(\bo \Sigma_3 \bo\Sigma_2)^{\frac12}\Big) \log (p \vee N)^2 \\ 
            \lesssim  &  N_2 \sqrt{N_1 N_3} \tr(\bo \Sigma_1 \bo \Sigma_2) \norm{\bo \Sigma_2 \bo \Sigma_3}^{\frac12} (\log(p \vee N))^2  + \sqrt{N_1 N_2 N_3} \tr(\bo \Sigma_1 \bo \Sigma_2)^{\frac12} \tr(\bo \Sigma_3 \bo \Sigma_2)^{\frac12} (\log(p \vee N))^2
        \end{align}
        with probability at least $1- O(d^{-c})$, where we similarly reason in the last inequality as in \eqref{eq: decoupled matrix product concentration 1 (bounded matrix)} and \eqref{eq: decoupled matrix product concentration 2 (bounded matrix)}. 

        \bigskip

        Lastly, we control the terms $\norm{\mb 1_{N_1}\t \mb G^{(1)} {\mb G^{(2)}}\t \mb G^{(2)} \mb T}_2 $ and $\norm{\mb 1_{N_1}\t \mb E^{(1)} {\mb E^{(2)}}\t \mb E^{(2)} \mb T}_2$ in what follows. From the concentration bounds for Gaussian distributed random matrix, one obtain that
        \begin{align}
            & \norm{\bo \Sigma_1^{\frac12} {\mb G^{(2)}}\t \mb G^{(2)} \mb T}_F \leq \sqrt{r} \norm{\bo \Sigma_1^{\frac12} {\mb G^{(2)}}\t \mb G^{(2)} \mb T}, \\ 
            & \norm{ \bo \Sigma_1^{\frac12} {\mb G^{(2)}}\t \mb G^{(2)} \mb T} \lesssim (\sqrt{N_2}\norm{\bo \Sigma_1 \bo \Sigma_2}^{\frac12} + \tr(\bo \Sigma_1 \bo \Sigma_2)^{\frac12}) (\norm{\mb T\t \bo \Sigma_2 \mb T}^{\frac12} \sqrt{N_2 } )
            \label{eq: bounds for G2tG2T}
        \end{align}
        with probability at least $1- O(d^{-c})$. A direct application of the Hanson-Wright inequality gives that 
        \begin{align}
            &\Big| \norm{\mb 1_{N_1}\t \mb G^{(1)} {\mb G^{(2)}}\t \mb G^{(2)} \mb T}_2^2 - N_1 \norm{\bo \Sigma_1^{\frac12} {\mb G^{(2)}}\t \mb G^{(2)} \mb T}_F^2 \Big| \\
            \lesssim &  N_1 \norm{\bo \Sigma_1^{\frac12} {\mb G^{(2)}}\t \mb G^{(2)} \mb T}_F \norm{\bo \Sigma_1^{\frac12} {\mb G^{(2)}}\t \mb G^{(2)} \mb T}\sqrt{\log(p\vee N)} + N_1 \norm{\bo \Sigma_1^{\frac12} {\mb G^{(2)}}\t \mb G^{(2)} \mb T}^2 \log(p\vee N)
        \end{align}
        with probability at least $1- O(d^{-c})$. Taken together with \eqref{eq: bounds for G2tG2T}, one has with probability at least $1- O(d^{-c})$ that  
        \begin{align}
            & \norm{\mb 1_{N_1}\t \mb G^{(1)} {\mb G^{(2)}}\t \mb G^{(2)} \mb T}_2 \lesssim \sqrt{r}(\sqrt{N_2}\norm{\bo \Sigma_1 \bo \Sigma_2}^{\frac12} + \tr(\bo \Sigma_1 \bo \Sigma_2)^{\frac12}) (\norm{\mb T\t \bo \Sigma_2 \mb T}^{\frac12} \sqrt{N_2 } ) \sqrt{\log(p\vee N)}. 
        \end{align}

        For the non-Gaussian case, we invoke the matrix Bernstein inequality to derive that 
        \begin{align}
            & \norm{\mb 1_{N_1}\t \mb E^{(1)} {\mb E^{(2)}}\t \mb E^{(2)} \mb T}_2 \lesssim  \norm{\bo \Sigma_1^{\frac12} {\mb E^{(2)}}\t \mb E^{(2)} \mb T}_F \sqrt{N_1\log(d\vee N)} + \sqrt{m}B \max_{c\in[l]}\norm{(\mb E^{(2)})_{\cdot, S_c}\t \mb E^{(2)} \mb T} \log(d\vee N) \\ 
            \lesssim & (\sqrt{N_2}\norm{\bo \Sigma_1 \bo \Sigma_2}^{\frac12} + \tr(\bo \Sigma_1 \bo \Sigma_2)^{\frac12}) (\norm{\mb T\t \bo \Sigma_2 \mb T}^{\frac12} \sqrt{N_2 } ) \sqrt{rN_1\log(d\vee N)} \\ 
            & \qquad + \sqrt{m}B  N_2 \norm{\bo \Sigma_2}^{\frac12} \norm{\mb T\t \bo \Sigma_2 \mb T}^{\frac12} \log(d\vee N) \\ 
            \lesssim &(\sqrt{N_2}\norm{\bo \Sigma_1 \bo \Sigma_2}^{\frac12} + \tr(\bo \Sigma_1 \bo \Sigma_2)^{\frac12}) (\norm{\mb T\t \bo \Sigma_2 \mb T}^{\frac12} \sqrt{N_2 } ) \sqrt{rN_1\log(p\vee N)}
        \end{align}
        with probability at least $1- O(d^{-c})$, since we have the assumption that $\sqrt{m}B \norm{\bo \Sigma_2}^{\frac12}\ll \tr(\bo \Sigma_1 \bo \Sigma_2)^{\frac12}$ and the concentrations that
        \longeq{
            & \norm{\bo \Sigma_1^{\frac12} {\mb E^{(2)}}\t \mb E^{(2)} \mb T}_F \leq \norm{\bo \Sigma_1^{\frac12} {\mb E^{(2)}}\t}_F \norm{\mb E^{(2)} \mb T}, \qquad \norm{(\mb E^{(2)})_{\cdot, S_c}\t \mb E^{(2)} \mb T} \leq \norm{(\mb E^{(2)})_{\cdot, S_c}} \norm{\mb E^{(2)} \mb T}, \\ 
            & \norm{\bo \Sigma_1^{\frac12} {\mb E^{(2)}}\t}_F \lesssim \sqrt{N_2} \tr(\bo \Sigma_1 \bo \Sigma_2)^{\frac12}, \qquad  \norm{\mb E^{(2)} \mb T} \lesssim \sqrt{N_2} \norm{\mb T\t \bo \Sigma_2 \mb T}^{\frac12}, \quad \max_{c \in[l]} \norm{(\mb E^{(2)})_{\cdot, S_c}}  \lesssim \sqrt{N_2} \norm{\bo \Sigma_2}^{\frac12}
        }
        hold with probability at least $1- O(d^{-c})$.

\qed

\subsubsection{Proof of Lemmas~\ref{lemma: noise matrix concentrations using the universality (Gaussian)},~\ref{lemma: noise matrix concentrations using the universality (bounded)}, and~\ref{lemma: 2k moment of At Ei}}
\label{subsubsec: proof of noise matrix concentration}

\begin{proof}[Proof of Lemma~\ref{lemma: noise matrix concentrations using the universality (Gaussian)}]

For a tight concentration in terms of $K$ on the quantities with Gaussian ensembles, we resort to the results in \cite{bandeira2023matrix}.  
To begin with, an upper bound for  $\norm{\mb E}$ arises by the universality result on Gaussian matrices \cite[Corollary 2.2]{bandeira2023matrix}:
    \begin{align}
        \norm{\mb E} \leq \norm{\mb E_{\mathrm{free}}} + C v(\mb E)^{\frac{1}{2}}\sigma (\mb E)^{\frac{1}{2}}(\log d )^{\frac{3}{4}} + C \sigma_*(\mb E) \sqrt{\log d},
    \end{align}
    with probability at least $1 - O(d^{-c-2})$ where $C$ is a universal constant. Here, similar to the free probability element in the derivation of \eqref{eq: universality and free prob for covariance estimation}, $\mb E_{\text{free}}$ is an element in the tensor product space of the real-valued $d\times d$ matrix space $M_{d}(\bb R)$ and a free semi-circle family $\mc A$.

    To upper bound the norm of $\mb E_{\text{free}}$, we employ Lemma \cite[Lemma 2.5]{bandeira2023matrix} to derive that,  given $\mb X = \mb A_0 + \sum_{i\in[n]} \mb A_i g_i \in \bb C^{n_1\times n_2}$ where $\{\mb A_i\}_{i=0}^n \subset \bb R^{n_1 \times n_2}$ are determinstic matrices and $\{g_i\}_{i\in[n]}$ are independent standard Gaussian, Lemma 2.5 in \cite{bandeira2023matrix} implies that 
    \eq{
        \norm{\mb X_{\mathrm{free}}} \leq \norm{\mb A_0} + \Bignorm{\sum_{i\in[n]}\mb A_i^* \mb A_i}^{\frac{1}{2}} + \Bignorm{\sum_{i\in[n]}\mb A_i \mb A_i^*}^{\frac{1}{2}}. 
    }
    This combines with \eqref{eq: bound on sigma(E)}, \eqref{eq: bound on v(E)}, and \eqref{eq: bound on sigma_*(E)}, the concentration inequality could be further simplified as 
    \begin{align}
         \norm{\mb E} \lesssim &\sigma \sqrt{p} + \tilde \sigma \sqrt{n} + \tilde \sigma^{\frac{1}{2}}(\sigma^{\frac{1}{2}}p^{\frac{1}{4}} + \tilde \sigma^{\frac{1}{2}} n^{\frac{1}{4}})(\log d )^{\frac{3}{4}} +\tilde \sigma \sqrt{\log d} 
        \lesssim \sigma \sqrt{p} + \tilde \sigma \sqrt{n}
    \end{align}
    with probability at least $1- O(d^{-c-2})$. 

    Similarly, the bounds for $\mb E_{\mc I_k(\mb z^*), \cdot}$, $\mb E_{i,\cdot}$, and $\mb E\mb V^*$ follow from similar arguments: 
    \begin{align}
        &\norm{\mb E_{\mc I_k(\mb z^*), \cdot}} \lesssim \sigma \sqrt{p} +  \tilde\sigma \sqrt{n}, \\ 
        & \norm{\mb E_{i}}_2 \lesssim  \norm{(\mb E_{i})_{\mathrm{free}}} + Cv(\mb E_{i})^{\frac{1}{2}}\sigma(\mb E_{i})^{\frac{1}{2}} (\log d)^{\frac{3}{4}} + C \sigma_*(\mb E_{i}) \sqrt{\log d}  
        \lesssim  \sigma \sqrt{p}, \\ 
        &\norm{\mb E\mb V^*} \lesssim  \norm{(\mb E \mb V^*)_{\mathrm{free}}} + Cv(\mb E \mb V^*)^{\frac{1}{2}}\sigma(\mb E \mb V^*)^{\frac{1}{2}} (\log d)^{\frac{3}{4}} + C \sigma_*(\mb E \mb V^*) \sqrt{\log d} 
        \lesssim  \bar \sigma \sqrt{n}
    \end{align}
    hold with probability at least $1- O(d^{-c-2})$.

    Next, applying the Hanson-Wright inequality gives that 
    \begin{align}
        & \Bignorm{\sum_{i\in \mc I_k(\mb z^*)} \mb E_i\t \mb V^*}_2 \lesssim \bar\sigma \sqrt{K n_k} + \bar \sigma (K n_k \log d)^{\frac14}+ \bar \sigma \sqrt{n_k \log d} \lesssim \bar \sigma \sqrt{n_k K \log d}, \\ 
        & \norm{\mb E \mb V^*}\ti \lesssim \bar \sigma \sqrt{K \log d},\qquad  \norm{\mc H(\mb E \mb E\t) \mb U^*}\ti \lesssim \tilde \sigma \sigma \sqrt{pK \log d},  
    \end{align}
    with probability at least $1- O(d^{-c-2})$.

 To control $\norm{{\mb U^*}^\top \mc H(\mb E \mb E^\top)\mb U^*}$, note that there exists an orthogonal matrix $\mb Q\in\mathbb O(K)$ such that  
$
\mb U^* \mb Q = \mb Z^* \operatorname{diag}\!\left(\frac{1}{\sqrt{n_k}}\right)_{k\in[K]} 
$.
Therefore,
\begin{align}
\norm{{\mb U^*}^\top \mc H(\mb E \mb E^\top)\mb U^*}
\le r \Bigg[
&\max_{k_1\neq k_2\in[r]} \frac{1}{\sqrt{n_{k_1}n_{k_2}}}
\Big|\!\sum_{i_1\in\mc I_{k_1}(\mb z^*),\, i_2\in\mc I_{k_2}(\mb z^*)}
\langle \mb E_{i_1}, \mb E_{i_2}\rangle\Big|  \notag\\
&+ \max_{k\in[r]} \frac{1}{n_k}
\Big|\!\sum_{\substack{i_1,i_2\in\mc I_k(\mb z^*)\\ i_1\neq i_2}}
\langle \mb E_{i_1}, \mb E_{i_2}\rangle\Big|
\Bigg].
\label{eq: decomposition of Ut H(EEt) U}
\end{align}

Conditioning on $\mb E_{\mc I_{k_2}(\mb z^*),\cdot}$ and applying Hoeffding’s inequality to the cross-cluster term yields, with probability at least $1-O(d^{-c-2})$,
\[
\max_{k_1\neq k_2} \frac{1}{\sqrt{n_{k_1}n_{k_2}}}
\Big|\sum_{i_1\in\mc I_{k_1},\, i_2\in\mc I_{k_2}}
\langle \mb E_{i_1}, \mb E_{i_2}\rangle\Big|
\lesssim \tilde\sigma \sigma \sqrt{p\log d}.
\]

The within-cluster term admits the same bound by combining the decoupling argument of \cite[Theorem 1]{de1995decoupling} with Hoeffding’s inequality. Consequently, with probability at least $1-O(d^{-c-2})$,
\[
\norm{{\mb U^*}^\top \mc H(\mb E \mb E^\top)\mb U^*}
 \lesssim  K\tilde\sigma \sigma \sqrt{p\log d}.
\]
    
    Regarding the last inequality, it suffices to control $\mb 1_{n_{a}}\t \mb E_{\mc I_a(\mb z^*),\cdot} \bo \Sigma_k \mb E_{\mc I_a(\mb z^*),\cdot}\t \mb 1_{n_{a}}$ for $a \in [K] \backslash \{k\}$, and $\mb 1_{n_{k}}\t \mb E_{\mc I_k(\mb z^*)\backslash\{i\},\cdot} \bo \Sigma_k \mb E_{\mc I_k(\mb z^*) \backslash\{i\},\cdot}\t \mb 1_{n_{k}}$ for $i \in[n]$. By the Hanson-Wright inequality, one has with probability at least $1- O(d^{-c})$ that 
    \begin{align}
        & \Big| \mb 1_{n_{a}}\t \mb E_{\mc I_a(\mb z^*),\cdot} \bo \Sigma_k \mb E_{\mc I_a(\mb z^*),\cdot}\t \mb 1_{n_{a}} - \bb E\big[ \mb 1_{n_{a}}\t \mb E_{\mc I_a(\mb z^*),\cdot} \bo \Sigma_k \mb E_{\mc I_a(\mb z^*),\cdot}\t \mb 1_{n_{a}}\big] \Big|  \\
        \lesssim & n_a \sqrt{\tr(\bo \sigma_a \bo \Sigma_k \bo \Sigma_a \bo \sigma_k) \log d} + n_a \norm{\bo \sigma_a \bo \Sigma_k \bo \Sigma_a \bo \sigma_k} \log d 
        \lesssim  n_a \sqrt{p}\sigma \tilde\sigma^3 \log d,
        \label{eq: concentration on 1t Ea Sigmak Eat 1}
        \\ 
        & \max_{i \in \mc I_k(\mb z^*)}\Big| \mb 1_{n_{k}}\t \mb E_{\mc I_k(\mb z^*)\backslash\{i\},\cdot} \bo \Sigma_k \mb E_{\mc I_k(\mb z^*) \backslash\{i\},\cdot}\t \mb 1_{n_{k}} - \bb E\big[ \mb 1_{n_{k}}\t \mb E_{\mc I_k(\mb z^*)\backslash\{i\},\cdot} \bo \Sigma_k \mb E_{\mc I_k(\mb z^*) \backslash\{i\},\cdot}\t \mb 1_{n_{k}}\big]\Big|  \\ 
        \lesssim &  n_k \sqrt{p}\sigma \tilde\sigma^3 \log d.
    \end{align}

    Besides, for $a \neq b \in[K] \backslash\{k\}$ one has by the Hanson-Wright inequality that
    \begin{align}
        & \Big|\mb 1_{n_{a}}\t \mb E_{\mc I_a(\mb z^*),\cdot} \bo \Sigma_k \mb E_{\mc I_b(\mb z^*),\cdot}\t \mb 1_{n_{b}} \Big| \lesssim \sqrt{n_a n_bp}\sigma\tilde\sigma^3 \log d, \\ 
        &  \max_{i\in\mc I_k(\mb z^*)}\Big|\mb 1_{n_{a}}\t \mb E_{\mc I_a(\mb z^*),\cdot} \bo \Sigma_k \mb E_{\mc I_k(\mb z^*)\backslash\{i\},\cdot}\t \mb 1_{n_{k}} \Big| \lesssim \sqrt{n_a n_k p}\sigma\tilde\sigma^3 \log d
    \end{align}
    with probability at least $1- O(d^{-c})$. Combining these pieces implies that 
    \begin{equation}
        \max_{i\in[n],~k_1,k_2,k\in [K]}  \Bignorm{{\mb U_{\cdot,k_1}^*}\t \mc P_{-i,\cdot}(\mb E) \bo \Sigma_k \mc P_{-i,\cdot}(\mb E)\t \mb U_{\cdot,k_2}^* - \sum_{i'\in[n]\backslash\{i\}} \tr(\bo \Sigma_{z_{i'}^*} \bo \Sigma_{z_i^*}){\mb U^*_{i',k_1}}\t \mb U_{i',k_2}^*} \lesssim \sqrt{p}\sigma\tilde\sigma^3 \log d, 
        \label{eq: single entry of Ut E Sigmak Et U}
    \end{equation}
    which leads to 
    \begin{align}
        & \max_{i \in[n], k\in[K]} \Bignorm{{\mb U^*}\t \mc P_{-i,\cdot}(\mb E) \bo \Sigma_k \mc P_{-i,\cdot}(\mb E)\t \mb U^* - \sum_{i'\in[n]\backslash\{i\}} \tr(\bo \Sigma_{z_{i'}^*} \bo \Sigma_{z_i^*}){\mb U^*_{i',\cdot}}\t \mb U_{i',\cdot}^*} \lesssim K\sqrt{p}\sigma\tilde\sigma^3 \log d 
    \end{align}
    with probability at least $1- O(d^{-c})$. 
\end{proof}

\begin{proof}[Proof of Lemma~\ref{lemma: noise matrix concentrations using the universality (bounded)}]

Throughout the proof, we will repetitively make use of the following lemma to upper bound the spectral norm of the matrices in interest: 
\begin{lemma}[Corollary 2.15 in \cite{brailovskaya2022universality}]
\label{lemma: corollary 2.15 in Brailovskaya's paper}
    Let $\mb Y = \sum_{i\in[n]}\mb Z_i$, where $\mb Z_1, \cdots, \mb Z_n$ are independent (possibly not self-disjoint) $d\times d$ random matrices with $\bb E[\mb Z_i]=0$. Then 
    \begin{align}
        & \bb P \big[\norm{\mb Y} \geq 2 \sigma(\mb Y) + C\big(v(\mb Y)^{\frac{1}{2}}\sigma(\mb Y)^{\frac{1}{2}} + \sigma_*(\mb Y)t^{\frac{1}{2}} \\ 
        & \qquad + R(\mb Y)^{\frac{1}{3}}\sigma(\mb Y)^{\frac{2}{3}}t^{\frac{2}{3}} + R(\mb Y)t
        \big)\big]\leq 4de^{-t}.
    \end{align}
\end{lemma}

Now everything boils down to upper bounding the quantities in the above lemma. We inspect each case as follows: 
\begin{itemize}[leftmargin=*]
    \item For the full-size matrix $\mb E$, we have $\bignorm{\bb E[\mb E^\top\mb E]}^{\frac12} \le \tilde \sigma \sqrt{n}$ and $ \bignorm{\bb E[\mb E\mb E^\top]}^{\frac12} \le \sigma \sqrt{p} $, 
which leads to 
\begin{align}
    & \sigma(\mb E) \leq \sigma \sqrt{p} + \tilde \sigma \sqrt{n}.\label{eq: bound on sigma(E)}
\end{align}

Moreover, we have 
\begin{align}
    & v(\mb E)  = \norm{\mathrm{Cov}\big(\mathrm{vec} (\mb E) \big)}^{\frac{1}{2}} = \max_{i\in[n]}\bb E[\mb E_{i,\cdot}\mb E_{i,\cdot}\t]^{\frac12} \leq \tilde \sigma, \label{eq: bound on v(E)}\\
    & \sigma_*(\mb E)  = \sup_{\norm{\mb v}_2 = \norm{\mb w}_2 = 1}\bb E\big[\big|\ip{\mb v,\mb E\mb w}\big|^2\big] =  \sup_{\norm{\mb w} = 1}\max_{i\in[n]}\mb \|\bb E[|\mb E_i\t \mb w|^2 ]\mb\|^{\frac12} \leq \tilde \sigma, \label{eq: bound on sigma_*(E)}\\ 
    & R(\mb E) = \Big\|\max_{i,j}\|\mb E_{i,S_j}\|\Big\|_{\infty} \leq \sqrt{m}B. \label{eq: bound on R(E)}
\end{align}	

Substituions of \eqref{eq: bound on sigma(E)}, \eqref{eq: bound on v(E)}, \eqref{eq: bound on sigma_*(E)}, and \eqref{eq: bound on R(E)} into Lemma~\ref{lemma: corollary 2.15 in Brailovskaya's paper} yields that
\eq{
\norm{\mb E} \lesssim  \sigma \sqrt{p} + \tilde \sigma \sqrt{n} + m^{\frac{1}{6}} B^{\frac{1}{3}} \big( \sigma^{\frac{2}{3}}p^{\frac{1}{3}} + \tilde \sigma^{\frac{2}{3}} n^{\frac{1}{3}} \big) \big( \log d\big)^{\frac{2}{3}} + \sqrt{m} B \log d
}
holds with probability at least $1- O(d^{-c-2})$
by taking $t = c\log d$ for some sufficiently large $c$. 
\item The upper bound for $\mb E_{\mc I_k(\mb z^*),\cdot}$ follow from a similar derivation as above, with the substitution of $n$ by $n_k$. 

\item Regarding the $i$-th row of $\mb E$, a similar derivation yields that 
\begin{align}
    & \sigma(\mb E_{i})^2 = \bb E\norm{\mb E_{i}}_2^2 \leq \sigma^2 p, \quad  v(\mb E) = \norm{\cov(\mb E_{i})}^{\frac{1}{2}} \leq \tilde \sigma, \quad  \sigma_*(\mb E_{i}) \leq \tilde\sigma, \quad  R(\mb E_{i}) \leq \sqrt{m}B. 
\end{align}

Taking these collectively into Lemma~\ref{lemma: corollary 2.15 in Brailovskaya's paper} yields that 
\begin{align}
    & \norm{\mb E_i}_2 \lesssim \sigma \sqrt{p} + \tilde\sigma\sqrt{\log d} + m^{\frac{1}{6}}B^{\frac{1}{3}} \sigma^{\frac{4}{3}} p^{\frac{3}{2}} (\log d)^{\frac23} + \sqrt{m}B \log d \\ 
    \lesssim & \sigma \sqrt{p} + \tilde\sigma\sqrt{\log d} + m^{\frac{1}{6}}B^{\frac{1}{3}} \sigma^{\frac{4}{3}} p^{\frac{3}{2}} (\log d)^{\frac23} + \sqrt{m}B \log d 
\end{align}
holds with probability at least $1- O(d^{-c-2})$ since $\sqrt{m \log d} \ll \sqrt{p}$. 

\item In terms of $\mb E\mb V^*$, we have
\begin{align}
    &\sigma(\mb E\mb V^*) = \max\left\{\bignorm{\bb E[\mb E\mb V^*{\mb V^{*\top}} \mb E\t ] }, \bignorm{\bb E[{\mb V^{*\top}} \mb E\t \mb E\mb V^*] }\right\}^{\frac{1}{2}} \leq \max\{\bar \sigma, \sqrt{n} \bar \sigma \} = \sqrt{n} \bar \sigma, \\
    & v(\mb E\mb V^*) = \norm{\mathrm{Cov}\big( \mathrm{vec}(\mb E\mb V^*)\big)}^{\frac{1}{2}} = \bar \sigma, \quad \sigma_*(\mb E\mb V^*) = \sup_{\norm{\mb v}_2 = \norm{\mb w}_2 = 1} \bb E\big[\big|\ip{\mb v, \mb E\mb w} \big|^2\big] = \bar \sigma ,\quad  R(\mb E\mb V^*) \leq mB \norm{\mb V^*}\ti . 
\end{align} 
Therefore, applying Lemma~\ref{lemma: corollary 2.15 in Brailovskaya's paper} yields that 
\begin{align}
    \norm{\mb E\mb V^*} \lesssim 
    & \sqrt{n}\bar \sigma + n^{\frac{1}{4}} \bar \sigma^{\frac{1}{2}} \tilde\sigma^{\frac{1}{2}} + \bar \sigma(\log d)^{\frac{1}{2}}  + \big(mB \norm{\mb V^*}\ti\big)^{\frac{1}{3}} n^{\frac{1}{3}} \tilde\sigma^{\frac23}(\log d)^{\frac23} + m B\norm{\mb V^{*}}\ti\log d \\
    \lesssim &  \sqrt{n}\bar \sigma  +  \big(mB \norm{\mb V^{*}}\ti\big)^{\frac{1}{3}} n^{\frac{1}{3}} \tilde\sigma^{\frac23}(\log d)^{\frac23} + m B\norm{\mb V^{*}}\ti\log d
\end{align}
holds with probability at least $1- O(d^{-c})$. 

\item We apply the matrix Bernstein inequality (cf. \cite[Theorem 6.1.1]{tropp2015introduction}) to derive that
\begin{align}
    & \bignorm{\sum_{i\in[n], z_i^* = k}\mb E_i\t \mb V^*}_2 \lesssim \bar \sigma \sqrt{n_k K \log d}  + mB\norm{\mb V^{*}}\ti \log d \lesssim \bar \sigma \sqrt{n_k K \log d}, \\ 
    & \norm{\mb E \mb V^*}\ti \lesssim \bar \sigma \sqrt{K \log d}, ~\norm{\mc H(\mb E \mb E\t) \mb U^*}\ti \lesssim \tilde \sigma \sigma \sqrt{p K \log d},~ \max_{c\in[l]} \norm{\mb E_{\cdot, S_c}\t \mb U^*}  \lesssim \tilde \sigma \sqrt{m K\log d}
\end{align}
hold with probablity at least $1- O(d^{-c-2})$.

\item Combining the decomposition in \eqref{eq: decomposition of Ut H(EEt) U} with the decoupling result of \cite[Theorem~1]{de1995decoupling} and an application of Bernstein’s inequality yields the desired upper bound for $\norm{{\mb U^*}^\top \mc H(\mb E \mb E^\top)\mb U^*}$. 

\item Denote by $\mb G$ an independent Gaussian analog of $\mb E$ sharing the same mean and covariance structures. 
By applying the universality result \cite[Theorem~2.5]{brailovskaya2022universality}, one has with probability at least $1- O(d^{-c})$ that 
\begin{align}
    &\Big| \norm{\mb 1_{n_{a}}\t \mb E_{\mc I_a(\mb z^*),\cdot} \bo \Sigma_k^{\frac12}}_2 - \norm{\mb 1_{n_{a}}\t \mb G_{\mc I_a(\mb z^*),\cdot} \bo \Sigma_k^{\frac12}}_2 \Big| \lesssim \sqrt{n_a}\tilde\sigma^2 \sqrt{\log d} + m^{\frac16}B^{\frac13}  \tilde\sigma \sigma^{\frac23}p^{\frac13} (\log d)^{\frac23} + \sqrt{m}B\tilde\sigma \log d.
\end{align}
Combining it with the inequality \eqref{eq: concentration on 1t Ea Sigmak Eat 1} gives that 
\begin{align}
    & \norm{\mb 1_{n_{a}}\t \mb E_{\mc I_a(\mb z^*),\cdot} \bo \Sigma_k^{\frac12}}_2 - \bb E\Big[ \norm{\mb 1_{n_{a}}\t \mb E_{\mc I_a(\mb z^*),\cdot} \bo \Sigma_k^{\frac12}}_2^2\Big]^{\frac12} \\ 
    \lesssim & \big(n_a\tr(\bo \Sigma_a \bo \Sigma_k)\big)^{-\frac12}\Big(  n_a \big( 1 \vee \frac{\tilde\sigma}{\sqrt{p}\sigma} \big)\sqrt{p}\sigma \tilde\sigma^3 \log d\Big) + \sqrt{n_a}\tilde\sigma^2 \log d + \tilde\sigma^{\frac43} \sigma^{\frac23}n_a^{\frac16} p^{\frac13} (\log d)^{\frac23}
\end{align}
with probability at least $1- O(d^{-c})$, which leads to 
\begin{align}
    & \norm{\mb 1_{n_{a}}\t \mb E_{\mc I_a(\mb z^*),\cdot} \bo \Sigma_k^{\frac12}}_2^2 - \bb E\Big[ \norm{\mb 1_{n_{a}}\t \mb E_{\mc I_a(\mb z^*),\cdot} \bo \Sigma_k^{\frac12}}_2^2\Big] \lesssim 
     n_a \big( 1 \vee \frac{\tilde\sigma}{\sqrt{p}\sigma} \big)\sqrt{p}\sigma \tilde\sigma^3 \log d + \tilde\sigma^{\frac73} \sigma^{\frac53}n_a^{\frac23} p^{\frac56} (\log d)^{\frac23}.
\end{align} 
Similarly, we can prove that 
\begin{align}
    & \max_{i\in[n]}\norm{\mb 1_{n_{k}}\t \mb E_{\mc I_k(\mb z^*)\backslash\{i\},\cdot} \bo \Sigma_k^{\frac12}}_2^2 - \bb E\Big[ \norm{\mb 1_{n_{k}}\t \mb E_{\mc I_k(\mb z^*)\backslash\{i\},\cdot} \bo \Sigma_k^{\frac12}}_2^2\Big] \\
    \lesssim & 
     n_a \big( 1 \vee \frac{\tilde\sigma}{\sqrt{p}\sigma} \big)\sqrt{p}\sigma \tilde\sigma^3 \log d + \tilde\sigma^{\frac73} \sigma^{\frac53}n_a^{\frac23} p^{\frac56} (\log d)^{\frac23}.
\end{align}
holds with probability at least $1- O(d^{-c})$. 

Moreover, owing to the matrix Bernstein inequality, one has for $a \neq b \in [K] \backslash\{k\}$ that 
\begin{align}
    & \Big|\mb 1_{n_{a}}\t \mb E_{\mc I_a(\mb z^*),\cdot} \bo \Sigma_k \mb E_{\mc I_b(\mb z^*),\cdot}\t \mb 1_{n_{b}} \Big| \\
    \lesssim & \sqrt{n_a}\norm{\bo \Sigma_a^{\frac12} \bo \Sigma_k}\norm{\mb 1_{n_b}\t \mb E_{\mc I_b(\mb z^*),\cdot}}_2 \sqrt{\log d} + \sqrt{m}B \max_{c\in[l]} \norm{(\bo \Sigma_k)_{\mc I_c(\mb z^*),\cdot} \mb E_{\mc I_b(\mb z^*),\cdot}\t \mb 1_{n_{b}}}_2\log d \\ 
    \lesssim &  \sqrt{n_an_bp} \tilde\sigma^3 \sigma \sqrt{\log d} + \sqrt{m}B \tilde\sigma^2 \big(\tilde\sigma \sqrt{n_b \log d }\big) \log d  \lesssim  \sqrt{n_an_bp} \tilde\sigma^3 \sigma \sqrt{\log d}, \\ 
    & \max_{i\in\mc I_k(\mb z^*)}\Big|\mb 1_{n_{a}}\t \mb E_{\mc I_a(\mb z^*),\cdot} \bo \Sigma_k \mb E_{\mc I_k(\mb z^*)\backslash\{i\},\cdot}\t \mb 1_{n_{k}} \Big|  \lesssim  \sqrt{n_an_k p} \tilde\sigma^3 \sigma \sqrt{\log d}
\end{align}
with probability at least $1- O(d^{-c})$. 

Then mimicing the arguments in \eqref{eq: single entry of Ut E Sigmak Et U} yields that 
\longeq{
    & \max_{i\in[n],k_1,k_2,k\in [K]}  \Bignorm{{\mb U_{\cdot,k_1}^{*\top}} \mc P_{-i,\cdot}(\mb E) \bo \Sigma_k \mc P_{-i,\cdot}(\mb E)\t \mb U_{\cdot,k_2}^* - \sum_{i'\in[n]\backslash\{i\}} \tr(\bo \Sigma_{z_{i'}^*} \bo \Sigma_{z_i^*}){\mb U^{*}_{i',k_1}} \mb U_{i',k_2}^*} \lesssim \sqrt{p}\sigma\tilde\sigma^3 \sqrt{\log d},  \\ 
    & \max_{i \in[n], k\in[K]} \Bignorm{{\mb U^*}\t \mc P_{-i,\cdot}(\mb E) \bo \Sigma_k \mc P_{-i,\cdot}(\mb E)\t \mb U^* - \sum_{i'\in[n]\backslash\{i\}} \tr(\bo \Sigma_{z_{i'}^*} \bo \Sigma_{z_i^*}){\mb U^*_{i',\cdot}}\t \mb U_{i',\cdot}^*} \lesssim K\sqrt{p}\sigma\tilde\sigma^3 \sqrt{\log d}
}
    with probability at least $1- O(d^{-c})$.

\end{itemize}
\end{proof}

\begin{proof}[Proof of Lemma~\ref{lemma: 2k moment of At Ei}]
    The bound for $\bb E\big[\norm{\mb A\t \mb E_i}_2^{2k} \big]$ follows by Theorem 2.8 in \cite{brailovskaya2022universality} that 
    \begin{align}
        & \bb E\big[\bignorm{\mb A\t \mb E_i}_2^{2k} \big]^{\frac{1}{2k}} \leq   \bb E\big[\bignorm{\mb A\t \mb G_i}_2^{2k} \big]^{\frac{1}{2k}} + C \Big(\bb E\big[\sum_{c\in[l]} \frac{1}{K}\norm{\mb E_{i,S_c} \mb A_{S_c,\cdot}}_2^{2k} \big] \Big)^{\frac{1}{2k}}k^2 \\
        \leq & \bb E\big[\bignorm{\mb A\t \mb G_i}_2^{2k} \big]^{\frac{1}{2k}} + C \big(\tr(\mb A\t \bo \Sigma_{z_i^*} \mb A)/K \big)^{\frac{1}{2k}} \big(\sqrt{m}B \max_{c\in[l]}\norm{\mb A_{S_c,\cdot}}\big)^{\frac{k-1}{k}} k^2
    \end{align}
    where $\mb G$ is a standard Gaussian matrix with identical covariance as $\mb E$.

\end{proof}

\subsection{Stability of Perturbed Decision Boundary}
To make sure that the decision boundary is stable in the presence of randomness, we are going to state some technical lemmas, part of which directly come from \cite{chen2024optimal} or can be proved following the same route.

\subsubsection{Proof of Lemma~\ref{lemma: SNR and distance}}
\label{subsubsec: proof of lemma SNR and distance}
    The proof of this lemma is similar to the proof of \cite[Lemma 6.3]{chen2021optimal}. 

    \emph{Upper Bound for $\mathsf{SNR}_{a,b}$} Consider $\mb x_0 = {\mb S_a^*}^{-\frac{1}{2}}(\bo \theta_b^* - \bo \theta_a^*)$. Observe that 
    \begin{align}
        & \mb x_0\t\big(\mb I - {\mb S_{a}^*}^{\frac{1}{2}}{\mb S_{b}^*}^{-1} {\mb S_{a}^*}^{\frac{1}{2}}\big)\mb x_0 + 2\mb x_0\t{\mb S_{a}^*}^{\frac{1}{2}}{\mb S_{b}^*}^{-1}{\mb V^{*\top}}\big(\bo \theta_{b}^* - \bo \theta_{a}^*\big) \\ 
        & -  \big(\bo \theta_{b}^* - \bo \theta_{a}^*\big)\t \mb V^*{\mb S_{b}^*}^{-1}{\mb V^{*\top}}\big(\bo \theta_{b}^* - \bo \theta_{a}^*\big)= (\bo \theta_b^* - \bo \theta_a^*)\t {\mb S_a^*}^{-1} (\bo \theta_b^* - \bo \theta_a^*) \geq 0.
    \end{align}
    Therefore, we can know that $\mathsf{SNR}_{a,b} \leq\norm{{\mb S_a^*}^{-\frac{1}{2}}(\bo \theta_b^* - \bo \theta_a^*)}\leq \frac{\norm{\bo \theta_a^* - \bo \theta_b^*}}{\underline \sigma }$. 

    \emph{Lower Bound for $\mathsf{SNR}_{a,b}$} It is straightforward by the Cauchy inequality that 
    \begin{align}
        &\mb x\t\big(\mb I - {\mb S_{a}^*}^{\frac{1}{2}}{\mb S_{b}^*}^{-1} {\mb S_{a}^*}^{\frac{1}{2}}\big)\mb x + 2\mb x\t{\mb S_{a}^*}^{\frac{1}{2}}{\mb S_{b}^*}^{-1}{\mb V^{*\top}}\big(\bo \theta_{b}^* - \bo \theta_{a}^*\big) -  \big(\bo \theta_{b}^* - \bo \theta_{a}^*\big)\t \mb V^*{\mb S_{b}^*}^{-1}{\mb V^{*\top}}\big(\bo \theta_{b}^* - \bo \theta_{a}^*\big) \\ 
        \leq & \norm{\mb I -  {\mb S_{a}^*}^{\frac{1}{2}}{\mb S_{b}^*}^{-1} {\mb S_{a}^*}^{\frac{1}{2}}} \norm{\mb x}_2^2 + 2\norm{{\mb S_{a}^*}^{\frac{1}{2}}{\mb S_{b}^*}^{-1}{\mb V^{*\top}}\big(\bo \theta_{b}^* - \bo \theta_{a}^*\big)}_2 \norm{\mb x}  - \big(\bo \theta_{b}^* - \bo \theta_{a}^*\big)\t \mb V^*{\mb S_{b}^*}^{-1}{\mb V^{*\top}}\big(\bo \theta_{b}^* - \bo \theta_{a}^*\big)\\ 
        \leq & (1 + \frac{\lambda_{\max}}{\lambda_{\min}}) \norm{\mb x}_2^2 + 2 \frac{\sqrt{\lambda_{\max}}}{\lambda_{\min}}\norm{\bo \theta_b^* - \bo \theta_a^*}_2\norm{\mb x}_2 - \frac{\norm{\bo \theta_b^* - \bo \theta_a^*}_2^2 }{\lambda_{\max}}. 
    \end{align}

    Since the only positive zero of $(1 + \frac{\lambda_{\max}}{\lambda_{\min}})t^2 + 2 \frac{\sqrt{\lambda_{\max}}}{\lambda_{\min}}t - \frac{1}{\lambda_{\max}}$ is 
    $$\frac{-\sqrt{\lambda_{\max}} + \sqrt{\lambda_{\max} + \frac{\lambda_{\max}(\lambda_{\min} + \lambda_{\max})}{\lambda_{\max}}}}{2(\lambda_{\max} + \lambda_{\min})},$$ 
    we can conclude that 
    \begin{equation}
        \mathsf{SNR}_{a,b} \geq \frac{-\sqrt{\lambda_{\max}} + \sqrt{\lambda_{\max} + \frac{\lambda_{\min}(\lambda_{\min} + \lambda_{\max})}{\lambda_{\max}}}}{2(\lambda_{\max} + \lambda_{\min})} \norm{\bo \theta_b^* - \bo \theta_a^*}_2.
    \end{equation}
    To simplify the above lower bound, we write it as 
    \begin{align}
        & \frac{-\sqrt{\lambda_{\max}} + \sqrt{\lambda_{\max} + \frac{\lambda_{\min}(\lambda_{\min} + \lambda_{\max})}{\lambda_{\max}}}}{\lambda_{\min} + \lambda_{\max}} = \lambda_{\min}^{-\frac{1}{2}}\frac{-\sqrt{\tau} + \sqrt{\tau + \frac{1 + \tau}{\tau}}}{1 + \tau}.
    \end{align}
    If $\tau \in [1,2]$, there exists a positive constant $c_1$ that lower-bounds $\frac{-\sqrt{\tau} + \sqrt{\tau + \frac{1 + \tau}{\tau}}}{1 + \tau}$. If $\tau \geq 2$, by the Taylor expansion on $\sqrt{x}$, for some positive constant $c_2$ one has 
    \begin{align}
        & \frac{-\sqrt{\tau} + \sqrt{\tau + \frac{1 + \tau}{\tau}}}{1 + \tau} = \frac{1}{1 +\tau} \frac{\frac{1+\tau}{\tau}\cdot }{2\sqrt{\tau}} - \frac{1}{1 + \tau} \frac{(\frac{1 +\tau}{\tau})^2}{4\tau_0^{\frac{3}{2}}} \geq \frac{1}{2\sqrt{\tau}} - \frac{1 + \tau}{2^{\frac{7}{2}}\tau^2} \geq \frac{c_2}{\sqrt{\tau}}
    \end{align}
    where $\tau_0 \in [2, \tau]$ arises by the mean-value theorem.
    Therefore, we can conclude that $\mathsf{SNR}_{a,b} \geq c\lambda_{\min}^{-\frac{1}{2}}\tau^{-\frac{1}{2}}\norm{\bo \theta_a^* - \bo \theta_b^*}_2$ for some positive constant $c$.

    For the last inequality, we notice that 
    \begin{align}
         &\mathsf{SNR}_{a,b}^2 = \min_{\mb x \in \bb R^K}\big\{(\mb x - \mb w_a^*)\t {\mb S^*_a}^{-1}(\mb x- \mb w_a^*):    \underbrace{(\mb x - \mb w_{a}^*)\t {\mb S^*_{a}}^{-1}(\mb x- \mb w_{a}^*)}_{\geq \frac{1}{\bar\sigma^2}\norm{\mb x - \mb w_{a}^*}_2^2 } =\underbrace{ (\mb x - \mb w_{b}^*)\t {\mb S^*_{b}}^{-1}(\mb x- \mb w_{b}^*)}_{\geq \frac{1}{\bar\sigma^2}\norm{\mb x - \mb w_{b}^*}_2^2}
\big\} \\ 
& \geq   \min_{\mb x \in \bb R^K}\big\{\frac{1}{\bar\sigma^2} \norm{\mb x- \mb w_a^*}_2^2:\frac{1}{\bar\sigma^2}\norm{\mb x - \mb w_{a}^*}_2^2  = \frac{1}{\bar\sigma^2}\norm{\mb x - \mb w_{b}^*}_2^2 \big\}  
=  \frac{1}{4\bar\sigma^2 }\triangle^2 \geq \frac{1}{4\tau^2}\underline \omega, 
    \end{align}
    which leads to 
\eq{
    \frac{1}{2\tau} \underline \omega^{\frac12}\leq \mathsf{SNR}_{a,b}.
}

Regarding the upper bound, without loss of generality, we suppose that $\omega_{1,2} = \underline \omega$. Then $\mathsf{SNR}_{1,2}^2 \leq\big( t_0 (\mb w_2^*- \mb w_1^*) \big)\t{\mb S_1^*}^{-1}\big( t_0 (\mb w_2^*- \mb w_1^*) \big)\leq \underline \omega $ with $t_0\in (0,1)$ satisfying 
\eq{
\big( t_0 (\mb w_2^*- \mb w_1^*) \big)\t{\mb S_1^*}^{-1}\big( t_0 (\mb w_2^*- \mb w_1^*) \big)=  \big((1 - t_0 )(\mb w_2^*- \mb w_1^*)\big)\t{\mb S_2^*}^{-1}\big((1 - t_0 )(\mb w_2^*- \mb w_1^*)\big). 
}

By definition, we have ${\mathsf{SNR}}^2 \leq \underline \omega$. 

\subsubsection{Proof of Lemma~\ref{lemma: SNR' and SNR}}
\label{subsubsec: proof of lemma SNR' and SNR}
This lemma is a direct conclusion of \cite[Lemma~C.9]{chen2024optimal}. For completeness, we present their result as follows.  
\begin{lemma}[Lemma~C.9 in \cite{chen2024optimal}]
\label{lemma: C9 in chen2024optimal}
    Consider any $\bo \theta\in \bb R^d \backslash \{0\}$ and $\mb \Sigma \in \bb R^{d\times d}$ that is positive semi-definite. Let $\lambda_{\max}, \lambda_{\min} >0$ be the largest and smallest eigenvalues of $\mb \Sigma$, respectively. For any $t \in \bb R$, define 
    $$ D(t) =\left\{x \in \bb R^d: \norm{\mb x}_2^2 \geq \big( \mb x - \bo \theta\big)\t \mb \Sigma^{-1}\big( \mb x - \bo \theta\big) + t \right\},
    $$
    and $s(t) = \min_{x\in D(t)} \norm{\mb x}_2$. Then if $ -\norm{\bo \theta}_2^2 / (8\lambda_{\max}) < t, t' <  \norm{\bo \theta}_2^2 / 8 $, we have 
    \eq{
    \left| s(t') - s(t) \right| \leq \lambda_{\max} \frac{t' - t }{2 \min\{\sqrt{\lambda_{\min} / 8}, 1/2\} \norm{\bo \theta}_2 }. 
    }
    
\end{lemma}

Let the quantities in the expression $D(t)$
from \cite[Lemma~C.9]{chen2024optimal} be 
\begin{equation}
    d = K, \quad \mb \Sigma = {\mb S^*_a}^{-\frac{1}{2}}{\mb S_b^*}{\mb S^*_a}^{-\frac{1}{2}},\quad  \bo \theta = {\mb S^*_a}^{-\frac{1}{2}}\big(\mb w_b^* - \mb w_a^*\big). 
\end{equation}

We first notice that $1/\kappa_{\mathsf{cov}}^2 \leq \lambda_{\min} \leq \lambda_{\max} \leq \kappa_{\mathsf{cov}}^2$ and $\norm{\bo \theta}_2^2 = \omega_{a,b}$. 
Given $\delta = o(\frac{1}{\kappa_{\mathsf{cov}}^2})$, the condition to apply their conclusion is satisfied for every sufficiently large $n$. Letting $t = 0$, $t' = \delta \omega_{a,b}$, one has: 
\begin{align}
    & |\mathsf{SNR}^{\mathsf{perturbed}}_{a,b} - \mathsf{SNR}_{a,b} | \lesssim  \kappa_{\mathsf{cov}}^3 \delta \omega_{a,b}^{\frac12}
\end{align}
for some constant $c_{\mathsf{SNR}} \geq 0$. Finally, invoking the relation $\mathsf{SNR}_{a,b} \gtrsim \tau^{-1} \omega_{a,b}^{\frac{1}{2}}$ from Lemma~\ref{lemma: SNR and distance} yields that 
\begin{align}
    \mathsf{SNR}_{a,b}^{\mathsf{perturbed}}(\delta) \geq (1 - c_{\mathsf{SNR}} \kappa_{\mathsf{cov}}^4 \delta \omega_{a,b}^{\frac12}) \mathsf{SNR}_{a,b}. 
\end{align}

\subsection{Proof of Proposition~\ref{proposition: spectral initialization}}
\label{subsec: proof of initialization proposition}
This proposition departs from prior theory \cite{abbe2022,zhang2024leave}, which focuses on bounding the expectation $\mathbb{E}[h(\hat{\mathbf z}^{(0)}, \mathbf z^*)]$. It is obtained by synthesizing several key components developed in the proofs of Theorems~\ref{theorem: upper bound for algorithm (gaussian)} and~\ref{theorem: upper bound for algorithm (bounded)}.

A fundamental limitation of expectation-based bounds is that an inequality of the form $\bb E[h(\hat{\mb z}^{(0)}, \mb z^*)] \leq \exp(-(1+o(1))\xi^2)$ only implies $h(\hat{\mb z}^{(0)}, \mb z^*) \leq \exp(-(1+o(1))\xi^2 + \xi'^2)$ with probabilty $1 - \exp(-\xi'^2)$.  For this bound to be nontrivial, the tail probability $\exp(-\xi'^2)$ must be at least on the order of $\exp\big(-(1+o(1))\xi^2\big)$, which is typically not smaller than $\exp\big(-(1+o(1))\mathsf{SNR}^2/2\big)$. As a result, such bounds generally fail to yield sufficiently sharp high-probability guarantees. 
To overcome this issue, we instead leverage our singular subspace perturbation theory to derive a high-probability bound directly. A major technical challenge is controlling the sum $\sum_{i\in[n]}\ind\big\{\bignorm{\mb E_{i,\cdot} \mc P_{-i,\cdot}(\mb E)\t \mb U^*{\bo \Lambda^*}^{-2} }_2\geq \frac{1}{16} \sqrt{K/\beta n}  \big\}$. 

We start with invoking Lemma~D.1 in \cite{abbe2022} to build up the relationship between $h(\hat{\mb z}^{(0)}, \mb z^*)$ and $\norm{\mb U\mb O - \tilde{\mb U^*}}_F^2$, where the $i$-th row of $\tilde{\mb U}^*$ is defined as $\mb U^*_{i,\cdot} + \mathsf{bias}_{z_i^*}$. To apply their lemma, we first leverage the inequalities in the proof of Lemma~\ref{lemma: projected covariance matrix estimation error with true z} to imply that 
\begin{align}
    & \norm{\mb U \mb O - \tilde{\mb U}^*}_F^2 \lesssim \frac{K \bar \sigma^2_{\mathsf{cov}} n }{(\sigma_{K}^*)^2} \ll \frac{K}{\beta n} \cdot \frac{n}{\beta K^2} \lesssim \frac{\min_{k_1,k_2\in[K]}\norm{\tilde{\mb c}_{k_1}^* - \tilde{\mb c}_{k_2}^*}_2^2 n}{\beta K^2}   
\end{align}
with probability at least $1-O(d^{-c})$ for some constant $c$,
which verifies the ``Low-noise condition'' therein. Here, the second inequality follows from our assumptions on $\snr$ along with \eqref{eq: fact B}. 
As a result, one has 
\begin{align}
    & nh(\hat{\mb z}^{(0)}, \mb z^*) \lesssim \left|\big\{i\in[n]: \bignorm{\mb U_{i}\mb O - \tilde{\mb c}_{z_i^*}^*}_2 \geq \frac{1}{4}\sqrt{K/ \beta n}\big\}\right|
    \label{eq: upper bound for nh(z,z^*)}
\end{align}
with probability at least $1-O(d^{-c})$. 
To further simplify, we invoke Corollary~\ref{corollary: singular subspace perturbation theory in clustering problem} along with \eqref{eq: alternative upper bound for delta_u}, \eqref{eq: fact A}, and \eqref{eq: fact B} to obtain that 
\begin{align}
    & \left|\big\{i\in[n]: \bignorm{\mb U_{i}\mb O - \tilde{\mb c}_{z_i^*}^*}_2 \geq \frac{1}{4}\sqrt{K/ \beta n}\big\}\right|  \\
    \leq & \left|\big\{i\in[n]: \bignorm{\mb U_{i}\mb O - \tilde{\mb c}_{z_i^*}^* - \mk L_i}_2 \geq \deltau \big\}\right| + \left|\big\{i\in[n]: \bignorm{\mk L_i}_2 \geq \frac{1}{8}\sqrt{K/ \beta n} \big\}\right| \\ 
    \leq &  \left|\big\{i\in[n]: \bignorm{\mb U_{i}\mb O - \tilde{\mb c}_{z_i^*}^* - \mk L_i}_2 \geq \deltau \big\}\right| + \left|\big\{i\in[n]: \bignorm{\mb E_{i,\cdot} \mb V^*{\bo \Lambda^*}^{-1}}_2 \geq \frac{1}{16}\sqrt{K/ \beta n} \big\}\right|  \\ 
    & +\left|\big\{i\in[n]: \bignorm{\mb E_{i,\cdot} \mc P_{-i,\cdot}(\mb E)\t \mb U^*{\bo \Lambda^*}^{-2}}_2 \geq \frac{1}{16}\sqrt{K/ \beta n} \big\}\right|. 
    \label{eq: error decomposition for initialization}
\end{align}

We first consider the regime
$\snr \lesssim \kappa \kappa_{\mathsf{cov}}\sqrt{\log n}$.
We analyze the three terms on the right-hand side of
\eqref{eq: error decomposition for initialization} separately.

\emph{First term. }
For the first term, it follows from Markov's inequality, the definitiion of $\delta_u$, and Corollary~\ref{corollary: singular subspace perturbation theory in clustering problem} that 
\begin{equation}
    \left|\big\{i\in[n]: \bignorm{\mb U_{i}\mb O - \tilde{\mb c}_{z_i^*}^* - \mk L_i}_2 \geq \deltau \big\}\right| \ll \frac{n}{\beta K(\log d)^4 \snr^2 \nu}
\end{equation}
with probability $1 -  \beta K (\log d)^4 \snr^2 \nu \exp(-(1 + o(1))\frac{\snr^2}{2}) = 1 - \exp(-(1 + o(1))\frac{\snr^2}{2})$. 

\emph{Second term. } For each $i\in[n]$, we have 
\begin{align}
    & \bb P\big[\bignorm{\mb E_{i,\cdot} \mb V^*{\bo \Lambda^*}^{-1}}_2 \geq \frac{1}{16}\sqrt{K/ \beta n} \big\}\big] \leq  \bb P\big[\bignorm{\mb E_{i,\cdot} \mb V^*}_2 \geq \frac{1}{16}\sqrt{K/ \beta n} \big\} \sigma_K^* \big]  \\ 
    \stackrel{\text{by \eqref{eq: fact B}}}{\leq} & \bb P\big[\bignorm{\mb E_{i,\cdot} \mb V^*}_2 \geq \frac{1}{16}\frac{\snr \bar \sigma_{\mathsf{cov}}}{\kappa \kappa_{\mathsf{cov}} \sqrt{\beta}}\big\} \sigma_K^* \big]   
    \leq   \bb P\big[\bignorm{\mb E_{i,\cdot} \mb V^*}_2 \geq C_1  \bar \sigma_{\mathsf{cov}} \sqrt{\log \log d + \log (\kappa\kappa_{\mathsf{cov}}) } \big]  \\ 
   \leq & \bb P\big[\bignorm{\mb E_{i,\cdot} \mb V^*}_2 \geq C_2  \bar \sigma_{\mathsf{cov}} \sqrt{\log \log d + \log \snr  } \big]  \leq \frac{C_3}{ \nu \snr^2 \beta K(\log d)^4} 
   \label{eq: second term concentration in initialization analysis}
\end{align}
with an sufficiently large $C_1$ and $C_2 = \frac{1}{2}C_1$ and some sufficiently small $C_3$. In the second-to-last inequality, we used Hanson-Wright inequality and Lemma~\ref{eq: 2k moment of At Ei}. 
Notice that $\ind\{\bignorm{\mb E_{i,\cdot} \mb V^*{\bo \Lambda^*}^{-1}}_2 \geq \frac{1}{16}\sqrt{K/ \beta n} \}$ are independent across $i$, Bernstein's inequality yields
\begin{align}
      \left|\big\{i\in[n]: \bignorm{\mb E_{i,\cdot} \mb V^*{\bo \Lambda^*}^{-1}}_2 \geq \frac{1}{16}\sqrt{K/ \beta n} \big\}\right| \leq& n  \frac{C_3}{ \nu \snr^2 \beta K(\log d)^4} + c_b \Big[n^{\frac34}\sqrt{ \frac{C_3}{ \nu \snr^2 \beta K(\log d)^4}} + n^{\frac14}\Big] \\ 
    \lesssim & 2n  \frac{C_3}{ \nu \snr^2 \beta K(\log d)^4} 
    \label{eq: Bersntein application for second term in initialization analysis}
\end{align}
with probability at least $1- e^{-\sqrt{n}}$ for some constant $c_b$, where we used the fact that $$\frac{C_3}{ \nu \snr^2 \beta K(\log d)^4} \gtrsim \frac{C_3}{ \kappa \kappa_{\mathsf{cov}}^2 \nu \beta K(\log d)^5}\gg n^{-\frac12}. $$

\emph{Third Term. }The last term in \eqref{eq: error decomposition for initialization} is more delicate, since the indicator summands are no longer independent. 
To control the dependence, we first recall the definition of $\bar \sigma_{\mathsf{cov}}$ and obtain from the Hanson-Wright inequality or the Bernstein inequality that, with probability at least $1-O(d^{-c})$,
\begin{equation}
\label{eq: E_i E_j inner product upper bound}
|\ip{\mb E_i,\mb E_j}|
 \lesssim 
\bar\sigma_{\mathsf{cov}}\,\sigma_K^* \sqrt{\log d},
\qquad \forall\, i\neq j\in[n].
\end{equation}
In what follows we work on the event \eqref{eq: E_i E_j inner product upper bound}. 
Then one has 
\begin{align}
    & \ind\big\{\bignorm{\mb E_{i,\cdot} \mc P_{-i,\cdot}(\mb E)\t \mb U^*{\bo \Lambda^*}^{-2} }_2\geq \frac{1}{16} \sqrt{K/\beta n}  \big\} \leq \ind\big\{\bignorm{\mb E_{i,\cdot} \mc P_{-i,\cdot}(\mb E)\t \mb U^* }_2 \geq \frac{(\sigma_K^*)^2}{16} \sqrt{K/\beta n}  \big\} \\ 
    \leq & \sum_{k\in[K]} \ind\big\{\bignorm{\mb E_{i,\cdot} \mc P_{\mc I_{k}(\mb z^*)\backslash\{i\},\cdot}(\mb E)\t }_2^2 \geq \frac{(\sigma_K^*)^4}{(16)^2 \beta^2 }  \big\} =  \sum_{k\in[K]} \ind\big\{\sum_{i' \in \mc I_k(\mb z^*) \backslash\{i\}} |\ip{\mb E_{i}, \mb E_{i'}}|^2 \geq \frac{(\sigma_K^*)^4}{(16)^2 \beta^2 }   \big\}  \\ 
    \stackrel{\text{by \eqref{eq: fact B}}}{\leq} &  \sum_{k\in[K]} \ind\big\{\sum_{i' \in \mc I_k(\mb z^*) \backslash\{i\}} |\ip{\mb E_{i}, \mb E_{i'}}|^2 \geq \frac{n\snr^2 (\bar\sigma_{\mathsf{cov}})^2 (\sigma_K^*)^2 }{(16)^2 K \kappa^2 \kappa_{\mathsf{cov}}^2 \beta^2}  \big\} .
\end{align}

For $k\in[K]$, we let $a_k$ be $\sum_{i' \in \mc I_k(\mb z^*) \backslash\{i\}} \ind\big\{ |\ip{\mb E_{i}, \mb E_{i'}}|^2 \geq \frac{\snr^2 (\bar\sigma_{\mathsf{cov}})^2 (\sigma_K^*)^2 }{2 (16)^2  \kappa^2  \kappa_{\mathsf{cov}}^2 \beta^3 }  \big\}$. Under the condition \eqref{eq: E_i E_j inner product upper bound}, the event $\sum_{i' \in \mc I_k(\mb z^*) \backslash\{i\}} |\ip{\mb E_{i}, \mb E_{i'}}|^2 \geq \frac{n\snr^2 (\bar\sigma_{\mathsf{cov}})^2 (\sigma_K^*)^2 }{16 K \kappa^2 \kappa_{\mathsf{cov}}^2 \beta^2 }$ implies that 
\eq{
    a_k \bar \sigma_{\mathsf{cov}}^2 (\sigma_K^*)^2 \log d + (\frac{\beta n}{K} - a_k) \frac{\snr^2 (\bar\sigma_{\mathsf{cov}})^2 (\sigma_K^*)^2}{2 (16)^2 \kappa^2 \kappa_{\mathsf{cov}}^2 \beta^3  } \geq \frac{n\snr^2 (\bar\sigma_{\mathsf{cov}})^2 (\sigma_K^*)^2}{(16)^2 K \kappa^2 \kappa_{\mathsf{cov}}^2 \beta^2},
}
which is equivalent to
\longeq{
    & a_k \geq  \Big[\bar \sigma_{\mathsf{cov}}^2 (\sigma_K^*)^2 \log d - \frac{\snr^2 (\bar\sigma_{\mathsf{cov}})^2(\sigma_K^*)^2 }{2 (16)^2 \kappa^2 \kappa_{\mathsf{cov}}^2 \beta^3 } \Big]^{-1} \frac{n\snr^2 (\bar\sigma_{\mathsf{cov}})^2 (\sigma_K^*)^2}{2(16)^2 K \kappa^2 \kappa_{\mathsf{cov}}^2 \beta^2} \geq \frac{\snr^2 }{2(16)^2 \log d \kappa^2 \kappa_{\mathsf{cov}}^2 \beta^3 } \cdot \frac{n \beta }{K} . 
}

As a consequence, one can rewrite the sum of $n$ indicators as the sum of $n^2$ indicator, which is later viewed as a U-statistic:
\begin{align}
    & \sum_{i\in[n]}\ind\big\{\bignorm{\mb E_{i,\cdot} \mc P_{-i,\cdot}(\mb E)\t \mb U^*{\bo \Lambda^*}^{-2} }_2\geq \frac{1}{16} \sqrt{K/\beta n}  \big\}\\ 
    \leq & \sum_{i\in[n]} \sum_{k\in[K]} \frac{\sum_{i' \in \mc I_k(\mb z^*) \backslash\{i\}} \ind\big\{ |\ip{\mb E_{i}, \mb E_{i'}}|^2 \geq \frac{\snr^2 (\bar\sigma_{\mathsf{cov}})^2 (\sigma_K^*)^2 }{2 (16)^2  \kappa^2  \kappa_{\mathsf{cov}}^2 \beta^3 }  \big\}}{\frac{\snr^2 }{2(16)^2 \log d \kappa^2 \kappa_{\mathsf{cov}}^2 \beta^3 } \cdot \frac{n \beta }{K}} \\ 
    = & \sum_{i_1\neq i_2 \in[n]} \ind\big\{ |\ip{\mb E_{i_1}, \mb E_{i_2}}|^2 \geq \frac{\snr^2 (\bar\sigma_{\mathsf{cov}})^2 (\sigma_K^*)^2 }{2 (16)^2  \kappa^2  \kappa_{\mathsf{cov}}^2 \beta^3 }  \big\} \cdot \Big[\frac{\snr^2 }{2(16)^2 \log d \kappa^2 \kappa_{\mathsf{cov}}^2 \beta^3 } \cdot \frac{n \beta }{K}\Big]^{-1}. 
    \label{eq: decomposition for the third term in initialization analysis}
\end{align}
We therefore apply the decoupling inequality
\cite[Theorem~1]{de1995decoupling} to obtain that, for any $x\in\bb R$,
\begin{align}
    & \bb P\Big[\sum_{i_1\neq i_2 \in[n]} \ind\big\{ |\ip{\mb E_{i_1}, \mb E_{i_2}}|^2 \geq \frac{\snr^2 (\bar\sigma_{\mathsf{cov}})^2 (\sigma_K^*)^2 }{2 (16)^2  \kappa^2  \kappa_{\mathsf{cov}}^2 \beta^3 }  \big\} \cdot \Big[\frac{\snr^2 }{2(16)^2 \log d \kappa^2 \kappa_{\mathsf{cov}}^2 \beta^3 } \cdot \frac{n \beta }{K}\Big]^{-1} \geq x \Big] \\ 
    \leq & C \bb P\Big[\sum_{i_1\neq i_2 \in[n]} \ind\big\{ |\ip{\mb E_{i_1}, \mb E'_{i_2}}|^2 \geq \frac{\snr^2 (\bar\sigma_{\mathsf{cov}})^2 (\sigma_K^*)^2 }{2 (16)^2  \kappa^2  \kappa_{\mathsf{cov}}^2 \beta^3 }  \big\} \cdot \Big[\frac{\snr^2 }{2(16)^2 \log d \kappa^2 \kappa_{\mathsf{cov}}^2 \beta^3 } \cdot \frac{n \beta }{K}\Big]^{-1}\geq C x \Big]      
\end{align}
for every $x \in\bb R$ with some constant $C$, where $\mb E'$ is an independent copy of $\mb E$. This together with \eqref{eq: decomposition for the third term in initialization analysis} gives rise to 
\begin{align}
    & \bb P\bigg[\Big|\big\{i\in[n]: \bignorm{\mb E_{i,\cdot} \mc P_{-i,\cdot}(\mb E)\t \mb U^*{\bo \Lambda^*}^{-2}}_2 \geq \frac{1}{16}\sqrt{K/ \beta n} \big\}\Big| > n \cdot \frac{C_4}{\nu \snr^2 \beta K(\log d)^4}\bigg]  \\ 
    \leq & O(d^{-c}) + \bb P\big[ \sum_{i_1\neq i_2 \in[n]} \ind\big\{ |\ip{\mb E_{i_1}, \mb E'_{i_2}}|^2 \geq \frac{\snr^2 (\bar\sigma_{\mathsf{cov}})^2 (\sigma_K^*)^2 }{2 (16)^2  \kappa^2  \kappa_{\mathsf{cov}}^2 \beta^3 }  \big\} \geq C n^2\cdot \frac{C_4}{2(16)^2 \kappa^2 \kappa_{\mathsf{cov}}^2\nu \beta^3 K^2 (\log d)^5}\big] \\ 
    \leq & O(d^{-c}) + \bb P\big[ \sum_{i_1\neq i_2 \in[n]} \ind\big\{ |\ip{\mb E_{i_1}, \mb E'_{i_2}}| \geq \bar \sigma_{\mathsf{cov}} \sigma_K^* \sqrt{\log \log d + \log (\kappa\kappa_{\mathsf{cov}})} \big\} \geq C n^2\cdot \frac{C_4}{ \snr^2 (\log d)^5}\big] \\ 
    \leq & O(d^{-c}) + O(e^{-\sqrt{n}}),
\end{align}
where the last step follows from an analogous argument to \eqref{eq: second term concentration in initialization analysis} and \eqref{eq: Bersntein application for second term in initialization analysis} together with Lemmas~\ref{lemma: noise matrix concentrations using the universality (Gaussian)},\ref{lemma: noise matrix concentrations using the universality (bounded)}. 

With these high-probability bounds in place, we complete the proof by combining
\eqref{eq: upper bound for nh(z,z^*)} and \eqref{eq: error decomposition for initialization}. 

\bigskip 
On the other hand, in the high-SNR regime
$\snr \gg \kappa \kappa_{\mathsf{cov}}\sqrt{\log n}$,
combining Corollary~\ref{corollary: singular subspace perturbation theory in clustering problem},
Lemma~\ref{lemma: noise matrix concentrations using the universality (Gaussian)}
or~\ref{lemma: noise matrix concentrations using the universality (bounded)}
along with \eqref{eq: upper bound for nh(z,z^*)} and \eqref{eq: error decomposition for initialization} yields 
$ h(\hat{\mb z}^{(0)}, \mb z^*) =0
$
with probability at least $1-O(n^{-c})$, where we used the fact that 
$$\left|\big\{i\in[n]: \bignorm{\mb E_{i,\cdot} \mb V^*{\bo \Lambda^*}^{-1}}_2 \geq \frac{1}{16}\sqrt{K/ \beta n} \big\}\right|=\left|\big\{i\in[n]: \bignorm{\mb E_{i,\cdot} \mc P_{-i,\cdot}(\mb E)\t \mb U^*{\bo \Lambda^*}^{-2}}_2 \geq \frac{1}{16}\sqrt{K/ \beta n} \big\}\right|=0$$ 
holds with probability at least $1- O(d^{-c})$, 
since $\bignorm{\mb E_{i,\cdot} \mb V^*{\bo \Lambda^*}^{-1}}_2 \geq \frac{1}{16}\sqrt{K/ \beta n} $ implies that 
\begin{align}
    & \bignorm{\mb E_{i,\cdot} \mb V^*}_2 \geq \frac{\sigma_1^*}{\kappa}\frac{1}{16}\sqrt{K/ \beta n} \geq \frac{n}{\kappa K} \min_{k_1\neq k_2\in[K]} \norm{\bo \theta_{k_1}^* - \bo \theta_{k_2}^*}_2\geq \frac{n}{\kappa\kappa_{\mathsf{cov}} K}  \bar \sigma_{\mathsf{cov}} \snr  \gg \bar \sigma \sqrt{\log n},
\end{align} 
and the same argument applies to $\bignorm{\mb E_{i,\cdot} \mc P_{-i,\cdot}(\mb E)\t \mb U^*{\bo \Lambda^*}^{-2}}_2$.

\section{Auxiliary Lemmas}\label{sec-proof of lemmas}
The following lemma comes from \cite[Proposition 3.2.1]{PeccatiTaqqu2011}. 
     	\begin{lemma}[Leonov-Shiryaev]
       	\label{lemma: Leonov-Shiryaev}
       		We can write 
       		\longeq{
       		&\mathbb E[W_1\cdots W_m] = \sum_{\pi\in \mc P([m])}\prod_{p\in \pi} \kappa(W_p),
       		}
         where $\mc P([m])$ denotes all possible partitions of $[m]$. 
       		Moreover, we have 
       		\longeq{
       		& \kappa(W_1, \cdots, W_m) = \sum_{\pi\in P([m])} (-1)^{|\pi|-1} (|\pi| - 1)!\prod_{p\in \pi} \bb E\Big[\prod_{j\in p}W_j\Big].
       		}
       	\end{lemma}
        
\begin{lemma}[Generalized Modified Logarithmic Sobolev Inequality I]\label{lemma: Generalized Modified Logarithmic Sobolev Inequality I}
Let $\mb X_1, \cdots, \mb X_n \in \bb R^p$ be independent random vectors and let $f: ([0,1]^p)^n \rightarrow \bb R$ be a separately convex function, namely, $f(\mb x_1, \cdots, \mb x_{i-1}, \cdot, \mb x_{i+1}, \cdots, \mb x_n)$ be a convex function of the $i$-th vector if the rest of the vectors are fixed. We also assume that $|f(\mb x) - f(\mb y)| \leq \norm{\mb x - \mb y}_2$ for all $\mb x = (\mb x_1, \cdots, \mb x_n), \mb y = (\mb y_1, \cdots, \mb y_n) \in ([0,1]^p)^n$. Then for $Z = f(\mb X_1, \cdots,   \mb X_n)$, it holds for all $t >0$ that $\bb P[Z > \bb E Z + t] \leq e^{-\frac{t^2}{2p}}$.
Moreover, the moments of $Z$ could be bounded by 
$
\bb E[Z^l] \leq C(pl)^{\frac{l}{2}}
$
for some constant $C$. 
\end{lemma}

\begin{proof}[Proof of Lemma~\ref{lemma: Generalized Modified Logarithmic Sobolev Inequality I}]
    Denote that $Z_i = \inf_{\mb x \in [0,1]^p} f(\mb X_1, \cdots, \mb X_{i-1}, \mb x_{i}, \mb X_{i+1}, \cdots, \mb X_n) $. Then by \cite[Theorem 6.6]{boucheron2003concentration} one has 
    \longeq{
     & \lambda \bb E[Ze^{\lambda Z}] - \bb E[e^{\lambda Z}] \log \bb E[e^{\lambda Z}] \leq \sum_{i\in[n]}\bb E\Big[e^{\lambda Z}\frac{\lambda^2( Z - Z_i)^2}{2}\Big]\\ 
    \leq &\bb E\Big[e^{\lambda Z}\frac{\lambda^2\sum_{i\in[n]}\norm{\nabla_if(\mb X)}_2^2\norm{\mb X_i - \mb X_i'}_2^2}{2}\Big] \leq \bb E\Big[e^{\lambda Z}\frac{\lambda^2\norm{\nabla f(\mb X)}_2^2 p}{2}\Big] \leq \bb E[e^{\lambda Z}]\frac{p\lambda^2}{2}
    }
    for all $\lambda\in \bb R$ where $\nabla_i f(\mb X)$ denotes the gradient vector of the function $\tilde f_i \coloneqq f$ of the $i$-th vector. 

    Then it could be written as $\frac{\mathrm d}{\mathrm d\lambda }(\frac{\log\bb E[e^{\lambda (Z - \bb E[Z])}]}{\lambda}) \leq \frac{p}{2}$, 
    which leads to $
        \bb E[e^{\lambda (Z - \bb E[Z])}] \leq \frac{p\lambda^2 }{2}$.
    Finally, invoking the Markov inequality yields the desired conclusion. 
\end{proof}

\begin{lemma}
\label{lemma: average upper bound}
Consider (not necessarily independent) random variables $X_1,\dots,X_n$ satisfying 
$|X_i| \le A$ with probability at least $1-p_A$ and $|X_i| \le B$ with probability at least $1-p_B$, where $p_B \ge np_A$ and $B \le A$. Then $
\frac{1}{n}\sum_{i=1}^n |X_i| \le cA + (1 - c)B$
with probability at least $1 - np_A - p_B/c$. 
\end{lemma}

\begin{proof}
Let $\mathcal E = \{ |X_i| \le A \text{ for all } i \in [n]\}$. Then $\Pr(\mathcal E^c) \le np_A$. Conditional on $\mathcal E$, each $X_i$ satisfies $|X_i| \le B$ with probability at least $(1 - p_B - np_A)/(1 - np_A)$. Hence
$\mathbb E\big[ n^{-1}\sum_{i=1}^n \mathbf 1\{|X_i| \ge B\} \,\big|\, \mathcal E \big] \le (p_B + np_A)/(1 - np_A)$.  
By Markov's inequality, for any $c>0$,
\[
\bb P \big[ n^{-1}\sum_{i=1}^n \mathbf 1\{|X_i| \ge B\} \ge c | \mathcal E \big]
\le \frac{p_B}{(1 - np_A)c}.
\]
Thus, with probability at least $1 - np_A - p_B /c$, at most a fraction $c$ of the variables exceed $B$, and therefore 
$\sum_{i=1}^n |X_i| \le cA + (1-c)B$.  
\end{proof}

\section{Additional Experiments}
\label{sec: additional experiments}
In addition to the experiments in the main text, we examine the Gaussian mixture model over a broad range of dimensions, benchmarking our approach against the EM algorithm, spectral methods, and K-means. Recall that the spectral methods apply K-means to the rows of low-dimensional embedding $\hat{\mb U} \hat{\bo \Lambda}$; in \textbf{spectral clustering}, $\hat{\mb U}$ and the diagonal entries of $\hat{\bo \Lambda}$ correspond to the top-$K$ singular vectors and singular values of $\mb R$; 
in \textbf{hollowed spectral clustering}, they correspond to the top-$K$ eigenvectors of $\mc H(\mb R \mb R^\top)$ and the square roots of the corresponding eigenvalues.
The cluster centers are defined as $\bo \theta_1^* = (\alpha \cdot \mb 1_{p/2}, \mb 0_{p/2})$ and $\bo \theta_2^* = (\mb 0_{p/2}, \alpha \cdot \mb 1_{p/2})$. To introduce heterogeneity in covariance structures, we set the covariance matrices as $\diag(25 \cdot \mb I_{p/2}, \mb I_{p/2})$ for one cluster and $\diag(\mb I_{p/2}, 25 \cdot \mb I_{p/2})$ for the other. 
For the EM implementation, we use the \texttt{mvnormalmixEM} function from the R-package \texttt{mixtools} and assign each data point to the class with the largest posterior probability, based on the estimated parameters. 
Note that the empirical performance of the hard-EM algorithm proposed by \cite{chen2024optimal}  is similar to that of EM presented here, because they both require inverting $p\times p$ sample covariance matrices. 
As shown in Table~\ref{table: em comparison}, the {EM} algorithm frequently encounters singularity issues when inverting $p \times p$ covariance matrices, leading to failures in a significant proportion of Monte Carlo simulations; e.g., for $p=200$, EM can only run without failures for 40.5\% of the 200 simulation trials. So we only present results for the EM algorithm for $p\leq 200$. Note that we calculate ``EM err.'' by averaging over the successful trials, which actually leads to an optimistic approximation to EM's clustering performance in the first five rows of Table~\ref{table: em comparison}. On the other hand, the {K-Means} algorithm and the spectral method show similar clustering performance but are surpassed by our proposed method.
Table~\ref{table: em comparison} also shows that our COPO method is computationally very efficient, taking only 0.2 second on average for $(n,p)=(500,5000)$.

\begin{table}[htbp]
\small\addtolength{\tabcolsep}{-4pt}
\centering
\begin{tabular}{ccccccccc}
\toprule
    $n$ & $p$  & K-means & Spec.  & Hollowed-Spec. & COPO  & EM  (\%Suc.) & COPO time & EM time   \\ 
    \midrule
    500 &40 & 0.10  & 0.10 & 0.10 & 0.07 & 0 (97.0\%) & 0.05  & 3.1 \\ 
    500 & 80 & 0.10 & 0.10 & 0.10 & 0.08 & 0.05 (93.5\%) & 0.06  & 18.3 \\
    500 & 120 & 0.11 & 0.11 & 0.11 & 0.08 & 0.20 (85.5\%) & 0.06  & 91.5 \\ 
    500 & 160 & 0.12 & 0.11 & 0.11 & 0.09 & 0.31 (64.5\%) & 0.07  & 64.7 \\ 
    500 & 200 & 0.13 & 0.12 & 0.12 & 0.10 & 0.29 (44.5\%) & 0.07  & 76.3 \\ 
    500 & 500 & 0.20 & 0.18 & 0.18 & 0.16 & -  & 0.13  & - \\ 
    500 & 1000 &  0.30 & 0.29 & 0.28 & 0.27 & -  & 0.25 & - \\ 
    \bottomrule
    \end{tabular}
    \caption{Clustering error rates and computation times with varying dimensions for Gaussian mixtures. ``Spec.'' refers to spectral clustering. ``err.'' refers to the average clustering error rates. The unit of time is seconds. The (\%Suc.) in the second-to-last column means the proportion of simulation trials in which the EM algorithm runs without failures.}
    \label{table: em comparison}
\end{table}

\end{document}